\pdfoutput=1

\documentclass[11pt, a4paper]{article}

\usepackage[T1]{fontenc} 
\usepackage[utf8]{inputenc} 

\usepackage[noblocks]{authblk}

\title{An Additive-Noise Approximation to {K}eller--{S}egel--{D}ean--{K}awasaki Dynamics: {L}ocal Well-Posedness of Paracontrolled Solutions}
\author{Adrian Martini}
\affil{Department of Statistics, University of Oxford;\protect\\ \texttt{martini@stats.ox.ac.uk}, ORCID iD: 0000-0001-9350-1338}

\author{Avi Mayorcas}
\affil{Department of Mathematical Sciences, University of Bath;\protect\\ \texttt{am2735@bath.ac.uk}, ORCID iD: 0000-0003-4133-9740}
\date{\today}

\usepackage{etoolbox} 

\newtoggle{diagrams}
\toggletrue{diagrams} 

\newtoggle{colours}
\toggletrue{colours} 

\newtoggle{details}
\togglefalse{details}

\usepackage[UKenglish]{babel} 
\usepackage{csquotes} 

\usepackage[
	babel=true,
	expansion=false,
	factor=1100
]{microtype} 

\usepackage[skip=10pt plus 1pt, indent=0pt]{parskip} 

\usepackage[all]{nowidow} 

\usepackage{amsmath,amssymb,amsthm} 
\usepackage{mathtools} 
\usepackage{amstext} 

\usepackage{dsfont} 
\usepackage{mathrsfs} 
\usepackage{stmaryrd} 

\usepackage{verbatim} 
\usepackage{mdframed} 

\PassOptionsToPackage{hyphens}{url} 

\usepackage{booktabs}

\usepackage{tikz} 

\usepackage{xcolor}

\definecolor{detailscolor}{RGB}{0, 153, 136} 

\iftoggle{colours}{%
	\definecolor{vectorcolor}{RGB}{238, 51, 119} 
	
	\definecolor{rootcolor}{RGB}{0, 119, 187} 
	}{%
	\definecolor{vectorcolor}{RGB}{180, 180, 180} 
	
	\definecolor{rootcolor}{RGB}{0, 0, 0} 
}%

\usepackage[
	hyperref=true,
	backend=biber,
	style=ext-alphabetic,
	useprefix=true,
	giveninits=true,
	doi=false,
	isbn=false,
	url=false,
	eprint=false
]{biblatex}

\DeclareFieldFormat[article]{number}{\mkbibparens{#1}} 

\DeclareFieldFormat{issuedate}{#1} 

\DeclareFieldFormat[article]{title}{#1} 
\DeclareFieldFormat[inproceedings]{title}{#1} 
\DeclareFieldFormat[incollection]{title}{#1} 
\DeclareFieldFormat[misc]{title}{\mkbibemph{#1}}

\DeclareFieldFormat{journaltitle}{\mkbibemph{#1}\addcomma} 

\renewbibmacro{in:}{} 

\usepackage{hyperref} 

\hypersetup{
	pdfauthor={Adrian Martini, Avi Mayorcas},
	pdftitle={An Additive-Noise Approximation to {K}eller--{S}egel--{D}ean--{K}awasaki Dynamics: {L}ocal Well-Posedness of Paracontrolled Solutions},
	pdfsubject={},
	pdfkeywords={},
	pdfproducer={LaTeX using hyperref},
	pdfcreator={pdfLaTeX from TeXLive},
	pageanchor={true},
	pdfborder={0 0 .33 [3 1]},
	colorlinks=true,
	linkcolor=blue,
	citecolor=blue,
	filecolor=blue,
	urlcolor=blue,
	linkbordercolor=teal,
	citebordercolor=teal,
	filebordercolor=teal,
	urlbordercolor=blue
}

\usepackage{geometry} 

\usepackage{fancyhdr}

\usepackage{ellipsis}

\usepackage{lmodern}

\usepackage{bm} 

\numberwithin{equation}{section}
\newtheorem{theorem}{Theorem}[section]

\newtheorem{lemma}[theorem]{Lemma}
\newtheorem{proposition}[theorem]{Proposition}

\newtheorem{definition}[theorem]{Definition} 
\theoremstyle{remark}
\newtheorem{remark}[theorem]{Remark}
\newtheorem{example}[theorem]{Example}

\let\oldproofname=\proofname
\renewcommand{\proofname}{\rm\bf{\oldproofname}}

\newcommand{\dd}{\mathop{}\!\mathrm{d}\mkern0.5mu} 
\newcommand{\euler}{\mathrm{e}\mkern0.7mu} 
\newcommand{\upi}{\mathrm{i}\mkern0.7mu} 
\newcommand{\uppi}{\pi}

\newcommand{\mcB}{\mathcal{B}}
\newcommand{\mcC}{\mathcal{C}}

\newcommand{\mcF}{\mathcal{F}}

\newcommand{\mcH}{\mathcal{H}}
\newcommand{\mcI}{\mathcal{I}}

\newcommand{\mcS}{\mathcal{S}}

\newcommand{\mcX}{\mathcal{X}}

\newcommand{\mbC}{\mathbb{C}}

\newcommand{\mbE}{\mathbb{E}}

\newcommand{\mbN}{\mathbb{N}}

\newcommand{\mbP}{\mathbb{P}}

\newcommand{\mbR}{\mathbb{R}}

\newcommand{\mbT}{\mathbb{T}}

\newcommand{\mbX}{\mathbb{X}}
\newcommand{\mbY}{\mathbb{Y}}
\newcommand{\mbZ}{\mathbb{Z}}

\newcommand{\bxi}{\boldsymbol{\xi}}

\newcommand{\mfB}{\mathfrak{B}}

\newcommand{\msC}{\mathscr{C}}
\newcommand{\msD}{\mathscr{D}}

\newcommand{\msF}{\mathscr{F}}
\newcommand{\msG}{\mathscr{G}}
\newcommand{\msH}{\mathscr{H}}

\newcommand{\msL}{\mathscr{L}}

\newcommand{\msP}{\mathscr{P}}

\DeclarePairedDelimiter{\floor}{\lfloor}{\rfloor}
\DeclarePairedDelimiter{\abs}{\lvert}{\rvert}
\DeclarePairedDelimiter{\norm}{\lVert}{\rVert}
\newcommand{\inner}[2]{\langle#1,#2\rangle}

\DeclareMathOperator{\trace}{Tr}
\newcommand{\supp}{\text{supp}}
\newcommand{\eps}{\varepsilon}

\newcommand{\tzero}{|_{t=0}}
\newcommand{\bigtzero}{\big|_{t=0}}

\newcommand{\from}{\colon}
\newcommand{\defeq}{\coloneq} 
\newcommand{\embed}{\hookrightarrow}
\newcommand{\per}{\textnormal{per}}

\newcommand{\pa}{\mathbin{\varolessthan}}
\newcommand{\re}{\mathbin{\varodot}}
\newcommand{\vdiv}{\nabla\cdot}

\newcommand{\om}{\omega}
\let\oldhat\hat
\renewcommand{\hat}{\widehat}
\newcommand{\rksnoise}[2]{\mcX^{#1,#2}}
\newcommand{\rdet}{\rho_{\textnormal{det}}}
\newcommand{\het}{\sigma}
\newcommand{\srdet}{\sqrt{\rdet}}
\newcommand{\can}{\textnormal {can}}
\newcommand{\chem}{\chi}

\iftoggle{details}{
	\newenvironment{details}%
	{\color{detailscolor}\textbullet\textbf{Details: }}%
	{\ignorespacesafterend}%
}{%
	\newenvironment{details}%
	{\comment}%
	{\endcomment}%
}%

\newcommand{\multiquad}[1][1]{\hspace*{#1em}\ignorespaces}

\newcommand{\place}{\,\cdot\,}
\newcommand{\shift}[1]{{#1}+\place}

\newcommand{\Tmax}{T_{\textnormal{max}}}

\usepackage{halloweenmath}

\let\oldskull\skull
\def\skull{\mathord{\oldskull}}
\newcommand{\cem}{\skull}

\newcommand{\target}{F} 
\newcommand{\sol}{\textnormal{sol}}
\newcommand{\wt}{\textnormal{wt}}
\newcommand{\weight}{\textnormal{-}\wt}

\makeatletter
\newcommand{\mytextsize}{\f@size pt}
\makeatother

\newlength\RSu 
\newlength\RSwidth 
\newlength\RSbaseline 
\newtoggle{isoption}

\iftoggle{diagrams}{
	\iftoggle{colours}{
		\newcommand{\Cherry}[1]{
			\togglefalse{isoption}
			\ifnumequal{#1}{1}{\toggletrue{isoption}\,\includegraphics[page = 1, width = 0.41cm*\ratio{\mytextsize}{10.95 pt}]{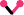}\,}{}
			\ifnumequal{#1}{10}{\toggletrue{isoption}\,\includegraphics[page = 2, width = 0.44cm*\ratio{\mytextsize}{10.95 pt}]{diagrams_coloured.pdf}\,}{}
			\iftoggle{isoption}{}{\PackageError{Cherry}{Undefined option to Cherry: #1}{}}
		}%
		\newcommand{\Ypsilon}[1]{
			\togglefalse{isoption}
			\ifnumequal{#1}{1}{\toggletrue{isoption}\raisebox{-0.4ex}{\includegraphics[page = 3, width = 0.37cm*\ratio{\mytextsize}{10.95 pt}]{diagrams_coloured.pdf}}}{}
			\ifnumequal{#1}{10}{\toggletrue{isoption}\raisebox{-0.4ex}{\includegraphics[page = 4, width = 0.4cm*\ratio{\mytextsize}{10.95 pt}]{diagrams_coloured.pdf}}}{}
			\iftoggle{isoption}{}{\PackageError{Ypsilon}{Undefined option to Ypsilon: #1}{}}
		}%
		\newcommand{\PreThreeloop}[1]{
			\togglefalse{isoption}
			\ifnumequal{#1}{1}{\toggletrue{isoption}\raisebox{-0.7ex}{\includegraphics[page = 6, width = 0.48cm*\ratio{\mytextsize}{10.95 pt}]{diagrams_coloured.pdf}\,}}{}
			\ifnumequal{#1}{10}{\toggletrue{isoption}\raisebox{-0.7ex}{\includegraphics[page = 7, width = 0.49cm*\ratio{\mytextsize}{10.95 pt}]{diagrams_coloured.pdf}\,}}{}
			\ifnumequal{#1}{2}{\toggletrue{isoption}\raisebox{-0.7ex}{\includegraphics[page = 8, width = 0.48cm*\ratio{\mytextsize}{10.95 pt}]{diagrams_coloured.pdf}\,}}{}
			\ifnumequal{#1}{20}{\toggletrue{isoption}\raisebox{-0.7ex}{\includegraphics[page = 9, width = 0.49cm*\ratio{\mytextsize}{10.95 pt}]{diagrams_coloured.pdf}\,}}{}
			\iftoggle{isoption}{}{\PackageError{PreThreeloop}{Undefined option to PreThreeloop: #1}{}}
		}%
		\newcommand{\Checkmark}[1]{
			\togglefalse{isoption}
			\ifnumequal{#1}{1}{\toggletrue{isoption}\raisebox{-0.7ex}{\includegraphics[page = 10, width = 0.54cm*\ratio{\mytextsize}{10.95 pt}]{diagrams_coloured.pdf}\,}}{}
			\ifnumequal{#1}{10}{\toggletrue{isoption}\raisebox{-0.7ex}{\includegraphics[page = 11, width = 0.58cm*\ratio{\mytextsize}{10.95 pt}]{diagrams_coloured.pdf}\,}}{}
			\ifnumequal{#1}{2}{\toggletrue{isoption}\raisebox{-0.7ex}{\includegraphics[page = 12, width = 0.54cm*\ratio{\mytextsize}{10.95 pt}]{diagrams_coloured.pdf}\,}}{}
			\ifnumequal{#1}{20}{\toggletrue{isoption}\raisebox{-0.7ex}{\includegraphics[page = 13, width = 0.58cm*\ratio{\mytextsize}{10.95 pt}]{diagrams_coloured.pdf}\,}}{}
			\iftoggle{isoption}{}{\PackageError{Checkmark}{Undefined option to Checkmark: #1}{}}
		}%
		\newcommand{\Triangle}[1]{
			\togglefalse{isoption}
			\ifnumequal{#1}{1}{\toggletrue{isoption}\raisebox{-0.165ex}{\,\includegraphics[page = 14, width = 0.33cm*\ratio{\mytextsize}{10.95 pt}]{diagrams_coloured.pdf}\,}}{}
			\ifnumequal{#1}{2}{\toggletrue{isoption}\raisebox{-0.165ex}{\,\includegraphics[page = 15, width = 0.33cm*\ratio{\mytextsize}{10.95 pt}]{diagrams_coloured.pdf}\,}}{}
			\ifnumequal{#1}{3}{\toggletrue{isoption}\raisebox{-0.3ex}{\,\includegraphics[page = 16, width = 0.33cm*\ratio{\mytextsize}{10.95 pt}]{diagrams_coloured.pdf}\,}}{}
			\iftoggle{isoption}{}{\PackageError{Triangle}{Undefined option to Triangle: #1}{}}
		}%
		\newcommand{\Loop}{\,\includegraphics[page = 17, width = 0.27cm*\ratio{\mytextsize}{10.95 pt}]{diagrams_coloured.pdf}\,}
		\newcommand{\PreCocktail}[1]{
			\togglefalse{isoption}
			\ifnumequal{#1}{1}{\toggletrue{isoption}\raisebox{-0.7ex}{\includegraphics[page = 18, width = 0.51cm*\ratio{\mytextsize}{10.95 pt}]{diagrams_coloured.pdf}\,}}{}
			\ifnumequal{#1}{10}{\toggletrue{isoption}\raisebox{-0.7ex}{\includegraphics[page = 19, width = 0.52cm*\ratio{\mytextsize}{10.95 pt}]{diagrams_coloured.pdf}\,}}{}
			\ifnumequal{#1}{2}{\toggletrue{isoption}\raisebox{-0.7ex}{\includegraphics[page = 20, width = 0.51cm*\ratio{\mytextsize}{10.95 pt}]{diagrams_coloured.pdf}\,}}{}
			\ifnumequal{#1}{20}{\toggletrue{isoption}\raisebox{-0.7ex}{\includegraphics[page = 21, width = 0.52cm*\ratio{\mytextsize}{10.95 pt}]{diagrams_coloured.pdf}\,}}{}
			\ifnumequal{#1}{3}{\toggletrue{isoption}\raisebox{-0.7ex}{\includegraphics[page = 22, width = 0.51cm*\ratio{\mytextsize}{10.95 pt}]{diagrams_coloured.pdf}\,}}{}
			\ifnumequal{#1}{30}{\toggletrue{isoption}\raisebox{-0.7ex}{\includegraphics[page = 23, width = 0.52cm*\ratio{\mytextsize}{10.95 pt}]{diagrams_coloured.pdf}\,}}{}
			\ifnumequal{#1}{4}{\toggletrue{isoption}\raisebox{-0.7ex}{\includegraphics[page = 24, width = 0.51cm*\ratio{\mytextsize}{10.95 pt}]{diagrams_coloured.pdf}\,}}{}
			\ifnumequal{#1}{40}{\toggletrue{isoption}\raisebox{-0.7ex}{\includegraphics[page = 25, width = 0.51cm*\ratio{\mytextsize}{10.95 pt}]{diagrams_coloured.pdf}\,}}{}
			\ifnumequal{#1}{5}{\toggletrue{isoption}\raisebox{-0.7ex}{\includegraphics[page = 26, width = 0.52cm*\ratio{\mytextsize}{10.95 pt}]{diagrams_coloured.pdf}\,}}{}
			\ifnumequal{#1}{50}{\toggletrue{isoption}\raisebox{-0.835ex}{\includegraphics[page = 27, width = 0.52cm*\ratio{\mytextsize}{10.95 pt}]{diagrams_coloured.pdf}\,}}{}
			\iftoggle{isoption}{}{\PackageError{PreCocktail}{Undefined option to PreCocktail: #1}{}}
		}%
		\newcommand{\PreThree}[1]{
			\togglefalse{isoption}
			\ifnumequal{#1}{1}{\toggletrue{isoption}\raisebox{-0.7ex}{\includegraphics[page = 28, width = 0.54cm*\ratio{\mytextsize}{10.95 pt}]{diagrams_coloured.pdf}\,}}{}
			\ifnumequal{#1}{10}{\toggletrue{isoption}\raisebox{-0.7ex}{\includegraphics[page = 29, width = 0.58cm*\ratio{\mytextsize}{10.95 pt}]{diagrams_coloured.pdf}\,}}{}
			\ifnumequal{#1}{2}{\toggletrue{isoption}\raisebox{-0.7ex}{\includegraphics[page = 30, width = 0.54cm*\ratio{\mytextsize}{10.95 pt}]{diagrams_coloured.pdf}\,}}{}
			\ifnumequal{#1}{20}{\toggletrue{isoption}\raisebox{-0.7ex}{\includegraphics[page = 31, width = 0.58cm*\ratio{\mytextsize}{10.95 pt}]{diagrams_coloured.pdf}\,}}{}
			\iftoggle{isoption}{}{\PackageError{PreThree}{Undefined option to PreThree: #1}{}}
		}%
		\newcommand{\ti}{\mathchoice
			{\,\includegraphics[page = 32, width = 0.1cm*\ratio{\mytextsize}{10.95 pt}]{diagrams_coloured.pdf}\,}
			{\,\includegraphics[page = 32, width = 0.1cm*\ratio{\mytextsize}{10.95 pt}]{diagrams_coloured.pdf}\,}
			{\scalebox{0.8}{\,\includegraphics[page = 32, width = 0.1cm*\ratio{\mytextsize}{10.95 pt}]{diagrams_coloured.pdf}\,}}
			{\scalebox{0.64}{\,\includegraphics[page = 32, width = 0.1cm*\ratio{\mytextsize}{10.95 pt}]{diagrams_coloured.pdf}\,}}
		}%
		\newcommand{\ty}{\mathchoice
			{\,\includegraphics[page = 33, width = 0.25cm*\ratio{\mytextsize}{10.95 pt}]{diagrams_coloured.pdf}\,}
			{\,\includegraphics[page = 33, width = 0.25cm*\ratio{\mytextsize}{10.95 pt}]{diagrams_coloured.pdf}\,}
			{\scalebox{0.8}{\,\includegraphics[page = 33, width = 0.25cm*\ratio{\mytextsize}{10.95 pt}]{diagrams_coloured.pdf}\,}}
			{\scalebox{0.64}{\,\includegraphics[page = 33, width = 0.25cm*\ratio{\mytextsize}{10.95 pt}]{diagrams_coloured.pdf}\,}}
		}%
		\newcommand{\tp}{\mathchoice
			{\,\includegraphics[page = 34, width = 0.36cm*\ratio{\mytextsize}{10.95 pt}]{diagrams_coloured.pdf}\,}
			{\,\includegraphics[page = 34, width = 0.36cm*\ratio{\mytextsize}{10.95 pt}]{diagrams_coloured.pdf}\,}
			{\scalebox{0.8}{\,\includegraphics[page = 34, width = 0.36cm*\ratio{\mytextsize}{10.95 pt}]{diagrams_coloured.pdf}\,}}
			{\scalebox{0.64}{\,\includegraphics[page = 34, width = 0.36cm*\ratio{\mytextsize}{10.95 pt}]{diagrams_coloured.pdf}\,}}
		}%
		\newcommand{\tc}{\mathchoice
			{\,\includegraphics[page = 35, width = 0.36cm*\ratio{\mytextsize}{10.95 pt}]{diagrams_coloured.pdf}\,}
			{\,\includegraphics[page = 35, width = 0.36cm*\ratio{\mytextsize}{10.95 pt}]{diagrams_coloured.pdf}\,}
			{\scalebox{0.8}{\,\includegraphics[page = 35, width = 0.36cm*\ratio{\mytextsize}{10.95 pt}]{diagrams_coloured.pdf}\,}}
			{\scalebox{0.64}{\,\includegraphics[page = 35, width = 0.36cm*\ratio{\mytextsize}{10.95 pt}]{diagrams_coloured.pdf}\,}}
		}%
		\newcommand{\tl}{\mathchoice
			{\,\includegraphics[page = 36, width = 0.18cm*\ratio{\mytextsize}{10.95 pt}]{diagrams_coloured.pdf}\,}
			{\,\includegraphics[page = 36, width = 0.18cm*\ratio{\mytextsize}{10.95 pt}]{diagrams_coloured.pdf}\,}
			{\scalebox{0.8}{\,\includegraphics[page = 36, width = 0.18cm*\ratio{\mytextsize}{10.95 pt}]{diagrams_coloured.pdf}\,}}
			{\scalebox{0.64}{\,\includegraphics[page = 36, width = 0.18cm*\ratio{\mytextsize}{10.95 pt}]{diagrams_coloured.pdf}\,}}
		}%
		\newcommand{\SLolli}{
			\raisebox{-0.4ex}{\,\includegraphics[page = 37, width = 0.13cm*\ratio{\mytextsize}{10.95 pt}]{diagrams_coloured.pdf}\,}
		}%
		\newcommand{\SYpsilon}{
			\raisebox{-0.4ex}{\includegraphics[page = 38, width = 0.37cm*\ratio{\mytextsize}{10.95 pt}]{diagrams_coloured.pdf}}
		}%
		\newcommand{\SPreThree}{
			\raisebox{-0.165ex}{\includegraphics[page = 39, width = 0.54cm*\ratio{\mytextsize}{10.95 pt}]{diagrams_coloured.pdf}\,}
		}%
		\newcommand{\SVee}{
			{\,\includegraphics[page = 40, width = 0.33cm*\ratio{\mytextsize}{10.95 pt}]{diagrams_coloured.pdf}\,}
		}%
		\newcommand{\STriangle}{
			{\,\includegraphics[page = 41, width = 0.33cm*\ratio{\mytextsize}{10.95 pt}]{diagrams_coloured.pdf}\,}
		}%
	}{
		\newcommand{\Cherry}[1]{
			\togglefalse{isoption}
			\ifnumequal{#1}{1}{\toggletrue{isoption}\,\includegraphics[page = 1, width = 0.41cm*\ratio{\mytextsize}{10.95 pt}]{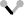}\,}{}
			\ifnumequal{#1}{10}{\toggletrue{isoption}\,\includegraphics[page = 2, width = 0.44cm*\ratio{\mytextsize}{10.95 pt}]{diagrams_grayscale.pdf}\,}{}
			\iftoggle{isoption}{}{\PackageError{Cherry}{Undefined option to Cherry: #1}{}}
		}%
		\newcommand{\Ypsilon}[1]{
			\togglefalse{isoption}
			\ifnumequal{#1}{1}{\toggletrue{isoption}\raisebox{-0.4ex}{\includegraphics[page = 3, width = 0.37cm*\ratio{\mytextsize}{10.95 pt}]{diagrams_grayscale.pdf}}}{}
			\ifnumequal{#1}{10}{\toggletrue{isoption}\raisebox{-0.4ex}{\includegraphics[page = 4, width = 0.4cm*\ratio{\mytextsize}{10.95 pt}]{diagrams_grayscale.pdf}}}{}
			\iftoggle{isoption}{}{\PackageError{Ypsilon}{Undefined option to Ypsilon: #1}{}}
		}%
		\newcommand{\PreThreeloop}[1]{
			\togglefalse{isoption}
			\ifnumequal{#1}{1}{\toggletrue{isoption}\raisebox{-0.7ex}{\includegraphics[page = 6, width = 0.48cm*\ratio{\mytextsize}{10.95 pt}]{diagrams_grayscale.pdf}\,}}{}
			\ifnumequal{#1}{10}{\toggletrue{isoption}\raisebox{-0.7ex}{\includegraphics[page = 7, width = 0.49cm*\ratio{\mytextsize}{10.95 pt}]{diagrams_grayscale.pdf}\,}}{}
			\ifnumequal{#1}{2}{\toggletrue{isoption}\raisebox{-0.7ex}{\includegraphics[page = 8, width = 0.48cm*\ratio{\mytextsize}{10.95 pt}]{diagrams_grayscale.pdf}\,}}{}
			\ifnumequal{#1}{20}{\toggletrue{isoption}\raisebox{-0.7ex}{\includegraphics[page = 9, width = 0.49cm*\ratio{\mytextsize}{10.95 pt}]{diagrams_grayscale.pdf}\,}}{}
			\iftoggle{isoption}{}{\PackageError{PreThreeloop}{Undefined option to PreThreeloop: #1}{}}
		}%
		\newcommand{\Checkmark}[1]{
			\togglefalse{isoption}
			\ifnumequal{#1}{1}{\toggletrue{isoption}\raisebox{-0.7ex}{\includegraphics[page = 10, width = 0.54cm*\ratio{\mytextsize}{10.95 pt}]{diagrams_grayscale.pdf}\,}}{}
			\ifnumequal{#1}{10}{\toggletrue{isoption}\raisebox{-0.7ex}{\includegraphics[page = 11, width = 0.58cm*\ratio{\mytextsize}{10.95 pt}]{diagrams_grayscale.pdf}\,}}{}
			\ifnumequal{#1}{2}{\toggletrue{isoption}\raisebox{-0.7ex}{\includegraphics[page = 12, width = 0.54cm*\ratio{\mytextsize}{10.95 pt}]{diagrams_grayscale.pdf}\,}}{}
			\ifnumequal{#1}{20}{\toggletrue{isoption}\raisebox{-0.7ex}{\includegraphics[page = 13, width = 0.58cm*\ratio{\mytextsize}{10.95 pt}]{diagrams_grayscale.pdf}\,}}{}
			\iftoggle{isoption}{}{\PackageError{Checkmark}{Undefined option to Checkmark: #1}{}}
		}%
		\newcommand{\Triangle}[1]{
			\togglefalse{isoption}
			\ifnumequal{#1}{1}{\toggletrue{isoption}\raisebox{-0.165ex}{\,\includegraphics[page = 14, width = 0.33cm*\ratio{\mytextsize}{10.95 pt}]{diagrams_grayscale.pdf}\,}}{}
			\ifnumequal{#1}{2}{\toggletrue{isoption}\raisebox{-0.165ex}{\,\includegraphics[page = 15, width = 0.33cm*\ratio{\mytextsize}{10.95 pt}]{diagrams_grayscale.pdf}\,}}{}
			\ifnumequal{#1}{3}{\toggletrue{isoption}\raisebox{-0.3ex}{\,\includegraphics[page = 16, width = 0.33cm*\ratio{\mytextsize}{10.95 pt}]{diagrams_grayscale.pdf}\,}}{}
			\iftoggle{isoption}{}{\PackageError{Triangle}{Undefined option to Triangle: #1}{}}
		}%
		\newcommand{\Loop}{\,\includegraphics[page = 17, width = 0.27cm*\ratio{\mytextsize}{10.95 pt}]{diagrams_grayscale.pdf}\,}
		\newcommand{\PreCocktail}[1]{
			\togglefalse{isoption}
			\ifnumequal{#1}{1}{\toggletrue{isoption}\raisebox{-0.7ex}{\includegraphics[page = 18, width = 0.51cm*\ratio{\mytextsize}{10.95 pt}]{diagrams_grayscale.pdf}\,}}{}
			\ifnumequal{#1}{10}{\toggletrue{isoption}\raisebox{-0.7ex}{\includegraphics[page = 19, width = 0.52cm*\ratio{\mytextsize}{10.95 pt}]{diagrams_grayscale.pdf}\,}}{}
			\ifnumequal{#1}{2}{\toggletrue{isoption}\raisebox{-0.7ex}{\includegraphics[page = 20, width = 0.51cm*\ratio{\mytextsize}{10.95 pt}]{diagrams_grayscale.pdf}\,}}{}
			\ifnumequal{#1}{20}{\toggletrue{isoption}\raisebox{-0.7ex}{\includegraphics[page = 21, width = 0.52cm*\ratio{\mytextsize}{10.95 pt}]{diagrams_grayscale.pdf}\,}}{}
			\ifnumequal{#1}{3}{\toggletrue{isoption}\raisebox{-0.7ex}{\includegraphics[page = 22, width = 0.51cm*\ratio{\mytextsize}{10.95 pt}]{diagrams_grayscale.pdf}\,}}{}
			\ifnumequal{#1}{30}{\toggletrue{isoption}\raisebox{-0.7ex}{\includegraphics[page = 23, width = 0.52cm*\ratio{\mytextsize}{10.95 pt}]{diagrams_grayscale.pdf}\,}}{}
			\ifnumequal{#1}{4}{\toggletrue{isoption}\raisebox{-0.7ex}{\includegraphics[page = 24, width = 0.51cm*\ratio{\mytextsize}{10.95 pt}]{diagrams_grayscale.pdf}\,}}{}
			\ifnumequal{#1}{40}{\toggletrue{isoption}\raisebox{-0.7ex}{\includegraphics[page = 25, width = 0.51cm*\ratio{\mytextsize}{10.95 pt}]{diagrams_grayscale.pdf}\,}}{}
			\ifnumequal{#1}{5}{\toggletrue{isoption}\raisebox{-0.7ex}{\includegraphics[page = 26, width = 0.52cm*\ratio{\mytextsize}{10.95 pt}]{diagrams_grayscale.pdf}\,}}{}
			\ifnumequal{#1}{50}{\toggletrue{isoption}\raisebox{-0.835ex}{\includegraphics[page = 27, width = 0.52cm*\ratio{\mytextsize}{10.95 pt}]{diagrams_grayscale.pdf}\,}}{}
			\iftoggle{isoption}{}{\PackageError{PreCocktail}{Undefined option to PreCocktail: #1}{}}
		}%
		\newcommand{\PreThree}[1]{
			\togglefalse{isoption}
			\ifnumequal{#1}{1}{\toggletrue{isoption}\raisebox{-0.7ex}{\includegraphics[page = 28, width = 0.54cm*\ratio{\mytextsize}{10.95 pt}]{diagrams_grayscale.pdf}\,}}{}
			\ifnumequal{#1}{10}{\toggletrue{isoption}\raisebox{-0.7ex}{\includegraphics[page = 29, width = 0.58cm*\ratio{\mytextsize}{10.95 pt}]{diagrams_grayscale.pdf}\,}}{}
			\ifnumequal{#1}{2}{\toggletrue{isoption}\raisebox{-0.7ex}{\includegraphics[page = 30, width = 0.54cm*\ratio{\mytextsize}{10.95 pt}]{diagrams_grayscale.pdf}\,}}{}
			\ifnumequal{#1}{20}{\toggletrue{isoption}\raisebox{-0.7ex}{\includegraphics[page = 31, width = 0.58cm*\ratio{\mytextsize}{10.95 pt}]{diagrams_grayscale.pdf}\,}}{}
			\iftoggle{isoption}{}{\PackageError{PreThree}{Undefined option to PreThree: #1}{}}
		}%
		\newcommand{\ti}{\mathchoice
			{\,\includegraphics[page = 32, width = 0.1cm*\ratio{\mytextsize}{10.95 pt}]{diagrams_grayscale.pdf}\,}
			{\,\includegraphics[page = 32, width = 0.1cm*\ratio{\mytextsize}{10.95 pt}]{diagrams_grayscale.pdf}\,}
			{\scalebox{0.8}{\,\includegraphics[page = 32, width = 0.1cm*\ratio{\mytextsize}{10.95 pt}]{diagrams_grayscale.pdf}\,}}
			{\scalebox{0.64}{\,\includegraphics[page = 32, width = 0.1cm*\ratio{\mytextsize}{10.95 pt}]{diagrams_grayscale.pdf}\,}}
		}%
		\newcommand{\ty}{\mathchoice
			{\,\includegraphics[page = 33, width = 0.25cm*\ratio{\mytextsize}{10.95 pt}]{diagrams_grayscale.pdf}\,}
			{\,\includegraphics[page = 33, width = 0.25cm*\ratio{\mytextsize}{10.95 pt}]{diagrams_grayscale.pdf}\,}
			{\scalebox{0.8}{\,\includegraphics[page = 33, width = 0.25cm*\ratio{\mytextsize}{10.95 pt}]{diagrams_grayscale.pdf}\,}}
			{\scalebox{0.64}{\,\includegraphics[page = 33, width = 0.25cm*\ratio{\mytextsize}{10.95 pt}]{diagrams_grayscale.pdf}\,}}
		}%
		\newcommand{\tp}{\mathchoice
			{\,\includegraphics[page = 34, width = 0.36cm*\ratio{\mytextsize}{10.95 pt}]{diagrams_grayscale.pdf}\,}
			{\,\includegraphics[page = 34, width = 0.36cm*\ratio{\mytextsize}{10.95 pt}]{diagrams_grayscale.pdf}\,}
			{\scalebox{0.8}{\,\includegraphics[page = 34, width = 0.36cm*\ratio{\mytextsize}{10.95 pt}]{diagrams_grayscale.pdf}\,}}
			{\scalebox{0.64}{\,\includegraphics[page = 34, width = 0.36cm*\ratio{\mytextsize}{10.95 pt}]{diagrams_grayscale.pdf}\,}}
		}%
		\newcommand{\tc}{\mathchoice
			{\,\includegraphics[page = 35, width = 0.36cm*\ratio{\mytextsize}{10.95 pt}]{diagrams_grayscale.pdf}\,}
			{\,\includegraphics[page = 35, width = 0.36cm*\ratio{\mytextsize}{10.95 pt}]{diagrams_grayscale.pdf}\,}
			{\scalebox{0.8}{\,\includegraphics[page = 35, width = 0.36cm*\ratio{\mytextsize}{10.95 pt}]{diagrams_grayscale.pdf}\,}}
			{\scalebox{0.64}{\,\includegraphics[page = 35, width = 0.36cm*\ratio{\mytextsize}{10.95 pt}]{diagrams_grayscale.pdf}\,}}
		}%
		\newcommand{\tl}{\mathchoice
			{\,\includegraphics[page = 36, width = 0.18cm*\ratio{\mytextsize}{10.95 pt}]{diagrams_grayscale.pdf}\,}
			{\,\includegraphics[page = 36, width = 0.18cm*\ratio{\mytextsize}{10.95 pt}]{diagrams_grayscale.pdf}\,}
			{\scalebox{0.8}{\,\includegraphics[page = 36, width = 0.18cm*\ratio{\mytextsize}{10.95 pt}]{diagrams_grayscale.pdf}\,}}
			{\scalebox{0.64}{\,\includegraphics[page = 36, width = 0.18cm*\ratio{\mytextsize}{10.95 pt}]{diagrams_grayscale.pdf}\,}}
		}%
		\newcommand{\SLolli}{
			\raisebox{-0.4ex}{\,\includegraphics[page = 37, width = 0.13cm*\ratio{\mytextsize}{10.95 pt}]{diagrams_grayscale.pdf}\,}
		}%
		\newcommand{\SYpsilon}{
			\raisebox{-0.4ex}{\includegraphics[page = 38, width = 0.37cm*\ratio{\mytextsize}{10.95 pt}]{diagrams_grayscale.pdf}}
		}%
		\newcommand{\SPreThree}{
			\raisebox{-0.165ex}{\includegraphics[page = 39, width = 0.54cm*\ratio{\mytextsize}{10.95 pt}]{diagrams_grayscale.pdf}\,}
		}%
		\newcommand{\SVee}{
			{\,\includegraphics[page = 40, width = 0.33cm*\ratio{\mytextsize}{10.95 pt}]{diagrams_grayscale.pdf}\,}
		}%
		\newcommand{\STriangle}{
			{\,\includegraphics[page = 41, width = 0.33cm*\ratio{\mytextsize}{10.95 pt}]{diagrams_grayscale.pdf}\,}
		}%
	}%
}{
	\newcommand{\Cherry}[1]{\normalfont\textbf{Chr(#1)}}
	\newcommand{\Ypsilon}[1]{\normalfont\textbf{Y(#1)}}
	\newcommand{\PreThreeloop}[1]{\normalfont\textbf{PTl(#1)}}
	\newcommand{\Checkmark}[1]{\normalfont\textbf{Ch(#1)}}
	\newcommand{\Triangle}[1]{\normalfont\textbf{Tr(#1)}}
	\newcommand{\Loop}{\normalfont\textbf{Lo}}
	\newcommand{\PreCocktail}[1]{\normalfont\textbf{PC(#1)}}
	\newcommand{\PreThree}[1]{\normalfont\textbf{PT(#1)}}
	\newcommand{\ti}{\normalfont\textbf{ti}}
	\newcommand{\ty}{\normalfont\textbf{ty}}
	\newcommand{\tp}{\normalfont\textbf{tp}}
	\newcommand{\tc}{\normalfont\textbf{tc}}
	\newcommand{\tl}{\normalfont\textbf{tl}}
	\newcommand{\SLolli}{\normalfont\textbf{L}}
	\newcommand{\SYpsilon}{\normalfont\textbf{Y}}
	\newcommand{\SPreThree}{\normalfont\textbf{PT}}
	\newcommand{\SVee}{\normalfont\textbf{V}}
	\newcommand{\STriangle}{\normalfont\textbf{Tr}}
}%

\usetikzlibrary{shapes}
\usetikzlibrary{decorations.markings}

\tikzset{
	root/.style={circle, fill=rootcolor, inner sep=0ex, minimum size=1.1ex},
	dot/.style={circle, fill=black, inner sep=0ex, minimum size=0.55ex},
	rootvar/.style={circle,fill=rootcolor!10, draw=rootcolor, inner sep=0ex, minimum size=1.1ex},
	scalarkernel/.style={black, line width=\RSwidth, shorten <=0.2ex, decoration={markings,mark=at position 0.96 with {\arrow[scale=1.1,vectorcolor]{>}}}, postaction={decorate},shorten >=0.4ex},
	vectorkernel/.style={vectorcolor, line width=\RSwidth, shorten <=0.2ex, decoration={markings,mark=at position 0.96 with {\arrow[scale=1.1,vectorcolor]{>}}}, postaction={decorate},shorten >=0.4ex},
	vectorcircle/.style={circle, fill=none, draw=vectorcolor, inner sep=0ex, minimum size=0.66ex, line width=0.7\RSwidth, scale=1.15},
}%
\begin{document}

\maketitle

\begin{abstract}
	\noindent
	Using the method of paracontrolled distributions, we show the local well-posedness of an additive-noise approximation to the fluctuating hydrodynamics of the Keller--Segel model on the two-dimensional torus. Our approximation is a non-linear, non-local, parabolic-elliptic stochastic PDE with an irregular, heterogeneous space-time noise. As a consequence of the irregularity and heterogeneity, solutions to this equation must be renormalised by a sequence of diverging fields. Using the symmetry of the elliptic Green's function, which appears in our non-local term, we establish that the renormalisation diverges at most logarithmically, an improvement over the linear divergence one would expect by power counting. Similar cancellations also serve to reduce the number of diverging counterterms.
\end{abstract}
{%
}%
\microtypesetup{protrusion=false} 

\tableofcontents

\microtypesetup{protrusion=true}
\section{Introduction}\label{sec:introduction}
In this work we are concerned with the local well-posedness of singular stochastic partial differential equations (singular SPDEs) of the kind
\begin{equation}\label{eq:gen_rKS_intro}
	\begin{cases}
		\begin{aligned}
			(\partial_t-\Delta)\rho&=\vdiv(\rho\nabla\Phi_{\rho})+\vdiv(\het\boldsymbol{\xi}), && \text{ in } (0,T]\times\mbT^2,\\
			-\Delta\Phi_{\rho}&=\rho- \inner{\rho}{1}_{L^2(\mbT^{2})},&& \text{ in }(0,T]\times \mbT^2,\\
			\rho\tzero&=\rho_{0}, && \text{ on }\mbT^2,
		\end{aligned}
	\end{cases}
\end{equation}
where $\boldsymbol{\xi}=(\xi^1,\xi^2)$ is a two-dimensional vector of i.i.d.\ space-time white noises, $\mbT^2=\mbR^2/\mbZ^2$ is the two-dimensional torus, $T\in(0,\infty)$ is a time horizon, $\het\in C([0,T];\mcH^2(\mbT^2))$ is a space-time heterogeneity and $\rho_0$ is some suitable initial data which we specify later. The advection in~\eqref{eq:gen_rKS_intro} comes from the Keller--Segel model of chemotaxis~\cite{keller_segel_70} and the noise stems from the theory of fluctuating hydrodynamics where one would ideally set $\het = \sqrt{\rho}$ to obtain the Dean--Kawasaki noise~\cite{dean_96,kawasaki_94}. However, it was shown in~\cite{konarovskyi_lehmann_vRenesse_19,konarovskyi_lehmann_vRenesse_20} that the equation with smooth drift only admits solutions which are empirical measures of the underlying interacting particle system. Hence one does not expect~\eqref{eq:gen_rKS_intro} with $\het = \sqrt{\rho}$ to admit non-atomic solutions.

Motivated by the theory of linear fluctuating hydrodynamics, our main example is instead given by the choice $\het = \srdet$ where $\rdet$ solves the deterministic PDE
\begin{equation}\label{eq:det_KS_intro}
	\begin{cases}
		\begin{aligned}
			(\partial_t-\Delta)\rdet&=\vdiv(\rdet \nabla\Phi_{\rdet}), && \text{ in } (0,T]\times\mbT^2,\\
			-\Delta\Phi_{\rdet}&=\rdet-\inner{\rdet}{1}_{L^2(\mbT^{2})}, && \text{ in }(0,T]\times \mbT^2,\\
			\rdet\tzero&=\rho_{0}, && \text{ on }\mbT^2.
		\end{aligned}
	\end{cases}
\end{equation}
This choice will be applied in a follow-up paper, \cite{martini_mayorcas_22_LDP}, to an additive-noise approximation of the Dean--Kawasaki equation associated to the Keller--Segel model. More precisely, given a \emph{noise intensity} $\eps>0$ and a \emph{correlation length} $\delta=\delta(\eps)>0$, we define the additive-noise approximation $\rho^{(\eps)}_{\delta}$ as a solution to
\begin{equation}\label{eq:gen_rKS_canonical}
	\begin{cases}
		\begin{aligned}
			(\partial_t-\Delta)\rho&=\vdiv(\rho\nabla\Phi_{\rho})+\eps^{1/2}\vdiv(\sqrt{\rdet}\boldsymbol{\xi}^{\delta}), && \text{ in } (0,T]\times\mbT^2,\\
			-\Delta\Phi_{\rho}&=\rho- \inner{\rho}{1}_{L^2(\mbT^{2})},&& \text{ in }(0,T]\times \mbT^2,\\
			\rho\tzero&=\rho_{0}, && \text{ on }\mbT^2,
		\end{aligned}
	\end{cases}
\end{equation}
where $\boldsymbol{\xi}^\delta \defeq(\psi_{\delta}\ast \xi^1,\psi_{\delta}\ast \xi^2)$ denotes a mollified noise and $\psi_{\delta}$ a standard, symmetric mollifier. Under appropriate scaling assumptions on $(\eps,\delta(\eps))$, using the theory developed in this paper, we will show in~\cite{martini_mayorcas_22_LDP} that $(\rho^{(\eps)}_{\delta(\eps)})_{\eps>0}$ satisfies a law of large numbers with limit $\rdet$, a central limit theorem with leading-order fluctuations given by the generalized Ornstein--Uhlenbeck process and a large deviation principle. This relates our additive-noise approximation~\eqref{eq:gen_rKS_canonical} to stochastic particle approximations to the Keller--Segel model~\cite{fournier_jourdain_17,fournier_tardy_21}, which were shown to have the same law of large numbers~\cite{fournier_jourdain_17,tardy_23,bresch_jabin_wang_19,bresch_jabin_wang_20,bresch_jabin_wang_23}---and are expected to have the same central limit theorem\footnote{%
	See~\cite{wang_zhao_zhu_23} for a related model.
}%
---as our approximation~\eqref{eq:gen_rKS_canonical}.
%
%
\begin{remark}
	We may also consider equations of the form
	\begin{equation}\label{eq:det_KS_intro_chi}
		\begin{cases}
			\begin{aligned}
				(\partial_t-\Delta)\rdet&=-\chem\vdiv(\rdet \nabla\Phi_{\rdet}), && \text{ in } (0,T]\times\mbT^2,\\
				-\Delta\Phi_{\rdet}&=\rdet-\inner{\rdet}{1}_{L^2(\mbT^{2})}, && \text{ in }(0,T]\times \mbT^2,\\
				\rdet\tzero&=\rho_{0}, && \text{ on }\mbT^2,
			\end{aligned}
		\end{cases}
	\end{equation}
	where $\chem \in \mbR$. In this setting, when one restricts $\rho_0$ to be non-negative and to integrate to $1$ (i.e.\ the density of a probability measure) one recovers the usual parabolic-elliptic Keller--Segel equation,~\cite{keller_segel_70}, the analysis of which has received much attention~\cite{horstmann_03_KS_I,horstmann_04_KS_II,hillen_painter_09_users,painter_19_chemotaxis}. The global existence of~\eqref{eq:det_KS_intro_chi} in spatial dimension two depends on the size and sign of $\chem$, \cite{jager_luckhaus_92_explosions,corrias_perthame_zaag_04,blanchet_dolbeault_perthame_06}. Since we are only concerned with local existence and all of our analysis is agnostic as to the size and sign of $\chem$ we set it to be $-1$ and work with equations of the form~\eqref{eq:gen_rKS_intro} and~\eqref{eq:det_KS_intro}.
\end{remark}
In this paper, rather than the particular~\eqref{eq:gen_rKS_canonical}, we will treat the general~\eqref{eq:gen_rKS_intro} where $\het$ is an arbitrary function, continuous in time and $\mcH^{2}(\mbT^{2})$ in space. Due to the singularity of the noise, defining a suitable notion of solution to~\eqref{eq:gen_rKS_intro} is non-trivial and we will implement a paracontrolled approach~\cite{gubinelli_15_GIP} to obtain local well-posedness (Theorem~\ref{thm:main_theorem}). To see why such an approach is necessary we consider the terms of~\eqref{eq:gen_rKS_intro} under a formal power counting argument. The proper definition of all function spaces used below can be found in Appendix~\ref{app:Besov}.

The white noise almost surely takes values in $\mcC_{\textnormal{par}}^{-2-}([0,T]\times\mbT^{2};\mbR^{2})$ for any $T>0$.\footnote{%
	Here $\mcC^\alpha_{\textnormal{par}}$ denotes the set of space-time H\"{o}lder-regular distributions of parameter $\alpha \in \mbR$ equipped with the parabolic scaling, i.e.\ regularity in time counts twice as much as regularity in space. These spaces are not used beyond the introduction and so we refer to \cite[Lem.~2.12 \& Def.~3.7]{hairer_14_RegStruct} for an example definition. The shorthand $\alpha\pm$ is used to denote $\alpha\pm\eps$ for some $\eps>0$ arbitrarily small but fixed.
}
Let us assume for now that we can define the product $\het \boldsymbol{\xi}$ intrinsically and that it is no more regular than the white noise itself. Due to the regularising effect of the heat equation (Lemma~\ref{lem:Schauder}), it follows that the solution $\ti$ to the linear equation,
\begin{equation*}
	\begin{cases}
		\begin{aligned}
			(\partial_t -\Delta)\ti &= \vdiv(\het \boldsymbol{\xi}), &&\text{ in }(0,T]\times \mbT^2,\\
			\ti\tzero &= 0, && \text{ on }\mbT^2,
		\end{aligned}
	\end{cases}
\end{equation*}
may be no more regular than $C_T\mcC^{-1-}(\mbT^{2})$, where $\mcC^{\alpha}(\mbT^{2})$ denotes the H\"{o}lder--Besov space of regularity $\alpha\in\mbR$ and $C_{T}\mcC^{\alpha}(\mbT^{2})$ denotes the space of continuous functions on $[0,T]$ taking values in $\mcC^{\alpha}(\mbT^{2})$. Certainly we cannot, in general, expect $\rho$ to be any more regular than this. Assuming that this regularity is passed to $\rho$ and applying the regularising effect of the elliptic equation (Lemma~\ref{lem:elliptic_regularity}) we would have $\nabla\Phi_{\rho}\in C_{T}\mcC^{0-}(\mbT^{2};\mbR^{2})$. However, by Bony's estimate the product $f g$ is only a priori well-defined for $f\in\mcC^{\alpha}(\mbT^{2})$ and $g\in \mcC^{\beta}(\mbT^{2})$ with $\alpha+\beta>0$, which does not hold for any coordinate of the vector $\rho\nabla\Phi_{\rho}$.

The theories of regularity structures, paracontrolled distributions, flow equations and various recent extensions and adaptations thereof have revolutionised the study of singular SPDEs~\cite{hairer_14_RegStruct,gubinelli_15_GIP, kupiainen_16, otto_sauer_smith_weber_21,linares_otto_tempelmayr_tsatsoulis_24,duch_21}. The common thread throughout these theories is to notice that the factors of the ill-defined products are not generic distributions but inherit structure from the noise. This inheritance allows one to define renormalised products, which excise the singular part, allowing one to give meaning to a renormalised equation which is continuous in a finite tuple of \emph{diagrams} built from the noise. These diagrams are referred to as an enhancement.

The theory of paracontrolled distributions was first developed by M.~Gubinelli, P.~Imkeller and N.~Perkowski in~\cite{gubinelli_15_GIP}. The central idea is to use harmonic analysis to construct regular commutators, which allow us to decompose the equation into exogenous noise terms and terms that can be constructed as fixed points. Paracontrolled distributions have been successfully applied to analyse a range of singular SPDEs and operators including; the parabolic Anderson model (PAM)~\cite{gubinelli_15_GIP,koenig_perkowski_vanzuijlen_20}, the Anderson Hamiltonian~\cite{allez_chouk_15,gubinelli_ugurcan_zachhuber_20,chouk_van_zuijlen_21}, the $\Phi^{4}_{3}$ model~\cite{mourrat_weber_17_CDI,catellier_chouk_18}, the Kardar--Parisi--Zhang equation~\cite{gubinelli_perkowski_17}, the stochastic Burgers and Navier--Stokes equations~\cite{zhu_zhu_15,gubinelli_perkowski_17} and the stochastic non-linear wave equation~\cite{gubinelli_koch_oh_18}.

Let $(\psi_{\delta})_{\delta>0}$ be a sequence of symmetric spatial mollifiers. In our case we find that there exists a family of deterministic fields $(f^{\delta})_{\delta>0}$ satisfying
\begin{equation*}
	f^{\delta}\in L_{T}^{\infty}\mcC^{-1-}(\mbT^{2};\mbR^{2}),\qquad\norm{f^{\delta}}_{L_{T}^{\infty}\mcC^{-1-}}\lesssim(1\vee\log(\delta^{-1}))\norm{\het}^{2}_{C_{T}\mcH^{2}}\quad\text{for every}~\delta>0
\end{equation*}
such that the sequence of solutions $(\rho^\delta)_{\delta>0}$ each solving,
\begin{equation}\label{eq:gen_rKS_renormalised}
	\begin{cases}
		\begin{aligned}
			(\partial_t-\Delta)\rho&=\vdiv(\rho\nabla\Phi_{\rho}-f^\delta)+\vdiv(\het\boldsymbol{\xi}^{\delta}), && \text{ in } (0,T]\times\mbT^2,\\
			-\Delta\Phi_{\rho}&=\rho-\inner{\rho}{1}_{L^2(\mbT^{2})}, && \text{ in }(0,T]\times \mbT^2,\\
			\rho\tzero&=\rho_{0}, && \text{ on }\mbT^2,
		\end{aligned}
	\end{cases}
\end{equation}
converges in probability to a unique limit $\rho$, which we identify as the renormalised solution to~\eqref{eq:gen_rKS_intro}. For a definition of the renormalised solution, see Definition~\ref{def:renormalised_solution}; for its existence, see Theorem~\ref{thm:convergence_renorm_sol}, Part~\ref{it:thm_existence_renormalized}; and for the convergence, see Theorem~\ref{thm:convergence_renorm_sol}, Part~\ref{it:thm_convergence_to_renormalized}; see also Subsection~\ref{subsec:gen_strategy} for a formal discussion of these results.

We remark that the eventual application to fluctuating hydrodynamics motivates us to consider mollified noise terms of the form $\het\boldsymbol{\xi}^{\delta}\defeq\het(\psi_{\delta}\ast\boldsymbol{\xi})$, rather than mollifying the whole product $(\het\boldsymbol{\xi})^{\delta}\defeq\psi_{\delta}\ast(\het\boldsymbol{\xi})$, see~\cite[Sec.~3.2]{cornalba_fischer_ingmanns_raithel_23} and~\cite{fehrman_gess_23_zero_range,dirr_fehrman_gess_20}. Further, we mollify the noise in space only so that the Markov property of the underlying dynamics is retained.

Three points of interest arise from~\eqref{eq:gen_rKS_renormalised}. Firstly, in the case where $\het$ is genuinely heterogeneous the field $f^\delta$ is in general also heterogeneous. This has been observed elsewhere, having been pointed out as a possibility in~\cite{hairer_14_RegStruct} and seen explicitly in the renormalisation of singular SPDEs on bounded domains,~\cite{gerencser_hairer_19}. Secondly, if $\het$ is a constant, so that our noise agrees with that of the stochastic Burgers' equation, then $f^{\delta}$ is zero. In this case the renormalised equation agrees exactly with the singular equation, i.e.\ the products are not explicitly renormalised when $\delta=0$. This phenomenon has also been observed in~\cite{daprato_debussche_temam_94,daprato_debussche_02,zhu_zhu_15,gubinelli_perkowski_17}. Thirdly, using the informal power counting described above, one might expect the singular product $\rho^\delta\nabla \Phi_{\rho^\delta}$ to diverge at the order of $\delta^{-1}$, since this is the gap in regularity between the singular factors. However,~\eqref{eq:gen_rKS_renormalised} shows that this divergence is at most logarithmic. This improvement arises from symmetries in the fundamental solution of the elliptic problem, leading to non-trivial cancellations in our stochastic estimates. In the case of constant $\het$ it is exactly these cancellations which show that no explicit renormalisation in~\eqref{eq:gen_rKS_renormalised} is necessary.

To demonstrate the underlying principle, let us consider a one-dimensional example. We assume that $u^{\delta}\to u$ as $\delta \to 0$ in a space of regularity $-1/2-$. The product rule gives the identity,
\begin{equation}\label{eq:product_rule}
	u^{\delta}\partial_{x}\partial_{x}^{-2}u^{\delta}=\frac{1}{2}\Bigl(\partial_{x}^{2}(\partial_{x}^{-1}u^{\delta}\partial_{x}^{-2}u^{\delta})-\partial_{x}(u^{\delta}\partial_{x}^{-2}u^{\delta})\Bigr),
\end{equation}
where we write $\partial_x^{-1}$ as a shorthand for integration in $x$, with the (arbitrary) normalisation that the primitive is mean-free. While the product on the left-hand side, between an object converging in $-1/2-$ and an object converging in $1/2-$ looks ill-posed, the right-hand side is in fact classically well-posed; the first term is the second derivative of a product between objects in $1/2-$ and $3/2-$, while the second is the derivative of a product between an object in $-1/2-$ and one in $3/2-$. Hence the anticipated logarithmic divergence of the left hand side is removed by expanding as on the right-hand side. This basic observation extends to our higher-dimensional case through the symmetry of the Green's function for Poisson's equation. We see that the symmetry alleviates divergences by one order. Linear divergences of $\delta^{-1}$ are improved to logarithmic, and logarithmic divergences are improved to well-posedness. The heterogeneity $\het$ makes these improvements visible, as when $\het$ is constant the same symmetries lead to perfect cancellations removing the need for renormalising counterterms all together. Similar observations have also been made in the context of the KPZ equation,~\cite[Lem.~9.5]{gubinelli_perkowski_17}.

We observe that~\eqref{eq:gen_rKS_intro} is an example of a singular SPDE involving anisotropic regularity since the regularising effect of the elliptic equation only takes place in the spatial variable. It is for this reason that we choose to work with the theory of paracontrolled distributions, since it naturally treats space and time separately. This is in contrast to the theory of regularity structures which naturally treats space and time simultaneously. While there are examples of works which study anisotropic equations and associated regularity structures, \cite{berglund_kuehn_16,hairer_matetski_18,berglund_kuehn_19}, the analysis is somewhat ad hoc and requires a significant amount of bespoke machinery to be developed. It is for this reason that we opt for a paracontrolled approach which more natively applies to the anisotropic regularisation present in~\eqref{eq:gen_rKS_intro}.
\paragraph{Structure of the Paper:}
In the rest of this section we first recall some basic notations and conventions which are used throughout the text. Some of these are accompanied by more rigorous presentations in the appendices. We then present an outline of the general strategy and our main result in Subsection~\ref{subsec:gen_strategy}. Section~\ref{sec:enhancements} contains a detailed proof of the existence and regularity of the various stochastic objects which we are required to construct and constitute our enhanced noise. The careful analysis of these stochastic objects and control over the diverging fields is the main contribution of this paper. In Section~\ref{sec:existence} we show the local well-posedness of paracontrolled solutions given a suitable enhancement of the noise. Finally we include three appendices: Appendix~\ref{app:Besov} recalls some useful and well-known results concerning Besov spaces and paraproducts; Appendix~\ref{sec:shape_coefficient_estimates} provides various estimates on the so-called shape coefficients which we introduce in Section~\ref{sec:enhancements}
and Appendix~\ref{app:summation_estimates} contains a number of summation and discrete convolution estimates that we make repeated use of throughout.
\paragraph{Acknowledgements:} 
We would like to express our gratitude to A.~Etheridge, B.~Fehrman, N.~Perkowski and W.~van Zuijlen for helpful discussions during the writing of this manuscript. We also wish to thank S.~Mahdisoltani for bringing the subject of fluctuating hydrodynamics to our attention, in particular through the paper~\cite{mahdisoltani_et_al_21}.

A.~Martini was supported by the Engineering and Physical Sciences Research Council Doctoral Training Partnerships [grant number EP/R513295/1] and by the Lamb \& Flag Scholarship of St John's College, Oxford. Part of this work was completed during A.~Martini's participation in the Junior Trimester Program `Stochastic modelling in the life science: From evolution to medicine' at the Hausdorff Research Institute for Mathematics, funded by the Deutsche Forschungsgemeinschaft (DFG, German Research Foundation) under Germany's Excellence Strategy -- EXC-2047/1 -- 390685813.

Work on this paper was undertaken during A.~Mayorcas's tenure as INI-Simons Post Doctoral Research Fellow hosted by the Isaac Newton Institute for Mathematical Sciences (INI) participating in programme \emph{Frontiers in Kinetic Theory}, and by the Department of Pure Mathematics and Mathematical Statistics (DPMMS) at the University of Cambridge.  This author would like to thank INI and DPMMS for support and hospitality during this fellowship, which was supported by Simons Foundation (award ID 316017) and by Engineering and Physical Sciences Research Council (EPSRC) grant number EP/R014604/1.
\subsection{Notations and Conventions}\label{sec:notations}
We write $\mbN$ for the natural numbers excluding zero, $\mbN_0 \defeq \mbN\cup\{0\}$ and $\mbN_{-1} \defeq \mbN_0\cup\{-1\}$. We define the two-dimensional torus by $\mbT^{2}\defeq\mbR^{2}/\mbZ^{2}$. Throughout $\lvert\,\cdot\,\rvert$ will indicate the norm $\lvert x\rvert=\big(\sum_{i=1}^2 \lvert x_i\rvert^2\big)^{1/2}$. Occasionally we write $\lvert x\rvert_{\infty}\defeq \max_{i =1,2} \lvert x_i \rvert$ to indicate the maximum norm on $\mbT^2$ or $\mbR^2$. For $r>0$ we use the notation $B(0,r)\defeq \{x\in \mbR^2:\lvert x\rvert<r\}$. From now on we will write $\inner{a}{b}$ to denote the inner product on any Hilbert space which we either specify or leave clear from the context. For $k,n\in\mbN$ we denote by $C^k(\mbT^2;\mbR^n)$ (resp.\ $C^\infty(\mbT^2;\mbR^n)$) the space of $k$-times continuously differentiable (resp.\ smooth) functions on the torus with values in $\mbR^{n}$. We write $\mcS'(\mbT^{2};\mbR)$ for the dual of $C^{\infty}(\mbT^{2};\mbR)$ and $\mcS'(\mbT^{2};\mbR^{n})$ for vectors of elements in $\mcS'(\mbT^{2};\mbR)$. When the context is clear we will remove the target space so as to lighten notation.

For $f\in C^\infty(\mbT^2;\mbR)$ (resp.\ complex sequences $(\zeta(\om))_{\om\in\mbZ^{2}}$ with $\overline{\zeta(-\om)}=\zeta(\om)$ that decay faster than any polynomial) we define its Fourier transform (resp.\ inverse Fourier transform) by the expression,
\begin{equation*}
	\msF f(\omega) \defeq \int_{\mbT^2}\euler^{-2\uppi\upi\inner{\om}{x}}f(x)\dd x,\qquad \msF^{-1} \zeta(x) \defeq \sum_{\om\in\mbZ^2}\euler^{2\uppi\upi\inner{\om}{x}}\zeta(\om).
\end{equation*} 
This is extended componentwise to vector-valued functions, by density to $f\in L^p(\mbT^2;\mbR^n)$ for $p\in [1,\infty)$ and by duality to $f\in \mcS'(\mbT^2;\mbR^n)$. Where convenient we use the shorthand $\hat{f}(\omega)\defeq\msF f (\omega)$. We define the Sobolev space $\mcH^{\alpha}(\mbT^{2};\mbR^{n})$ of regularity $\alpha\in\mbR$ as the space of distributions $u\in\mcS'(\mbT^{2};\mbR^{n})$ such that
\begin{equation*}
	\norm{u}_{\mcH^{\alpha}}\defeq\Bigl(\sum_{\om\in\mbZ^{2}}(1+\abs{2\uppi\om}^{2})^{\alpha}\abs{\hat{u}(\om)}^{2}\Bigr)^{1/2}<\infty.
\end{equation*}
These notions also extend to complex-valued distributions, which we denote by $\mcS'(\mbT^{2};\mbC^{n})$ and $\mcH^{\alpha}(\mbT^{2};\mbC^{n})$.

We often work in the scale of Besov and H\"{o}lder--Besov spaces whose definitions and basic properties are recalled in Appendix~\ref{app:Besov}. We let:
\begin{equation}\label{eq:dyadic_partition_of_unity_intro}
	\begin{split}
		&\varrho_{-1},\varrho_{0}\in C^{\infty}(\mbR^{2};[0,1])~\text{be radially symmetric and such that}\\
		&\supp(\varrho_{-1})\subset B(0,1/2),\qquad\supp(\varrho_{0})\subset\{x\in\mbR^{2}:9/32\leq\abs{x}\leq 1\},\\
		&\sum_{k=-1}^{\infty}\varrho_{k}(x)=1\quad\text{for all}~x\in\mbR^{2}\quad\text{where}~\varrho_{k}(x)\defeq\varrho_{0}(2^{-k}x)\quad\text{for each}~k\in\mbN.
	\end{split}
\end{equation}
This defines a dyadic partition of unity as in Subsection~\ref{subsec:Besov_spaces}. Given $k\geq -1$ we write $\Delta_{k}u\defeq\msF^{-1}(\varrho_{k}\msF u)$ for the associated Littlewood--Paley block and given $\alpha\in \mbR$, $p,q\in [1,\infty]$, we define the Besov-norm $\norm{u}_{\mcB^{\alpha}_{p,q}(\mbT^2;\mbR^{n})}\defeq\norm{(2^{k\alpha}\norm{\Delta_{k}u}_{L^{p}(\mbT^2;\mbR^{n})})_{k\in\mbN_{-1}}}_{\ell^{q}}$ for all $u\in\mcS'(\mbT^{2};\mbR^{n})$. We use $\mcB^\alpha_{p,q}(\mbT^2;\mbR^n)$ to denote the completion of $C^{\infty}(\mbT^{2};\mbR^{n})$ under $\norm{\,\cdot\,}_{\mcB^\alpha_{p,q}(\mbT^2;\mbR^n)}$ and use the shorthand $\mcC^\alpha(\mbT^2;\mbR^n)\defeq\mcB^\alpha_{\infty,\infty}(\mbT^2;\mbR^n)$. As above, we will often remove the domain and target spaces when the context is clear. For $\alpha \in \mbR$ and $p,q\in [1,\infty]$ we use the notation $\mcB^{\alpha-}_{p,q}(\mbT^{2};\mbR^{n})\defeq\cap_{\alpha'<\alpha}\mcB^{\alpha'}_{p,q}(\mbT^{2};\mbR^{n})$. 

We define the action of the heat semigroup on $f \in L^1(\mbT^2;\mbR)$ by the flow,
\begin{equation*}
	[0,\infty)\ni t\mapsto P_{t} f \defeq\msF^{-1}(\euler^{-t\abs{2\uppi\place}^{2}}\hat{f}(\place))=\msH_t\ast f
\end{equation*}
where the convolution against the heat kernel $\msH_{t}$ on $\mbT^{2}$ is defined by
\begin{equation*}
	\msH_{t}\ast f=\int_{\mbT^{2}}\msH_{t}(\place-y)f(y)\dd y=\int_{(0,1)^{2}}\msH_{t}^{\per}(\place-y)f^{\per}(y)\dd y
\end{equation*}
with $f^{\per}\from\mbR^{2}\to\mbR$ the periodization of $f\from\mbT^{2}\to\mbR$ and $\msH^{\per}_{t}$ given by
\begin{equation*}
	\begin{split}
		\msH_{t}^{\per}(x)&\defeq\frac{1}{4\uppi t}\sum_{n\in\mbZ^2}\euler^{-\frac{\abs{x-n}^2}{4t}}=\sum_{\om\in\mbZ^{2}}\euler^{2\uppi\upi\inner{\om}{x}}\euler^{-t\abs{2\uppi\om}^2}\qquad\text{for}~t>0,\\
		\msH_{0}^{\per}&\defeq\sum_{n\in\mbZ^{2}}\delta_{n},\qquad\text{and}\qquad\msH_{t}^{\per}(x)\defeq0\qquad\text{for}~t<0.
	\end{split}
\end{equation*}
For $f\from[0,T]\times\mbT^{2}\to \mbR$, we define the resolution of the heat equation as
\begin{equation*}
	\mcI[f]_{t}\defeq\int_{0}^{t}P_{t-s}f_s\dd s=\int_{0}^{t}\msH_{t-s}\ast f_s\dd s.
\end{equation*}
We define the solution to the elliptic equation $-\Delta\Phi_{f}=f-\inner{f}{1}_{L^{2}(\mbT^{2})}$ by
\begin{equation*}
	\Phi_{f}\defeq\msG\ast f,\qquad\text{where}\quad f\in\mcS'(\mbT^{2};\mbR)\quad\text{and}
\end{equation*}
\begin{equation*}
	\msG(x)\defeq\sum_{\om\in\mbZ^{2}\setminus\{0\}}\euler^{2\uppi\upi\inner{\om}{x}}\frac{1}{\abs{2\uppi\om}^{2}},\qquad(\msG\ast f)(x)=\sum_{\om\in\mbZ^{2}\setminus\{0\}}\euler^{2\uppi\upi\inner{\om}{x}}\frac{1}{\abs{2\uppi\om}^{2}}\hat{f}(\om).
\end{equation*}
We denote for $\om=(\om^{1},\om^{2})\in\mbZ^{2}$, $t\in\mbR$ and $j=1,2$ the Fourier multipliers $H^{j}_{t}(\om)\defeq2\uppi\upi\om^{j}\exp(-t\abs{2\uppi\om}^{2})\mathds{1}_{t\geq0}$ and $G^{j}(\om)\defeq2\uppi\upi\om^{j}\abs{2\uppi\om}^{-2}\mathds{1}_{\om\neq0}$.

Given a Banach space $E$ and an interval $I\subseteq[0,\infty)$ we write $C_{I}E\defeq C(I;E)$ for the space of continuous maps $f\from I\to E$. For compact $I$ we equip $C_{I}E$ with the norm $\norm{f}_{C_{I}E}\defeq\sup_{t\in I}\norm{f_{t}}_{E}$. For $\kappa\in(0,1)$ we define $C^{\kappa}_{I}E\defeq C^{\kappa}(I;E)$ the space of $\kappa$-H\"{o}lder continuous maps $f\from I\to E$ equipped with the norm $\norm{f}_{C_{I}^{\kappa}E}\defeq\norm{f}_{C_{I}E}+\sup_{s\neq t\in I}\frac{\norm{f_{t}-f_{s}}_{E}}{\abs{t-s}^{\kappa}}$. For $T>0$, we use the shorthands $C_{T}E=C_{[0,T]}E$ and $C_{T}^{\kappa}E=C_{[0,T]}^\kappa E$. Note that the norm $\norm{f}_{C_T^{\kappa}E}$ is equivalent to $\norm{f_{0}}_{E}+\sup_{s\neq t\in[0,T]}\frac{\norm{f_{t}-f_{s}}_{E}}{\abs{t-s}^{\kappa}}$. For $\eta\geq0$ we define the Banach space
\begin{equation*}
	\begin{split}
		C_{\eta;T}E\defeq\Bigl\{f\from(0,T]\to E:~&(0,T]\ni t\mapsto(1\wedge t)^{\eta}f_{t}~\text{is continuous in}~E,\\
		&\multiquad[6](1\wedge t)^{\eta}f_{t}\bigtzero\defeq\lim_{t\to0}(1\wedge t)^{\eta}f_{t}\in E\Bigr\}
	\end{split}
\end{equation*}
equipped with the norm
\begin{equation*}
	\norm{f}_{C_{\eta;T}E}\defeq\sup_{t\in(0,T]}(1\wedge t)^{\eta}\norm{f_{t}}_{E}.
\end{equation*}
We refer to $\eta$ as the weight at $0$. For $\kappa\in(0,1]$ we let $C^\kappa_{\eta;T}E\defeq C^\kappa_\eta((0,T];E)$ denote the Banach space of functions $f\from(0,T]\to E$ which are finite under the norm
\begin{equation*}
	 \norm{f}_{C^{\kappa}_{\eta;T}E}\defeq\norm{f}_{C_{\eta;T}E}+\sup_{s\neq t\in(0,T]}(1\wedge s\wedge t)^{\eta}\frac{\norm{f_{t}-f_{s}}_{E}}{\abs{t-s}^{\kappa}}.
\end{equation*}
We also make use of the notation $\msL_{\eta;T}^{\kappa}\mcC^{\alpha}(\mbT^{2})\defeq C_{\eta;T}^{\kappa}\mcC^{\alpha-2\kappa}(\mbT^{2})\cap C_{\eta;T}\mcC^{\alpha}(\mbT^{2})$ to denote a weighted interpolation space, on which we define
\begin{equation*}
	\norm{u}_{\msL_{\eta;T}^{\kappa}\mcC^{\alpha}}\defeq\max\{\norm{u}_{C_{\eta;T}^{\kappa}\mcC^{\alpha-2\kappa}},\norm{u}_{C_{\eta;T}\mcC^{\alpha}}\}.
\end{equation*}
We set $\msL_{T}^{\kappa}\mcC^{\alpha}(\mbT^{2})\defeq\msL_{0;T}^{\kappa}\mcC^{\alpha}(\mbT^{2})\defeq C_{T}^{\kappa}\mcC^{\alpha-2\kappa}(\mbT^{2})\cap C_{T}\mcC^{\alpha}(\mbT^{2})$ and understand $\msL^{0}_{\eta;T}\mcC^{\alpha}(\mbT^{2})=C_{\eta;T}\mcC^{\alpha}(\mbT^{2})$.

We write $\lesssim$ to indicate that an inequality holds up to a constant depending on quantities that we do not keep track of or are fixed throughout. When we do wish to emphasise the dependence on certain quantities $\alpha$, $p$, $d$, we either write $\lesssim_{\alpha,p,d}$ or define $C\defeq C(\alpha,p,d)>0$ and write $\leq C$.

For two vectors $\om_{1},\om_{2}\in\mbZ^{2}$, we use $(\om_{1}\perp\om_{2})$ to denote the formula $\inner{\om_{1}}{\om_{2}}=0$ and $\lnot(\om_{1}\perp\om_{2})$ to denote its logical negation.

Recall the dyadic partition of unity $(\varrho_{k})_{k\in\mbN_{-1}}$ introduced in~\eqref{eq:dyadic_partition_of_unity_intro}. Let $u,v\in\mcS'(\mbT^{2})$, we define the truncated sums
\begin{equation}\label{eq:definition_sim}
	\sum_{\substack{\om_1,\om_2\in\mbZ^{2}\\\om_1\sim\om_2}}\hat{u}(\om_1)\hat{v}(\om_2)\defeq\sum_{\om_1,\om_2\in\mbZ^{2}}\hat{u}(\om_1)\hat{v}(\om_2)\sum_{\substack{k,l\in\mbN_{-1}\\\abs{k-l}\leq1}}\varrho_{l}(\om_1)\varrho_{k}(\om_2)
\end{equation}
and
\begin{equation}\label{eq:definition_precsim}
	\sum_{\substack{\om_1,\om_2\in\mbZ^{2}\\\om_1\precsim\om_2}}\hat{u}(\om_1)\hat{v}(\om_2)\defeq\sum_{\om_1,\om_2\in\mbZ^{2}}\hat{u}(\om_1)\hat{v}(\om_2)\sum_{k=1}^{\infty}\sum_{l=-1}^{k-2}\varrho_{l}(\om_1)\varrho_{k}(\om_2),
\end{equation}
where we implicitly assume that those sums converge absolutely. This is a discrete analogue of the usual paraproduct decomposition, cf.\ Subsection~\ref{subsec:paraproducts}. If $\om_1,\om_2\in\mbZ^{2}\setminus\{0\}$, then $\om_1\sim\om_2$ implies $9/64\abs{\om_1}\leq\abs{\om_2}\leq64/9\abs{\om_1}$ and $\om_1\precsim\om_2$ implies $\abs{\om_1}\leq8/9\abs{\om_2}$. The constants $9/64$ and $8/9$ are invariant under dilations of the chosen partition of unity $(\varrho_{k})_{k\in\mbN_{-1}}$, see e.g.~\cite[(64)]{mourrat_weber_xu_16}.
\begin{details}
	The restriction $\sim$ satisfies
	\begin{equation*}
		\sum_{\substack{k,l\in\mbN_{-1}\\\abs{k-l}\leq1}}\varrho_{l}(\om_1)\varrho_{k}(\om_2)
		\begin{cases}
			\in[0,1]\quad\text{for any }\om_1,\om_2,\\
			=0\quad\text{if }(\abs{\om_1}\geq1\text{ or }\abs{\om_2}\geq1)\text{ and }\frac{\abs{\om_1}}{\abs{\om_2}}\not\in[\frac{9}{64},\frac{64}{9}].
		\end{cases}
	\end{equation*}
	Assume $\abs{\om_2}\geq1$ and $\abs{\om_1}/\abs{\om_2}\not\in[9/64,64/9]$. In particular $\om_2\in\supp(\varrho_{k})$ for some $k\in\mbN_{1}$. We have either $\abs{\om_1}>\frac{64}{9}\abs{\om_2}\geq\frac{64}{9}\frac{9}{32}2^{k}=2^{k+1}$ or $\abs{\om_1}<\frac{9}{64}\abs{\om_2}<\frac{9}{64}2^{k}=\frac{9}{32}2^{k-1}$. In both cases, $\om_1\not\in\bigcup_{j=k-1}^{k+1}\supp(\varrho_{j})$ and so $\sum_{\abs{k-l}\leq1}\varrho_{l}(\om_1)\varrho_{k}(\om_2)=0$. If instead $\abs{\om_1}\geq1$, we can use $\abs{\om_2}/\abs{\om_1}\not\in[9/64,64/9]$ and exchange r\^{o}les.
	
	Occasionally when we consider second moments, we shall also use the same notation with the cut-off function $\abs{\sum_{\abs{k-l}\leq1}\varrho_{l}(\om_1)\varrho_{k}(\om_2)}^2$.
	
	Assume further $\om_1,\om_2\not=0$, as is usual the case in our diagrams. In this case, the conditions $\abs{\om_1}\geq1$, $\abs{\om_1}\geq1$ are always satisfied.
	
	Assume $\om_1,\om_2\in\mbZ^{2}\setminus\{0\}$ are in the support of the cut-off $\sum_{k=1}^{\infty}\sum_{l=-1}^{k-2}\varrho_{l}(\om_1)\varrho_{k}(\om_2)$. By the definition of the dyadic partition of unity $(\varrho_{q})_{q\in\mbN_{-1}}$, we obtain
	\begin{equation*}
		\abs{\om_1}\leq 2^{k-2}=\frac{8}{9}\frac{9}{32}2^{k}\leq\frac{8}{9}\abs{\om_2}
	\end{equation*}
	which yields the lower bound on $\om_2$. Furthermore, if $\om=\om_1+\om_2$, then $\abs{\om}\leq\abs{\om_1}+\abs{\om_2}\leq\frac{17}{9}\abs{\om_2}$.
\end{details}

We also define a probability space $(\Omega,\mcF,\mbP)$ which we assume to be large enough to support a countable family of Brownian motions. This probability space will be fixed, so that whenever a property holds almost surely, it will do so with respect to $\mbP$.
\subsection{Strategy and Main Result}\label{subsec:gen_strategy}
We first outline the paracontrolled approach to~\eqref{eq:gen_rKS_intro} in a relatively loose manner, identifying the main steps of the method and the diagrams that we will need to give meaning to. Recall that we wish to define a sufficiently robust notion of solution to~\eqref{eq:gen_rKS_intro} which in particular is stable under regular approximations to the noise. To do so we first write~\eqref{eq:gen_rKS_intro} in mild form, setting
\begin{equation}\label{eq:mild_gen_rKS}
	\rho=P\rho_0+\vdiv\mcI[\rho\nabla\Phi_{\rho}]+\vdiv\mcI[\het\bxi].
\end{equation}
Note that $\rho$ may blow up before time $T$ due to the non-linearity on the right-hand side of~\eqref{eq:mild_gen_rKS}. For the purpose of this discussion, we will assume that $\rho$ exists until time $T$; see Subsection~\ref{subsec:maximal_time_of_existence} for results on the maximal time of existence. In the remainder of this section we will assume that all terms on the right-hand side of~\eqref{eq:mild_gen_rKS} are continuous in time while taking values in a H\"{o}lder--Besov space $\mcC^{\alpha}(\mbT^2)$, where $\alpha$ is possibly negative.

Working, for now, with smooth initial data, we may assume that the final term on the right-hand side is the least regular component of $\rho$. Using the same stochastic estimates alluded to in the introduction, along with the regularising effect of the heat kernel and the effect of the derivative, we will work under the assumption that $\ti\defeq\nabla\cdot\mcI[\het\bxi]\in C_{T}\mcC^{-1-}(\mbT^{2})$. Passing this regularity to $\rho$ and applying the regularising effect of the elliptic equation (Lemma~\ref{lem:elliptic_regularity}) we expect to have $\nabla\Phi_{\rho} \in C_{T}\mcC^{0-}(\mbT^{2};\mbR^{2})$. Therefore, as discussed in the introduction, the product $\rho\nabla\Phi_{\rho}$ is not a priori well-defined. Our first step is to employ the so called Da Prato--Debussche trick~\cite{daprato_debussche_03} to remove the most singular term by defining $u\defeq \rho-\ti$ so that if $\rho$ is a solution to~\eqref{eq:mild_gen_rKS},
\begin{equation*}
	u=P\rho_0+\vdiv\mcI[u\nabla\Phi_u]+\vdiv\mcI[u\nabla\Phi_{\ti}]+\vdiv\mcI[\ti\nabla\Phi_u]+\vdiv \mcI[\ti \nabla \Phi_{\ti}].
\end{equation*}

We notice that the product $\ti \nabla \Phi_{\ti}$ is not classically well-posed, however it can be renormalised and replaced with the symbol $\ty \defeq \vdiv \mcI[\ti \nabla \Phi_{\ti}] - \tl$, where $\tl\defeq\mbE[\vdiv\mcI[\ti\nabla\Phi_{\ti}]]$ denotes the singular part of this product. The term $\ty$ is well-defined and will have the regularity that one would formally expect of $\vdiv\mcI[\ti\nabla\Phi_{\ti}]$, namely $\ty\in C_{T}\mcC^{0-}(\mbT^{2})$ (see Subsection~\ref{sec:diagrams_of_order_2_and_3}). From now on we continue with our expansion, replacing the singular product by its renormalised counterpart $\ty$ so that we have in fact changed the equation solved by $\rho$.

We are now in better shape, as we may now work with $u\in C_{T}\mcC^{0-}(\mbT^{2})$, which renders the first product on the right-hand side classically well-posed. However, the second and third products remain ill-defined. We may repeat the same trick, defining $w \defeq u- \ty$, which should solve,
\begin{equation*}
	w=P\rho_0+\vdiv\mcI [w\nabla \Phi_{\ti}]+\vdiv\mcI[\ti\nabla\Phi_{w}]+\vdiv\mcI[\nabla\Phi_{\ty}\pa\ti]+\vdiv\mcI[\tp] + Q(w,\ti,\ty),
\end{equation*}
where $Q(w,\ti,\ty)$ denotes a finite sum of classically well-posed terms involving $w$, $\ti$ and $\ty$. The formal definition of Bony's decomposition into para and resonant products is given in Appendix~\ref{app:Besov}, however, for now we simply recall the rules that for $f\in\mcC^{\alpha}(\mbT^{2})$, $g\in\mcC^{\beta}(\mbT^{2})$, one has
\begin{equation*}
	f\pa g\in\mcC^{\beta\wedge(\alpha+\beta)}(\mbT^{2})\text{ for any }\alpha\in\mbR\setminus\{0\},~\beta\in\mbR\qquad~\text{and}~\qquad f \re g\in\mcC^{\alpha+\beta}(\mbT^{2})~\text{if}~\alpha+\beta> 0.
\end{equation*}
The new symbol appearing on the right-hand side for $w$ is a shorthand for
\begin{equation*}
	\tp\defeq\ty\re\nabla\Phi_{\ti}+\nabla\Phi_{\ty}\re\ti.
\end{equation*}
%
%
Although those resonant products are not classically well-defined, further stochastic arguments show that they can in fact be defined as objects finite in $C_T\mcC^{0-}(\mbT^{2};\mbR^{2})$ without the subtraction of any infinite counterterms. 
The full definition of $\tp$, through stochastic calculus, is contained in Subsection~\ref{subsec:Feynman} and we show in Subsections~\ref{sec:diagrams_of_order_2_and_3} and~\ref{sec:Wick_contractions} that $\tp\in C_T\mcC^{0-}(\mbT^{2};\mbR^{2})$. So even though it requires significant work to define, it is not the least regular term on the right-hand side.

Instead this is given by the paraproduct term, $\vdiv\mcI[\nabla\Phi_{\ty}\pa\ti]$, which using Bony's estimate (Lemma~\ref{lem:Bony}) is only finite in $C_T\mcC^{0-}(\mbT^{2})$, and the formal product term $\vdiv\mcI[\ti\nabla \Phi_w]$, which is not even a priori well-defined. Hence, as before we can only expect to find $w\in C_T\mcC^{0-}(\mbT^{2})$ which is not regular enough to define the products $w\nabla \Phi_{\ti}$ and $\ti\nabla \Phi_w$ a priori.

One sees that further applications of the Da Prato--Debussche trick will not improve the situation. Instead we employ the core idea that solutions should resemble the noise at small scales. This is formalised through the \emph{paracontrolled Ansatz}, that is, we only look for solutions such that for each coordinate $j=1,2$,
\begin{equation*}
	w=\sum_{k=1}^{2}\partial_{k}\Phi_{w+\ty}\pa\partial_{k}\mcI[\ti]+w^{\#},\qquad\partial_{j}\Phi_{w}=\sum_{k=1}^{2}\partial_{k}\Phi_{w+\ty}\pa\partial_{k,j}\mcI[\Phi_{\ti}]+(\nabla\Phi_w)^{\#}_{j},
\end{equation*}
or in vector notation,
\begin{equation}\label{eq:intro_paracontrolled_Ansatz}
	w=\nabla \Phi_{w+\ty} \pa \nabla\mcI[\ti] +w^{\#},\qquad	\nabla\Phi_w=\nabla\Phi_{w+\ty}\pa\nabla^2\mcI[\Phi_{\ti}]+(\nabla \Phi_w)^{\#}, 
\end{equation}
where the \emph{paracontrolled remainders} $w^{\#}$ and $(\nabla \Phi_w)^{\#}$ are terms to be fixed by the equation which we stipulate to be finite in $C_T\mcC^{0+}(\mbT^{2})$ and $C_T\mcC^{1+}(\mbT^{2};\mbR^{2})$ respectively.\footnote{
	The remainder $(\nabla \Phi_w)^{\#}$ can be expressed in terms of $\nabla\Phi_{w^{\#}}$ and a more regular commutator, cf.\ Lemma~\ref{lem:commutator_Fourier_multiplier_paraproduct}. Hence, the paracontrolled Ansatz for $w$ implies the Ansatz for $\nabla\Phi_{w}$. In fact in Section~\ref{sec:existence} we make use of an equivalent Ansatz which makes certain technical steps easier but is less clear to present---see Remark~\ref{rem:Ansatz_discussion}.
}
This ensures that the products $w^{\#}\nabla \Phi_{\ti}$ and $\ti (\nabla \Phi_w)^{\#}$ are classically well-defined. Rearranging, using the linearity of the map $f\mapsto \nabla \Phi_f$ and applying Bony's decomposition to the products $w\nabla \Phi_{\ti}$ and $\ti\nabla \Phi_w$, we find the identity
\begin{align*}
	w^{\#} &= P\rho_0+ \vdiv\mcI [w\re \nabla \Phi_{\ti}]+ \vdiv\mcI[\ti\re \nabla\Phi_{w}]+ \vdiv\mcI[\nabla\Phi_{w+\ty}\pa \ti] -\nabla \Phi_{w+\ty} \pa \nabla \mcI[\ti]\\
	&\quad+\tilde{Q}(w,\ti,\ty,\tp),
\end{align*} 
where $\tilde{Q}(w,\ti,\ty,\tp)$ is a new polynomial of its arguments and can be expected to be of strictly positive regularity. Hence, the regularity of $w^{\#}$ is governed by that of the commutator and that of the resonant products $\vdiv\mcI [w\re \nabla \Phi_{\ti}]$ and $\vdiv\mcI[\ti\re \nabla\Phi_{w}]$. The commutator can be controlled by Lemmas~\ref{lem:commutator_Fourier_multiplier_paraproduct} and~\ref{lem:commutator_heat_paraproduct}, which show that
\begin{equation*}
	\vdiv\mcI[\nabla\Phi_{w+\ty}\pa \ti] -\nabla \Phi_{w+\ty} \pa \nabla\mcI[\ti] \in C_T\mcC^{1-}(\mbT^{2}).
\end{equation*} 
To treat the resonant products we make use of the \emph{Ansatz} again, writing for each coordinate $j=1,2$,
\begin{equation*}
	w\re\partial_{j}\Phi_{\ti}=\sum_{k=1}^{2}(\partial_{k}\Phi_{w+\ty}\pa\partial_{k}\mcI[\ti])\re\partial_{j}\Phi_{\ti}+w^{\#}\re\partial_{j}\Phi_{\ti},
\end{equation*}
and
\begin{equation*}
	\ti\re\partial_{j}\Phi_{w}=\sum_{k=1}^{2}\ti\re(\partial_{k}\Phi_{w+\ty}\pa\partial_{k,j}\mcI[\Phi_{\ti}])+\ti\re(\nabla \Phi_w)^{\#}_{j};
\end{equation*}
or again in vector notation,
\begin{equation*}
	 w\re \nabla \Phi_{\ti} = (\nabla \Phi_{w+\ty} \pa \nabla\mcI[\ti])\re \nabla \Phi_{\ti} +w^{\#}\re \nabla \Phi_{\ti},
\end{equation*}
and
\begin{equation*}
	\ti\re \nabla \Phi_{w} = \ti\re(\nabla\Phi_{w+\ty}\pa\nabla^2\mcI[\Phi_{\ti}])+\ti\re (\nabla \Phi_w)^{\#}.
\end{equation*}
Under our stipulation that $w^{\#}\in C_T\mcC^{0+}(\mbT^{2})$ and $(\nabla\Phi_w)^{\#}\in C_{T}\mcC^{1+}(\mbT^{2};\mbR^{2})$, the final two resonant products are classically well-defined and so it only remains to check that the first term of each expansion is finite. To achieve this last step we consider a commutator for the triple product, 
\begin{equation*}
	(\nabla \Phi_{w+\ty} \pa \nabla\mcI[\ti])\re \nabla \Phi_{\ti} = (\nabla \Phi_{w+\ty} \pa \nabla\mcI[\ti])\re \nabla \Phi_{\ti} - \nabla \Phi_{w+\ty} ( \nabla\mcI[\ti]\re \nabla \Phi_{\ti}) + \nabla \Phi_{w+\ty} (\nabla\mcI[\ti]\re \nabla \Phi_{\ti}).
\end{equation*}
Lemma~\ref{lem:commutator_paraproduct_resonant} shows that the commutator lies in $C_{T}\mcC^{1-}(\mbT^{2};\mbR^{2})$. We apply a similar trick to the resonant product $\ti \re \nabla \Phi_w$, writing
\begin{equation*}
	\ti\re(\nabla\Phi_{w+\ty}\pa\nabla^2\mcI[\Phi_{\ti}]) = \ti\re(\nabla\Phi_{w+\ty}\pa\nabla^2\mcI[\Phi_{\ti}]) - \nabla\Phi_{w+\ty}(\nabla^2\mcI[\Phi_{\ti}]\re\ti) + \nabla\Phi_{w+\ty}(\nabla^2\mcI[\Phi_{\ti}]\re\ti).
\end{equation*}
Again, the regularity of the commutator follows from Lemma~\ref{lem:commutator_paraproduct_resonant}. Taken together the last two exogenous terms produce the final diagram we are required to construct,
\begin{equation*}
	\tc \defeq \nabla\mcI[\ti]\re \nabla \Phi_{\ti} + \nabla^2\mcI[\Phi_{\ti}]\re\ti.
\end{equation*}
Note that the first resonant product above should be read as a vector outer product so that $\tc$ is matrix valued. We would naively expect both summands of $\tc$ to diverge logarithmically if we replace $\boldsymbol{\xi}$ by $\boldsymbol{\xi}^{\delta}$ and let $\delta\to 0$. However, the symmetry of the Green's function allows us to show that after summing both terms, $\tc$ is well-defined in a sufficiently strong topology even for $\delta=0$.

Reversing all of the above steps we find the modified equation solved by our paracontrolled object,
\begin{equation}\label{eq:paracon_sol_intro_I}
	\rho =\ti + \ty + \nabla \Phi_{w+\ty} \pa \nabla\mcI[\ti] +w^{\#},
\end{equation}
with $w^{\#}$ a solution to
\begin{equation}\label{eq:paracon_sol_intro_II}
	w^{\#} = P\rho_0+ \vdiv\mcI [w^{\#}\re \nabla \Phi_{\ti}]+ \vdiv\mcI[\nabla \Phi_{w+\ty} \tc]+ \vdiv\mcI[\ti\re (\nabla\Phi_w)^{\#}] + \bar{Q}(w,\ti,\ty,\tp),
\end{equation}
for $\bar{Q}(w,\ti,\ty,\tp)$ a third polynomial of its arguments, their paraproducts and commutators. We call this $\rho$ a paracontrolled solution to~\eqref{eq:gen_rKS_intro} with enhancement $\mbX$. For a rigorous definition, see Definition~\ref{def:paracontrolled_solution}.

In the paracontrolled decomposition~\eqref{eq:paracon_sol_intro_I} of $\rho$ the first three terms lie in spaces of negative regularity. Hence, the singular parts of the product $\rho \nabla \Phi_{\rho}$ will be determined by non-linear combinations of the first three terms. Since $\ti$ and $\ty$ will be supplied as data these terms can be handled directly. However, as $w$ also carries information from $\rho$, products involving $\nabla\Phi_{w+\ty}\pa\nabla\mcI[\ti]$ cannot be handled in the same way. Instead we make use of the commutator estimates above. To see this in practice and to identify the possibly diverging field $f^\delta$ alluded to in the introduction, we recall our notion of a mollified noise by setting $\ti^\delta \defeq \vdiv\mcI[\het (\psi_{\delta} \ast \boldsymbol{\xi}) ]$, where $\psi_{\delta}$ is a standard mollifier. We use the notations $\ty^{\delta}$, $\tp^{\delta}$, $\tc^{\delta}$ to denote the same diagrams now constructed from $\ti^\delta$, and use $\rho^{\delta}$ to denote the solution of the mild equation
\begin{equation}\label{eq:smooth_mild_sol}
	\rho^\delta = P\rho_0 + \vdiv \mcI[\rho^\delta \nabla \Phi_{\rho^\delta}] - \tl^\delta + \ti^\delta.
\end{equation}
We further denote the second Da~Prato--Debussche remainder by $w^{\delta}\defeq\rho^{\delta}-\ti^{\delta}-\ty^{\delta}$ and define $\ty^{\delta}_{\can}\defeq\vdiv\mcI[\ti^{\delta}\nabla\Phi_{\ti^{\delta}}]=\ty^{\delta}+\tl^{\delta}$. 

We have the identity,
\begin{equation}\label{eq:Par_Sol_Nonlinearity}
	\rho^\delta \nabla \Phi_{\rho^\delta} = \ti^\delta \nabla \Phi_{\ti^\delta} + \ty^\delta_{\can}\re \nabla \Phi_{\ti^{\delta}} +\nabla \Phi_{\ty^\delta_{\can}} \re \ti^{\delta}-\tl^{\delta}\re\nabla \Phi_{\ti^\delta}-\nabla \Phi_{\tl^{\delta}}\re\ti^\delta+\nabla\Phi_{w^{\delta}+\ty^{\delta}}\tc^{\delta}+\ldots.
\end{equation}
Here we have only kept track of terms that are either not classically well-defined or contain stochastic diagrams which require construction. The final term involving $\tc^\delta$ arises from applying commutators to the paraproduct term in the expansion of $\rho^\delta$ where the more regular parts have been left implicit above. Since we only expect to have $\rho^\delta\rightarrow \rho$ in $C_T\mcC^{-1-}(\mbT^{2})$ we do not expect \eqref{eq:Par_Sol_Nonlinearity} to converge directly. We have already identified the possibly diverging field which renormalises the first term, since
 \begin{equation*}
 	\vdiv \mcI[\ti^\delta \nabla \Phi_{\ti^\delta}]-\tl^\delta = \ty^\delta \rightarrow \ty \in C_T\mcC^{0-}(\mbT^{2}). 
 \end{equation*}
As discussed in the introduction, formal power counting would lead one to expect $\tl^{\delta}$ to diverge at order $\delta^{-1}$, however, exploiting the symmetry of the elliptic Green's function we have that $\norm{\tl^{\delta}}_{C_{T}\mcC^{0-}}\lesssim (1\vee\log (\delta^{-1}))\norm{\het}_{C_{T}\mcH^{2}}^{2}$. 
 
The diverging diagram $\tl^{\delta}$ is also contained in the terms $\tl^{\delta}\re\nabla \Phi_{\ti^{\delta}}$ and $\nabla \Phi_{\tl^{\delta}}\re\ti^{\delta}$. However, since $\tl^{\delta}$ is of regularity $0-$ and $\ti^{\delta}$ of regularity $-1-$, it is not directly clear how to make sense of those products. Note that if $\tl^{\delta}$ were a diverging constant rather than a field this would simply be scalar multiplication and we would have no trouble. Nevertheless, using that $\tl^{\delta}$ is deterministic, one can define the products $\tl^{\delta}\re\nabla \Phi_{\ti^{\delta}}$ and $\nabla \Phi_{\tl^{\delta}}\re \ti^{\delta}$ directly as It\^{o} objects and show that they diverge at a rate no worse than $\tl^{\delta}$. We refer to Subsection~\ref{sec:canonical} for this argument. 
 
Since $\ty^{\delta}_{\can}=\ty^{\delta}+\tl^{\delta}$, we may expand the product $\ty^\delta_{\can}\re \nabla \Phi_{\ti^\delta} + \nabla \Phi_{\ty^\delta_{\can}}\re \ti^\delta$ to cancel the diverging terms $\tl^{\delta}\re\nabla \Phi_{\ti^\delta}$ and $\nabla \Phi_{\tl^{\delta}}\re\ti^{\delta}$ in~\eqref{eq:Par_Sol_Nonlinearity}. Hence, we can construct the renormalised product $\lim_{\delta\to0}(\vdiv\mcI[\rho^{\delta}\nabla\Phi_{\rho^{\delta}}]-\tl^{\delta})$ without further modifications. In particular, this identifies the renormalising sequence $(f^{\delta})_{\delta>0}$ of~\eqref{eq:gen_rKS_renormalised} as $f^{\delta}=\mbE[\ti^{\delta}\nabla\Phi_{\ti^{\delta}}]$ and~\eqref{eq:smooth_mild_sol} as the corresponding mild formulation.

We have therefore identified both the solution $\rho$ and the non-linear term in~\eqref{eq:mild_gen_rKS} as trilinear functions of a suitable enhancement of the noise. To conclude this section we paraphrase the main result of this paper (Theorem~\ref{thm:convergence_renorm_sol}). In the following, $(\alpha,p,q,\beta,\beta',\beta^{\#},\beta_{0},\kappa,\eta)$ denotes a tuple of exponents which determines the regularity and integrability of each object in our paracontrolled decomposition. We assume that those exponents satisfy assumption~\eqref{eq:exponents} in Section~\ref{sec:existence}, see also Examples~\ref{ex:exponents_regular_initial_data}--\ref{ex:exponents_Lp_initial_data} for various special cases of interest.
\begin{theorem}\label{thm:main_theorem}
	Let $(\alpha,p,q,\beta,\beta',\beta^{\#},\beta_{0},\kappa,\eta)$ be a tuple of exponents satisfying~\eqref{eq:exponents}, $\rho_{0}\in\mcB_{p,q}^{\beta_{0}}(\mbT^{2})$ be some initial data, $T>0$ be a time horizon, $\boldsymbol{\xi}=(\xi^{1},\xi^{2})$ be a vector-valued space-time white noise on $[0,T]\times\mbT^{2}$, $\het\in C_{T}\mcH^{2}(\mbT^{2})$ be a heterogeneity and $(\psi_\delta)_{\delta>0}$ be a family of symmetric mollifiers. Then there exist enhancements $\mbX=(\ti,\ty,\tp,\tc)$, $\mbX^{\delta}=(\ti^\delta,\ty^\delta,\tp^\delta,\tc^\delta)$ as described above (in particular $\mbX^{\delta}$ is built from $\het\boldsymbol{\xi}^{\delta}$ with $\boldsymbol{\xi}^{\delta}=\psi_{\delta}\ast \bxi$) and a random variable $\Tmax\in(0,T]$ (the maximal time of existence) with the following properties:
	\begin{enumerate}
		\item There exists a unique paracontrolled solution $\rho$ to~\eqref{eq:gen_rKS_intro} on $[0,\Tmax)$ with enhancement $\mbX$ and initial data $\rho_{0}$ (in the sense of Definition~\ref{def:paracontrolled_solution}), which we call the renormalised solution. The maximal time of existence $\Tmax$ satisfies
		\begin{equation*}
			\Tmax=T\qquad\text{or}\qquad\lim_{t\uparrow\Tmax}\norm{(1\wedge t)^{\eta}\rho(t)}_{\mcC^{\alpha+1}}=\infty\qquad\text{almost surely}.
		\end{equation*}
		\item For each $\lambda>0$,
		\begin{equation*}
			\lim_{\delta\to0}\mbP\Bigl(\sum_{L=1}^{\infty}\frac{2^{-L}}{L}\frac{\norm{\rho-\rho^{\delta}}_{C_{\eta;T_{L}}\mcC^{\alpha+1}}}{1+\norm{\rho-\rho^{\delta}}_{C_{\eta;T_{L}}\mcC^{\alpha+1}}}>\lambda\Bigr)=0,
		\end{equation*}
		where $\rho^{\delta}$ is the solution to~\eqref{eq:smooth_mild_sol} and
		\begin{equation*}
			T_{L}\defeq T\wedge L\wedge\inf\Bigl\{t\in[0,T]:\norm{(1\wedge t)^{\eta}\rho(t)}_{\mcC^{-1-}}>L~\text{or}~\norm{(1\wedge t)^{\eta}\rho^{\delta}(t)}_{\mcC^{-1-}}>L\Bigr\}
		\end{equation*}
		for every $L\in\mbN$.
	\end{enumerate}
	Furthermore $\tl^\delta = \vdiv \mcI [\ti^\delta \nabla \Phi_{\ti^\delta}] -\ty^\delta = \mbE[\vdiv \mcI [\ti^\delta \nabla \Phi_{\ti^\delta}]]$. If $\het$ is a constant then $\tl^\delta \equiv 0$ while in general one has the bound $\norm{\tl^\delta}_{C_{T}\mcC^{2\alpha+4}}\lesssim(1\vee\log(\delta^{-1}))\norm{\het}_{C_{T}\mcH^{2}}^{2}$ (see Remark~\ref{rem:renormalisation_vanishes_homogeneous} and Lemma~\ref{lem:existence_canonical}).
\end{theorem}
\begin{remark}
	In the case of constant $\het$ it also holds that $\mbE[\ti^\delta\nabla \Phi_{\ti^\delta}]=0$. This is due to the symmetry of the elliptic Green's function, see the discussion of~\eqref{eq:product_rule}.
\end{remark}
\begin{remark}
	Since the submission of this manuscript it has been brought to our attention that in the particular case of the Keller--Segel non-linearity $\rho\nabla\Phi_{\rho}$, one may in fact simplify the analysis of the equation~\eqref{eq:mild_gen_rKS} in the following manner: Using that $\Phi$ denotes the resolution of the mean-free Laplacian on $\mbT^{2}$, one may cast the Keller--Segel non-linearity in the following form
	\begin{equation*}
		\rho\nabla\Phi_{\rho}=-\Delta\Phi_{\rho}\nabla\Phi_{\rho}+\inner{\rho}{1}_{L^{2}}\nabla\Phi_{\rho},
	\end{equation*}
	where 
	%
	%
	the first term can be written by the product rule as
	\begin{equation*}
		\begin{split}
			\Delta\Phi_{\rho}\nabla\Phi_{\rho}=\vdiv(\nabla\Phi_{\rho})^{\otimes2}-\frac{1}{2}\nabla\trace(\nabla\Phi_{\rho})^{\otimes2},
		\end{split}
	\end{equation*}
	with
	\begin{equation*}
		(\nabla\Phi_{\rho})^{\otimes2}_{i,j}\defeq(\nabla\Phi_{\rho}\otimes\nabla\Phi_{\rho})_{i,j}\defeq\partial_{i}\Phi_{\rho}\partial_{j}\Phi_{\rho},\qquad i,j=1,2.
	\end{equation*}
	\begin{details}
		For every coordinate $j=1,2$,
		\begin{equation*}
			\begin{split}
				\Delta\Phi_{\rho}\partial_{j}\Phi_{\rho}=\sum_{i=1}^{2}\partial_{i}(\partial_{i}\Phi_{\rho}\partial_{j}\Phi_{\rho})-\frac{1}{2}\sum_{i=1}^{2}\partial_{j}(\partial_{i}\Phi_{\rho}\partial_{i}\Phi_{\rho}).
			\end{split}
		\end{equation*}
	\end{details}
	Hence, we may re-write~\eqref{eq:mild_gen_rKS} as
	\begin{equation}\label{eq:Hendrik_and_Massimiliano_insight}
		\begin{split}
			\rho&=P\rho_{0}-\vdiv\mcI[\vdiv(\nabla\Phi_{\rho})^{\otimes2}]+\frac{1}{2}\vdiv\mcI[\nabla\trace(\nabla\Phi_{\rho})^{\otimes2}]+\vdiv\mcI[\inner{\rho}{1}_{L^{2}}\nabla\Phi_{\rho}]+\ti.
		\end{split}
	\end{equation}
	Using the bilinearity of each operator on the right-hand side of~\eqref{eq:Hendrik_and_Massimiliano_insight}, it follows that the Da~Prato--Debussche remainder $u=\rho-\ti$ satisfies the equation
	\begin{equation*}
		u=P\rho_{0}+\vdiv\mcI[\ti\nabla\Phi_{\ti}]+\oldhat{Q}(u,\ti),
	\end{equation*}
	where $\oldhat{Q}(u,\ti)$ denotes a finite sum of classically well-posed terms involving $\nabla\Phi_{u}$ and $\nabla\Phi_{\ti}$ rather than $u$ and $\ti$. Consequently, it suffices to renormalise the product $\vdiv\mcI[\ti\nabla\Phi_{\ti}]$ by subtracting $\tl$, without the need to perform a paracontrolled decomposition. Nevertheless our methods are still of independent interest, as they demonstrate how one may construct a noise enhancement that involves heterogeneities and anisotropic regularities. We thank M.~Gubinelli and H.~Weber for these insights.
\end{remark}
\section{Noise Enhancement}\label{sec:enhancements}
In this section we construct the enhancements required in Theorem~\ref{thm:main_theorem} and establish their regularities.
\subsection{Outline and Regularities}
We begin by defining a vector $\boldsymbol{\xi}=(\xi^1,\xi^2)$ of space-time white noises as in~\cite{mourrat_weber_xu_16}. Let $(W^{j}(\cdot,m))_{m\in\mbZ^2,j=1,2}$ be a family of complex-valued Brownian motions on $\mbR_{+}$ starting from $0$ that satisfy $\overline{W^j(\cdot,m)}=W^j(\cdot,-m)$ and are otherwise independent.
\begin{details}
	For example, those can be defined in the following way. Let $B^j(m)$, $\tilde{B}^j(m)$, $j=1,2$, $m\in\mbZ^2$, be real-valued Brownian motions on $\mbR_{+}$ starting from $0$ such that $\mbE[(B^j_t(m))^2]=\mbE[(\tilde{B}^j_t(m))^2]=t/2$. We assume the families $(B^j(m))_{m\in\mbZ^2,j=1,2}$, $(\tilde{B}^j(m))_{m\in\mbZ^2,j=1,2}$ satisfy $B^j(-m)=B^j(m)$ and $\tilde{B}^j(-m)=-\tilde{B}^j(m)$ but are otherwise independent in $j$, $m$ and are also mutually independent. We can then set $W^j(t,m)=B^j_t(m)+\upi\tilde{B}^j_t(m)$.
\end{details}
We define for $j=1,2$, the space-time white noise $\xi^j$ by setting for any $\phi\in L^{2}((0,\infty)\times\mbT^{2};\mbC)$,
\begin{equation}\label{eq:definition_white_noise}
	\xi^{j}(\phi)\defeq\sum_{m_1\in\mbZ^2}\int_{0}^{\infty} \dd W^j(u_1,m_1)\hat{\phi}(u_1,-m_1).
\end{equation}
\begin{details}
	Let $\phi,\psi\in L^2(\mbR_{+}\times\mbT^2;\mbC)$. By It\^{o}'s isometry and Parseval's theorem,
	\begin{equation*}
		\mbE[\xi^{j}(\phi)\overline{\xi^{j}(\psi)}]=\sum_{m_1\in\mbZ^2}\int_{0}^{\infty}\dd u_1\hat{\phi}(u_1,-m_1)\overline{\hat{\psi}(u_1,-m_1)}=\inner{\phi}{\psi}_{L^2(\mbR_{+}\times\mbT^2;\mbC)}.
	\end{equation*}
	Hence this noise is space-time $\delta$-correlated and indeed white.
	
\end{details}
\begin{details}
	Let us now construct the rough Keller--Segel noise $\boldsymbol{\zeta}\defeq\het\boldsymbol{\xi}$ and its mollification $\boldsymbol{\zeta}_{\delta}\defeq\het(\psi_{\delta}\ast\boldsymbol{\xi})$.
	
	We define the noise $\boldsymbol{\zeta}\defeq(\zeta^{1},\zeta^{2})$ by testing $\zeta^j=\het\xi^j$, $j=1,2$, with $\phi\in L^2(\mbR_{+}\times\mbT^2;\mbC)$,
	\begin{equation*}
		\zeta^{j}(\phi)=\xi^{j}(\het\phi)=\sum_{m_1\in\mbZ^2}\int_{0}^{\infty}\dd W^{j}(u_1,m_1)\mathscr{F}(\het\phi)(u_1,-m_1).
	\end{equation*}
	In order to define $\boldsymbol{\zeta}_{\delta}$, let us first construct $\psi_{\delta}$. Let $\varphi\in C^{\infty}(\mbR^2)$ be of compact support, even and such that $\varphi(0)=1$. We define for $\delta>0$, $\psi_{\delta}(x)\defeq\delta^{-2}\sum_{\om\in\mbZ^2}\mathscr{F}_{\mbR^2}^{-1}\varphi(\delta^{-1}(\om+x))$, $x\in\mbR^2$. We obtain by an application of the Poisson summation formula,
	\begin{equation*}
		\psi_{\delta}(x)=\sum_{\om\in\mbZ^2}\euler^{2\uppi\upi\inner{\om}{x}}\varphi(\delta\om).
	\end{equation*}
	We define $\boldsymbol{\zeta}_{\delta}\defeq(\zeta^{1}_{\delta},\zeta^{2}_{\delta})$, $\zeta^{j}_{\delta}(x)\defeq\het(x)\xi^{j}(\psi_{\delta}(x-\cdot))$, $x\in\mbT^2$, $j=1,2$. By testing $\zeta^{j}_{\delta}$ against $\phi\in L^2(\mbR_{+}\times\mbT^2;\mbC)$,
	\begin{equation*}
		\zeta^{j}_{\delta}(\phi)=\xi^{j}(\psi_{\delta}(-\cdot)\ast(\het\phi))=\sum_{m_1\in\mbZ^2}\int_{0}^{\infty}\dd W^{j}(u_1,m_1)\mathscr{F}(\psi_{\delta}(-\cdot)\ast(\het\phi))(u_1,-m_1).
	\end{equation*}
	Let formally $\Dot{W}^{j}(t,\om)$ be the temporal derivative of a complex Brownian motion. We take $\phi(s,x)=\delta_{t}(s)\exp(-2\uppi\upi\inner{\om}{x})$, $t>0$, $\om\in\mbZ^2$, as a test function for $\zeta^{j}_{\delta}$ and obtain
	\begin{equation*}
		\begin{split}
			\hat{\zeta^{j}_{\delta}}(t,\om)&=\sum_{m_1\in\mbZ^2}\Dot{W}^{j}(t,m_1)\msF(\psi_{\delta}(-\cdot)\ast(\het(t,\cdot)\exp(-2\uppi\upi\inner{\om}{\cdot})))(t,-m_1)\\
			&=\sum_{m_1\in\mbZ^2}\Dot{W}^{j}(t,m_1)\hat{\psi_{\delta}}(m_1)\sum_{m_2\in\mbZ^{2}}\hat{\het}(t,m_2)\delta_{-\om,-m_1-m_2}\\
			&=\sum_{m_1\in\mbZ^2}\Dot{W}^{j}(t,m_1)\varphi(\delta m_1)\hat{\het}(t,\om-m_1).
		\end{split}
	\end{equation*}
	Next, we motivate our eventual definitions of $\ti^{\delta}=\vdiv\mcI[\boldsymbol{\zeta}_{\delta}]$ and $\ti=\vdiv\mcI[\boldsymbol{\zeta}]$. Formally, $(\partial_t-\Delta)\ti^{\delta}=\vdiv\boldsymbol{\zeta}_{\delta}$, hence for the spatial Fourier coefficients,
	\begin{equation*}
		(\partial_t+\abs{2\uppi\om}^2)\hat{\ti^{\delta}}(t,\om)=\hat{\vdiv\boldsymbol{\zeta}_{\delta}}(t,\om)=\sum_{j_1=1}^{2}2\uppi\upi\om^{j_1}\hat{\zeta^{j_1}_{\delta}}(t,\om).
	\end{equation*}
	As a consequence,
	\begin{equation*}
		\begin{split}
			&\hat{\ti^{\delta}}(t,\om)-\hat{\ti^{\delta}}(0,\om)\\
			&=-\abs{2\uppi\om}^2\int_{0}^{t}\hat{\ti^{\delta}}(u_1,\om)\dd u_1+\sum_{j_1=1}^{2}2\uppi\upi\om^{j_1}\sum_{m_1\in\mbZ^2}\int_{0}^{t}\dd W^{j_1}(u_1,m_1)\hat{\het}(u_1,\om-m_1)\varphi(\delta m_1)
		\end{split}
	\end{equation*}
	and
	\begin{equation*}
		\dd\hat{\ti^{\delta}}(t,\om)=-\abs{2\uppi\om}^2\hat{\ti^{\delta}}(t,\om)\dd t+\sum_{j_1=1}^{2}2\uppi\upi\om^{j_1}\sum_{m_1\in\mbZ^2}\dd W^{j_1}(t,m_1)\hat{\het}(t,\om-m_1)\varphi(\delta m_1).
	\end{equation*}
	This is an Ornstein--Uhlenbeck process. By applying It\^{o}'s formula to $\hat{\ti^{\delta}}(t,\om)\exp(t\abs{2\uppi\om}^2)$, we arrive at the explicit solution
	\begin{equation*}
		\begin{split}
			&\hat{\ti^{\delta}}(t,\om)\\
			&=\euler^{-t\abs{2\uppi\om}^2}\hat{\ti^{\delta}}(0,\om)+\int_{0}^{t}\euler^{-\abs{t-u_1}\abs{2\uppi\om}^2}\sum_{j_1=1}^{2}2\uppi\upi\om^{j_1}\sum_{m_1\in\mbZ^2}\hat{\het}(u_1,\om-m_1)\varphi(\delta m_1)\dd W^{j_1}(u_1,m_1).
		\end{split}
	\end{equation*}
	We further define $H_{t}^j(\om)\defeq2\uppi\upi \om^j\exp(-t\abs{2\uppi\om}^2)\mathds{1}_{t\geq0}$, the multiplier associated to $\partial_{j}\mcI$. By introducing $H^{j_1}_{t-u_1}(\om)$ and setting $\hat{\ti^{\delta}}(0,\om)=0$, we obtain
	\begin{equation*}
		\hat{\ti^{\delta}}(t,\om)=\sum_{j_1=1}^{2}\sum_{m_1\in\mbZ^2}\int_{0}^{t}\dd W^{j_1}(u_1,m_1)\hat{\het}(u_1,\om-m_1)\varphi(\delta m_1) H_{t-u_1}^{j_1}(\om),
	\end{equation*}
	which we will simply use as the definition of $\ti^{\delta}=\vdiv\mcI[\het(\psi_{\delta}\ast\boldsymbol{\xi})]$.
	
\end{details}
We define our choice of mollifiers.
\begin{definition}\label{def:mollifiers}
	We define the cut-off 
	\begin{equation}\label{eq:def_cut_off}
		\begin{split}
			&\varphi\in C^{\infty}(\mbR^{2})~\text{to be of compact support},~\supp(\varphi)\subset B(0,1),~\text{even}~\text{and such that}~\varphi(0)=1.
		\end{split}
	\end{equation}
	Given $\varphi$ satisfying~\eqref{eq:def_cut_off}, we define a sequence of mollifiers $(\psi_{\delta})_{\delta>0}$ on $\mbT^{2}$ by
	\begin{equation}\label{eq:def_mollifiers}
		\psi_{\delta}(x)\defeq\sum_{\om\in\mbZ^{2}}\euler^{2\uppi\upi\inner{\om}{x}}\varphi(\delta\om).
	\end{equation}
\end{definition}
Our enhancements will lie in the following space of enhanced rough noises, which quantifies their regularities in interpolation spaces (see Section~\ref{sec:notations} and Definition~\ref{def:interpolation_space}).
\begin{definition}[Enhanced rough noise]\label{def:noise_space}
	For all $T>0$, $\alpha\in(-5/2,-2)$ and $\kappa\in(0,1/2)$ let the map
	\begin{equation*}
		\begin{split}
			\Theta&\from(\msL^{\kappa}_T\mcC^{\alpha+2}(\mbT^{2};\mbR)\times\msL^{\kappa}_T\mcC^{2\alpha+5}(\mbT^{2};\mbR))\\
			&\to\msL^{\kappa}_T\mcC^{\alpha+1}(\mbT^{2};\mbR)\times\msL^{\kappa}_T\mcC^{2\alpha+4}(\mbT^{2};\mbR)\times\msL^{\kappa}_T\mcC^{3\alpha+6}(\mbT^{2};\mbR^{2})\times\msL^{\kappa}_T\mcC^{2\alpha+4}(\mbT^{2};\mbR^{2\times2}),\\
			(v,f)&\mapsto\Theta(v,f),
		\end{split}
	\end{equation*}
	be given by 
	\begin{align*}
		Y&\defeq\vdiv\mcI[{v}\nabla\Phi_{v}]-f,\\
		\Theta(v,f)&\defeq(v,Y,Y\re\nabla\Phi_{v}+\nabla\Phi_Y\re v,\nabla \mcI[v]\re\nabla\Phi_{v}+\nabla^2\mcI[\Phi_{v}]\re {v})
	\end{align*}
	and define the space $\rksnoise{\alpha}{\kappa}_T$ to be the closure of the subset
	\begin{equation*}
		\begin{split}
			&\{\Theta(v,f):(v,f)\in\msL^{\kappa}_T\mcC^{\alpha+2}(\mbT^{2};\mbR)\times\msL^{\kappa}_T\mcC^{2\alpha+5}(\mbT^{2};\mbR),~v_{0}=f_{0}=0\}\\
			&\qquad\subset\msL^{\kappa}_T\mcC^{\alpha+1}(\mbT^{2};\mbR)\times\msL^{\kappa}_T\mcC^{2\alpha+4}(\mbT^{2};\mbR)\times\msL^{\kappa}_T\mcC^{3\alpha+6}(\mbT^{2};\mbR^{2})\times\msL^{\kappa}_T\mcC^{2\alpha+4}(\mbT^{2};\mbR^{2\times2}).
		\end{split}
	\end{equation*}
	We shall denote a generic element of this closure by $\mbX=(\ti,\ty,\tp,\tc)\in\rksnoise{\alpha}{\kappa}_T$ and equip it with the metric induced by the norm
	\begin{equation*}
		\norm{\mbX}_{\rksnoise{\alpha}{\kappa}_T}\defeq\max\{\norm{\ti}_{\msL^{\kappa}_T\mcC^{\alpha+1}},\norm{\ty}_{\msL^{\kappa}_T\mcC^{2\alpha+4}},\norm{\tp}_{\msL^{\kappa}_T\mcC^{3\alpha+6}},\norm{\tc}_{\msL^{\kappa}_T\mcC^{2\alpha+4}}\}.
	\end{equation*}
\end{definition}
\begin{details}
	Here, we needed the assumption $\alpha>-5/2$ to make sure that $v\nabla\Phi_{v}$ is well-defined.
	
\end{details}
The main result of this section is the following theorem, reminiscent of~\cite[Thm.~9.1]{gubinelli_perkowski_17}.
\begin{theorem}\label{thm:enhancement_existence}
	Let $T>0$, $\het\in C_{T}\mcH^{2}(\mbT^{2})$, $\boldsymbol{\xi}$ be a two-dimensional vector of space-time white noises, $(\psi_{\delta})_{\delta>0}$ be a sequence of mollifiers as in~\eqref{eq:def_mollifiers}, $\boldsymbol{\xi}^{\delta}\defeq\psi_{\delta}\ast\boldsymbol{\xi}\defeq(\psi_{\delta}\ast\xi^{1},\psi_{\delta}\ast\xi^{2})$ and $\mbX^\delta\defeq(\ti^\delta,\ty^\delta,\tp^\delta,\tc^\delta)$ be given by
	\begin{align*}
		&\ti^{\delta}\defeq\vdiv\mcI[\het\boldsymbol{\xi}^{\delta}],\quad\ty^{\delta}\defeq\vdiv\mcI[\ti^{\delta}\nabla\Phi_{\ti^{\delta}}]-\mbE[\vdiv\mcI[\ti^{\delta}\nabla\Phi_{\ti^{\delta}}]],\\
		&\tp^{\delta}\defeq\ty^{\delta}\re\nabla\Phi_{\ti^{\delta}}+\nabla\Phi_{\ty^{\delta}}\re\ti^{\delta},\quad\tc^{\delta}\defeq\nabla\mcI[\ti^{\delta}]\re\nabla\Phi_{\ti^{\delta}}+\nabla^{2}\mcI[\Phi_{\ti^{\delta}}]\re\ti^{\delta}.
	\end{align*}
	Then for all $\alpha\in(-5/2,-2)$ and $\kappa\in(0,1/2)$ the following hold
	\begin{enumerate}
		\item Almost surely, $\mbX^{\delta}\in\rksnoise{\alpha}{\kappa}_{T}$ and
		\begin{equation*}
			\mbX^\delta\in\msL^{\kappa}_{T}\mcC^{\alpha+2}(\mbT^{2};\mbR)\times\msL^{\kappa}_{T}\mcC^{2\alpha+5}(\mbT^{2};\mbR)\times\msL^{\kappa}_{T}\mcC^{3\alpha+8}(\mbT^{2};\mbR^{2})\times\msL^{\kappa}_{T}\mcC^{2\alpha+6}(\mbT^{2};\mbR^{2\times2}).
		\end{equation*}
		\item Almost surely, there exists some $\mbX=(\ti,\ty,\tp,\tc)\in\rksnoise{\alpha}{\kappa}_{T}$ such that for any $p\in[1,\infty)$ we have $\lim_{\delta\to0}\mbE[\norm{\mbX-\mbX^{\delta}}_{\rksnoise{\alpha}{\kappa}_{T}}^p]^{1/p}=0$ and $\mbE[\norm{\mbX}_{\rksnoise{\alpha}{\kappa}_{T}}^p]^{1/p}<\infty$.
		\item Defining
		\begin{equation*}
			\begin{split}
				&\tl^{\delta}\defeq\mbE[\vdiv\mcI[\ti^{\delta}\nabla\Phi_{\ti^{\delta}}]],\qquad\ty^{\delta}_{\can}\defeq\vdiv\mcI[\ti^{\delta}\nabla\Phi_{\ti^{\delta}}]=\ty^{\delta}+\tl^{\delta},\\
				&\multiquad[3]\tp^{\delta}_{\can}\defeq\ty^{\delta}_{\can}\re\nabla\Phi_{\ti^{\delta}}+\nabla\Phi_{\ty^{\delta}_{\can}}\re\ti^{\delta},
			\end{split}
		\end{equation*}
		it holds that 
		\begin{equation*}
			\norm{\tl^{\delta}}_{\msL_{T}^{\kappa}\mcC^{2\alpha+4}}\lesssim(1\vee\log(\delta^{-1}))\norm{\het}_{C_T\mcH^{2}}^{2}
		\end{equation*}
		and for any $p\in[1,\infty)$,
		\begin{equation*}
			\begin{split}
				&\mbE[\norm{\ty^{\delta}_{\can}}_{\msL_{T}^{\kappa}\mcC^{2\alpha+4}}^{p}]^{1/p}\lesssim(1\vee\log(\delta^{-1}))\norm{\het}_{C_T\mcH^{2}}^{2},\\
				&\mbE[\norm{\tp^{\delta}_{\can}}_{\msL_{T}^{\kappa}\mcC^{3\alpha+6}}^{p}]^{1/p}\lesssim(1\vee\log(\delta^{-1}))\norm{\het}_{C_T\mcH^{2}}^{3}.
			\end{split}
		\end{equation*}
	\end{enumerate}
\end{theorem}
An explicit definition of the limit $\mbX=(\ti,\ty,\tp,\tc)$ can be found in Subsection~\ref{subsec:Feynman}. We call $\mbX$ the \emph{renormalised enhancement} and $\mbX^{\delta}_{\can}=(\ti^{\delta},\ty^{\delta}_{\can},\tp^{\delta}_{\can},\tc^{\delta})$ the \emph{canonical enhancement}.

The result will be shown in several parts, namely in Lemma~\ref{lem:existence_lolli} ($\ti$), Lemma~\ref{lem:existence_ypsilon}~($\ty$), Lemma~\ref{lem:existence_three} and Lemma~\ref{lem:existence_precocktail}~($\tp$), Lemma~\ref{lem:existence_checkmark} and Lemma~\ref{lem:existence_triangle}~($\tc$) and Lemma~\ref{lem:existence_canonical}~($\tl^{\delta}$, $\ty_{\can}^{\delta}$, $\tp_{\can}^{\delta}$).
\begin{remark}
	Different aspects of Theorem~\ref{thm:enhancement_existence} require different assumptions on the heterogeneity $\het$. For example the regularity of $\ti$ and $\ty$ only requires $\het\in C_{T}L^{\infty}(\mbT^{2})$ (Lemma~\ref{lem:existence_lolli} and Lemma~\ref{lem:existence_ypsilon}) while the regularities of the contractions contained in $\tp$, $\tc$ require that uniformly over $t\in [0,T]$ one has $\sup_{\om \in \mbZ^2}|\hat{\het}(t,\omega)|(1+|\omega|^2) <\infty$. The assumption $\het\in C_{T}\mcH^{2}(\mbT^{2})$ implies both of these conditions and provides a convenient norm and well-studied space that controls the latter quantity; hence we choose to work with this simpler, if sub-optimal restriction. Furthermore, with a view to setting $\het = \sqrt{\rdet}$ (cf.~\eqref{eq:gen_rKS_canonical}) the condition $\het\in C_{T}\mcH^{2}(\mbT^{2})$ is more straightforward to check.
\end{remark}
\begin{remark}
	As discussed in the introduction, we build our \emph{regular} enhancement from $\het\boldsymbol{\xi}^{\delta}=\het(\psi_{\delta}\ast\boldsymbol{\xi})$ instead of $(\het\boldsymbol{\xi})^{\delta}=\psi_{\delta}\ast(\het\boldsymbol{\xi})$ and with only a spatial convolution. Hence, for $\ti^{\delta}$ to be more regular, we need to assume some regularity on the heterogeneity $\het$. We again make use of the condition $\sup_{t\in[0,T]}\sup_{\om\in\mbZ^{2}}\abs{\hat{\het}(t,\omega)}(1+\abs{\om}^{2})<\infty$ (see the proof of Lemma~\ref{lem:existence_lolli}), though other choices are possible.
\end{remark}
\begin{remark}
	Our methods also allow us to establish that $\ty\in\msL_{T}^{1-}\mcC^{0-}(\mbT^{2})$. However, since we do not make use of the additional time regularity, we omit the proof.
\end{remark}
In the remainder of this section, we outline the basic arguments involved in proving Theorem~\ref{thm:enhancement_existence}. We motivate the definition of $\ti$ and establish its existence.

It is well-known that the Fourier frequencies of the stochastic heat equation are given by Ornstein--Uhlenbeck processes and by a similar argument, we can find an expression for the Fourier transform of $\ti=\vdiv\mcI[\het\boldsymbol{\xi}]$. For all $\om\in\mbZ^2$ and $t\in\mbR$, let $H_{t}^j(\om)\defeq2\uppi\upi\om^{j}\exp(-t\abs{2\uppi\om}^{2})\mathds{1}_{t\geq0}$ be the multiplier appearing in $\partial_{j}\mcI$. We define $\ti$ by applying the inverse Fourier transform to the sequence
\begin{equation}\label{eq:lolli_example}
	\hat{\ti}(t,\om)\defeq\sum_{j_1=1}^{2}\sum_{m_1\in\mbZ^2}\int_{0}^{t}\dd W^{j_1}(u_1,m_1)\hat{\het}(u_1,\om-m_1)H_{t-u_1}^{j_1}(\om).
\end{equation}
We also introduce the Fourier transform of $\tau\defeq\vdiv\mcI[\boldsymbol{\xi}]$ by
\begin{equation}\label{eq:lolli_homogeneous_after_isometry}
	\hat{\tau}(t,\om)\defeq\sum_{j_1=1}^{2}\int_{0}^{t}\dd W^{j_1}(u_1,\om)H_{t-u_1}^{j_1}(\om).
\end{equation}
\begin{lemma}\label{lem:existence_lolli}
	Let $T>0$, $\alpha<-2$, $\kappa\in(0,1/2)$ and $\het\in C_{T}L^{\infty}(\mbT^{2})$. Then for any $p\in[1,\infty)$ we have $\mbE[\norm{\ti}_{\msL^{\kappa}_{T}\mcC^{\alpha+1}}^p]^{1/p}\lesssim\norm{\het}_{C_{T}L^{\infty}}$ and in particular $\ti\in\msL^{\kappa}_T\mcC^{\alpha+1}(\mbT^{2})$ a.s. Assume in addition $\het\in C_{T}\mcH^{2}(\mbT^{2})$ and $\delta>0$, then it holds that $\mbE[\norm{\ti^{\delta}}_{\msL^{\kappa}_{T}\mcC^{\alpha+2}}^p]^{1/p}\lesssim(1+\delta^{-2})^{1/2}\norm{\het}_{C_{T}\mcH^{2}}$ and in particular $\ti^{\delta}\in\msL^{\kappa}_{T}\mcC^{\alpha+2}(\mbT^{2})$ a.s. What is more, $\lim_{\delta\to0}\mbE[\norm{\ti-\ti^{\delta}}_{\msL^{\kappa}_T\mcC^{\alpha+1}}^{p}]^{1/p}=0$ for any $p\in[1,\infty)$.
\end{lemma}
\begin{proof}
	Let $\gamma\in(0,1]$, $\eps\in(0,\gamma/2)$ and $\max\{1/\eps,2\}<p<\infty$. To establish the existence and regularity of $\ti$ in a Besov space, we apply Nelson's estimate (Lemma~\ref{lem:Nelson_estimate}), Kolmogorov's continuity criterion (Lemma~\ref{lem:existence_criterion}) and the Besov embedding~\eqref{eq:Besov_cts_embedding}. Therefore, in order to establish that $\ti\in C_{T}^{\gamma/2-\eps}\mcC^{-1-\gamma-\eps-2/p}(\mbT^{2})\embed C_{T}^{\gamma/2-\eps}\mcC^{-1-\gamma-3\eps}(\mbT^{2})$ almost surely, it suffices to show that the quantity
	\begin{equation*}
		\sup_{s\neq t\in[0,T]}\abs{t-s}^{-p\gamma/2}\sum_{q\in\mbN_{-1}}2^{pq(-1-\gamma-\eps)}\int_{\mbT^{2}}\mbE[\abs{\Delta_q\ti(t,x)-\Delta_q\ti(s,x)}^2]^{p/2}\dd x
	\end{equation*}
	is finite. Using that $\het\in C_{T}L^{\infty}(\mbT^{2})$ is bounded, we can pass to real space with~\eqref{eq:isometry_bound_delta=0} to deduce for $\tau$ as in~\eqref{eq:lolli_homogeneous_after_isometry},
	\begin{equation*}
		\mbE[\abs{\Delta_{q}\ti(t,x)-\Delta_{q}\ti(s,x)}^2]\leq\norm{\het}_{C_{T}L^\infty}^{2}\mbE[\abs{\Delta_{q}\tau(t,x)-\Delta_{q}\tau(s,x)}^2].
	\end{equation*}
	Using that $\mbE[\hat{\tau}(t,\om)\overline{\hat{\tau}(s,\om')}]=0$ if $\om\neq\om'\in\mbZ^2$, we obtain
	\begin{equation*}
		\begin{split}
			&\mbE[\abs{\Delta_{q}\tau(t,x)-\Delta_{q}\tau(s,x)}^2]\\
			&=\sum_{\om\in\mbZ^{2}}\sum_{\om'\in\mbZ^{2}}\euler^{2\uppi\upi\inner{\om}{x}}\euler^{-2\uppi\upi\inner{\om'}{x}}\varrho_{q}(\om)\varrho_{q}(\om')\mbE[(\hat{\tau}(t,\om)-\hat{\tau}(s,\om))\overline{(\hat{\tau}(t,\om')-\hat{\tau}(s,\om'))}]\\
			&=\sum_{\om\in\mbZ^{2}}\varrho_{q}(\om)^{2}\mbE[\abs{\hat{\tau}(t,\om)-\hat{\tau}(s,\om)}^{2}].
		\end{split}
	\end{equation*}
	It follows by It\^{o}'s isometry and interpolation~\eqref{eq:interpolation},
	\begin{equation*}
		\mbE[\abs{\hat{\tau}(t,\om)-\hat{\tau}(s,\om)}^2]\leq\sum_{j_1=1}^{2}\int_{-\infty}^{\infty}\dd u_1\abs{H^{j_1}_{t-u_1}(\om)-H^{j_1}_{s-u_1}(\om)}^2\lesssim\abs{t-s}^{\gamma}\abs{\om}^{2\gamma}
	\end{equation*}
	\begin{details}
		We bound the right-hand side with~\eqref{eq:interpolation} for $\gamma\in[0,1]$,
		\begin{equation*}
			\begin{split}
				&\sum_{j_1=1}^{2}\int_{-\infty}^{\infty}\dd u_1\abs{H^{j_1}_{t-u_1}(\om)-H^{j_1}_{s-u_1}(\om)}^2\\
				&=\abs{2\uppi\om}^{2}(1-\euler^{-\abs{t-s}\abs{2\uppi\om}^2})^{2}\int_{-\infty}^{s}\dd u_1\euler^{-2\abs{s-u_1}\abs{2\uppi\om}^2}+\frac{1}{2}(1-\euler^{-2\abs{t-s}\abs{2\uppi\om}^2})\\
				&\lesssim\abs{t-s}^{\gamma}\abs{\om}^{2\gamma}.
			\end{split}
		\end{equation*}
	\end{details}
	and therefore uniformly in $x\in\mbT^{2}$,
	\begin{equation*}
		\mbE[\abs{\Delta_{q}\tau(t,x)-\Delta_{q}\tau(s,x)}^2]\lesssim\abs{t-s}^\gamma2^{q(2+2\gamma)}.
	\end{equation*}
	To summarize, we obtain by Lemma~\ref{lem:Nelson_estimate}, Lemma~\ref{lem:existence_criterion} and the arguments above for each $\max\{1/\eps,2\}<p<\infty$,
	\begin{equation}\label{eq:lolli_sum_decomposition}
		\begin{split}
			&\mbE[\norm{\ti}_{C_T^{\gamma/2-\eps}\mcC^{-1-\gamma-3\eps}}^p]\\
			&\lesssim\sup_{s\neq t\in[0,T]}\abs{t-s}^{-p\gamma/2}\sum_{q\in\mbN_{-1}}2^{pq(-1-\gamma-\eps)}\int_{\mbT^{2}}\mbE[\abs{\Delta_q\ti(t,x)-\Delta_q\ti(s,x)}^2]^{p/2}\dd x\\
			&\leq\sup_{s\neq t\in[0,T]}\abs{t-s}^{-p\gamma/2}\norm{\het}_{C_{T}L^{\infty}}^{p}\sum_{q\in\mbN_{-1}}2^{pq(-1-\gamma-\eps)}\sup_{x\in\mbT^{2}}\mbE[\abs{\Delta_q\tau(t,x)-\Delta_q\tau(s,x)}^2]^{p/2}\\
			&\lesssim\norm{\het}_{C_{T}L^{\infty}}^{p}\sum_{q\in\mbN_{-1}}2^{-pq\eps}\\
			&\lesssim\norm{\het}_{C_{T}L^{\infty}}^{p},
		\end{split}
	\end{equation}
	which we can generalize to $p\in[1,\infty)$ by H\"{o}lder's inequality. In particular it follows by Lemma~\ref{lem:existence_criterion} that $\ti\in C_T^{\gamma/2-\eps}\mcC^{-1-\gamma-3\eps}(\mbT^{2})$ almost surely, where $\eps\in(0,\gamma/2)$ can be chosen arbitrarily small.
	\begin{details}
		By choosing $\gamma$ arbitrarily small we obtain $\ti\in C_T\mcC^{-1-}(\mbT^{2})$ and hence $\ti\in\msL^{\kappa}_T\mcC^{-1-}(\mbT^{2})$ a.s.\ for any $\kappa\in(0,1/2)$.
	\end{details}
	
	Next we show that the approximating sequence $(\ti^\delta)_{\delta>0}$ is more regular if $\het\in C_{T}\mcH^{2}(\mbT^{2})$. In Fourier space,
	\begin{equation}\label{eq:lolli_approximation_example}
		\hat{\ti^{\delta}}(t,\om)=\sum_{j_1=1}^{2}\sum_{m_1\in\mbZ^2}\int_{0}^{t}\dd W^{j_1}(u_1,m_1)\hat{\het}(u_1,\om-m_1)\varphi(\delta m_1)H_{t-u_1}^{j_1}(\om).
	\end{equation}
	We apply It\^{o}'s isometry, the triangle inequality and interpolation~\eqref{eq:interpolation} to estimate
	\begin{equation*}
		\begin{split}
			&\Bigl\lvert\mbE\Bigl[\Bigl(\hat{\ti^{\delta}}(t,\om)-\hat{\ti^{\delta}}(s,\om)\Bigr)\overline{\Bigl(\hat{\ti^{\delta}}(t,\om')-\hat{\ti^{\delta}}(s,\om')\Bigr)}\Bigr]\Bigr\rvert\\
			&\lesssim\abs{t-s}^{\gamma}\abs{\om}^{\gamma}\abs{\om'}^{\gamma}\norm{\het}_{C_{T}\mcH^{2}}^{2}\sum_{m_1\in\mbZ^{2}}(1+\abs{\om-m_1}^{2})^{-1}(1+\abs{\om'-m_1}^{2})^{-1}(1+\abs{\delta m_1}^{2})^{-1},
		\end{split}
	\end{equation*}
	where we used that $\abs{\hat{\het}(u,\om)}\lesssim(1+\abs{\om}^{2})^{-1}\norm{\het}_{C_{T}\mcH^{2}}$ uniformly in $u\in[0,T]$ and $\om\in\mbZ^{2}$, and that $(1+x^2)^{1/2}\abs{\varphi(x)}\lesssim\norm{\varphi}_{L^{\infty}}$ uniformly in $x\in\mbR^{2}$ (the latter of which we leave implicit.)
	\begin{details}
		By It\^{o}'s isometry,
		\begin{equation*}
			\begin{split}
				&\mbE\Bigl[\Bigl(\hat{\ti^{\delta}}(t,\om)-\hat{\ti^{\delta}}(s,\om)\Bigr)\overline{\Bigl(\hat{\ti^{\delta}}(t,\om')-\hat{\ti^{\delta}}(s,\om')\Bigr)}\Bigr]\\
				&=\sum_{j_1=1}^{2}\sum_{m_1\in\mbZ^{2}}\int_{0}^{\infty}\dd u_1\hat{\het}(u_1,\om-m_1)\overline{\hat{\het}(u_1,\om'-m_1)}\abs{\varphi(\delta m_1)}^{2}(H_{t-u_1}^{j_1}(\om)-H_{s-u_1}^{j_1}(\om))\overline{(H_{t-u_1}^{j_1}(\om')-H_{s-u_1}^{j_1}(\om'))}.
			\end{split}
		\end{equation*}
		Assume $s<t$. We estimate with~\eqref{eq:interpolation} for $\gamma\in[0,1]$,
		\begin{equation*}
			\begin{split}
				&\sum_{j_1=1}^{2}\int_{-\infty}^{\infty}\dd u_1 \abs{H_{t-u_1}^{j_1}(\om)-H_{s-u_1}^{j_1}(\om)}\abs{H_{t-u_1}^{j_1}(\om')-H_{s-u_1}^{j_1}(\om')}\\
				&\lesssim\abs{\om}\abs{\om'}\int_{-\infty}^{s}\dd u_1(\euler^{-\abs{s-u_1}\abs{2\uppi\om}^{2}}-\euler^{-\abs{t-u_1}\abs{2\uppi\om}^{2}})(\euler^{-\abs{s-u_1}\abs{2\uppi\om'}^{2}}-\euler^{-\abs{t-u_1}\abs{2\uppi\om'}^{2}})\\
				&\quad+\abs{\om}\abs{\om'}\int_{s}^{t}\dd u_1\euler^{-\abs{t-u_1}(\abs{2\pi\om}^{2}+\abs{2\pi\om'}^{2})}\\
				&=\abs{\om}\abs{\om'}(1-\euler^{-\abs{t-s}\abs{2\uppi\om}^{2}})(1-\euler^{-\abs{t-s}\abs{2\uppi\om'}^{2}})\int_{-\infty}^{s}\dd u_1\euler^{-\abs{s-u_1}(\abs{2\uppi\om}^{2}+\abs{2\uppi\om'}^{2})}\\
				&\quad+\abs{\om}\abs{\om'}(\abs{2\pi\om}^{2}+\abs{2\pi\om'}^{2})^{-1}(1-\euler^{-\abs{t-s}(\abs{2\pi\om}^{2}+\abs{2\pi\om'}^{2})})\\
				&\lesssim\abs{t-s}^{\gamma}\abs{\om}^{\gamma}\abs{\om'}^{\gamma}+\abs{\om}\abs{\om'}\abs{t-s}^{\gamma}(\abs{\om}^{2}+\abs{\om'}^{2})^{\gamma-1}\\
				&\lesssim\abs{t-s}^{\gamma}\abs{\om}^{\gamma}\abs{\om'}^{\gamma}.
			\end{split}
		\end{equation*}
	\end{details}
	We may assume $\om,\om'\in\mbZ^{2}\setminus\{0\}$, since $\hat{\ti^{\delta}}(t,0)=0$. We decompose the sum over $m_1\in\mbZ^{2}$ into the domains $m_1=0$, $m_1=\om$, $m_1=\om'$ and $m_1\in\mbZ^{2}\setminus\{0,\om,\om'\}$,
	\begin{equation*}
		\begin{split}
			&\sum_{m_1\in\mbZ^{2}}(1+\abs{\om-m_1}^{2})^{-1}(1+\abs{\om'-m_1}^{2})^{-1}(1+\abs{\delta m_1}^{2})^{-1}\\
			&\leq\abs{\om}^{-2}\abs{\om'}^{-2}+\delta^{-2}(1+\abs{\om-\om'}^{2})^{-1}(\abs{\om}^{-2}+\abs{\om'}^{-2})\\
			&\quad+\delta^{-2}\sum_{m_1\in\mbZ^{2}\setminus\{0,\om,\om'\}}\abs{\om-m_1}^{-2}\abs{\om'-m_1}^{-2}\abs{m_1}^{-2}.
		\end{split}
	\end{equation*}
	We estimate the sum over $m_1\in\mbZ^{2}\setminus\{0,\om,\om'\}$. Assume $\om=\om'$, then by Lemma~\ref{lem:convolution_estimate_paraproduct} and~\eqref{eq:summation_estimates_-alpha},
	\begin{equation*}
		\sum_{m_1\in\mbZ^{2}\setminus\{0,\om,\om'\}}\abs{\om-m_1}^{-4}\abs{m_1}^{-2}\lesssim\abs{\om}^{-2}.
	\end{equation*}
	\begin{details}
		We decompose the sum
		\begin{equation*}
			\begin{split}
				\sum_{m_1\in\mbZ^{2}\setminus\{0,\om,\om'\}}\abs{\om-m_1}^{-4}\abs{m_1}^{-2}&=\sum_{\substack{m_1\in\mbZ^{2}\setminus\{0,\om,\om'\}\\m_{1}\prec(\om-m_{1})}}\abs{\om-m_1}^{-4}\abs{m_1}^{-2}+\sum_{\substack{m_1\in\mbZ^{2}\setminus\{0,\om,\om'\}\\(\om-m_{1})\prec m_{1}}}\abs{\om-m_1}^{-4}\abs{m_1}^{-2}\\
				&\quad+\sum_{\substack{m_1\in\mbZ^{2}\setminus\{0,\om,\om'\}\\m_{1}\sim(\om-m_{1})}}\abs{\om-m_1}^{-4}\abs{m_1}^{-2}.
			\end{split}
		\end{equation*}
		Let $\eps<1$, by the bound induced by~\eqref{eq:definition_precsim} and Lemma~\ref{lem:convolution_estimates},
		\begin{equation*}
			\sum_{\substack{m_1\in\mbZ^{2}\setminus\{0,\om,\om'\}\\m_{1}\prec(\om-m_{1})}}\abs{\om-m_1}^{-4}\abs{m_1}^{-2}\lesssim\abs{\om}^{-2}\sum_{\substack{m_1\in\mbZ^{2}\setminus\{0,\om,\om'\}\\m_{1}\prec(\om-m_{1})}}\abs{\om-m_1}^{-2}\abs{m_1}^{-2}\lesssim\abs{\om}^{-2}\abs{\om}^{-2+2\eps}\lesssim\abs{\om}^{-2}.
		\end{equation*}
		By Lemma~\ref{lem:convolution_estimate_paraproduct},
		\begin{equation*}
			\sum_{\substack{m_1\in\mbZ^{2}\setminus\{0,\om,\om'\}\\(\om-m_{1})\prec m_{1}}}\abs{\om-m_1}^{-4}\abs{m_1}^{-2}\lesssim\abs{\om}^{-2}.
		\end{equation*}
		By Lemma~\ref{lem:convolution_estimates},
		\begin{equation*}
			\sum_{\substack{m_1\in\mbZ^{2}\setminus\{0,\om,\om'\}\\m_{1}\sim(\om-m_{1})}}\abs{\om-m_1}^{-4}\abs{m_1}^{-2}\lesssim\abs{\om}^{-4}.
		\end{equation*}
		This yields the claim.
		
	\end{details}
	Assume $\om\neq\om'$, we apply Lemma~\ref{lem:convolution_twofold} to estimate for $\eps<1$,
	\begin{equation*}
		\sum_{m_1\in\mbZ^{2}\setminus\{0,\om,\om'\}}\abs{\om-m_1}^{-2}\abs{\om'-m_1}^{-2}\abs{m_1}^{-2}\lesssim\abs{\om-\om'}^{-2+\eps}(\abs{\om}^{-2+2\eps}+\abs{\om'}^{-2+2\eps}).
	\end{equation*}
	Having established the decay of the Fourier coefficients, we can bound the Littlewood--Paley blocks. Let $q\in\mbN_{-1}$, we obtain uniformly in $x\in\mbT^{2}$,
	\begin{equation*}
		\begin{split}
			\mbE[\abs{\Delta_q\ti^{\delta}(t,x)-\Delta_q\ti^{\delta}(s,x)}^{2}]&\leq\sum_{\om,\om'\in\mbZ^{2}\setminus\{0\}}\varrho_q(\om)\varrho_q(\om')\Bigl\lvert\mbE\Bigl[\Bigl(\hat{\ti^{\delta}}(t,\om)-\hat{\ti^{\delta}}(s,\om)\Bigr)\overline{\Bigl(\hat{\ti^{\delta}}(t,\om')-\hat{\ti^{\delta}}(s,\om')\Bigr)}\Bigr]\Bigr\rvert\\
			&\lesssim(1+\delta^{-2})\norm{\het}_{C_{T}\mcH^{2}}^{2}\abs{t-s}^{\gamma}2^{q(2\gamma+3\eps)}.
		\end{split}
	\end{equation*}
	Consequently, by Lemma~\ref{lem:Nelson_estimate} and Lemma~\ref{lem:existence_criterion} for any $p\in[1,\infty)$ and $\eps\in(0,\gamma/2)$,
	\begin{details}
		\begin{equation*}
			\begin{split}
				&\mbE[\norm{\ti^{\delta}}_{C_T^{\gamma/2-\eps}\mcC^{-\gamma-4\eps}}^p]\\
				&\lesssim\sup_{s\neq t\in[0,T]}\abs{t-s}^{-p\gamma/2}\sum_{q\in\mbN_{-1}}2^{pq(-\gamma-2\eps)}\int_{\mbT^{2}}\mbE[\abs{\Delta_q\ti^{\delta}(t,x)-\Delta_q\ti^{\delta}(s,x)}^2]^{p/2}\dd x\\
				&\lesssim(1+\delta^{-2})^{p/2}\norm{\het}_{C_{T}\mcH^{2}}^{p}\sum_{q\in\mbN_{-1}}2^{pq(-\gamma-2\eps)}2^{pq(\gamma+\frac{3}{2}\eps)}\\
				&\lesssim(1+\delta^{-2})^{p/2}\norm{\het}_{C_{T}\mcH^{2}}^{p}\sum_{q\in\mbN_{-1}}2^{-pq\frac{1}{2}\eps}\\
				&\lesssim(1+\delta^{-2})^{p/2}\norm{\het}_{C_{T}\mcH^{2}}^{p}
			\end{split}
		\end{equation*}
		and so
	\end{details}
	\begin{equation*}
		\mbE[\norm{\ti^{\delta}}_{C_T^{\gamma/2-\eps}\mcC^{-\gamma-4\eps}}^p]^{1/p}\lesssim(1+\delta^{-2})^{1/2}\norm{\het}_{C_T\mcH^{2}}
	\end{equation*}
	and in particular, $\ti^{\delta}\in C_{T}^{\gamma/2-\eps}\mcC^{-\gamma-4\eps}(\mbT^{2})$ almost surely, where $\eps\in(0,\gamma/2)$ can be chosen arbitrarily small.
	
	To deduce $\lim_{\delta\to0}\mbE[\norm{\ti-\ti^{\delta}}_{\msL_{T}^{\kappa}\mcC^{\alpha+1}}^{p}]^{1/p}=0$ for all $p\in[1,\infty)$, $\alpha<-2$ and $\kappa\in(0,1/2)$, let $\gamma\in(0,1]$, $\eps\in(0,\gamma/2)$, $\max\{1/\eps,2\}<p<\infty$ and $0<\vartheta'<\vartheta\leq1$. We use Nelson's estimate (Lemma~\ref{lem:Nelson_estimate}) and Kolmogorov's continuity criterion (Lemma~\ref{lem:existence_criterion}) to bound
	\begin{equation*}
		\begin{split}
			&\mbE[\norm{\ti-\ti^{\delta}}_{C_{T}^{\gamma/2-\eps}\mcC^{-1-\gamma-\vartheta-3\eps}}^{p}]\\
			&\lesssim\int_{0}^{T}\int_{0}^{T}\abs{t-s}^{-2-p(\gamma/2-\eps)}\sum_{q\in\mbN_{-1}}2^{pq(-1-\gamma-\vartheta-\eps)}\\
			&\quad\times\int_{\mbT^{2}}\mbE[\abs{\Delta_q(\ti-\ti^{\delta})(t,x)-\Delta_q(\ti-\ti^{\delta})(s,x)}^2]^{p/2}\dd x\dd s\dd t.
		\end{split}
	\end{equation*}
	We then use Lemma~\ref{lem:isometry_interpolation} to bound
	\begin{equation*}
		\begin{split}
			&\mbE[\abs{\Delta_q(\ti-\ti^{\delta})(t,x)-\Delta_q(\ti-\ti^{\delta})(s,x)}^{2}]\\
			&\lesssim\delta^{2\vartheta'}\norm{1-\varphi}_{C^{1}}^{2}\norm{\het}_{C_{T}\mcH^{1}}^{2}\sum_{\om\in\mbZ^{2}}\varrho_{q}(\om)^{2}\sum_{j_{1}=1}^{2}\int_{-\infty}^{\infty}\dd u_{1}(1+\abs{2\uppi\om}^{2})^{\vartheta}\abs{H_{t-u_{1}}^{j_{1}}(\om)-H_{s-u_{1}}^{j_{1}}(\om)}^{2}.
		\end{split}
	\end{equation*}
	Using the same estimates as in the derivation of~\eqref{eq:lolli_sum_decomposition}, we obtain
	\begin{equation*}
		\begin{split}
			&\mbE[\abs{\Delta_q(\ti-\ti^{\delta})(t,x)-\Delta_q(\ti-\ti^{\delta})(s,x)}^{2}]\\
			&\lesssim\delta^{2\vartheta'}\abs{t-s}^{\gamma}\norm{1-\varphi}_{C^{1}}^{2}\norm{\het}_{C_{T}\mcH^{1}}^{2}\sum_{\om\in\mbZ^{2}}\varrho_{q}(\om)^{2}(1+\abs{2\uppi\om}^{2})^{\vartheta}\abs{\om}^{2\gamma}\\
			&\lesssim\delta^{2\vartheta'}\abs{t-s}^{\gamma}\norm{1-\varphi}_{C^{1}}^{2}\norm{\het}_{C_{T}\mcH^{1}}^{2}2^{2q(1+\gamma+\vartheta)}.
		\end{split}
	\end{equation*}
	Therefore,
	\begin{equation*}
		\begin{split}
			&\mbE[\norm{\ti-\ti^{\delta}}_{C_{T}^{\gamma/2-\eps}\mcC^{-1-\gamma-\vartheta-3\eps}}^{p}]\\
			&\lesssim\delta^{p\vartheta'}\norm{1-\varphi}_{C^{1}}^{p}\norm{\het}_{C_{T}\mcH^{1}}^{p}\int_{0}^{T}\int_{0}^{T}\abs{t-s}^{-2+p\eps}\dd s\dd t\sum_{q\in\mbN_{-1}}2^{-pq\eps}\\
			&\lesssim\delta^{p\vartheta'}\norm{1-\varphi}_{C^{1}}^{p}\norm{\het}_{C_{T}\mcH^{1}}^{p},
		\end{split}
	\end{equation*}
	which implies $\lim_{\delta\to0}\mbE[\norm{\ti-\ti^{\delta}}_{C_{T}^{\gamma/2-\eps}\mcC^{-1-\gamma-\vartheta-3\eps}}^{p}]=0$. An application of H\"{o}lder's inequality allows us to deduce the convergence for all $p\in[1,\infty)$, where $\eps$ and $\vartheta$ can be chosen arbitrarily small. This yields the claim.
\end{proof}
The convergence of the other approximations in $\mbX^{\delta}$ is analogous. The only crucial observation here is that we can choose $\vartheta'$ and $\vartheta$ arbitrarily small when applying Lemma~\ref{lem:isometry_interpolation}, which ensures that the various applications of Lemma~\ref{lem:convolution_estimates} in the proof of Theorem~\ref{thm:enhancement_existence} carry over. Hence, we will only show the existence of $\mbX$ and $\mbX^{\delta}_{\can}$, the second part of Theorem~\ref{thm:enhancement_existence}, that is, the convergence $\mbX^{\delta}\to\mbX$, then follows as above.
\begin{details}
	See Subsection~\ref{sec:diagrams_of_order_2_and_3} for an example.
\end{details}
\subsection{Feynman Diagrams}\label{subsec:Feynman}
%
%
As demonstrated by~\eqref{eq:lolli_example}, we may construct objects derived from white noise as (iterated) stochastic integrals. However, as we continue to multiply terms to construct the enhancement of Theorem~\ref{thm:enhancement_existence}, we need to apply It\^{o}'s product rule to increasingly complicated expressions. To implement this procedure efficiently, we use an extension of a graphical representation that was developed by~\cite{mourrat_weber_xu_16,gubinelli_perkowski_17}, which relates our stochastic objects to Feynman diagrams.

We make use of several types of vertices, which will be introduced throughout this subsection. A root
$\begin{tikzpicture}[scale=0.7,baseline=-0.6ex]
	\node at (0,0) [root] (middle) {};
	\node at (0.7,0) [] {\scriptsize $(t,\om)$};
\end{tikzpicture}\!\!$
represents the argument $(t,\om)$.
A circle
$\begin{tikzpicture}[scale=0.7,baseline=-0.6ex]
	\node at (0,0) [vectorcircle] (middle) {};
\end{tikzpicture}$
denotes an instance of stochastic integration in time against a two-dimensional Brownian field with heterogeneity $\het$. Graphically this integrator is given by
\begin{equation*}
	\begin{tikzpicture}[scale=0.7,baseline=-0.6ex]
		\node at (0,0) [vectorcircle] (middle) {};
		\node at (0,0.4)[]{\scriptsize $(u_1,\om_1,m_1,j_1)$};
	\end{tikzpicture}
	=\sum_{\om_{1}\in\mbZ^{2}}\sum_{j_1=1}^{2}\sum_{m_1\in\mbZ^{2}}\int_{0}^{\infty}\dd W^{j_1}(u_1,m_1)\hat{\het}(u_1,\om_1-m_1)\ldots,
\end{equation*}
where the placeholder $\ldots$ stands for an integrand in $u_1$, $\om_1$, $m_1$, $j_1$, which is to be determined from the remaining diagram.

Generally, circles are enumerated by $k\in\mbN$ and equipped with tuples of dummy variables $(u_k,\om_k,m_k,j_k)$, where $u_k\in(0,\infty)$, $\om_k\in\mbZ^{2}$, $m_k\in\mbZ^{2}$ and $j_k=1,2$, which we point out explicitly for the reader's convenience. We refer to the $\om$, $\om_k$ as the \emph{frequencies} and the $m_k$ as the \emph{modes}. We denote coordinates of those by superscripts.

Vertices are connected by different types of directed edges, i.e.\ arrows, that represent integrands. Black arrows
$\!\!\begin{tikzpicture}[scale=0.7,baseline=-0.6ex]
	\node at (0,0) [] (middle){};
	\node at (1.2,0) [] (right){};
	\node at (2.47,0)[]{\scriptsize $(u_k,\om_k,m_k,j_k)$};
	\node at (-0.3,0) [] {\scriptsize $(t,\om)$}; 
	\draw[scalarkernel] (middle) to (right); 
\end{tikzpicture}\!\!$
are associated to the integrand $H_{t-u_k}^{j_k}(\om_k)$, which is the Fourier multiplier appearing in $\vdiv\mcI$. Highlighted arrows
$\!\!\begin{tikzpicture}[scale=0.7,baseline=-0.0ex]
	\node at (0,0) [] (middle){};
	\node at (1.2,0) [] (right){};
	\node at (2.47,0)[]{\scriptsize $(u_k,\om_k,m_k,j_k)$};
	\node at (-0.3,0) [] {\scriptsize $(t,\om)$};
	\node at (0.55,0.3) [] {\scriptsize $j$}; 
	\draw[vectorkernel] (middle) to (right); 
\end{tikzpicture}\!\!$,
for $j=1,2$, are associated to $G^j(\om_k)H_{t-u_k}^{j_k}(\om_k)$, where 
\begin{equation*}
	G^j(\om_k)\defeq2\uppi\upi \om_k^j\abs{2\uppi\om_k}^{-2}\mathds{1}_{\om_k\neq0}
\end{equation*}
is the multiplier for $\partial_j\Phi$. The direction of an arrow indicates the smaller time variable $u_k$ in the integration. Furthermore, if a single arrow emerges from the root
$\begin{tikzpicture}[scale=0.7,baseline=-0.6ex]
	\node at (0,0) [root] (middle) {};
	\node at (0.7,0) [] {\scriptsize $(t,\om)$};
\end{tikzpicture}\!\!\!\!$
, we multiply the integrand by~$\mathds{1}_{\om=\om_{1}}$. 

For example by applying those rules, we can represent~\eqref{eq:lolli_example} as 
\begin{equation*}
	\hat{\ti}(t,\om)=
	\begin{tikzpicture}[scale=0.7,baseline=1.2ex]
		\node at (0,0) [root] (below){};
		\node at (0,1.2) [vectorcircle] (above){};
		\node at (0,1.6)[]{\scriptsize $(u_1,\om,m_1,j_1)$};
		\node at (0,-0.4) [] {\scriptsize $(t,\om)$}; 
		\draw[scalarkernel] (below) to (above); 
	\end{tikzpicture}
	\defeq\sum_{j_1=1}^{2}\sum_{m_1\in\mbZ^{2}}\int_{0}^{t}\dd W^{j_1}(u_1,m_1)\hat{\het}(u_1,\om-m_1) H_{t-u_1}^{j_1}(\om),
\end{equation*}
and $\nabla\Phi_{\ti}$ as
\begin{equation*}
	\hat{\partial_j\Phi_{\ti}}(t,\om)=
	\begin{tikzpicture}[scale=0.7,baseline=1.2ex]
		\node at (0,0) [root] (below){};
		\node at (0,1.2) [vectorcircle] (above){};
		\node at (0,1.6)[]{\scriptsize $(u_1,\om,m_1,j_1)$};
		\node at (0,-0.4) [] {\scriptsize $(t,\om)$};
		\node at (0.3,0.55) [] {\scriptsize $j$};
		\draw[vectorkernel] (below) to (above); 
	\end{tikzpicture}
	\defeq\sum_{j_1=1}^{2}\sum_{m_1\in\mbZ^{2}}\int_{0}^{t}\dd W^{j_1}(u_1,m_1)\hat{\het}(u_1,\om-m_1)G^{j}(\om)H_{t-u_1}^{j_1}(\om).
\end{equation*}
In particular if $\om=0$, then $H^j_t(\om)=G^j(\om)=0$, hence we may assume $\om\neq0$ whenever it appears in either multiplier.

As a general rule, \textbf{black} objects in our diagrams are associated to scalars and {\color{vectorcolor}\textbf{highlighted}} objects are associated to vectors. Our arrows have highlighted arrowheads, indicating that they act on vector-valued objects. On the other hand, the type of object they return is determined by the arrow shaft. Accordingly,
$\!\!\begin{tikzpicture}[scale=0.7,baseline=-0.6ex]
	\node at (0,0) [] (below){};
	\node at (1.2,0) [] (right){};
	\node at (2.47,0)[]{\scriptsize $(u_k,\om_k,m_k,j_k)$};
	\node at (-0.3,0) [] {\scriptsize $(t,\om)$}; 
	\draw[scalarkernel] (below) to (right); 
\end{tikzpicture}\!\!$
produces a scalar and
$\!\!\begin{tikzpicture}[scale=0.7,baseline=-0.6ex]
	\node at (0,0) [] (middle){};
	\node at (1.2,0) [] (right){};
	\node at (2.47,0)[]{\scriptsize $(u_k,\om_k,m_k,j_k)$};
	\node at (-0.3,0) [] {\scriptsize $(t,\om)$};
 	\draw[vectorkernel] (middle) to (right); 
\end{tikzpicture}\!\!$
a vector.

The existence of $\vdiv\mcI[\ti\nabla\Phi_{\ti}]$ is not guaranteed by Bony's estimates (Lemma~\ref{lem:Bony}), since $\ti\in C_{T}\mcC^{-1-}(\mbT^{2})$ and hence $\nabla\Phi_{\ti}\in C_{T}\mcC^{0-}(\mbT^{2};\mbR^{2})$. In order to construct such non-linear objects, we formally apply It\^{o}'s product rule to identify the candidate Fourier transform. Let $n\in\mbN$ and assume $a_1,\ldots,a_n\in\mbN$ are distinct. We denote by $\Sigma(a_1,\ldots,a_n)$ the permutation group of $\{a_1,\ldots,a_n\}$. Let $\om_1,\om_2\in\mbZ^2$, we compute 
\begin{equation}\label{eq:diagram_Ito}
	\begin{split}
		&\mathscr{F}(\ti\nabla\Phi_{\ti})(t,\om,j)=\sum_{\substack{\om_1,\om_2\in\mbZ^{2}\\\om=\om_1+\om_2}}\hat{\ti}(t,\om_1)\hat{\partial_{j}\Phi_{\ti}}(t,\om_2)\\
		&=\sum_{\substack{\om_1,\om_2\in\mbZ^{2}\\\om=\om_1+\om_2}}\sum_{j_1,j_2=1}^{2}\sum_{m_1,m_2\in\mbZ^{2}}\int_{0}^{t}\dd W^{j_2}(u_2,m_2)\int_{0}^{u_2}\dd W^{j_1}(u_1,m_1)\\
		&\quad\sum_{\varsigma\in\Sigma(1,2)}\hat{\het}(u_{\varsigma(1)},\om_{\varsigma(1)}-m_{\varsigma(1)})\hat{\het}(u_{\varsigma(2)},\om_{\varsigma(2)}-m_{\varsigma(2)})H^{j_{{\varsigma(1)}}}_{t-u_{\varsigma(1)}}(\om_{\varsigma(1)})G^{j}(\om_{\varsigma(2)})H^{j_{\varsigma(2)}}_{t-u_{\varsigma(2)}}(\om_{\varsigma(2)})\\
		&\quad+\sum_{\substack{\om_1,\om_2\in\mbZ^{2}\\\om=\om_1+\om_2}}\sum_{j_1=1}^{2}\sum_{m_1\in\mbZ^{2}}\int_{0}^{t}\dd u_1\hat{\het}(u_1,\om_1-m_1)\hat{\het}(u_1,\om_2+m_1)H^{j_1}_{t-u_1}(\om_1)G^{j}(\om_2)H^{j_1}_{t-u_1}(\om_2)\\
		&\eqqcolon\hat{\Cherry{10}}(t,\om,j)+\hat{\Loop}(t,\om,j).
	\end{split}
\end{equation}
\begin{details}
	\begin{equation*}
		\begin{split}
			\sum_{\substack{\om_1,\om_2\in\mbZ^{2}\\\om=\om_1+\om_2}}\hat{\ti}(t,\om_1)\hat{\partial_{j}\Phi_{\ti}}(t,\om_2)&=\sum_{\substack{\om_1,\om_2\in\mbZ^{2}\\\om=\om_1+\om_2}}\Bigl(\sum_{j_1=1}^{2}\sum_{m_1\in\mbZ^{2}}\int_{0}^{t}\dd W^{j_1}(u_1,m_1)\hat{\het}(u_1,\om_1-m_1)H^{j_1}_{t-u_1}(\om_1)\Bigr)\\
			&\quad\times\Bigl(\sum_{j_2=1}^{2}\sum_{m_2\in\mbZ^{2}}\int_{0}^{t}\dd W^{j_2}(u_2,m_2)\hat{\het}(u_2,\om_2-m_2)G^{j}(\om_2)H^{j_2}_{t-u_2}(\om_2)\Bigr)\\
		\end{split}
	\end{equation*}
\end{details}
The symmetrization of the first integrand is a direct consequence of the It\^{o} product rule. Note that in the second term, $j_1=j_2$, $u_1=u_2$ but $m_1=-m_2$, which is a consequence of the Hermitean structure of complex Brownian motion. Such a decomposition of stochastic products into iterated (stochastic) integrals is often called a \emph{Wiener chaos decomposition} after~\cite{wiener_38}, see~\cite{hairer_16,mourrat_weber_xu_16} for more details.

To represent $\Cherry{10}$, let us extend our graphical rules. Two arrows emerging from a common vertex represent a convolution in Fourier space. Their integrands are multiplied, but are related by the \emph{Kirchhoff rule}~\cite{mourrat_weber_xu_16}: each vertex $v$ has a frequency $\om$ or $\om_k$ which is part of its variables. This frequency will be called \emph{ingoing} at the vertex $v$. An ingoing frequency at a vertex $v$ is \emph{outgoing} for the vertex $w$, if there exists an arrow pointing from $w$ to $v$. A vertex is called \emph{internal}, if there exists an arrow emerging from it. The rule states that at each internal vertex, the ingoing frequency (e.g.\ $\om$ above) equals the sum of the outgoing frequencies (e.g.\ $\om_1$, $\om_2$ above). In graphical notation,
\begin{equation*}
	\begin{tikzpicture}[scale=0.7,baseline=-0.6ex]
		\node at (0,0) [root] (middle) {}; 
		\node at (-0.6,1.2) [] (left) {}; 
		\node at (0.6,1.2) [] (right) {}; 
		\node at (0,-0.4) () {\scriptsize $(t,\om)$};
		\node at (-1.45,1.3) [] {\scriptsize $(u_1,\om_1,m_1,j_1)$}; 
		\node at (1.45,1.3) [] {\scriptsize $(u_2,\om_2,m_2,j_2)$};
		\node at (0.5,0.4) [] {\scriptsize $j$};
		\draw[scalarkernel] (middle) to (left); 
		\draw[vectorkernel] (middle) to (right); 
	\end{tikzpicture}
	\defeq\mathds{1}_{\om=\om_1+\om_2}H^{j_{1}}_{t-u_{1}}(\om_{1})G^{j}(\om_{2})H^{j_{2}}_{t-u_{2}}(\om_{2}).
\end{equation*}
Those arrows will target the integrators
$\begin{tikzpicture}[scale=0.7,baseline=-0.6ex]
	\node at (0,0) [vectorcircle] (middle) {};
	\node at (1.6,0)[]{\scriptsize $(u_1,\om_1,m_1,j_1)$};
\end{tikzpicture}$
and
$\begin{tikzpicture}[scale=0.7,baseline=-0.6ex]
	\node at (0,0) [vectorcircle] (middle) {};
	\node at (1.6,0)[]{\scriptsize $(u_2,\om_2,m_2,j_2)$};
\end{tikzpicture}$
which will be multiplied and integrated over. The integral is then restricted to the simplex $u_1<u_2$ to ensure that the integrand is adapted. To obtain the integral over the full domain, we symmetrize the integrand by permuting the indices that appear in the simplex. For example,
\begin{equation*}
	\begin{tikzpicture}[scale=0.7,baseline=-0.6ex]
		\node at (0,0) [root] (middle) {}; 
		\node at (-0.6,1.2) [vectorcircle] (left) {}; 
		\node at (0.6,1.2) [vectorcircle] (right) {}; 
		\node at (0,-0.4) () {\scriptsize $(t,\om)$};
		\node at (-1.45,1.6) [] {\scriptsize $(u_1,\om_1,m_1,j_1)$}; 
		\node at (1.45,1.6) [] {\scriptsize $(u_2,\om_2,m_2,j_2)$};
		\node at (0.5,0.4) [] {\scriptsize $j$};
		\draw[scalarkernel] (middle) to (left); 
		\draw[vectorkernel] (middle) to (right); 
	\end{tikzpicture}=\hat{\Cherry{10}}(t,\om,j)
\end{equation*}
is the first object in the decomposition~\eqref{eq:diagram_Ito}.

Next, let us discuss $\Loop$. As can be seen in~\eqref{eq:diagram_Ito}, instances of Lebesgue integration
%
%
%
arise through It\^{o} correction terms. It\^{o} corrections will be denoted by \emph{contractions}, i.e.\ two arrows pointing at different vertices
\begin{tikzpicture}[scale=0.7,baseline=-0.6ex]
	\node at (0,0) [vectorcircle] (middle) {};
\end{tikzpicture}
and
\begin{tikzpicture}[scale=0.7,baseline=-0.6ex]
	\node at (0,0) [vectorcircle] (middle) {};
\end{tikzpicture}
are merged at a common vertex
$\begin{tikzpicture}[scale=0.7,baseline=-0.6ex]
	\node at (0,0) [dot] (middle) {};
\end{tikzpicture}\,$.
Graphically,
\begin{equation*}
	\begin{tikzpicture}[scale=0.7,baseline=-0.6ex]
		\node at (0,0) [] (middle) {};
		\node at (-0.6,1.2) [vectorcircle] (left) {}; 
		\node at (0.6,1.2) [vectorcircle] (right) {}; 
		\draw[scalarkernel] (middle) to (left); 
		\draw[vectorkernel] (middle) to (right);
		\node at (4.2,0.6)[]{are contracted to};
		\node at (7.8,1.2) [dot] (above) {}; 
		\draw[vectorkernel, bend right=60] (7.8,0) to (above); 
		\draw[scalarkernel, bend left=60] (7.8,0) to (above);
	\end{tikzpicture}.
\end{equation*}
Using the orthogonality of $(W^{j}(u,m))_{u\geq0,m\in\mbZ^{2},j=1,2}$, we can identify some of the dummy variables of the two vertices that are being merged. Indeed as in~\eqref{eq:diagram_Ito} we set $j_1=j_2$, $u_1=u_2$, $m_1=-m_2$, but leave $\om_1$, $\om_2$ as they are. To make it easier for the reader to discern the multipliers attached to each arrow, we give both tuples of dummy variables, even after the contraction. Graphically,
\begin{equation*}
	\begin{tikzpicture}[scale=0.7,baseline=-0.6ex]
		\node at (0,0) [] (middle) {};
		\node at (-0.6,1.2) [vectorcircle] (left) {};
		\node at (0.6,1.2) [vectorcircle] (right) {};
		\node at (-1.45,1.6) [] {\scriptsize $(u_1,\om_1,m_1,j_1)$}; 
		\node at (1.45,1.6) [] {\scriptsize $(u_2,\om_2,m_2,j_2)$}; 
		\draw[scalarkernel] (middle) to (left); 
		\draw[vectorkernel] (middle) to (right);
		\node at (4.2,0.6)[]{are contracted to};
		\node at (7.8,1.2) [dot] (above) {};
		\node at (6.36,1.6) [] {\scriptsize $(u_1,\om_1,m_1,j_1)$}; 
		\node at (9.41,1.6) [] {\scriptsize $(u_1,\om_2,-m_1,j_1)$};
		\draw[vectorkernel, bend right=60] (7.8,0) to (above); 
		\draw[scalarkernel, bend left=60] (7.8,0) to (above);
	\end{tikzpicture}.
\end{equation*}
This results in the diagram
\begin{equation*}
	\begin{tikzpicture}[scale=0.7,baseline=-0.6ex]
		\node at (0,0) [root] (middle) {}; 
		\node at (0,1.2) [dot] (above) {}; 
		\node at (-1.44,1.6) [] {\scriptsize $(u_1,\om_1,m_1,j_1)$}; 
		\node at (1.61,1.6) [] {\scriptsize $(u_1,\om_2,-m_1,j_1)$};
		\node at (0.65,0.6) [] {\scriptsize $j$};
		\node at (0,-0.4) {\scriptsize $(t,\om)$}; 
		\draw[vectorkernel, bend right=60] (middle) to (above); 
		\draw[scalarkernel, bend left=60] (middle) to (above);
	\end{tikzpicture}=\hat{\Loop}(t,\om,j),
\end{equation*}
which is the It\^{o} correction in the decomposition~\eqref{eq:diagram_Ito} and coincides with the mean, $\Loop=\mbE[\ti\nabla\Phi_{\ti}]$. Diagrams carrying contractions are often called \emph{Wick contractions} after~\cite{wick_50}, see~\cite{gubinelli_perkowski_17} for more details. 
%

Depending on $\het$, $\Loop$ may be infinite. Hence, we consider the \emph{renormalised enhancement}, where we only keep the first term $\Cherry{10}$ of the decomposition~\eqref{eq:diagram_Ito} to define $\ty$. Formally, this is equivalent to subtracting the mean as a counterterm,
\begin{equation*}
	\Cherry{10}=\ti\nabla\Phi_{\ti}-\mbE[\ti\nabla\Phi_{\ti}].
\end{equation*}
This identity can be made rigorous with suitable regularization and limiting procedures, see the discussion of the canonical enhancement at the end of this section.

Another source of Lebesgue integrals is the concatenation of the $\mcI$ operation. We obtain arrows pointing at other arrows, connected through a vertex
$\begin{tikzpicture}[scale=0.7,baseline=-0.6ex]
	\node at (0,0) [dot] (middle) {};
\end{tikzpicture}\,$.
We multiply their integrands and make sure to respect the Kirchhoff rule. The multiplier of the incoming arrow will be determined by a tuple of dummy variables $(u_k,\om_k,j_k)$ at the connecting vertex. For example, $\vdiv\mcI[\nabla\Phi_{\ti}]$ can be expressed as 
\begin{equation*}
	\begin{split}
		&\msF(\vdiv\mcI[\nabla\Phi_{\ti}])(t,\om)=
		\begin{tikzpicture}[scale=0.7,baseline=-0.6ex]
			\node at (0,0) [root] (left){};
			\node at (1.4,0) [dot] (middle){};
			\node at (2.8,0) [vectorcircle] (right){};
			\node at (-0.7,0)[]{\scriptsize $(t,\om)$};
			\node at (1.4,0.45)[]{\scriptsize $(u_2,\om,j_2)$};
			\node at (4.3,0) [] {\scriptsize $(u_1,\om,m_1,j_1)$}; 
			\draw[scalarkernel] (left) to (middle);
			\draw[vectorkernel] (middle) to (right); 
		\end{tikzpicture}\\
		&\defeq\sum_{j_1,j_2=1}^{2}\sum_{m_1\in\mbZ^{2}}\int_{0}^{t}\dd u_2\int_{0}^{u_2}\dd W^{j_1}(u_1,m_1)\hat{\het}(u_1,\om-m_1)H_{t-u_2}^{j_2}(\om)G^{j_2}(\om)H_{u_2-u_1}^{j_1}(\om).
	\end{split}
\end{equation*}
The renormalised stochastic object $\ty=\vdiv\mcI[\ti\nabla\Phi_{\ti}]-\mbE[\vdiv\mcI[\ti\nabla\Phi_{\ti}]]$ can then be expressed as
\begin{equation*}
	\begin{split}
		&\hat{\ty}(t,\om)=
		\begin{tikzpicture}[scale=0.7,baseline=-4.2ex]
			\node at (0,0) [dot] (middle) {}; 
			\node at (-0.6,1.2) [vectorcircle] (left) {}; 
			\node at (0.6,1.2) [vectorcircle] (right) {}; 
			\node at (0,-1.2) [root] (below) {}; 
			\node at (0,-1.6) () {\scriptsize $(t,\om)$};
			\node at (-1.45,1.6) [] {\scriptsize $(u_1,\om_1,m_1,j_1)$}; 
			\node at (1.45,1.6) [] {\scriptsize $(u_2,\om_2,m_2,j_2)$};
			\node at (1.75,0) {\scriptsize $(u_3,\om_1+\om_2,j_3)$}; 
			\draw[scalarkernel] (middle) to (left); 
			\draw[vectorkernel] (middle) to (right); 
			\draw[scalarkernel] (below) to (middle);
		\end{tikzpicture}\\
		&\defeq\sum_{\substack{\om_1,\om_2\in\mbZ^{2}\\\om=\om_1+\om_2}}\sum_{j_1,j_2,j_3=1}^{2}\sum_{m_1,m_2\in\mbZ^{2}}\int_{0}^{t}\dd u_3\int_{0}^{u_3}\dd W^{j_2}(u_2,m_2)\int_{0}^{u_2}\dd W^{j_1}(u_1,m_1)\\
		&\quad\hat{\het}(u_{1},\om_{1}-m_{1})\hat{\het}(u_{2},\om_{2}-m_{2})\sum_{\varsigma\in\Sigma(1,2)}H^{j_3}_{t-u_3}(\om_{\varsigma(1)}+\om_{\varsigma(2)})H^{j_{\varsigma(1)}}_{u_3-u_{\varsigma(1)}}(\om_{\varsigma(1)})G^{j_3}(\om_{\varsigma(2)})H^{j_{\varsigma(2)}}_{u_3-u_{\varsigma(2)}}(\om_{\varsigma(2)}).
	\end{split}
\end{equation*}
%

In fact, we will not construct $\Cherry{10}$ in itself. As we will see in Lemma~\ref{lem:existence_ypsilon}, $\ty$ can be constructed as a continuous function in time that takes values in a space of distributions. On the other hand, we do not expect $\Cherry{10}$ to admit pointwise-in-time values; instead we expect it to exist as a proper space-time distribution, which resembles the situation discussed in~\cite[pp.~23--24 \& pp.~32--33]{mourrat_weber_xu_16} and~\cite{catellier_chouk_18,hairer_matetski_18}.

A particular variant of the root
$\begin{tikzpicture}[scale=0.7,baseline=-0.6ex]
	\node at (0,0) [root] (middle) {};
\end{tikzpicture}$
is the vertex
$\begin{tikzpicture}[scale=0.7,baseline=-0.6ex]
	\node at (0,0) [rootvar] (middle) {}; 
	\draw[fill] (0,0) circle (0.02);
\end{tikzpicture}$
which arises through applications of the resonant product. The vertex 
$\begin{tikzpicture}[scale=0.7,baseline=-0.6ex]
	\node at (0,0) [rootvar] (middle) {}; 
	\draw[fill] (0,0) circle (0.02);
\end{tikzpicture}$
relates the frequencies of the arrows that it joins through the $\sim$-relation defined in~\eqref{eq:definition_sim}, see also~\cite[(64)]{mourrat_weber_xu_16}.

Let us consider more complicated objects. We have the Wiener chaos decomposition
\begin{equation*}
	\tp=\PreThree{10}+\PreThree{20}+\PreCocktail{30}+(\PreCocktail{10}+\PreCocktail{20})+\PreCocktail{40}.
\end{equation*}
The contractions $\PreThreeloop{10}$ and $\PreThreeloop{20}$ that one might expect are absent in the renormalised enhancement due to our definition of $\ty=\vdiv\mcI[\ti\nabla\Phi_{\ti}]-\tl$. We express the third-order Wiener chaos term $\PreThree{10}$ as follows. For $j=1,2$,
\begin{equation*}
	\begin{split}
		&\hat{\PreThree{10}}(t,\om,j)=
		\begin{tikzpicture}[scale=0.7,baseline=-1.8ex]
			\node at (-0.6,0) [dot] (middle) {}; 
			\node at (-1.2,1.2) [vectorcircle] (aboveleft) {};
			\node at (0,1.2) [vectorcircle] (aboveright) {}; 
			\node at (0.6,0) [vectorcircle] (right) {}; 
			\node at (0,-1.2) [rootvar] (below) {}; 
			\node at (0,-1.6) () {\scriptsize $(t,\om)$};
			\node at (-2.05,1.6) [] {\scriptsize $(u_1,\om_1,m_1,j_1)$}; 
			\node at (0.85,1.6) [] {\scriptsize $(u_2,\om_2,m_2,j_2)$};
			\node at (1.3,0.4) [] {\scriptsize $(u_4,\om_4,m_4,j_4)$};
			\node at (-2.3,0) [] {\scriptsize $(u_3,\om_1+\om_2,j_3)$};
			\node at (0.5,-0.8) [] {\scriptsize $j$};
			\draw[fill] (0,-1.2) circle (0.02);
			\draw[scalarkernel] (middle) to (aboveleft);
			\draw[vectorkernel] (middle) to (aboveright); 
			\draw[vectorkernel] (below) to (right); 
			\draw[scalarkernel] (below) to (middle);
		\end{tikzpicture}\\
		&\defeq\sum_{\substack{\om_1,\om_2,\om_4\in\mbZ^{2}\\\om=\om_1+\om_2+\om_4}}\sum_{j_1,j_2,j_3,j_4=1}^{2}\sum_{m_1,m_2,m_4\in\mbZ^{2}}\int_{0}^{t}\dd u_3\int_{0}^{t}\dd W^{j_4}(u_4,m_4)\int_{0}^{u_4}\dd W^{j_1}(u_1,m_1)\int_{0}^{u_1}\dd W^{j_2}(u_2,m_2)\\
		&\quad\hat{\het}(u_1,\om_1-m_1)\hat{\het}(u_2,\om_2-m_2)\hat{\het}(u_4,\om_4-m_4)\sum_{\varsigma\in \Sigma(1,2,4)}H_{t-u_3}^{j_3}(\om_{\varsigma(1)}+\om_{\varsigma(2)})
		G^{j}(\om_{\varsigma(4)})H_{t-u_{\varsigma(4)}}^{j_{\varsigma(4)}}(\om_{\varsigma(4)})\\
		&\quad\times H_{u_3-u_{\varsigma(1)}}^{j_{\varsigma(1)}}(\om_{\varsigma(1)})G^{j_3}(\om_{\varsigma(2)})H_{u_3-u_{\varsigma(2)}}^{j_{\varsigma(2)}}(\om_{\varsigma(2)})\sum_{\substack{k,l\in\mbN_{-1}\\\abs{k-l}\leq1}}\varrho_k(\om_{\varsigma(1)}+\om_{\varsigma(2)})\varrho_l(\om_{\varsigma(4)}).
	\end{split}
\end{equation*}
The diagrams $\PreCocktail{10}$ and $\PreCocktail{20}$ may not exist in themselves, but the summed object $\PreCocktail{50}\defeq\PreCocktail{10}+\PreCocktail{20}$ does. Its iterated integral representation is given by
\begin{equation*}
	\begin{split}
		&\hat{\PreCocktail{50}}(t,\om,j)=
		\begin{tikzpicture}[scale=0.7,baseline=-1.8ex]
			\node at (-0.6,0) [dot] (middle) {}; 
			\node at (-1.2,1.2) [vectorcircle] (aboveleft) {}; 
			\node at (0.6,0) [dot] (right) {};
			\node at (0,-1.2) [rootvar] (below) {}; 
			\node at (0,-1.6) () {\scriptsize $(t,\om)$}; 
			\node at (0.5,-0.8) [] {\scriptsize $j$}; 
			\node at (-1.2,1.6) [] {\scriptsize $(u_1,\om_1,m_1,j_1)$}; 
			\node at (0.85,0.4) [] {\scriptsize $(u_2,\om_2,m_2,j_2)$}; 
			\node at (2.35,-0.15) [] {\scriptsize $(u_2,\om_4,-m_2,j_2)$};
			\node at (-2.3,0) [] {\scriptsize $(u_3,\om_1+\om_2,j_3)$};
			\draw[fill] (0,-1.2) circle (0.02);
			\draw[vectorkernel] (middle) to (aboveleft); 
			\draw[scalarkernel] (middle) to (right); 
			\draw[vectorkernel] (below) to (right); 
			\draw[scalarkernel] (below) to (middle);
		\end{tikzpicture}
		+
		\begin{tikzpicture}[scale=0.7,baseline=-1.8ex]
			\node at (-0.6,0) [dot] (middle) {}; 
			\node at (-1.2,1.2) [vectorcircle] (aboveleft) {}; 
			\node at (0.6,0) [dot] (right) {};
			\node at (0,-1.2) [rootvar] (below) {}; 
			\node at (0,-1.6) () {\scriptsize $(t,\om)$}; 
			\node at (-0.5,-0.8) [] {\scriptsize $j$}; 
			\node at (-1.2,1.6) [] {\scriptsize $(u_1,\om_1,m_1,j_1)$}; 
			\node at (0.85,0.4) [] {\scriptsize $(u_2,\om_2,m_2,j_2)$}; 
			\node at (2.35,-0.15) [] {\scriptsize $(u_2,\om_4,-m_2,j_2)$};
			\node at (-2.3,0) [] {\scriptsize $(u_3,\om_1+\om_2,j_3)$};
			\draw[fill] (0,-1.2) circle (0.02);
			\draw[vectorkernel] (middle) to (aboveleft); 
			\draw[scalarkernel] (middle) to (right); 
			\draw[scalarkernel] (below) to (right); 
			\draw[vectorkernel] (below) to (middle);
		\end{tikzpicture}\\
		%
		&\defeq\sum_{\substack{\om_1,\om_2,\om_4\in\mbZ^{2}\\\om=\om_1+\om_2+\om_4\\(\om_1+\om_2)\sim\om_4}}\sum_{j_1,j_2,j_3=1}^{2}\sum_{m_1,m_2\in\mbZ^{2}}\int_{0}^{t}\dd u_3\int_{0}^{u_3}\dd u_2\int_{0}^{u_3}\dd W^{j_1}(u_1,m_1)\\
		&\quad\hat{\het}(u_1,\om_1-m_1)\hat{\het}(u_2,\om_2-m_2)\hat{\het}(u_2,\om_4+m_2)(G^{j}(\om_4)+G^{j}(\om_1+\om_2))\\
		&\quad\times H^{j_3}_{t-u_3}(\om_1+\om_2)H^{j_2}_{t-u_2}(\om_4)G^{j_3}(\om_1)H^{j_1}_{u_3-u_1}(\om_1)H^{j_2}_{u_3-u_2}(\om_2).
	\end{split}
\end{equation*}

We will show in Lemma~\ref{lem:elliptic_difference} that the resulting factor $G^{j}(\om_4)+G^{j}(\om_1+\om_2)$ has better decay in $\om_4$ than $G^{j}(\om_4)$. This is due to the symmetry of $G^{j}$, which allows us to write $G^{j}(\om_1+\om_2)+G^{j}(\om_4)=G^{j}(\om-\om_4)-G^{j}(-\om_4)$. The improved decay leads to the well-posedness of $\PreCocktail{50}$ and is a higher-dimensional analogue of the product rule discussed in~\eqref{eq:product_rule}.
\begin{remark}
	One could simplify the contraction by using the identity
	\begin{equation}\label{eq:contraction_identity}
		\sum_{m_2\in\mbZ^{2}}\hat{\het}(u_2,\om_2-m_2)\hat{\het}(u_2,\om_4+m_2)=\hat{\het^{2}}(u_2,\om_2+\om_4).
	\end{equation}
	This idea would allow us to derive bounds in terms of $\norm{\het^2}_{C_T\mcH^{2}}$ rather than $\norm{\het}_{C_T\mcH^{2}}^{2}$. However, \eqref{eq:contraction_identity} is no longer applicable in our prelimiting enhancement due to the cut-off $\varphi(\delta m_2)$ (cf.~\eqref{eq:lolli_approximation_example}). Instead we use direct estimates that do not rely on~\eqref{eq:contraction_identity}.
\end{remark}
The remaining diagrams $\PreThree{20}$, $\PreCocktail{30}$ and $\PreCocktail{40}$ are similar to the ones given above.
\begin{details}
	\begin{equation*}
		\begin{split}
			&\hat{\PreThree{20}}(t,\om,j)=
			\begin{tikzpicture}[scale=0.7,baseline=-1.8ex]
				\node at (-0.6,0) [dot] (middle) {}; 
				\node at (-1.2,1.2) [vectorcircle] (aboveleft) {};
				\node at (0,1.2) [vectorcircle] (aboveright) {}; 
				\node at (0.6,0) [vectorcircle] (right) {}; 
				\node at (0,-1.2) [rootvar] (below) {}; 
				\node at (0,-1.6) () {\scriptsize $(t,\om)$};
				\node at (-2.05,1.6) [] {\scriptsize $(u_1,\om_1,m_1,j_1)$}; 
				\node at (0.85,1.6) [] {\scriptsize $(u_2,\om_2,m_2,j_2)$};
				\node at (1.3,0.4) [] {\scriptsize $(u_4,\om_4,m_4,j_4)$};
				\node at (-2.3,0) [] {\scriptsize $(u_3,\om_1+\om_2,j_3)$};
				\node at (-0.5,-0.8) [] {\scriptsize $j$};
				\draw[fill] (0,-1.2) circle (0.02);
				\draw[scalarkernel] (middle) to (aboveleft);
				\draw[vectorkernel] (middle) to (aboveright); 
				\draw[scalarkernel] (below) to (right); 
				\draw[vectorkernel] (below) to (middle);
			\end{tikzpicture}\\
			&\defeq\sum_{\substack{\om_1,\om_2,\om_4\in\mbZ^{2}\\\om=\om_1+\om_2+\om_4}}\sum_{j_1,j_2,j_3,j_4=1}^{2}\sum_{m_1,m_2,m_4\in\mbZ^{2}}\int_{0}^{t}\dd u_3\int_{0}^{t}\dd W^{j_4}(u_4,m_4)\int_{0}^{u_4}\dd W^{j_1}(u_1,m_1)\int_{0}^{u_1}\dd W^{j_2}(u_2,m_2)\\
			&\quad\hat{\het}(u_1,\om_1-m_1)\hat{\het}(u_2,\om_2-m_2)\hat{\het}(u_4,\om_4-m_4)\sum_{\varsigma\in\Sigma(1,2,4)}G^{j}(\om_{\varsigma(1)}+\om_{\varsigma(2)})H^{j_3}_{t-u_3}(\om_{\varsigma(1)}+\om_{\varsigma(2)})\\
			&\quad\times H^{j_{\varsigma(4)}}_{t-u_{\varsigma(4)}}(\om_{\varsigma(4)})H^{j_{\varsigma(1)}}_{u_3-u_{\varsigma(1)}}(\om_{\varsigma(1)})G^{j_3}(\om_{\varsigma(2)})H^{j_{\varsigma(2)}}_{u_3-u_{\varsigma(2)}}(\om_{\varsigma(2)})\sum_{\substack{k,l\in\mbN_{-1}\\\abs{k-l}\leq 1}}\varrho_k(\om_{\varsigma(1)}+\om_{\varsigma(2)})\varrho_l(\om_{\varsigma(4)}),
		\end{split}
	\end{equation*}
	\begin{equation*}
		\begin{split}
			&\hat{\PreCocktail{30}}(t,\om,j)=
			\begin{tikzpicture}[scale=0.7,baseline=-1.8ex]
				\node at (-0.6,0) [dot] (middle) {}; 
				\node at (-1.2,1.2) [vectorcircle] (aboveleft) {}; 
				\node at (0.6,0) [dot] (right) {};
				\node at (0,-1.2) [rootvar] (below) {}; 
				\node at (0,-1.6) () {\scriptsize $(t,\om)$}; 
				\node at (0.5,-0.8) [] {\scriptsize $j$}; 
				\node at (-1.2,1.6) [] {\scriptsize $(u_1,\om_1,m_1,j_1)$}; 
				\node at (0.85,0.4) [] {\scriptsize $(u_2,\om_2,m_2,j_2)$}; 
				\node at (2.35,-0.15) [] {\scriptsize $(u_2,\om_4,-m_2,j_2)$};
				\node at (-2.3,0) [] {\scriptsize $(u_3,\om_1+\om_2,j_3)$};
				\draw[fill] (0,-1.2) circle (0.02);
				\draw[scalarkernel] (middle) to (aboveleft); 
				\draw[vectorkernel] (middle) to (right); 
				\draw[vectorkernel] (below) to (right); 
				\draw[scalarkernel] (below) to (middle);
			\end{tikzpicture}\\
			%
			&\defeq\sum_{\substack{\om_1,\om_2,\om_4\in\mbZ^{2}\\\om=\om_1+\om_2+\om_4\\(\om_1+\om_2)\sim\om_4}}\sum_{j_1,j_2,j_3=1}^{2}\sum_{m_1,m_2\in\mbZ^{2}}\int_{0}^{t}\dd u_3\int_{0}^{u_3}\dd u_2\int_{0}^{u_3}\dd W^{j_1}(u_1,m_1)\\
			&\quad\hat{\het}(u_1,\om_1-m_1)\hat{\het}(u_2,\om_2-m_2)\hat{\het}(u_2,\om_4+m_2)\\
			&\quad\times H_{t-u_3}^{j_3}(\om_1+\om_2)G^{j}(\om_4)H_{t-u_2}^{j_2}(\om_4)H_{u_3-u_1}^{j_1}(\om_1)G^{j_3}(\om_2)H_{u_3-u_2}^{j_2}(\om_2),
		\end{split}
	\end{equation*}
	\begin{equation*}
		\begin{split}
			&\hat{\PreCocktail{40}}(t,\om,j)
			=\begin{tikzpicture}[scale=0.7,baseline=-1.8ex]
				\node at (-0.6,0) [dot] (middle) {}; 
				\node at (-1.2,1.2) [vectorcircle] (aboveleft) {}; 
				\node at (0.6,0) [dot] (right) {};
				\node at (0,-1.2) [rootvar] (below) {}; 
				\node at (0,-1.6) () {\scriptsize $(t,\om)$}; 
				\node at (-0.5,-0.8) [] {\scriptsize $j$}; 
				\node at (-1.2,1.6) [] {\scriptsize $(u_1,\om_1,m_1,j_1)$}; 
				\node at (0.85,0.4) [] {\scriptsize $(u_2,\om_2,m_2,j_2)$}; 
				\node at (2.35,-0.15) [] {\scriptsize $(u_2,\om_4,-m_2,j_2)$};
				\node at (-2.3,0) [] {\scriptsize $(u_3,\om_1+\om_2,j_3)$};
				\draw[fill] (0,-1.2) circle (0.02);
				\draw[scalarkernel] (middle) to (aboveleft); 
				\draw[vectorkernel] (middle) to (right); 
				\draw[scalarkernel] (below) to (right); 
				\draw[vectorkernel] (below) to (middle);
			\end{tikzpicture}\\
			%
			&\defeq\sum_{\substack{\om_1,\om_2,\om_4\in\mbZ^{2}\\\om=\om_1+\om_2+\om_4\\(\om_1+\om_2)\sim\om_4}}\sum_{j_1,j_2,j_3=1}^{2}\sum_{m_1,m_2\in\mbZ^{2}}\int_{0}^{t}\dd u_3\int_{0}^{u_3}\dd u_2\int_{0}^{u_3}\dd W^{j_1}(u_1,m_1)\\
			&\quad\hat{\het}(u_1,\om_1-m_1)\hat{\het}(u_2,\om_2-m_2)\hat{\het}(u_2,\om_4+m_2)\\
			&\quad\times G^{j}(\om_1+\om_2)H^{j_3}_{t-u_3}(\om_1+\om_2)H^{j_2}_{t-u_2}(\om_4)H^{j_1}_{u_3-u_1}(\om_1)G^{j_3}(\om_2)H^{j_2}_{u_3-u_2}(\om_2).
		\end{split}
	\end{equation*}
\end{details}

We extend our graphical rules to incorporate the operator $\partial_{l}\mcI$ for $l=1,2$. An indexed black arrow pointing at a scalar object
$\!\!\begin{tikzpicture}[scale=0.7,baseline=0ex]
	\node at (0,0) [] (middle){};
	\node at (1.2,0) [] (right){};
	\node at (1.75,0)[]{\scriptsize $(u_k,\om_k)$};
	\node at (-0.3,0) [] {\scriptsize $(t,\om)$};
	\node at (0.55,0.3) [] {\scriptsize $l$}; 
	\draw[scalarkernel] (middle) to (right); 
\end{tikzpicture}\!\!$
is associated to the multiplier $H^{l}_{t-u_k}(\om_k)$. On the other hand, a doubly-indexed, highlighted arrow pointing at a scalar object
$\!\!\begin{tikzpicture}[scale=0.7,baseline=0.6ex]
	\node at (0,0) [] (middle){};
	\node at (1.5,0) [] (right){};
	\node at (2.05,0)[]{\scriptsize $(u_k,\om_k)$};
	\node at (-0.3,0) [] {\scriptsize $(t,\om)$};
	\node at (0.65,0.3) [] {\scriptsize $l,j$}; 
	\draw[vectorkernel] (middle) to (right); 
\end{tikzpicture}\!\!$
is associated to $G^{j}(\om_k)H^{l}_{t-u_k}(\om_k)$.

The remaining object in the enhancement is $\tc_{k,j}$ for $k,j=1,2$. We consider the Wiener chaos decomposition
\begin{equation*}
	\tc_{k,j}=\Checkmark{10}(k,j)+\Checkmark{20}(k,j)+(\Triangle{1}(k,j)+\Triangle{2}(k,j)).
\end{equation*}
The first term is given by
\begin{equation*}
	\begin{split}
		&\hat{\Checkmark{10}}(t,\om,k,j)=
		\begin{tikzpicture}[scale=0.7,baseline=-1.8ex]
			\node at (-0.6,0) [dot] (middle) {}; 
			\node at (-1.2,1.2) [vectorcircle] (aboveleft) {};
			\node at (0.6,0) [vectorcircle] (right) {}; 
			\node at (0,-1.2) [rootvar] (below) {}; 
			\node at (0,-1.6) () {\scriptsize $(t,\om)$};
			\node at (-1.2,1.6) [] {\scriptsize $(u_1,\om_1,m_1,j_1)$}; 
			\node at (0.9,0.4) [] {\scriptsize $(u_2,\om_2,m_2,j_2)$};
			\node at (0.5,-0.8) [] {\scriptsize $j$};
			\node at (-0.5,-0.8) [] {\scriptsize $k$};
			\node at (-1.5,0) [] {\scriptsize $(u_3,\om_1)$};
			\draw[fill] (0,-1.2) circle (0.02);
			\draw[scalarkernel] (middle) to (aboveleft);
			\draw[vectorkernel] (below) to (right); 
			\draw[scalarkernel] (below) to (middle); 
		\end{tikzpicture}\\
		&\defeq\sum_{\substack{\om_1,\om_2\in\mbZ^{2}\\\om=\om_1+\om_2\\\om_1\sim\om_2}}\sum_{j_1,j_2=1}^{2}\sum_{m_1,m_2\in\mbZ^{2}}\int_{0}^{t}\dd u_3\int_{0}^{t}\dd W^{j_1}(u_1,m_1)\int_{0}^{u_1}\dd W^{j_2}(u_2,m_2)\\
		&\quad\hat{\het}(u_1,\om_1-m_1)\hat{\het}(u_2,\om_2-m_2)\sum_{\varsigma\in\Sigma(1,2)}H_{t-u_3}^{k}(\om_{\varsigma(1)})H_{u_3-u_{\varsigma(1)}}^{j_{\varsigma(1)}}(\om_{\varsigma(1)})G^{j}(\om_{\varsigma(2)})H_{t-u_{\varsigma(2)}}^{j_{\varsigma(2)}}(\om_{\varsigma(2)})
	\end{split}
\end{equation*}
and the second term $\Checkmark{20}$ is again similar.
\begin{details}
	\begin{equation*}
		\begin{split}
			&\hat{\Checkmark{20}}(t,\om,k,j)=
			\begin{tikzpicture}[scale=0.7,baseline=-1.8ex]
				\node at (-0.6,0) [dot] (middle) {}; 
				\node at (-1.2,1.2) [vectorcircle] (aboveleft) {};
				\node at (0.6,0) [vectorcircle] (right) {}; 
				\node at (0,-1.2) [rootvar] (below) {}; 
				\node at (0,-1.6) () {\scriptsize $(t,\om)$};
				\node at (-1.2,1.6) [] {\scriptsize $(u_1,\om_1,m_1,j_1)$}; 
				\node at (0.9,0.4) [] {\scriptsize $(u_2,\om_2,m_2,j_2)$};
				\node at (-0.8,-0.8) [] {\scriptsize $k,j$};
				\node at (-1.5,0) [] {\scriptsize $(u_3,\om_1)$};
				\draw[fill] (0,-1.2) circle (0.02);
				\draw[scalarkernel] (middle) to (aboveleft);
				\draw[scalarkernel] (below) to (right); 
				\draw[vectorkernel] (below) to (middle); 
			\end{tikzpicture}\\
			&\defeq\sum_{\substack{\om_1,\om_2\in\mbZ^{2}\\\om=\om_1+\om_2\\\om_1\sim\om_2}}\sum_{j_1,j_2=1}^{2}\sum_{m_1,m_2\in\mbZ^{2}}\int_{0}^{t}\dd u_3\int_{0}^{t}\dd W^{j_1}(u_1,m_1)\int_{0}^{u_1}\dd W^{j_2}(u_2,m_2)\\
			&\quad\hat{\het}(u_1,\om_1-m_1)\hat{\het}(u_2,\om_2-m_2)\sum_{\varsigma\in\Sigma(1,2)}H^k_{t-u_3}(\om_{\varsigma(1)})G^j(\om_{\varsigma(1)})H^{j_{\varsigma(1)}}_{u_3-u_{\varsigma(1)}}(\om_{\varsigma(1)})H^{j_{\varsigma(2)}}_{t-u_{\varsigma(2)}}(\om_{\varsigma(2)}).
		\end{split}
	\end{equation*}
\end{details}
We consider the contractions as a summed object $\Triangle{3}\defeq\Triangle{1}+\Triangle{2}$. We obtain
\begin{equation*}
	\begin{split}
		&\hat{\Triangle{3}}(t,\om,k,j)=
		\begin{tikzpicture}[scale=0.7,baseline=-1.8ex]
			\node at (-0.6,0) [dot] (middle) {}; 
			\node at (0.6,0) [dot] (right) {}; 
			\node at (0,-1.2) [rootvar] (below) {}; 
			\node at (0,-1.6) () {\scriptsize $(t,\om)$};
			\node at (1.95,0.4) [] {\scriptsize $(u_2,\om_1,-m_2,j_2)$}; 
			\node at (2.2,-0.15) [] {\scriptsize $(u_2,\om_2,m_2,j_2)$};
			\node at (0.5,-0.8) [] {\scriptsize $j$};
			\node at (-1.05,0.4) [] {\scriptsize $(u_3,\om_1)$}; 
			\node at (-0.5,-0.8) [] {\scriptsize $k$};
			\draw[fill] (0,-1.2) circle (0.02);
			\draw[scalarkernel] (middle) to (right); 
			\draw[vectorkernel] (below) to (right); 
			\draw[scalarkernel] (below) to (middle);
		\end{tikzpicture}
		+
		\begin{tikzpicture}[scale=0.7,baseline=-1.8ex]
			\node at (-0.6,0) [dot] (middle) {}; 
			\node at (0.6,0) [dot] (right) {}; 
			\node at (0,-1.2) [rootvar] (below) {}; 
			\node at (0,-1.6) () {\scriptsize $(t,\om)$};
			\node at (1.95,0.4) [] {\scriptsize $(u_2,\om_1,-m_2,j_2)$}; 
			\node at (2.2,-0.15) [] {\scriptsize $(u_2,\om_2,m_2,j_2)$};
			\node at (-0.8,-0.8) [] {\scriptsize $k,j$};
			\node at (-1.05,0.4) [] {\scriptsize $(u_3,\om_1)$}; 
			\draw[fill] (0,-1.2) circle (0.02);
			\draw[scalarkernel] (middle) to (right);
			\draw[scalarkernel] (below) to (right); 
			\draw[vectorkernel] (below) to (middle); 
		\end{tikzpicture}\\
		&\defeq\sum_{\substack{\om_1,\om_2\in\mbZ^{2}\\\om=\om_1+\om_2\\\om_1\sim\om_2}}\sum_{j_2=1}^{2}\sum_{m_2\in\mbZ^{2}}\int_{0}^{t}\dd u_3\int_{0}^{u_3}\dd u_2\hat{\het}(u_2,\om_1+m_2)\hat{\het}(u_2,\om_2-m_2)(G^{j}(\om_2)+G^j(\om_1))\\[-20pt]
		&\multiquad[16]\times H_{t-u_3}^{k}(\om_1)H_{u_3-u_2}^{j_2}(\om_1)H_{t-u_2}^{j_2}(\om_2).
	\end{split}
\end{equation*}

We can define approximate diagrams as in~\eqref{eq:lolli_approximation_example} by multiplying the cut-off $\varphi(\delta m_k)$ (cf.~\eqref{eq:def_cut_off}) to each instance of the noise
$\begin{tikzpicture}[scale=0.7,baseline=-0.6ex]
	\node at (0,0) [vectorcircle] (middle) {};
	\node at (1.61,0)[]{\scriptsize $(u_k,\om_k,m_k,j_k)$};
\end{tikzpicture}\!\!$.
\begin{details}
	For instance,
	\begin{equation*}
		\begin{split}
			&\hat{\ty^{\delta}}(t,\om)\\
			&\defeq\sum_{\substack{\om_1,\om_2\in\mbZ^{2}\\\om=\om_1+\om_2}}\sum_{j_1,j_2,j_3=1}^{2}\sum_{m_1,m_2\in\mbZ^{2}}\int_{0}^{t}\dd u_3\int_{0}^{u_3}\dd W^{j_2}(u_2,m_2)\int_{0}^{u_2}\dd W^{j_1}(u_1,m_1)\varphi(\delta m_1)\varphi(\delta m_2)\\
			&\quad\hat{\het}(u_{1},\om_{1}-m_{1})\hat{\het}(u_{2},\om_{2}-m_{2})\sum_{\varsigma\in\Sigma(1,2)}H^{j_3}_{t-u_3}(\om_{\varsigma(1)}+\om_{\varsigma(2)})H^{j_{\varsigma(1)}}_{u_3-u_{\varsigma(1)}}(\om_{\varsigma(1)})G^{j_3}(\om_{\varsigma(2)})H^{j_{\varsigma(2)}}_{u_3-u_{\varsigma(2)}}(\om_{\varsigma(2)}).
		\end{split}
	\end{equation*}
\end{details}
In general, we denote the regularization of a diagram by a superscript $\delta$. The canonical enhancement $\mbX^{\delta}_{\can}=(\ti^{\delta},\ty^{\delta}_{\can},\tp^{\delta}_{\can},\tc^{\delta})$ is then built from regularized noise terms, but retains the diverging sequences that are removed in the renormalised enhancement $\mbX$. Repeating~\eqref{eq:diagram_Ito}, we may consider the decomposition of the diagram with cut-off,
\begin{equation*}
	\ty^{\delta}_{\can}=\vdiv\mcI[\ti^{\delta}\nabla\Phi_{\ti^{\delta}}]=\ty^{\delta}+\tl^{\delta},
\end{equation*}
where $\tl^{\delta}=\mbE[\vdiv\mcI[\ti^{\delta}\nabla\Phi_{\ti^{\delta}}]]$. In addition to $\ty^{\delta}$, we also have to control the mean,
\begin{equation*}
	\begin{split}
		&\hat{\tl^{\delta}}(t,\om)=
		\begin{tikzpicture}[scale=0.7,baseline=-4.2ex]
			\node at (0,0) [dot] (middle) {}; 
			\node at (0,1.2) [dot] (above) {}; 
			\node at (0,-1.2) [root] (below) {}; 
			\node at (0,-1.6) () {\scriptsize $(t,\om)$};
			\node at (-1.44,1.6) [] {\scriptsize $(u_1,\om_1,m_1,j_1)$}; 
			\node at (1.61,1.6) [] {\scriptsize $(u_1,\om_2,-m_1,j_1)$};
			\node at (1.3,0) {\scriptsize $(u_3,\om,j_3)$}; 
			\draw[vectorkernel, bend right=60] (middle) to (above); 
			\draw[scalarkernel, bend left=60] (middle) to (above);
			\draw[scalarkernel] (below) to (middle); 
		\end{tikzpicture}\\
		&\defeq\sum_{\substack{\om_1,\om_2\in\mbZ^{2}\\\om=\om_1+\om_2}}\sum_{j_1,j_3=1}^{2}\sum_{m_1\in\mbZ^{2}}\int_{0}^{t}\dd u_3\int_{0}^{u_3}\dd u_1\hat{\het}(u_1,\om_1-m_1)\hat{\het}(u_1,\om_2+m_1)\abs{\varphi(\delta m_1)}^{2}\\[-16pt]
		&\multiquad[17]\times H^{j_3}_{t-u_3}(\om)H^{j_1}_{u_3-u_1}(\om_1)G^{j_3}(\om_2)H^{j_1}_{u_3-u_1}(\om_2).
	\end{split}
\end{equation*}
Including $\tl^{\delta}$ in $\ty^{\delta}_{\can}$ generates additional terms in the decomposition of $\tp^{\delta}_{\can}$,
\begin{equation*}
	\tp^{\delta}_{\can}=\tp^{\delta}+{\PreThreeloop{10}\!}^{\delta}+{\PreThreeloop{20}\!}^{\delta}.
\end{equation*}
The diagram ${\PreThreeloop{10}\!}^{\delta}$ is given by
\begin{equation*}
	\begin{split}
		&\hat{{\PreThreeloop{10}\!}^{\delta}}(t,\om,j)= 
		\begin{tikzpicture}[scale=0.7,baseline=-1.8ex]
			\node at (-0.6,0) [dot] (middle) {}; 
			\node at (-0.6,1.2) [dot] (abovemiddle) {};
			\node at (0.6,0) [vectorcircle] (right) {}; 
			\node at (0,-1.2) [rootvar] (below) {}; 
			\node at (0,-1.6) () {\scriptsize $(t,\om)$};
			\node at (-2.04,1.6) [] {\scriptsize $(u_1,\om_1,m_1,j_1)$}; 
			\node at (1.01,1.6) [] {\scriptsize $(u_1,\om_2,-m_1,j_1)$};
			\node at (1.4,0.4) [] {\scriptsize $(u_4,\om_4,m_4,j_4)$};
			\node at (-2.5,0) [] {\scriptsize $(u_3,\om_1+\om_2,j_3)$};
			\node at (0.5,-0.8) [] {\scriptsize $j$};
			\draw[fill] (0,-1.2) circle (0.02);
			\draw[scalarkernel, bend left=60] (middle) to (abovemiddle);
			\draw[vectorkernel, bend right=60] (middle) to (abovemiddle); 
			\draw[vectorkernel] (below) to (right); 
			\draw[scalarkernel] (below) to (middle);
		\end{tikzpicture}\\
		&\defeq\sum_{\substack{\om_1,\om_2,\om_4\in\mbZ^{2}\\\om=\om_1+\om_2+\om_4\\(\om_1+\om_2)\sim\om_4}}\sum_{j_1,j_3,j_4=1}^{2}\sum_{m_1,m_4\in\mbZ^{2}}\int_{0}^{t}\dd W^{j_4}(u_4,m_4)\int_{0}^{t}\dd u_3\int_{0}^{u_3}\dd u_1\\
		&\quad\hat{\het}(u_1,\om_1-m_1)\hat{\het}(u_1,\om_2+m_1)\hat{\het}(u_4,\om_4-m_4)\abs{\varphi(\delta m_1)}^{2}\varphi(\delta m_4)\\
		&\quad\times H_{t-u_3}^{j_3}(\om_1+\om_2)G^{j}(\om_4)H_{t-u_4}^{j_4}(\om_4)H_{u_3-u_1}^{j_1}(\om_1)G^{j_3}(\om_2)H_{u_3-u_1}^{j_1}(\om_2)
	\end{split}
\end{equation*}
and ${\PreThreeloop{20}\!}^{\delta}$ is again similar. Here, we have implicitly changed our graphical rules to include the cut-off.
\begin{details}
	\begin{equation*}
		\begin{split}
			&\hat{{\PreThreeloop{20}\!}^{\delta}}(t,\om,j)=
			\begin{tikzpicture}[scale=0.7,baseline=-1.8ex]
				\node at (-0.6,0) [dot] (middle) {}; 
				\node at (-0.6,1.2) [dot] (abovemiddle) {};
				\node at (0.6,0) [vectorcircle] (right) {}; 
				\node at (0,-1.2) [rootvar] (below) {}; 
				\node at (0,-1.6) () {\scriptsize $(t,\om)$};
				\node at (-2.04,1.6) [] {\scriptsize $(u_1,\om_1,m_1,j_1)$}; 
				\node at (1.01,1.6) [] {\scriptsize $(u_1,\om_2,-m_1,j_1)$};
				\node at (1.4,0.4) [] {\scriptsize $(u_4,\om_4,m_4,j_4)$};
				\node at (-2.5,0) [] {\scriptsize $(u_3,\om_1+\om_2,j_3)$};
				\node at (-0.5,-0.8) [] {\scriptsize $j$};
				\draw[fill] (0,-1.2) circle (0.02);
				\draw[scalarkernel, bend left=60] (middle) to (abovemiddle);
				\draw[vectorkernel, bend right=60] (middle) to (abovemiddle); 
				\draw[scalarkernel] (below) to (right); 
				\draw[vectorkernel] (below) to (middle);
			\end{tikzpicture}\\
			&\defeq\sum_{\substack{\om_1,\om_2,\om_4\in\mbZ^{2}\\\om=\om_1+\om_2+\om_4\\(\om_1+\om_2)\sim\om_4}}\sum_{j_1,j_3,j_4=1}^{2}\sum_{m_1,m_4\in\mbZ^{2}}\int_{0}^{t}\dd W^{j_4}(u_4,m_4)\int_{0}^{t}\dd u_3\int_{0}^{u_3}\dd u_1\\
			&\quad\hat{\het}(u_1,\om_1-m_1)\hat{\het}(u_1,\om_2+m_1)\hat{\het}(u_4,\om_4-m_4)\abs{\varphi(\delta m_1)}^{2}\varphi(\delta m_4)\\
			&\quad\times G^{j}(\om_1+\om_2)H_{t-u_3}^{j_3}(\om_1+\om_2)H_{t-u_4}^{j_4}(\om_4)H_{u_3-u_1}^{j_1}(\om_1)G^{j_3}(\om_2)H_{u_3-u_1}^{j_1}(\om_2).
		\end{split}
	\end{equation*}
\end{details}
\subsection{Existence and Regularity of Stochastic Objects}
In this subsection we define the notion of an iterated It\^{o} integral (Definition~\ref{def:iterated_Ito}), introduce It\^{o}'s isometry (Lemma~\ref{lem:Ito_isometry}) and Nelson's estimate (Lemma~\ref{lem:Nelson_estimate}). We further define an iterated It\^{o} integral with mollification and heterogeneity (Definition~\ref{def:iterated_Ito_mollification_heterogeneity}), discuss in Lemma~\ref{lem:isometry_argument} how those can be controlled by passing to real space and then introduce an interpolation argument to show that mollified iterated It\^{o} integrals converge to their un-mollified counterparts (Lemma~\ref{lem:isometry_interpolation}). Finally we derive a general criterion of existence for stochastic objects taking values in Besov spaces (Lemma~\ref{lem:existence_criterion}). See also~\cite{mourrat_weber_xu_16,gubinelli_perkowski_17} for different instances of the same arguments.

Let $n\in\mbN$, $D\subset(0,\infty)^{n}$ and $\phi\in L^2(D\times\mbT^{2n}\times\{1,2\}^{n};\mbC)$. We define the spatial Fourier transform of $\phi$ by
\begin{equation*}
	\hat{\phi}(u_1,\om_1,j_1,\ldots,u_n,\om_n,j_n)\defeq\int_{(\mbT^2)^n}\euler^{-2\uppi\upi(\inner{\om_1}{x_1}+\ldots+\inner{\om_n}{x_n})}\phi(u_1,x_1,j_1,\ldots,u_n,x_n,j_n)\dd x_1\ldots\dd x_n.
\end{equation*}
We can now define iterated It\^{o} integrals.
\begin{definition}[Iterated It\^{o} integral]\label{def:iterated_Ito}
	Let $n\in\mbN$ and let 
	\begin{equation*}
		(0,\infty)^{n}_{>}\defeq\{(u_{1},\ldots,u_{n})\in(0,\infty)^{n}:u_{1}>u_{2}>\ldots>u_{n}\}.
	\end{equation*}
	We define the iterated It\^{o} integral acting on $\phi\in L^2((0,\infty)^n_{>}\times\mbT^{2n}\times\{1,2\}^{n};\mbC)$ by
	\begin{equation*}
		\begin{split}
			I^{n}(\phi)&\defeq\sum_{\om_1,\ldots,\om_n\in\mbZ^2}\sum_{j_1,\ldots,j_n=1}^{2}\int_{0}^{\infty}\dd W^{j_1}(u_1,\om_1)\ldots\int_{0}^{u_{n-1}}\dd W^{j_n}(u_n,\om_n)\hat{\phi}(u_1,-\om_1,j_1,\ldots,u_n,-\om_n,j_n).
		\end{split}
	\end{equation*}
\end{definition}
We can identify second moments of iterated It\^{o} integrals by a version of It\^{o}'s isometry.
\begin{lemma}[It\^{o}'s isometry]\label{lem:Ito_isometry}
	Let $n\in\mbN$ and $\phi\in L^2((0,\infty)^n_{>}\times\mbT^{2n}\times\{1,2\}^{n};\mbC)$, then
	\begin{equation*}
		\begin{split}
			\mbE[\abs{I^{n}(\phi)}^{2}]&=\sum_{\om_1,\ldots,\om_n\in\mbZ^2}\sum_{j_1,\ldots,j_n=1}^{2}\int_{0}^{\infty}\dd u_{1}\ldots\int_{0}^{u_{n-1}}\dd u_{n}\abs{\hat{\phi}(u_1,-\om_1,j_1,\ldots,u_n,-\om_n,j_n)}^{2}\\
			&=\norm{\phi}_{L^2((0,\infty)^n_{>}\times\mbT^{2n}\times\{1,2\}^{n};\mbC)}^{2}.
		\end{split}
	\end{equation*}
\end{lemma}
\begin{proof}
	The claim follows by It\^{o}'s isometry, applied inductively to each stochastic integral.
\end{proof}
The following result, Nelson's estimate, allows us to bound $p\textnormal{th}$-moments of iterated It\^{o} integrals by their second moments.
\begin{lemma}[Nelson's estimate]\label{lem:Nelson_estimate}
	Let $n\in\mbN$ and $p\in[2,\infty)$. Then there exists a $C=C(n,p)>0$ such that for any $\phi\in L^2((0,\infty)^{n}_{>}\times\mbT^{2n}\times\{1,2\}^{n};\mbC)$,
	\begin{equation*}
		\mbE[\abs{I^n(\phi)}^p]^{1/p}\leq C \mbE[\abs{I^n(\phi)}^2]^{1/2}.
	\end{equation*}
\end{lemma}
\begin{proof}
	For a proof, see~\cite{nualart_06,mourrat_weber_xu_16}.
\end{proof}
\begin{details}
	We prove Lemma~\ref{lem:Nelson_estimate} as in~\cite{mourrat_weber_xu_16}. We first establish the Burkholder--Davis--Gundy inequality for complex, continuous, local martingales.
	\begin{lemma}\label{lem:complex_Burkholder_Davis_Gundy_inequality}
		Let $p\in[1,\infty)$, then there exists some $C>0$ such that for any $T>0$ and complex, continuous, local martingale $M$ started from $0$,
		\begin{equation*}
			\mbE[\sup_{0\leq t\leq T}\lvert M_t\rvert^p]\leq C\mbE[[M,\overline{M}]_T^{p/2}].
		\end{equation*}
	\end{lemma}
	\begin{proof}
		We decompose $M=\Re M+\upi\Im M$ and write
		\begin{equation*}
			\mbE[\sup_{0\leq t\leq T}\lvert M_t\rvert^p]=\mbE[\sup_{0\leq t\leq T}(\lvert\Re M_t\rvert^2+\lvert\Im M_t\rvert^2)^{p/2}].
		\end{equation*}
		By the two-dimensional Burkholder--Davis--Gundy inequality~\cite[Thm.~14.6]{friz_victoir_10}, there exists some $C=C(p) >0$ such that
		\begin{equation*}
			\mbE[\sup_{0\leq t\leq T}(\lvert\Re M_t\rvert^2+\lvert\Im M_t\rvert^2)^{p/2}]\leq C\mbE[([\Re M]_T+[\Im M]_T)^{p/2}]=C\mbE[[M]_T^{p/2}].
		\end{equation*}
	\end{proof}
	We can now prove Nelson's estimate.
	\begin{proof}[Proof of Lemma~\ref{lem:Nelson_estimate}]
		We define 
		\begin{equation*}
			\begin{split}
				I^{n}(\phi,t)&\defeq\sum_{\om_1,\ldots,\om_n\in\mbZ^2}\sum_{j_1,\ldots,j_n=1}^{2}\int_{0}^{t}\dd W^{j_1}(u_1,\om_1)\int_{0}^{u_1}\dd W^{j_2}(u_2,\om_2)\ldots\int_{0}^{u_{n-1}}\dd W^{j_n}(u_n,\om_n)\\
				&\quad\hat{\phi}(u_1,-\om_1,j_1,\ldots,u_n,-\om_n,j_n),
			\end{split}
		\end{equation*}
		and claim that this is a martingale in $t$. Let $\delta>0$. Indeed the claim is clear for the truncation
		\begin{equation*}
			\begin{split}
				I^{n}_{\delta}(\phi,t)&\defeq\sum_{\substack{\om_1,\ldots,\om_n\in\mbZ^{2}\\\abs{\om_1}\leq \delta^{-1},\ldots,\abs{\om_n}\leq \delta^{-1}}}\sum_{j_1,\ldots,j_n=1}^{2}\int_{0}^{t}\dd W^{j_1}(u_1,\om_1)\int_{0}^{u_1}\dd W^{j_2}(u_2,\om_2)\ldots\int_{0}^{u_{n-1}}\dd W^{j_n}(u_n,\om_n)\\
				&\quad\hat{\phi}(u_1,-\om_1,j_1,\ldots,u_n,-\om_n,j_n),
			\end{split}
		\end{equation*}
		and it suffices to show $I^{n}_{\delta}(\phi,t)\to I^{n}(\phi,t)$ in $L^2(\mbP)$ as $\delta\to0$. We apply It\^{o}'s isometry,
		\begin{equation*}
			\begin{split}
				&\mbE[(I^{n}(\phi,t)-I^{n}_{\delta}(\phi,t))\overline{(I^{n}(\phi,t)-I^{n}_{\delta}(\phi,t))}]\\
				&=\sum_{\om_1,\ldots,\om_n\in\mbZ^2}\sum_{j_1,\ldots,j_n=1}^{2}\int_{0}^{t}\dd u_1\int_{0}^{u_1}\dd u_2\ldots\int_{0}^{u_{n-1}}\dd u_n\abs{\hat{\phi}(u_1,-\om_1,j_1,\ldots,u_n,-\om_n,j_n)}^2\\
				&\quad-\sum_{\substack{\om_1,\ldots,\om_n\in\mbZ^{2}\\\abs{\om_1}\leq \delta^{-1},\ldots,\abs{\om_n}\leq \delta^{-1}}}\sum_{j_1,\ldots,j_n=1}^{2}\int_{0}^{t}\dd u_1\int_{0}^{u_1}\dd u_2\ldots\int_{0}^{u_{n-1}}\dd u_n\abs{\hat{\phi}(u_1,-\om_1,j_1,\ldots,u_n,-\om_n,j_n)}^2.
			\end{split}
		\end{equation*}
		The sum on the RHS vanishes as $\delta\to0$. Hence $I^{n}_{\delta}(\phi,t)\to I^{n}(\phi,t)$ in $L^2(\mbP)$, which implies that $I^{n}(\phi,t)$ is a martingale.
		
		We can now inductively apply the complex Burkholder--Davis--Gundy inequality, Lemma~\ref{lem:complex_Burkholder_Davis_Gundy_inequality}. Let $n=1$. We compute the quadratic variation $[I^{1}(\phi),\overline{I^{1}(\phi)}]_t$,
		\begin{equation*}
			[I^{1}(\phi),\overline{I^{1}(\phi)}]_t=\sum_{\om_1\in\mbZ^{2}}\sum_{j_1=1}^{2}\int_{0}^{t}\dd u_1\abs{\hat{\phi}(u_1,-\om_1,j_1)}^{2}.
		\end{equation*}
		We apply Lemma~\ref{lem:complex_Burkholder_Davis_Gundy_inequality} to estimate for any $p\in[1,\infty)$,
		\begin{equation*}
			\mbE[\abs{I^{1}(\phi,t)}^p]\leq C\Bigl(\sum_{\om_1\in\mbZ^{2}}\sum_{j_1=1}^{2}\int_{0}^{t}\abs{\phi(u_1,-\om_1,j_1)}^{2}\dd u_1\Bigr)^{p/2}=C \mbE[\abs{I^{1}(\phi,t)}^2]^{p/2}.
		\end{equation*}
		Assume the claim is true for iterated integrals of order up to $n-1$. We compute the quadratic variation of $I^{n}(\phi,t)$,
		\begin{equation*}
			\begin{split}
				&[I^{n}(\phi),\overline{I^{n}(\phi)}]_t\\
				&=\sum_{\om_1\in\mbZ^{2}}\sum_{j_1=1}^{2}\int_{0}^{t}\dd u_1\\
				&\quad\Bigl\lvert\sum_{\om_2,\ldots,\om_n\in\mbZ^{2}}\sum_{j_2,\ldots,j_n=1}^{2}\int_{0}^{u_1}\dd W^{j_2}(u_2,\om_2)\ldots\int_{0}^{u_{n-1}}\dd W^{j_n}(u_n,\om_n)\hat{\phi}(u_1,-\om_1,j_1,\ldots,u_n,-\om_n,j_n)\Bigr\rvert^{2}.
			\end{split}
		\end{equation*}
		Assume $p\in[2,\infty)$. We apply Minkowski's inequality and the induction hypothesis to estimate
		\begin{equation*}
			\begin{split}
				&\mbE[\abs{I^{n}(\phi,t)}^{p}]^{2/p}\lesssim\mbE[[I^{n}(\phi),\overline{I^{n}(\phi)}]_t^{p/2}]^{2/p}\\
				&\leq\sum_{\om_1\in\mbZ^{2}}\sum_{j_1=1}^{2}\int_{0}^{t}\dd u_1\mbE\Bigl[\Bigl\lvert\sum_{\om_2,\ldots,\om_n\in\mbZ^{2}}\sum_{j_2,\ldots,j_n=1}^{2}\int_{0}^{u_1}\dd W^{j_2}(u_2,\om_2)\ldots\int_{0}^{u_{n-1}}\dd W^{j_n}(u_n,\om_n)\\
				&\multiquad[12]\hat{\phi}(u_1,-\om_1,j_1,\ldots,u_n,-\om_n,-j_n)\Bigr\rvert^{p}\Bigr]^{2/p}\\
				&\lesssim \sum_{\om_1,\ldots,\om_n\in\mbZ^{2}}\sum_{j_1,\ldots,j_n=1}^{2}\int_{0}^{t}\dd u_1\ldots\int_{0}^{u_{n-1}}\dd u_n\abs{\hat{\phi}(u_1,-\om_1,j_1,\ldots,u_n,-\om_n,j_n)}^{2}\\
				&=\mbE[\abs{I^{n}(\phi,t)}^{2}],
			\end{split}
		\end{equation*}
		where the last line follows by It\^{o}'s isometry. We conclude by applying the $L^p(\mbP)$-martingale convergence theorem to let $t\to\infty$.
	\end{proof}
\end{details}
Let $T>0$ be a time horizon. Next we define iterated It\^{o} integrals with mollification~\eqref{eq:def_cut_off}--\eqref{eq:def_mollifiers} and heterogeneity $\het\from[0,T]\times\mbT^{2}\to\mbR$ and show how one may control the influence of the heterogeneity by passing to real space (Lemma~\ref{lem:isometry_argument}).
\begin{definition}\label{def:iterated_Ito_mollification_heterogeneity}
	Let $n\in\mbN$, $T>0$ and
	\begin{equation*}
		(0,T)^{n}_{>}\defeq\{(u_{1},\ldots,u_{n})\in(0,T)^{n}:u_{1}>u_{2}>\ldots>u_{n}\}
	\end{equation*}
	For every $\het\in C_TL^{\infty}(\mbT^{2})$, $\delta\geq0$ and $\varphi$ as in~\eqref{eq:def_cut_off} we define the heterogeneous iterated It\^{o} integral acting on $\phi\in L^2((0,T)^n_{>}\times\mbT^{2n};\mbC^{2n})$ (with mollification if $\delta>0$) by
	\begin{equation*}
		\begin{split}
			I^{n}_{\delta;\het}(\phi)&\defeq\sum_{\om_1,\ldots,\om_n\in\mbZ^2}\sum_{j_1,\ldots,j_n=1}^{2}\sum_{m_1,\ldots,m_n\in\mbZ^2}\varphi(\delta m_1)\dots\varphi(\delta m_n)\\
			&\quad\times\int_{0}^{T}\dd W^{j_1}(u_1,m_1)\int_{0}^{u_1}\dd W^{j_2}(u_2,m_2)\ldots\int_{0}^{u_{n-1}}\dd W^{j_n}(u_n,m_n)\\
			&\quad\times\hat{\het}(u_1,\om_1-m_1)\dots\hat{\het}(u_n,\om_n-m_n)\hat{\phi}(u_1,-\om_1,j_1,\ldots,u_n,-\om_n,j_n).
		\end{split}
	\end{equation*}
\end{definition}
Next we show how one may control the heterogeneity.
\begin{lemma}\label{lem:isometry_argument}
	Let $n\in\mbN$, $T>0$, $\het\in C_TL^{\infty}(\mbT^{2})$, $\delta\geq0$ and $\varphi$ be as in~\eqref{eq:def_cut_off}. If $\delta=0$, then
	\begin{equation}\label{eq:isometry_bound_delta=0}
		\mbE[\abs{I^{n}_{0;\het}(\phi)}^{2}]\leq\norm{\het}_{C_TL^\infty}^{2n}\mbE[\abs{I^{n}(\phi)}^{2}].
	\end{equation}
	If $\delta>0$, then
	\begin{equation}\label{eq:isometry_bound_delta>0}
		\mbE[\abs{I^{n}_{\delta;\het}(\phi)}^{2}]\leq\norm{\varphi}_{L^{\infty}}^{2n}\norm{\het}_{C_TL^\infty}^{2n}\mbE[\abs{I^{n}(\phi)}^{2}].
	\end{equation}
\end{lemma}
\begin{proof}
	Let $\boldsymbol{\om}=(\om_1,\ldots,\om_n)$, $\boldsymbol{m}=(m_1,\ldots,m_n)$, $\boldsymbol{u}=(u_1,\ldots,u_n)$ and $\boldsymbol{j}=(j_1,\ldots,j_n)$. We represent
	\begin{equation*}
		\hat{\het}(u_1,\om_1-m_1)\dots\hat{\het}(u_n,\om_n-m_n)\hat{\phi}(u_1,-\om_1,j_1,\ldots,u_n,-\om_n,j_n)=\hat{\het^{\otimes n}}(\boldsymbol{u},\boldsymbol{\om}-\boldsymbol{m})\hat{\phi}(\boldsymbol{u},-\boldsymbol{\om},\boldsymbol{j}).
	\end{equation*}
	Using that the Fourier transform turns products of $L^2(\mbT^{2n})$-functions into convolutions, we obtain
	\begin{equation*}
		\sum_{\boldsymbol{\om}\in(\mbZ^{2})^{\times n}}\hat{\het^{\otimes n}}(\boldsymbol{u},\boldsymbol{\om}-\boldsymbol{m})\hat{\phi}(\boldsymbol{u},-\boldsymbol{\om},\boldsymbol{j})=\mathscr{F}(\het^{\otimes n}\phi)(\boldsymbol{u},-\boldsymbol{m},\boldsymbol{j})
	\end{equation*}
	and consequently, $I^{n}_{\delta;\het}(\phi)=I^{n}(\psi_{\delta}^{\otimes n}\ast(\het^{\otimes n}\phi))$ for $\delta>0$ and $I^{n}_{0;\het}(\phi)=I^{n}(\het^{\otimes n}\phi)$ for $\delta=0$. To establish~\eqref{eq:isometry_bound_delta>0} for $\delta>0$, we apply It\^{o}'s isometry (Lemma~\ref{lem:Ito_isometry}), Parseval's theorem and use the boundedness of $\varphi$ and $\het$ to control
	\begin{equation*}
		\begin{split}
			&\mbE[\abs{I^n_{\delta;\het}(\phi)}^2]\\
			&=\sum_{\boldsymbol{j}\in\{1,2\}^{\times n}}\sum_{\boldsymbol{m}\in(\mbZ^{2})^{\times n}}\int_{0}^{T}\dd u_1\ldots\int_{0}^{u_{n-1}}\dd u_n\abs{\varphi^{\otimes n}(\delta \boldsymbol{m})}^{2}\abs{\mathscr{F}(\het^{\otimes n}\phi)(\boldsymbol{u},-\boldsymbol{m},\boldsymbol{j})}^{2}\\
			&\leq\norm{\varphi}_{L^{\infty}}^{2n}\sum_{\boldsymbol{j}\in\{1,2\}^{\times n}}\sum_{\boldsymbol{m}\in(\mbZ^{2})^{\times n}}\int_{0}^{T}\dd u_1\ldots\int_{0}^{u_{n-1}}\dd u_n\abs{\mathscr{F}(\het^{\otimes n}\phi)(\boldsymbol{u},-\boldsymbol{m},\boldsymbol{j})}^{2}\\
			&=\norm{\varphi}_{L^{\infty}}^{2n}\sum_{\boldsymbol{j}\in\{1,2\}^{\times n}}\int_{0}^{T}\dd u_1\ldots\int_{0}^{u_{n-1}}\dd u_n\int_{(\mbT^2)^n}\dd \boldsymbol{x}\abs{\het^{\otimes n}(\boldsymbol{u},\boldsymbol{x})}^2\abs{\phi(\boldsymbol{u},\boldsymbol{x},\boldsymbol{j})}^2\\
			&\leq\norm{\varphi}_{L^{\infty}}^{2n}\norm{\het}_{C_TL^\infty}^{2n}\sum_{\boldsymbol{j}\in\{1,2\}^{\times n}}\int_{0}^{T}\dd u_1\ldots\int_{0}^{u_{n-1}}\dd u_n\int_{(\mbT^2)^n}\dd \boldsymbol{x}\abs{\phi(\boldsymbol{u},\boldsymbol{x},\boldsymbol{j})}^2\\
			&=\norm{\varphi}_{L^{\infty}}^{2n}\norm{\het}_{C_TL^\infty}^{2n}\mbE[\abs{I^n(\phi)}^2],
		\end{split}
	\end{equation*}
	which yields~\eqref{eq:isometry_bound_delta>0}. The derivation of~\eqref{eq:isometry_bound_delta=0} for $\delta=0$ is analogous.
\end{proof}
Next we consider the second moment of the difference $I^{n}_{0;\het}(\phi)-I^{n}_{\delta;\het}(\phi)$, where we combine the idea of Lemma~\ref{lem:isometry_argument} with an interpolation argument to extract a small power of $\delta$. The resulting bound~\eqref{eq:isometry_difference} is reminiscent of~\cite[Lem.~8.7]{gubinelli_perkowski_17} in that we also need to pass to a space of lower regularity.
\begin{lemma}\label{lem:isometry_interpolation}
	Let $n\in\mbN$, $T>0$, $\het\in C_T\mcH^{1}(\mbT^{2})$, $\delta\geq0$ and $\varphi$ be as in~\eqref{eq:def_cut_off}. Then for every $0<\vartheta'<\vartheta\leq1$ and $\phi\in L^2((0,T)^n_{>};\mcH^{\vartheta}(\mbT^{2n};\mbC^{2n}))$,
	\begin{equation}\label{eq:isometry_difference}
		\mbE[\abs{I^{n}_{0;\het}(\phi)-I^{n}_{\delta;\het}(\phi)}^2]\lesssim\delta^{2\vartheta'}\norm{1-\varphi}_{C^{1}}^{2}(1+\norm{\varphi}_{L^{\infty}}^{n-1})^{2}\norm{\het}_{C_{T}\mcH^{1}}^{2n}\mbE[\abs{I^{n}((1-\Delta)^{\vartheta/2}\phi)}^{2}].
	\end{equation}
\end{lemma}
\begin{proof}
	We use induction and interpolation to bound for every $\vartheta'\in(0,1)$,
	\begin{equation}\label{eq:tensor_difference_interpolation_bound}
		\abs{1-\varphi^{\otimes n}(\delta\boldsymbol{m})}\leq\delta^{\vartheta'}\norm{1-\varphi}_{C^{1}}(1+\norm{\varphi}_{L^{\infty}}^{n-1})n(1+\abs{\boldsymbol{m}}^{2})^{\vartheta'/2},
	\end{equation}
	which combined with It\^{o}'s isometry and the definition of $I^{n}_{0;\het}(\phi)-I^{n}_{\delta;\het}(\phi)$ yields
	\begin{equation}\label{eq:iterated_integral_interpolation_bound}
		\begin{split}
			&\mbE[\abs{I^{n}_{0;\het}(\phi)-I^{n}_{\delta;\het}(\phi)}^2]\\
			&\lesssim\delta^{2\vartheta'}\norm{1-\varphi}_{C^{1}}^{2}(1+\norm{\varphi}_{L^{\infty}}^{n-1})^{2}\\
			&\quad\times\sum_{\boldsymbol{j}\in\{1,2\}^{\times n}}\sum_{\boldsymbol{m}\in(\mbZ^{2})^{\times n}}\int_{0}^{T}\dd u_1\ldots\int_{0}^{u_{n-1}}\dd u_n(1+\abs{2\uppi\boldsymbol{m}}^{2})^{\vartheta'}\abs{\mathscr{F}(\het^{\otimes n}\phi)(\boldsymbol{u},-\boldsymbol{m},\boldsymbol{j})}^{2}.
		\end{split}
	\end{equation}
	\begin{details}
		\paragraph{Proof of~\eqref{eq:tensor_difference_interpolation_bound}.}
		It follows inductively that
		\begin{equation}\label{eq:tensor_diff_identity_mollifier}
			1-\varphi^{\otimes n}(\delta\boldsymbol{m})=1-\varphi(\delta m_{n})+\sum_{l=1}^{n-1}(1-\varphi(\delta m_{n-l}))\prod_{k=n-l+1}^{n}\varphi(\delta m_{k}).
		\end{equation}
		Indeed, we claim the more general identity that for every $(x_{k})_{k=1,\ldots,n}$ and $(y_{k})_{k=1,\ldots,n}$,
		\begin{equation}\label{eq:inductive_hypothesis}
			\Bigl(\prod_{k=1}^{n}x_{k}\Bigr)-\Bigl(\prod_{k=1}^{n}y_{k}\Bigr)=(x_{n}-y_{n})\Bigl(\prod_{k=1}^{n-1}x_{k}\Bigr)+\sum_{l=1}^{n-1}(x_{n-l}-y_{n-l})\Bigl(\prod_{k=1}^{n-l-1}x_{k}\Bigr)\Bigl(\prod_{k=n-l+1}^{n}y_{k}\Bigr).
		\end{equation}
		To obtain~\eqref{eq:tensor_diff_identity_mollifier}, we can then apply~\eqref{eq:inductive_hypothesis} with $x_{k}=1$ and $y_{k}=\varphi(\delta m_{k})$ for every $k\in\{1,\ldots,n\}$.
		
		The identity~\eqref{eq:inductive_hypothesis} is clear for $n=1$. For the induction step, assume~\eqref{eq:inductive_hypothesis} holds true with $n$ replaced by $n-1$. We decompose
		\begin{equation*}
			\Bigl(\prod_{k=1}^{n}x_{k}\Bigr)-\Bigl(\prod_{k=1}^{n}y_{k}\Bigr)=(x_{n}-y_{n})\Bigl(\prod_{k=1}^{n-1}x_{k}\Bigr)+\Bigl(\Bigl(\prod_{k=1}^{n-1}x_{k}\Bigr)-\Bigl(\prod_{k=1}^{n-1}y_{k}\Bigr)\Bigr)y_{n}.
		\end{equation*}
		Applying the inductive hypothesis, we obtain
		\begin{equation*}
			\begin{split}
				&\Bigl(\prod_{k=1}^{n}x_{k}\Bigr)-\Bigl(\prod_{k=1}^{n}y_{k}\Bigr)\\
				&=(x_{n}-y_{n})\Bigl(\prod_{k=1}^{n-1}x_{k}\Bigr)+\Bigl((x_{n-1}-y_{n-1})\Bigl(\prod_{k=1}^{n-2}x_{k}\Bigr)+\sum_{l=1}^{n-2}(x_{n-1-l}-y_{n-1-l})\Bigl(\prod_{k=1}^{n-l-2}x_{k}\Bigr)\Bigl(\prod_{k=n-l}^{n-1}y_{k}\Bigr)\Bigr)y_{n}\\
				&=(x_{n}-y_{n})\Bigl(\prod_{k=1}^{n-1}x_{k}\Bigr)+(x_{n-1}-y_{n-1})\Bigl(\prod_{k=1}^{n-2}x_{k}\Bigr)y_{n}+\sum_{l=2}^{n-1}(x_{n-l}-y_{n-l})\Bigl(\prod_{k=1}^{n-l-1}x_{k}\Bigr)\Bigl(\prod_{k=n-l+1}^{n}y_{k}\Bigr)\\
				&=(x_{n}-y_{n})\Bigl(\prod_{k=1}^{n-1}x_{k}\Bigr)+\sum_{l=1}^{n-1}(x_{n-l}-y_{n-l})\Bigl(\prod_{k=1}^{n-l-1}x_{k}\Bigr)\Bigl(\prod_{k=n-l+1}^{n}y_{k}\Bigr),
			\end{split}
		\end{equation*}
		which yields~\eqref{eq:inductive_hypothesis}.
		
		Having established~\eqref{eq:tensor_diff_identity_mollifier}, we next prove~\eqref{eq:tensor_difference_interpolation_bound} by interpolation. We estimate for every $m\in\mbZ^{2}$,
		\begin{equation*}
			\abs{1-\varphi(\delta m)}\leq\norm{1-\varphi}_{L^{\infty}(\mbR^{2})}
		\end{equation*}
		and by the mean value theorem,
		\begin{equation*}
			\abs{1-\varphi(\delta m)}\leq\norm{\nabla\varphi}_{L^{\infty}(\mbR^{2})}\delta\abs{m}.
		\end{equation*}
		Interpolation then yields for every $\vartheta'\in[0,1]$,
		\begin{equation}\label{eq:mollifier_interpolation}
			\abs{1-\varphi(\delta m)}\leq\norm{1-\varphi}_{L^{\infty}(\mbR^{2})}^{1-\vartheta'}\norm{\nabla\varphi}_{L^{\infty}(\mbR^{2})}^{\vartheta'}\delta^{\vartheta'}\abs{m}^{\vartheta'}\leq\norm{1-\varphi}_{C^{1}(\mbR^{2})}\delta^{\vartheta'}\abs{m}^{\vartheta'}.
		\end{equation}
		Hence, combining~\eqref{eq:tensor_diff_identity_mollifier} and~\eqref{eq:mollifier_interpolation},
		\begin{equation*}
			\begin{split}
				\abs{1-\varphi^{\otimes n}(\delta\boldsymbol{m})}&\leq\abs{1-\varphi(\delta m_{n})}+\sum_{l=1}^{n-1}\abs{1-\varphi(\delta m_{n-l})}\prod_{k=n-l+1}^{n}\abs{\varphi(\delta m_{k})}\\
				&\leq\delta^{\vartheta'}\norm{1-\varphi}_{C^{1}}\abs{m_{n}}^{\vartheta'}+\delta^{\vartheta'}\norm{1-\varphi}_{C^{1}}\norm{\varphi}_{L^{\infty}}^{n-1}\sum_{l=1}^{n-1}\abs{m_{n-l}}^{\vartheta'}\\
				&\leq\delta^{\vartheta'}\norm{1-\varphi}_{C^{1}}(1+\norm{\varphi}_{L^{\infty}}^{n-1})\sum_{l=0}^{n-1}\abs{m_{n-l}}^{\vartheta'}.
			\end{split}
		\end{equation*}
		Finally, we bound
		\begin{equation*}
			\sum_{l=0}^{n-1}\abs{m_{n-l}}^{\vartheta'}\leq\sum_{l=0}^{n-1}(1+\abs{\boldsymbol{m}}^{2})^{\vartheta'/2}=n(1+\abs{\boldsymbol{m}}^{2})^{\vartheta'/2},
		\end{equation*}
		which yields~\eqref{eq:tensor_difference_interpolation_bound}.
	\end{details}
	
	Let $0<\vartheta'<\vartheta\leq1$, to control the integrand, we use standard product estimates in Bessel potential spaces to obtain
	\begin{details}
		by~\cite[Thm.~6.18]{hairer_09},
	\end{details}
	\begin{equation}\label{eq:integrand_bound}
		\begin{split}
			\sum_{\boldsymbol{m}\in(\mbZ^{2})^{\times n}}(1+\abs{2\uppi\boldsymbol{m}}^{2})^{\vartheta'}\abs{\mathscr{F}(\het^{\otimes n}\phi)(\boldsymbol{u},-\boldsymbol{m},\boldsymbol{j})}^{2}&=\norm{(\het^{\otimes n}\phi)(\boldsymbol{u},\place,\boldsymbol{j})}_{\mcH^{\vartheta'}(\mbT^{2n})}^{2}\\
			&\lesssim\norm{\het}_{C_{T}\mcH^{1}(\mbT^{2})}^{2n}\norm{\phi(\boldsymbol{u},\place,\boldsymbol{j})}_{\mcH^{\vartheta}(\mbT^{2n})}^{2}.
		\end{split}
	\end{equation}
	\begin{details}
		where we used that
		\begin{equation*}
			\norm{\het^{\otimes n}(\boldsymbol{u},\place)}_{\mcH^{1}(\mbT^{2n})}^{2}=\sum_{\boldsymbol{m}\in\mbZ^{2n}}(1+\abs{2\uppi m_{1}}^{2}+\ldots+\abs{2\uppi m_{n}}^{2})\abs{\hat{\het}(u_{1},m_{1})\ldots\hat{\het}(u_{n},m_{n})}^{2}\leq\norm{\het}_{C_{T}\mcH^{1}(\mbT^{2})}^{2n}.
		\end{equation*}
	\end{details}
	\begin{details}
		By Parseval's theorem and Tonelli's theorem,
		\begin{equation*}
			\begin{split}
				&\mbE[\abs{I^{n}_{0;\het}(\phi)-I^{n}_{\delta;\het}(\phi)}^2]\\
				&\lesssim\delta^{2\vartheta'}\norm{1-\varphi}_{C^{1}}^{2}(1+\norm{\varphi}_{L^{\infty}}^{n-1})^{2}\norm{\het}_{C_{T}\mcH^{1}}^{2n}\\
				&\quad\times\sum_{\boldsymbol{\om}\in(\mbZ^{2})^{\times n}}\sum_{\boldsymbol{j}\in\{1,2\}^{\times n}}\int_{0}^{T}\dd u_1\ldots\int_{0}^{u_{n-1}}\dd u_n(1+\abs{2\uppi\boldsymbol{\om}}^{2})^{\vartheta}\abs{\hat{\phi}(\boldsymbol{u},-\boldsymbol{\om},\boldsymbol{j})}^{2}.
			\end{split}
		\end{equation*}
	\end{details}
	Plugging~\eqref{eq:integrand_bound} into~\eqref{eq:iterated_integral_interpolation_bound}, we arrive at 
	\begin{equation*}
		\mbE[\abs{I^{n}_{0;\het}(\phi)-I^{n}_{\delta;\het}(\phi)}^2]\lesssim\delta^{2\vartheta'}\norm{1-\varphi}_{C^{1}}^{2}(1+\norm{\varphi}_{L^{\infty}}^{n-1})^{2}\norm{\het}_{C_{T}\mcH^{1}}^{2n}\mbE[\abs{I^{n}((1-\Delta)^{\vartheta/2}\phi)}^{2}],
	\end{equation*}
	which yields~\eqref{eq:isometry_difference}.
\end{proof}
\begin{remark}
	In~\eqref{eq:isometry_difference}, to extract an exponent $\vartheta'$ of $\delta^{2}$, we had to assume that $\het\in C_{T}\mcH^{1}(\mbT^{2})$ and $\phi\in L^2((0,T)^n_{>};\mcH^{\vartheta}(\mbT^{2n};\mbC^{2n}))$ for some $0<\vartheta'<\vartheta\leq1$; in other words, we had to assume that $\het$ is regular and also had to incur a small loss in the regularity of $\phi$. Both assumptions can be lifted if one considers iterated It\^{o} integrals based on the noise $\psi_{\delta}\ast(\het\boldsymbol{\xi})$ rather than $\het(\psi_{\delta}\ast\boldsymbol{\xi})$, where one can use that $\psi_{\delta}\ast(\het\boldsymbol{\xi})$ may be interpreted as a mollified, vector-valued space-time white noise on an $L^{2}$-space with measure $\het(u,x)^{2}\dd u\dd x$.
	\begin{details}
		\paragraph{Analogue of Lemma~\ref{lem:isometry_interpolation} for iterated It\^{o} integrals based on $\psi_{\delta}\ast(\het\boldsymbol{\xi})$.}
		For iterated It\^{o} integrals based on $\psi_{\delta}\ast(\het\boldsymbol{\xi})$, we can estimate by It\^{o}'s isometry (Lemma~\ref{lem:Ito_isometry}),
		\begin{equation*}
			\begin{split}
			\mbE[\abs{I^{n}(\het^{\otimes n}\phi)-I^{n}(\het^{\otimes n}(\psi_{\delta}^{\otimes n}\ast\phi))}^{2}]&=\norm{\het^{\otimes n}(\phi-\psi_{\delta}^{\otimes n}\ast\phi)}_{L^2((0,T)^n_{>}\times\mbT^{2n}\times\{1,2\}^{n};\mbC)}^{2}\\
			&\leq\norm{\het}_{C_{T}L^{\infty}}^{2n}\norm{\phi-\psi_{\delta}^{\otimes n}\ast\phi}_{L^2((0,T)^n_{>}\times\mbT^{2n}\times\{1,2\}^{n};\mbC)}^{2}
			\end{split}
		\end{equation*} 
		and apply~\eqref{eq:tensor_difference_interpolation_bound} to deduce
		\begin{equation*}
			\norm{\phi-\psi_{\delta}^{\otimes n}\ast\phi}_{L^2((0,T)^n_{>}\times\mbT^{2n}\times\{1,2\}^{n};\mbC)}^{2}\lesssim\delta^{2\vartheta'}\norm{1-\varphi}_{C^{1}}^{2}(1+\norm{\varphi}_{L^{\infty}}^{n-1})^{2}\norm{\phi}_{L^2((0,T)^n_{>};\mcH^{\vartheta'}(\mbT^{2n};\mbC^{2n}))}^{2},
		\end{equation*}
		which yields the bound
		\begin{equation*}
			\mbE[\abs{I^{n}(\het^{\otimes n}\phi)-I^{n}(\het^{\otimes n}(\psi_{\delta}^{\otimes n}\ast\phi))}^{2}]\lesssim\norm{\het}_{C_{T}L^{\infty}}^{2n}\delta^{2\vartheta'}\norm{1-\varphi}_{C^{1}}^{2}(1+\norm{\varphi}_{L^{\infty}}^{n-1})^{2}\norm{\phi}_{L^2((0,T)^n_{>};\mcH^{\vartheta'}(\mbT^{2n};\mbC^{2n}))}^{2}.
		\end{equation*}
		Hence, for integrals based on $\psi_{\delta}\ast(\het\boldsymbol{\xi})$, we neither have to assume that $\het$ is regular nor incur a loss of regularity. However, since our integrals are based on $\het(\psi_{\delta}\ast\boldsymbol{\xi})$, we need to apply Lemma~\ref{lem:isometry_interpolation} and have to do both.
	\end{details}
\end{remark}
The following Kolmogorov criterion provides an efficient method for establishing the regularity of stochastic processes in H\"{o}lder--Besov spaces. The presentation of this lemma is reminiscent of~\cite[Prop.~4.1]{perkowski_20}.
\begin{lemma}\label{lem:existence_criterion}
	Let $\alpha\in\mbR$, $p\in(1,\infty)$, $\gamma\in(1/p,1]$, $\gamma'\in(0,\gamma-1/p)$ and $X\from[0,T]\to\mcS'(\mbT^{2})$ be a stochastic process such that $X(0)=0$. Assume there exists some $K>0$ such that
	\begin{equation}\label{eq:Kolmogorov_condition}
		\int_{0}^{T}\int_{0}^{T}\abs{t-s}^{-2-p\gamma'}\sum_{q\in\mbN_{-1}}2^{pq\alpha}\int_{\mbT^{2}}\mbE[\abs{\Delta_qX(t,x)-\Delta_qX(s,x)}^p]\dd x\dd s\dd t\leq K<\infty.
	\end{equation}
	Then there exists a modification of $X$ (which we do not relabel) such that
	\begin{equation*}
		\mbE\Bigl[\norm{X}^p_{C^{\gamma'}_{T}\mcC^{\alpha-2/p}}\Bigr]\lesssim\mbE\Bigl[\norm{X}^p_{C^{\gamma'}_{T}\mcB_{p,p}^{\alpha}}\Bigr]\lesssim_{p,\gamma'}K
	\end{equation*}
	and in particular $X \in C^{\gamma'}_T \mcC^{\alpha-2/p}(\mbT^{2})$ a.s. Assume that
	\begin{equation*}
		M\defeq\sup_{s\neq t\in[0,T]}\abs{t-s}^{-p\gamma}\sum_{q\in\mbN_{-1}}2^{pq\alpha}\int_{\mbT^{2}}\mbE[\abs{\Delta_qX(t,x)-\Delta_qX(s,x)}^p]\dd x<\infty,
	\end{equation*}
	then for each $\gamma'\in(0,\gamma-1/p)$, \eqref{eq:Kolmogorov_condition} is satisfied with 
	\begin{equation*}
		K=M\int_{0}^{T}\int_{0}^{T}\abs{t-s}^{-2+p(\gamma-\gamma')}\dd s\dd t<\infty.
	\end{equation*}
\end{lemma}
\begin{proof}
	The bound follows by the definition of $\mcB_{p,p}^{\alpha}(\mbT^{2})$ (Definition~\ref{def:Besov_space}), the Besov–H\"{o}lder embedding in time~\cite[Proof~of~Thm.~A.10]{friz_victoir_10} and the Besov embedding in space~\eqref{eq:Besov_cts_embedding}. To show $X\in C^{\gamma'}_{T}\mcB^{\alpha}_{p,p}(\mbT^{2})\embed C^{\gamma'}_{T}\mcC^{\alpha-2/p}(\mbT^{2})$, it suffices to exhibit a smooth, approximating sequence. This can be achieved by combining the Besov embedding~\eqref{eq:Besov_cts_embedding} with the fact that the smooth functions are dense in $\mcB^{\alpha}_{p,p}(\mbT^{2})$.
\end{proof}
\begin{details}
	\begin{proof}
		By the definition of the Besov norm and Tonelli's theorem,
		\begin{equation*}
			\mbE[\norm{X(t)-X(s)}_{\mcB^{\alpha}_{p,p}}^{p}]=\sum_{q\in\mbN_{-1}}2^{pq\alpha}\int_{\mbT^{2}}\mbE[\abs{\Delta_qX(t,x)-\Delta_qX(s,x)}^p]\dd x,
		\end{equation*}
		hence by assumption,
		\begin{equation*}
			\int_{0}^{T}\int_{0}^{T}\abs{t-s}^{-2-p\gamma'}\mbE[\norm{X(t)-X(s)}_{\mcB^{\alpha}_{p,p}}^{p}]\dd s\dd t\leq K<\infty.
		\end{equation*}
		We may apply the Besov–H\"{o}lder embedding in time~\cite[Proof~of~Thm.~A.10]{friz_victoir_10} to bound the H\"{o}lder seminorm of $X$,
		\begin{equation*}
			\mbE\Bigl[\Bigl(\sup_{s\not=t\in[0,T]}\frac{\norm{X(t)-X(s)}_{\mcB_{p,p}^{\alpha}}}{\abs{t-s}^{\gamma'}}\Bigr)^{p}\Bigr]\lesssim_{p,\gamma'}\int_{0}^{T}\int_{0}^{T}\abs{t-s}^{-2-p\gamma'}\mbE[\norm{X(t)-X(s)}_{\mcB^{\alpha}_{p,p}}^{p}]\dd s\dd t\leq K.
		\end{equation*}
		Using that $X(0)=0$, we can estimate the uniform norm,
		\begin{equation*}
			\sup_{t\in[0,T]}\norm{X(t)}_{\mcB_{p,p}^{\alpha}}=\sup_{t\in[0,T]}\norm{X(t)-X(0)}_{\mcB_{p,p}^{\alpha}}\leq T^{\gamma'}\sup_{s\neq t\in[0,T]}\frac{\norm{X(t)-X(s)}_{\mcB_{p,p}^{\alpha}}}{\abs{t-s}^{\gamma'}}.
		\end{equation*}
		By the Besov embedding~\eqref{eq:Besov_cts_embedding}, $\norm{\place}_{\mcC^{\alpha-2/p}}\lesssim\norm{\,\cdot\,}_{\mcB_{p,p}^{\alpha}}$, so that the same holds in $\mcC^{\alpha-2/p}(\mbT^{2})$ as well. This yields the claim.
	\end{proof}
\end{details}
\subsection{Diagrams of Order 2 and 3}\label{sec:diagrams_of_order_2_and_3}
In this section we construct the second-order diagrams $\ty$, $\Checkmark{10}$ and $\Checkmark{20}$ and the third-order diagrams $\PreThree{10}$ and $\PreThree{20}$.
\begin{lemma}\label{lem:existence_ypsilon}
	Let $T>0$, $\alpha<-2$, $\kappa\in(0,1/2)$ and $\het\in C_{T}L^{\infty}(\mbT^{2})$. Then for any $p\in[1,\infty)$ we have $\mbE[\norm{\ty}_{\msL^{\kappa}_{T}\mcC^{2\alpha+4}}^p]^{1/p}\lesssim\norm{\het}_{C_{T}L^\infty}^{2}$ and in particular $\ty\in\msL^{\kappa}_{T}\mcC^{2\alpha+4}(\mbT^{2})$ a.s.
\end{lemma}
From now on we denote $\Ypsilon{10}\defeq\ty$ to emphasize the separate r\^{o}les of colour and shape.
\begin{details}
	
	We first prove that $\Delta_{q}\Ypsilon{10}(t,x)$ can be represented as an iterated It\^{o} integral. Let
	\begin{equation*}
		\begin{split}
			&\phi_{t,x}^{q}(u_1,x_1,j_1,u_2,x_2,j_2)\\
			&\defeq\sum_{\varsigma\in\Sigma(1,2)}\sum_{j_3=1}^{2}\int_{0}^{t}\dd u_3\int_{\mbT^2}\dd y\\
			&\quad(\partial_{j_3}\mathscr{H}_{t-u_3}\ast\msF^{-1}\varrho_q)(y+x)(\partial_{j_{\varsigma(1)}}\mathscr{H}_{u_3-u_{\varsigma(1)}})(-x_{\varsigma(1)}-y)(\partial_{j_3}\mathscr{G}\ast\partial_{j_{\varsigma(2)}}\mathscr{H}_{u_3-u_{\varsigma(2)}})(-x_{\varsigma(2)}-y),
		\end{split}
	\end{equation*}
	where $\partial_k\mathscr{H}_u$ denotes the kernel to $H^{k}_u$, for $k=1,2$ and $u\geq0$, and $\partial_k\mathscr{G}$ the kernel to $G^{k}$. In this case,
	\begin{equation*}
		\begin{split}
			&\hat{\phi_{t,x}^{q}}(u_1,-\om_1,j_1,u_2,-\om_2,j_2)\\
			&=\euler^{2\uppi\upi\inner{\om_1+\om_2}{x}}\varrho_q(\om_1+\om_2)\\
			&\quad\times\sum_{\varsigma\in\Sigma(1,2)}\sum_{j_3=1}^{2}\int_{0}^{t}\dd u_3H^{j_3}_{t-u_3}(\om_{1}+\om_{2})H^{j_{\varsigma(1)}}_{u_3-u_{\varsigma(1)}}(\om_{\varsigma(1)})G^{j_3}(\om_{\varsigma(2)})H^{j_{\varsigma(2)}}_{u_3-u_{\varsigma(2)}}(\om_{\varsigma(2)}),
		\end{split}
	\end{equation*}
	such that we have the iterated It\^{o} integral representation
	\begin{equation*}
		\begin{split}
			\Delta_q\Ypsilon{10}(t,x)&=\sum_{\om_1,\om_2\in\mbZ^{2}}\sum_{j_1,j_2=1}^{2}\sum_{m_1,m_2\in\mbZ^{2}}\int_{0}^{\infty}\dd W^{j_2}(u_2,m_2)\int_{0}^{u_2}\dd W^{j_1}(u_1,m_1)\\
			&\quad\hat{\het}(u_1,\om_1-m_1)\hat{\het}(u_2,\om_2-m_2)\hat{\phi_{t,x}^{q}}(u_1,-\om_1,j_1,u_2,-\om_2,j_2)\\
			&=I^{2}(\phi_{t,x}^{q}).
		\end{split}
	\end{equation*}
	The Fourier transform $\hat{\Ypsilon{10}}(t,\om)$ is also an iterated It\^{o} integral. Indeed, $\hat{\Ypsilon{10}}(t,\om)=I^{2}(\phi_{t,\om})$, where
	\begin{equation*}
		\begin{split}
			&\phi_{t,\om}(u_1,x_1,j_1,u_2,x_2,j_2)\\
			&\defeq\sum_{\varsigma\in\Sigma(1,2)}\sum_{j_3=1}^{2}\int_{0}^{t}\dd u_3\int_{\mbT^2}\dd y\\
			&\quad(\partial_{j_3}\mathscr{H}_{t-u_3}\ast\euler^{2\uppi\upi\inner{\om}{\place}})(y)(\partial_{j_{\varsigma(1)}}\mathscr{H}_{u_3-u_{\varsigma(1)}})(-x_{\varsigma(1)}-y)(\partial_{j_3}\mathscr{G}\ast\partial_{j_{\varsigma(2)}}\mathscr{H}_{u_3-u_{\varsigma(2)}})(-x_{\varsigma(2)}-y).
		\end{split}
	\end{equation*}
\end{details}
We first derive a useful upper bound on the second moments of $\Ypsilon{10}$ in terms of an explicit, time-dependent function $\mathsf{S}_{s,t}\SYpsilon$. We call this function the \emph{shape coefficient}.
\begin{definition}\label{def:shape_coefficient_ypsilon}
	Let $s,t\geq0$ and $\om_1,\om_2\in2\uppi\mbZ^2\setminus\{0\}$. We define the \emph{shape coefficient}
	\begin{equation}\label{eq:shape_coefficient_ypsilon}
		\begin{split}
			\mathsf{S}_{s,t}\SYpsilon(\om_1,\om_2)\defeq&\int_{0}^{t}\dd u_3\int_{0}^{s}\dd u_3'\int_{-\infty}^{u_3\wedge u_3'}\dd u_2\int_{-\infty}^{u_3\wedge u_3'}\dd u_1\\
			&\quad \euler^{-\abs{t+s-(u_3+u_3')}\abs{\om_1+\om_2}^2}\euler^{-\abs{u_3+u_3'-2u_1}\abs{\om_1}^2}\euler^{-\abs{u_3+u_3'-2u_2}\abs{\om_2}^2},
		\end{split}
	\end{equation}
	and the \emph{increment shape coefficient}
	\begin{equation}\label{eq:increment_shape_coefficient_ypsilon}
		\mathsf{D}_{s,t}\SYpsilon\defeq\mathsf{S}_{t,t}\SYpsilon+\mathsf{S}_{s,s}\SYpsilon-\mathsf{S}_{s,t}\SYpsilon-\mathsf{S}_{t,s}\SYpsilon.
	\end{equation}
\end{definition}
Here, the letter $\mathsf{S}$ stands for shape and the letter $\mathsf{D}$ for difference. It will be clear from the proof of Lemma~\ref{lem:ypsilon_upper_bound} that $\mathsf{D}_{s,t}\SYpsilon\geq0$.

Shape coefficients play a central r\^{o}le in our bounds, as they capture the iterated applications of $\vdiv\mcI$; they fundamentally depend on the shape of the diagram, as opposed to the additional colouring induced by $\nabla\Phi$. Using this notation, we obtain the following bound.
\begin{lemma}\label{lem:ypsilon_upper_bound}
	Let $s,t\in[0,T]$, $x\in\mbT^{2}$ and $q\in\mbN_{-1}$. It holds that
	\begin{equation}\label{eq:ypsilon_upper_bound}
		\begin{split}
			&\mbE[\abs{\Delta_{q}\Ypsilon{10}(t,x)-\Delta_{q}\Ypsilon{10}(s,x)}^2]\\
			&\leq\norm{\het}_{C_{T}L^{\infty}}^{4}2!(2\uppi)^{4}\sum_{\om\in\mbZ^{2}}\varrho_{q}(\om)^{2}\sum_{\substack{\om_1,\om_2\in\mbZ^{2}\setminus\{0\}\\\om=\om_1+\om_2}}\abs{\om_1}^{2}\abs{\om_2}^{-2}\abs{\inner{\om_2}{\om_1+\om_2}}^2\mathsf{D}_{s,t}\SYpsilon(2\uppi\om_1,2\uppi\om_2).
		\end{split}
	\end{equation}
\end{lemma}
We refer to the prefactor $\abs{\om_1}^{2}\abs{\om_2}^{-2}\abs{\inner{\om_2}{\om_1+\om_2}}^2$ as the \emph{colouring} of $\Ypsilon{10}$.

Before we give the proof of Lemma~\ref{lem:ypsilon_upper_bound}, let us comment on its general strategy. In $\Ypsilon{10}$, $\PreThree{10}$ and $\PreThree{20}$ it does not suffice to apply the triangle inequality to push the absolute value past the integral sign. This is related to the appearance of the sub-diagram $\Cherry{10}$, which we do not expect to be pointwise evaluable. Instead, we rely on bilinearity and expand the integrand according to the identity
\begin{equation*}
	(f(t)-f(s))\overline{(g(t)-g(s))}=f(t)\overline{g(t)}+f(s)\overline{g(s)}-f(s)\overline{g(t)}-f(t)\overline{g(s)},
\end{equation*}
which leads to the common equation for this type of shape coefficient,
\begin{equation*}
	\mathsf{D}_{s,t}=\mathsf{S}_{t,t}+\mathsf{S}_{s,s}-\mathsf{S}_{s,t}-\mathsf{S}_{t,s}.
\end{equation*}
We refer to Lemmas~\ref{lem:checkmark_upper_bound}--\ref{lem:existence_canonical} for instances where we can simplify our calculations by applying the triangle inequality.
\begin{proof}[Proof of Lemma~\ref{lem:ypsilon_upper_bound}]
	Let $s,t\in[0,T]$, $x\in\mbT^{2}$ and $q\in\mbN_{-1}$. An application of~\eqref{eq:isometry_bound_delta=0} yields
	\begin{equation*}
		\mbE[\abs{\Delta_{q}\Ypsilon{10}(t,x)-\Delta_{q}\Ypsilon{10}(s,x)}^2]\leq\norm{\het}_{C_{T}L^{\infty}}^{4}\mbE[\abs{\Delta_{q}\Ypsilon{1}(t,x)-\Delta_{q}\Ypsilon{1}(s,x)}^2]
	\end{equation*}
	where $\Ypsilon{1}$ is defined by
	\begin{equation*}
		\begin{split}
			\hat{\Ypsilon{1}}(t,\om)&\defeq\sum_{\substack{\om_1,\om_2\in\mbZ^{2}\\\om=\om_1+\om_2}}\sum_{j_1,j_2,j_3=1}^{2}\int_{0}^{t}\dd u_3\int_{0}^{u_3}\dd W^{j_2}(u_2,\om_2)\int_{0}^{u_2}\dd W^{j_1}(u_1,\om_1)\\
			&\quad\times\sum_{\varsigma\in\Sigma(1,2)}H^{j_3}_{t-u_3}(\om_{\varsigma(1)}+\om_{\varsigma(2)})H^{j_{\varsigma(1)}}_{u_3-u_{\varsigma(1)}}(\om_{\varsigma(1)})G^{j_3}(\om_{\varsigma(2)})H^{j_{\varsigma(2)}}_{u_3-u_{\varsigma(2)}}(\om_{\varsigma(2)}).
		\end{split}
	\end{equation*}
	Using that
	\begin{equation*}
		\mbE\Bigl[\hat{\Ypsilon{1}}(t,\om)\overline{\hat{\Ypsilon{1}}(s,\om')}\Bigr]=0\quad\text{if}\quad\om\neq\om'\in\mbZ^{2},
	\end{equation*}
	 we obtain
	\begin{equation*}
		\mbE[\abs{\Delta_{q}\Ypsilon{1}(t,x)-\Delta_{q}\Ypsilon{1}(s,x)}^2]=\sum_{\om\in\mbZ^{2}}\varrho_{q}(\om)^{2}\mbE\Bigl[\Bigl\lvert\hat{\Ypsilon{1}}(t,\om)-\hat{\Ypsilon{1}}(s,\om)\Bigr\rvert^2\Bigr].
	\end{equation*}
	It follows by an application of It\^{o}'s isometry and Jensen's inequality,
	\begin{equation*}
		\begin{split}
			&\mbE\Bigl[\Bigl\lvert\hat{\Ypsilon{1}}(t,\om)-\hat{\Ypsilon{1}}(s,\om)\Bigr\rvert^2\Bigr]\\
			&\leq\sum_{\substack{\om_1,\om_2\in\mbZ^{2}\\\om=\om_1+\om_2}}\sum_{j_1,j_2,j_3,j_3'=1}^{2}2!\int_{-\infty}^{\infty}\dd u_2\int_{-\infty}^{\infty}\dd u_1\int_{0}^{\infty}\dd u_3\int_{0}^{\infty}\dd u_3'\\
			&\quad(H^{j_3}_{t-u_3}(\om_{1}+\om_{2})-H^{j_3}_{s-u_3}(\om_{1}+\om_{2}))H^{j_{1}}_{u_3-u_{1}}(\om_{1})G^{j_3}(\om_{2})H^{j_{2}}_{u_3-u_{2}}(\om_{2})\\
			&\qquad\times\overline{(H^{j_3'}_{t-u_3'}(\om_{1}+\om_{2})-H^{j_3'}_{s-u_3'}(\om_{1}+\om_{2}))H^{j_{1}}_{u_3'-u_{1}}(\om_{1})G^{j_3'}(\om_{2})H^{j_{2}}_{u_3'-u_{2}}(\om_{2})},
		\end{split}
	\end{equation*}
	where we used that the complex absolute value of $z\in\mbC$ is given by $\abs{z}^{2}=z\overline{z}$.
	\begin{details}
		We apply It\^{o}'s isometry to derive for $s,t\in[0,T]$,
		\begin{equation*}
			\begin{split}
				&\mbE\Bigl[\Bigl\lvert\hat{\Ypsilon{1}}(t,\om)-\hat{\Ypsilon{1}}(s,\om)\Bigr\rvert^2\Bigr]\\
				&=\sum_{\substack{\om_1,\om_2\in\mbZ^{2}\\\om=\om_1+\om_2}}\sum_{j_1,j_2=1}^{2}\int_{0}^{\infty}\dd u_2\int_{0}^{u_2}\dd u_1\\
				&\Bigl\lvert\sum_{\varsigma\in\Sigma(1,2)}\sum_{j_3=1}^{2}\int_{0}^{\infty}\dd u_3(H^{j_3}_{t-u_3}(\om_{1}+\om_{2})-H^{j_3}_{s-u_3}(\om_{1}+\om_{2}))H^{j_{\varsigma(1)}}_{u_3-u_{\varsigma(1)}}(\om_{\varsigma(1)})G^{j_3}(\om_{\varsigma(2)})H^{j_{\varsigma(2)}}_{u_3-u_{\varsigma(2)}}(\om_{\varsigma(2)}) \Bigr\rvert^2.
			\end{split}
		\end{equation*}
		We use that the integrand is symmetric, implying
		\begin{equation*}
			\int_{0}^{\infty}\dd u_2\int_{0}^{u_2}\dd u_1=\int_{0}^{\infty}\dd u_2\int_{0}^{\infty}\dd u_1\mathds{1}_{u_1\leq u_2}=\frac{1}{2!}\int_{0}^{\infty}\dd u_2\int_{0}^{\infty}\dd u_1.
		\end{equation*}
		In order to remove the symmetrization, we apply Jensen's inequality to estimate
		\begin{equation*}
			\begin{split}
				&\mbE\Bigl[\Bigl\lvert\hat{\Ypsilon{1}}(t,\om)-\hat{\Ypsilon{1}}(s,\om)\Bigr\rvert^2\Bigr]\\
				&\leq\sum_{\substack{\om_1,\om_2\in\mbZ^{2}\\\om=\om_1+\om_2}}\sum_{j_1,j_2=1}^{2}2!\int_{0}^{\infty}\dd u_2\int_{0}^{\infty}\dd u_1\\
				&\quad\Bigl\lvert\sum_{j_3=1}^{2}\int_{0}^{\infty}\dd u_3(H^{j_3}_{t-u_3}(\om_{1}+\om_{2})-H^{j_3}_{s-u_3}(\om_{1}+\om_{2}))H^{j_{1}}_{u_3-u_{1}}(\om_{1})G^{j_3}(\om_{2})H^{j_{2}}_{u_3-u_{2}}(\om_{2})\Bigr\rvert^2.
			\end{split}
		\end{equation*}
		To simplify the resulting expressions, we extend the integrals over $u_1$, $u_2$ to $(-\infty,\infty)$,
		\begin{equation*}
			\begin{split}
				&\mbE\Bigl[\Bigl\lvert\hat{\Ypsilon{1}}(t,\om)-\hat{\Ypsilon{1}}(s,\om)\Bigr\rvert^2\Bigr]\\
				&\leq\sum_{\substack{\om_1,\om_2\in\mbZ^{2}\\\om=\om_1+\om_2}}\sum_{j_1,j_2=1}^{2}2!\int_{-\infty}^{\infty}\dd u_2\int_{-\infty}^{\infty}\dd u_1\\
				&\quad\Bigl\lvert\sum_{j_3=1}^{2}\int_{0}^{\infty}\dd u_3(H^{j_3}_{t-u_3}(\om_{1}+\om_{2})-H^{j_3}_{s-u_3}(\om_{1}+\om_{2}))H^{j_{1}}_{u_3-u_{1}}(\om_{1})G^{j_3}(\om_{2})H^{j_{2}}_{u_3-u_{2}}(\om_{2})\Bigr\rvert^2.
			\end{split}
		\end{equation*}
	\end{details}
	Recalling the definition of $\mathsf{D}_{s,t}\SYpsilon$ from~\eqref{eq:increment_shape_coefficient_ypsilon}, we obtain
	\begin{equation*}
		\mbE\Bigl[\Bigl\lvert\hat{\Ypsilon{1}}(t,\om)-\hat{\Ypsilon{1}}(s,\om)\Bigr\rvert^2\Bigr]\leq2!(2\uppi)^{4}\sum_{\substack{\om_1,\om_2\in\mbZ^{2}\setminus\{0\}\\\om=\om_1+\om_2}}\abs{\om_1}^{2}\abs{\om_2}^{-2}\abs{\inner{\om_2}{\om_1+\om_2}}^2\mathsf{D}_{s,t}\SYpsilon(2\uppi\om_1,2\uppi\om_2).
	\end{equation*}
	This yields the claim.
\end{proof}
To evaluate the integrals in $\mathsf{S}_{s,t}\SYpsilon$, we use a case distinction over $(\om_1\perp\om_2)$ and $\lnot(\om_1\perp\om_2)$. We can then find explicit expressions for $\mathsf{D}_{s,t}\SYpsilon$ via~\eqref{eq:increment_shape_coefficient_ypsilon}, which can be used to derive the necessary bounds. This is the content of Lemma~\ref{lem:difference_ypsilon_bound}. We can now give the proof of Lemma~\ref{lem:existence_ypsilon}.
\begin{proof}[Proof of Lemma~\ref{lem:existence_ypsilon}]
	Let $T>0$ and $\gamma\in[0,1]$. To decompose the right-hand side of~\eqref{eq:ypsilon_upper_bound}, we introduce the orthogonal sum
	\begin{equation*}
		\boldsymbol{\mathrm{E}}^{\perp}\Bigl[\Bigl\lvert\hat{\Ypsilon{1}}(t,\om)-\hat{\Ypsilon{1}}(s,\om)\Bigr\rvert^2\Bigr]\defeq\sum_{\substack{\om_1,\om_2\in\mbZ^{2}\setminus\{0\}\\\om=\om_1+\om_2\\\om_1\perp\om_2}}\abs{\om_1}^{2}\abs{\om_2}^{-2}\abs{\inner{\om_2}{\om_1+\om_2}}^2\mathsf{D}_{s,t}\SYpsilon(2\uppi\om_1,2\uppi\om_2)
	\end{equation*}
	and the non-orthogonal sum
	\begin{equation*}
		\boldsymbol{\mathrm{E}}^{\lnot}\Bigl[\Bigl\lvert\hat{\Ypsilon{1}}(t,\om)-\hat{\Ypsilon{1}}(s,\om)\Bigr\rvert^2\Bigr]\defeq\sum_{\substack{\om_1,\om_2\in\mbZ^{2}\setminus\{0\}\\\om=\om_1+\om_2\\\lnot(\om_1\perp\om_2)}}\abs{\om_1}^{2}\abs{\om_2}^{-2}\abs{\inner{\om_2}{\om_1+\om_2}}^2\mathsf{D}_{s,t}\SYpsilon(2\uppi\om_1,2\uppi\om_2).
	\end{equation*}
	We obtain the decomposition
	\begin{equation*}
		\begin{split}
			&\mbE[\abs{\Delta_{q}\Ypsilon{10}(t,x)-\Delta_{q}\Ypsilon{10}(s,x)}^2]\\
			&\leq\norm{\het}_{C_{T}L^{\infty}}^{4}2!(2\uppi)^{4}\sum_{\om\in\mbZ^{2}\setminus\{0\}}\varrho_{q}(\om)^{2}\Bigl(\boldsymbol{\mathrm{E}}^{\perp}\Bigl[\Bigl\lvert\hat{\Ypsilon{1}}(t,\om)-\hat{\Ypsilon{1}}(s,\om)\Bigr\rvert^2\Bigr]+\boldsymbol{\mathrm{E}}^{\lnot}\Bigl[\Bigl\lvert\hat{\Ypsilon{1}}(t,\om)-\hat{\Ypsilon{1}}(s,\om)\Bigr\rvert^2\Bigr]\Bigr),
		\end{split}
	\end{equation*}
	where we used that $\hat{\Ypsilon{1}}(t,0)=0$. In the orthogonal sum $\boldsymbol{\mathrm{E}}^{\perp}$, we obtain by Lemma~\ref{lem:difference_ypsilon_bound},
	\begin{equation*}
		\mathsf{D}_{s,t}\SYpsilon(2\uppi\om_1,2\uppi\om_2)\lesssim\abs{t-s}^{\gamma}\abs{\om_1}^{-2}\abs{\om_2}^{-2}\abs{\om_1+\om_2}^{-4+2\gamma},
	\end{equation*}
	so that
	\begin{equation*}
		\boldsymbol{\mathrm{E}}^{\perp}\Bigl[\Bigl\lvert\hat{\Ypsilon{1}}(t,\om)-\hat{\Ypsilon{1}}(s,\om)\Bigr\rvert^2\Bigr]\lesssim\abs{t-s}^{\gamma}\abs{\om}^{-4+2\gamma}\sum_{\substack{\om_1,\om_2\in\mbZ^{2}\setminus\{0\}\\\om=\om_1+\om_2\\\om_1\perp\om_2}}1,
	\end{equation*}
	where we used the orthogonality $(\om_1\perp\om_2)$ to identify $\abs{\inner{\om_2}{\om_1+\om_2}}^2=\abs{\om_2}^{4}$. Using the orthogonality again, we obtain the bound $\abs{\om_1}^{2}\leq\abs{\om_1}^{2}+\abs{\om_2}^{2}=\abs{\om}^{2}$. By applying~\eqref{eq:summation_estimates_0} to the finite sum over $\om_1\in\mbZ^{2}\setminus\{0\}$, $\abs{\om_1}\leq\abs{\om}$, we arrive at
	\begin{equation*}
		\boldsymbol{\mathrm{E}}^{\perp}\Bigl[\Bigl\lvert\hat{\Ypsilon{1}}(t,\om)-\hat{\Ypsilon{1}}(s,\om)\Bigr\rvert^2\Bigr]\lesssim\abs{t-s}^{\gamma}\abs{\om}^{-2+2\gamma}.
	\end{equation*}
	Next we consider the non-orthogonal sum $\boldsymbol{\mathrm{E}}^{\lnot}$. Lemma~\ref{lem:difference_ypsilon_bound} yields
	\begin{equation*}
		\mathsf{D}_{s,t}\SYpsilon(2\uppi\om_1,2\uppi\om_2)\lesssim\abs{t-s}^{\gamma}\abs{\om_1}^{-4+2\gamma}\abs{\om_2}^{-2}\abs{\om_1+\om_2}^{-2}+\abs{t-s}^{\gamma}\abs{\om_1}^{-4}\abs{\om_2}^{-2}\abs{\om_1+\om_2}^{-2+2\gamma},
	\end{equation*}
	so that by Lemma~\ref{lem:convolution_estimates} it follows that for any $\gamma\in(0,1)$ and $\eps\in(0,(2-2\gamma\wedge 1))$,
	\begin{equation*}
		\begin{split}
			\boldsymbol{\mathrm{E}}^{\lnot}\Bigl[\Bigl\lvert\hat{\Ypsilon{1}}(t,\om)-\hat{\Ypsilon{1}}(s,\om)\Bigr\rvert^2\Bigr]&\lesssim\abs{t-s}^{\gamma}\sum_{\substack{\om_1,\om_2\in\mbZ^{2}\setminus\{0\}\\\om=\om_1+\om_2\\\lnot(\om_1\perp\om_2)}}(\abs{\om_1}^{-2+2\gamma}\abs{\om_2}^{-2}+\abs{\om}^{2\gamma}\abs{\om_1}^{-2}\abs{\om_2}^{-2})\\
			&\lesssim\abs{t-s}^{\gamma}\abs{\om}^{-2+2\gamma+2\eps},
		\end{split}
	\end{equation*}
	where we used the Cauchy--Schwarz inequality to control $\abs{\inner{\om_2}{\om_1+\om_2}}^2\lesssim\abs{\om_2}^{2}\abs{\om_1+\om_2}^{2}$. Applying these results to the bound~\eqref{eq:ypsilon_upper_bound}, we arrive at
	\begin{equation*}
		\mbE[\abs{\Delta_{q}\Ypsilon{10}(t,x)-\Delta_{q}\Ypsilon{10}(s,x)}^2]\lesssim\abs{t-s}^{\gamma}\norm{\het}_{C_{T}L^{\infty}}^{4}2^{q(2\gamma+2\eps)},
	\end{equation*}
	which is uniform in $x\in\mbT^{2}$. Assume in addition $\eps<\gamma/2$. Using that $\Delta_q\Ypsilon{10}$ denotes an iterated It\^{o} integral, we obtain by Lemma~\ref{lem:Nelson_estimate} and Lemma~\ref{lem:existence_criterion} for any $p\in[1,\infty)$,
	\begin{equation*}
		\mbE[\norm{\Ypsilon{10}}_{C_T^{\gamma/2-\eps}\mcC^{-\gamma-4\eps}}^{p}]^{1/p}\lesssim\norm{\het}_{C_TL^{\infty}}^{2}
	\end{equation*}
	and therefore $\Ypsilon{10}\in\msL^{\kappa}_{T}\mcC^{0-}(\mbT^{2})$ a.s.\ for any $\kappa\in(0,1/2)$.
\end{proof}
\begin{details}
	\paragraph{Proof that $\ty^{\delta}\in\msL_{T}^{\kappa}\mcC^{2\alpha+4}(\mbT^{2})$.}
	We give a generic argument that shows that the mollified diagram has at least the same regularity as its un-mollified counterpart. An application of~\eqref{eq:isometry_bound_delta>0} yields the estimate
	\begin{equation*}
		\mbE[\abs{\Delta_{q}\ty^{\delta}(t,x)-\Delta_{q}\ty^{\delta}(s,x)}^2]\leq\norm{\varphi}_{L^{\infty}}^{4}\norm{\het}_{C_{T}L^{\infty}}^{4}\mbE[\abs{\Delta_{q}\Ypsilon{1}(t,x)-\Delta_{q}\Ypsilon{1}(s,x)}^2].
	\end{equation*}
	Hence, it follows by the proof of Lemma~\ref{lem:existence_ypsilon} that $\ty^{\delta}\in\msL_{T}^{\kappa}\mcC^{2\alpha+4}(\mbT^{2})$.
	\paragraph{Proof that $\ty^{\delta}\in\msL_{T}^{\kappa}\mcC^{2\alpha+5}(\mbT^{2})$.}
	We have shown in Lemma~\ref{lem:existence_lolli} that $\ti^{\delta}\in\msL_{T}^{\kappa}\mcC^{\alpha+2}(\mbT^{2})$. It then follows by the definition $\ty^{\delta}=\vdiv\mcI[\ti^{\delta}\nabla\Phi_{\ti^{\delta}}]-\tl^{\delta}$ and the triangle inequality, that
	\begin{equation*}
		\norm{\ty^{\delta}}_{\msL_{T}^{\kappa}\mcC^{2\alpha+5}}\leq\norm{\vdiv\mcI[\ti^{\delta}\nabla\Phi_{\ti^{\delta}}]}_{\msL_{T}^{\kappa}\mcC^{2\alpha+5}}+\norm{\tl^{\delta}}_{\msL_{T}^{\kappa}\mcC^{2\alpha+5}},
	\end{equation*}
	where the second term $\norm{\tl^{\delta}}_{\msL_{T}^{\kappa}\mcC^{2\alpha+5}}$ is finite by Lemma~\ref{lem:existence_canonical}. To estimate the first term, we apply Schauder's estimate (Lemma~\ref{lem:Schauder}) and Bony's estimates (Lemma~\ref{lem:Bony}), to bound
	\begin{equation*}
		\norm{\vdiv\mcI[\ti^{\delta}\nabla\Phi_{\ti^{\delta}}]}_{\msL_{T}^{\kappa}\mcC^{2\alpha+5}}\lesssim\norm{\ti^{\delta}\nabla\Phi_{\ti^{\delta}}}_{C_{T}\mcC^{2\alpha+4}}\lesssim\norm{\ti^{\delta}\nabla\Phi_{\ti^{\delta}}}_{C_{T}\mcC^{\alpha+2}}\lesssim\norm{\ti^{\delta}}_{C_{T}\mcC^{\alpha+2}}^{2},
	\end{equation*}
	where we used that $2\alpha+5>0$.
	\paragraph{Convergence of $(\ty^{\delta})_{\delta>0}$ to $\ty$ in $\msL_{T}^{\kappa}\mcC^{2\alpha+4}(\mbT^{2})$.}
	It suffices to show $\lim_{\delta\to0}\mbE[\norm{\ty-\ty^{\delta}}_{\msL_{T}^{\kappa}\mcC^{2\alpha+4}}^{p}]=0$. Let $\gamma\in(0,1]$, $\eps\in(0,\gamma/2)$, $\max\{1/\eps,2\}<p<\infty$ and $0<\vartheta'<\vartheta\leq1$ be such that $2\vartheta<1-\gamma$. We use Nelson's estimate (Lemma~\ref{lem:Nelson_estimate}) and Kolmogorov's continuity criterion (Lemma~\ref{lem:existence_criterion}) to bound
	\begin{equation}\label{eq:ypsilon_Kolmogorov_Nelson_moll_difference}
		\begin{split}
			&\mbE[\norm{\ty-\ty^{\delta}}_{C_{T}^{\gamma/2-\eps}\mcC^{-\gamma-2\vartheta-3\eps}}^{p}]\\
			&\lesssim\int_{0}^{T}\int_{0}^{T}\abs{t-s}^{-2-p(\gamma/2-\eps)}\sum_{q\in\mbN_{-1}}2^{pq(-\gamma-2\vartheta-\eps)}\int_{\mbT^{2}}\mbE[\abs{\Delta_q(\ty-\ty^{\delta})(t,x)-\Delta_q(\ty-\ty^{\delta})(s,x)}^2]^{p/2}\dd x\dd s\dd t.
		\end{split}
	\end{equation}
	We use~\eqref{eq:isometry_difference} and the representation of $\ty$ as an iterated It\^{o} integral to estimate
	\begin{equation*}
		\begin{split}
			&\mbE[\abs{\Delta_{q}(\ty-\ty^{\delta})(t,x)-\Delta_{q}(\ty-\ty^{\delta})(s,x)}^{2}]\\
			&=\mbE[\abs{I^{2}_{0;\het}(\phi_{t,x}^{q})-I^{2}_{\delta;\het}(\phi_{t,x}^{q})-(I^{2}_{0;\het}(\phi_{s,x}^{q})-I^{2}_{\delta;\het}(\phi_{s,x}^{q}))}^{2}]\\
			&=\mbE[\abs{I^{2}_{0;\het}(\phi^{q}_{t,x}-\phi^{q}_{s,x})-I^{2}_{\delta;\het}(\phi^{q}_{t,x}-\phi^{q}_{s,x})}^{2}]\\
			&\lesssim\delta^{2\vartheta'}\norm{1-\varphi}_{C^{1}(\mbR^{2})}^{2}(1+\norm{\varphi}_{L^{\infty}})^{2}\norm{\het}_{C_{T}\mcH^{1}}^{4}\mbE[\abs{I^{2}((1-\Delta)^{\vartheta/2}(\phi^{q}_{t,x}-\phi^{q}_{s,x}))}^{2}].
		\end{split}
	\end{equation*} 
	By It\^{o}'s isometry (Lemma~\ref{lem:Ito_isometry}),
	\begin{equation*}
		\begin{split}
			&\mbE[\abs{I^{2}((1-\Delta)^{\vartheta/2}(\phi^{q}_{t,x}-\phi^{q}_{s,x}))}^{2}]\\
			&\lesssim\sum_{\om_{1},\om_{2}\in\mbZ^{2}}\sum_{j_{1},j_{2}=1}^{2}\int_{0}^{\infty}\dd u_{1}\int_{0}^{u_{1}}\dd u_{2}(1+\abs{2\uppi\om_{1}}^{2}+\abs{1\uppi\om_{2}}^{2})^{\vartheta}\abs{(\hat{\phi^{q}_{t,x}}-\hat{\phi^{q}_{s,x}})(u_{1},-\om_{1},j_{1},u_{2},-\om_{2},j_{2})}^{2}.
		\end{split}
	\end{equation*}
	We use the explicit form of $\hat{\phi_{t,x}^{q}}$, Jensen's inequality and Lemma~\ref{lem:ypsilon_upper_bound} to bound
	\begin{equation*}
		\begin{split}
			&\sum_{\om_{1},\om_{2}\in\mbZ^{2}}\sum_{j_{1},j_{2}=1}^{2}\int_{0}^{\infty}\dd u_{1}\int_{0}^{u_{1}}\dd u_{2}(1+\abs{2\uppi\om_{1}}^{2}+\abs{1\uppi\om_{2}}^{2})^{\vartheta}\abs{(\hat{\phi^{q}_{t,x}}-\hat{\phi^{q}_{s,x}})(u_{1},-\om_{1},j_{1},u_{2},-\om_{2},j_{2})}^{2}\\
			&=\sum_{\om_{1},\om_{2}\in\mbZ^{2}}\sum_{j_{1},j_{2}=1}^{2}\int_{0}^{\infty}\dd u_{1}\int_{0}^{u_{1}}\dd u_{2}(1+\abs{2\uppi\om_{1}}^{2}+\abs{1\uppi\om_{2}}^{2})^{\vartheta}\varrho_{q}(\om_{1}+\om_{2})^{2}\\
			&\quad\times\Bigl\lvert\sum_{\varsigma\in\Sigma(1,2)}\sum_{j_{3}=1}^{2}\int_{0}^{\infty}\dd u_{3}(H^{j_{3}}_{t-u_{3}}(\om_{1}+\om_{2})-H^{j_{3}}_{s-u_{3}}(\om_{1}+\om_{2}))H^{j_{\varsigma(1)}}_{u_{3}-u_{\varsigma(1)}}(\om_{\varsigma(1)})G^{j_{3}}(\om_{\varsigma(2)})H^{j_{\varsigma(2)}}_{u_{3}-u_{\varsigma(2)}}(\om_{\varsigma(2)})\Bigr\rvert^{2}\\
			&\leq\sum_{\om\in\mbZ^{2}}\varrho_{q}(\om)^{2}\sum_{\substack{\om_{1},\om_{2}\in\mbZ^{2}\\\om=\om_{1}+\om_{2}}}\sum_{j_{1},j_{2}=1}^{2}2!\int_{-\infty}^{\infty}\dd u_{2}\int_{-\infty}^{\infty}\dd u_{1}(1+\abs{2\uppi\om_{1}}^{2}+\abs{2\uppi\om_{2}}^{2})^{\vartheta}\\
			&\quad\times\Bigl\lvert\sum_{j_{3}=1}^{2}\int_{0}^{\infty}\dd u_{3}(H^{j_{3}}_{t-u_{3}}(\om_{1}+\om_{2})-H^{j_{3}}_{s-u_{3}}(\om_{1}+\om_{2}))H^{j_{1}}_{u_{3}-u_{1}}(\om_{1})G^{j_{3}}(\om_{2})H^{j_{2}}_{u_{3}-u_{2}}(\om_{1})\Bigr\rvert^{2}\\
			&\leq\sum_{\om\in\mbZ^{2}}\varrho_{q}(\om)^{2}\sum_{\substack{\om_{1},\om_{2}\in\mbZ^{2}\\\om=\om_{1}+\om_{2}}}(1+\abs{2\uppi\om_{1}}^{2}+\abs{2\uppi\om_{2}}^{2})^{\vartheta}2!(2\uppi)^{4}\abs{\om_{1}}^{2}\abs{\om_{2}}^{-2}\abs{\inner{\om_{2}}{\om_{1}+\om_{2}}}^{2}\mathsf{D}_{s,t}\SYpsilon(2\uppi\om_1,2\uppi\om_2).
		\end{split}
	\end{equation*}
	To summarize, by an application of~\eqref{eq:isometry_difference} and Lemma~\ref{lem:ypsilon_upper_bound} we may control
	\begin{equation}\label{eq:ypsilon_upper_bound_moll_difference}
		\begin{split}
			&\mbE[\abs{\Delta_{q}(\ty-\ty^{\delta})(t,x)-\Delta_{q}(\ty-\ty^{\delta})(s,x)}^2]\\
			&\lesssim\delta^{2\vartheta'}\norm{1-\varphi}_{C^{1}(\mbR^{2})}^{2}(1+\norm{\varphi}_{L^{\infty}})^{2}\norm{\het}_{C_{T}\mcH^{1}}^{4}\sum_{\om\in\mbZ^{2}}\varrho_{q}(\om)^{2}\\
			&\quad\times\sum_{\substack{\om_1,\om_2\in\mbZ^{2}\setminus\{0\}\\\om=\om_1+\om_2}}(1+\abs{\om_{1}}^{2}+\abs{\om_{2}}^{2})^{\vartheta}\abs{\om_1}^{2}\abs{\om_2}^{-2}\abs{\inner{\om_2}{\om_1+\om_2}}^2\mathsf{D}_{s,t}\SYpsilon(2\uppi\om_1,2\uppi\om_2).
		\end{split}
	\end{equation}
	To avoid a case distinction, we estimate
	\begin{equation*}
		(1+\abs{\om_{1}}^{2}+\abs{\om_{2}}^{2})^{\vartheta}\leq(1+\abs{\om_{1}}^{2})^{\vartheta}(1+\abs{\om_{2}}^{2})^{\vartheta},
	\end{equation*}
	This does result in a loss of regularity of order $\vartheta$, which is negligible since $\vartheta$ can be chosen arbitrarily small.
	
	We continue to estimate~\eqref{eq:ypsilon_upper_bound_moll_difference} as in the proof of Lemma~\ref{lem:existence_ypsilon}. We obtain by Lemma~\ref{lem:difference_ypsilon_bound} for $(\om_{1}\perp\om_{2})$,
	\begin{equation*}
		\mathsf{D}_{s,t}\SYpsilon(2\uppi\om_1,2\uppi\om_2)\lesssim\abs{t-s}^{\gamma}\abs{\om_1}^{-2}\abs{\om_2}^{-2}\abs{\om_1+\om_2}^{-4+2\gamma},
	\end{equation*}
	so that
	\begin{equation*}
		\begin{split}
			&\sum_{\substack{\om_1,\om_2\in\mbZ^{2}\setminus\{0\}\\\om=\om_1+\om_2\\\om_{1}\perp\om_{2}}}(1+\abs{\om_{1}}^{2})^{\vartheta}(1+\abs{\om_{2}}^{2})^{\vartheta}\abs{\om_1}^{2}\abs{\om_2}^{-2}\abs{\inner{\om_2}{\om_1+\om_2}}^2\mathsf{D}_{s,t}\SYpsilon(2\uppi\om_1,2\uppi\om_2)\\
			&\lesssim\abs{t-s}^{\gamma}\abs{\om}^{-4+2\gamma+4\vartheta}\sum_{\substack{\om_1,\om_2\in\mbZ^{2}\setminus\{0\}\\\om=\om_1+\om_2\\\om_1\perp\om_2}}1\lesssim\abs{t-s}^{\gamma}\abs{\om}^{-4+2\gamma+4\vartheta}(1\vee\abs{\om}^{2})=\abs{t-s}^{\gamma}\abs{\om}^{-2+2\gamma+4\vartheta},
		\end{split}
	\end{equation*}
	where we used the orthogonality $(\om_1\perp\om_2)$ to identify $\abs{\inner{\om_2}{\om_1+\om_2}}^2=\abs{\om_2}^{4}$ and to bound $\abs{\om_1}^{2}\leq\abs{\om_1}^{2}+\abs{\om_2}^{2}=\abs{\om}^{2}$; and applied~\eqref{eq:summation_estimates_0} to the finite sum over $\om_1\in\mbZ^{2}\setminus\{0\}$ with $\abs{\om_1}\leq\abs{\om}$.
	
	Lemma~\ref{lem:difference_ypsilon_bound} yields for $\lnot(\om_{1}\perp\om_{2})$,
	\begin{equation*}
		\mathsf{D}_{s,t}\SYpsilon(2\uppi\om_1,2\uppi\om_2)\lesssim\abs{t-s}^{\gamma}\abs{\om_1}^{-4+2\gamma}\abs{\om_2}^{-2}\abs{\om_1+\om_2}^{-2}+\abs{t-s}^{\gamma}\abs{\om_1}^{-4}\abs{\om_2}^{-2}\abs{\om_1+\om_2}^{-2+2\gamma},
	\end{equation*}
	so that by Lemma~\ref{lem:convolution_estimates} for all $\gamma\in(0,1)$ and $2\vartheta<1-\gamma$,
	\begin{equation*}
		\begin{split}
			&\sum_{\substack{\om_1,\om_2\in\mbZ^{2}\setminus\{0\}\\\om=\om_1+\om_2\\\lnot(\om_{1}\perp\om_{2})}}(1+\abs{\om_{1}}^{2})^{\vartheta}(1+\abs{\om_{2}}^{2})^{\vartheta}\abs{\om_1}^{2}\abs{\om_2}^{-2}\abs{\inner{\om_2}{\om_1+\om_2}}^2\mathsf{D}_{s,t}\SYpsilon(2\uppi\om_1,2\uppi\om_2)\\
			&\lesssim\abs{t-s}^{\gamma}\sum_{\substack{\om_1,\om_2\in\mbZ^{2}\setminus\{0\}\\\om=\om_1+\om_2\\\lnot(\om_1\perp\om_2)}}(\abs{\om_1}^{-2+2\gamma+2\vartheta}\abs{\om_2}^{-2+2\vartheta}+\abs{\om}^{2\gamma}\abs{\om_1}^{-2+2\vartheta}\abs{\om_2}^{-2+2\vartheta})\\
			&\lesssim\abs{t-s}^{\gamma}\abs{\om}^{-2+2\gamma+4\vartheta},
		\end{split}
	\end{equation*}
	where we used the Cauchy--Schwarz inequality to control $\abs{\inner{\om_2}{\om_1+\om_2}}^2\lesssim\abs{\om_2}^{2}\abs{\om_1+\om_2}^{2}$.
	
	Applying these results to the bound~\eqref{eq:ypsilon_upper_bound_moll_difference}, we arrive at
	\begin{equation*}
		\mbE[\abs{\Delta_{q}(\ty-\ty^{\delta})(t,x)-\Delta_{q}(\ty-\ty^{\delta})(s,x)}^2]\lesssim\delta^{2\vartheta'}\abs{t-s}^{\gamma}\norm{1-\varphi}_{C^{1}(\mbR^{2})}^{2}(1+\norm{\varphi}_{L^{\infty}})^{2}\norm{\het}_{C_{T}\mcH^{1}}^{4}2^{q(2\gamma+4\vartheta)}
	\end{equation*}
	which is uniform in $x\in\mbT^{2}$. Plugging this into~\eqref{eq:ypsilon_Kolmogorov_Nelson_moll_difference} yields
	\begin{equation*}
		\begin{split}
			&\mbE[\norm{\ty-\ty^{\delta}}_{C_{T}^{\gamma/2-\eps}\mcC^{-\gamma-2\vartheta-3\eps}}^p]\\
			&\lesssim\delta^{p\vartheta'}\norm{1-\varphi}_{C^{1}(\mbR^{2})}^{p}(1+\norm{\varphi}_{L^{\infty}})^{p}\norm{\het}_{C_{T}\mcH^{1}}^{2p}\int_{0}^{T}\int_{0}^{T}\abs{t-s}^{-2+p\eps}\sum_{q\in\mbN_{-1}}2^{pq(-\gamma-2\vartheta-\eps)}2^{pq(\gamma+2\vartheta)}\dd s\dd t\\
			&\lesssim\delta^{p\vartheta'}\norm{1-\varphi}_{C^{1}(\mbR^{2})}^{p}(1+\norm{\varphi}_{L^{\infty}})^{p}\norm{\het}_{C_{T}\mcH^{1}}^{2p},
		\end{split}
	\end{equation*}
	which implies $\lim_{\delta\to0}\mbE[\norm{\ty-\ty^{\delta}}_{C_{T}^{\gamma/2-\eps}\mcC^{-\gamma-2\vartheta-3\eps}}^{p}]=0$ for every $p>\max\{1/\eps,2\}$. By an application of Markov's inequality, we can conclude that $\lim_{\delta\to0}\mbE[\norm{\ty-\ty^{\delta}}_{C_{T}^{\gamma/2-\eps}\mcC^{-\gamma-2\vartheta-3\eps}}^p]=0$ for every $p\in[1,\infty)$. Here, $\eps>0$, $\vartheta'>0$ and $\vartheta>\vartheta'$ can be chosen arbitrarily small, which yields convergence in $\msL_{T}^{1/2-}\mcC^{0-}(\mbT^{2})$.
	
\end{details}
Next we consider the third-order diagrams $\PreThree{10}$ and $\PreThree{20}$.
\begin{lemma}\label{lem:existence_three}
	Let $T>0$, $\alpha<-2$, $\kappa\in(0,1/2)$ and $\het\in C_{T}L^{\infty}(\mbT^{2})$. Then for any $p\in[1,\infty)$ we have
	\begin{equation*}
		\mbE[\norm{\PreThree{10}}_{\msL^{\kappa}_{T}\mcC^{3\alpha+6}}^p]^{1/p}+\mbE[\norm{\PreThree{20}}_{\msL^{\kappa}_{T}\mcC^{3\alpha+6}}^p]^{1/p}\lesssim\norm{\het}_{C_{T}L^{\infty}}^{3}
	\end{equation*}
	and in particular $\PreThree{10},\PreThree{20}\in\msL^{\kappa}_{T}\mcC^{3\alpha+6}(\mbT^{2};\mbR^{2})$ a.s.
\end{lemma}
\begin{proof}
	The proof of this lemma is similar to the one for Lemma~\ref{lem:existence_ypsilon}, so we only provide a sketch. The key idea is to consider the shape coefficient
	\begin{equation*}
		\mathsf{S}_{s,t}\SPreThree(\om_1,\om_2,\om_4)\defeq\mathsf{S}_{s,t}\SYpsilon(\om_1,\om_2)\mathsf{S}_{s,t}\SLolli(\om_4),\qquad s,t\geq0,\quad\om_1,\om_2,\om_4\in2\uppi\mbZ^2\setminus\{0\},
	\end{equation*}
	where the factor $\mathsf{S}_{s,t}\SYpsilon$ was already defined in~\eqref{eq:shape_coefficient_ypsilon}, and $\mathsf{S}_{s,t}\SLolli$ is given by
	\begin{equation*}
		\mathsf{S}_{s,t}\SLolli(\om_4)\defeq\int_{-\infty}^{s\wedge t}\dd u_4\euler^{-\abs{t-u_4}\abs{\om_4}^2}\euler^{-\abs{s-u_4}\abs{\om_4}^2}=\frac{1}{2}\abs{\om_4}^{-2}\euler^{-\abs{t-s}\abs{\om_4}^2}.
	\end{equation*}
	The factorization $\mathsf{S}_{s,t}\SPreThree=\mathsf{S}_{s,t}\SYpsilon\,\mathsf{S}_{s,t}\SLolli$ follows since there is no arrow \emph{pointing} at the common root between the vertices labelled by $u_3$ and $u_4$.
	
	We can then find explicit expressions for $\mathsf{D}_{s,t}\SPreThree$, which we use to bound the second moments of $\PreThree{10}$ and $\PreThree{20}$. The claim then follows by~\eqref{eq:isometry_bound_delta=0}, Lemma~\ref{lem:Nelson_estimate} and Lemma~\ref{lem:existence_criterion}.
\end{proof}
\begin{details}
	We define a shape coefficient for those diagrams.
	\begin{definition}
		Let $s,t\geq0$, $\om_1,\om_2,\om_4\in2\uppi\mbZ^2\setminus\{0\}$. We define the shape coefficient
		\begin{equation*}
			\mathsf{S}_{s,t}\SPreThree(\om_1,\om_2,\om_4)\defeq\mathsf{S}_{s,t}\SYpsilon(\om_1,\om_2)\mathsf{S}_{s,t}\SLolli(\om_4),
		\end{equation*}
		where the factor $\mathsf{S}_{s,t}\SYpsilon$ was already considered in~\eqref{eq:shape_coefficient_ypsilon}, and $\mathsf{S}_{s,t}\SLolli$ is given by
		\begin{equation*}
			\mathsf{S}_{s,t}\SLolli(\om_4)\defeq\int_{-\infty}^{s\wedge t}\dd u_4\euler^{-\abs{t-u_4}\abs{\om_4}^2}\euler^{-\abs{s-u_4}\abs{\om_4}^2}=\frac{1}{2}\abs{\om_4}^{-2}\euler^{-\abs{t-s}\abs{\om_4}^2}.
		\end{equation*}
		We define the increment shape coefficient by
		\begin{equation*}
			\mathsf{D}_{s,t}\SPreThree\defeq\mathsf{S}_{t,t}\SPreThree+\mathsf{S}_{s,s}\SPreThree-\mathsf{S}_{s,t}\SPreThree-\mathsf{S}_{t,s}\SPreThree.
		\end{equation*}
	\end{definition}
	We first consider $\PreThree{10}$, whose second moments we bound with the following lemma.
	\begin{lemma}
		Let $s,t\in[0,T]$, $x\in\mbT^{2}$, $j=1,2$ and $q\in\mbN_{-1}$. It holds that
		\begin{equation}\label{eq:prethree_upper_bound}
			\begin{split}
				&\mbE[\abs{\Delta_{q}\PreThree{10}(t,x,j)-\Delta_{q}\PreThree{10}(s,x,j)}^2]\\
				&\leq\norm{\het}_{C_{T}L^{\infty}}^{6}3!(2\uppi)^{4}\sum_{\om\in\mbZ^{2}}\varrho_{q}(\om)^{2}\sum_{\substack{\om_1,\om_2,\om_4\in\mbZ^{2}\setminus\{0\}\\\om_1+\om_2\in\mbZ^{2}\setminus\{0\}\\\om=\om_1+\om_2+\om_4\\(\om_1+\om_2)\sim\om_4}}\abs{\om_4^j}^2\abs{\om_4}^{-2}\abs{\om_2}^{-2}\abs{\om_1}^2\abs{\inner{\om_2}{\om_1+\om_2}}^{2}\\[-35pt]
				&\multiquad[27]\times\mathsf{D}_{s,t}\SPreThree(2\uppi\om_1,2\uppi\om_2,2\uppi\om_4).
			\end{split}
		\end{equation}
	\end{lemma} 
	\begin{proof}
		Let $s,t\in[0,T]$, $x\in\mbT^{2}$, $j=1,2$ and $q\in\mbN_{-1}$. An application of~\eqref{eq:isometry_bound_delta=0} yields
		\begin{equation*}
			\mbE[\abs{\Delta_{q}\PreThree{10}(t,x,j)-\Delta_{q}\PreThree{10}(s,x,j)}^2]\leq\norm{\het}_{C_{T}L^{\infty}}^{6}\mbE[\abs{\Delta_{q}\PreThree{1}(t,x,j)-\Delta_{q}\PreThree{1}(s,x,j)}^2],
		\end{equation*}
		where $\PreThree{1}$ is defined by
		\begin{equation*}
			\begin{split}
				\hat{\PreThree{1}}(t,\om,j)&\defeq\sum_{\substack{\om_1,\om_2,\om_4\in\mbZ^{2}\\\om=\om_1+\om_2+\om_4}}\sum_{j_1,j_2,j_3,j_4=1}^{2}\int_{0}^{t}\dd u_3\int_{0}^{t}\dd W^{j_4}(u_4,\om_4)\int_{0}^{u_4}\dd W^{j_1}(u_1,\om_1)\int_{0}^{u_1}\dd W^{j_2}(u_2,\om_2)\\
				&\quad\sum_{\varsigma\in\Sigma(1,2,4)}H_{t-u_3}^{j_3}(\om_{\varsigma(1)}+\om_{\varsigma(2)})
				G^{j}(\om_{\varsigma(4)})H_{t-u_{\varsigma(4)}}^{j_{\varsigma(4)}}(\om_{\varsigma(4)})H_{u_3-u_{\varsigma(1)}}^{j_{\varsigma(1)}}(\om_{\varsigma(1)})\\
				&\quad\times G^{j_3}(\om_{\varsigma(2)})H_{u_3-u_{\varsigma(2)}}^{j_{\varsigma(2)}}(\om_{\varsigma(2)})\sum_{\substack{k,l\in\mbN_{-1}\\\abs{k-l}\leq 1}}\varrho_k(\om_{\varsigma(1)}+\om_{\varsigma(2)})\varrho_l(\om_{\varsigma(4)}).
			\end{split}
		\end{equation*}
		Using that
		\begin{equation*}
			\mbE\Bigl[\hat{\PreThree{1}}(t,\om,j)\overline{\hat{\PreThree{1}}(s,\om',j)}\Bigr]=0\quad\text{if}\quad\om\neq\om'\in\mbZ^{2},
		\end{equation*}
		we obtain
		\begin{equation*}
			\mbE[\abs{\Delta_{q}\PreThree{1}(t,x,j)-\Delta_{q}\PreThree{1}(s,x,j)}^2]=\sum_{\om\in\mbZ^{2}}\varrho_{q}(\om)^{2}\mbE\Bigl[\Bigl\lvert\hat{\PreThree{1}}(t,\om,j)-\hat{\PreThree{1}}(s,\om,j)\Bigr\rvert^2\Bigr].
		\end{equation*}
		By It\^{o}'s isometry,
		\begin{equation*}
			\begin{split}
				&\mbE\Bigl[\Bigl\lvert\hat{\PreThree{1}}(t,\om,j)-\hat{\PreThree{1}}(s,\om,j)\Bigr\rvert^2\Bigr]\\
				&=\sum_{\substack{\om_1,\om_2,\om_4\in\mbZ^{2}\setminus\{0\}\\\om=\om_1+\om_2+\om_4
				}}\sum_{j_1,j_2,j_4=1}^{2}\int_{0}^{\infty}\dd u_4 \int_{0}^{u_4}\dd u_1\int_{0}^{u_1}\dd u_2\\
				&\quad\Bigl\lvert\sum_{\varsigma\in\Sigma(1,2,4)}\sum_{j_3=1}^{2}\int_{0}^{\infty}\dd u_3(H^{j_3}_{t-u_3}(\om_{\varsigma(1)}+\om_{\varsigma(2)})H^{j_{\varsigma(4)}}_{t-u_{\varsigma(4)}}(\om_{\varsigma(4)})-H^{j_3}_{s-u_3}(\om_{\varsigma(1)}+\om_{\varsigma(2)})H^{j_{\varsigma(4)}}_{s-u_{\varsigma(4)}}(\om_{\varsigma(4)}))\\
				&\quad\times H^{j_{\varsigma(1)}}_{u_3-u_{\varsigma(1)}}(\om_{\varsigma(1)})H^{j_{\varsigma(2)}}_{u_3-u_{\varsigma(2)}}(\om_{\varsigma(2)})G^{j}(\om_{\varsigma(4)})G^{j_3}(\om_{\varsigma(2)})\sum_{\substack{k,l\in\mbN_{-1}\\\abs{k-l}\leq 1}}\varrho_k(\om_{\varsigma(1)}+\om_{\varsigma(2)})\varrho_l(\om_{\varsigma(4)})\Bigr\rvert^2.
			\end{split}
		\end{equation*}
		By Jensen's inequality,
		\begin{equation*}
			\begin{split}
				&\mbE\Bigl[\Bigl\lvert\hat{\PreThree{1}}(t,\om,j)-\hat{\PreThree{1}}(s,\om,j)\Bigr\rvert^2\Bigr]\\
				&\leq\sum_{\substack{\om_1,\om_2,\om_4\in\mbZ^{2}\\\om=\om_1+\om_2+\om_4}
				}\sum_{j_1,j_2,j_4=1}^{2}3!\int_{-\infty}^{\infty}\dd u_4 \int_{-\infty}^{\infty}\dd u_1\int_{-\infty}^{\infty}\dd u_2\\
				&\quad\Bigl\lvert\sum_{j_3=1}^{2}\int_{0}^{\infty}\dd u_3(H^{j_3}_{t-u_3}(\om_{1}+\om_{2})H^{j_{4}}_{t-u_{4}}(\om_{4})-H^{j_3}_{s-u_3}(\om_{1}+\om_{2})H^{j_{4}}_{s-u_{4}}(\om_{4}))\\
				&\quad\times H^{j_{1}}_{u_3-u_{1}}(\om_{1})H^{j_{2}}_{u_3-u_{2}}(\om_{2})G^{j}(\om_{4})G^{j_3}(\om_{2})\sum_{\substack{k,l\in\mbN_{-1}\\\abs{k-l}\leq 1}}\varrho_k(\om_{1}+\om_{2})\varrho_l(\om_{4})\Bigr\rvert^2.
			\end{split}
		\end{equation*}
		We arrive at 
		\begin{equation*}
			\begin{split}
				&\mbE\Bigl[\Bigl\lvert\hat{\PreThree{1}}(t,\om,j)-\hat{\PreThree{1}}(s,\om,j)\Bigr\rvert^2\Bigr]\\
				&\leq\sum_{\substack{\om_1,\om_2,\om_4\in\mbZ^{2}\\\om=\om_1+\om_2+\om_4\\(\om_1+\om_2)\sim\om_4}
				}\sum_{j_1,j_2,j_3,j_3',j_4=1}^{2}3!\int_{-\infty}^{\infty}\dd u_4 \int_{-\infty}^{\infty}\dd u_1\int_{-\infty}^{\infty}\dd u_2\int_{0}^{\infty}\dd u_3\int_{0}^{\infty}\dd u_3'\\
				&\quad (H^{j_3}_{t-u_3}(\om_{1}+\om_{2})H^{j_{4}}_{t-u_{4}}(\om_{4})-H^{j_3}_{s-u_3}(\om_{1}+\om_{2})H^{j_{4}}_{s-u_{4}}(\om_{4}))H^{j_{1}}_{u_3-u_{1}}(\om_{1})H^{j_{2}}_{u_3-u_{2}}(\om_{2})G^{j}(\om_{4})G^{j_3}(\om_{2})\\ 
				&\quad\times\overline{(H^{j_3'}_{t-u_3'}(\om_{1}+\om_{2})H^{j_{4}}_{t-u_{4}}(\om_{4})-H^{j_3'}_{s-u_3'}(\om_{1}+\om_{2})H^{j_{4}}_{s-u_{4}}(\om_{4}))H^{j_{1}}_{u_3'-u_{1}}(\om_{1})H^{j_{2}}_{u_3'-u_{2}}(\om_{2})G^{j}(\om_{4})G^{j_3'}(\om_{2})}.
			\end{split}
		\end{equation*}
		By introducing the increment shape coefficient,
		\begin{equation*}
			\begin{split}
				&\mbE\Bigl[\Bigl\lvert\hat{\PreThree{1}}(t,\om,j)-\hat{\PreThree{1}}(s,\om,j)\Bigr\rvert^2\Bigr]\\
				&\leq3!(2\uppi)^{4}\sum_{\substack{\om_1,\om_2,\om_4\in\mbZ^{2}\setminus\{0\}\\\om_1+\om_2\in\mbZ^{2}\setminus\{0\}\\\om=\om_1+\om_2+\om_4\\(\om_1+\om_2)\sim\om_4}}\abs{\om_4^j}^2\abs{\om_4}^{-2}\abs{\om_2}^{-2}\abs{\om_1}^2\abs{\inner{\om_2}{\om_1+\om_2}}^{2}\mathsf{D}_{s,t}\SPreThree(2\uppi\om_1,2\uppi\om_2,2\uppi\om_4).
			\end{split}
		\end{equation*}
		This yields the claim.
	\end{proof}
	We bound the increment shape coefficient in an appendix, see Lemma~\ref{lem:difference_three_bound}. We can now prove Lemma~\ref{lem:existence_three}.
	\begin{proof}[Proof of Lemma~\ref{lem:existence_three}]
		Assume $T>0$ and $\gamma\in(0,1)$. To decompose the right-hand side of~\eqref{eq:prethree_upper_bound}, we introduce the orthogonal sum
		\begin{equation*}
			\begin{split}
				&\boldsymbol{\mathrm{E}}^{\perp}\Bigl[\Bigl\lvert\hat{\PreThree{1}}(t,\om,j)-\hat{\PreThree{1}}(s,\om,j)\Bigr\rvert^2\Bigr]\\
				&\defeq\sum_{\substack{\om_1,\om_2,\om_4\in\mbZ^{2}\setminus\{0\}\\\om_1+\om_2\in\mbZ^{2}\setminus\{0\}\\\om=\om_1+\om_2+\om_4\\(\om_1+\om_2)\sim\om_4\\\om_1\perp\om_2}}\abs{\om_4^j}^2\abs{\om_4}^{-2}
				\abs{\om_2}^{-2}\abs{\om_1}^2\abs{\inner{\om_2}{\om_1+\om_2}}^{2}\mathsf{D}_{s,t}\SPreThree(2\uppi\om_1,2\uppi\om_2,2\uppi\om_4),
			\end{split}
		\end{equation*}
		and the non-orthogonal sum
		\begin{equation*}
			\begin{split}
				&\boldsymbol{\mathrm{E}}^{\lnot}\Bigl[\Bigl\lvert\hat{\PreThree{1}}(t,\om,j)-\hat{\PreThree{1}}(s,\om,j)\Bigr\rvert^2\Bigr]\\
				&\defeq\sum_{\substack{\om_1,\om_2,\om_4\in\mbZ^{2}\setminus\{0\}\\\om_1+\om_2\in\mbZ^{2}\setminus\{0\}\\\om=\om_1+\om_2+\om_4\\(\om_1+\om_2)\sim\om_4\\\lnot(\om_1\perp\om_2)}}\abs{\om_4^j}^2\abs{\om_4}^{-2}\abs{\om_2}^{-2}\abs{\om_1}^2\abs{\inner{\om_2}{\om_1+\om_2}}^{2}\mathsf{D}_{s,t}\SPreThree(2\uppi\om_1,2\uppi\om_2,2\uppi\om_4).
			\end{split}
		\end{equation*}
		We obtain the decomposition
		\begin{equation*}
			\begin{split}
				&\mbE[\abs{\Delta_{q}\PreThree{10}(t,x,j)-\Delta_{q}\PreThree{10}(s,x,j)}^2]\\
				&\leq\norm{\het}_{C_{T}L^{\infty}}^{6}3!(2\uppi)^{4}\sum_{\om\in\mbZ^{2}}\varrho_{q}(\om)^{2}\Bigl(\boldsymbol{\mathrm{E}}^{\perp}\Bigl[\Bigl\lvert\hat{\PreThree{1}}(t,\om,j)-\hat{\PreThree{1}}(s,\om,j)\Bigr\rvert^2\Bigr]+\boldsymbol{\mathrm{E}}^{\lnot}\Bigl[\Bigl\lvert\hat{\PreThree{1}}(t,\om,j)-\hat{\PreThree{1}}(s,\om,j)\Bigr\rvert^2\Bigr]\Bigr).
			\end{split}
		\end{equation*}
		We start with the orthogonal sum $\boldsymbol{\mathrm{E}}^{\perp}$. By Lemma~\ref{lem:difference_three_bound},
		\begin{equation*}
			\mathsf{D}_{s,t}\SPreThree(2\uppi\om_1,2\uppi\om_2,2\uppi\om_4)\lesssim\abs{t-s}^{\gamma}\abs{\om_1}^{-2}\abs{\om_2}^{-2}\abs{\om_4}^{-6+2\gamma},
		\end{equation*}
		so that by~\eqref{eq:summation_estimates_0} and Lemma~\ref{lem:convolution_estimates},
		\begin{equation*}
			\begin{split}
				\boldsymbol{\mathrm{E}}^{\perp}\Bigl[\Bigl\lvert\hat{\PreThree{1}}(t,\om,j)-\hat{\PreThree{1}}(s,\om,j)\Bigr\rvert^2\Bigr]&\lesssim\abs{t-s}^{\gamma}\sum_{\substack{\om_3,\om_4\in\mbZ^{2}\setminus\{0\}\\\om=\om_3+\om_4\\\om_3\sim\om_4}}\abs{\om_4}^{-6+2\gamma}\sum_{\substack{\om_1,\om_2\in\mbZ^{2}\setminus\{0\}\\\om_3=\om_1+\om_2\\\om_1\perp\om_2}}1\\
				&\lesssim\abs{t-s}^{\gamma}\sum_{\substack{\om_3,\om_4\in\mbZ^{2}\setminus\{0\}\\\om=\om_3+\om_4\\\om_3\sim\om_4}}\abs{\om_4}^{-6+2\gamma}(1\vee\abs{\om_{3}}^{2})\lesssim\abs{t-s}^{\gamma}(1\vee\abs{\om})^{-2+2\gamma},
			\end{split}
		\end{equation*}
		where we used the orthogonality $(\om_1\perp\om_2)$ to identify $\abs{\inner{\om_2}{\om_1+\om_2}}^2=\abs{\om_2}^{4}$ and to bound $\abs{\om_1}^{2}\leq\abs{\om_1}^{2}+\abs{\om_2}^{2}=\abs{\om_{3}}^{2}$.
		
		Next we consider the non-orthogonal part. An application of Lemma~\ref{lem:difference_three_bound} yields for $\eps\in(0,1-\gamma)$,
		\begin{equation*}
			\begin{split}
				\mathsf{D}_{s,t}\SPreThree(2\uppi\om_1,2\uppi\om_2,2\uppi\om_4)&\lesssim\abs{t-s}^{\gamma}\abs{\om_1}^{-2-2\eps}\abs{\om_2}^{-2}\abs{\om_4}^{-6+2\gamma+2\eps}+\abs{t-s}^{\gamma}\abs{\om_1}^{-4}\abs{\om_2}^{-2}\abs{\om_4}^{-4+2\gamma}\\
				&\quad+\abs{t-s}^{\gamma}\abs{\om_1}^{-4+2\gamma}\abs{\om_2}^{-2}\abs{\om_4}^{-4},
			\end{split}
		\end{equation*}
		so that
		\begin{equation*}
			\begin{split}
				\boldsymbol{\mathrm{E}}^{\lnot}\Bigl[\Bigl\lvert\hat{\PreThree{1}}(t,\om,j)-\hat{\PreThree{1}}(s,\om,j)\Bigr\rvert^2\Bigr]&\lesssim\abs{t-s}^{\gamma}\sum_{\substack{\om_3,\om_4\in\mbZ^{2}\setminus\{0\}\\\om=\om_3+\om_4\\\om_3\sim\om_4}}\abs{\om_4}^{-4+2\gamma+2\eps}\sum_{\substack{\om_1,\om_2\in\mbZ^{2}\setminus\{0\}\\\om_3=\om_1+\om_2\\\lnot(\om_1\perp\om_2)}}\abs{\om_1}^{-2\eps}\abs{\om_2}^{-2}\\
				&\quad+\abs{t-s}^{\gamma}\sum_{\substack{\om_3,\om_4\in\mbZ^{2}\setminus\{0\}\\\om=\om_3+\om_4\\\om_3\sim\om_4}}\abs{\om_4}^{-2+2\gamma}\sum_{\substack{\om_1,\om_2\in\mbZ^{2}\setminus\{0\}\\\om_3=\om_1+\om_2\\\lnot(\om_1\perp\om_2)}}\abs{\om_1}^{-2}\abs{\om_2}^{-2}\\
				&\quad+\abs{t-s}^{\gamma}\sum_{\substack{\om_3,\om_4\in\mbZ^{2}\setminus\{0\}\\\om=\om_3+\om_4\\\om_3\sim\om_4}}\abs{\om_4}^{-2}\sum_{\substack{\om_1,\om_2\in\mbZ^{2}\setminus\{0\}\\\om_3=\om_1+\om_2\\\lnot(\om_1\perp\om_2)}}\abs{\om_1}^{-2+2\gamma}\abs{\om_2}^{-2},
			\end{split}
		\end{equation*}
		where we used the Cauchy--Schwarz inequality to control $\abs{\inner{\om_2}{\om_1+\om_2}}^2\lesssim\abs{\om_2}^{2}\abs{\om_1+\om_2}^{2}$ and $(\om_1+\om_2)\sim\om_4$ to bound $\abs{\om_1+\om_2}^{2}\lesssim\abs{\om_4}^{2}$.
		
		Assume $0<\delta/2<\eps<1-\gamma$. We estimate the first double sum by several applications of Lemma~\ref{lem:convolution_estimates},
		\begin{equation*}
			\sum_{\substack{\om_3,\om_4\in\mbZ^{2}\setminus\{0\}\\\om=\om_3+\om_4\\\om_3\sim\om_4}}\abs{\om_4}^{-4+2\gamma+2\eps}\sum_{\substack{\om_1,\om_2\in\mbZ^{2}\setminus\{0\}\\\om_3=\om_1+\om_2\\\lnot(\om_1\perp\om_2)}}\abs{\om_1}^{-2\eps}\abs{\om_2}^{-2}\lesssim\sum_{\substack{\om_3,\om_4\in\mbZ^{2}\setminus\{0\}\\\om=\om_3+\om_4\\\om_3\sim\om_4}}\abs{\om_4}^{-4+2\gamma+2\eps}\abs{\om_3}^{\delta-2\eps}\lesssim(1\vee\abs{\om})^{-2+2\gamma+\delta}.
		\end{equation*}
		The remaining sums can be bounded in a similar manner,
		\begin{equation*}
			\begin{split}
				&\sum_{\substack{\om_3,\om_4\in\mbZ^{2}\setminus\{0\}\\\om=\om_3+\om_4\\\om_3\sim\om_4}}\abs{\om_4}^{-2+2\gamma}\sum_{\substack{\om_1,\om_2\in\mbZ^{2}\setminus\{0\}\\\om_3=\om_1+\om_2\\\lnot(\om_1\perp\om_2)}}\abs{\om_1}^{-2}\abs{\om_2}^{-2}\leq\sum_{\substack{\om_3,\om_4\in\mbZ^{2}\setminus\{0\}\\\om=\om_3+\om_4\\\om_3\sim\om_4}}\abs{\om_4}^{-2+2\gamma}\sum_{\substack{\om_1,\om_2\in\mbZ^{2}\setminus\{0\}\\\om_3=\om_1+\om_2\\\lnot(\om_1\perp\om_2)}}\abs{\om_1}^{-2+\delta/2}\abs{\om_2}^{-2+\delta/2}\\
				&\lesssim\sum_{\substack{\om_3,\om_4\in\mbZ^{2}\setminus\{0\}\\\om=\om_3+\om_4\\\om_3\sim\om_4}}\abs{\om_4}^{-2+2\gamma}\abs{\om_3}^{-2+\delta}\lesssim(1\vee\abs{\om})^{-2+2\gamma+\delta}
			\end{split}
		\end{equation*}
		and
		\begin{equation*}
			\begin{split}
				&\sum_{\substack{\om_3,\om_4\in\mbZ^{2}\setminus\{0\}\\\om=\om_3+\om_4\\\om_3\sim\om_4}}\abs{\om_4}^{-2}\sum_{\substack{\om_1,\om_2\in\mbZ^{2}\setminus\{0\}\\\om_3=\om_1+\om_2\\\lnot(\om_1\perp\om_2)}}\abs{\om_1}^{-2+2\gamma}\abs{\om_2}^{-2}\leq\sum_{\substack{\om_3,\om_4\in\mbZ^{2}\setminus\{0\}\\\om=\om_3+\om_4\\\om_3\sim\om_4}}\abs{\om_4}^{-2}\sum_{\substack{\om_1,\om_2\in\mbZ^{2}\setminus\{0\}\\\om_3=\om_1+\om_2\\\lnot(\om_1\perp\om_2)}}\abs{\om_1}^{-2+2\gamma}\abs{\om_2}^{-2+\delta}\\
				&\lesssim\sum_{\substack{\om_3,\om_4\in\mbZ^{2}\setminus\{0\}\\\om=\om_3+\om_4\\\om_3\sim\om_4}}\abs{\om_4}^{-2}\abs{\om_3}^{-2+2\gamma+\delta}\lesssim(1\vee\abs{\om})^{-2+2\gamma+\delta}.
			\end{split}
		\end{equation*}
		To summarize,
		\begin{equation*}
			\boldsymbol{\mathrm{E}}^{\lnot}\Bigl[\Bigl\lvert\hat{\PreThree{1}}(t,\om,j)-\hat{\PreThree{1}}(s,\om,j)\Bigr\rvert^2\Bigr]\lesssim\abs{t-s}^{\gamma}(1\vee\abs{\om})^{-2+2\gamma+\delta}
		\end{equation*}
		and therefore uniformly in $x\in\mbT^{2}$,
		\begin{equation*}
			\mbE[\abs{\Delta_{q}\PreThree{10}(t,x,j)-\Delta_{q}\PreThree{10}(s,x,j)}^2]\lesssim\norm{\het}_{C_{T}L^{\infty}}^{6}\abs{t-s}^{\gamma}2^{q(2\gamma+\delta)}.
		\end{equation*}
		Assume in addition $\eps<\gamma/2$. We obtain by~\eqref{eq:isometry_bound_delta=0}, Lemma~\ref{lem:Nelson_estimate} and Lemma~\ref{lem:existence_criterion} for any $p\in[1,\infty)$,
		\begin{equation*}
			\mbE[\norm{\PreThree{10}}_{C_T^{\gamma/2-\eps}\mcC^{-\gamma-4\eps}}^{p}]^{1/p}\lesssim\norm{\het}_{C_TL^{\infty}}^{3}
		\end{equation*}
		and therefore $\PreThree{10}\in\msL^{\kappa}_{T}\mcC^{0-}(\mbT^{2};\mbR^{2})$ a.s.\ for any $\kappa\in(0,1/2)$.
		
		The next diagram is $\PreThree{20}$, whose covariance can be bounded by
		\begin{equation*}
			\begin{split}
				&\mbE[\abs{\Delta_{q}\PreThree{20}(t,x,j)-\Delta_{q}\PreThree{20}(s,x,j)}^2]\\
				&\leq\norm{\het}_{C_{T}L^{\infty}}^{6}3!(2\uppi)^{4}\sum_{\om\in\mbZ^{2}}\varrho_{q}(\om)^{2}\sum_{\substack{\om_1,\om_2,\om_4\in\mbZ^{2}\setminus\{0\}\\\om_1+\om_2\in\mbZ^{2}\setminus\{0\}\\\om=\om_1+\om_2+\om_4\\(\om_1+\om_2)\sim\om_4}}\frac{\abs{\om_1^j+\om_2^j}^{2}}{\abs{\om_1+\om_2}^{4}}\abs{\om_4}^{2}\abs{\om_2}^{-2}\abs{\om_1}^{2}\abs{\inner{\om_2}{\om_1+\om_2}}^2\\[-30pt]
				&\multiquad[30]\times\mathsf{D}_{s,t}\SPreThree(2\uppi\om_1,2\uppi\om_2,2\uppi\om_4).
			\end{split}
		\end{equation*}
		We see that $\PreThree{20}$ differs from $\PreThree{10}$ only in its colouring of the $\om_1+\om_2$ and $\om_4$ arrows. However, the sum of the exponents is preserved leading to the same overall regularity. Consequently by the same arguments as for $\PreThree{10}$, it follows that for any $T>0$, $p\in[1,\infty)$, $\gamma\in(0,1)$, $\eps\in(0,(1-\gamma)\wedge\gamma/2)$,
		\begin{equation*}
			\mbE[\norm{\PreThree{20}}_{C_T^{\gamma/2-\eps}\mcC^{-\gamma-4\eps}}^{p}]^{1/p}\lesssim\norm{\het}_{C_TL^{\infty}}^{3}
		\end{equation*}
		and therefore $\PreThree{20}\in\msL^{\kappa}_{T}\mcC^{0-}(\mbT^{2};\mbR^{2})$ a.s.\ for any $\kappa\in(0,1/2)$.
	\end{proof}
\end{details}
We can also show the existence of the diagrams $\Checkmark{10}$ and $\Checkmark{20}$.
\begin{lemma}\label{lem:existence_checkmark}
	Let $T>0$, $\alpha<-2$, $\kappa\in(0,1/2)$ and $\het\in C_{T}L^{\infty}(\mbT^{2})$. Then for any $p\in[1,\infty)$ we have
	\begin{equation*}
		\mbE[\norm{\Checkmark{10}}_{\msL^{\kappa}_{T}\mcC^{2\alpha+4}}^p]^{1/p}+\mbE[\norm{\Checkmark{20}}_{\msL^{\kappa}_{T}\mcC^{2\alpha+4}}^p]^{1/p}\lesssim\norm{\het}_{C_{T}L^{\infty}}^{2}
	\end{equation*}
	and in particular $\Checkmark{10},\Checkmark{20}\in\msL^{\kappa}_{T}\mcC^{2\alpha+4}(\mbT^{2};\mbR^{2\times2})$ a.s.
\end{lemma}
We define a shape coefficient for $\Checkmark{10}$ and $\Checkmark{20}$. Since those do not contain the problematic sub-diagram $\Cherry{10}$, it suffices to push the absolute value past the integral sign. We denote this fact by the letter $\mathsf{A}$ for absolute value. In particular, we may bound any integral over $[0,\infty)$ by $(-\infty,\infty)$, which simplifies our calculations.
\begin{definition}\label{def:shape_coefficient_vee}
	Let $s,t\geq0$, $\om_1,\om_1',\om_2\in\mbZ^{2}\setminus\{0\}$ and $k,k'=1,2$. We set
	\begin{equation*}
		\begin{split}
			&\mathsf{A}^{k,k'}_{s,t}\SVee(\om_1,\om_1',\om_2)\\
			&\defeq\sum_{j=1}^{2}\int_{-\infty}^{\infty}\dd u_2\int_{-\infty}^{\infty}\dd u_1\int_{-\infty}^{\infty}\dd u_1'\abs{H_{t-u_1}^{k}(\om_1)H_{t-u_2}^{j}(\om_2)-H_{s-u_1}^{k}(\om_1)H_{s-u_2}^{j}(\om_2)}\\
			&\multiquad[14]\times\abs{H_{t-u_1'}^{k'}(\om_1')H_{t-u_2}^{j}(\om_2)-H_{s-u_1'}^{k'}(\om_1')H_{s-u_2}^{j}(\om_2)}.
		\end{split}
	\end{equation*}
\end{definition}
We can then bound the second moment of $\Checkmark{10}$ in terms of this object.
\begin{lemma}\label{lem:checkmark_upper_bound}
	Let $s,t\in[0,T]$, $x\in\mbT^{2}$, $k,j=1,2$ and $q\in\mbN_{-1}$. It holds that
	\begin{equation*}
		\begin{split}
			&\mbE[\abs{\Delta_{q}\Checkmark{10}(t,x,k,j)-\Delta_{q}\Checkmark{10}(s,x,k,j)}^2]\lesssim\norm{\het}_{C_TL^{\infty}}^{4}\sum_{\om\in\mbZ^{2}}\varrho_{q}(\om)^{2}\sum_{\substack{\om_1,\om_2\in\mbZ^{2}\setminus\{0\}\\\om=\om_1+\om_2\\\om_1\sim\om_2}}\abs{\om_2}^{-2}\mathsf{A}^{k,k}_{s,t}\SVee(\om_1,\om_1,\om_2).
		\end{split}
	\end{equation*}
\end{lemma}
\begin{proof}
	The claim follows by an application of~\eqref{eq:isometry_bound_delta=0}, It\^{o}'s isometry (Lemma~\ref{lem:Ito_isometry}), the triangle inequality and direct computation.
\end{proof}
\begin{details}
	\begin{proof}
		Let $s,t\in[0,T]$, $x\in\mbT^{2}$, $k,j=1,2$ and $q\in\mbN_{-1}$. An application of~\eqref{eq:isometry_bound_delta=0} yields
		\begin{equation*}
			\begin{split}
				&\mbE[\abs{\Delta_{q}\Checkmark{10}(t,x,k,j)-\Delta_{q}\Checkmark{10}(s,x,k,j)}^2]\leq\norm{\het}_{C_TL^{\infty}}^{4}\mbE[\abs{\Delta_{q}\Checkmark{1}(t,x,k,j)-\Delta_{q}\Checkmark{1}(s,x,k,j)}^2],
			\end{split}
		\end{equation*}
		where $\Checkmark{1}$ is defined by
		\begin{equation*}
			\begin{split}
				\hat{\Checkmark{1}}(t,\om,k,j)&\defeq\sum_{\substack{\om_1,\om_2\in\mbZ^{2}\\\om=\om_1+\om_2\\\om_1\sim\om_2}}\sum_{j_1,j_2=1}^{2}\int_{0}^{t}\dd u_3\int_{0}^{t}\dd W^{j_1}(u_1,\om_1)\int_{0}^{u_1}\dd W^{j_2}(u_2,\om_2)\\
				&\quad\sum_{\varsigma\in\Sigma(1,2)}H_{t-u_3}^{k}(\om_{\varsigma(1)})H_{u_3-u_{\varsigma(1)}}^{j_{\varsigma(1)}}(\om_{\varsigma(1)})G^{j}(\om_{\varsigma(2)})H_{t-u_{\varsigma(2)}}^{j_{\varsigma(2)}}(\om_{\varsigma(2)}).
			\end{split}
		\end{equation*}
		Using that
		\begin{equation*}
			\mbE\Bigl[\hat{\Checkmark{1}}(t,\om,k,j)\overline{\hat{\Checkmark{1}}(s,\om',k,j)}\Bigr]=0\quad\text{if}\quad\om\neq\om'\in\mbZ^{2},
		\end{equation*}
		we obtain
		\begin{equation*}
			\mbE[\abs{\Delta_{q}\Checkmark{1}(t,x,k,j)-\Delta_{q}\Checkmark{1}(s,x,k,j)}^2]=\sum_{\om\in\mbZ^{2}}\varrho_{q}(\om)^{2}\mbE\Bigl[\Bigl\lvert\hat{\Checkmark{1}}(t,\om,k,j)-\hat{\Checkmark{1}}(s,\om,k,j)\Bigr\rvert^2\Bigr].
		\end{equation*}
		By It\^{o}'s isometry,
		\begin{equation*}
			\begin{split}
				&\mbE[\abs{\hat{\Checkmark{1}}(t,\om,k,j)-\hat{\Checkmark{1}}(s,\om,k,j)}^2]\\
				&=\sum_{\substack{\om_1,\om_2\in\mbZ^{2}\\\om=\om_1+\om_2\\\om_1\sim\om_2}}\sum_{j_1,j_2=1}^{2}\frac{1}{2!}\int_{0}^{\infty}\dd u_1\int_{0}^{\infty}\dd u_2\\
				&\Bigl\lvert\sum_{\varsigma\in\Sigma(1,2)}\int_{0}^{\infty}\dd u_3 (H^k_{t-u_3}(\om_{\varsigma(1)})H^{j_{\varsigma(2)}}_{t-u_{\varsigma(2)}}(\om_{\varsigma(2)})-H^k_{s-u_3}(\om_{\varsigma(1)})H^{j_{\varsigma(2)}}_{s-u_{\varsigma(2)}}(\om_{\varsigma(2)}))H^{j_{\varsigma(1)}}_{u_3-u_{\varsigma(1)}}(\om_{\varsigma(1)})G^j(\om_{\varsigma(2)})\Bigr\rvert^2.
			\end{split}
		\end{equation*}
		By Jensen's inequality,
		\begin{equation*}
			\begin{split}
				&\mbE\Bigl[\Bigl\lvert\hat{\Checkmark{1}}(t,\om,k,j)-\hat{\Checkmark{1}}(s,\om,k,j)\Bigr\rvert^2\Bigr]\\
				&\leq\sum_{\substack{\om_1,\om_2\in\mbZ^{2}\\\om=\om_1+\om_2\\\om_1\sim\om_2}}\sum_{j_1,j_2=1}^{2}2!\int_{0}^{\infty}\dd u_1\int_{0}^{\infty}\dd u_2\\
				&\quad\Bigl\lvert\int_{0}^{\infty}\dd u_3 (H^k_{t-u_3}(\om_{1})H^{j_{2}}_{t-u_{2}}(\om_{2})-H^k_{s-u_3}(\om_{1})H^{j_{2}}_{s-u_{2}}(\om_{2}))H^{j_{1}}_{u_3-u_{1}}(\om_{1})G^j(\om_{2})\Bigr\rvert^2.
			\end{split}
		\end{equation*}
		We apply the triangle inequality,
		\begin{equation*}
			\begin{split}
				&\mbE\Bigl[\Bigl\lvert\hat{\Checkmark{1}}(t,\om,k,j)-\hat{\Checkmark{1}}(s,\om,k,j)\Bigr\rvert^2\Bigr]\\
				&\lesssim\sum_{\substack{\om_1,\om_2\in\mbZ^{2}\\\om=\om_1+\om_2\\\om_1\sim\om_2}}\sum_{j_1,j_2=1}^{2}\int_{0}^{\infty}\dd u_1\int_{0}^{\infty}\dd u_2\int_{0}^{\infty}\dd u_3\int_{0}^{\infty}\dd u_3'\abs{G^j(\om_{2})}^{2}\\
				&\quad\times\lvert (H^k_{t-u_3}(\om_{1})H^{j_{2}}_{t-u_{2}}(\om_{2})-H^k_{s-u_3}(\om_{1})H^{j_{2}}_{s-u_{2}}(\om_{2}))H^{j_{1}}_{u_3-u_{1}}(\om_{1})\rvert\\
				&\quad\times\lvert (H^k_{t-u_3'}(\om_{1})H^{j_{2}}_{t-u_{2}}(\om_{2})-H^k_{s-u_3'}(\om_{1})H^{j_{2}}_{s-u_{2}}(\om_{2}))H^{j_{1}}_{u_3'-u_{1}}(\om_{1})\rvert.
			\end{split}
		\end{equation*}
		We compute 
		\begin{equation*}
			\sum_{j_1=1}^{2}\int_{-\infty}^{u_3\wedge u_3'}\dd u_1\abs{H_{u_3-u_1}^{j_1}(\om_1)}\abs{H_{u_3'-u_1}^{j_1}(\om_1)}=\frac{1}{2}\euler^{-\abs{u_3-u_3'}\abs{2\uppi\om_1}^2}\lesssim1,
		\end{equation*}
		which yields
		\begin{equation*}
			\begin{split}
				&\mbE\Bigl[\Bigl\lvert\hat{\Checkmark{1}}(t,\om,k,j)-\hat{\Checkmark{1}}(s,\om,k,j)\Bigr\rvert^2\Bigr]\\
				&\lesssim\sum_{\substack{\om_1,\om_2\in\mbZ^{2}\\\om=\om_1+\om_2\\\om_1\sim\om_2}}\sum_{j_2=1}^{2}\int_{0}^{\infty}\dd u_2\int_{0}^{\infty}\dd u_3\int_{0}^{\infty}\dd u_3'\abs{G^j(\om_{2})}^{2}\\
				&\quad\times\lvert H^k_{t-u_3}(\om_{1})H^{j_{2}}_{t-u_{2}}(\om_{2})-H^k_{s-u_3}(\om_{1})H^{j_{2}}_{s-u_{2}}(\om_{2})\rvert\lvert H^k_{t-u_3'}(\om_{1})H^{j_{2}}_{t-u_{2}}(\om_{2})-H^k_{s-u_3'}(\om_{1})H^{j_{2}}_{s-u_{2}}(\om_{2})\rvert.
			\end{split}
		\end{equation*}
		We introduce the shape coefficient $\mathsf{A}^{k,k}_{s,t}\SVee(\om_1,\om_1,\om_2)$, from which the bound follows.
	\end{proof}
\end{details}
\begin{proof}[Proof of Lemma~\ref{lem:existence_checkmark}]
	We provide a sketch. The shape coefficient is controlled in Lemma~\ref{lem:v_regularity}, we can then use fairly direct estimates and apply Lemma~\ref{lem:Nelson_estimate} and Lemma~\ref{lem:existence_criterion} as before. We observe that $\Checkmark{20}$ differs from $\Checkmark{10}$ only in its colouring of the $\om_1$ and $\om_2$ arrows, with the sum of the exponents preserved. Consequently by the same arguments as for $\Checkmark{10}$, we can also construct $\Checkmark{20}$.
\end{proof}
\begin{details}
	\begin{proof}[Proof of Lemma~\ref{lem:existence_checkmark}]
		Let $T>0$ and $\gamma\in(0,1)$. We first establish the existence of $\Checkmark{10}$. Assume $s\leq t$. An application of Lemma~\ref{lem:v_regularity} yields
		\begin{equation*}
			\mathsf{A}^{k,k}_{s,t}\SVee(\om_1,\om_1,\om_2)\lesssim\abs{t-s}^{\gamma}\abs{\om_1}^{-2+2\gamma},
		\end{equation*}
		so that by Lemma~\ref{lem:checkmark_upper_bound} and an application of Lemma~\ref{lem:convolution_estimates}, uniformly in $x\in\mbT^{2}$,
		\begin{equation*}
			\begin{split}
				\mbE[\abs{\Delta_{q}\Checkmark{10}(t,x,k,j)-\Delta_{q}\Checkmark{10}(s,x,k,j)}^2]&\lesssim\norm{\het}_{C_{T}L^{\infty}}^{4}\abs{t-s}^{\gamma}\sum_{\om\in\mbZ^{2}}\varrho_{q}(\om)^{2}\sum_{\substack{\om_1,\om_2\in\mbZ^{2}\setminus\{0\}\\\om=\om_1+\om_2\\\om_1\sim\om_2}}\abs{\om_2}^{-2}\abs{\om_1}^{-2+2\gamma}\\
				&\lesssim\norm{\het}_{C_{T}L^{\infty}}^{4}\abs{t-s}^{\gamma}\sum_{\om\in\mbZ^{2}}\varrho_{q}(\om)^{2}(1\vee\abs{\om})^{-2+2\gamma}\\
				&=\norm{\het}_{C_{T}L^{\infty}}^{4}\abs{t-s}^{\gamma}2^{q2\gamma}.
			\end{split}
		\end{equation*}
		Assume $\eps<\gamma/2$. We obtain by Lemma~\ref{lem:Nelson_estimate} and Lemma~\ref{lem:existence_criterion} for any $p\in[1,\infty)$,
		\begin{equation*}
			\mbE[\norm{\Checkmark{10}}_{C_T^{\gamma/2-\eps}\mcC^{-\gamma-3\eps}}^{p}]^{1/p}\lesssim\norm{\het}_{C_TL^{\infty}}^{2}
		\end{equation*}
		and therefore $\Checkmark{10}\in\msL^{\kappa}_{T}\mcC^{0-}(\mbT^{2};\mbR^{2\times2})$ a.s.\ for any $\kappa\in(0,1/2)$.
		
		We observe that $\Checkmark{20}$ differs from $\Checkmark{10}$ only in its colouring of the $\om_1$ and $\om_2$ arrows, with the sum of the exponents preserved. Consequently by the same arguments as for $\Checkmark{10}$, it follows that for any $p\in[1,\infty)$, $\eps<\gamma/2$,
		\begin{equation*}
			\mbE[\norm{\Checkmark{20}}_{C_T^{\gamma/2-\eps}\mcC^{-\gamma-3\eps}}^{p}]^{1/p}\lesssim\norm{\het}_{C_TL^{\infty}}^{2}
		\end{equation*}
		and therefore $\Checkmark{20}\in\msL^{\kappa}_{T}\mcC^{0-}(\mbT^{2};\mbR^{2\times2})$ a.s.\ for any $\kappa\in(0,1/2)$.
	\end{proof}
\end{details}
\subsection{Wick Contractions}\label{sec:Wick_contractions}
In this section we construct the contractions $\PreCocktail{30}$, $\PreCocktail{40}$, $\PreCocktail{50}=\PreCocktail{10}+\PreCocktail{20}$ and $\Triangle{3}=\Triangle{1}+\Triangle{2}$.

The diagrams $\PreCocktail{30}$, $\PreCocktail{40}$ differ from $\PreCocktail{10}$, $\PreCocktail{20}$ and $\Triangle{1}$, $\Triangle{2}$, despite their similarity in structure. The first two are well-defined, since two applications of $\nabla\Phi$ appear inside the $\STriangle$-shaped sub-diagram. This is not the case for $\PreCocktail{10}$, $\PreCocktail{20}$ and $\Triangle{1}$, $\Triangle{2}$ as one may tell by the distribution of highlighted arrows. 

However, by adding the problematic diagrams, $\PreCocktail{50}=\PreCocktail{10}+\PreCocktail{20}$ and $\Triangle{3}=\Triangle{1}+\Triangle{2}$ we can make use of the symmetry $G^{j}(\om_{4})=-G^{j}(-\om_{4})$, for all $j=1,2$ and $\om_{4}\in\mbZ^{2}$, to establish the existence of the summed objects.

We define a shape coefficient for those diagrams.
\begin{definition}\label{def:shape_coefficient_triangle}
	Let $s,t\geq0$, $\om_1,\om_2,\om_3\in\mbZ^{2}\setminus\{0\}$ and $k=1,2$. We set 
	\begin{equation*}
		\begin{split}
			&\mathsf{A}^{k}_{s,t}\STriangle(\om_1,\om_2,\om_3)\\
			&\defeq\sum_{j=1}^{2}\int_{-\infty}^{t}\dd u_1\int_{-\infty}^{u_1}\dd u_2\abs{(H^{k}_{t-u_1}(\om_1)H^{j}_{t-u_2}(\om_2)-H^{k}_{s-u_1}(\om_1)H^{j}_{s-u_2}(\om_2))H^{j}_{u_1-u_2}(\om_3)}.
		\end{split}
	\end{equation*}
\end{definition}
In Lemma~\ref{lem:existence_precocktail} we establish the existence of $\PreCocktail{30}$, $\PreCocktail{40}$, $\PreCocktail{50}$ and in Lemma~\ref{lem:existence_triangle}, we establish the existence of $\Triangle{3}$.
\begin{lemma}\label{lem:existence_precocktail}
	Let $T>0$, $\alpha<-2$, $\kappa\in(0,1/2)$ and $\het\in C_{T}\mcH^{2}(\mbT^{2})$. Then for any $p\in[1,\infty)$ we have
	\begin{equation*}
		\mbE[\norm{\PreCocktail{30}}_{\msL^{\kappa}_{T}\mcC^{3\alpha+6}}^p]^{1/p}+\mbE[\norm{\PreCocktail{40}}_{\msL^{\kappa}_{T}\mcC^{3\alpha+6}}^p]^{1/p}+\mbE[\norm{\PreCocktail{50}}_{\msL^{\kappa}_{T}\mcC^{3\alpha+6}}^p]^{1/p}\lesssim\norm{\het}_{C_{T}L^{\infty}}\norm{\het}_{C_{T}\mcH^{2}}^{2}
	\end{equation*}
	and in particular $\PreCocktail{30},\PreCocktail{40},\PreCocktail{50}\in\msL^{\kappa}_{T}\mcC^{3\alpha+6}(\mbT^{2};\mbR^{2})$ a.s.
\end{lemma}
\begin{lemma}\label{lem:existence_triangle}
	Let $T>0$, $\alpha<-2$, $\kappa\in(0,1/2)$ and $\het\in C_{T}\mcH^{2}(\mbT^{2})$. Then it holds that $\Triangle{3}\in\msL^{\kappa}_{T}\mcC^{2\alpha+4}(\mbT^{2};\mbR^{2\times2})$ and 
	\begin{equation*}
		\norm{\Triangle{3}}_{\msL^{\kappa}_{T}\mcC^{2\alpha+4}}\lesssim\norm{\het}_{C_{T}\mcH^{2}}^{2}.
	\end{equation*}
\end{lemma}
We first show Lemma~\ref{lem:existence_precocktail}. Let us focus on $\PreCocktail{50}$, the derivation for $\PreCocktail{30}$ and $\PreCocktail{40}$ is similar, but easier.
\begin{proof}[Proof of Lemma~\ref{lem:existence_precocktail}]
	Let $s,t\in[0,T]$, $x\in\mbT^{2}$, $j=1,2$ and $q\in\mbN_{-1}$. An application of~\eqref{eq:isometry_bound_delta=0} yields 
	\begin{equation*}
		\mbE[\abs{\Delta_{q}\PreCocktail{50}(t,x,j)-\Delta_{q}\PreCocktail{50}(s,x,j)}^2]\leq\norm{\het}_{C_{T}L^{\infty}}^{2}\mbE[\abs{\Delta_{q}\PreCocktail{5}(t,x,j)-\Delta_{q}\PreCocktail{5}(s,x,j)}^2],
	\end{equation*}
	where $\PreCocktail{5}$ is defined by 
	\begin{equation*}
		\begin{split}
			&\hat{\PreCocktail{5}}(t,\om,j)\\
			&\defeq\sum_{\substack{\om_1,\om_2,\om_4\in\mbZ^{2}\\\om=\om_1+\om_2+\om_4\\(\om_1+\om_2)\sim\om_4}}\sum_{j_1,j_2,j_3=1}^{2}\sum_{m_2\in\mbZ^{2}}\int_{0}^{t}\dd u_3\int_{0}^{u_3}\dd u_2\int_{0}^{u_3}\dd W^{j_1}(u_1,\om_1)\hat{\het}(u_2,\om_2-m_2)\hat{\het}(u_2,\om_4+m_2)\\
			&\quad\times(G^{j}(\om_4)+G^{j}(\om_1+\om_2))H^{j_3}_{t-u_3}(\om_1+\om_2)H^{j_2}_{t-u_2}(\om_4)G^{j_3}(\om_1)H^{j_1}_{u_3-u_1}(\om_1)H^{j_2}_{u_3-u_2}(\om_2).
		\end{split}
	\end{equation*}
	Using the definition of the Littlewood--Paley block $\Delta_{q}$, uniformly in $x\in\mbT^{2}$,
	\begin{equation*}
		\begin{split}
			&\mbE[\abs{\Delta_q\PreCocktail{5}(t,x,j)-\Delta_q\PreCocktail{5}(s,x,j)}^2]\\
			&\leq\sum_{\om,\om'\in\mbZ^{2}} \varrho_q(\om)\varrho_q(\om')\Bigl\lvert\mbE\Bigl[\Bigl(\hat{\PreCocktail{5}}(t,\om,j)-\hat{\PreCocktail{5}}(s,\om,j)\Bigr)\overline{\Bigl(\hat{\PreCocktail{5}}(t,\om',j)-\hat{\PreCocktail{5}}(s,\om',j)\Bigr)}\Bigr]\Bigr\rvert.
		\end{split}
	\end{equation*}
	We apply It\^{o}'s isometry and a decay estimate (Lemma~\ref{lem:elliptic_difference}) for the symmetrized elliptic multiplier $G^{j}(\om_{1}+\om_{2})+G^{j}(\om_{4})=G^{j}(\om-\om_{4})-G^{j}(-\om_{4})$. We obtain the bound
	\begin{equation}\label{eq:precocktail_sum_upper_bound}
		\begin{split}
			&\mbE[\abs{\Delta_{q}\PreCocktail{50}(t,x,j)-\Delta_{q}\PreCocktail{50}(s,x,j)}^2]\\
			&\lesssim\norm{\het}_{C_{T}L^{\infty}}^{2}\norm{\het}_{C_T\mcH^{2}}^{4}\sum_{\om,\om'\in\mbZ^{2}}\varrho_{q}(\om)\varrho_{q}(\om')\abs{\om}\abs{\om'}\sum_{\substack{\om_1,\om_2,\om_4\in\mbZ^{2}\setminus\{0\}\\\om_1+\om_2\in\mbZ^{2}\setminus\{0\}\\\om=\om_1+\om_2+\om_4\\(\om_1+\om_2)\sim\om_4}}\sum_{\substack{\om _2',\om_4'\in\mbZ^{2}\setminus\{0\}\\\om_1+\om_2'\in\mbZ^{2}\setminus\{0\}\\\om'=\om_1+\om_2'+\om_4'\\(\om_1+\om_2')\sim\om_4'}}\sum_{j_3,j_3'=1}^{2}\sum_{m_2,m_2'\in\mbZ^{2}}\\
			&(1+\abs{\om_2-m_2}^2)^{-1}(1+\abs{\om_2'-m_2'}^2)^{-1}(1+\abs{\om_4+m_2}^2)^{-1}(1+\abs{\om_4'+m_2'}^2)^{-1}\abs{\om_1}^{-2}\\
			&\times\abs{\om-\om_4}^{-2}\abs{\om'-\om_4'}^{-2}(1+\abs{\om}\abs{\om_4}^{-1})(1+\abs{\om'}\abs{\om_4'}^{-1})\mathsf{A}_{s,t}^{j_3}\STriangle(\om_1+\om_2,\om_4,\om_2)\mathsf{A}_{s,t}^{j_3'}\STriangle(\om_1+\om_2',\om_4',\om_2').
		\end{split}
	\end{equation}
	\begin{details}
		By It\^{o}'s isometry and the triangle inequality,
		\begin{equation*}
			\begin{split}
				&\Bigl\lvert\mbE\Bigl[\Bigl(\hat{\PreCocktail{5}}(t,\om,j)-\hat{\PreCocktail{5}}(s,\om,j)\Bigr)\overline{\Bigl(\hat{\PreCocktail{5}}(t,\om',j)-\hat{\PreCocktail{5}}(s,\om',j)\Bigr)}\Bigr]\Bigr\rvert\\
				&\leq\sum_{\substack{\om_1,\om_2,\om_4\in\mbZ^{2}\setminus\{0\}\\\om_1+\om_2\in\mbZ^{2}\setminus\{0\}\\\om=\om_1+\om_2+\om_4\\(\om_1+\om_2)\sim\om_4}}\sum_{j_1,j_2,j_3=1}^{2}\sum_{m_2\in\mbZ^{2}}\int_{0}^{\infty}\dd u_3\int_{0}^{u_3}\dd u_2\sum_{\substack{\om_2',\om_4'\in\mbZ^{2}\setminus\{0\}\\\om_1+\om_2'\in\mbZ^{2}\setminus\{0\}\\\om'=\om_1+\om_2'+\om_4'\\(\om_1+\om_2')\sim\om_4'}}\sum_{j_2',j_3'=1}^{2}\sum_{m_2'\in\mbZ^{2}}\int_{0}^{\infty}\dd u_3'\int_{0}^{u_3'}\dd u_2'\int_{0}^{\infty}\dd u_1\\
				&\quad\abs{\hat{\het}(u_2,\om_2-m_2)}\abs{\hat{\het}(u_2',\om_2'-m_2')}\abs{\hat{\het}(u_2,\om_4+m_2)}\abs{\hat{\het}(u_2',\om_4'+m_2')}\\
				&\quad\times\abs{G^{j}(\om_1+\om_2)+G^{j}(\om_4)}\abs{G^{j}(\om_1+\om_2')+G^{j}(\om_4')}\\
				&\quad\times\abs{(H^{j_3}_{t-u_3}(\om_1+\om_2)H^{j_2}_{t-u_2}(\om_4)-H^{j_3}_{s-u_3}(\om_1+\om_2)H^{j_2}_{s-u_2}(\om_4))G^{j_3}(\om_1)H^{j_1}_{u_3-u_1}(\om_1)H^{j_2}_{u_3-u_2}(\om_2)}\\
				&\quad\times\abs{(H^{j_3'}_{t-u_3'}(\om_1+\om_2')H^{j_2'}_{t-u_2'}(\om_4')-H^{j_3'}_{s-u_3'}(\om_1+\om_2')H^{j_2'}_{s-u_2'}(\om_4'))G^{j_3'}(\om_1)H^{j_1}_{u_3'-u_1}(\om_1)H^{j_2'}_{u_3'-u_2'}(\om_2')}.
			\end{split}
		\end{equation*}
		After estimating $\het$ by $\abs{\hat{\het}(u_2,\om_2-m_2)}\lesssim(1+\abs{\om_2-m_2}^2)^{-1}\norm{\het}_{C_T\mcH^{2}}$, $u_2\in[0,T]$, we can extend the domain of integration and bound the integral over $u_1$,
		
		\begin{equation*}
			\sum_{j_1=1}^{2}\int_{-\infty}^{u_3\wedge u_3'}\dd u_1\abs{H^{j_1}_{u_3-u_1}(\om_1)H^{j_1}_{u_3'-u_1}(\om_1)}=\euler^{-\abs{u_3-u_3'}\abs{2\uppi\om_1}^2}\leq 1.
		\end{equation*}
		Hence
		\begin{equation*}
			\begin{split}
				&\Bigl\lvert\mbE\Bigl[\Bigl(\hat{\PreCocktail{5}}(t,\om,j)-\hat{\PreCocktail{5}}(s,\om,j)\Bigr)\overline{\Bigl(\hat{\PreCocktail{5}}(t,\om',j)-\hat{\PreCocktail{5}}(s,\om',j)\Bigr)}\Bigr]\Bigr\rvert\\
				&\lesssim\norm{\het}_{C_T\mcH^{2}}^{4}\sum_{\substack{\om_1,\om_2,\om_4\in\mbZ^{2}\setminus\{0\}\\\om_1+\om_2\in\mbZ^{2}\setminus\{0\}\\\om=\om_1+\om_2+\om_4\\(\om_1+\om_2)\sim\om_4}}\sum_{\substack{\om _2',\om_4'\in\mbZ^{2}\setminus\{0\}\\\om_1+\om_2'\in\mbZ^{2}\setminus\{0\}\\\om'=\om_1+\om_2'+\om_4'\\(\om_1+\om_2')\sim\om_4'}}\sum_{j_3,j_3'=1}^{2}\sum_{m_2,m_2'\in\mbZ^{2}}\\
				&\quad(1+\abs{\om_2-m_2}^2)^{-1}(1+\abs{\om_2'-m_2'}^2)^{-1}(1+\abs{\om_4+m_2}^2)^{-1}(1+\abs{\om_4'+m_2'}^2)^{-1}\\
				&\quad\times\abs{G^{j}(\om_1+\om_2)+G^{j}(\om_4)}\abs{G^{j}(\om_1+\om_2')+G^{j}(\om_4')}\abs{G^{j_3}(\om_1)G^{j_3'}(\om_1)}\\
				&\quad\times\sum_{j_2=1}^{2}\int_{-\infty}^{\infty}\dd u_3\int_{-\infty}^{u_3}\dd u_2\abs{(H^{j_3}_{t-u_3}(\om_1+\om_2)H^{j_2}_{t-u_2}(\om_4)-H^{j_3}_{s-u_3}(\om_1+\om_2)H^{j_2}_{s-u_2}(\om_4))H^{j_2}_{u_3-u_2}(\om_2)}\\
				&\quad\times\sum_{j_2'=1}^{2}\int_{-\infty}^{\infty}\dd u_3'\int_{-\infty}^{u_3'}\dd u_2'\abs{(H^{j_3'}_{t-u_3'}(\om_1+\om_2')H^{j_2'}_{t-u_2'}(\om_4')-H^{j_3'}_{s-u_3'}(\om_1+\om_2')H^{j_2'}_{s-u_2'}(\om_4'))H^{j_2'}_{u_3'-u_2'}(\om_2')}.
			\end{split}
		\end{equation*}
		We can now apply Lemma~\ref{lem:elliptic_difference} and introduce the shape coefficient.
	\end{details}
	We control the shape coefficient with Lemma~\ref{lem:triangle_regularity} and obtain for all $\gamma\in[0,1]$,
	\begin{equation*}
		\mathsf{A}_{s,t}^{j_3}\STriangle(\om_1+\om_2,\om_4,\om_2)\lesssim\abs{t-s}^{\gamma}\abs{\om_4}^{2\gamma}\abs{\om_2}^{-1}.
	\end{equation*}
	We can then plug this expression into~\eqref{eq:precocktail_sum_upper_bound} and apply Lemma~\ref{lem:double_sum_PreCocktail_renormalised} to control the sums over $\om_4$ and $m_2\in\mbZ^{2}$. Let $\gamma\in(0,1/2)$ and $\eps\in(0,1-2\gamma)$, it follows that
	\begin{equation}\label{eq:precocktail_sum_upper_bound_2}
		\begin{split}
			&\mbE[\abs{\Delta_{q}\PreCocktail{50}(t,x,j)-\Delta_{q}\PreCocktail{50}(s,x,j)}^2]\\
			&\lesssim\norm{\het}_{C_{T}L^{\infty}}^{2}\norm{\het}_{C_{T}\mcH^{2}}^{4}\abs{t-s}^{2\gamma}\sum_{\om,\om'\in\mbZ^{2}}\varrho_{q}(\om)\varrho_{q}(\om')\abs{\om}^{2\gamma+\eps}\abs{\om'}^{2\gamma+\eps}\\
			&\quad\times\sum_{\om_1\in\mbZ^{2}\setminus\{0\}}\abs{\om_1}^{-2}(1\vee\abs{\om-\om_1})^{-2+\eps}(1\vee\abs{\om'-\om_1})^{-2+\eps}.
		\end{split}
	\end{equation}
	Hence it suffices to control the remaining sum over $\om_1\in\mbZ^{2}\setminus\{0\}$,
	\begin{equation*}
		\sum_{\om_1\in\mbZ^{2}\setminus\{0\}}\abs{\om_1}^{-2}(1\vee\abs{\om-\om_1})^{-2+\eps}(1\vee\abs{\om'-\om_1})^{-2+\eps}.
	\end{equation*}
	We distinguish the cases $\om=\om'$ and $\om\neq\om'$. In the case $\om=\om'$, we decompose the sum into the regions $\om_1=\om$ and $\om_1\in\mbZ^{2}\setminus\{0,\om\}$. We then estimate by Lemma~\ref{lem:convolution_estimates},
	\begin{equation*}
		\sum_{\om_1\in\mbZ^{2}\setminus\{0\}}\abs{\om_1}^{-2}(1\vee\abs{\om-\om_1})^{-4+2\eps}=\abs{\om}^{-2}+\sum_{\om_1\in\mbZ^{2}\setminus\{0,\om\}}\abs{\om_1}^{-2}\abs{\om-\om_1}^{-4+2\eps}\lesssim\abs{\om}^{-2+2\eps}.
	\end{equation*}
	In the case $\om\neq\om'$, we decompose the sum into the regions $\om_1=\om$, $\om_1=\om'$ and $\om_1\in\mbZ^{2}\setminus\{0,\om,\om'\}$,
	\begin{equation*}
		\begin{split}
			&\sum_{\om_1\in\mbZ^{2}\setminus\{0\}}\abs{\om_1}^{-2}(1\vee\abs{\om-\om_1})^{-2+\eps}(1\vee\abs{\om'-\om_1})^{-2+\eps}\\
			&=(\abs{\om}^{-2}+\abs{\om'}^{-2})\abs{\om-\om'}^{-2+\eps}+\sum_{\om_1\in\mbZ^{2}\setminus\{0,\om,\om'\}}\abs{\om_1}^{-2}\abs{\om-\om_1}^{-2+\eps}\abs{\om'-\om_1}^{-2+\eps}
		\end{split}
	\end{equation*}
	and apply Lemma~\ref{lem:convolution_twofold} to bound
	\begin{equation*}
		\sum_{\om_1\in\mbZ^{2}\setminus\{0,\om,\om'\}}\abs{\om_1}^{-2}\abs{\om-\om_1}^{-2+\eps}\abs{\om'-\om_1}^{-2+\eps}\lesssim\abs{\om-\om'}^{-2+\eps}\abs{\om}^{-2+2\eps}+\abs{\om-\om'}^{-2+\eps}\abs{\om'}^{-2+2\eps}.
	\end{equation*}
	Assume $q\in\mbN_{0}$ and $\om,\om'\in\supp(\varrho_q)$, it follows that $2^{q}\lesssim\abs{\om}\lesssim2^{q}$, $2^{q}\lesssim\abs{\om'}\lesssim2^{q}$ and if $\om\neq\om'$ then $2^{q}\lesssim\abs{\om-\om'}\lesssim2^{q}$ as well. We obtain by~\eqref{eq:precocktail_sum_upper_bound_2},
	\begin{equation*}
		\mbE[\abs{\Delta_q\PreCocktail{50}(t,x,j)-\Delta_q\PreCocktail{50}(s,x,j)}^2]\lesssim\norm{\het}_{C_{T}L^{\infty}}^{2}\norm{\het}_{C_{T}\mcH^{2}}^{4}\abs{t-s}^{2\gamma}2^{q(4\gamma+5\eps)}.
	\end{equation*}
	Assume in addition $\eps<\gamma$, we obtain by Lemma~\ref{lem:Nelson_estimate} and Lemma~\ref{lem:existence_criterion} for any $p\in[1,\infty)$,
	\begin{equation*}
		\mbE[\norm{\PreCocktail{50}}_{C_T^{\gamma-\eps}\mcC^{-2\gamma-5\eps}}^{p}]^{1/p}\lesssim\norm{\het}_{C_TL^{\infty}}\norm{\het}_{C_T\mcH^{2}}^{2}
	\end{equation*}
	and therefore $\PreCocktail{50}\in\msL^{\kappa}_{T}\mcC^{0-}(\mbT^{2};\mbR^{2})$ a.s.\ for any $\kappa\in(0,1/2)$.
	\begin{details}
		
		Next we construct $\PreCocktail{30}$. An application of~\eqref{eq:isometry_bound_delta=0} yields uniformly in $x\in\mbT^{2}$,
		\begin{equation*}
			\mbE[\abs{\Delta_{q}\PreCocktail{30}(t,x,j)-\Delta_{q}\PreCocktail{30}(s,x,j)}^2]\leq\norm{\het}_{C_{T}L^{\infty}}^{2}\mbE[\abs{\Delta_{q}\PreCocktail{3}(t,x,j)-\Delta_{q}\PreCocktail{3}(s,x,j)}^2],
		\end{equation*}
		where $\PreCocktail{3}$ is defined by
		\begin{equation*}
			\begin{split}
				&\hat{\PreCocktail{3}}(t,\om,j)\\
				&\defeq\sum_{\substack{\om_1,\om_2,\om_4\in\mbZ^{2}\\\om=\om_1+\om_2+\om_4\\(\om_1+\om_2)\sim\om_4}}\sum_{j_1,j_2,j_3=1}^{2}\sum_{m_2\in\mbZ^{2}}\int_{0}^{t}\dd u_3\int_{0}^{u_3}\dd u_2\int_{0}^{u_3}\dd W^{j_1}(u_1,\om_1)\hat{\het}(u_2,\om_2-m_2)\hat{\het}(u_2,\om_4+m_2)\\
				&\quad\times H_{t-u_3}^{j_3}(\om_1+\om_2)G^{j}(\om_4)H_{t-u_2}^{j_2}(\om_4)H_{u_3-u_1}^{j_1}(\om_1)G^{j_3}(\om_2)H_{u_3-u_2}^{j_2}(\om_2).
			\end{split}
		\end{equation*}
		We obtain by It\^{o}'s isometry for $\om,\om'\in\mbZ^2$, $s,t\in[0,T]$,
		\begin{equation*}
			\begin{split}
				&\Bigl\lvert\mbE\Bigl[\Bigl(\hat{\PreCocktail{3}}(t,\om,j)-\hat{\PreCocktail{3}}(s,\om,j)\Bigr)\overline{\Bigl(\hat{\PreCocktail{3}}(t,\om',j)-\hat{\PreCocktail{3}}(s,\om',j)\Bigr)}\Bigr]\Bigr\rvert\\
				&\leq\sum_{\substack{\om_1,\om_2,\om_4\in\mbZ^{2}\\\om=\om_1+\om_2+\om_4\\(\om_1+\om_2)\sim\om_4}}\sum_{j_1,j_2,j_3=1}^{2}\sum_{m_2\in\mbZ^{2}}\int_{0}^{\infty}\dd u_3\int_{0}^{u_3}\dd u_2\sum_{\substack{\om_2',\om_4'\in\mbZ^{2}\\\om'=\om_1+\om_2'+\om_4'\\(\om_1+\om_2')\sim\om_4'}}\sum_{j_2',j_3'=1}^{2}\sum_{m_2'\in\mbZ^{2}}\int_{0}^{\infty}\dd u_3'\int_{0}^{u_3'}\dd u_2'\int_{0}^{u_3\wedge u_3'}\dd u_1\\
				&\quad\abs{\hat{\het}(u_2,\om_2-m_2)}\abs{\hat{\het}(u_2,\om_2'-m_2')}\abs{\hat{\het}(u_2,\om_4+m_2)}\abs{\hat{\het}(u_2,\om_4'+m_2')}\\
				&\quad\times\abs{G^{j}(\om_4)}\abs{G^{j}(\om_4')}\abs{G^{j_3}(\om_2)}\abs{G^{j_3'}(\om_2')}\\
				&\quad\times\abs{(H^{j_3}_{t-u_3}(\om_1+\om_2)H^{j_2}_{t-u_2}(\om_4)-H^{j_3}_{s-u_3}(\om_1+\om_2)H^{j_2}_{s-u_2}(\om_4))H^{j_1}_{u_3-u_1}(\om_1)H^{j_2}_{u_3-u_2}(\om_2)}\\
				&\quad\times\abs{(H^{j_3'}_{t-u_3'}(\om_1+\om_2')H^{j_2'}_{t-u_2'}(\om_4')-H^{j_3'}_{s-u_3'}(\om_1+\om_2')H^{j_2'}_{s-u_2'}(\om_4'))H^{j_1}_{u_3'-u_1}(\om_1)H^{j_2'}_{u_3'-u_2'}(\om_2')}.
			\end{split}
		\end{equation*}
		We estimate $\abs{\hat{\het}(u_2,\om_2-m_2)}\lesssim(1+\abs{\om_2-m_2}^2)^{-1}\norm{\het}_{C_T\mcH^{2}}$, $u_2\in[0,T]$. We can then extend the domain of integration and compute 
		\begin{equation*}
			\sum_{j_1=1}^{2}\int_{-\infty}^{u_3\wedge u_3'}\dd u_1\abs{H^{j_1}_{u_3-u_1}(\om_1)H^{j_1}_{u_3'-u_1}(\om_1)}=\euler^{-\abs{u_3-u_3'}\abs{2\uppi\om_1}^2}\leq 1.
		\end{equation*}
		Introducing the shape coefficient yields
		\begin{equation*}
			\begin{split}
				&\Bigl\lvert\mbE\Bigl[\Bigl(\hat{\PreCocktail{3}}(t,\om,j)-\hat{\PreCocktail{3}}(s,\om,j)\Bigr)\overline{\Bigl(\hat{\PreCocktail{3}}(t,\om',j)-\hat{\PreCocktail{3}}(s,\om',j)\Bigr)}\Bigr]\Bigr\rvert\\
				&\lesssim\norm{\het}_{C_T\mcH^{2}}^{4}\sum_{\substack{\om_1,\om_2,\om_4\in\mbZ^{2}\setminus\{0\}\\\om_1+\om_2\in\mbZ^{2}\setminus\{0\}\\\om=\om_1+\om_2+\om_4\\(\om_1+\om_2)\sim\om_4}}\sum_{\substack{\om_2',\om_4'\in\mbZ^{2}\setminus\{0\}\\\om_1+\om_2'\in\mbZ^{2}\setminus\{0\}\\\om'=\om_1+\om_2'+\om_4'\\(\om_1+\om_2')\sim\om_4'}}\sum_{j_3,j_3'=1}^{2}\sum_{m_2,m_2'\in\mbZ^{2}}\\
				&\quad(1+\abs{\om_2-m_2}^2)^{-1}(1+\abs{\om_2'-m_2'}^2)^{-1}(1+\abs{\om_4+m_2}^2)^{-1}(1+\abs{\om_4'+m_2'}^2)^{-1}\\
				&\quad\times\abs{G^{j}(\om_4)}\abs{G^{j}(\om_4')}\abs{G^{j_3}(\om_2)}\abs{G^{j_3'}(\om_2')}\mathsf{A}_{s,t}^{j_3}\STriangle(\om_1+\om_2,\om_4,\om_2)\mathsf{A}_{s,t}^{j_3'}\STriangle(\om_1+\om_2',\om_4',\om_2').
			\end{split}
		\end{equation*}
		Assume $s\leq t$ and $\gamma\in(0,1/2)$. We apply Lemma~\ref{lem:triangle_regularity} to control $\mathsf{A}_{s,t}^{j_3}\STriangle(\om_1+\om_2,\om_4,\om_2)$,
		\begin{equation*}
			\begin{split}
				&\Bigl\lvert\mbE\Bigl[\Bigl(\hat{\PreCocktail{3}}(t,\om,j)-\hat{\PreCocktail{3}}(s,\om,j)\Bigr)\overline{\Bigl(\hat{\PreCocktail{3}}(t,\om',j)-\hat{\PreCocktail{3}}(s,\om',j)\Bigr)}\Bigr]\Bigr\rvert\\
				&\lesssim\norm{\het}_{C_T\mcH^{2}}^{4}\abs{t-s}^{2\gamma}\sum_{\om_1\in\mbZ^{2}\setminus\{0\}}\sum_{\substack{\om_2,\om_4\in\mbZ^{2}\setminus\{0\}\\\om_1+\om_2\in\mbZ^{2}\setminus\{0\}\\\om-\om_1=\om_2+\om_4\\(\om-\om_4)\sim\om_4}}\sum_{\substack{\om_2',\om_4'\in\mbZ^{2}\setminus\{0\}\\\om_1+\om_2'\in\mbZ^{2}\setminus\{0\}\\\om'-\om_1=\om_2'+\om_4'\\(\om'-\om_4')\sim\om_4'}}\sum_{m_2,m_2'\in\mbZ^{2}}\\
				&\quad(1+\abs{\om_2-m_2}^2)^{-1}(1+\abs{\om_2'-m_2'}^2)^{-1}(1+\abs{\om_4+m_2}^2)^{-1}(1+\abs{\om_4'+m_2'}^2)^{-1}\\
				&\quad\times\abs{\om_4}^{-1+2\gamma}\abs{\om_4'}^{-1+2\gamma}\abs{\om_2}^{-2}\abs{\om_2'}^{-2}.
			\end{split}
		\end{equation*}
		Let $\eps\in(0,(1-2\gamma)/2)$. We then apply Lemma~\ref{lem:double_sum_PreCocktail} to obtain
		\begin{equation*}
			\begin{split}
				&\Bigl\lvert\mbE\Bigl[\Bigl(\hat{\PreCocktail{3}}(t,\om,j)-\hat{\PreCocktail{3}}(s,\om,j)\Bigr)\overline{\Bigl(\hat{\PreCocktail{3}}(t,\om',j)-\hat{\PreCocktail{3}}(s,\om',j)\Bigr)}\Bigr]\Bigr\rvert\\
				&\lesssim\norm{\het}_{C_T\mcH^{2}}^{4}\abs{t-s}^{2\gamma}(1\vee\abs{\om})^{-1+2\gamma+2\eps}(1\vee\abs{\om'})^{-1+2\gamma+2\eps}\\
				&\quad\times\sum_{\om_1\in\mbZ^{2}\setminus\{0\}}(1\vee\abs{\om-\om_1})^{-2+\eps}(1\vee\abs{\om'-\om_1})^{-2+\eps}.
			\end{split}
		\end{equation*}
		Assume $\om=\om'$. We obtain by~\eqref{eq:summation_estimates_-alpha},
		\begin{equation*}
			\sum_{\om_1\in\mbZ^{2}\setminus\{0\}}(1\vee\abs{\om-\om_1})^{-4+2\eps}=1+\sum_{\om_1\in\mbZ^{2}\setminus\{0,\om\}}\abs{\om-\om_1}^{-4+2\eps}\lesssim1.
		\end{equation*}
		Assume $\om\neq\om'$. We obtain by Lemma~\ref{lem:convolution_estimates}, 
		\begin{equation*}
			\begin{split}
				&\sum_{\om_1\in\mbZ^{2}\setminus\{0\}}(1\vee\abs{\om-\om_1})^{-2+\eps}(1\vee\abs{\om'-\om_1})^{-2+\eps}\\
				&\leq2\abs{\om-\om'}^{-2+\eps}+\sum_{\om_1\in\mbZ^{2}\setminus\{0,\om,\om'\}}\abs{\om-\om_1}^{-2+\eps}\abs{\om'-\om_1}^{-2+\eps}\lesssim\abs{\om-\om'}^{-2+2\eps}.
			\end{split}
		\end{equation*}
		Consequently for $q\in\mbN_{0}$,
		\begin{equation*}
			\sum_{\om\in\mbZ^{2}} \abs{\varrho_q(\om)}^2\mbE\Bigl[\Bigl\lvert\hat{\PreCocktail{3}}(t,\om,j)-\hat{\PreCocktail{3}}(s,\om,j)\Bigr\rvert^2\Bigl]\lesssim\norm{\het}_{C_T\mcH^{2}}^{4}\abs{t-s}^{2\gamma}2^{q(4\gamma+4\eps)}
		\end{equation*}
		and for $q=-1$,
		\begin{equation*}
			\sum_{\om\in\mbZ^{2}} \abs{\varrho_{-1}(\om)}^2\mbE\Bigl[\Bigl\lvert\hat{\PreCocktail{3}}(t,\om,j)-\hat{\PreCocktail{3}}(s,\om,j)\Bigr\rvert^2\Bigl]\lesssim\norm{\het}_{C_T\mcH^{2}}^{4}\abs{t-s}^{2\gamma}\lesssim\norm{\het}_{C_T\mcH^{2}}^{4}\abs{t-s}^{2\gamma}2^{-(4\gamma+4\eps)}.
		\end{equation*}
		Furthermore for $q\in\mbN_{-1}$,
		\begin{equation*}
			\begin{split}
				&\sum_{\substack{\om,\om'\in\mbZ^{2}\\\om\not=\om'}} \varrho_q(\om)\varrho_q(\om')\Bigl\lvert\mbE\Bigl[\Bigl(\hat{\PreCocktail{3}}(t,\om,j)-\hat{\PreCocktail{3}}(s,\om,j)\Bigr)\overline{\Bigl(\hat{\PreCocktail{3}}(t,\om',j)-\hat{\PreCocktail{3}}(s,\om',j)\Bigr)}\Bigr]\Bigr\rvert\\
				&\lesssim\norm{\het}_{C_T\mcH^{2}}^{4}\abs{t-s}^{2\gamma}2^{q(4\gamma+6\eps)}.
			\end{split}
		\end{equation*}
		Assume in addition $\eps<\gamma$. We obtain by~\eqref{eq:isometry_bound_delta=0}, Lemma~\ref{lem:Nelson_estimate} and Lemma~\ref{lem:existence_criterion} for any $p\in[1,\infty)$,
		\begin{equation*}
			\mbE[\norm{\PreCocktail{30}}_{C_T^{\gamma-\eps}\mcC^{-2\gamma-6\eps}}^{p}]^{1/p}\lesssim\norm{\het}_{C_TL^{\infty}}\norm{\het}_{C_T\mcH^{2}}^{2}
		\end{equation*}
		and therefore $\PreCocktail{30}\in\msL^{\kappa}_{T}\mcC^{0-}(\mbT^{2};\mbR^{2})$ a.s.\ for any $\kappa\in(0,1/2)$.
		
		Next we construct $\PreCocktail{40}$. An application of~\eqref{eq:isometry_bound_delta=0} yields uniformly in $x\in\mbT^{2}$,
		\begin{equation*}
			\mbE[\abs{\Delta_{q}\PreCocktail{40}(t,x,j)-\Delta_{q}\PreCocktail{40}(s,x,j)}^2]\leq\norm{\het}_{C_{T}L^{\infty}}^{2}\mbE[\abs{\Delta_{q}\PreCocktail{4}(t,x,j)-\Delta_{q}\PreCocktail{4}(s,x,j)}^2],
		\end{equation*}
		where $\PreCocktail{4}$ is defined by
		\begin{equation*}
			\begin{split}
				&\hat{\PreCocktail{4}}(t,\om,j)\\
				&\defeq\sum_{\substack{\om_1,\om_2,\om_4\in\mbZ^{2}\\\om=\om_1+\om_2+\om_4\\(\om_1+\om_2)\sim\om_4}}\sum_{j_1,j_2,j_3=1}^{2}\sum_{m_2\in\mbZ^{2}}\int_{0}^{t}\dd u_3\int_{0}^{u_3}\dd u_2\int_{0}^{u_3}\dd W^{j_1}(u_1,\om_1)\hat{\het}(u_2,\om_2-m_2)\hat{\het}(u_2,\om_4+m_2)\\
				&\quad\times G^{j}(\om_1+\om_2)H^{j_3}_{t-u_3}(\om_1+\om_2)H^{j_2}_{t-u_2}(\om_4)H^{j_1}_{u_3-u_1}(\om_1)G^{j_3}(\om_2)H^{j_2}_{u_3-u_2}(\om_2).
			\end{split}
		\end{equation*}
		By It\^{o}'s isometry for $\om,\om'\in\mbZ^2$, $s,t\in[0,T]$,
		\begin{equation*}
			\begin{split}
				&\Bigl\lvert\mbE\Bigl[\Bigl(\hat{\PreCocktail{4}}(t,\om,j)-\hat{\PreCocktail{4}}(s,\om,j)\Bigr)\overline{\Bigl(\hat{\PreCocktail{4}}(t,\om',j)-\hat{\PreCocktail{4}}(s,\om',j)\Bigr)}\Bigr]\Bigr\rvert\\
				&\leq\sum_{\substack{\om_1,\om_2,\om_4\in\mbZ^{2}\\\om=\om_1+\om_2+\om_4\\(\om_1+\om_2)\sim\om_4}}\sum_{j_1,j_2,j_3=1}^{2}\sum_{m_2\in\mbZ^{2}}\int_{0}^{\infty}\dd u_3\int_{0}^{u_3}\dd u_2\sum_{\substack{\om_2',\om_4'\in\mbZ^{2}\\\om'=\om_1+\om_2'+\om_4'\\(\om_1+\om_2')\sim\om_4'}}\sum_{j_2',j_3'=1}^{2}\sum_{m_2'\in\mbZ^{2}}\int_{0}^{\infty}\dd u_3'\int_{0}^{u_3'}\dd u_2'\int_{0}^{u_3\wedge u_3'}\dd u_1\\
				&\quad\abs{\hat{\het}(u_2,\om_2-m_2)}\abs{\hat{\het}(u_2',\om_2'-m_2')}\abs{\hat{\het}(u_2,\om_4+m_2)}\abs{\hat{\het}(u_2',\om_4'+m_2')}\\
				&\quad\times\abs{G^{j}(\om_1+\om_2)}\abs{G^{j}(\om_1+\om_2')}\abs{G^{j_3}(\om_2)}\abs{G^{j_3'}(\om_2')}\\
				&\quad\times\abs{(H^{j_3}_{t-u_3}(\om_1+\om_2)H^{j_2}_{t-u_2}(\om_4)-H^{j_3}_{s-u_3}(\om_1+\om_2)H^{j_2}_{s-u_2}(\om_4))H^{j_1}_{u_3-u_1}(\om_1)H^{j_2}_{u_3-u_2}(\om_2)}\\
				&\quad\times\abs{(H^{j_3'}_{t-u_3'}(\om_1+\om_2')H^{j_2'}_{t-u_2'}(\om_4')-H^{j_3'}_{s-u_3'}(\om_1+\om_2')H^{j_2'}_{s-u_2'}(\om_4'))H^{j_1}_{u_3'-u_1}(\om_1)H^{j_2'}_{u_3'-u_2'}(\om_2')},
			\end{split}
		\end{equation*}
		and so,
		\begin{equation*}
			\begin{split}
				&\Bigl\lvert\mbE\Bigl[\Bigl(\hat{\PreCocktail{4}}(t,\om,j)-\hat{\PreCocktail{4}}(s,\om,j)\Bigr)\overline{\Bigl(\hat{\PreCocktail{4}}(t,\om',j)-\hat{\PreCocktail{4}}(s,\om',j)\Bigr)}\Bigr]\Bigr\rvert\\
				&\lesssim\norm{\het}_{C_T\mcH^{2}}^{4}\sum_{\substack{\om_1,\om_2,\om_4\in\mbZ^{2}\setminus\{0\}\\\om_1+\om_2\in\mbZ^{2}\setminus\{0\}\\\om=\om_1+\om_2+\om_4\\(\om_1+\om_2)\sim\om_4}}\sum_{\substack{\om_2',\om_4'\in\mbZ^{2}\setminus\{0\}\\\om_1+\om_2'\in\mbZ^{2}\setminus\{0\}\\\om'=\om_1+\om_2'+\om_4'\\(\om_1+\om_2')\sim\om_4'}}\sum_{j_3,j_3'=1}^{2}\sum_{m_2,m_2'\in\mbZ^{2}}\\
				&\quad(1+\abs{\om_2-m_2}^2)^{-1}(1+\abs{\om_2'-m_2'}^2)^{-1}(1+\abs{\om_4+m_2}^2)^{-1}(1+\abs{\om_4'+m_2'}^2)^{-1}\\
				&\quad\times\abs{G^{j}(\om_1+\om_2)}\abs{G^{j}(\om_1+\om_2')}\abs{G^{j_3}(\om_2)}\abs{G^{j_3'}(\om_2')}\mathsf{A}_{s,t}^{j_3}\STriangle(\om_1+\om_2,\om_4,\om_2)\mathsf{A}_{s,t}^{j_3'}\STriangle(\om_1+\om_2',\om_4',\om_2').
			\end{split}
		\end{equation*}
		The only difference to $\PreCocktail{3}$ is the factor $\abs{G^{j}(\om-\om_4)}\abs{G^{j}(\om'-\om_4')}$ that appears rather than $\abs{G^{j}(\om_4)}\abs{G^{j}(\om_4')}$. Hence, in the proof of Lemma~\ref{lem:double_sum_PreCocktail} we estimate instead of~\eqref{eq:PreCocktail_colour_difference},
		\begin{equation*}
			\begin{split}
				&\sum_{\substack{\om_4\in\mbZ^{2}\setminus\{0,\om,\om-\om_1\}\\(\om-\om_4)\sim\om_4}}\abs{\om-\om_4}^{-1}\abs{\om_4}^{2\gamma}\abs{\om-\om_1-\om_4}^{-2}\\
				&\leq\Bigl(\sum_{\substack{\om_4\in\mbZ^{2}\setminus\{0,\om\}\\(\om-\om_4)\sim\om_4}}\abs{\om-\om_4}^{-p}\abs{\om_4}^{p\delta}\Bigr)^{1/p}\Bigl(\sum_{\substack{\om_4\in\mbZ^{2}\setminus\{0,\om-\om_1\}}}\abs{\om_4}^{q(2\gamma-\delta )}\abs{\om-\om_1-\om_4}^{q(-2+\eps)}\Bigr)^{1/q},
			\end{split}
		\end{equation*}
		where $\eps\in(0,1-2\gamma)$, $p\in(1,\infty)$, $q=p/(p-1)$ and $\delta>0$. However, Lemma~\ref{lem:convolution_estimates} still applies, since the condition $p(1-\delta)>2$ for the resonant product is preserved.
		
		Assume in addition $\eps<\gamma$. We obtain by~\eqref{eq:isometry_bound_delta=0}, Lemma~\ref{lem:Nelson_estimate} and Lemma~\ref{lem:existence_criterion} for any $p\in[1,\infty)$,
		\begin{equation*}
			\mbE[\norm{\PreCocktail{40}}_{C_T^{\gamma-\eps}\mcC^{-2\gamma-6\eps}}^{p}]^{1/p}\lesssim\norm{\het}_{C_TL^{\infty}}\norm{\het}_{C_T\mcH^{2}}^{2}
		\end{equation*}
		and therefore $\PreCocktail{40}\in\msL^{\kappa}_{T}\mcC^{0-}(\mbT^{2};\mbR^{2})$ a.s.\ for any $\kappa\in(0,1/2)$.
	\end{details}
\end{proof}
Next we prove the existence of $\Triangle{3}$.
\begin{proof}[Proof of Lemma~\ref{lem:existence_triangle}]
	Let $0\leq s\leq t\leq T$, $\om\in\mbZ^2$ and $k,j=1,2$. We can bound the increment
	\begin{equation*}
		\begin{split}
			&\abs{\hat{\Triangle{3}}(t,\om,k,j)-\hat{\Triangle{3}}(s,\om,k,j)}\\
			&\lesssim\norm{\het}_{C_T\mcH^{2}}^{2}\sum_{\substack{\om_1,\om_2\in\mbZ^{2}\setminus\{0\}\\\om=\om_1+\om_2\\\om_1\sim\om_2}}\sum_{m_2\in\mbZ^{2}}\frac{\abs{G^j(\om_2)+G^j(\om_1)}}{(1+\abs{\om_1+m_2}^{2})(1+\abs{\om_2-m_2}^{2})}\mathsf{A}_{s,t}^{k}\STriangle(\om_1,\om_2,\om_1).
		\end{split}
	\end{equation*}
	\begin{details}
		\begin{equation*}
			\begin{split}
				&\hat{\Triangle{3}}(t,\om,k,j)-\hat{\Triangle{3}}(s,\om,k,j)\\
				&=\sum_{\substack{\om_1,\om_2\in\mbZ^{2}\\\om=\om_1+\om_2\\\om_1\sim\om_2}}\sum_{j_2=1}^{2}\sum_{m_2\in\mbZ^{2}}\int_{0}^{\infty}\dd u_3\int_{0}^{u_3}\dd u_2\hat{\het}(u_2,\om_1+m_2)\hat{\het}(u_2,\om_2-m_2)(G^j(\om_1)+G^j(\om_2))\\
				&\quad\times(H^k_{t-u_3}(\om_1)H^{j_2}_{t-u_2}(\om_2)-H^k_{s-u_3}(\om_1)H^{j_2}_{s-u_2}(\om_2))H^{j_2}_{u_3-u_2}(\om_1).
			\end{split}
		\end{equation*}
	\end{details}
	We control the shape coefficient with Lemma~\ref{lem:triangle_regularity} and the elliptic multiplier with Lemma~\ref{lem:elliptic_difference}. Let $\kappa\in[0,1]$, we arrive at
	\begin{equation*}
		\begin{split}
			&\abs{\hat{\Triangle{3}}(t,\om,k,j)-\hat{\Triangle{3}}(s,\om,k,j)}\lesssim\norm{\het}_{C_T\mcH^{2}}^{2}\abs{t-s}^{\kappa}\abs{\om}\sum_{\substack{\om_1,\om_2\in\mbZ^{2}\setminus\{0\}\\\om=\om_1+\om_2\\\om_1\sim\om_2}}\sum_{m_2\in\mbZ^{2}}\frac{(1+\abs{\om}\abs{\om_2}^{-1})\abs{\om_1}^{-3+2\kappa}}{(1+\abs{\om_1+m_2}^{2})(1+\abs{\om_2-m_2}^{2})}.
		\end{split}
	\end{equation*}
	Let $\eps\in(0,1)$, we bound by Lemma~\ref{lem:convolution_estimates},
	\begin{equation*}
		\begin{split}
			&\sum_{m_2\in\mbZ^{2}}(1+\abs{\om_1+m_2}^{2})^{-1}(1+\abs{\om_2-m_2}^{2})^{-1}\\
			&=2(1+\abs{\om}^{2})^{-1}+\sum_{m_2\in\mbZ^{2}\setminus\{-\om_1,\om_2\}}(1+\abs{\om_1+m_2}^{2})^{-1}(1+\abs{\om_2-m_2}^{2})^{-1}\\
			&\lesssim\abs{\om}^{-2+2\eps}
		\end{split}
	\end{equation*}
	and obtain
	\begin{equation*}
		\abs{\hat{\Triangle{3}}(t,\om,k,j)-\hat{\Triangle{3}}(s,\om,k,j)}\lesssim\norm{\het}_{C_T\mcH^{2}}^{2}\abs{t-s}^{\kappa}\abs{\om}^{-1+2\eps}\sum_{\substack{\om_1,\om_2\in\mbZ^{2}\setminus\{0\}\\\om=\om_1+\om_2\\\om_1\sim\om_2}}(1+\abs{\om}\abs{\om_2}^{-1})\abs{\om_1}^{-3+2\kappa}.
	\end{equation*}
	For any $\kappa\in(0,1/2)$, we bound by Lemma~\ref{lem:convolution_estimates},
	\begin{equation*}
		\abs{\hat{\Triangle{3}}(t,\om,k,j)-\hat{\Triangle{3}}(s,\om,k,j)}\lesssim\norm{\het}_{C_T\mcH^{2}}^{2}\abs{t-s}^{\kappa}\abs{\om}^{-2+2\kappa+2\eps}.
	\end{equation*}
	It follows directly that
	\begin{equation*}
		\norm{\Triangle{3}}_{C_T^{\kappa}\mcC^{-2\kappa-2\eps}}\lesssim\norm{\het}_{C_T\mcH^{2}}^{2}
	\end{equation*}
	\begin{details}
		\begin{equation*}
			\begin{split}
				\norm{\Triangle{3}(t)-\Triangle{3}(s)}_{\mcC^{-2\kappa-2\eps}}&=\norm{(2^{-q(2\kappa+2\eps)}\norm{\Delta_q\Triangle{3}(t)-\Delta_q\Triangle{3}(s)}_{L^\infty})_{q\in\mbN_{-1}}}_{\ell^\infty}\lesssim\norm{\het}_{C_T\mcH^{2}}^{2}\abs{t-s}^{\kappa}.
			\end{split}
		\end{equation*}
	\end{details}
	and we obtain $\Triangle{3}\in\msL_{T}^{\kappa}\mcC^{0-}(\mbT^{2};\mbR^{2\times2})$ for any $\kappa\in(0,1/2)$.
	\begin{details}
		We need to embed $\Triangle{3}$ into a slightly weaker space to ensure it is an element of the closure of smooth functions. We apply Lemma~\ref{lem:Besov_criterion} and use $\norm{\Triangle{3}}_{C_T^{\kappa}\mcC^{-2\kappa-2\eps}}<\infty$ to deduce $\Triangle{3}\in C_{T}^{\kappa}\mcC^{-2\kappa-3\eps}(\mbT^{2};\mbR^{2\times2})$.
	\end{details}
\end{proof}
\subsection{Construction of the Canonical Enhancement}\label{sec:canonical}
In this section we construct $\tl^{\delta}$ and ${\PreThreeloop{10}\!}^{\delta}$, ${\PreThreeloop{20}\!}^{\delta}$ for correlation lengths $\delta>0$. Additionally, we bound their speeds of divergence as $\delta\to0$ by a logarithmic rate using the symmetry of the elliptic equation.
\begin{lemma}\label{lem:existence_canonical}
	Let $T>0$, $\alpha<-2$, $\kappa\in(0,1)$, $\delta>0$, $\het\in C_{T}\mcH^{2}(\mbT^{2})$ and $\varphi$ be as in~\eqref{eq:def_cut_off}. Then it holds that
	\begin{equation*}
		\norm{\tl^{\delta}}_{\msL^{\kappa}_{T}\mcC^{2\alpha+4}}\lesssim(1\vee\log(\delta^{-1}))\norm{\varphi}_{L^{\infty}}^{2}\norm{\het}_{C_{T}\mcH^{2}}^{2},\qquad\norm{\tl^{\delta}}_{\msL^{\kappa}_{T}\mcC^{2\alpha+5}}\lesssim(1\vee(\delta^{-1}\log(\delta^{-1})))\norm{\varphi}_{L^{\infty}}^{2}\norm{\het}_{C_{T}\mcH^{2}}^{2},
	\end{equation*}
	and for any $\kappa\in(0,1/2)$, $p\in[1,\infty)$, 
	\begin{equation*}
		\mbE[\norm{{\PreThreeloop{10}\!}^{\delta}}_{\msL^{\kappa}_{T}\mcC^{3\alpha+6}}^p]^{1/p}\lesssim(1\vee\log(\delta^{-1}))\norm{\varphi}_{L^{\infty}}^{3}\norm{\het}_{C_{T}L^{\infty}}\norm{\het}_{C_{T}\mcH^{2}}^{2},
	\end{equation*}
	and
	\begin{equation*}
		\mbE[\norm{{\PreThreeloop{20}\!}^{\delta}}_{\msL^{\kappa}_{T}\mcC^{3\alpha+6}}^p]^{1/p}\lesssim(1\vee\log(\delta^{-1}))\norm{\varphi}_{L^{\infty}}^{3}\norm{\het}_{C_{T}L^{\infty}}\norm{\het}_{C_{T}\mcH^{2}}^{2}.
	\end{equation*}
\end{lemma}
\begin{proof}
	We first establish $\tl^{\delta}\in\msL_{T}^{1-}\mcC^{0-}(\mbT^{2})$. Let $s,t\in[0,T]$, $\om\in\mbZ^{2}$ and $\kappa\in(0,1)$. It suffices to consider $\om\in\mbZ^{2}\setminus\{0\}$, since $\hat{\tl^{\delta}}(t,0)=0$. We symmetrize the contraction. Changing the r\^{o}les of $\om_1$, $\om_2$ in the definition of $\hat{\tl^{\delta}}(t,\om)$ and using that $\varphi$ is even, we obtain
	\begin{equation*}
		\begin{split}
			\hat{\tl^{\delta}}(t,\om)&=\frac{1}{2}\sum_{\substack{\om_1,\om_2\in\mbZ^{2}\\\om=\om_1+\om_2}}\sum_{j_1,j_3=1}^{2}\sum_{m_1\in\mbZ^{2}}\int_{0}^{t}\dd u_3\int_{0}^{u_3}\dd u_1\hat{\het}(u_1,\om_1-m_1)\hat{\het}(u_1,\om_2+m_1)\abs{\varphi(\delta m_1)}^{2}\\[-16pt]
			&\multiquad[17]\times H^{j_3}_{t-u_3}(\om)H^{j_1}_{u_3-u_1}(\om_1)H^{j_1}_{u_3-u_1}(\om_2)(G^{j_3}(\om_1)+G^{j_3}(\om_2)).
		\end{split}
	\end{equation*}
	We apply the triangle inequality, Lemma~\ref{lem:elliptic_difference} and~\eqref{eq:interpolation} to bound
	\begin{equation*}
		\begin{split}
			&\abs{\hat{\tl^{\delta}}(t,\om)-\hat{\tl^{\delta}}(s,\om)}\\
			&\lesssim\norm{\varphi}_{L^{\infty}}^{2}\norm{\het}_{C_{T}\mcH^{2}}^{2}\abs{t-s}^{\kappa}\abs{\om}^{2\kappa}\sum_{\substack{m_1\in\mbZ^{2}\\\abs{m_1}\leq\delta^{-1}}}\sum_{\om_1\in\mbZ^{2}\setminus\{0,\om\}}\frac{\abs{\om_1}^{-2}(1+\abs{\om}\abs{\om-\om_1}^{-1})}{(1+\abs{\om_1-m_1}^{2})(1+\abs{\om-\om_1+m_1}^{2})}.
		\end{split}
	\end{equation*}
	We can now apply Lemma~\ref{lem:sum_m_om_canonical} to control the sums over $m_1\in\mbZ^{2}$, $\abs{m_1}\leq\delta^{-1}$, and $\om_1\in\mbZ^{2}\setminus\{0,\om\}$ for every $\eps\in(0,1/2)$,
	\begin{equation*}
		\abs{\hat{\tl^{\delta}}(t,\om)-\hat{\tl^{\delta}}(s,\om)}\lesssim(1\vee\log(\delta^{-1}))\norm{\varphi}_{L^{\infty}}^{2}\norm{\het}_{C_{T}\mcH^{2}}^{2}\abs{t-s}^{\kappa}\abs{\om}^{-2+2\kappa+3\eps}
	\end{equation*}
	and so
	\begin{equation*}
		\norm{\tl^{\delta}}_{\msL_{T}^{\kappa}\mcC^{-3\eps}}\lesssim(1\vee\log(\delta^{-1}))\norm{\varphi}_{L^{\infty}}^{2}\norm{\het}_{C_{T}\mcH^{2}}^{2}.
	\end{equation*}
	The proof that $\tl^{\delta}\in\msL_{T}^{\kappa}\mcC^{1-}(\mbT^{2})$ with a divergence of order $1\vee(\delta^{-1}\log(\delta^{-1}))$ follows by a similar derivation, where we skip the symmetrization and use that $\abs{\varphi(\delta m_1)}\lesssim(1+\abs{\delta m_1}^{2})^{-1/4}\norm{\varphi}_{L^{\infty}}$.
	\begin{details}
		We use $\supp(\varphi)\subset B(0,1)$ to estimate
		\begin{equation*}
			\sup_{x\in\mbR^{2}}(1+\abs{x}^{2})^{1/4}\abs{\varphi(x)}\lesssim\norm{\varphi}_{L^{\infty}},
		\end{equation*}
		which implies
		\begin{equation*}
			\abs{\varphi(\delta m_1)}=(1+\abs{\delta m_{1}}^{2})^{-1/4}(1+\abs{\delta m_{1}}^{2})^{1/4}\abs{\varphi(\delta m_1)}\lesssim(1+\abs{\delta m_{1}}^{2})^{-1/4}\norm{\varphi}_{L^{\infty}}.
		\end{equation*}
	\end{details}
	\begin{details}
		We estimate as before,
		\begin{equation*}
			\begin{split}
				&\abs{\hat{\tl^{\delta}}(t,\om)-\hat{\tl^{\delta}}(s,\om)}\\
				&\lesssim\norm{\varphi}_{L^{\infty}}^{2}\norm{\het}_{C_{T}\mcH^{2}}^{2}\abs{t-s}^{\kappa}\abs{\om}^{2\kappa-1}\sum_{\substack{m_1\in\mbZ^{2}\\\abs{m_1}\leq\delta^{-1}}}\sum_{\om_1\in\mbZ^{2}\setminus\{0,\om\}}\frac{\abs{\om_1}^{-1}}{(1+\abs{\delta m_1}^{2})^{1/2}(1+\abs{\om_1-m_1}^{2})(1+\abs{\om-\om_1+m_1}^{2})},
			\end{split}
		\end{equation*}
		where we used that $\abs{\varphi(\delta m_1)}\lesssim(1+\abs{\delta m_1}^{2})^{-1/4}\norm{\varphi}_{L^{\infty}}$. We can control the double sum by Lemma~\ref{lem:sum_m_om_canonical_regular} and obtain for any $\eps\in(0,1)$,
		\begin{equation*}
			\abs{\hat{\tl^{\delta}}(t,\om)-\hat{\tl^{\delta}}(s,\om)}\lesssim(1\vee(\delta^{-1}\log(\delta^{-1})))\norm{\varphi}_{L^{\infty}}^{2}\norm{\het}_{C_{T}\mcH^{2}}^{2}\abs{t-s}^{\kappa}\abs{\om}^{2\kappa-3+3\eps}.
		\end{equation*}
		This yields the claim.
	\end{details}
	
	Next we establish that ${\PreThreeloop{10}\!}^{\delta},{\PreThreeloop{20}\!}^{\delta}\in\msL_{T}^{1/2-}\mcC^{0-}(\mbT^{2};\mbR^{2})$. We first consider ${\PreThreeloop{10}\!}^{\delta}$. We apply It\^{o}'s isometry, \eqref{eq:isometry_bound_delta>0} and Lemma~\ref{lem:elliptic_difference} to bound uniformly in $x\in\mbT^{2}$,
	\begin{equation*}
		\begin{split}
			&\mbE[\abs{\Delta_q{\PreThreeloop{10}\!}^{\delta}(t,x,j)-\Delta_q{\PreThreeloop{10}\!}^{\delta}(s,x,j)}^2]\\
			&\lesssim\norm{\varphi}_{L^{\infty}}^{6}\norm{\het}_{C_{T}L^{\infty}}^{2}\norm{\het}_{C_{T}\mcH^{2}}^{4}\sum_{\om,\om'\in\mbZ^{2}}\varrho_{q}(\om)\varrho_{q}(\om')\\
			&\quad\times\sum_{\substack{\om_4\in\mbZ^{2}\setminus\{0,\om,\om'\}\\(\om-\om_4)\sim\om_4\\(\om'-\om_4)\sim\om_4}}\sum_{\substack{\om_1,\om_2\in\mbZ^{2}\setminus\{0\}\\\om-\om_4=\om_1+\om_2}}\sum_{\substack{\om_1',\om_2'\in\mbZ^{2}\setminus\{0\}\\\om'-\om_4=\om_1'+\om_2'}}\sum_{j_3,j_3'=1}^{2}\sum_{\substack{m_1\in\mbZ^{2}\\\abs{m_1}\leq\delta^{-1}}}\sum_{\substack{m_1'\in\mbZ^{2}\\\abs{m_1'}\leq\delta^{-1}}}\\
			&\quad(1+\abs{\om_1-m_1}^{2})^{-1}(1+\abs{\om_2+m_1}^{2})^{-1}(1+\abs{\om_1'-m_1'}^{2})^{-1}(1+\abs{\om_2'+m_1'}^{2})^{-1}\abs{\om_4}^{-2}\\
			&\quad\times\abs{\om-\om_4}\abs{\om'-\om_4}\abs{\om_2}^{-2}\abs{\om_2'}^{-2}(1+\abs{\om-\om_4}\abs{\om_1}^{-1})(1+\abs{\om'-\om_4}\abs{\om_1'}^{-1})\\
			&\quad\times\mathsf{A}_{s,t}^{j_3,j_3'}\SVee(\om_1+\om_2,\om_1'+\om_2',\om_4).
		\end{split}
	\end{equation*}
	\begin{details}
		We symmetrize
		\begin{equation*}
			\begin{split}
				&\hat{{\PreThreeloop{10}\!}^{\delta}}(t,\om,j)\\
				&=\frac{1}{2}\sum_{\substack{\om_1,\om_2,\om_4\in\mbZ^{2}\\\om=\om_1+\om_2+\om_4\\(\om_1+\om_2)\sim\om_4}}\sum_{j_1,j_3,j_4=1}^{2}\sum_{m_1,m_4\in\mbZ^{2}}\int_{0}^{t}\dd W^{j_4}(u_4,m_4)\int_{0}^{t}\dd u_3\int_{0}^{u_3}\dd u_1\\
				&\quad\hat{\het}(u_1,\om_1-m_1)\hat{\het}(u_1,\om_2+m_1)\hat{\het}(u_4,\om_4-m_4)\abs{\varphi(\delta m_1)}^{2}\varphi(\delta m_4)\\
				&\quad\times H_{t-u_3}^{j_3}(\om_1+\om_2)G^{j}(\om_4)H_{t-u_4}^{j_4}(\om_4)H_{u_3-u_1}^{j_1}(\om_1)(G^{j_3}(\om_2)+G^{j_3}(\om_1))H_{u_3-u_1}^{j_1}(\om_2).
			\end{split}
		\end{equation*}
		Let $x\in\mbT^{2}$, $j=1,2$, $\delta>0$ and $q\in\mbN_{-1}$. An application of~\eqref{eq:isometry_bound_delta>0} yields uniformly in $x\in\mbT^{2}$,
		\begin{equation*}
			\mbE[\abs{\Delta_q{\PreThreeloop{10}\!}^{\delta}(t,x,j)-\Delta_q{\PreThreeloop{10}\!}^{\delta}(s,x,j)}^2]\leq\norm{\varphi}_{L^{\infty}}^{2}\norm{\het}_{C_{T}L^{\infty}}^{2}\mbE[\abs{\Delta_q{\PreThreeloop{1}\!}^{\delta}(t,x,j)-\Delta_q{\PreThreeloop{1}\!}^{\delta}(s,x,j)}^2],
		\end{equation*}
		where ${\PreThreeloop{1}\!}^{\delta}$ is defined by
		\begin{equation*}
			\begin{split}
				&\hat{{\PreThreeloop{1}\!}^{\delta}}(t,\om,j)\\
				&\defeq\frac{1}{2}\sum_{\substack{\om_1,\om_2,\om_4\in\mbZ^{2}\\\om=\om_1+\om_2+\om_4\\(\om_1+\om_2)\sim\om_4}}\sum_{j_1,j_3,j_4=1}^{2}\sum_{m_1\in\mbZ^{2}}\int_{0}^{t}\dd W^{j_4}(u_4,\om_4)\int_{0}^{t}\dd u_3\int_{0}^{u_3}\dd u_1\hat{\het}(u_1,\om_1-m_1)\hat{\het}(u_1,\om_2+m_1)\abs{\varphi(\delta m_1)}^{2}\\
				&\quad\times H_{t-u_3}^{j_3}(\om_1+\om_2)G^{j}(\om_4)H_{t-u_4}^{j_4}(\om_4)H_{u_3-u_1}^{j_1}(\om_1)(G^{j_3}(\om_2)+G^{j_3}(\om_1))H_{u_3-u_1}^{j_1}(\om_2).
			\end{split}
		\end{equation*}
		By the definition of the Littlewood-Paley block $\Delta_{q}$,
		\begin{equation*}
			\begin{split}
				&\mbE[\abs{\Delta_q{\PreThreeloop{1}\!}^{\delta}(t,x,j)-\Delta_q{\PreThreeloop{1}\!}^{\delta}(s,x,j)}^2]\\
				&\leq\sum_{\om,\om'\in\mbZ^{2}}\varrho_q(\om)\varrho_q(\om')\Bigl\lvert\mbE\Bigl[\Bigl(\hat{{\PreThreeloop{1}\!}^{\delta}}(t,\om,j)-\hat{{\PreThreeloop{1}\!}^{\delta}}(s,\om,j)\Bigr)\overline{\Bigl(\hat{{\PreThreeloop{1}\!}^{\delta}}(t,\om',j)-\hat{{\PreThreeloop{1}\!}^{\delta}}(s,\om',j)\Bigr)}\Bigr]\Bigr\rvert,
			\end{split}
		\end{equation*}
		where
		\begin{equation*}
			\begin{split}
				&\Bigl\lvert\mbE\Bigl[\Bigl(\hat{{\PreThreeloop{1}\!}^{\delta}}(t,\om,j)-\hat{{\PreThreeloop{1}\!}^{\delta}}(s,\om,j)\Bigr)\overline{\Bigl(\hat{{\PreThreeloop{1}\!}^{\delta}}(t,\om',j)-\hat{{\PreThreeloop{1}\!}^{\delta}}(s,\om',j)\Bigr)}\Bigr]\Bigr\rvert\\
				&=\frac{1}{4}\sum_{\substack{\om_1,\om_2,\om_4\in\mbZ^{2}\\\om=\om_1+\om_2+\om_4\\(\om_1+\om_2)\sim\om_4}}\sum_{\substack{\om_1',\om_2'\in\mbZ^{2}\\\om'=\om_1'+\om_2'+\om_4\\(\om_1'+\om_2')\sim\om_4}}\sum_{j_1,j_1',j_3,j_3',j_4=1}^{2}\sum_{m_1,m_1'\in\mbZ^{2}}\int_{0}^{\infty}\dd u_4\int_{0}^{\infty}\dd u_3\int_{0}^{u_3}\dd u_1\int_{0}^{\infty}\dd u_3'\int_{0}^{u_3'}\dd u_1'\\
				&\quad\hat{\het}(u_1,\om_1-m_1)\hat{\het}(u_1,\om_2+m_1)\overline{\hat{\het}(u_1',\om_1'-m_1')\hat{\het}(u_1',\om_2'+m_1')}\\
				&\quad\times\abs{\varphi(\delta m_1)}^{2}\abs{\varphi(\delta m_1')}^{2} G^{j}(\om_4)(G^{j_3}(\om_2)+G^{j_3}(\om_1))\overline{G^{j}(\om_4)(G^{j_3'}(\om_2')+G^{j_3'}(\om_1'))}\\
				&\quad\times (H_{t-u_3}^{j_3}(\om_1+\om_2)H_{t-u_4}^{j_4}(\om_4)-H_{s-u_3}^{j_3}(\om_1+\om_2)H_{s-u_4}^{j_4}(\om_4))H_{u_3-u_1}^{j_1}(\om_1)H_{u_3-u_1}^{j_1}(\om_2)\\
				&\quad\times \overline{(H_{t-u_3'}^{j_3'}(\om_1'+\om_2')H_{t-u_4}^{j_4}(\om_4)-H_{s-u_3'}^{j_3'}(\om_1'+\om_2')H_{s-u_4}^{j_4}(\om_4))H_{u_3'-u_1'}^{j_1'}(\om_1')H_{u_3'-u_1'}^{j_1'}(\om_2')}.
			\end{split}
		\end{equation*}
		Taking the absolute value,
		\begin{equation*}
			\begin{split}
				&\Bigl\lvert\mbE\Bigl[\Bigl(\hat{{\PreThreeloop{1}\!}^{\delta}}(t,\om,j)-\hat{{\PreThreeloop{1}\!}^{\delta}}(s,\om,j)\Bigr)\overline{\Bigl(\hat{{\PreThreeloop{1}\!}^{\delta}}(t,\om',j)-\hat{{\PreThreeloop{1}\!}^{\delta}}(s,\om',j)\Bigr)}\Bigr]\Bigr\rvert\\
				&\lesssim\sum_{\substack{\om_1,\om_2,\om_4\in\mbZ^{2}\\\om=\om_1+\om_2+\om_4\\(\om_1+\om_2)\sim\om_4}}\sum_{\substack{\om_1',\om_2'\in\mbZ^{2}\\\om'=\om_1'+\om_2'+\om_4\\(\om_1'+\om_2')\sim\om_4}}\sum_{j_1,j_1',j_3,j_3',j_4=1}^{2}\sum_{m_1,m_1'\in\mbZ^{2}}\int_{0}^{\infty}\dd u_4\int_{0}^{\infty}\dd u_3\int_{0}^{u_3}\dd u_1\int_{0}^{\infty}\dd u_3'\int_{0}^{u_3'}\dd u_1'\\
				&\quad\abs{\hat{\het}(u_1,\om_1-m_1)}\abs{\hat{\het}(u_1,\om_2+m_1)}\abs{\hat{\het}(u_1',\om_1'-m_1')}\abs{\hat{\het}(u_1',\om_2'+m_1')}\\
				&\quad\times\abs{\varphi(\delta m_1)}^{2}\abs{\varphi(\delta m_1')}^{2}\abs{G^{j}(\om_4)}^2\abs{G^{j_3}(\om_2)+G^{j_3}(\om_1)}\abs{G^{j_3'}(\om_2')+G^{j_3'}(\om_1')}\\
				&\quad\times \abs{(H_{t-u_3}^{j_3}(\om_1+\om_2)H_{t-u_4}^{j_4}(\om_4)-H_{s-u_3}^{j_3}(\om_1+\om_2)H_{s-u_4}^{j_4}(\om_4))H_{u_3-u_1}^{j_1}(\om_1)H_{u_3-u_1}^{j_1}(\om_2)}\\
				&\quad\times \abs{(H_{t-u_3'}^{j_3'}(\om_1'+\om_2')H_{t-u_4}^{j_4}(\om_4)-H_{s-u_3'}^{j_3'}(\om_1'+\om_2')H_{s-u_4}^{j_4}(\om_4))H_{u_3'-u_1'}^{j_1'}(\om_1')H_{u_3'-u_1'}^{j_1'}(\om_2')}.
			\end{split}
		\end{equation*}
		We estimate $\sup_{u_1\in[0,T]}\abs{\hat{\het}(u_1,\om_1-m_1)}\lesssim\norm{\het}_{C_{T}\mcH^{2}}(1+\abs{\om_1-m_1}^{2})^{-1}$ and subsequently
		\begin{equation*}
			\sum_{j_1=1}^{2}\int_{-\infty}^{u_3}\dd u_1\abs{H_{u_3-u_1}^{j_1}(\om_1)H_{u_3-u_1}^{j_1}(\om_2)}=\frac{\inner{\om_1}{\om_2}}{\abs{\om_1}^2+\abs{\om_2}^2}\lesssim 1.
		\end{equation*}
		So,
		\begin{equation*}
			\begin{split}
				&\Bigl\lvert\mbE\Bigl[\Bigl(\hat{{\PreThreeloop{1}\!}^{\delta}}(t,\om,j)-\hat{{\PreThreeloop{1}\!}^{\delta}}(s,\om,j)\Bigr)\overline{\Bigl(\hat{{\PreThreeloop{1}\!}^{\delta}}(t,\om',j)-\hat{{\PreThreeloop{1}\!}^{\delta}}(s,\om',j)\Bigr)}\Bigr]\Bigr\rvert\\
				&\lesssim\norm{\varphi}_{L^{\infty}}^{4}\norm{\het}_{C_{T}\mcH^{2}}^{4}\sum_{\substack{\om_1,\om_2,\om_4\in\mbZ^{2}\\\om=\om_1+\om_2+\om_4\\(\om_1+\om_2)\sim\om_4}}\sum_{\substack{\om_1',\om_2'\in\mbZ^{2}\\\om'=\om_1'+\om_2'+\om_4\\(\om_1'+\om_2')\sim\om_4}}\sum_{j_3,j_3',j_4=1}^{2}\sum_{\substack{m_1\in\mbZ^{2}\\\abs{m_1}\leq\delta^{-1}}}\sum_{\substack{m_1'\in\mbZ^{2}\\\abs{m_1'}\leq\delta^{-1}}}\\
				&\quad(1+\abs{\om_1-m_1}^{2})^{-1}(1+\abs{\om_2+m_1}^{2})^{-1}(1+\abs{\om_1'-m_1'}^{2})^{-1}(1+\abs{\om_2'+m_1'}^{2})^{-1}\\
				&\quad\times\abs{G^{j}(\om_4)}^2\abs{G^{j_3}(\om_2)+G^{j_3}(\om_1)}\abs{G^{j_3'}(\om_2')+G^{j_3'}(\om_1')}\\
				&\quad\times\int_{-\infty}^{\infty}\dd u_4\int_{-\infty}^{\infty}\dd u_3\int_{-\infty}^{\infty}\dd u_3'\abs{H_{t-u_3}^{j_3}(\om_1+\om_2)H_{t-u_4}^{j_4}(\om_4)-H_{s-u_3}^{j_3}(\om_1+\om_2)H_{s-u_4}^{j_4}(\om_4)}\\
				&\multiquad[14]\times \abs{H_{t-u_3'}^{j_3'}(\om_1'+\om_2')H_{t-u_4}^{j_4}(\om_4)-H_{s-u_3'}^{j_3'}(\om_1'+\om_2')H_{s-u_4}^{j_4}(\om_4)}.
			\end{split}
		\end{equation*}
		Hence, for $\om,\om'\in\mbZ^{2}$ by It\^{o}'s isometry,
		\begin{equation*}
			\begin{split}
				&\Bigl\lvert\mbE\Bigl[\Bigl(\hat{{\PreThreeloop{1}\!}^{\delta}}(t,\om,j)-\hat{{\PreThreeloop{1}\!}^{\delta}}(s,\om,j)\Bigr)\overline{\Bigl(\hat{{\PreThreeloop{1}\!}^{\delta}}(t,\om',j)-\hat{{\PreThreeloop{1}\!}^{\delta}}(s,\om',j)\Bigr)}\Bigr]\Bigr\rvert\\
				&\lesssim\norm{\varphi}_{L^{\infty}}^{4}\norm{\het}_{C_{T}\mcH^{2}}^{4}\sum_{\substack{\om_4\in\mbZ^{2}\setminus\{0,\om,\om'\}\\(\om-\om_4)\sim\om_4\\(\om'-\om_4)\sim\om_4}}\sum_{\substack{\om_1,\om_2\in\mbZ^{2}\setminus\{0\}\\\om-\om_4=\om_1+\om_2}}\sum_{\substack{\om_1',\om_2'\in\mbZ^{2}\setminus\{0\}\\\om'-\om_4=\om_1'+\om_2'}}\sum_{j_3,j_3'=1}^{2}\sum_{\substack{m_1\in\mbZ^{2}\\\abs{m_1}\leq\delta^{-1}}}\sum_{\substack{m_1'\in\mbZ^{2}\\\abs{m_1'}\leq\delta^{-1}}}\\
				&\quad(1+\abs{\om_1-m_1}^{2})^{-1}(1+\abs{\om_2+m_1}^{2})^{-1}(1+\abs{\om_1'-m_1'}^{2})^{-1}(1+\abs{\om_2'+m_1'}^{2})^{-1}\\
				&\quad\times\abs{G^{j}(\om_4)}^2\abs{G^{j_3}(\om_2)+G^{j_3}(\om_1)}\abs{G^{j_3'}(\om_2')+G^{j_3'}(\om_1')}\mathsf{A}_{s,t}^{j_3,j_3'}\SVee(\om_1+\om_2,\om_1'+\om_2',\om_4).
			\end{split}
		\end{equation*}
	\end{details}
	Assume $s\leq t$, $\gamma\in[0,1]$, $\delta>0$ and $\eps\in(0,1/2)$ whose range we continue to restrict throughout the proof. We apply Lemma~\ref{lem:v_regularity} to control the shape coefficient and subsequently Lemma~\ref{lem:sum_m_om_canonical} to obtain
	\begin{equation*}
		\begin{split}
			&\mbE[\abs{\Delta_q{\PreThreeloop{10}\!}^{\delta}(t,x,j)-\Delta_q{\PreThreeloop{10}\!}^{\delta}(s,x,j)}^2]\\
			&\lesssim(1\vee\log(\delta^{-1}))^{2}\norm{\varphi}_{L^{\infty}}^{6}\norm{\het}_{C_{T}L^{\infty}}^{2}\norm{\het}_{C_{T}\mcH^{2}}^{4}\abs{t-s}^{\gamma}\sum_{\om,\om'\in\mbZ^{2}}\varrho(\om)\varrho(\om')\\
			&\quad\times\sum_{\substack{\om_4\in\mbZ^{2}\setminus\{0,\om,\om'\}\\(\om-\om_4)\sim\om_4\\(\om'-\om_4)\sim\om_4}}\abs{\om_4}^{-2}\abs{\om-\om_4}^{-2+\gamma+3\eps}\abs{\om'-\om_4}^{-2+\gamma+3\eps}.
		\end{split}
	\end{equation*}
	Assume $6\eps<4-2\gamma$, we apply H\"{o}lder's inequality,
	\begin{equation}\label{eq:prethreeloop_colour_difference}
		\begin{split}
			&\sum_{\substack{\om_4\in\mbZ^{2}\setminus\{0,\om,\om'\}\\(\om-\om_4)\sim\om_4\\(\om'-\om_4)\sim\om_4}}\abs{\om_4}^{-2}\abs{\om-\om_4}^{-2+\gamma+3\eps}\abs{\om'-\om_4}^{-2+\gamma+3\eps}\\
			&\leq\Bigl(\sum_{\substack{\om_4\in\mbZ^{2}\setminus\{0,\om\}\\(\om-\om_4)\sim\om_4}}\abs{\om_4}^{-2}\abs{\om-\om_4}^{-4+2\gamma+6\eps}\Bigr)^{1/2}\Bigl(\sum_{\substack{\om_4\in\mbZ^{2}\setminus\{0,\om'\}\\(\om'-\om_4)\sim\om_4}}\abs{\om_4}^{-2}\abs{\om'-\om_4}^{-4+2\gamma+6\eps}\Bigr)^{1/2}\\
			&\lesssim(1\vee\abs{\om})^{-2+\gamma+3\eps}(1\vee\abs{\om'})^{-2+\gamma+3\eps},
		\end{split}
	\end{equation}
	which implies
	\begin{equation*}
		\begin{split}
			&\mbE[\abs{\Delta_q{\PreThreeloop{10}\!}^{\delta}(t,x,j)-\Delta_q{\PreThreeloop{10}\!}^{\delta}(s,x,j)}^2]\\
			&\lesssim(1\vee\log(\delta^{-1}))^{2}\norm{\varphi}_{L^{\infty}}^{6}\norm{\het}_{C_{T}L^{\infty}}^{2}\norm{\het}_{C_{T}\mcH^{2}}^{4}\abs{t-s}^{\gamma}\sum_{\om,\om'\in\mbZ^{2}}\varrho_{q}(\om)\varrho_{q}(\om')(1\vee\abs{\om})^{-2+\gamma+3\eps}(1\vee\abs{\om'})^{-2+\gamma+3\eps}.
		\end{split}
	\end{equation*}
	Assume in addition $\gamma\in(0,1)$ and $\eps\in(0,\gamma/2)$, we obtain by Lemma~\ref{lem:Nelson_estimate} and Lemma~\ref{lem:existence_criterion} for any $p\in[1,\infty)$, 
	\begin{equation*}
		\mbE[\norm{{\PreThreeloop{10}\!}^{\delta}}_{C_T^{\gamma/2-\eps}\mcC^{-\gamma-6\eps}}^{p}]^{1/p}\lesssim(1\vee\log(\delta^{-1}))\norm{\varphi}_{L^{\infty}}^{3}\norm{\het}_{C_{T}L^{\infty}}\norm{\het}_{C_{T}\mcH^{2}}^{2}
	\end{equation*}
	and therefore ${\PreThreeloop{10}\!}^{\delta}\in\msL^{\kappa}_{T}\mcC^{0-}(\mbT^{2};\mbR^{2})$ a.s.\ for any $\kappa\in(0,1/2)$.
	
	The only difference between $\hat{{\PreThreeloop{20}\!}^{\delta}}$ and $\hat{{\PreThreeloop{10}\!}^{\delta}}$ is that the factor $G^{j}(\om_1+\om_2)$ replaces $G^{j}(\om_4)$. Instead of~\eqref{eq:prethreeloop_colour_difference}, we estimate by H\"{o}lder's inequality,
	\begin{equation*}
		\begin{split}
			&\sum_{\substack{\om_4\in\mbZ^{2}\setminus\{0,\om,\om'\}\\(\om-\om_4)\sim\om_4\\(\om'-\om_4)\sim\om_4}}\abs{\om-\om_4}^{-3+\gamma+3\eps}\abs{\om'-\om_4}^{-3+\gamma+3\eps}\\
			&\leq\Bigl(\sum_{\substack{\om_4\in\mbZ^{2}\setminus\{\om\}\\(\om-\om_4)\sim\om_4}}\abs{\om-\om_4}^{-6+2\gamma+6\eps}\Bigr)^{1/2}\Bigl(\sum_{\substack{\om_4\in\mbZ^{2}\setminus\{\om'\}\\(\om'-\om_4)\sim\om_4}}\abs{\om'-\om_4}^{-6+2\gamma+6\eps}\Bigr)^{1/2}\\
			&\lesssim(1\vee\abs{\om})^{-2+\gamma+3\eps}(1\vee\abs{\om'})^{-2+\gamma+3\eps}.
		\end{split}
	\end{equation*}
	As before, we obtain by~\eqref{eq:isometry_bound_delta>0}, Lemma~\ref{lem:Nelson_estimate} and Lemma~\ref{lem:existence_criterion} for any $\gamma\in(0,1)$, $6\eps<4-2\gamma$, $\eps<\gamma/2$, $\delta>0$ and $p\in[1,\infty)$,
	\begin{equation*}
		\mbE[\norm{{\PreThreeloop{20}\!}^{\delta}}_{C_T^{\gamma/2-\eps}\mcC^{-\gamma-6\eps}}^{p}]^{1/p}\lesssim(1\vee\log(\delta^{-1}))\norm{\varphi}_{L^{\infty}}^{3}\norm{\het}_{C_{T}L^{\infty}}\norm{\het}_{C_{T}\mcH^{2}}^{2}
	\end{equation*}
	and therefore ${\PreThreeloop{20}\!}^{\delta}\in\msL^{\kappa}_{T}\mcC^{0-}(\mbT^{2};\mbR^{2})$ a.s.\ for any $\kappa\in(0,1/2)$.
\end{proof}
\begin{remark}
	We may also construct $\Loop^{\delta}=\mbE[\ti^{\delta}\nabla\Phi_{\ti^{\delta}}]$, which is similar to $\tl^{\delta}$ but without the lower stem (cf.~\eqref{eq:diagram_Ito}). However, to obtain $\kappa$-time regularity, we need to trade $2\kappa$-space regularity in the parabolic multipliers $H_{t-u_1}^{j_1}(\om_1)H_{t-u_1}^{j_1}(\om_2)$. We then sum over $\om_1,\om_2\in\mbZ^{2}$, hence we will be stuck with a divergence of $\delta^{-2\kappa}$, where $\kappa$ is arbitrarily small but positive.
\end{remark}
\begin{remark}\label{rem:renormalisation_vanishes_homogeneous}
	We can show that $\Loop^{\delta}\equiv0$, if $\het\equiv1$. Indeed, if we choose $\het\equiv1$, then $\hat{\het}(u,\om)=\mathds{1}_{\om=0}$. Consequently on the right-hand side of~\eqref{eq:diagram_Ito}, $\om_1=m_1$ and $\om_2=-m_1$. By the symmetrization, we obtain the factor $G^{j}(\om_1)+G^{j}(\om_2)$. Using that $G^{j}$ is odd and that $\om_1=-\om_2$, we see that this term is zero, so that $\Loop^{\delta}\equiv0$. Similarly, $\tl^{\delta}={\PreThreeloop{10}\!}^{\delta}={\PreThreeloop{20}\!}^{\delta}\equiv0$, if $\het\equiv1$.
\end{remark}
\section{Existence of Paracontrolled Solutions}\label{sec:existence}
In this section we show the existence and uniqueness of a paracontrolled solution to~\eqref{eq:gen_rKS_intro} given an abstract enhancement, where the notion of paracontrolled solution has previously been motivated in~\eqref{eq:paracon_sol_intro_I}--\eqref{eq:paracon_sol_intro_II} and will be made precise in Definition~\ref{def:paracontrolled_solution}.

In Subsection~\ref{subsec:local_well_posedness} we consider the solution on small time intervals (Lemma~\ref{lem:construction_paracontrolled}), which we then extend in Subsection~\ref{subsec:maximal_time_of_existence} to a maximal time of existence (Lemma~\ref{lem:paracontrolled_sol_in_blow_up_space}). In Subsection~\ref{subsec:ren_sol} we then combine the deterministic solution theory above with the stochastic existence of the renormalised enhancement (Theorem~\ref{thm:enhancement_existence}) to construct the renormalised solution to~\eqref{eq:gen_rKS_intro} as a random variable (Theorem~\ref{thm:convergence_renorm_sol}, Part~\ref{it:thm_existence_renormalized}). We can then show that solutions to~\eqref{eq:smooth_mild_sol} converge in probability to the renormalised solution (Theorem~\ref{thm:convergence_renorm_sol}, Part~\ref{it:thm_convergence_to_renormalized}).
\subsection{Local Well-Posedness}\label{subsec:local_well_posedness}
Throughout we fix exponents satisfying the assumptions below. To explain their usage: $p$ and $\beta_0$ will be the integrability and regularity exponents of the admissible initial condition in the Besov scale $\mcB^{\beta_0}_{p,q}(\mbT^{2})$, where $q$ is the microscopic parameter; $\alpha$ will be the regularity of the space-time white noise, so that almost surely $\ti\in C_{T}\mcC^{\alpha+1}(\mbT^{2})$; $\beta$ measures the regularity of the second Da~Prato--Debussche remainder, $w$, in the H\"{o}lder scale and $\eta$ the allowed blow-up of $w$ at $t=0$; $\beta'$ measures the regularity of the Gubinelli derivative $w'$; $\beta^{\#}$ measures the maximal spatial regularity of the paracontrolled remainder and finally $\kappa$ will be used to denote time regularity.

From now on we fix $(\alpha,p,q,\beta,\beta',\beta^{\#},\beta_{0},\kappa,\eta)$ satisfying
\begin{equation}\label{eq:exponents}
	\begin{split}
		\alpha\in(-9/4,-2),&\qquad q\in[1,\infty],\\
		p\in(2/(\alpha+3),\infty],&\qquad\beta\in(-1/2,\alpha+2),\\
		\beta'\in(-2\alpha-4,(\beta+1)\wedge(2\alpha+5)],&\qquad\beta^{\#}\in(-\alpha-2,\alpha+\beta'+2),\\
		\beta_{0}\in(\beta+2/p-(\alpha+3),\beta^{\#}],&\qquad\kappa\in((\beta^{\#}-\alpha-2)/2,1/2),\\
		\eta\in\Bigl[\Bigl(\Bigl(\frac{\beta-\beta_{0}}{2}\vee0\Bigr)+\frac{1}{p}\Bigr)&\vee\Bigl(\frac{\beta^{\#}-\beta_0}{4}+\frac{1}{2p}\Bigr),\frac{\alpha+3}{2}\Bigr).
	\end{split}
\end{equation}
\begin{details}
	\paragraph{Derivation of~\eqref{eq:exponents}.}
	For our a priori bounds, it is sufficient to assume
	\begin{equation*}
		\begin{split}
			\alpha\in(-9/4,-2),&\qquad q\in[1,\infty],\\
			p\in[1,\infty],&\qquad\beta\in(-1/2,\alpha+2),\\			
			\beta'\in(-2\alpha-4,(\beta+1)\wedge(2\alpha+5)],&\qquad\beta^{\#}\in(-\alpha-2,\alpha+\beta'+2),\\
			\beta_{0}\in(-\infty,\beta^{\#}],&\qquad\kappa\in((\beta^{\#}-\alpha-2)/2,1/2),\\
			\eta\in\Bigl[\Bigl(\Bigl(\frac{\beta-\beta_{0}}{2}\vee0\Bigr)+\frac{1}{p}\Bigr)&\vee\Bigl(\frac{\beta^{\#}-\beta_{0}}{4}+\frac{1}{2p}\Bigr),\frac{1}{2}\wedge\frac{\alpha+3}{2}\Bigr).
		\end{split}
	\end{equation*}
	We can now resolve the case distinctions above, subject to the condition that those intervals need to be non-empty.
	
	We start by resolving the case distinction in the upper bound on $\eta$,
	\begin{equation*}
		\frac{\alpha+3}{2}<\frac{1}{2}\qquad\iff\qquad\alpha<-2.
	\end{equation*}
	To proceed with the lower bound on $\eta$, we distinguish between $\beta_{0}\in(-\infty,\beta]$ and $\beta_{0}\in(\beta,\beta^{\#}]$.
	\paragraph{Case 1.}
	Assume $\beta_{0}\leq\beta$. To ensure that the resulting interval for  $\eta\in[(\frac{\beta-\beta_{0}}{2}+\frac{1}{p})\vee(\frac{\beta^{\#}-\beta_{0}}{4}+\frac{1}{2p}),\frac{\alpha+3}{2})$ is non-empty, we need to assume
	\begin{equation*}
		\frac{\beta-\beta_{0}}{2}+\frac{1}{p}<\frac{\alpha+3}{2}\qquad\iff\qquad\beta+\frac{2}{p}-(\alpha+3)<\beta_{0}
	\end{equation*}
	and
	\begin{equation*}
		\frac{\beta^{\#}-\beta_{0}}{4}+\frac{1}{2p}<\frac{\alpha+3}{2}\qquad\iff\qquad\beta^{\#}+\frac{2}{p}-2(\alpha+3)<\beta_{0}.
	\end{equation*}
	We place both as lower bounds on $\beta_{0}$ and note
	\begin{equation*}
		\beta^{\#}+\frac{2}{p}-2(\alpha+3)<\beta+\frac{2}{p}-(\alpha+3)\qquad\iff\qquad\beta^{\#}<\alpha+\beta+3,
	\end{equation*}
	which is true by the assumptions $\beta^{\#}<\alpha+\beta'+2$ and $\beta'<\beta+1$. Hence, our intervals simplify to
	\begin{equation*}
		\beta_{0}\in(\beta+2/p-(\alpha+3),\beta],\qquad\eta\in\Bigl[\Bigl(\frac{\beta-\beta_{0}}{2}+\frac{1}{p}\Bigr)\vee\Bigl(\frac{\beta^{\#}-\beta_{0}}{4}+\frac{1}{2p}\Bigr),\frac{\alpha+3}{2}\Bigr).
	\end{equation*}
	To ensure 
	\begin{equation*}
		\beta+\frac{2}{p}-(\alpha+3)<\beta\qquad\iff\qquad\frac{2}{p}<\alpha+3,
	\end{equation*}
	we assume $p>2/(\alpha+3)$. Hence in the case $\beta_{0}\leq\beta$, we can simplify our intervals to
	\begin{equation*}
		\begin{split}
			\alpha\in(-9/4,-2),&\qquad q\in[1,\infty],\\
			p\in(2/(\alpha+3),\infty],&\qquad\beta\in(-1/2,\alpha+2),\\
			\beta'\in(-2\alpha-4,(\beta+1)\wedge(2\alpha+5)],&\qquad\beta^{\#}\in(-\alpha-2,\alpha+\beta'+2),\\
			\beta_0\in(\beta+2/p-(\alpha+3),\beta],&\qquad\kappa\in((\beta^{\#}-\alpha-2)/2,1/2),\\
			\eta\in\Bigl[\Bigl(\frac{\beta-\beta_{0}}{2}+\frac{1}{p}\Bigr)&\vee\Bigl(\frac{\beta^{\#}-\beta_{0}}{4}+\frac{1}{2p}\Bigr),\frac{\alpha+3}{2}\Bigr).
		\end{split}
	\end{equation*}
	\paragraph{Case 2.}
	Assume $\beta_{0}\in(\beta,\beta^{\#}]$. To ensure that the resulting interval for  $\eta\in[\frac{1}{p}\vee(\frac{\beta^{\#}-\beta_{0}}{4}+\frac{1}{2p}),\frac{\alpha+3}{2})$ is non-empty, we need to assume
	\begin{equation*}
		\frac{1}{p}<\frac{\alpha+3}{2}
	\end{equation*}
	and
	\begin{equation*}
		\frac{\beta^{\#}-\beta_{0}}{4}+\frac{1}{2p}<\frac{\alpha+3}{2}\qquad\iff\qquad\beta^{\#}+\frac{2}{p}-2(\alpha+3)<\beta_{0}.
	\end{equation*}
	Both are satisfied if we assume $p>2/(\alpha+3)$, since 
	\begin{equation*}
		\beta^{\#}<\alpha+\beta'+2\leq\alpha+\beta+3<\alpha+\beta+3+(\alpha+3)-\frac{2}{p}=\beta+2(\alpha+3)-\frac{2}{p},
	\end{equation*}
	which yields
	\begin{equation*}
		\beta^{\#}+\frac{2}{p}-2(\alpha+3)<\beta<\beta_{0}.
	\end{equation*}
	Hence in the case $\beta_{0}\in(\beta,\beta^{\#}]$, we can simplify our intervals to
	\begin{equation*}
		\begin{split}
			\alpha\in(-9/4,-2),&\qquad q\in[1,\infty],\\
			p\in(2/(\alpha+3),\infty],&\qquad \beta\in(-1/2,\alpha+2),\\
			\beta'\in(-2\alpha-4,(\beta+1)\wedge(2\alpha+5)],&\qquad\beta^{\#}\in(-\alpha-2,\alpha+\beta'+2),\\
			\beta_0\in(\beta,\beta^{\#}],&\qquad\kappa\in((\beta^{\#}-\alpha-2)/2,1/2),\\
			\eta\in\Bigl[\frac{1}{p}\vee\Bigl(\frac{\beta^{\#}-\beta_{0}}{4}&+\frac{1}{2p}\Bigr),\frac{\alpha+3}{2}\Bigr).
		\end{split}
	\end{equation*}
	Combining both cases, we obtain for all $\beta_0\in(\beta+2/p-(\alpha+3),\beta^{\#}]$,
	\begin{equation*}
		\begin{split}
			\alpha\in(-9/4,-2),&\qquad q\in[1,\infty],\\
			p\in(2/(\alpha+3),\infty],&\qquad\beta\in(-1/2,\alpha+2),\\
			\beta'\in(-2\alpha-4,(\beta+1)\wedge(2\alpha+5)],&\qquad\beta^{\#}\in(-\alpha-2,\alpha+\beta'+2),\\
			\beta_0\in(\beta+2/p-(\alpha+3),\beta^{\#}],&\qquad\kappa\in((\beta^{\#}-\alpha-2)/2,1/2),\\
			\eta\in\Bigl[\Bigl(\Bigl(\frac{\beta-\beta_{0}}{2}\vee0\Bigr)+\frac{1}{p}\Bigr)&\vee\Bigl(\frac{\beta^{\#}-\beta_{0}}{4}+\frac{1}{2p}\Bigr),\frac{\alpha+3}{2}\Bigr).
		\end{split}
	\end{equation*}
	The interval for $\kappa\in((\beta^{\#}-\alpha-2)/2,1/2)$ is non-empty, since $\beta^{\#}<\alpha+\beta+3$ and $\beta<0$ imply $\beta^{\#}-\alpha-2<1$.
	
	The interval for $\beta^{\#}\in(-\alpha-2,\alpha+\beta'+2)$ is non-empty, since $-2\alpha-4<\beta'$.
	
	The interval for $\beta'\in(-2\alpha-4,(\beta+1)\wedge(2\alpha+5)]$ is non-empty, since $-2\alpha-4<1/2<\beta+1$ and $-2\alpha-4<2\alpha+5$.
	
	The interval for $\beta\in(-1/2,\alpha+2)$ is non-empty, since $-1/2=-9/4+7/4<-9/4+2<\alpha+2$.
	
\end{details}
One can confirm that the intervals in~\eqref{eq:exponents} are non-empty. 

Particular choices of exponents allow us to consider regular and irregular initial data:
\begin{example}\label{ex:exponents_regular_initial_data}
	By taking $p=\infty$, $\beta_{0}=\beta^{\#}$ and $\eta=0$, we can choose as initial data any $\rho_{0}\in\mcB_{\infty,q}^{\beta^{\#}}(\mbT^{2})$, without incurring a blow-up at $0$.
\end{example}
\begin{example}\label{ex:exponents_Calpha+1_initial_data}
	By taking $p=q=\infty$ and $\beta<2\alpha+4$,
	\begin{details}
		which is possible, since
		\begin{equation*}
			-1/2<2\alpha+4\quad\iff\quad-9/4<\alpha,
		\end{equation*}
	\end{details}
	 we can choose as initial data any $\rho_{0}\in\mcC^{\alpha+1}(\mbT^{2})$, which is the natural regularity of $\ti$ and the state space of $\rho=\ti+\ty+w$. 
	 \begin{details}
	 	Indeed, $\beta_{0}=\alpha+1\in(\beta-(\alpha+3),\beta^{\#}]$, since
	 	\begin{equation*}
	 		\beta-(\alpha+3)<\alpha+1\qquad\iff\qquad\beta<2\alpha+4
	 	\end{equation*}
	 	and
	 	\begin{equation*}
	 		\alpha+1<-1<-\frac{1}{2}<\beta<\beta^{\#}.
	 	\end{equation*}
	 \end{details}
\end{example}
\begin{example}\label{ex:exponents_Lp_initial_data}
	Using that $L^{p}(\mbT^{2})\embed\mcB_{p,\infty}^{0}(\mbT^{2})$ (see Subsection~\ref{subsec:Besov_spaces}), we can choose as initial data any $\rho_{0}\in L^{p}(\mbT^{2})$ with $p>2$.
	\begin{details}
		Let $p>2$, $q=\infty$, $\eps<1/4$, $\alpha=-2-\eps$ and $\beta=-1/2+\eps$. Then for $\eps$ sufficiently small depending on $p$, we can ensure $p>2/(\alpha+3)$, since
		\begin{equation*}
			p>\frac{2}{\alpha+3}=\frac{2}{1-\eps}\quad\iff\quad\eps<1-\frac{2}{p}.
		\end{equation*}
		We can also ensure $\beta_{0}=0\in(\beta+2/p-(\alpha+3),\beta^{\#}]$, since
		\begin{equation*}
			\beta+\frac{2}{p}-(\alpha+3)=-\frac{1}{2}+\eps+\frac{2}{p}-(1-\eps)<0\quad\iff\quad2\eps<\frac{3}{2}-\frac{2}{p}
		\end{equation*}
		and
		\begin{equation*}
			\beta_{0}=0<-\alpha-2<\beta^{\#}.
		\end{equation*}
	\end{details}
\end{example}
Let us fix a $T>0$, we define the space of paracontrolled distributions.
\begin{definition}\label{def:paracontrolled_space}
	Let $\mbX\in\rksnoise{\alpha}{\kappa}_{T}$ and $\rho_0\in\mcB_{p,q}^{\beta_0}(\mbT^{2})$. We define the space
	\begin{equation*}
		\msD_{T}\subset\msL_{\eta;T}^{\kappa}\mcC^{\beta}(\mbT^2;\mbR)\times\msL_{\eta;T}^{\kappa}\mcC^{\beta'}(\mbT^2;\mbR^{2})\times(\msL_{\eta;T}^{\kappa}\mcC^{\beta}(\mbT^2;\mbR)\cap\msL_{2\eta;T}^{\kappa}\mcC^{\beta^{\#}}(\mbT^2;\mbR))
	\end{equation*}
	of distributions paracontrolled by $\mbX$ as those triples $\boldsymbol{w}\defeq(w,w', w^{\#})$ such that
	\begin{equation}\label{eq:paracontrolled_Ansatz}
		w=\vdiv\mcI[w'\pa \ti]+w^{\#}
	\end{equation}
	as well as $\lim_{t\to0}w_{t}=\lim_{t\to0}w_{t}^{\#}=\rho_{0}$ in $\mcS'(\mbT^{2};\mbR)$ and $\lim_{t\to0}w'_{t}=\nabla\Phi_{\rho_{0}}$ in $\mcS'(\mbT^{2};\mbR^{2})$. We equip this space with the metric induced by the norm
	\begin{equation*}
		\norm{\boldsymbol{w}}_{\msD_{T}}\defeq\max\{\norm{w}_{\msL_{\eta;T}^{\kappa}\mcC^{\beta}},\norm{w'}_{\msL_{\eta;T}^{\kappa}\mcC^{\beta'}},\norm{w^{\#}}_{\msL_{\eta;T}^{\kappa}\mcC^{\beta}},\norm{w^{\#}}_{\msL_{2\eta;T}^{\kappa}\mcC^{\beta^{\#}}}\}.
	\end{equation*}
\end{definition}
\begin{remark}\label{rem:Ansatz_discussion}
	The Ansatz~\eqref{eq:paracontrolled_Ansatz} allows us to write the mild equation for $w^{\#}$ as $w^{\#}=P\rho_{0}+\vdiv\mcI[\Omega^{\#}(\boldsymbol{w})]$ with some $\Omega^{\#}(\boldsymbol{w})$ determined by~\eqref{eq:mild_gen_rKS}. We can hence simplify our estimates by using the space-time regularization of $\mcI$. Note that~\eqref{eq:paracontrolled_Ansatz} is equivalent to the Ansatz~\eqref{eq:intro_paracontrolled_Ansatz} discussed in the introduction, up to commutators.
\end{remark}
\begin{lemma}\label{lem:parspace_non_trivial}
	Given $\mbX\in\rksnoise{\alpha}{\kappa}_{T}$ and $\rho_0\in\mcB_{p,q}^{\beta_0}(\mbT^{2})$, the space $\msD_T$ is a non-empty, complete metric space.
\end{lemma}
\begin{proof}
	To show that $\msD_T$ is non-empty, we can choose $w'=\nabla\Phi_{P\rho_0}$, $w^{\#}=P\rho_0$ and subsequently set $w\defeq\vdiv\mcI[w'\pa\ti]+w^{\#}$. The initial condition is satisfied, since $\lim_{t\to0}w'_{t}=\nabla\Phi_{\rho_0}$ in $\mcS'(\mbT^{2};\mbR^{2})$ and $\lim_{t\to0}w^{\#}_{t}=\rho_{0}$ in $\mcS'(\mbT^{2};\mbR)$. By Lemma~\ref{lem:Schauder} and Lemma~\ref{lem:elliptic_regularity}, using that $\beta'\leq\beta+1$, and $((\beta-\beta_{0})/2\vee0)+1/p\leq\eta$, we obtain $w'=\nabla\Phi_{P\rho_0}\in\msL_{\eta;T}^{\kappa}\mcC^{\beta'}(\mbT^{2};\mbR^{2})$, $w^{\#}=P\rho_0\in\msL_{\eta;T}^{\kappa}\mcC^{\beta}(\mbT^{2};\mbR)$ and
	\begin{equation*}
		\norm{\nabla\Phi_{P\rho_0}}_{\msL^{\kappa}_{\eta;T}\mcC^{\beta'}}\lesssim\norm{\nabla\Phi_{P\rho_0}}_{\msL^{\kappa}_{\eta;T}\mcC^{\beta+1}}\lesssim\norm{P\rho_0}_{\msL^{\kappa}_{\eta;T}\mcC^{\beta}}\lesssim_{T}\norm{\rho_0}_{\mcB_{p,q}^{\beta_{0}}}.
	\end{equation*}
	Using that $\beta_{0}\leq\beta^{\#}$ and $(\beta^{\#}-\beta_{0})/2+1/p\leq2\eta$, we obtain $w^{\#}\in\msL_{2\eta;T}^{\kappa}\mcC^{\beta^{\#}}(\mbT^{2};\mbR)$ and
	\begin{equation*}
		\norm{P\rho_0}_{\msL^{\kappa}_{2\eta;T}\mcC^{\beta^{\#}}}\lesssim_{T}\norm{\rho_0}_{\mcB_{p,q}^{\beta_{0}}}.
	\end{equation*}
	Since $w=\vdiv\mcI[w'\pa\ti]+w^{\#}$, we find by an application of the triangle inequality, Lemma~\ref{lem:Schauder} and Lemma~\ref{lem:Bony}, that $w\in\msL_{\eta;T}^{\kappa}\mcC^{\beta}(\mbT^{2};\mbR)$ and
	\begin{equation*}
		\norm{w}_{\msL^{\kappa}_{\eta;T}\mcC^{\beta}}\lesssim_{T}\norm{w'}_{C_{\eta;T}\mcC^{\beta'}}\norm{\ti}_{C_{T}\mcC^{\alpha+1}}+\norm{w^{\#}}_{\msL^{\kappa}_{\eta;T}\mcC^{\beta}},
	\end{equation*}
	where we used that $\beta<\alpha+2$ and $\beta'>0$.
	\begin{details}
		By the triangle inequality,
		\begin{equation*}
			\norm{w}_{\msL^{\kappa}_{\eta;T}\mcC^{\beta}}\leq\norm{\vdiv\mcI[w'\pa\ti]}_{\msL^{\kappa}_{\eta;T}\mcC^{\beta}}+\norm{w^{\#}}_{\msL^{\kappa}_{\eta;T}\mcC^{\beta}}
		\end{equation*}
		and by Lemma~\ref{lem:Schauder} and Lemma~\ref{lem:Bony}, using that $\beta<\alpha+2$ and $\beta'>0$,
		\begin{equation*}
			\norm{\vdiv\mcI[w'\pa\ti]}_{\msL^{\kappa}_{\eta;T}\mcC^{\beta}}\lesssim_{T}\norm{\vdiv(w'\pa\ti)}_{C_{\eta;T}\mcC^{\alpha}}\lesssim\norm{w'}_{C_{\eta;T}\mcC^{\beta'}}\norm{\ti}_{C_{T}\mcC^{\alpha+1}}.
		\end{equation*}
	\end{details}
	%
	%
	To show completeness, let $(\boldsymbol{w}_n)_{n\in\mbN}$ be a Cauchy sequence in $\msD_{T}$. By the completeness of $\msL_{\eta;T}^{\kappa}\mcC^{\beta}(\mbT^{2};\mbR)\times\msL_{\eta;T}^{\kappa}\mcC^{\beta'}(\mbT^{2};\mbR^{2})\times(\msL_{\eta;T}^{\kappa}\mcC^{\beta}(\mbT^{2};\mbR)\cap\msL_{2\eta;T}^{\kappa}\mcC^{\beta^{\#}}(\mbT^{2};\mbR))$, we obtain $w_n\to w\in\msL_{\eta;T}^{\kappa}\mcC^{\beta}(\mbT^{2};\mbR)$, $w'_{n}\to w'\in\msL_{\eta;T}^{\kappa}\mcC^{\beta'}(\mbT^{2};\mbR^{2})$, $w^{\#}_{n}\to w^{\#}\in\msL_{\eta;T}^{\kappa}\mcC^{\beta}(\mbT^{2};\mbR)$ and $w^{\#}_{n}\to w^{\#}\in\msL_{2\eta;T}^{\kappa}\mcC^{\beta^{\#}}(\mbT^{2};\mbR)$ all with the correct initial conditions. It suffices to show $w=\vdiv\mcI[w'\pa\ti]+w^{\#}$ so that $\boldsymbol{w}\in\msD_{T}$. Then again by Lemma~\ref{lem:Schauder} and Lemma~\ref{lem:Bony},
	%
	\begin{equation*}
		w_n-\vdiv\mcI[w'\pa\ti]-w^{\#}=\vdiv\mcI[(w'_{n}-w')\pa\ti]+w^{\#}_{n}-w^{\#}\to 0\qquad\text{in}\quad\msL_{\eta;T}^{\kappa}\mcC^{\beta}(\mbT^{2};\mbR),
	\end{equation*}
	\begin{details}
		\begin{equation*}
			\norm{\vdiv\mcI[(w'_{n}-w')\pa\ti]}_{\msL_{\eta;T}^{\kappa}\mcC^{\beta}}\lesssim_{T}\norm{w'_{n}-w'}_{C_{\eta;T}\mcC^{\beta'}}\norm{\ti}_{C_T\mcC^{\alpha+1}}\to0,
		\end{equation*}
	\end{details}
	which combined with $w_n\to w\in\msL_{\eta;T}^{\kappa}\mcC^{\beta}(\mbT^{2};\mbR)$ yields $w=\vdiv\mcI[w'\pa\ti]+w^{\#}$.
\end{proof}
In the next lemma we show that $w\re\nabla\Phi_{\ti}+\ti\re\nabla\Phi_{w}$ is well-defined in $\msD_{T}\times\rksnoise{\alpha}{\kappa}_{T}$.
\begin{lemma}\label{lem:renormalised_products_estimates}
	There exists a continuous operator $\msP\from\msD_T\times\rksnoise{\alpha}{\kappa}_{T}\to C_{2\eta;T}\mcC^{2\alpha+4}(\mbT^{2};\mbR^{2})$, such that when all objects are smooth,
	\begin{equation*}
		\msP(\boldsymbol{w},\mbX)=w\re\nabla\Phi_{\ti}+\ti\re\nabla\Phi_{w}.
	\end{equation*} 
\end{lemma}
\begin{proof}
	Using the notation $\msC(f,g,h)=(f\pa g)\re h-f(g\re h)$ (see Lemma~\ref{lem:commutator_paraproduct_resonant}) and recalling that $\tc=\nabla\mcI[\ti]\re\nabla\Phi_{\ti}+\nabla^{2}\mcI[\Phi_{\ti}]\re\ti$ we can expand the product into
	\begin{equation*}
		\begin{split}
			\msP(\boldsymbol{w},\mbX)&\defeq\msC(w',\nabla\mcI[\ti],\nabla\Phi_{\ti})+\msC(w',\nabla^{2}\mcI[\Phi_{\ti}],\ti)+(w^{\#}+\vdiv\mcI[w'\pa\ti]-w'\pa\nabla\mcI[\ti])\re\nabla\Phi_{\ti}\\
			&\quad+(\nabla\Phi_{w^{\#}}+\nabla\vdiv\Phi_{\mcI[w'\pa\ti]}-w'\pa\nabla^{2}\mcI[\Phi_{\ti}])\re\ti+w'\tc,
		\end{split}
	\end{equation*}
	where $\boldsymbol{w}\in\msD_{T}$ and $\mbX\in\rksnoise{\alpha}{\kappa}_{T}$. In what follows, we establish the bound
	\begin{equation}\label{eq:renormalised_products_estimates_bound_3}
		\norm{\msP(\boldsymbol{w},\mbX)}_{C_{2\eta;T}\mcC^{2\alpha+4}}\lesssim_{T}(\norm{w'}_{\msL_{\eta;T}^{\kappa}\mcC^{\beta'}}+\norm{w^{\#}}_{C_{2\eta;T}\mcC^{\beta^{\#}}})(1+\norm{\ti}_{C_T\mcC^{\alpha+1}}+\norm{\tc}_{C_T\mcC^{2\alpha+4}})^{2}
	\end{equation}
	by various applications of our commutator results. The regularity $\msP(\boldsymbol{w},\mbX)\in C_{2\eta;T}\mcC^{2\alpha+4}(\mbT^{2};\mbR^{2})$ follows by the same arguments.
	
	Using Lemma~\ref{lem:commutator_paraproduct_resonant} and that $\beta'\in(0,1)$, $2\alpha+4<0<2\alpha+4+\beta'$, we obtain
	\begin{equation*}
		\norm{\msC(w',\nabla\mcI[\ti],\nabla\Phi_{\ti})}_{C_{2\eta;T}\mcC^{2\alpha+4}}\lesssim_{T}\norm{w'}_{C_{\eta;T}\mcC^{\beta'}}\norm{\nabla\mcI[\ti]}_{C_T\mcC^{\alpha+2}}\norm{\nabla\Phi_{\ti}}_{C_T\mcC^{\alpha+2}}
	\end{equation*}
	and
	\begin{equation*}
		\norm{\msC(w',\nabla^{2}\mcI[\Phi_{\ti}],\ti)}_{C_{2\eta;T}\mcC^{2\alpha+4}}\lesssim_{T}\norm{w'}_{C_{\eta;T}\mcC^{\beta'}}\norm{\nabla^{2}\mcI[\Phi_{\ti}]}_{C_T\mcC^{\alpha+3}}\norm{\ti}_{C_T\mcC^{\alpha+1}}.
	\end{equation*}
	
	Further, by Lemma~\ref{lem:Bony}, using that $2\alpha+4<0<\beta^{\#}+\alpha+2$,
	\begin{equation*}
		\norm{(w^{\#}+\vdiv\mcI[w'\pa\ti]-w'\pa\nabla\mcI[\ti])\re\nabla\Phi_{\ti}}_{C_{2\eta;T}\mcC^{2\alpha+4}}\lesssim\norm{w^{\#}+\vdiv\mcI[w'\pa\ti]-w'\pa\nabla\mcI[\ti]}_{C_{2\eta;T}\mcC^{\beta^{\#}}}\norm{\nabla\Phi_{\ti}}_{C_T\mcC^{\alpha+2}}.
	\end{equation*}
	To control the remainder, we apply Lemma~\ref{lem:commutator_Fourier_multiplier_paraproduct} and Lemma~\ref{lem:commutator_heat_paraproduct}, using that $\kappa\in((\beta^{\#}-\alpha-2)/2,1/2)$ and $\beta^{\#}<\alpha+\beta'+2$,
	\begin{equation*}
		\norm{w^{\#}+\vdiv\mcI[w'\pa\ti]-w'\pa\nabla\mcI[\ti]}_{C_{2\eta;T}\mcC^{\beta^{\#}}}\lesssim_{T}\norm{w^{\#}}_{C_{2\eta;T}\mcC^{\beta^{\#}}}+\norm{w'}_{\msL_{\eta;T}^{\kappa}\mcC^{\beta'}}\norm{\ti}_{C_T\mcC^{\alpha+1}}.
	\end{equation*}
	\begin{details}
		The remainder $w^{\#}+\vdiv\mcI[w'\pa\ti]-w'\pa\nabla\mcI[\ti]$ can be decomposed into
		\begin{equation*}
			w^{\#}+\vdiv\mcI[w'\pa\ti]-w'\pa\nabla\mcI[\ti]=w^{\#}+\mcI[\vdiv (w'\pa \ti)-w'\pa \nabla\ti]+\mcI[w'\pa \nabla\ti]-w'\pa \mcI[\nabla\ti].
		\end{equation*}
		We estimate the first commutator by Lemma~\ref{lem:Schauder} and Lemma~\ref{lem:commutator_Fourier_multiplier_paraproduct}, using that $\beta^{\#}<\alpha+\beta'+2$, $\beta'<1$,
		\begin{equation*}
			\norm{\mcI[\vdiv(w'\pa\ti)-w'\pa\nabla\ti]}_{\msL_{2\eta;T}^{\kappa}\mcC^{\beta^{\#}}}\lesssim_{T}\norm{\vdiv(w'\pa\ti)-w'\pa\nabla\ti}_{C_{\eta;T}\mcC^{\alpha+\beta'}}\lesssim\norm{w'}_{C_{\eta;T}\mcC^{\beta'}}\norm{\ti}_{C_T\mcC^{\alpha+1}}.
		\end{equation*}
		For the second commutator we apply Lemma~\ref{lem:commutator_heat_paraproduct}, using that $\kappa\in(0,1/2)$ and that there exists some $\gamma<2\kappa\wedge\beta'$ such that $\beta^{\#}\in(\gamma+\alpha,\gamma+\alpha+2)$, to estimate
		\begin{equation*}
			\norm{\mcI[w'\pa\nabla\ti]-w'\pa\mcI[\nabla\ti]}_{C_{2\eta;T}\mcC^{\beta^{\#}}}\lesssim_{T}\norm{w'}_{\msL_{\eta;T}^{\kappa}\mcC^{\gamma}}\norm{\nabla\ti}_{C_T\mcC^{\alpha}}\lesssim\norm{w'}_{\msL_{\eta;T}^{\kappa}\mcC^{\beta'}}\norm{\nabla\ti}_{C_T\mcC^{\alpha}}.
		\end{equation*}
		Indeed, if we choose $\gamma>\beta^{\#}-\alpha-2$, the upper bound $\beta^{\#}<\gamma+\alpha+2$ becomes immediate. Choosing $\gamma$ sufficiently close to $\beta^{\#}-\alpha-2$ provides the lower bound $\gamma+\alpha<\beta^{\#}$. We only need to make sure that this choice is compatible with the assumption $\gamma<2\kappa\wedge\beta'$. However, the intervals $(\beta^{\#}-\alpha-2,2\kappa)$ and $(\beta^{\#}-\alpha-2,\beta')$ are non-empty so that we can find a suitable $\gamma$.
		
		This yields
		\begin{equation*}
			\norm{w^{\#}+\vdiv\mcI[w'\pa\ti]-w'\pa\nabla\mcI[\ti]}_{C_{2\eta;T}\mcC^{\beta^{\#}}}\lesssim_{T}\norm{w'}_{\msL_{\eta;T}^{\kappa}\mcC^{\beta'}}\norm{\ti}_{C_T\mcC^{\alpha+1}}+\norm{w^{\#}}_{C_{2\eta;T}\mcC^{\beta^{\#}}}.
		\end{equation*}
	\end{details}
	Similarly, by Lemma~\ref{lem:Bony}, Lemma~\ref{lem:commutator_Fourier_multiplier_paraproduct} and Lemma~\ref{lem:commutator_heat_paraproduct}, using that $2\alpha+4<0<\beta^{\#}+\alpha+2$,
	\begin{equation*}
		\begin{split}
			&\norm{(\nabla\Phi_{w^{\#}}+\nabla\vdiv\Phi_{\mcI[w'\pa\ti]}-w'\pa\nabla^{2}\mcI[\Phi_{\ti}])\re\ti}_{C_{2\eta;T}\mcC^{2\alpha+4}}\\
			&\lesssim_{T}(\norm{w^{\#}}_{C_{2\eta;T}\mcC^{\beta^{\#}}}+\norm{w'}_{\msL_{\eta;T}^{\kappa}\mcC^{\beta'}}\norm{\ti}_{C_T\mcC^{\alpha+1}})\norm{\ti}_{C_T\mcC^{\alpha+1}}.
		\end{split}
	\end{equation*}
	\begin{details}
		The remainder $\nabla\Phi_{w^{\#}}+\nabla\vdiv\Phi_{\mcI[w'\pa\ti]}-w'\pa\nabla^{2}\mcI[\Phi_{\ti}]$ can be decomposed into
		\begin{equation}\label{eq:Phi_tilde_sharp_w}
			\begin{split}
				&\nabla\Phi_{w^{\#}}+\nabla\vdiv\Phi_{\mcI[w'\pa\ti]}-w'\pa\nabla^{2}\mcI[\Phi_{\ti}]\\
				&=\mcI[\nabla\vdiv (w'\pa\Phi_{\ti})-w'\pa\nabla^2\Phi_{\ti}]+\mcI[w'\pa\nabla^2\Phi_{\ti}]-w'\pa\mcI[\nabla^2\Phi_{\ti}]+\nabla\Phi^{\#}_{w},
			\end{split}
		\end{equation}
		where
		\begin{equation}\label{eq:Phi_sharp_w}
			\Phi^{\#}_{w}\defeq\vdiv\mcI[\Phi(D)(w'\pa\ti)-w'\pa\Phi_{\ti}]+\Phi_{w^{\#}}.
		\end{equation}
		We estimate the first commutator in~\eqref{eq:Phi_tilde_sharp_w} by using Lemma~\ref{lem:Schauder} and Lemma~\ref{lem:commutator_Fourier_multiplier_paraproduct}, and that $\beta^{\#}<\alpha+\beta'+2$, $\beta'<1$,
		\begin{equation*}
			\begin{split}
				\norm{\mcI[\nabla\vdiv(w'\pa\Phi_{\ti})-w'\pa\nabla^{2}\Phi_{\ti}]}_{\msL_{2\eta;T}^{\kappa}\mcC^{\beta^{\#}+1}}&\lesssim_{T}\norm{\nabla\vdiv(w'\pa\Phi_{\ti})-w'\pa\nabla^{2}\Phi_{\ti}}_{C_{\eta;T}\mcC^{\alpha+\beta'+1}}\\
				&\lesssim\norm{w'}_{C_{\eta;T}\mcC^{\beta'}}\norm{\Phi_{\ti}}_{C_T\mcC^{\alpha+3}}.
			\end{split}
		\end{equation*}
		For the second commutator in~\eqref{eq:Phi_tilde_sharp_w} we apply Lemma~\ref{lem:commutator_heat_paraproduct}, using that $\kappa\in(0,1/2)$ and that there exists some $\gamma<2\kappa\wedge\beta'$ such that $\beta^{\#}\in(\gamma+\alpha,\gamma+\alpha+2)$ to estimate
		\begin{equation*}
			\norm{\mcI[w'\pa\nabla^2\Phi_{\ti}]-w'\pa\mcI[\nabla^2\Phi_{\ti}]}_{C_{2\eta;T}\mcC^{\beta^{\#}+1}}\lesssim_{T}\norm{w'}_{\msL_{\eta;T}^{\kappa}\mcC^{\gamma}}\norm{\nabla^2\Phi_{\ti}}_{C_T\mcC^{\alpha+1}}\lesssim\norm{w'}_{\msL_{\eta;T}^{\kappa}\mcC^{\beta'}}\norm{\nabla^2\Phi_{\ti}}_{C_T\mcC^{\alpha+1}}.
		\end{equation*}
		Finally, we estimate with Lemma~\ref{lem:Fourier_multiplier},
		\begin{equation*}
			\norm{\nabla\Phi^{\#}_{w}}_{C_{2\eta;T}\mcC^{\beta^{\#}+1}}\lesssim\norm{\Phi^{\#}_{w}}_{C_{2\eta;T}\mcC^{\beta^{\#}+2}},
		\end{equation*}
		which shows that~\eqref{eq:Phi_tilde_sharp_w} can be bounded by
		\begin{equation*}
			\norm{\nabla\Phi_{w^{\#}}+\nabla\vdiv\Phi_{\mcI[w'\pa\ti]}-w'\pa\nabla^{2}\mcI[\Phi_{\ti}]}_{C_{2\eta;T}\mcC^{\beta^{\#}+1}}\lesssim_{T}\norm{w'}_{\msL_{\eta;T}^{\kappa}\mcC^{\beta'}}\norm{\ti}_{C_T\mcC^{\alpha+1}}+\norm{\Phi^{\#}_{w}}_{C_{2\eta;T}\mcC^{\beta^{\#}+2}}.
		\end{equation*}
		We control $\Phi^{\#}_{w}$ defined by~\eqref{eq:Phi_sharp_w}. We bound the commutator in~\eqref{eq:Phi_sharp_w} with Schauder's estimate, Lemma~\ref{lem:Schauder}, and Lemma~\ref{lem:commutator_Fourier_multiplier_paraproduct}, using that $\beta^{\#}<\alpha+\beta'+2$ and $\beta'<1$,
		\begin{equation*}
			\begin{split}
				\norm{\vdiv\mcI[\Phi(D)(w'\pa\ti)-w'\pa\Phi_{\ti}]}_{\msL_{2\eta;T}^{\kappa}\mcC^{\beta^{\#}+2}}&\lesssim_{T}\norm{\vdiv(\Phi(D)(w'\pa\ti)-w'\pa\Phi_{\ti})}_{C_{\eta;T}\mcC^{\alpha+\beta'+2}}\\
				&\lesssim\norm{\Phi(D)(w'\pa \ti)-w'\pa \Phi_{\ti}}_{C_{\eta;T}\mcC^{\alpha+\beta'+3}}\\
				&\lesssim\norm{w'}_{C_{\eta;T}\mcC^{\beta'}}\norm{\ti}_{C_{T}\mcC^{\alpha+1}}.
			\end{split}
		\end{equation*}
		The remaining term can be bounded by Lemma~\ref{lem:Fourier_multiplier},
		\begin{equation*}
			\norm{\Phi_{w^{\#}}}_{C_{2\eta;T}\mcC^{\beta^{\#}+2}}\lesssim\norm{w^{\#}}_{C_{2\eta;T}\mcC^{\beta^{\#}}}.
		\end{equation*}
		This yields the estimate
		\begin{equation*}
			\norm{\nabla\Phi_{w^{\#}}+\nabla\vdiv\Phi_{\mcI[w'\pa\ti]}-w'\pa\nabla^{2}\mcI[\Phi_{\ti}]}_{C_{2\eta;T}\mcC^{\beta^{\#}+1}}\lesssim_{T}\norm{w'}_{\msL_{\eta;T}^{\kappa}\mcC^{\beta'}}\norm{\ti}_{C_T\mcC^{\alpha+1}}+\norm{w^{\#}}_{C_{2\eta;T}\mcC^{\beta^{\#}}}.
		\end{equation*}
	\end{details}
	
	Finally, by Lemma~\ref{lem:Bony}, using that $2\alpha+4<0<2\alpha+4+\beta'$,
	\begin{equation*}
		\norm{w'\tc}_{C_{2\eta;T}\mcC^{2\alpha+4}}\lesssim_{T}\norm{w'}_{C_{\eta;T}\mcC^{\beta'}}\norm{\tc}_{C_T\mcC^{2\alpha+4}}.
	\end{equation*}
	This yields the claim.
\end{proof}
We can now derive a priori bounds for our solution map.
\begin{lemma}\label{lem:bounds_apriori}
	For every $\mbX\in\rksnoise{\alpha}{\kappa}_{T}$ and $\rho_0\in\mcB_{p,q}^{\beta_{0}}(\mbT^{2})$, let $\boldsymbol{\Psi}$, acting on $\boldsymbol{u}=(u,u',u^{\#})\in\msD_T$, be given by $\boldsymbol{\Psi}(\boldsymbol{u})\defeq(w,w',w^{\#})$, where
	\begin{equation*}
		\begin{cases}
			\begin{aligned}
				w&\defeq\vdiv\mcI[w'\pa \ti]+w^{\#},\quad w'\defeq\nabla\Phi_ u+\nabla\Phi_{\ty},\\
				w^{\#}&\defeq P\rho_0+\vdiv\mcI[\Omega^{\#}(\boldsymbol{u})],
			\end{aligned}
		\end{cases}
	\end{equation*}
	and
	\begin{equation*}
		\begin{split}
			\Omega^{\#}(\boldsymbol{u})&\defeq u\nabla\Phi_u+u\nabla\Phi_{\ty}+\ty\nabla\Phi_u+\ty \nabla\Phi_{\ty}+\tp+\nabla \Phi_{\ti}\pa\ty+\ty\pa\nabla\Phi_{\ti}\\
			&\quad+\ti\pa\nabla \Phi_{\ty}+ u\pa\nabla\Phi_{\ti}+\nabla\Phi_{\ti}\pa u+\ti\pa\nabla \Phi_{u}+\msP(\boldsymbol{u},\mbX).
		\end{split}
	\end{equation*}
	Then $\boldsymbol{\Psi}(\boldsymbol{u})\in\msD_{T}$ and there exists some $\theta>0$ depending only on the chosen exponents and the dimension, such that for every $T\leq1$,
	\begin{align}
		\max\{\norm{w}_{\msL_{\eta;T}^{\kappa}\mcC^{\beta}},\norm{w^{\#}}_{\msL_{\eta;T}^{\kappa}\mcC^{\beta}},\norm{w^{\#}}_{\msL_{2\eta;T}^{\kappa}\mcC^{\beta^{\#}}}\}&\lesssim (1+T^{\theta}\norm{\boldsymbol{u}}_{\msD_T})^{2}(1+\norm{\mbX}_{\rksnoise{\alpha}{\kappa}_{T}}+\norm{\rho_0}_{\mcB_{p,q}^{\beta_0}})^{2},\label{eq:Psi_bounds_w_w_sharp}\\
		\norm{w'}_{\msL^{\kappa}_{\eta;T}\mcC^{\beta'}}&\lesssim\norm{u}_{\msL_{\eta;T}^{\kappa}\mcC^{\beta}}+\norm{\mbX}_{\rksnoise{\alpha}{\kappa}_{T}}.\label{eq:Psi_bounds_w_prime}
	\end{align}
\end{lemma}
\begin{proof}
	We derive bounds for our solution map in several steps. By the same arguments it also follows that $\boldsymbol{\Psi}(\boldsymbol{u})\in\msD_{T}$. In the remainder of the proof, we will assume $T\leq1$.
	
	\emph{Step $1$. The $\msL_{\eta;T}^{\kappa}\mcC^{\beta}(\mbT^{2};\mbR)$-regularity of $w$.} As in the proof of Lemma~\ref{lem:parspace_non_trivial}, but this time keeping track of the dependency on $T$, we see that
	\begin{equation*}
		\norm{w}_{\msL_{\eta;T}^{\kappa}\mcC^{\beta}}\lesssim(T^{1-\frac{\beta-\alpha}{2}}\vee T^{1-\kappa})\norm{w'}_{C_{\eta;T}\mcC^{\beta'}}\norm{\ti}_{C_T\mcC^{\alpha+1}}+\norm{w^{\#}}_{\msL_{\eta;T}^{\kappa}\mcC^{\beta}}.
	\end{equation*}
	\emph{Step $2$. The $\msL_{\eta;T}^{\kappa}\mcC^{\beta}(\mbT^{2};\mbR)$ and $\msL_{2\eta;T}^{\kappa}\mcC^{\beta^{\#}}(\mbT^{2};\mbR)$-regularity of $w^{\#}$.} To establish the $\msL_{\eta;T}^{\kappa}\mcC^{\beta}(\mbT^{2};\mbR)$-regularity of $w^{\#}$, we apply Lemma~\ref{lem:Schauder}, using that $((\beta-\beta_{0})/2\vee0)+1/p\leq\eta$, $\eta<1/2$, $\beta+1<\alpha+\beta+4$ and $-(\alpha+1)/2\vee\kappa<1-\eta$,
	\begin{equation*}
		\norm{w^{\#}}_{\msL_{\eta;T}^{\kappa}\mcC^{\beta}}\lesssim T^{\eta-(\frac{\beta-\beta_0}{2}\vee0)-\frac{1}{p}}\norm{\rho_0}_{\mcB_{p,q}^{\beta_0}}+(T^{\frac{\alpha+3}{2}-\eta}\vee T^{1-\kappa-\eta})\norm{\Omega^{\#}(\boldsymbol{u})}_{C_{2\eta;T}\mcC^{\alpha+\beta+2}}.
	\end{equation*}	
	To establish the $\msL_{2\eta;T}^{\kappa}\mcC^{\beta^{\#}}(\mbT^{2};\mbR)$-regularity of $w^{\#}$, we apply Lemma~\ref{lem:Schauder} as above, but this time using that $\beta_{0}\leq\beta^{\#}$, $(\beta^{\#}-\beta_{0})/2+1/p\leq2\eta$, $\eta<1/2$, $\beta^{\#}+1<\alpha+\beta+4$ and $(\beta^{\#}-\alpha-\beta-1)/2\vee\kappa<1$,
	\begin{equation*}
		\norm{w^{\#}}_{\msL_{2\eta;T}^{\kappa}\mcC^{\beta^{\#}}}\lesssim T^{2\eta-\frac{\beta^{\#}-\beta_0}{2}-\frac{1}{p}}\norm{\rho_0}_{\mcB_{p,q}^{\beta_0}}+(T^{1-\frac{\beta^{\#}-\alpha-\beta-1}{2}}\vee T^{1-\kappa})\norm{\Omega^{\#}(\boldsymbol{u})}_{C_{2\eta;T}\mcC^{\alpha+\beta+2}}.
	\end{equation*}	
	\emph{Step $3$. The $C_{2\eta;T}\mcC^{\alpha+\beta+2}(\mbT^{2};\mbR^{2})$-regularity of $\Omega^{\#}(\boldsymbol{u})$.} We obtain by various applications of Lemma~\ref{lem:Bony}, using in particular that $4\alpha+9>0$ and $\beta>-1/2$,
	\begin{equation*}
		\norm{u\nabla\Phi_ u}_{C_{2\eta;T}\mcC^{\alpha+\beta+2}}\lesssim\norm{u}_{C_{\eta;T}\mcC^{\beta}}^{2},
	\end{equation*}
	\begin{equation*}
		\begin{split}
			\max\Bigl\{\norm{u\nabla\Phi_{\ty}}_{C_{2\eta;T}\mcC^{\alpha+\beta+2}},&\norm{\ty\nabla\Phi_u }_{C_{2\eta;T}\mcC^{\alpha+\beta+2}},\norm{u\pa\nabla\Phi_{\ti}}_{C_{2\eta;T}\mcC^{\alpha+\beta+2}},\\
			&\norm{\nabla\Phi_{\ti}\pa u}_{C_{2\eta;T}\mcC^{\alpha+\beta+2}},\norm{\ti\pa\nabla\Phi_u}_{C_{2\eta;T}\mcC^{\alpha+\beta+2}}\Bigr\}\lesssim T^{\eta} \norm{u}_{C_{\eta;T}\mcC^{\beta}}\norm{\mbX}_{\rksnoise{\alpha}{\kappa}_{T}},
		\end{split}
	\end{equation*}
	and
	\begin{equation*}
		\begin{split}
			\max\Bigl\{\norm{{\ty}\nabla\Phi_{\ty}}_{C_{2\eta;T}\mcC^{\alpha+\beta+2}},&\norm{\tp}_{C_{2\eta;T}\mcC^{\alpha+\beta+2}},\norm{\nabla\Phi_{\ti}\pa{\ty}}_{C_{2\eta;T}\mcC^{\alpha+\beta+2}},\\
			&\norm{{\ty}\pa\nabla\Phi_{\ti}}_{C_{2\eta;T}\mcC^{\alpha+\beta+2}},\norm{{\ti}\pa\nabla\Phi_{\ty}}_{C_{2\eta;T}\mcC^{\alpha+\beta+2}}\Bigr\}\lesssim T^{2\eta}(1+\norm{\mbX}_{\rksnoise{\alpha}{\kappa}_T})^{2}.
		\end{split}
	\end{equation*}
	\begin{details}
		Using that $\alpha+\beta+2\leq\beta$, $0<2\beta+1$
		\begin{equation*}
			\norm{u\nabla\Phi_ u}_{C_{2\eta;T}\mcC^{\alpha+\beta+2}}\lesssim\norm{u}_{C_{\eta;T}\mcC^{\beta}}\norm{\nabla\Phi_{u}}_{C_{\eta;T}\mcC^{\beta+1}}.
		\end{equation*}
		Using that $\alpha+\beta+2\leq\beta$, $2\alpha+5>0$ and $\beta+2\alpha+5>0$,
		\begin{equation*}
			\norm{u\nabla\Phi_{\ty}}_{C_{2\eta;T}\mcC^{\alpha+\beta+2}}\lesssim T^{\eta} \norm{u}_{C_{\eta;T}\mcC^{\beta}}\norm{\nabla\Phi_{\ty}}_{C_{T}\mcC^{2\alpha+5}}.
		\end{equation*}
		Using that $\alpha+\beta+2\leq2\alpha+4$, $\beta+1>0$ and $\beta+2\alpha+5>0$,
		\begin{equation*}
			\norm{\ty\nabla\Phi_u }_{C_{2\eta;T}\mcC^{\alpha+\beta+2}}\lesssim T^{\eta}\norm{\ty}_{C_{T}\mcC^{2\alpha+4}}\norm{\nabla\Phi_{u}}_{C_{\eta;T}\mcC^{\beta+1}}.
		\end{equation*}
		Using that $\alpha+\beta+2\leq2\alpha+4$ and $4\alpha+9>0$,
		\begin{equation*}
			\norm{\ty\nabla\Phi_{\ty}}_{C_{2\eta;T}\mcC^{\alpha+\beta+2}}\lesssim T^{2\eta}\norm{\ty}_{C_T\mcC^{2\alpha+4}}\norm{\nabla\Phi_{\ty}}_{C_T\mcC^{2\alpha+5}}. 
		\end{equation*}
		Using that $\alpha+\beta+2\leq3\alpha+6$,
		\begin{equation*}
			\norm{\tp}_{C_{2\eta;T}\mcC^{\alpha+\beta+2}}\lesssim T^{2\eta}\norm{\tp}_{C_{T}\mcC^{3\alpha+6}}.
		\end{equation*}
		Using that $\alpha+\beta+2\leq3\alpha+6$, $\alpha+2<0$,
		\begin{equation*}
			\norm{\nabla\Phi_{\ti}\pa{\ty}}_{C_{2\eta;T}\mcC^{\alpha+\beta+2}}\lesssim T^{2\eta}\norm{\nabla\Phi_{\ti}}_{C_T\mcC^{\alpha+2}}\norm{\ty}_{C_T\mcC^{2\alpha+4}}.
		\end{equation*}
		Using that $\alpha+\beta+2\leq3\alpha+6$, $2\alpha+4<0$,
		\begin{equation*}
			\norm{{\ty}\pa\nabla\Phi_{\ti}}_{C_{2\eta;T}\mcC^{\alpha+\beta+2}}\lesssim T^{2\eta}\norm{\ty}_{C_T\mcC^{2\alpha+4}}\norm{\nabla\Phi_{\ti}}_{C_T\mcC^{\alpha+2}}.
		\end{equation*}
		Using that $\alpha+\beta+2\leq3\alpha+6$, $\alpha+1<0$,
		\begin{equation*}
			\norm{{\ti}\pa\nabla\Phi_{\ty}}_{C_{2\eta;T}\mcC^{\alpha+\beta+2}}\lesssim T^{2\eta}\norm{\ti}_{C_T\mcC^{\alpha+1}}\norm{\nabla\Phi_{\ty}}_{C_T\mcC^{2\alpha+5}}.
		\end{equation*}
		Using that $\beta<0$,
		\begin{equation*}
			\norm{u\pa\nabla\Phi_{\ti}}_{C_{2\eta;T}\mcC^{\alpha+\beta+2}}\lesssim T^{\eta}\norm{u}_{C_{\eta;T}\mcC^{\beta}}\norm{\nabla\Phi_{\ti}}_{C_T\mcC^{\alpha+2}}.
		\end{equation*}
		Similarly using that $\alpha+2<0$,
		\begin{equation*}
			\norm{\nabla\Phi_{\ti}\pa u}_{C_{2\eta;T}\mcC^{\alpha+\beta+2}}\lesssim T^{\eta}\norm{\nabla\Phi_{\ti}}_{C_T\mcC^{\alpha+2}}\norm{u}_{C_{\eta;T}\mcC^{\beta}}.
		\end{equation*}
		Using that $\alpha+1<0$,
		\begin{equation*}
			\norm{\ti\pa\nabla\Phi_u}_{C_{2\eta;T}\mcC^{\alpha+\beta+2}}\lesssim T^{\eta}\norm{\ti}_{C_{T}\mcC^{\alpha+1}}\norm{\nabla\Phi_{u}}_{C_{\eta;T}\mcC^{\beta+1}}.
		\end{equation*}
	\end{details}
	By~\eqref{eq:renormalised_products_estimates_bound_3} of Lemma~\ref{lem:renormalised_products_estimates}, using that $\alpha+\beta+2\leq2\alpha+4$,
	\begin{equation*}
		\norm{\msP(\boldsymbol{u},\mbX)}_{C_{2\eta;T}\mcC^{\alpha+\beta+2}}\lesssim(\norm{u'}_{\msL_{\eta;T}^{\kappa}\mcC^{\beta'}}+\norm{u^{\#}}_{C_{2\eta;T}\mcC^{\beta^{\#}}})(1+\norm{\ti}_{C_T\mcC^{\alpha+1}}+\norm{\tc}_{C_T\mcC^{2\alpha+4}})^{2}.
	\end{equation*}
	\emph{Step $4$. The $\msL_{\eta;T}^{\kappa}\mcC^{\beta'}(\mbT^{2};\mbR^{2})$-regularity of $w'$.} By definition, $w'=\nabla\Phi_u+\nabla\Phi_{\ty}$. Using that $\beta'\leq\beta+1$ and $\beta'\leq2\alpha+5$, we obtain
	\begin{equation*}
		\norm{w'}_{\msL_{\eta;T}^{\kappa}\mcC^{\beta'}}\lesssim\norm{u}_{\msL_{\eta;T}^{\kappa}\mcC^{\beta}}+T^{\eta}\norm{\ty}_{\msL_{T}^{\kappa}\mcC^{2\alpha+4}}.
	\end{equation*}
	\emph{Step $5$. Closing the bounds.} Using that $T\leq1$, we can collect all of the terms above and cast them in the form~\eqref{eq:Psi_bounds_w_w_sharp}--\eqref{eq:Psi_bounds_w_prime}. This yields the claim.
\end{proof}
\begin{details}
	In order to construct a solution, we apply Banach's fixed-point theorem.
	\begin{theorem}[Banach's fixed-point theorem]
		Let $(S,d)$ be a non-empty, complete metric space. Let $\Psi\from S\to S$ be a contraction, that is there exists some $c\in(0,1)$ such that for any $x,y\in S$,
		\begin{align*}
			d(\Psi(x),\Psi(y))\leq cd(x,y).
		\end{align*}
		Then there exists a unique fixed point of $\Psi$ in $S$.
	\end{theorem}
\end{details}
While Lemma~\ref{lem:bounds_apriori} shows that $\boldsymbol{\Psi}$ is a map from $\msD_T$ to itself, it is not a contraction for small $T$, since there is no small time parameter on the right-hand side of~\eqref{eq:Psi_bounds_w_prime}. The remedy is to apply $\boldsymbol{\Psi}$ twice and argue that a fixed point of $\boldsymbol{\Psi}^{\circ 2}$ is also a fixed point of $\boldsymbol{\Psi}$ itself.
\begin{proposition}\label{prop:w_local_existence}
	Let $\mbX\in\rksnoise{\alpha}{\kappa}_{T}$ and $\rho_0\in\mcB_{p,q}^{\beta_{0}}(\mbT^{2})$. Then there exists some $\bar{T}\in (0,T\wedge1]$ such that there is a unique solution $\boldsymbol{w}=(w,w',w^{\#})\in\msD_{\bar{T}}$ to the equation
	\begin{equation}\label{eq:paracontrolled_equation}
		\begin{cases}
			\begin{aligned}
				w&=\vdiv\mcI[w'\pa \ti]+w^{\#},\quad w'=\nabla\Phi_{w}+\nabla\Phi_{\ty},\\
				w^{\#}&=P\rho_0+\vdiv\mcI[\Omega^{\#}(\boldsymbol{w})].
			\end{aligned}
		\end{cases}
	\end{equation}
	The time of existence $\bar{T}$ depends only on $\norm{\mbX}_{\rksnoise{\alpha}{\kappa}_{T}}$, $\norm{\rho_{0}}_{\mcB_{p,q}^{\beta_{0}}}$, the chosen exponents and the dimension.
\end{proposition}
\begin{proof}
	We let $\bar{T}\in(0,T]$, $\bar{T}\leq1$, $\mbX \in \rksnoise{\alpha}{\kappa}_{T}$, $\boldsymbol{u}\in\msD_{\bar{T}}$ and define
	\begin{equation*}
		(\Psi(\boldsymbol{u}),\Psi(\boldsymbol{u})',\Psi(\boldsymbol{u})^{\#})\defeq\boldsymbol{\Psi}(\boldsymbol{u}).
	\end{equation*}
	By Lemma~\ref{lem:bounds_apriori} there exists some $\theta>0$ such that
	\begin{equation}\label{eq:Psi_bound_1}
		\max\{\norm{\Psi(\boldsymbol{u})}_{\msL_{\eta;\bar{T}}^{\kappa}\mcC^{\beta}},\norm{\Psi(\boldsymbol{u})^{\#}}_{\msL_{\eta;\bar{T}}^{\kappa}\mcC^{\beta}},\norm{\Psi(\boldsymbol{u})^{\#}}_{\msL_{2\eta;\bar{T}}^{\kappa}\mcC^{\beta^{\#}}}\}\lesssim(1+\bar{T}^{\theta}\norm{\boldsymbol{u}}_{\msD_{\bar{T}}})^{2}(1+\norm{\mbX}_{\rksnoise{\alpha}{\kappa}_{\bar{T}}}+\norm{\rho_0}_{\mcB_{p,q}^{\beta_{0}}})^{2},
	\end{equation}
	and
	\begin{equation}\label{eq:Psi_bound_2}
		\norm{\Psi(\boldsymbol{u})'}_{\msL_{\eta;\bar{T}}^{\kappa}\mcC^{\beta'}}\lesssim\norm{u}_{\msL_{\eta;\bar{T}}^{\kappa}\mcC^{\beta}}+\norm{\mbX}_{\rksnoise{\alpha}{\kappa}_{\bar{T}}}.
	\end{equation}
	Now denote
	\begin{equation*}
		(\Psi^{\circ2}(\boldsymbol{u}),\Psi^{\circ2}(\boldsymbol{u})',\Psi^{\circ2}(\boldsymbol{u})^{\#})\defeq\boldsymbol{\Psi}^{\circ2}(\boldsymbol{u}).
	\end{equation*}
	By iterating the bounds~\eqref{eq:Psi_bound_1}--\eqref{eq:Psi_bound_2}, using $\bar{T}\leq1$ to streamline exponents, we obtain
	\begin{equation}\label{eq:Psi_2_Bound}
		\norm{\boldsymbol{\Psi}^{\circ2}(\boldsymbol{u})}_{\msD_{\bar{T}}}\lesssim(1+\bar{T}^{\theta/2}\norm{\boldsymbol{u}}_{\msD_{\bar{T}}})^{4}(1+\norm{\mbX}_{\rksnoise{\alpha}{\kappa}_{\bar{T}}}+\norm{\rho_0}_{\mcB_{p,q}^{\beta_0}})^{6}.
	\end{equation}
	\begin{details}
		We obtain by~\eqref{eq:Psi_bound_1}--\eqref{eq:Psi_bound_2},
		\begin{equation}\label{eq:Psi_1_Bound}
			\norm{\boldsymbol{\Psi}(\boldsymbol{u})}_{\msD_{\bar{T}}}\lesssim(1+\norm{\boldsymbol{u}}_{\msD_{\bar{T}}})^{2}(1+\norm{\mbX}_{\rksnoise{\alpha}{\kappa}_{\bar{T}}}+\norm{\rho_0}_{\mcB_{p,q}^{\beta_{0}}})^{2}.
		\end{equation}
		By plugging~\eqref{eq:Psi_1_Bound} into~\eqref{eq:Psi_bound_1},
		\begin{equation*}
			\begin{split}
				&\max\{\norm{\Psi^{\circ2}(\boldsymbol{u})}_{\msL_{\eta;\bar{T}}^{\kappa}\mcC^{\beta}},\norm{\Psi^{\circ2}(\boldsymbol{u})^{\#}}_{\msL_{\eta;\bar{T}}^{\kappa}\mcC^{\beta}},\norm{\Psi^{\circ2}(\boldsymbol{u})^{\#}}_{\msL_{2\eta;\bar{T}}^{\kappa}\mcC^{\beta^{\#}}}\}\\
				&\lesssim(1+\bar{T}^{\theta}\norm{\boldsymbol{\Psi}(\boldsymbol{u})}_{\msD_{\bar{T}}})^{2}(1+\norm{\mbX}_{\rksnoise{\alpha}{\kappa}_{\bar{T}}}+\norm{\rho_0}_{\mcB_{p,q}^{\beta_{0}}})^{2}\\
				&\lesssim(1+\bar{T}^{\theta}(1+\norm{\boldsymbol{u}}_{\msD_{\bar{T}}})^{2}(1+\norm{\mbX}_{\rksnoise{\alpha}{\kappa}_{\bar{T}}}+\norm{\rho_{0}}_{\mcB_{p,q}^{\beta_{0}}})^{2})^{2}(1+\norm{\mbX}_{\rksnoise{\alpha}{\kappa}_{\bar{T}}}+\norm{\rho_0}_{\mcB_{p,q}^{\beta_{0}}})^{2}\\
				&\lesssim(1+\bar{T}^{\theta}(1+\norm{\boldsymbol{u}}_{\msD_{\bar{T}}})^{2})^{2}(1+\norm{\mbX}_{\rksnoise{\alpha}{\kappa}_{\bar{T}}}+\norm{\rho_0}_{\mcB_{p,q}^{\beta_{0}}})^{6}\\
				&\lesssim(1+\bar{T}^{\theta/2}\norm{\boldsymbol{u}}_{\msD_{\bar{T}}})^{4}(1+\norm{\mbX}_{\rksnoise{\alpha}{\kappa}_{\bar{T}}}+\norm{\rho_0}_{\mcB_{p,q}^{\beta_{0}}})^{6}.
			\end{split}
		\end{equation*}
		By plugging~\eqref{eq:Psi_bound_1} into~\eqref{eq:Psi_bound_2} and using that $\bar{T}^{\theta}\leq\bar{T}^{\theta/2}$,
		\begin{equation*}
			\begin{split}
				\norm{\Psi^{\circ2}(\boldsymbol{u})'}_{\msL_{\eta;\bar{T}}^{\kappa}\mcC^{\beta'}}&\lesssim\norm{\Psi(\boldsymbol{u})}_{\msL_{\eta;\bar{T}}^{\kappa}\mcC^{\beta}}+\norm{\mbX}_{\rksnoise{\alpha}{\kappa}_{\bar{T}}}\\
				&\lesssim(1+\bar{T}^{\theta}\norm{\boldsymbol{u}}_{\msD_{\bar{T}}})^{2}(1+\norm{\mbX}_{\rksnoise{\alpha}{\kappa}_{\bar{T}}}+\norm{\rho_0}_{\mcB_{p,q}^{\beta_{0}}})^{2}+\norm{\mbX}_{\rksnoise{\alpha}{\kappa}_{\bar{T}}}\\
				&\lesssim(1+\bar{T}^{\theta/2}\norm{\boldsymbol{u}}_{\msD_{\bar{T}}})^{4}(1+\norm{\mbX}_{\rksnoise{\alpha}{\kappa}_{\bar{T}}}+\norm{\rho_0}_{\mcB_{p,q}^{\beta_0}})^{6}.
			\end{split}
		\end{equation*}
		Together, this yields~\eqref{eq:Psi_2_Bound}.
	\end{details}
	
	Let $C>0$ be larger than the implicit constants of the inequalities \eqref{eq:Psi_bound_1} and \eqref{eq:Psi_2_Bound} above. Assume that $M,R>0$ are sufficiently large such that
	\begin{equation*}
		(1+\norm{\mbX}_{\rksnoise{\alpha}{\kappa}_{T}}+\norm{\rho_0}_{\mcB_{p,q}^{\beta_{0}}})<M,\qquad 2CM^{6}<R.
	\end{equation*}
	Assume further that $\norm{\boldsymbol{u}}_{\msD_{\bar{T}}}<R$. Using the bound \eqref{eq:Psi_2_Bound}, we can choose $\bar{T}=\bar{T}(R,\theta)\leq1$ smaller, if necessary, such that
	\begin{equation*}
		\norm{\boldsymbol{\Psi}^{\circ2}(\boldsymbol{u})}_{\msD_{\bar{T}}}\leq C(1+\bar{T}^{\theta/2}R)^{4}M^{6}\leq 2CM^{6}<R.
	\end{equation*}
	Consequently, $\boldsymbol{\Psi}^{\circ2}$ is a self-mapping on the ball
	\begin{equation*}
		\mfB_{R;\bar{T}}\defeq\{\boldsymbol{u}\in\msD_{\bar{T}}:\norm{\boldsymbol{u}}_{\msD_{\bar{T}}}<R\}.
	\end{equation*}
	Upon choosing $R>0$ sufficiently large, we can ensure that $\mfB_{R;\bar{T}}\subset\msD_{\bar{T}}$ is non-empty.
	\begin{details}
		Let $w'\defeq\nabla\Phi_{P\rho_{0}}$, $w^{\#}\defeq P\rho_{0}$ and $w\defeq\vdiv\mcI[w'\pa\ti]+w^{\#}$ as in Lemma~\ref{lem:parspace_non_trivial}. We can then choose $R>\norm{\boldsymbol{w}}_{\msD_{1}}$. It follows that $\norm{\boldsymbol{w}}_{\msD_{\bar{T}}}\leq\norm{\boldsymbol{w}}_{\msD_{1}}<R$ and $\boldsymbol{w}\in\mfB_{R;\bar{T}}$.
	\end{details}
	
	To achieve contractivity, we use the bilinearity of the equation. Let $\boldsymbol{v}=(v,v',v^{\#}),\boldsymbol{w}=(w,w',w^{\#})\in\msD_{\bar{T}}$ and denote
	\begin{equation*}
		(\Psi(\boldsymbol{v}),\Psi(\boldsymbol{v})',\Psi(\boldsymbol{v})^{\#})\defeq\boldsymbol{\Psi}(\boldsymbol{v}),\qquad(\Psi(\boldsymbol{w}),\Psi(\boldsymbol{w})',\Psi(\boldsymbol{w})^{\#})\defeq\boldsymbol{\Psi}(\boldsymbol{w}).
	\end{equation*}
	We see that
	\begin{align*}
		\Psi(\boldsymbol{v})-\Psi(\boldsymbol{w})&=\vdiv\mcI[(\Psi(\boldsymbol{v})'-\Psi(\boldsymbol{w})')\pa\ti]+\Psi(\boldsymbol{v})^{\#}-\Psi(\boldsymbol{w})^{\#},\\
		\Psi(\boldsymbol{v})'-\Psi(\boldsymbol{w})'&=\nabla\Phi_{v-w},\\
		\Psi(\boldsymbol{v})^{\#}-\Psi(\boldsymbol{w})^{\#}&=\vdiv\mcI[\Omega^{\#}(\boldsymbol{v})-\Omega^{\#}(\boldsymbol{w})],
	\end{align*}
	where
	\begin{equation*}
		\begin{split}
			\Omega^{\#}(\boldsymbol{v})-\Omega^{\#}(\boldsymbol{w})&=v\nabla\Phi_{v-w}+(v-w)\nabla\Phi_{w}+(v-w)\nabla\Phi_{\ty}+\ty\nabla\Phi_{v-w}+(v-w)\pa\nabla\Phi_{\ti}\\
			&\quad+\nabla\Phi_{\ti}\pa(v-w)+\ti\pa\nabla\Phi_{v-w}+\msP(\boldsymbol{v},\mbX)-\msP(\boldsymbol{w},\mbX).
		\end{split}
	\end{equation*}
	The difference of the renormalised products is given by
	\begin{equation*}
		\begin{split}
			\msP(\boldsymbol{v},\mbX)-\msP(\boldsymbol{w},\mbX)&=\msC(v'-w',\nabla\mcI[\ti],\nabla\Phi_{\ti})+\msC(v'-w',\nabla^{2}\mcI[\Phi_{\ti}],\ti)\\
			&\quad+(v^{\#}-w^{\#}+\vdiv\mcI[(v'-w')\pa\ti]-(v'-w')\pa\nabla\mcI[\ti])\re\nabla\Phi_{\ti}\\
			&\quad+(\nabla\Phi_{v^{\#}-w^{\#}}+\nabla\vdiv\Phi_{\mcI[(v'-w')\pa\ti]}-(v'-w')\pa\nabla^{2}\mcI[\Phi_{\ti}])\re\ti+(v'-w')\tc.
		\end{split}
	\end{equation*}
	Using the same bounds as before, we obtain, for some $\theta>0$,
	\begin{equation*}
		\begin{split}
			&\norm{\Psi(\boldsymbol{v})-\Psi(\boldsymbol{w})}_{\msL_{\eta;\bar{T}}^{\kappa}\mcC^{\beta}}\lesssim \bar{T}^{\theta}\norm{\Psi(\boldsymbol{v})'-\Psi(\boldsymbol{w})'}_{C_{\eta;\bar{T}}\mcC^{\beta'}}\norm{\ti}_{C_{\bar{T}}\mcC^{\alpha+1}}+\norm{\Psi(\boldsymbol{v})^{\#}-\Psi(\boldsymbol{w})^{\#}}_{\msL_{\eta;\bar{T}}^{\kappa}\mcC^{\beta}},
		\end{split}
	\end{equation*}
	and
	\begin{equation*}
		\max\{\norm{\Psi(\boldsymbol{v})^{\#}-\Psi(\boldsymbol{w})^{\#}}_{\msL_{\eta;\bar{T}}^{\kappa}\mcC^{\beta}},\norm{\Psi(\boldsymbol{v})^{\#}-\Psi(\boldsymbol{w})^{\#}}_{\msL_{2\eta;\bar{T}}^{\kappa}\mcC^{\beta^{\#}}}\}\lesssim \bar{T}^{\theta}\norm{\Omega^{\#}(\boldsymbol{v})-\Omega^{\#}(\boldsymbol{w})}_{C_{2\eta;\bar{T}}\mcC^{\alpha+\beta+2}},
	\end{equation*}
	as well as
	\begin{equation*}
		\norm{\Psi(\boldsymbol{v})'-\Psi(\boldsymbol{w})'}_{\msL_{\eta;\bar{T}}^{\kappa}\mcC^{\beta'}}\lesssim\norm{v-w}_{\msL_{\eta;\bar{T}}^{\kappa}\mcC^{\beta}}.
	\end{equation*}
	For the right-hand side,
	\begin{equation*}
		\begin{split}
			&\norm{\Omega^{\#}(\boldsymbol{v})-\Omega^{\#}(\boldsymbol{w})}_{C_{2\eta;\bar{T}}\mcC^{\alpha+\beta+2}}\\
			&\lesssim\norm{v}_{C_{\eta;\bar{T}}\mcC^{\beta}}\norm{v-w }_{C_{\eta;\bar{T}}\mcC^{\beta}}+\norm{v-w}_{C_{\eta;\bar{T}}\mcC^{\beta}}\norm{w}_{C_{\eta;\bar{T}}\mcC^{\beta}}+\norm{v-w}_{C_{\eta;\bar{T}}\mcC^{\beta}}\norm{\ty}_{C_{\bar{T}}\mcC^{2\alpha+4}}\\
			&\quad+\norm{v-w}_{C_{\eta;\bar{T}}\mcC^{\beta}}\norm{\ti}_{C_{\bar{T}}\mcC^{\alpha+1}}+\norm{\msP(\boldsymbol{v},\mbX)-\msP(\boldsymbol{w},\mbX)}_{C_{2\eta;\bar{T}}\mcC^{2\alpha+4}}.
		\end{split}
	\end{equation*}
	By the same arguments as in the proof of Lemma~\ref{lem:renormalised_products_estimates} we have
	\begin{equation*}
		\begin{split}
			&\norm{\msP(\boldsymbol{v},\mbX)-\msP(\boldsymbol{w},\mbX)}_{C_{2\eta;\bar{T}}\mcC^{2\alpha+4}}\\
			&\lesssim(\norm{v'-w'}_{\msL_{\eta;\bar{T}}^{\kappa}\mcC^{\beta'}}+\norm{v^{\#}-w^{\#}}_{C_{2\eta;\bar{T}}\mcC^{\beta^{\#}}})(1+\norm{\ti}_{C_{\bar{T}}\mcC^{\alpha+1}}+\norm{\tc}_{C_{\bar{T}}\mcC^{2\alpha+4}})^{2}.
		\end{split}
	\end{equation*}
	Combining the bounds above, we obtain
	\begin{align}
		&\max\{\norm{\Psi(\boldsymbol{v})-\Psi(\boldsymbol{w})}_{\msL_{\eta;\bar{T}}^{\kappa}\mcC^{\beta}},\norm{\Psi(\boldsymbol{v})^{\#}-\Psi(\boldsymbol{w})^{\#}}_{\msL_{\eta;\bar{T}}^{\kappa}\mcC^{\beta}},\norm{\Psi(\boldsymbol{v})^{\#}-\Psi(\boldsymbol{w})^{\#}}_{\msL_{2\eta;\bar{T}}^{\kappa}\mcC^{\beta^{\#}}}\}\notag\\
		&\qquad\lesssim\bar{T}^{\theta}\norm{\boldsymbol{v}-\boldsymbol{w}}_{\msD_{\bar{T}}}(1+\norm{\mbX}_{\rksnoise{\alpha}{\kappa}_{\bar{T}}}+\norm{v}_{\msL_{\eta;\bar{T}}^{\kappa}\mcC^{\beta}}+\norm{w}_{\msL_{\eta;\bar{T}}^{\kappa}\mcC^{\beta}})^{2},\label{eq:psi_bounds_diff_1}\\
		&\norm{\Psi(\boldsymbol{v})'-\Psi(\boldsymbol{w})'}_{\msL_{\eta;\bar{T}}^{\kappa}\mcC^{\beta'}}\lesssim\norm{v-w}_{\msL_{\eta;\bar{T}}^{\kappa}\mcC^{\beta}}.\label{eq:psi_bounds_diff_2}
	\end{align}
	
	Next we consider
	\begin{equation*}
		(\Psi^{\circ2}(\boldsymbol{v}),\Psi^{\circ2}(\boldsymbol{v})',\Psi^{\circ2}(\boldsymbol{v})^{\#})\defeq\boldsymbol{\Psi}^{\circ2}(\boldsymbol{v}),\qquad(\Psi^{\circ2}(\boldsymbol{w}),\Psi^{\circ2}(\boldsymbol{w})',\Psi^{\circ2}(\boldsymbol{w})^{\#})\defeq\boldsymbol{\Psi}^{\circ2}(\boldsymbol{w}).
	\end{equation*}
	Iterating~\eqref{eq:psi_bounds_diff_1}--\eqref{eq:psi_bounds_diff_2}, we arrive at
	\begin{equation*}
		\begin{split}
			&\norm{\boldsymbol{\Psi}^{\circ2}(\boldsymbol{v})-\boldsymbol{\Psi}^{\circ2}(\boldsymbol{w})}_{\msD_{\bar{T}}}\\
			&\lesssim \bar{T}^{\theta}\norm{\boldsymbol{v}-\boldsymbol{w}}_{\msD_{\bar{T}}}(1+\norm{\mbX}_{\rksnoise{\alpha}{\kappa}_{\bar{T}}}+\norm{v}_{\msL_{\eta;\bar{T}}^{\kappa}\mcC^{\beta}}+\norm{w}_{\msL_{\eta;\bar{T}}^{\kappa}\mcC^{\beta}}+\norm{\Psi(\boldsymbol{v})}_{\msL_{\eta;\bar{T}}^{\kappa}\mcC^{\beta}}+\norm{\Psi(\boldsymbol{w})}_{\msL_{\eta;\bar{T}}^{\kappa}\mcC^{\beta}})^{4}.
		\end{split}
	\end{equation*}
	\begin{details}
		We obtain by~\eqref{eq:psi_bounds_diff_1}--\eqref{eq:psi_bounds_diff_2},
		\begin{equation}\label{eq:Psi_1_Bound_diff}
			\norm{\boldsymbol{\Psi}(\boldsymbol{v})-\boldsymbol{\Psi}(\boldsymbol{w})}_{\msD_{\bar{T}}}\lesssim\norm{\boldsymbol{v}-\boldsymbol{w}}_{\msD_{\bar{T}}}(1+\norm{\mbX}_{\rksnoise{\alpha}{\kappa}_{\bar{T}}}+\norm{v}_{\msL_{\eta;\bar{T}}^{\kappa}\mcC^{\beta}}+\norm{w}_{\msL_{\eta;\bar{T}}^{\kappa}\mcC^{\beta}})^{2}.
		\end{equation} 
		By plugging~\eqref{eq:Psi_1_Bound_diff} into~\eqref{eq:psi_bounds_diff_1},
		\begin{equation*}
			\begin{split}
				&\max\{\norm{\Psi^{\circ2}(\boldsymbol{v})-\Psi^{\circ2}(\boldsymbol{w})}_{\msL_{\eta;\bar{T}}^{\kappa}\mcC^{\beta}},\norm{\Psi^{\circ2}(\boldsymbol{v})^{\#}-\Psi^{\circ2}(\boldsymbol{w})^{\#}}_{\msL_{\eta;\bar{T}}^{\kappa}\mcC^{\beta}},\norm{\Psi^{\circ2}(\boldsymbol{v})^{\#}-\Psi^{\circ2}(\boldsymbol{w})^{\#}}_{\msL_{2\eta;\bar{T}}^{\kappa}\mcC^{\beta^{\#}}}\}\\
				&\lesssim\bar{T}^{\theta}\norm{\boldsymbol{\Psi}(\boldsymbol{v})-\boldsymbol{\Psi}(\boldsymbol{w})}_{\msD_{\bar{T}}}(1+\norm{\mbX}_{\rksnoise{\alpha}{\kappa}_{\bar{T}}}+\norm{\Psi(\boldsymbol{v})}_{\msL_{\eta;\bar{T}}^{\kappa}\mcC^{\beta}}+\norm{\Psi(\boldsymbol{w})}_{\msL_{\eta;\bar{T}}^{\kappa}\mcC^{\beta}})^{2}\\
				&\lesssim\bar{T}^{\theta}\norm{\boldsymbol{v}-\boldsymbol{w}}_{\msD_{\bar{T}}}(1+\norm{\mbX}_{\rksnoise{\alpha}{\kappa}_{\bar{T}}}+\norm{v}_{\msL_{\eta;\bar{T}}^{\kappa}\mcC^{\beta}}+\norm{w}_{\msL_{\eta;\bar{T}}^{\kappa}\mcC^{\beta}}+\norm{\Psi(\boldsymbol{v})}_{\msL_{\eta;\bar{T}}^{\kappa}\mcC^{\beta}}+\norm{\Psi(\boldsymbol{w})}_{\msL_{\eta;\bar{T}}^{\kappa}\mcC^{\beta}})^{4}.
			\end{split}
		\end{equation*}
		By plugging~\eqref{eq:psi_bounds_diff_1} into~\eqref{eq:psi_bounds_diff_2},
		\begin{equation*}
			\begin{split}
				&\norm{\Psi^{\circ2}(\boldsymbol{v})'-\Psi^{\circ2}(\boldsymbol{w})'}_{\msL_{\eta;\bar{T}}^{\kappa}\mcC^{\beta'}}\\
				&\lesssim\norm{\Psi(\boldsymbol{v})-\Psi(\boldsymbol{w})}_{\msL_{\eta;\bar{T}}^{\kappa}\mcC^{\beta}}\\
				&\lesssim\bar{T}^{\theta}\norm{\boldsymbol{v}-\boldsymbol{w}}_{\msD_{\bar{T}}}(1+\norm{\mbX}_{\rksnoise{\alpha}{\kappa}_{\bar{T}}}+\norm{v}_{\msL_{\eta;\bar{T}}^{\kappa}\mcC^{\beta}}+\norm{w}_{\msL_{\eta;\bar{T}}^{\kappa}\mcC^{\beta}})^{2}\\
				&\lesssim\bar{T}^{\theta}\norm{\boldsymbol{v}-\boldsymbol{w}}_{\msD_{\bar{T}}}(1+\norm{\mbX}_{\rksnoise{\alpha}{\kappa}_{\bar{T}}}+\norm{v}_{\msL_{\eta;\bar{T}}^{\kappa}\mcC^{\beta}}+\norm{w}_{\msL_{\eta;\bar{T}}^{\kappa}\mcC^{\beta}}+\norm{\Psi(\boldsymbol{v})}_{\msL_{\eta;\bar{T}}^{\kappa}\mcC^{\beta}}+\norm{\Psi(\boldsymbol{w})}_{\msL_{\eta;\bar{T}}^{\kappa}\mcC^{\beta}})^{4}.
			\end{split}
		\end{equation*}
		This yields the claimed estimate.
		
	\end{details}
	Assume $\boldsymbol{v},\boldsymbol{w}\in \mfB_{R;\bar{T}}$, it follows by~\eqref{eq:Psi_bound_1}, the definition of $\bar{T}$ and $\mfB_{R;\bar{T}}$, that
	\begin{equation*}
		\norm{v}_{\msL_{\eta;\bar{T}}^{\kappa}\mcC^{\beta}}<R,\qquad\norm{w}_{\msL_{\eta;\bar{T}}^{\kappa}\mcC^{\beta}}<R,\qquad\norm{\Psi(\boldsymbol{v})}_{\msL_{\eta;\bar{T}}^{\kappa}\mcC^{\beta}}<R,\qquad\norm{\Psi(\boldsymbol{w})}_{\msL_{\eta;\bar{T}}^{\kappa}\mcC^{\beta}}<R.
	\end{equation*}
	\begin{details}
		We estimate with~\eqref{eq:Psi_bound_1},
		\begin{equation*}
			\begin{split}
				\norm{\Psi(\boldsymbol{v})}_{\msL_{\eta;\bar{T}}^{\kappa}\mcC^{\beta}}&\leq C(1+\bar{T}^{\theta}\norm{\boldsymbol{v}}_{\msD_{\bar{T}}})^{2}(1+\norm{\mbX}_{\rksnoise{\alpha}{\kappa}_{\bar{T}}}+\norm{\rho_0}_{\mcB_{p,q}^{\beta_{0}}})^{2}\\
				&<C(1+\bar{T}^{\theta}R)^{2}(1+\norm{\mbX}_{\rksnoise{\alpha}{\kappa}_{\bar{T}}}+\norm{\rho_0}_{\mcB_{p,q}^{\beta_{0}}})^{2}\\
				&\leq C(1+\bar{T}^{\theta/2}R)^{4}M^{6}\\
				&\leq2CM^{6}<R.
			\end{split}
		\end{equation*}
	\end{details}
	Choosing $\bar{T}$ still smaller, if necessary, we can arrange that for some $c<1$,
	\begin{equation*}
		\norm{\boldsymbol{\Psi}^{\circ2}(\boldsymbol{v})-\boldsymbol{\Psi}^{\circ2}(\boldsymbol{w})}_{\msD_{\bar{T}}}<c\norm{\boldsymbol{v}-\boldsymbol{w}}_{\msD_{\bar{T}}},
	\end{equation*}
	showing that $\boldsymbol{\Psi}^{\circ2}$ is a contraction on $\mfB_{R;\bar{T}}$.
	
	By Banach's fixed-point theorem there exists a unique fixed point for $\boldsymbol{\Psi}^{\circ2}$ in $\mfB_{R;\bar{T}}$. It suffices to argue that a fixed point to $\boldsymbol{\Psi}^{\circ2}$ is also a fixed point to $\boldsymbol{\Psi}$. The following is due to~\cite[Thm.~5.15]{perkowski_20}. Denote for the sake of notation, $\boldsymbol{w}=\boldsymbol{\Psi}^{\circ2}(\boldsymbol{w})$ and $\boldsymbol{v}\defeq\boldsymbol{\Psi}(\boldsymbol{w})$. We have $\boldsymbol{\Psi}^{\circ2}(\boldsymbol{v})=\boldsymbol{\Psi}(\boldsymbol{w})=\boldsymbol{v}$, hence $\boldsymbol{v}$ is itself a fixed point to $\boldsymbol{\Psi}^{\circ2}$, yielding by uniqueness that $\boldsymbol{v}=\boldsymbol{w}$. One can furthermore show that this fixed point is in fact unique in all of $\msD_{\bar{T}}$; it suffices to compare two putative solutions in $\msD_{\bar{T}}$ and similar estimates to those above show that they must be equal on a small time interval. Continuity then gives equality in all of $[0,\bar{T}]$.
\end{proof}
The utility of Proposition~\ref{prop:w_local_existence} is that it allows us to show existence and uniqueness of a suitable notion of solution to~\eqref{eq:gen_rKS_intro} by setting $\rho\defeq\ti+\ty+w$. We call this solution paracontrolled and are now in a position to give a rigorous definition (cf.~\eqref{eq:paracon_sol_intro_I}--\eqref{eq:paracon_sol_intro_II} for a formal motivation.)
\begin{definition}\label{def:paracontrolled_solution}
	Let $\mbX\in\rksnoise{\alpha}{\kappa}_{T}$ and $\rho_{0}\in\mcB_{p,q}^{\beta_{0}}(\mbT^{2})$. For every $S\leq T$, we call $\rho\from[0,S]\to\mcS'(\mbT^{2})$ a paracontrolled solution to~\eqref{eq:gen_rKS_intro} on $[0,S]$ with enhancement $\mbX$ and initial data $\rho_{0}$, if
	\begin{itemize}
		\item the triple $\boldsymbol{w}\defeq(w,w',w^{\#})$ of distributions
		\begin{equation*}
			w\defeq\rho-\ti-\ty,\qquad w'\defeq\nabla\Phi_{w}+\nabla\Phi_{\ty},\qquad w^{\#}\defeq\rho-\ti-\ty-\vdiv\mcI[w'\pa\ti];
		\end{equation*}
		satisfies $\boldsymbol{w}\in\msD_{S}$ (see Definition~\ref{def:paracontrolled_space});
		\item the solution $\rho$ satisfies $\rho=P\rho_{0}+\vdiv\mcI[\rho\nabla\Phi_{\rho}]+\ti$, where the product $\vdiv\mcI[\rho\nabla\Phi_{\rho}]$ needs to be interpreted in the renormalised sense, that is
		\begin{equation*}
			\vdiv\mcI[\rho\nabla\Phi_{\rho}]\defeq\ty+\vdiv\mcI[w'\pa\ti]+\vdiv\mcI[\Omega^{\#}(\boldsymbol{w})]
		\end{equation*}
		with $\Omega^{\#}(\boldsymbol{w})$ defined as in Lemma~\ref{lem:bounds_apriori}.
	\end{itemize}
	In particular it follows by Definition~\ref{def:paracontrolled_space}, that $\lim_{t\to0}\rho_{t}=\rho_{0}$ in $\mcS'(\mbT^{2})$.
	We call $\rho\from[0,S)\to\mcS'(\mbT^{2})$ a paracontrolled solution on $[0,S)$, if $\rho$ is a paracontrolled solution on $[0,S']$ for every $S'<S$.
	%
	%
\end{definition}
Using the uniqueness of the fixed point in Proposition~\ref{prop:w_local_existence}, one can show that paracontrolled solutions are unique.
\begin{lemma}\label{lem:paracontrolled_solution_unique}
	Let $\mbX\in\rksnoise{\alpha}{\kappa}_{T}$, $\rho_{0}\in\mcB_{p,q}^{\beta_{0}}(\mbT^{2})$ and $S\leq T$. Then there exists at most one paracontrolled solution $\rho\from[0,S]\to\mcS'(\mbT^{2})$ to~\eqref{eq:gen_rKS_intro} on $[0,S]$ with enhancement $\mbX$ and initial data $\rho_{0}$.
\end{lemma}
\begin{proof}
	Consider two paracontrolled solutions $\rho$, $\widetilde{\rho}$, which by definition are associated to two triples $\boldsymbol{w}$, $\widetilde{\boldsymbol{w}}$. However, these triples must both satisfy by definition the same fixed-point equation~\eqref{eq:paracontrolled_equation}, which has a unique solution by Proposition~\ref{prop:w_local_existence}. Therefore both paracontrolled solutions must agree.
\end{proof}
In the following lemma, we prove the existence of a paracontrolled solution given an abstract enhancement.
\begin{lemma}\label{lem:construction_paracontrolled}
	Let $\mbX\in\rksnoise{\alpha}{\kappa}_{T}$ and $\rho_0\in\mcB_{p,q}^{\beta_{0}}(\mbT^{2})$. Let $\bar{T}\in (0,T\wedge1]$ be as in Proposition~\ref{prop:w_local_existence} and $\boldsymbol{w}=(w,w',w^{\#})\in\msD_{\bar{T}}$ be the fixed point of~\eqref{eq:paracontrolled_equation} constructed therein. Then $\rho=\ti+\ty+w\in\msL_{\eta;\bar{T}}^{\kappa}\mcC^{\alpha+1}(\mbT^{2})$ is the unique paracontrolled solution to~\eqref{eq:gen_rKS_intro} on $[0,\bar{T}]$ with enhancement $\mbX$ and initial data $\rho_{0}$.
\end{lemma}
\begin{proof}
	Let $\boldsymbol{w}=(w,w',w^{\#})\in\msD_{\bar{T}}$ be the fixed point to~\eqref{eq:paracontrolled_equation} and $\rho\defeq\ti+\ty+w$. It follows by definition,
	\begin{equation*}
		w=\rho-\ti-\ty,\qquad w'=\nabla\Phi_{w}+\nabla\Phi_{\ty},\qquad w^{\#}=w-\vdiv\mcI[w'\pa\ti]=\rho-\ti-\ty-\vdiv\mcI[w'\pa\ti]
	\end{equation*}
	and
	\begin{equation*}
		\rho=P\rho_{0}+\ty+\vdiv\mcI[w'\pa\ti]+\vdiv\mcI[\Omega^{\#}(\boldsymbol{w})]+\ti.
	\end{equation*}
	Hence, $\rho$ is a paracontrolled solution to~\eqref{eq:gen_rKS_intro} on $[0,\bar{T}]$ with enhancement $\mbX$ and initial data $\rho_{0}$, which is unique by Lemma~\ref{lem:paracontrolled_solution_unique}.
\end{proof}
The next lemma shows that paracontrolled solutions are locally Lipschitz continuous in the noise enhancement and the initial data.
\begin{lemma}\label{lem:Local_Lipschitz}
	Let $R>0$, $\mbX=(\ti_{\textsf{X}},{\ty\!}_{\textsf{X}},\tp_{\textsf{X}},\tc_{\textsf{X}}),\mbY=(\ti_{\textsf{Y}},{\ty\!}_{\textsf{Y}},\tp_{\textsf{Y}},\tc_{\textsf{Y}})\in\rksnoise{\alpha}{\kappa}_{T}$ and $\rho_{0}^{\textsf{X}},\rho_{0}^{\textsf{Y}}\in\mcB_{p,q}^{\beta_{0}}(\mbT^{2})$ be such that
	\begin{equation*}
		\max\{\norm{\mbX}_{\rksnoise{\alpha}{\kappa}_{T}},\norm{\mbY}_{\rksnoise{\alpha}{\kappa}_{T}},\norm{\rho_{0}^{\textsf{X}}}_{\mcB_{p,q}^{\beta_{0}}},\norm{\rho_{0}^{\textsf{Y}}}_{\mcB_{p,q}^{\beta_{0}}}\}<R.
	\end{equation*}
	Then there exists some $\bar{T}=\bar{T}(R)\in (0,T]$ with the following properties
	\begin{enumerate}
		\item There exist solutions $\boldsymbol{w}_{\textsf{X}}\defeq(w_{\textsf{X}},w'_{\textsf{X}},w^{\#}_{\textsf{X}})$ and $\boldsymbol{w}_{\textsf{Y}}\defeq(w_{\textsf{Y}},w'_{\textsf{Y}},w^{\#}_{\textsf{Y}})$ to
		\begin{equation*}
			\begin{cases}
				\begin{aligned}
					w_{\textsf{X}}&\defeq\vdiv\mcI[w'_{\textsf{X}}\pa \ti_{\textsf{X}}]+w^{\#}_{\textsf{X}},\quad w'_{\textsf{X}}\defeq\nabla\Phi_{w_{\textsf{X}}}+\nabla\Phi_{{\ty\!}_{\textsf{X}}},\\
					w^{\#}_{\textsf{X}}&\defeq P\rho_{0}^{\textsf{X}}+\vdiv\mcI[\Omega^{\#}_{\textsf{X}}(\boldsymbol{w}_{\textsf{X}})],
				\end{aligned}
			\end{cases}
		\end{equation*}
		and
		\begin{equation*}
			\begin{cases}
				\begin{aligned}
					w_{\textsf{Y}}&\defeq\vdiv\mcI[w'_{\textsf{Y}}\pa \ti_{\textsf{Y}}]+w^{\#}_{\textsf{Y}},\quad w'_{\textsf{Y}}\defeq\nabla\Phi_{w_{\textsf{Y}}}+\nabla\Phi_{{\ty\!}_{\textsf{Y}}},\\
					w^{\#}_{\textsf{Y}}&\defeq P\rho_{0}^{\textsf{Y}}+\vdiv\mcI[\Omega^{\#}_{\textsf{Y}}(\boldsymbol{w}_{\textsf{Y}})],
				\end{aligned}
			\end{cases}
		\end{equation*}
		on $[0,\bar{T}]$ by an application of Proposition~\ref{prop:w_local_existence}. Here, $\Omega^{\#}_{\textsf{X}}(\boldsymbol{w}_{\textsf{X}})$ and $\Omega^{\#}_{\textsf{Y}}(\boldsymbol{w}_{\textsf{Y}})$ are defined as in Lemma~\ref{lem:bounds_apriori} with noises $\mbX$, $\mbY$ respectively.
		\item Setting
		\begin{equation*}
			\rho_{\textsf{X}}\defeq\ti_{\textsf{X}}+{\ty\!}_{\textsf{X}}+w_{\textsf{X}},\qquad\rho_{\textsf{Y}}\defeq\ti_{\textsf{Y}}+{\ty\!}_{\textsf{Y}}+w_{\textsf{Y}},
		\end{equation*}
		one has
		\begin{equation*}
			\begin{split}
				&\max\{\norm{w_{\textsf{X}}-w_{\textsf{Y}}}_{\msL_{\eta;\bar{T}}^{\kappa}\mcC^{\beta}},\norm{w^{\#}_{\textsf{X}}-w^{\#}_{\textsf{Y}}}_{\msL_{\eta;\bar{T}}^{\kappa}\mcC^{\beta}},\norm{w^{\#}_{\textsf{X}}-w^{\#}_{\textsf{Y}}}_{\msL_{2\eta;\bar{T}}^{\kappa}\mcC^{\beta^{\#}}},\norm{\rho_{\textsf{X}}-\rho_{\textsf{Y}}}_{\msL_{\eta;\bar{T}}^{\kappa}\mcC^{\alpha+1}}\}\\
				&\lesssim\norm{\rho^{\textsf{X}}_{0}-\rho^{\textsf{Y}}_{0}}_{\mcB_{p,q}^{\beta_{0}}}+\norm{\mbX-\mbY}_{\rksnoise{\alpha}{\kappa}_{\bar{T}}}(\norm{\boldsymbol{w}_{\textsf{X}}}_{\msD_{\bar{T}}}+\norm{\boldsymbol{w}_{\textsf{Y}}}_{\msD_{\bar{T}}}+\norm{\mbX}_{\rksnoise{\alpha}{\kappa}_{\bar{T}}}+\norm{\mbY}_{\rksnoise{\alpha}{\kappa}_{\bar{T}}}+1)^{2}.
			\end{split}
		\end{equation*}
	\end{enumerate}
\end{lemma}
\begin{proof}
	The claim follows as in Lemma~\ref{lem:bounds_apriori}, using the trilinearity of the equation.
\end{proof}
\begin{details}
	\begin{proof}
		We establish the local Lipschitz continuity by bounding $\norm{\rho_{\textsf{X}}-\rho_{\textsf{Y}}}_{\msL_{\eta;\bar{T}}^{\kappa}\mcC^{\alpha+1}}$, $\norm{w_{\textsf{X}}-w_{\textsf{Y}}}_{\msL_{\eta;\bar{T}}^{\kappa}\mcC^{\beta}}$, $\norm{w'_{\textsf{X}}-w'_{\textsf{Y}}}_{\msL_{\eta;\bar{T}}^{\kappa}\mcC^{\beta'}}$, $\norm{w^{\#}_{\textsf{X}}-w^{\#}_{\textsf{Y}}}_{\msL_{\eta;\bar{T}}^{\kappa}\mcC^{\beta}}$ and $\norm{w^{\#}_{\textsf{X}}-w^{\#}_{\textsf{Y}}}_{\msL_{2\eta;\bar{T}}^{\kappa}\mcC^{\beta^{\#}}}$.
		
		We start with the Gubinelli derivatives $w'_{\textsf{X}}-w'_{\textsf{Y}}$ and bound, using that $\beta'\leq\beta+1$ and $\beta'\leq2\alpha+5$,
		\begin{equation}\label{eq:Gubinelli_derivative_lipschitz}
			\begin{split}
				\norm{w'_{\textsf{X}}-w'_{\textsf{Y}}}_{\msL_{\eta;\bar{T}}^{\kappa}\mcC^{\beta'}}&=\norm{\nabla\Phi_{w_{\textsf{X}}}-\nabla\Phi_{w_{\textsf{Y}}}+\nabla\Phi_{{\ty\!}_{\textsf{X}}}-\nabla\Phi_{{\ty\!}_{\textsf{Y}}}}_{\msL_{\eta;\bar{T}}^{\kappa}\mcC^{\beta'}}\\
				&\lesssim\norm{w_{\textsf{X}}-w_{\textsf{Y}}}_{\msL_{\eta;\bar{T}}^{\kappa}\mcC^{\beta}}+\norm{\mbX-\mbY}_{\rksnoise{\alpha}{\kappa}_{\bar{T}}}.
			\end{split}
		\end{equation}
		We have
		\begin{equation*}
			\rho_{\textsf{X}}-\rho_{\textsf{Y}}=\ti_{\textsf{X}}-\ti_{\textsf{Y}}+{\ty\!}_{\textsf{X}}-{\ty\!}_{\textsf{Y}}+w_{\textsf{X}}-w_{\textsf{Y}}.
		\end{equation*}
		For the noise terms, we obtain using that $\alpha+1\leq 2\alpha+4$,
		\begin{equation*}
			\norm{\ti_{\textsf{X}}-\ti_{\textsf{Y}}}_{\msL_{\bar{T}}^{\kappa}\mcC^{\alpha+1}}+\norm{{\ty\!}_{\textsf{X}}-{\ty\!}_{\textsf{Y}}}_{\msL_{\bar{T}}^{\kappa}\mcC^{\alpha+1}}\lesssim\norm{\mbX-\mbY}_{\rksnoise{\alpha}{\kappa}_{\bar{T}}}.
		\end{equation*}
		We further decompose
		\begin{equation*}
			\begin{split}
				w_{\textsf{X}}-w_{\textsf{Y}}&=\vdiv\mcI[w'_{\textsf{X}}\pa\ti_{\textsf{X}}]-\vdiv\mcI[w'_{\textsf{Y}}\pa\ti_{\textsf{Y}}]+w^{\#}_{\textsf{X}}-w^{\#}_{\textsf{Y}}\\
				&=\vdiv\mcI[(w'_{\textsf{X}}-w'_{\textsf{Y}})\pa\ti_{\textsf{X}}]+\vdiv\mcI[w'_{\textsf{Y}}\pa(\ti_{\textsf{X}}-\ti_{\textsf{Y}})]+w^{\#}_{\textsf{X}}-w^{\#}_{\textsf{Y}}.
			\end{split}
		\end{equation*}
		By Lemma~\ref{lem:Schauder}, using that $\beta'>0$,
		\begin{equation*}
			\begin{split}
				\norm{\vdiv\mcI[w'_{\textsf{X}}\pa\ti_{\textsf{X}}]-\vdiv\mcI[w'_{\textsf{Y}}\pa\ti_{\textsf{Y}}]}_{\msL_{\eta;\bar{T}}^{\kappa}\mcC^{\beta}}&\lesssim(\bar{T}^{1-\frac{\beta-\alpha}{2}}\vee \bar{T}^{1-\kappa})\norm{w'_{\textsf{X}}-w'_{\textsf{Y}}}_{C_{\eta;\bar{T}}\mcC^{\beta'}}\norm{\ti_{\textsf{X}}}_{C_{\bar{T}}\mcC^{\alpha+1}}\\
				&\quad+(\bar{T}^{1-\frac{\beta-\alpha}{2}}\vee \bar{T}^{1-\kappa})\norm{w'_{\textsf{Y}}}_{C_{\eta;\bar{T}}\mcC^{\beta'}}\norm{\ti_{\textsf{X}}-\ti_{\textsf{Y}}}_{C_{\bar{T}}\mcC^{\alpha+1}}.
			\end{split}
		\end{equation*}
		Therefore
		\begin{equation*}
			\begin{split}
				\norm{w_{\textsf{X}}-w_{\textsf{Y}}}_{\msL_{\eta;\bar{T}}^{\kappa}\mcC^{\beta}}&\lesssim(\bar{T}^{1-\frac{\beta-\alpha}{2}}\vee \bar{T}^{1-\kappa})\norm{w'_{\textsf{X}}-w'_{\textsf{Y}}}_{C_{\eta;\bar{T}}\mcC^{\beta'}}\norm{\mbX}_{\rksnoise{\alpha}{\kappa}_{\bar{T}}}\\
				&\quad+(\bar{T}^{1-\frac{\beta-\alpha}{2}}\vee \bar{T}^{1-\kappa})\norm{w'_{\textsf{Y}}}_{C_{\eta;\bar{T}}\mcC^{\beta'}}\norm{\mbX-\mbY}_{\rksnoise{\alpha}{\kappa}_{\bar{T}}}+\norm{w^{\#}_{\textsf{X}}-w^{\#}_{\textsf{Y}}}_{\msL_{\eta;\bar{T}}^{\kappa}\mcC^{\beta}}.
			\end{split}
		\end{equation*}
		Plugging in~\eqref{eq:Gubinelli_derivative_lipschitz}, we obtain
		\begin{equation*}
			\begin{split}
				&\norm{w_{\textsf{X}}-w_{\textsf{Y}}}_{\msL_{\eta;\bar{T}}^{\kappa}\mcC^{\beta}}\\
				&\lesssim(\bar{T}^{1-\frac{\beta-\alpha}{2}}\vee \bar{T}^{1-\kappa})\norm{w_{\textsf{X}}-w_{\textsf{Y}}}_{\msL_{\eta;\bar{T}}^{\kappa}\mcC^{\beta}}\norm{\mbX}_{\rksnoise{\alpha}{\kappa}_{\bar{T}}}+(\bar{T}^{1-\frac{\beta-\alpha}{2}}\vee \bar{T}^{1-\kappa})\norm{\mbX-\mbY}_{\rksnoise{\alpha}{\kappa}_{\bar{T}}}\norm{\mbX}_{\rksnoise{\alpha}{\kappa}_{\bar{T}}}\\
				&\quad+(\bar{T}^{1-\frac{\beta-\alpha}{2}}\vee \bar{T}^{1-\kappa})\norm{w'_{\textsf{Y}}}_{C_{\eta;\bar{T}}\mcC^{\beta'}}\norm{\mbX-\mbY}_{\rksnoise{\alpha}{\kappa}_{\bar{T}}}+\norm{w^{\#}_{\textsf{X}}-w^{\#}_{\textsf{Y}}}_{\msL_{\eta;\bar{T}}^{\kappa}\mcC^{\beta}}.
			\end{split}
		\end{equation*}
		We can choose $\bar{T}=\bar{T}(\alpha,\beta,\kappa,\norm{\mbX}_{\rksnoise{\alpha}{\kappa}_{T}})\leq1$ smaller, if necessary, to absorb the first term, so that
		\begin{equation}\label{eq:w_lipschitz_1}
			\norm{w_{\textsf{X}}-w_{\textsf{Y}}}_{\msL_{\eta;\bar{T}}^{\kappa}\mcC^{\beta}}\lesssim\norm{\mbX-\mbY}_{\rksnoise{\alpha}{\kappa}_{\bar{T}}}(\norm{\mbX}_{\rksnoise{\alpha}{\kappa}_{\bar{T}}}+\norm{w'_{\textsf{Y}}}_{C_{\eta;\bar{T}}\mcC^{\beta'}})+\norm{w^{\#}_{\textsf{X}}-w^{\#}_{\textsf{X}}}_{\msL_{\eta;\bar{T}}^{\kappa}\mcC^{\beta}}.
		\end{equation}
		Next we decompose
		\begin{equation*}
			w^{\#}_{\textsf{X}}-w^{\#}_{\textsf{Y}}=P\rho^{\textsf{X}}_{0}-P\rho^{\textsf{Y}}_{0}+\vdiv\mcI[\Omega^{\#}_{\textsf{X}}(\boldsymbol{w}_{\textsf{X}})]-\vdiv\mcI[\Omega^{\#}_{\textsf{Y}}(\boldsymbol{w}_{\textsf{Y}})].
		\end{equation*}
		We control the $\msL_{\eta;\bar{T}}^{\kappa}\mcC^{\beta}$-regularity. By Lemma~\ref{lem:Schauder}, using that $(\frac{\beta-\beta_{0}}{2}\vee0)+\frac{1}{p}\leq\eta$,
		\begin{equation*}
			\norm{P\rho^{\textsf{X}}_{0}-P\rho^{\textsf{Y}}_{0}}_{\msL_{\eta;\bar{T}}^{\kappa}\mcC^{\beta}}\lesssim \bar{T}^{\eta-(\frac{\beta-\beta_{0}}{2}\vee0)-\frac{1}{p}}\norm{\rho^{\textsf{X}}_{0}-\rho^{\textsf{Y}}_{0}}_{\mcB_{p,q}^{\beta_{0}}}.
		\end{equation*}
		Next by Lemma~\ref{lem:Schauder},
		\begin{equation*}
			\norm{\vdiv\mcI[\Omega^{\#}_{\textsf{X}}(\boldsymbol{w}_{\textsf{X}})]-\vdiv\mcI[\Omega^{\#}_{\textsf{Y}}(\boldsymbol{w}_{\textsf{Y}})]}_{\msL_{\eta;\bar{T}}^{\kappa}\mcC^{\beta}}\lesssim (\bar{T}^{\frac{\alpha+3}{2}-\eta}\vee \bar{T}^{1-\kappa-\eta})\norm{\Omega^{\#}_{\textsf{X}}(\boldsymbol{w}_{\textsf{X}})-\Omega^{\#}_{\textsf{Y}}(\boldsymbol{w}_{\textsf{Y}})}_{C_{2\eta;\bar{T}}\mcC^{\alpha+\beta+2}}.
		\end{equation*}
		It follows that
		\begin{equation}\label{eq:w_sharp_lipschitz_beta_1}
			\norm{w^{\#}_{\textsf{X}}-w^{\#}_{\textsf{Y}}}_{\msL_{\eta;\bar{T}}^{\kappa}\mcC^{\beta}}\lesssim\norm{\rho^{\textsf{X}}_{0}-\rho^{\textsf{Y}}_{0}}_{\mcB_{p,q}^{\beta_{0}}}+(\bar{T}^{\frac{\alpha+3}{2}-\eta}\vee \bar{T}^{1-\kappa-\eta})\norm{\Omega^{\#}_{\textsf{X}}(\boldsymbol{w}_{\textsf{X}})-\Omega^{\#}_{\textsf{Y}}(\boldsymbol{w}_{\textsf{Y}})}_{C_{2\eta;\bar{T}}\mcC^{\alpha+\beta+2}}.
		\end{equation}
		We control the $\msL_{2\eta;\bar{T}}^{\kappa}\mcC^{\beta^{\#}}$-regularity. By Lemma~\ref{lem:Schauder}, using that $\beta_{0}\leq\beta^{\#}$ and $\frac{\beta^{\#}-\beta_{0}}{2}+\frac{1}{p}\leq2\eta$,
		\begin{equation*}
			\norm{P\rho^{\textsf{X}}_{0}-P\rho^{\textsf{Y}}_{0}}_{\msL_{2\eta;\bar{T}}^{\kappa}\mcC^{\beta^{\#}}}\lesssim \bar{T}^{2\eta-\frac{\beta^{\#}-\beta_{0}}{2}-\frac{1}{p}}\norm{\rho^{\textsf{X}}_{0}-\rho^{\textsf{Y}}_{0}}_{\mcB_{p,q}^{\beta_{0}}}.
		\end{equation*}
		Next by Lemma~\ref{lem:Schauder},
		\begin{equation*}
			\norm{\vdiv\mcI[\Omega^{\#}_{\textsf{X}}(\boldsymbol{w}_{\textsf{X}})]-\vdiv\mcI[\Omega^{\#}_{\textsf{Y}}(\boldsymbol{w}_{\textsf{Y}})]}_{\msL_{2\eta;\bar{T}}^{\kappa}\mcC^{\beta^{\#}}}\lesssim(\bar{T}^{1-\frac{\beta^{\#}-\alpha-\beta-1}{2}}\vee \bar{T}^{1-\kappa})\norm{\Omega^{\#}_{\textsf{X}}(\boldsymbol{w}_{\textsf{X}})-\Omega^{\#}_{\textsf{Y}}(\boldsymbol{w}_{\textsf{Y}})}_{C_{2\eta;\bar{T}}\mcC^{\alpha+\beta+2}}.
		\end{equation*}
		It follows that
		\begin{equation}\label{eq:w_sharp_lipschitz_beta_sharp_1}
			\norm{w^{\#}_{\textsf{X}}-w^{\#}_{\textsf{Y}}}_{\msL_{2\eta;\bar{T}}^{\kappa}\mcC^{\beta^{\#}}}\lesssim\norm{\rho^{\textsf{X}}_{0}-\rho^{\textsf{Y}}_{0}}_{\mcB_{p,q}^{\beta_{0}}}+(\bar{T}^{1-\frac{\beta^{\#}-\alpha-\beta-1}{2}}\vee \bar{T}^{1-\kappa})\norm{\Omega^{\#}_{\textsf{X}}(\boldsymbol{w}_{\textsf{X}})-\Omega^{\#}_{\textsf{Y}}(\boldsymbol{w}_{\textsf{Y}})}_{C_{2\eta;\bar{T}}\mcC^{\alpha+\beta+2}}.
		\end{equation}
		We decompose
		\begin{equation*}
			\begin{split}
				&\Omega_{\textsf{X}}^{\#}(\boldsymbol{w}_{\textsf{X}})-\Omega^{\#}_{\textsf{Y}}(\boldsymbol{w}_{\textsf{Y}})\\
				&=w_{\textsf{X}}\nabla\Phi_{w_{\textsf{X}}}-w_{\textsf{Y}}\nabla\Phi_{w_{\textsf{Y}}}+w_{\textsf{X}}\nabla\Phi_{{\ty\!}_{\textsf{X}}}-w_{\textsf{Y}}\nabla\Phi_{{\ty\!}_{\textsf{Y}}}+{\ty\!}_{\textsf{X}}\nabla\Phi_{w_{\textsf{X}}}-{\ty\!}_{\textsf{Y}}\nabla\Phi_{w_{\textsf{Y}}}+{\ty\!}_{\textsf{X}}\nabla\Phi_{{\ty\!}_{\textsf{X}}}-{\ty\!}_{\textsf{Y}}\nabla\Phi_{{\ty\!}_{\textsf{Y}}}\\
				&\quad+\tp_{\textsf{X}}-\tp_{\textsf{Y}}+\nabla\Phi_{\ti_{\textsf{X}}}\pa{\ty\!}_{\textsf{X}}-\nabla\Phi_{\ti_{\textsf{Y}}}\pa{\ty\!}_{\textsf{Y}}+{\ty\!}_{\textsf{X}}\pa\nabla\Phi_{\ti_{\textsf{X}}}-{\ty\!}_{\textsf{Y}}\pa\nabla\Phi_{\ti_{\textsf{Y}}}\\
				&\quad+\ti_{\textsf{X}}\pa\nabla\Phi_{{\ty\!}_{\textsf{X}}}-\ti_{\textsf{Y}}\pa\nabla\Phi_{{\ty\!}_{\textsf{Y}}}+w_{\textsf{X}}\pa\nabla\Phi_{\ti_{\textsf{X}}}-w_{\textsf{Y}}\pa\nabla\Phi_{\ti_{\textsf{Y}}}+\nabla\Phi_{\ti_{\textsf{X}}}\pa w_{\textsf{X}}-\nabla\Phi_{\ti_{\textsf{Y}}}\pa w_{\textsf{Y}}\\
				&\quad+\ti_{\textsf{X}}\pa\nabla\Phi_{w_{\textsf{X}}}-\ti_{\textsf{Y}}\pa\nabla\Phi_{w_{\textsf{Y}}}+\msP(\boldsymbol{w}_{\textsf{X}},\mbX)-\msP(\boldsymbol{w}_{\textsf{Y}},\mbY).
			\end{split}
		\end{equation*}
		Each term can be controlled as in Lemma~\ref{lem:bounds_apriori}. We decompose
		\begin{equation*}
			w_{\textsf{X}}\nabla\Phi_{w_{\textsf{X}}}-w_{\textsf{Y}}\nabla\Phi_{w_{\textsf{Y}}}=w_{\textsf{X}}\nabla\Phi_{w_{\textsf{X}}}-w_{\textsf{Y}}\nabla\Phi_{w_{\textsf{X}}}+w_{\textsf{Y}}\nabla\Phi_{w_{\textsf{X}}}-w_{\textsf{Y}}\nabla\Phi_{w_{\textsf{Y}}}.
		\end{equation*}
		Using that $\alpha+\beta+2\leq\beta$, $0<2\beta+1$,
		\begin{equation*}
			\begin{split}
				&\norm{w_{\textsf{X}}\nabla\Phi_{w_{\textsf{X}}}-w_{\textsf{Y}}\nabla\Phi_{w_{\textsf{Y}}}}_{C_{2\eta;\bar{T}}\mcC^{\alpha+\beta+2}}\\
				&\lesssim\norm{w_{\textsf{X}}-w_{\textsf{Y}}}_{C_{\eta;\bar{T}}\mcC^{\beta}}\norm{\nabla\Phi_{w_{\textsf{X}}}}_{C_{\eta;\bar{T}}\mcC^{\beta+1}}+\norm{w_{\textsf{Y}}}_{C_{\eta;\bar{T}}\mcC^{\beta}}\norm{\nabla\Phi_{w_{\textsf{X}}}-\nabla\Phi_{w_{\textsf{Y}}}}_{C_{\eta;\bar{T}}\mcC^{\beta+1}}.
			\end{split}
		\end{equation*}
		Using that $\alpha+\beta+2\leq\beta$, $2\alpha+5>0$ and $\beta+2\alpha+5>0$,
		\begin{equation*}
			\begin{split}
				\norm{w_{\textsf{X}}\nabla\Phi_{{\ty\!}_{\textsf{X}}}-w_{\textsf{Y}}\nabla\Phi_{{\ty\!}_{\textsf{Y}}}}_{C_{2\eta;\bar{T}}\mcC^{\alpha+\beta+2}}&\lesssim \bar{T}^{\eta}\norm{w_{\textsf{X}}-w_{\textsf{Y}}}_{C_{\eta;\bar{T}}\mcC^{\beta}}\norm{\nabla\Phi_{{\ty\!}_{\textsf{X}}}}_{C_{\bar{T}}\mcC^{2\alpha+5}}\\
				&\quad+\bar{T}^{\eta}\norm{w_{\textsf{Y}}}_{C_{\eta;\bar{T}}\mcC^{\beta}}\norm{\nabla\Phi_{{\ty\!}_{\textsf{X}}}-\nabla\Phi_{{\ty\!}_{\textsf{Y}}}}_{C_{\bar{T}}\mcC^{2\alpha+5}}.
			\end{split}
		\end{equation*}
		Using that $\alpha+\beta+2\leq2\alpha+4$, $\beta+1>0$ and $\beta+2\alpha+5>0$,
		\begin{equation*}
			\begin{split}
				\norm{{\ty\!}_{\textsf{X}}\nabla\Phi_{w_{\textsf{X}}}-{\ty\!}_{\textsf{Y}}\nabla\Phi_{w_{\textsf{Y}}}}_{C_{2\eta;\bar{T}}\mcC^{\alpha+\beta+2}}&\lesssim \bar{T}^{\eta}\norm{{\ty\!}_{\textsf{X}}-{\ty\!}_{\textsf{Y}}}_{C_{\bar{T}}\mcC^{2\alpha+4}}\norm{\nabla\Phi_{w_{\textsf{X}}}}_{C_{\eta;\bar{T}}\mcC^{\beta+1}}\\
				&\quad+\bar{T}^{\eta}\norm{{\ty\!}_{\textsf{Y}}}_{C_{\bar{T}}\mcC^{2\alpha+4}}\norm{\nabla\Phi_{w_{\textsf{X}}}-\nabla\Phi_{w_{\textsf{Y}}}}_{C_{\eta;\bar{T}}\mcC^{\beta+1}}.
			\end{split}
		\end{equation*}
		Using that $\alpha+\beta+2\leq2\alpha+4$ and $4\alpha+9>0$,
		\begin{equation*}
			\begin{split}
				\norm{{\ty\!}_{\textsf{X}}\nabla\Phi_{{\ty\!}_{\textsf{X}}}-{\ty\!}_{\textsf{Y}}\nabla\Phi_{{\ty\!}_{\textsf{Y}}}}_{C_{2\eta;\bar{T}}\mcC^{\alpha+\beta+2}}&\lesssim \bar{T}^{2\eta}\norm{{\ty\!}_{\textsf{X}}-{\ty\!}_{\textsf{Y}}}_{C_{\bar{T}}\mcC^{2\alpha+4}}\norm{\nabla\Phi_{{\ty\!}_{\textsf{X}}}}_{C_{\bar{T}}\mcC^{2\alpha+5}}\\
				&\quad+\bar{T}^{2\eta}\norm{{\ty\!}_{\textsf{Y}}}_{C_{\bar{T}}\mcC^{2\alpha+4}}\norm{\nabla\Phi_{{\ty\!}_{\textsf{X}}}-\nabla\Phi_{{\ty\!}_{\textsf{Y}}}}_{C_{\bar{T}}\mcC^{2\alpha+5}}.
			\end{split}
		\end{equation*}
		Using that $\alpha+\beta+2\leq3\alpha+6$,
		\begin{equation*}
			\norm{\tp_{\textsf{X}}-\tp_{\textsf{Y}}}_{C_{2\eta;\bar{T}}\mcC^{\alpha+\beta+2}}\lesssim \bar{T}^{2\eta}\norm{\tp_{\textsf{X}}-\tp_{\textsf{Y}}}_{C_{\bar{T}}\mcC^{3\alpha+6}}.
		\end{equation*}
		Using that $\alpha+\beta+2\leq 3\alpha+6$, $\alpha+2<0$,
		\begin{equation*}
			\begin{split}
				\norm{\nabla\Phi_{\ti_{\textsf{X}}}\pa{\ty\!}_{\textsf{X}}-\nabla\Phi_{\ti_{\textsf{Y}}}\pa{\ty\!}_{\textsf{Y}}}_{C_{2\eta;\bar{T}}\mcC^{\alpha+\beta+2}}&\lesssim \bar{T}^{2\eta}\norm{\nabla\Phi_{\ti_{\textsf{X}}}-\nabla\Phi_{\ti_{\textsf{Y}}}}_{C_{\bar{T}}\mcC^{\alpha+2}}\norm{{\ty\!}_{\textsf{X}}}_{C_{\bar{T}}\mcC^{2\alpha+4}}\\
				&\quad+\bar{T}^{2\eta}\norm{\nabla\Phi_{\ti_{\textsf{Y}}}}_{C_{\bar{T}}\mcC^{\alpha+2}}\norm{{\ty\!}_{\textsf{X}}-{\ty\!}_{\textsf{Y}}}_{C_{\bar{T}}\mcC^{2\alpha+4}}.
			\end{split}
		\end{equation*}
		Using that $\alpha+\beta+2\leq3\alpha+6$, $2\alpha+4<0$,
		\begin{equation*}
			\begin{split}
				\norm{{\ty\!}_{\textsf{X}}\pa\nabla\Phi_{\ti_{\textsf{X}}}-{\ty\!}_{\textsf{Y}}\pa\nabla\Phi_{\ti_{\textsf{Y}}}}_{C_{2\eta;\bar{T}}\mcC^{\alpha+\beta+2}}&\lesssim \bar{T}^{2\eta}\norm{{\ty\!}_{\textsf{X}}-{\ty\!}_{\textsf{Y}}}_{C_{\bar{T}}\mcC^{2\alpha+4}}\norm{\nabla\Phi_{\ti_{\textsf{X}}}}_{C_{\bar{T}}\mcC^{\alpha+2}}\\
				&\quad+\bar{T}^{2\eta}\norm{{\ty\!}_{\textsf{Y}}}_{C_{\bar{T}}\mcC^{2\alpha+4}}\norm{\nabla\Phi_{\ti_{\textsf{X}}}-\nabla\Phi_{\ti_{\textsf{Y}}}}_{C_{\bar{T}}\mcC^{\alpha+2}}.
			\end{split}
		\end{equation*}
		Using that $\alpha+\beta+2\leq3\alpha+6$, $\alpha+1<0$,
		\begin{equation*}
			\begin{split}
				\norm{\ti_{\textsf{X}}\pa\nabla\Phi_{{\ty\!}_{\textsf{X}}}-\ti_{\textsf{Y}}\pa\nabla\Phi_{{\ty\!}_{\textsf{Y}}}}_{C_{2\eta;\bar{T}}\mcC^{\alpha+\beta+2}}&\lesssim \bar{T}^{2\eta}\norm{\ti_{\textsf{X}}-\ti_{\textsf{Y}}}_{C_{\bar{T}}\mcC^{\alpha+1}}\norm{\nabla\Phi_{{\ty\!}_{\textsf{X}}}}_{C_{\bar{T}}\mcC^{2\alpha+5}}\\
				&\quad+\bar{T}^{2\eta}\norm{\ti_{\textsf{Y}}}_{C_{\bar{T}}\mcC^{\alpha+1}}\norm{\nabla\Phi_{{\ty\!}_{\textsf{X}}}-\nabla\Phi_{{\ty\!}_{\textsf{Y}}}}_{C_{\bar{T}}\mcC^{2\alpha+5}}.
			\end{split}
		\end{equation*}
		Using that $\beta<0$,
		\begin{equation*}
			\begin{split}
				\norm{w_{\textsf{X}}\pa\nabla\Phi_{\ti_{\textsf{X}}}-w_{\textsf{Y}}\pa\nabla\Phi_{\ti_{\textsf{Y}}}}_{C_{2\eta;\bar{T}}\mcC^{\alpha+\beta+2}}&\lesssim \bar{T}^{\eta}\norm{w_{\textsf{X}}-w_{\textsf{Y}}}_{C_{\eta;\bar{T}}\mcC^{\beta}}\norm{\nabla\Phi_{\ti_{\textsf{X}}}}_{C_{\bar{T}}\mcC^{\alpha+2}}\\
				&\quad+\bar{T}^{\eta}\norm{w_{\textsf{Y}}}_{C_{\eta;\bar{T}}\mcC^{\beta}}\norm{\nabla\Phi_{\ti_{\textsf{X}}}-\nabla\Phi_{\ti_{\textsf{Y}}}}_{C_{\bar{T}}\mcC^{\alpha+2}}.
			\end{split}
		\end{equation*}
		Using that $\alpha+2<0$,
		\begin{equation*}
			\begin{split}
				&\norm{\nabla\Phi_{\ti_{\textsf{X}}}\pa w_{\textsf{X}}-\nabla\Phi_{\ti_{\textsf{Y}}}\pa w_{\textsf{Y}}}_{C_{2\eta;\bar{T}}\mcC^{\alpha+\beta+2}}\\
				&\lesssim \bar{T}^{\eta}\norm{\nabla\Phi_{\ti_{\textsf{X}}}-\nabla\Phi_{\ti_{\textsf{Y}}}}_{C_{\bar{T}}\mcC^{\alpha+2}}\norm{w_{\textsf{X}}}_{C_{\eta;\bar{T}}\mcC^{\beta}}+\bar{T}^{\eta}\norm{\nabla\Phi_{\ti_{\textsf{Y}}}}_{C_{\bar{T}}\mcC^{\alpha+2}}\norm{w_{\textsf{X}}-w_{\textsf{Y}}}_{C_{\eta;\bar{T}}\mcC^{\beta}}.
			\end{split}
		\end{equation*}
		Using that $\alpha+1<0$,
		\begin{equation*}
			\begin{split}
				\norm{\ti_{\textsf{X}}\pa\nabla\Phi_{w_{\textsf{X}}}-\ti_{\textsf{Y}}\pa\nabla\Phi_{w_{\textsf{Y}}}}_{C_{2\eta;\bar{T}}\mcC^{\alpha+\beta+2}}&\lesssim \bar{T}^{\eta}\norm{\ti_{\textsf{X}}-\ti_{\textsf{Y}}}_{C_{\bar{T}}\mcC^{\alpha+1}}\norm{\nabla\Phi_{w_{\textsf{X}}}}_{C_{\eta;\bar{T}}\mcC^{\beta+1}}\\
				&\quad+\bar{T}^{\eta}\norm{\ti_{\textsf{Y}}}_{C_{\bar{T}}\mcC^{\alpha+1}}\norm{\nabla\Phi_{w_{\textsf{X}}}-\nabla\Phi_{w_{\textsf{Y}}}}_{C_{\eta;\bar{T}}\mcC^{\beta+1}}.
			\end{split}
		\end{equation*}
		Summing the bounds above, we estimate $\Omega^{\#}_{\textsf{X}}(\boldsymbol{w}_{\textsf{X}})-\Omega^{\#}_{\textsf{Y}}(\boldsymbol{w}_{\textsf{Y}})$ by
		\begin{equation}\label{eq:Omega_sharp_lipschitz_1}
			\begin{split}
				&\norm{\Omega^{\#}_{\textsf{X}}(\boldsymbol{w}_{\textsf{X}})-\Omega^{\#}_{\textsf{Y}}(\boldsymbol{w}_{\textsf{Y}})}_{C_{2\eta;\bar{T}}\mcC^{\alpha+\beta+2}}\\
				&\lesssim\norm{w_{\textsf{X}}-w_{\textsf{Y}}}_{C_{\eta;\bar{T}}\mcC^{\beta}}(\norm{w_{\textsf{X}}}_{C_{\eta;\bar{T}}\mcC^{\beta}}+\norm{w_{\textsf{Y}}}_{C_{\eta;\bar{T}}\mcC^{\beta}}+\norm{\mbX}_{\rksnoise{\alpha}{\kappa}_{\bar{T}}}+\norm{\mbY}_{\rksnoise{\alpha}{\kappa}_{\bar{T}}})\\
				&\quad+\norm{\mbX-\mbY}_{\rksnoise{\alpha}{\kappa}_{\bar{T}}}(\norm{w_{\textsf{Y}}}_{C_{\eta;\bar{T}}\mcC^{\beta}}+\norm{w_{\textsf{X}}}_{C_{\eta;\bar{T}}\mcC^{\beta}}+\norm{\mbX}_{\rksnoise{\alpha}{\kappa}_{\bar{T}}}+\norm{\mbY}_{\rksnoise{\alpha}{\kappa}_{\bar{T}}}+1)\\
				&\quad+\norm{\msP(\boldsymbol{w}_{\textsf{X}},\mbX)-\msP(\boldsymbol{w}_{\textsf{Y}},\mbY)}_{C_{2\eta;\bar{T}}\mcC^{\alpha+\beta+2}}.
			\end{split}
		\end{equation}
		We denote $\widetilde{w}^{\#}_{\textsf{X}}\defeq w^{\#}_{\textsf{X}}+\vdiv\mcI[w'_{\textsf{X}}\pa\ti_{\textsf{X}}]-w'_{\textsf{X}}\pa\nabla\mcI[\ti_{\textsf{X}}]$ and $\widetilde{\Phi}_{w_{\textsf{X}}}^{\#}\defeq\nabla\Phi_{w^{\#}_{\textsf{X}}}+\nabla\vdiv\Phi_{\mcI[w'_{\textsf{X}}\pa\ti_{\textsf{X}}]}-w'_{\textsf{X}}\pa\nabla^{2}\mcI[\Phi_{\ti_{\textsf{X}}}]$ and define $\widetilde{w}^{\#}_{\textsf{Y}}$, $\widetilde{\Phi}_{w_{\textsf{Y}}}^{\#}$ \emph{mutatis mutandis}. Next we consider the term $\msP(\boldsymbol{w}_{\textsf{X}},\mbX)-\msP(\boldsymbol{w}_{\textsf{Y}},\mbY)$. We decompose
		\begin{equation*}
			\begin{split}
				&\msP(\boldsymbol{w}_{\textsf{X}},\mbX)-\msP(\boldsymbol{w}_{\textsf{Y}},\mbY)\\
				&=\msC(w'_{\textsf{X}},\nabla\mcI[\ti_{\textsf{X}}],\nabla\Phi_{\ti_{\textsf{X}}})-\msC(w'_{\textsf{Y}},\nabla\mcI[\ti_{\textsf{Y}}],\nabla\Phi_{\ti_{\textsf{Y}}})+\msC(w'_{\textsf{X}},\nabla^{2}\mcI[\Phi_{\ti_{\textsf{X}}}],\ti_{\textsf{X}})-\msC(w'_{\textsf{Y}},\nabla^{2}\mcI[\Phi_{\ti_{\textsf{Y}}}],\ti_{\textsf{Y}})\\
				&\quad+\widetilde{w}^{\#}_{\textsf{X}}\re\nabla\Phi_{\ti_{\textsf{X}}}-\widetilde{w}^{\#}_{\textsf{Y}}\re\nabla\Phi_{\ti_{\textsf{Y}}}+\widetilde{\Phi}_{w_{\textsf{X}}}^{\#}\re\ti_{\textsf{X}}-\widetilde{\Phi}_{w_{\textsf{Y}}}^{\#}\re\ti_{\textsf{Y}}+w'_{\textsf{X}}\tc_{\textsf{X}}-w'_{\textsf{Y}}\tc_{\textsf{Y}}.
			\end{split}
		\end{equation*}
		We first estimate 
		\begin{equation}\label{eq:renorm_product_reg_bound}
			\norm{\msP(\boldsymbol{w}_{\textsf{X}},\mbX)-\msP(\boldsymbol{w}_{\textsf{Y}},\mbY)}_{C_{2\eta;\bar{T}}\mcC^{\alpha+\beta+2}}\lesssim\norm{\msP(\boldsymbol{w}_{\textsf{X}},\mbX)-\msP(\boldsymbol{w}_{\textsf{Y}},\mbY)}_{C_{2\eta;\bar{T}}\mcC^{2\alpha+4}}.
		\end{equation}
		We decompose the first term
		\begin{equation*}
			\begin{split}
				&\msC(w'_{\textsf{X}},\nabla\mcI[\ti_{\textsf{X}}],\nabla\Phi_{\ti_{\textsf{X}}})-\msC(w'_{\textsf{Y}},\nabla\mcI[\ti_{\textsf{Y}}],\nabla\Phi_{\ti_{\textsf{Y}}})\\
				&=\msC(w'_{\textsf{X}},\nabla\mcI[\ti_{\textsf{X}}],\nabla\Phi_{\ti_{\textsf{X}}})-\msC(w'_{\textsf{Y}},\nabla\mcI[\ti_{\textsf{X}}],\nabla\Phi_{\ti_{\textsf{X}}})\\
				&\quad+\msC(w'_{\textsf{Y}},\nabla\mcI[\ti_{\textsf{X}}],\nabla\Phi_{\ti_{\textsf{X}}})-\msC(w'_{\textsf{Y}},\nabla\mcI[\ti_{\textsf{Y}}],\nabla\Phi_{\ti_{\textsf{X}}})\\
				&\quad+\msC(w'_{\textsf{Y}},\nabla\mcI[\ti_{\textsf{Y}}],\nabla\Phi_{\ti_{\textsf{X}}})-\msC(w'_{\textsf{Y}},\nabla\mcI[\ti_{\textsf{Y}}],\nabla\Phi_{\ti_{\textsf{Y}}}).
			\end{split}
		\end{equation*}
		We estimate as in Lemma~\ref{lem:renormalised_products_estimates} using that $\beta'\in(0,1)$ and $2\alpha+4<0<2\alpha+4+\beta'$,
		\begin{equation*}
			\begin{split}
				&\norm{\msC(w'_{\textsf{X}},\nabla\mcI[\ti_{\textsf{X}}],\nabla\Phi_{\ti_{\textsf{X}}})-\msC(w'_{\textsf{Y}},\nabla\mcI[\ti_{\textsf{Y}}],\nabla\Phi_{\ti_{\textsf{Y}}})}_{C_{2\eta;\bar{T}}\mcC^{2\alpha+4}}\\
				&\lesssim_{\bar{T}}\norm{w'_{\textsf{X}}-w'_{\textsf{Y}}}_{C_{\eta;\bar{T}}\mcC^{\beta'}}\norm{\nabla\mcI[\ti_{\textsf{X}}]}_{C_{\bar{T}}\mcC^{\alpha+2}}\norm{\nabla\Phi_{\ti_{\textsf{X}}}}_{C_{\bar{T}}\mcC^{\alpha+2}}\\
				&\quad+\norm{w'_{\textsf{Y}}}_{C_{\eta;\bar{T}}\mcC^{\beta'}}\norm{\nabla\mcI[\ti_{\textsf{X}}]-\nabla\mcI[\ti_{\textsf{Y}}]}_{C_{\bar{T}}\mcC^{\alpha+2}}\norm{\nabla\Phi_{\ti_{\textsf{X}}}}_{C_{\bar{T}}\mcC^{\alpha+2}}\\
				&\quad+\norm{w'_{\textsf{Y}}}_{C_{\eta;\bar{T}}\mcC^{\beta'}}\norm{\nabla\mcI[\ti_{\textsf{Y}}]}_{C_{\bar{T}}\mcC^{\alpha+2}}\norm{\nabla\Phi_{\ti_{\textsf{X}}}-\nabla\Phi_{\ti_{\textsf{Y}}}}_{C_{\bar{T}}\mcC^{\alpha+2}}.
			\end{split}
		\end{equation*}
		Using that $\beta'\in(0,1)$ and $2\alpha+4<0<2\alpha+4+\beta'$,
		\begin{equation*}
			\begin{split}
				&\norm{\msC(w'_{\textsf{X}},\nabla^{2}\mcI[\Phi_{\ti_{\textsf{X}}}],\ti_{\textsf{X}})-\msC(w'_{\textsf{Y}},\nabla^{2}\mcI[\Phi_{\ti_{\textsf{Y}}}],\ti_{\textsf{Y}})}_{C_{2\eta;\bar{T}}\mcC^{2\alpha+4}}\\
				&\lesssim_{\bar{T}}\norm{w'_{\textsf{X}}-w'_{\textsf{Y}}}_{C_{\eta;\bar{T}}\mcC^{\beta'}}\norm{\nabla^{2}\mcI[\Phi_{\ti_{\textsf{X}}}]}_{C_{\bar{T}}\mcC^{\alpha+3}}\norm{\ti_{\textsf{X}}}_{C_{\bar{T}}\mcC^{\alpha+1}}\\
				&\quad+\norm{w'_{\textsf{Y}}}_{C_{\eta;\bar{T}}\mcC^{\beta'}}\norm{\nabla^{2}\mcI[\Phi_{\ti_{\textsf{X}}}]-\nabla^{2}\mcI[\Phi_{\ti_{\textsf{Y}}}]}_{C_{\bar{T}}\mcC^{\alpha+3}}\norm{\ti_{\textsf{X}}}_{C_{\bar{T}}\mcC^{\alpha+1}}\\
				&\quad+\norm{w'_{\textsf{Y}}}_{C_{\eta;\bar{T}}\mcC^{\beta'}}\norm{\nabla^{2}\mcI[\Phi_{\ti_{\textsf{Y}}}]}_{C_{\bar{T}}\mcC^{\alpha+3}}\norm{\ti_{\textsf{X}}-\ti_{\textsf{Y}}}_{C_{\bar{T}}\mcC^{\alpha+1}}.
			\end{split}
		\end{equation*}
		Using that $\beta^{\#}+\alpha+2>0$,
		\begin{equation*}
			\begin{split}
				&\norm{\widetilde{w}^{\#}_{\textsf{X}}\re\nabla\Phi_{\ti_{\textsf{X}}}-\widetilde{w}^{\#}_{\textsf{Y}}\re\nabla\Phi_{\ti_{\textsf{Y}}}}_{C_{2\eta;\bar{T}}\mcC^{2\alpha+4}}\\
				&\lesssim\norm{\widetilde{w}^{\#}_{\textsf{X}}-\widetilde{w}^{\#}_{\textsf{Y}}}_{C_{2\eta;\bar{T}}\mcC^{\beta^{\#}}}\norm{\nabla\Phi_{\ti_{\textsf{X}}}}_{C_{\bar{T}}\mcC^{\alpha+2}}+\norm{\widetilde{w}^{\#}_{\textsf{Y}}}_{C_{2\eta;\bar{T}}\mcC^{\beta^{\#}}}\norm{\nabla\Phi_{\ti_{\textsf{X}}}-\nabla\Phi_{\ti_{\textsf{Y}}}}_{C_{\bar{T}}\mcC^{\alpha+2}}.
			\end{split}
		\end{equation*}
		Using that $\beta^{\#}+\alpha+2>0$,
		\begin{equation*}
			\begin{split}
				&\norm{\widetilde{\Phi}^{\#}_{w_{\textsf{X}}}\re\ti_{\textsf{X}}-\widetilde{\Phi}^{\#}_{w_{\textsf{Y}}}\re\ti_{\textsf{Y}}}_{C_{2\eta;\bar{T}}\mcC^{2\alpha+4}}\\
				&\lesssim\norm{\widetilde{\Phi}^{\#}_{w_{\textsf{X}}}-\widetilde{\Phi}^{\#}_{w_{\textsf{Y}}}}_{C_{2\eta;\bar{T}}\mcC^{\beta^{\#}+1}}\norm{\ti_{\textsf{X}}}_{C_{\bar{T}}\mcC^{\alpha+1}}+\norm{\widetilde{\Phi}^{\#}_{w_{\textsf{Y}}}}_{C_{2\eta;\bar{T}}\mcC^{\beta^{\#}+1}}\norm{\ti_{\textsf{X}}-\ti_{\textsf{Y}}}_{C_{\bar{T}}\mcC^{\alpha+1}}.
			\end{split}
		\end{equation*}
		Using that $2\alpha+4+\beta'>0$,
		\begin{equation*}
			\norm{w'_{\textsf{X}}\tc_{\textsf{X}}-w'_{\textsf{Y}}\tc_{\textsf{Y}}}_{C_{2\eta;\bar{T}}\mcC^{2\alpha+4}}\lesssim_{\bar{T}}\norm{w'_{\textsf{X}}-w'_{\textsf{Y}}}_{C_{\eta;\bar{T}}\mcC^{\beta'}}\norm{\tc_{\textsf{X}}}_{C_{\bar{T}}\mcC^{2\alpha+4}}+\norm{w'_{\textsf{Y}}}_{C_{\eta;\bar{T}}\mcC^{\beta'}}\norm{\tc_{\textsf{X}}-\tc_{\textsf{Y}}}_{C_{\bar{T}}\mcC^{2\alpha+4}}.
		\end{equation*}
		Summing the bounds above, we obtain
		\begin{equation*}
			\begin{split}
				&\norm{\msP(\boldsymbol{w}_{\textsf{X}},\mbX)-\msP(\boldsymbol{w}_{\textsf{Y}},\mbY)}_{C_{2\eta;\bar{T}}\mcC^{2\alpha+4}}\\
				&\lesssim_{\bar{T}}\norm{w'_{\textsf{X}}-w'_{\textsf{Y}}}_{C_{\eta;\bar{T}}\mcC^{\beta'}}(\norm{\mbX}_{\rksnoise{\alpha}{\kappa}_{\bar{T}}}^{2}+\norm{\mbX}_{\rksnoise{\alpha}{\kappa}_{\bar{T}}})\\
				&\quad+\norm{\mbX-\mbY}_{\rksnoise{\alpha}{\kappa}_{\bar{T}}}(\norm{w'_{\textsf{Y}}}_{C_{\eta;\bar{T}}\mcC^{\beta'}}\norm{\mbX}_{\rksnoise{\alpha}{\kappa}_{\bar{T}}}+\norm{w'_{\textsf{Y}}}_{C_{\eta;{\bar{T}}}\mcC^{\beta'}}\norm{\mbY}_{\rksnoise{\alpha}{\kappa}_{\bar{T}}}\\
				&\multiquad[8]+\norm{\widetilde{w}^{\#}_{\textsf{Y}}}_{C_{2\eta;\bar{T}}\mcC^{\beta^{\#}}}+\norm{\widetilde{\Phi}^{\#}_{w_{\textsf{Y}}}}_{C_{2\eta;\bar{T}}\mcC^{\beta^{\#}+1}}+\norm{w'_{\textsf{Y}}}_{C_{\eta;\bar{T}}\mcC^{\beta'}})\\
				&\quad+\norm{\widetilde{w}^{\#}_{\textsf{X}}-\widetilde{w}^{\#}_{\textsf{Y}}}_{C_{2\eta;\bar{T}}\mcC^{\beta^{\#}}}\norm{\mbX}_{\rksnoise{\alpha}{\kappa}_{\bar{T}}}+\norm{\widetilde{\Phi}^{\#}_{w_{\textsf{X}}}-\widetilde{\Phi}^{\#}_{w_{\textsf{Y}}}}_{C_{2\eta;\bar{T}}\mcC^{\beta^{\#}+1}}\norm{\mbX}_{\rksnoise{\alpha}{\kappa}_{\bar{T}}}.
			\end{split}
		\end{equation*}
		Controlling $\norm{w'_{\textsf{X}}-w'_{\textsf{Y}}}_{C_{2\eta;\bar{T}}\mcC^{\beta'}}$ as in~\eqref{eq:Gubinelli_derivative_lipschitz}, we obtain
		\begin{equation}\label{eq:renorm_product_lipschitz_1}
			\begin{split}
				&\norm{\msP(\boldsymbol{w}_{\textsf{X}},\mbX)-\msP(\boldsymbol{w}_{\textsf{Y}},\mbY)}_{C_{2\eta;\bar{T}}\mcC^{2\alpha+4}}\\
				&\lesssim\norm{w_{\textsf{X}}-w_{\textsf{Y}}}_{C_{\eta;\bar{T}}\mcC^{\beta}}(\norm{\mbX}_{\rksnoise{\alpha}{\kappa}_{\bar{T}}}^{2}+\norm{\mbX}_{\rksnoise{\alpha}{\kappa}_{\bar{T}}})\\
				&\quad+\norm{\mbX-\mbY}_{\rksnoise{\alpha}{\kappa}_{\bar{T}}}(\norm{w'_{\textsf{Y}}}_{C_{\eta;\bar{T}}\mcC^{\beta'}}\norm{\mbX}_{\rksnoise{\alpha}{\kappa}_{\bar{T}}}+\norm{w'_{\textsf{Y}}}_{C_{\eta;\bar{T}}\mcC^{\beta'}}\norm{\mbY}_{\rksnoise{\alpha}{\kappa}_{\bar{T}}}+\norm{\widetilde{w}^{\#}_{\textsf{Y}}}_{C_{2\eta;\bar{T}}\mcC^{\beta^{\#}}}\\
				&\multiquad[11]+\norm{\widetilde{\Phi}^{\#}_{w_{\textsf{Y}}}}_{C_{2\eta;\bar{T}}\mcC^{\beta^{\#}+1}}+\norm{w'_{\textsf{Y}}}_{C_{\eta;\bar{T}}\mcC^{\beta'}}+\norm{\mbX}_{\rksnoise{\alpha}{\kappa}_{\bar{T}}}^{2}+\norm{\mbX}_{\rksnoise{\alpha}{\kappa}_{\bar{T}}})\\
				&\quad+\norm{\widetilde{w}^{\#}_{\textsf{X}}-\widetilde{w}^{\#}_{\textsf{Y}}}_{C_{2\eta;\bar{T}}\mcC^{\beta^{\#}}}\norm{\mbX}_{\rksnoise{\alpha}{\kappa}_{\bar{T}}}+\norm{\widetilde{\Phi}^{\#}_{w_{\textsf{X}}}-\widetilde{\Phi}^{\#}_{w_{\textsf{Y}}}}_{C_{2\eta;\bar{T}}\mcC^{\beta^{\#}+1}}\norm{\mbX}_{\rksnoise{\alpha}{\kappa}_{\bar{T}}}.
			\end{split}
		\end{equation}
		Next we estimate $\widetilde{w}^{\#}_{\textsf{X}}-\widetilde{w}^{\#}_{\textsf{Y}}$. We decompose
		\begin{equation}\label{eq:remainder_w_tilde_lipschitz}
			\begin{split}
				&\widetilde{w}^{\#}_{\textsf{X}}-\widetilde{w}^{\#}_{\textsf{Y}}\\
				&=\mcI[\vdiv(w'_{\textsf{X}}\pa\ti_{\textsf{X}})-w'_{\textsf{X}}\pa\nabla\ti_{\textsf{X}}]-\mcI[\vdiv(w'_{\textsf{Y}}\pa\ti_{\textsf{Y}})-w'_{\textsf{Y}}\pa\nabla\ti_{\textsf{Y}}]\\
				&\quad+\mcI[w'_{\textsf{X}}\pa\nabla\ti_{\textsf{X}}]-w'_{\textsf{X}}\pa\mcI[\nabla\ti_{\textsf{X}}]-\mcI[w'_{\textsf{Y}}\pa\nabla\ti_{\textsf{Y}}]+w'_{\textsf{Y}}\pa\mcI[\nabla\ti_{\textsf{Y}}]+w^{\#}_{\textsf{X}}-w^{\#}_{\textsf{Y}}.
			\end{split}
		\end{equation}
		For the first line of~\eqref{eq:remainder_w_tilde_lipschitz},
		\begin{equation*}
			\begin{split}
				&\mcI[\vdiv(w'_{\textsf{X}}\pa\ti_{\textsf{X}})-w'_{\textsf{X}}\pa\nabla\ti_{\textsf{X}}]-\mcI[\vdiv(w'_{\textsf{Y}}\pa\ti_{\textsf{Y}})-w'_{\textsf{Y}}\pa\nabla\ti_{\textsf{Y}}]\\
				&=\mcI[\vdiv((w'_{\textsf{X}}-w'_{\textsf{Y}})\pa\ti_{\textsf{X}})-(w'_{\textsf{X}}-w'_{\textsf{Y}})\pa\nabla\ti_{\textsf{X}}]+\mcI[\vdiv(w'_{\textsf{Y}}\pa(\ti_{\textsf{X}}-\ti_{\textsf{Y}}))-w'_{\textsf{Y}}\pa\nabla(\ti_{\textsf{X}}-\ti_{\textsf{Y}})].
			\end{split}
		\end{equation*}
		We estimate as in Lemma~\ref{lem:renormalised_products_estimates},
		\begin{equation*}
			\begin{split}
				&\norm{\mcI[\vdiv(w'_{\textsf{X}}\pa\ti_{\textsf{X}})-w'_{\textsf{X}}\pa\nabla\ti_{\textsf{X}}]-\mcI[\vdiv(w'_{\textsf{Y}}\pa\ti_{\textsf{Y}})-w'_{\textsf{Y}}\pa\nabla\ti_{\textsf{Y}}]}_{\msL_{2\eta;\bar{T}}^{\kappa}\mcC^{\beta^{\#}}}\\
				&\lesssim_{\bar{T}}\norm{w'_{\textsf{X}}-w'_{\textsf{Y}}}_{C_{\eta;\bar{T}}\mcC^{\beta'}}\norm{\ti_{\textsf{X}}}_{C_{\bar{T}}\mcC^{\alpha+1}}+\norm{w'_{\textsf{Y}}}_{C_{\eta;\bar{T}}\mcC^{\beta'}}\norm{\ti_{\textsf{X}}-\ti_{\textsf{Y}}}_{C_{\bar{T}}\mcC^{\alpha+1}}.
			\end{split}
		\end{equation*}
		For the second line of~\eqref{eq:remainder_w_tilde_lipschitz},
		\begin{equation*}
			\begin{split}
				&\mcI[w'_{\textsf{X}}\pa\nabla\ti_{\textsf{X}}]-w'_{\textsf{X}}\pa\mcI[\nabla\ti_{\textsf{X}}]-\mcI[w'_{\textsf{Y}}\pa\nabla\ti_{\textsf{Y}}]+w'_{\textsf{Y}}\pa\mcI[\nabla\ti_{\textsf{Y}}]\\
				&=\mcI[(w'_{\textsf{X}}-w'_{\textsf{Y}})\pa\nabla\ti_{\textsf{X}}]-(w'_{\textsf{X}}-w'_{\textsf{Y}})\pa\mcI[\nabla\ti_{\textsf{X}}]+\mcI[w'_{\textsf{Y}}\pa\nabla(\ti_{\textsf{X}}-\ti_{\textsf{Y}})]-w'_{\textsf{Y}}\pa\mcI[\nabla(\ti_{\textsf{X}}-\ti_{\textsf{Y}})].
			\end{split}
		\end{equation*}
		We estimate as in Lemma~\ref{lem:renormalised_products_estimates}, with $\gamma<2\kappa\wedge\beta'$ such that $\beta^{\#}\in(\gamma+\alpha,\gamma+\alpha+2)$, 
		\begin{equation*}
			\begin{split}
				&\norm{\mcI[w'_{\textsf{X}}\pa\nabla\ti_{\textsf{X}}]-w'_{\textsf{X}}\pa\mcI[\nabla\ti_{\textsf{X}}]-\mcI[w'_{\textsf{Y}}\pa\nabla\ti_{\textsf{Y}}]+w'_{\textsf{Y}}\pa\mcI[\nabla\ti_{\textsf{Y}}]}_{C_{2\eta;\bar{T}}\mcC^{\beta^{\#}}}\\
				&\lesssim_{\bar{T}}\norm{w'_{\textsf{X}}-w'_{\textsf{Y}}}_{\msL_{\eta;\bar{T}}^{\kappa}\mcC^{\gamma}}\norm{\nabla\ti_{\textsf{X}}}_{C_{\bar{T}}\mcC^{\alpha}}+\norm{w'_{\textsf{Y}}}_{\msL_{\eta;\bar{T}}^{\kappa}\mcC^{\gamma}}\norm{\nabla\ti_{\textsf{X}}-\nabla\ti_{\textsf{Y}}}_{C_{\bar{T}}\mcC^{\alpha}}\\
				&\lesssim\norm{w'_{\textsf{X}}-w'_{\textsf{Y}}}_{\msL_{\eta;\bar{T}}^{\kappa}\mcC^{\beta'}}\norm{\nabla\ti_{\textsf{X}}}_{C_{\bar{T}}\mcC^{\alpha}}+\norm{w'_{\textsf{Y}}}_{\msL_{\eta;\bar{T}}^{\kappa}\mcC^{\beta'}}\norm{\nabla\ti_{\textsf{X}}-\nabla\ti_{\textsf{Y}}}_{C_{\bar{T}}\mcC^{\alpha}}.
			\end{split}
		\end{equation*}
		Following Lemma~\ref{lem:Schauder} and Lemma~\ref{lem:commutator_heat_paraproduct}, we can bound the implicit constant uniformly in $\bar{T}\leq 1$, removing the time dependence of $\lesssim_{\bar{T}}$. We will make frequent use of this fact without explicitly pointing out every instance. This yields the estimate
		\begin{equation*}
			\begin{split}
				\norm{\widetilde{w}^{\#}_{\textsf{X}}-\widetilde{w}^{\#}_{\textsf{Y}}}_{C_{2\eta;\bar{T}}\mcC^{\beta^{\#}}}&\lesssim\norm{w'_{\textsf{X}}-w'_{\textsf{Y}}}_{\msL_{\eta;\bar{T}}^{\kappa}\mcC^{\beta'}}\norm{\ti_{\textsf{X}}}_{C_{\bar{T}}\mcC^{\alpha+1}}\\
				&\quad+\norm{\mbX-\mbY}_{\rksnoise{\alpha}{\kappa}_{\bar{T}}}\norm{w'_{\textsf{Y}}}_{\msL_{\eta;\bar{T}}^{\kappa}\mcC^{\beta'}}+\norm{w^{\#}_{\textsf{X}}-w^{\#}_{\textsf{Y}}}_{C_{2\eta;\bar{T}}\mcC^{\beta^{\#}}}.
			\end{split}
		\end{equation*}
		Plugging this into~\eqref{eq:renorm_product_lipschitz_1} yields
		\begin{equation}\label{eq:renorm_product_lipschitz_2}
			\begin{split}
				&\norm{\msP(\boldsymbol{w}_{\textsf{X}},\mbX)-\msP(\boldsymbol{w}_{\textsf{Y}},\mbY)}_{C_{2\eta;\bar{T}}\mcC^{2\alpha+4}}\\
				&\lesssim\norm{w_{\textsf{X}}-w_{\textsf{Y}}}_{C_{\eta;\bar{T}}\mcC^{\beta}}(\norm{\mbX}_{\rksnoise{\alpha}{\kappa}_{\bar{T}}}^{2}+\norm{\mbX}_{\rksnoise{\alpha}{\kappa}_{\bar{T}}})\\
				&\quad+\norm{\mbX-\mbY}_{\rksnoise{\alpha}{\kappa}_{\bar{T}}}\Bigl(\norm{w'_{\textsf{Y}}}_{\msL_{\eta;\bar{T}}^{\kappa}\mcC^{\beta'}}\norm{\mbX}_{\rksnoise{\alpha}{\kappa}_{\bar{T}}}+\norm{w'_{\textsf{Y}}}_{C_{\eta;\bar{T}}\mcC^{\beta'}}\norm{\mbY}_{\rksnoise{\alpha}{\kappa}_{\bar{T}}}+\norm{\widetilde{w}^{\#}_{\textsf{Y}}}_{C_{2\eta;\bar{T}}\mcC^{\beta^{\#}}}\\
				&\multiquad[11]+\norm{\widetilde{\Phi}^{\#}_{w_{\textsf{Y}}}}_{C_{2\eta;\bar{T}}\mcC^{\beta^{\#}+1}}+\norm{w'_{\textsf{Y}}}_{C_{\eta;\bar{T}}\mcC^{\beta'}}+\norm{\mbX}_{\rksnoise{\alpha}{\kappa}_{\bar{T}}}^{2}+\norm{\mbX}_{\rksnoise{\alpha}{\kappa}_{\bar{T}}}\Bigr)\\
				&\quad+\norm{w'_{\textsf{X}}-w'_{\textsf{Y}}}_{\msL_{\eta;\bar{T}}^{\kappa}\mcC^{\beta'}}\norm{\mbX}_{\rksnoise{\alpha}{\kappa}_{\bar{T}}}^{2}+\norm{w^{\#}_{\textsf{X}}-w^{\#}_{\textsf{Y}}}_{C_{2\eta;\bar{T}}\mcC^{\beta^{\#}}}\norm{\mbX}_{\rksnoise{\alpha}{\kappa}_{\bar{T}}}+\norm{\widetilde{\Phi}^{\#}_{w_{\textsf{X}}}-\widetilde{\Phi}^{\#}_{w_{\textsf{Y}}}}_{C_{2\eta;\bar{T}}\mcC^{\beta^{\#}+1}}\norm{\mbX}_{\rksnoise{\alpha}{\kappa}_{\bar{T}}}.
			\end{split}
		\end{equation}
		Next we estimate the remainder $\widetilde{\Phi}^{\#}_{w_{\textsf{X}}}-\widetilde{\Phi}^{\#}_{w_{\textsf{Y}}}$. We decompose
		\begin{equation}\label{eq:remainder_Phi_w_tilde_lipschitz}
			\begin{split}
				&\widetilde{\Phi}^{\#}_{w_{\textsf{X}}}-\widetilde{\Phi}^{\#}_{w_{\textsf{Y}}}\\
				&=\mcI[\nabla\vdiv(w'_{\textsf{X}}\pa\Phi_{\ti_{\textsf{X}}})-w'_{\textsf{X}}\pa\nabla^2\Phi_{\ti_{\textsf{X}}}]-\mcI[\nabla\vdiv(w'_{\textsf{Y}}\pa\Phi_{\ti_{\textsf{Y}}})-w'_{\textsf{Y}}\pa\nabla^2\Phi_{\ti_{\textsf{Y}}}]\\
				&\quad+\mcI[w'_{\textsf{X}}\pa\nabla^{2}\Phi_{\ti_{\textsf{X}}}]-w'_{\textsf{X}}\pa\mcI[\nabla^2\Phi_{\ti_{\textsf{X}}}]-\mcI[w'_{\textsf{Y}}\pa\nabla^2\Phi_{\ti_{\textsf{Y}}}]+w'_{\textsf{Y}}\pa\mcI[\nabla^{2}\Phi_{\ti_{\textsf{Y}}}]+\nabla\Phi^{\#}_{w_{\textsf{X}}}-\nabla\Phi^{\#}_{w_{\textsf{Y}}},
			\end{split}
		\end{equation}
		where $\Phi^{\#}_{w_{\textsf{X}}}=\vdiv\mcI[\Phi(D)(w'_{\textsf{X}}\pa\ti_{\textsf{X}})-w'_{\textsf{X}}\pa\Phi_{\ti_{\textsf{X}}}]+\Phi_{w^{\#}_{\textsf{X}}}$ and $\Phi^{\#}_{w_{\textsf{Y}}}$ is defined \emph{mutatis mutandis}. The first line of~\eqref{eq:remainder_Phi_w_tilde_lipschitz} can be decomposed as
		\begin{equation*}
			\begin{split}
				&\mcI[\nabla\vdiv(w'_{\textsf{X}}\pa\Phi_{\ti_{\textsf{X}}})-w'_{\textsf{X}}\pa\nabla^{2}\Phi_{\ti_{\textsf{X}}}]-\mcI[\nabla\vdiv(w'_{\textsf{Y}}\pa\Phi_{\ti_{\textsf{Y}}})+w'_{\textsf{Y}}\pa\nabla^2\Phi_{\ti_{\textsf{Y}}}]\\
				&=\mcI[\nabla\vdiv((w'_{\textsf{X}}-w'_{\textsf{Y}})\pa\Phi_{\ti_{\textsf{X}}})-(w'_{\textsf{X}}-w'_{\textsf{Y}})\pa\nabla^2\Phi_{\ti_{\textsf{X}}}]\\
				&\quad+\mcI[\nabla\vdiv(w'_{\textsf{Y}}\pa(\Phi_{\ti_{\textsf{X}}}-\Phi_{\ti_{\textsf{Y}}}))-w'_{\textsf{Y}}\pa\nabla^2(\Phi_{\ti_{\textsf{X}}}-\Phi_{\ti_{\textsf{X}}})].
			\end{split}
		\end{equation*}
		We estimate as in Lemma~\ref{lem:renormalised_products_estimates},
		\begin{equation*}
			\begin{split}
				&\norm{\mcI[\nabla\vdiv(w'_{\textsf{X}}\pa\Phi_{\ti_{\textsf{X}}})-w'_{\textsf{X}}\pa\nabla^{2}\Phi_{\ti_{\textsf{X}}}]-\mcI[\nabla\vdiv(w'_{\textsf{Y}}\pa\Phi_{\ti_{\textsf{Y}}})+w'_{\textsf{Y}}\pa\nabla^2\Phi_{\ti_{\textsf{Y}}}]}_{\msL_{2\eta;\bar{T}}^{\kappa}\mcC^{\beta^{\#}+1}}\\
				&\lesssim_{\bar{T}}\norm{w'_{\textsf{X}}-w'_{\textsf{Y}}}_{C_{\eta;\bar{T}}\mcC^{\beta'}}\norm{\Phi_{\ti_{\textsf{X}}}}_{C_{\bar{T}}\mcC^{\alpha+3}}+\norm{w'_{\textsf{Y}}}_{C_{\eta;\bar{T}}\mcC^{\beta'}}\norm{\Phi_{\ti_{\textsf{X}}}-\Phi_{\ti_{\textsf{Y}}}}_{C_{\bar{T}}\mcC^{\alpha+3}}.
			\end{split}
		\end{equation*}
		The second line of~\eqref{eq:remainder_Phi_w_tilde_lipschitz} can be decomposed as
		\begin{equation*}
			\begin{split}
				&\mcI[w'_{\textsf{X}}\pa\nabla^2\Phi_{\ti_{\textsf{X}}}]-w'_{\textsf{X}}\pa\mcI[\nabla^{2}\Phi_{\ti_{\textsf{X}}}]-\mcI[w'_{\textsf{Y}}\pa\nabla^2\Phi_{\ti_{\textsf{Y}}}]+w'_{\textsf{Y}}\pa\mcI[\nabla^{2}\Phi_{\ti_{\textsf{Y}}}]\\
				&=\mcI[(w'_{\textsf{X}}-w'_{\textsf{Y}})\pa\nabla^2\Phi_{\ti_{\textsf{X}}}]-(w'_{\textsf{X}}-w'_{\textsf{Y}})\pa\mcI[\nabla^{2}\Phi_{\ti_{\textsf{X}}}]\\
				&\quad+\mcI[w'_{\textsf{Y}}\pa(\nabla^2\Phi_{\ti_{\textsf{X}}}-\nabla^2\Phi_{\ti_{\textsf{Y}}})]-w'_{\textsf{Y}}\pa\mcI[\nabla^{2}\Phi_{\ti_{\textsf{X}}}-\nabla^2\Phi_{\ti_{\textsf{Y}}}].
			\end{split}
		\end{equation*}
		As in Lemma~\ref{lem:renormalised_products_estimates}, we estimate for some $\gamma<2\kappa\wedge\beta'$ such that $\beta^{\#}\in(\gamma+\alpha,\gamma+\alpha+2)$,
		\begin{equation*}
			\begin{split}
				&\norm{\mcI[w'_{\textsf{X}}\pa\nabla^2\Phi_{\ti_{\textsf{X}}}]-w'_{\textsf{X}}\pa\mcI[\nabla^{2}\Phi_{\ti_{\textsf{X}}}]-\mcI[w'_{\textsf{Y}}\pa\nabla^2\Phi_{\ti_{\textsf{Y}}}]+w'_{\textsf{Y}}\pa\mcI[\nabla^{2}\Phi_{\ti_{\textsf{Y}}}]}_{C_{2\eta;\bar{T}}\mcC^{\beta^{\#}+1}}\\
				&\lesssim_{\bar{T}}\norm{w'_{\textsf{X}}-w'_{\textsf{Y}}}_{\msL_{\eta;\bar{T}}^{\kappa}\mcC^{\gamma}}\norm{\nabla^{2}\Phi_{\ti_{\textsf{X}}}}_{C_{\bar{T}}\mcC^{\alpha+1}}+\norm{w'_{\textsf{Y}}}_{\msL_{\eta;\bar{T}}^{\kappa}\mcC^{\gamma}}\norm{\nabla^{2}\Phi_{\ti_{\textsf{X}}}-\nabla^{2}\Phi_{\ti_{\textsf{Y}}}}_{C_{\bar{T}}\mcC^{\alpha+1}}\\
				&\lesssim\norm{w'_{\textsf{X}}-w'_{\textsf{Y}}}_{\msL_{\eta;\bar{T}}^{\kappa}\mcC^{\beta'}}\norm{\nabla^{2}\Phi_{\ti_{\textsf{X}}}}_{C_{\bar{T}}\mcC^{\alpha+1}}+\norm{w'_{\textsf{Y}}}_{\msL_{\eta;\bar{T}}^{\kappa}\mcC^{\beta'}}\norm{\nabla^{2}\Phi_{\ti_{\textsf{X}}}-\nabla^{2}\Phi_{\ti_{\textsf{Y}}}}_{C_{\bar{T}}\mcC^{\alpha+1}}.
			\end{split}
		\end{equation*}
		The last term of~\eqref{eq:remainder_Phi_w_tilde_lipschitz} can be estimated by
		\begin{equation*}
			\begin{split}
				\norm{\nabla\Phi_{w_{\textsf{X}}}^{\#}-\nabla\Phi_{w_{\textsf{Y}}}^{\#}}_{C_{2\eta;\bar{T}}\mcC^{\beta^{\#}+1}}\lesssim\norm{\Phi_{w_{\textsf{X}}}^{\#}-\Phi_{w_{\textsf{Y}}}^{\#}}_{C_{2\eta;\bar{T}}\mcC^{\beta^{\#}+2}}.
			\end{split}
		\end{equation*}
		This yields the estimate
		\begin{equation}\label{eq:remainder_Phi_w_tilde_lipschitz_bound_1}
			\begin{split}
				&\norm{\widetilde{\Phi}^{\#}_{w_{\textsf{X}}}-\widetilde{\Phi}^{\#}_{w_{\textsf{Y}}}}_{C_{2\eta;\bar{T}}\mcC^{\beta^{\#}+1}}\\
				&\lesssim\norm{w'_{\textsf{X}}-w'_{\textsf{Y}}}_{\msL_{\eta;\bar{T}}^{\kappa}\mcC^{\beta'}}\norm{\mbX}_{\rksnoise{\alpha}{\kappa}_{\bar{T}}}+\norm{\mbX-\mbY}_{\rksnoise{\alpha}{\kappa}_{\bar{T}}}\norm{w'_{\textsf{Y}}}_{\msL_{\eta;\bar{T}}^{\kappa}\mcC^{\beta'}}+\norm{\Phi_{w_{\textsf{X}}}^{\#}-\Phi_{w_{\textsf{Y}}}^{\#}}_{C_{2\eta;\bar{T}}\mcC^{\beta^{\#}+2}}.
			\end{split}
		\end{equation}
		The remainder $\Phi_{w_{\textsf{X}}}^{\#}-\Phi_{w_{\textsf{Y}}}^{\#}$ can be decomposed as
		\begin{equation}\label{eq:remainder_Phi_w_lipschitz}
			\begin{split}
				\Phi^{\#}_{w_{\textsf{X}}}-\Phi^{\#}_{w_{\textsf{Y}}}&=\vdiv\mcI[\Phi(D)(w'_{\textsf{X}}\pa\ti_{\textsf{X}})-w'_{\textsf{X}}\pa\Phi_{\ti_{\textsf{X}}}]-\vdiv\mcI[\Phi(D)(w'_{\textsf{Y}}\pa\ti_{\textsf{Y}})-w'_{\textsf{Y}}\pa\Phi_{\ti_{\textsf{Y}}}]\\
				&\quad+\Phi_{w^{\#}_{\textsf{X}}}-\Phi_{w^{\#}_{\textsf{Y}}}.
			\end{split}
		\end{equation}
		The first line of~\eqref{eq:remainder_Phi_w_lipschitz} can be decomposed as
		\begin{equation*}
			\begin{split}
				&\vdiv\mcI[\Phi(D)(w'_{\textsf{X}}\pa\ti_{\textsf{X}})-w'_{\textsf{X}}\pa\Phi_{\ti_{\textsf{X}}}]-\vdiv\mcI[\Phi(D)(w'_{\textsf{Y}}\pa\ti_{\textsf{Y}})-w'_{\textsf{Y}}\pa\Phi_{\ti_{\textsf{Y}}}]\\
				&=\vdiv\mcI[\Phi(D)((w'_{\textsf{X}}-w'_{\textsf{Y}})\pa\ti_{\textsf{X}})-(w'_{\textsf{X}}-w'_{\textsf{Y}})\pa\Phi_{\ti_{\textsf{X}}}]\\
				&\quad+\vdiv\mcI[\Phi(D)(w'_{\textsf{Y}}\pa(\ti_{\textsf{X}}-\ti_{\textsf{Y}}))-w'_{\textsf{Y}}\pa(\Phi_{\ti_{\textsf{X}}}-\Phi_{\ti_{\textsf{Y}}})].
			\end{split}
		\end{equation*}
		We estimate as in Lemma~\ref{lem:renormalised_products_estimates},
		\begin{equation*}
			\begin{split}
				&\norm{\vdiv\mcI[\Phi(D)(w'_{\textsf{X}}\pa\ti_{\textsf{X}})-w'_{\textsf{X}}\pa\Phi_{\ti_{\textsf{X}}}]-\vdiv\mcI[\Phi(D)(w'_{\textsf{Y}}\pa\ti_{\textsf{Y}})-w'_{\textsf{Y}}\pa\Phi_{\ti_{\textsf{Y}}}]}_{\msL_{2\eta;\bar{T}}^{\kappa}\mcC^{\beta^{\#}+2}}\\
				&\lesssim_{\bar{T}}\norm{w'_{\textsf{X}}-w'_{\textsf{Y}}}_{C_{\eta;\bar{T}}\mcC^{\beta'}}\norm{\ti_{\textsf{X}}}_{C_{\bar{T}}\mcC^{\alpha+1}}+\norm{w'_{\textsf{Y}}}_{C_{\eta;\bar{T}}\mcC^{\beta'}}\norm{\ti_{\textsf{X}}-\ti_{\textsf{Y}}}_{C_{\bar{T}}\mcC^{\alpha+1}}.
			\end{split}
		\end{equation*}
		The second line of~\eqref{eq:remainder_Phi_w_lipschitz} can be estimated by
		\begin{equation*}
			\norm{\Phi_{w^{\#}_{\textsf{X}}}-\Phi_{w^{\#}_{\textsf{Y}}}}_{C_{2\eta;\bar{T}}\mcC^{\beta^{\#}+2}}\lesssim\norm{w^{\#}_{\textsf{X}}-w^{\#}_{\textsf{Y}}}_{C_{2\eta;\bar{T}}\mcC^{\beta^{\#}}}.
		\end{equation*}
		We obtain the bound
		\begin{equation}\label{eq:remainder_Phi_w_lipschitz_bound}
			\norm{\Phi^{\#}_{w_{\textsf{X}}}-\Phi^{\#}_{w_{\textsf{Y}}}}_{C_{2\eta;\bar{T}}\mcC^{\beta^{\#}+2}}\lesssim\norm{w'_{\textsf{X}}-w'_{\textsf{Y}}}_{C_{\eta;\bar{T}}\mcC^{\beta'}}\norm{\mbX}_{\rksnoise{\alpha}{\kappa}_{\bar{T}}}+\norm{\mbX-\mbY}_{\rksnoise{\alpha}{\kappa}_{\bar{T}}}\norm{w'_{\textsf{Y}}}_{C_{\eta;\bar{T}}\mcC^{\beta'}}+\norm{w^{\#}_{\textsf{X}}-w^{\#}_{\textsf{Y}}}_{C_{2\eta;\bar{T}}\mcC^{\beta^{\#}}}.
		\end{equation}
		We plug~\eqref{eq:remainder_Phi_w_lipschitz_bound} into~\eqref{eq:remainder_Phi_w_tilde_lipschitz_bound_1} to obtain
		\begin{equation}\label{eq:remainder_Phi_w_tilde_lipschitz_bound_2}
			\norm{\widetilde{\Phi}^{\#}_{w_{\textsf{X}}}-\widetilde{\Phi}^{\#}_{w_{\textsf{Y}}}}_{C_{2\eta;\bar{T}}\mcC^{\beta^{\#}+1}}\lesssim\norm{w'_{\textsf{X}}-w'_{\textsf{Y}}}_{\msL_{\eta;\bar{T}}^{\kappa}\mcC^{\beta'}}\norm{\mbX}_{\rksnoise{\alpha}{\kappa}_{\bar{T}}}+\norm{\mbX-\mbY}_{\rksnoise{\alpha}{\kappa}_{\bar{T}}}\norm{w'_{\textsf{Y}}}_{\msL_{\eta;\bar{T}}^{\kappa}\mcC^{\beta'}}+\norm{w^{\#}_{\textsf{X}}-w^{\#}_{\textsf{Y}}}_{C_{2\eta;\bar{T}}\mcC^{\beta^{\#}}}.
		\end{equation}
		We can now plug~\eqref{eq:remainder_Phi_w_tilde_lipschitz_bound_2} into~\eqref{eq:renorm_product_lipschitz_2} to obtain
		\begin{equation}\label{eq:renorm_product_lipschitz_3}
			\begin{split}
				&\norm{\msP(\boldsymbol{w}_{\textsf{X}},\mbX)-\msP(\boldsymbol{w}_{\textsf{Y}},\mbY)}_{C_{2\eta;\bar{T}}\mcC^{2\alpha+4}}\\
				&\lesssim\norm{w_{\textsf{X}}-w_{\textsf{Y}}}_{C_{\eta;\bar{T}}\mcC^{\beta}}(\norm{\mbX}_{\rksnoise{\alpha}{\kappa}_{\bar{T}}}^{2}+\norm{\mbX}_{\rksnoise{\alpha}{\kappa}_{\bar{T}}})\\
				&\quad+\norm{\mbX-\mbY}_{\rksnoise{\alpha}{\kappa}_{\bar{T}}}\Bigl(\norm{w'_{\textsf{Y}}}_{\msL_{\eta;\bar{T}}^{\kappa}\mcC^{\beta'}}\norm{\mbX}_{\rksnoise{\alpha}{\kappa}_{\bar{T}}}+\norm{w'_{\textsf{Y}}}_{C_{\eta;\bar{T}}\mcC^{\beta'}}\norm{\mbY}_{\rksnoise{\alpha}{\kappa}_{\bar{T}}}+\norm{\widetilde{w}^{\#}_{\textsf{Y}}}_{C_{2\eta;\bar{T}}\mcC^{\beta^{\#}}}\\
				&\multiquad[11]+\norm{\widetilde{\Phi}^{\#}_{w_{\textsf{Y}}}}_{C_{2\eta;\bar{T}}\mcC^{\beta^{\#}+1}}+\norm{w'_{\textsf{Y}}}_{C_{\eta;\bar{T}}\mcC^{\beta'}}+\norm{\mbX}_{\rksnoise{\alpha}{\kappa}_{\bar{T}}}^{2}+\norm{\mbX}_{\rksnoise{\alpha}{\kappa}_{\bar{T}}}\Bigr)\\
				&\quad+\norm{w'_{\textsf{X}}-w'_{\textsf{Y}}}_{\msL_{\eta;\bar{T}}^{\kappa}\mcC^{\beta'}}\norm{\mbX}_{\rksnoise{\alpha}{\kappa}_{\bar{T}}}^{2}+\norm{w^{\#}_{\textsf{X}}-w^{\#}_{\textsf{Y}}}_{C_{2\eta;\bar{T}}\mcC^{\beta^{\#}}}\norm{\mbX}_{\rksnoise{\alpha}{\kappa}_{\bar{T}}}.
			\end{split}
		\end{equation}
		We can now use~\eqref{eq:renorm_product_reg_bound} and plug~\eqref{eq:renorm_product_lipschitz_3} into~\eqref{eq:Omega_sharp_lipschitz_1} to obtain
		\begin{equation}\label{eq:Omega_sharp_lipschitz_2}
			\begin{split}
				&\norm{\Omega^{\#}_{\textsf{X}}(\boldsymbol{w}_{\textsf{X}})-\Omega^{\#}_{\textsf{Y}}(\boldsymbol{w}_{\textsf{Y}})}_{C_{2\eta;\bar{T}}\mcC^{\alpha+\beta+2}}\\
				&\lesssim\norm{w_{\textsf{X}}-w_{\textsf{Y}}}_{C_{\eta;\bar{T}}\mcC^{\beta}}(\norm{w_{\textsf{X}}}_{C_{\eta;\bar{T}}\mcC^{\beta}}+\norm{w_{\textsf{Y}}}_{C_{\eta;\bar{T}}\mcC^{\beta}}+\norm{\mbX}_{\rksnoise{\alpha}{\kappa}_{\bar{T}}}+\norm{\mbY}_{\rksnoise{\alpha}{\kappa}_{\bar{T}}}+\norm{\mbX}_{\rksnoise{\alpha}{\kappa}_{\bar{T}}}^{2})\\
				&\quad+\norm{\mbX-\mbY}_{\rksnoise{\alpha}{\kappa}_{\bar{T}}}\Bigl(\norm{w'_{\textsf{Y}}}_{\msL_{\eta;\bar{T}}^{\kappa}\mcC^{\beta'}}\norm{\mbX}_{\rksnoise{\alpha}{\kappa}_{\bar{T}}}+\norm{w'_{\textsf{Y}}}_{C_{\eta;\bar{T}}\mcC^{\beta'}}\norm{\mbY}_{\rksnoise{\alpha}{\kappa}_{\bar{T}}}+\norm{\widetilde{w}^{\#}_{\textsf{Y}}}_{C_{2\eta;\bar{T}}\mcC^{\beta^{\#}}}\\
				&\multiquad[8]+\norm{\widetilde{\Phi}^{\#}_{w_{\textsf{Y}}}}_{C_{2\eta;\bar{T}}\mcC^{\beta^{\#}+1}}+\norm{w'_{\textsf{Y}}}_{C_{\eta;\bar{T}}\mcC^{\beta'}}+\norm{\mbX}_{\rksnoise{\alpha}{\kappa}_{\bar{T}}}^{2}+\norm{w_{\textsf{Y}}}_{C_{\eta;\bar{T}}\mcC^{\beta}}\\
				&\multiquad[18]+\norm{w_{\textsf{X}}}_{C_{\eta;\bar{T}}\mcC^{\beta}}+\norm{\mbX}_{\rksnoise{\alpha}{\kappa}_{\bar{T}}}+\norm{\mbY}_{\rksnoise{\alpha}{\kappa}_{\bar{T}}}+1\Bigr)\\
				&\quad+\norm{w'_{\textsf{X}}-w'_{\textsf{Y}}}_{\msL_{\eta;\bar{T}}^{\kappa}\mcC^{\beta'}}\norm{\mbX}_{\rksnoise{\alpha}{\kappa}_{\bar{T}}}^{2}+\norm{w^{\#}_{\textsf{X}}-w^{\#}_{\textsf{Y}}}_{C_{2\eta;\bar{T}}\mcC^{\beta^{\#}}}\norm{\mbX}_{\rksnoise{\alpha}{\kappa}_{\bar{T}}}.
			\end{split}
		\end{equation}
		Plugging~\eqref{eq:Gubinelli_derivative_lipschitz} into~\eqref{eq:Omega_sharp_lipschitz_2}, we obtain
		\begin{equation}\label{eq:Omega_sharp_lipschitz_3}
			\begin{split}
				&\norm{\Omega^{\#}_{\textsf{X}}(\boldsymbol{w}_{\textsf{X}})-\Omega^{\#}_{\textsf{Y}}(\boldsymbol{w}_{\textsf{Y}})}_{C_{2\eta;\bar{T}}\mcC^{\alpha+\beta+2}}\\
				&\lesssim\norm{w_{\textsf{X}}-w_{\textsf{Y}}}_{\msL_{\eta;\bar{T}}^{\kappa}\mcC^{\beta}}(\norm{w_{\textsf{X}}}_{C_{\eta;\bar{T}}\mcC^{\beta}}+\norm{w_{\textsf{Y}}}_{C_{\eta;\bar{T}}\mcC^{\beta}}+\norm{\mbX}_{\rksnoise{\alpha}{\kappa}_{\bar{T}}}+\norm{\mbY}_{\rksnoise{\alpha}{\kappa}_{\bar{T}}}+\norm{\mbX}_{\rksnoise{\alpha}{\kappa}_{\bar{T}}}^{2})\\
				&\quad+\norm{\mbX-\mbY}_{\rksnoise{\alpha}{\kappa}_{\bar{T}}}\Bigl(\norm{w'_{\textsf{Y}}}_{\msL_{\eta;\bar{T}}^{\kappa}\mcC^{\beta'}}\norm{\mbX}_{\rksnoise{\alpha}{\kappa}_{\bar{T}}}+\norm{w'_{\textsf{Y}}}_{C_{\eta;\bar{T}}\mcC^{\beta'}}\norm{\mbY}_{\rksnoise{\alpha}{\kappa}_{\bar{T}}}+\norm{\widetilde{w}^{\#}_{\textsf{Y}}}_{C_{2\eta;\bar{T}}\mcC^{\beta^{\#}}}\\
				&\multiquad[8]+\norm{\widetilde{\Phi}^{\#}_{w_{\textsf{Y}}}}_{C_{2\eta;\bar{T}}\mcC^{\beta^{\#}+1}}+\norm{w'_{\textsf{Y}}}_{C_{\eta;\bar{T}}\mcC^{\beta'}}+\norm{\mbX}_{\rksnoise{\alpha}{\kappa}_{\bar{T}}}^{2}+\norm{w_{\textsf{Y}}}_{C_{\eta;\bar{T}}\mcC^{\beta}}\\
				&\multiquad[18]+\norm{w_{\textsf{X}}}_{C_{\eta;\bar{T}}\mcC^{\beta}}+\norm{\mbX}_{\rksnoise{\alpha}{\kappa}_{\bar{T}}}+\norm{\mbY}_{\rksnoise{\alpha}{\kappa}_{\bar{T}}}+1\Bigr)\\
				&\quad+\norm{w^{\#}_{\textsf{X}}-w^{\#}_{\textsf{Y}}}_{C_{2\eta;\bar{T}}\mcC^{\beta^{\#}}}\norm{\mbX}_{\rksnoise{\alpha}{\kappa}_{\bar{T}}}.
			\end{split}
		\end{equation}
		We can now plug~\eqref{eq:Omega_sharp_lipschitz_3} into~\eqref{eq:w_sharp_lipschitz_beta_sharp_1} to obtain
		\begin{equation}\label{eq:w_sharp_lipschitz_beta_sharp_2}
			\begin{split}
				&\norm{w^{\#}_{\textsf{X}}-w^{\#}_{\textsf{X}}}_{\msL_{2\eta;\bar{T}}^{\kappa}\mcC^{\beta^{\#}}}\\
				&\lesssim\norm{\rho^{\textsf{X}}_{0}-\rho^{\textsf{Y}}_{0}}_{\mcB_{p,q}^{\beta_{0}}}\\
				&\quad+(\bar{T}^{1-\frac{\beta^{\#}-\alpha-\beta-1}{2}}\vee \bar{T}^{1-\kappa})\norm{w_{\textsf{X}}-w_{\textsf{Y}}}_{\msL_{\eta;\bar{T}}^{\kappa}\mcC^{\beta}}\Bigl(\norm{w_{\textsf{X}}}_{C_{\eta;\bar{T}}\mcC^{\beta}}+\norm{w_{\textsf{Y}}}_{C_{\eta;\bar{T}}\mcC^{\beta}}\\
				&\multiquad[22]+\norm{\mbX}_{\rksnoise{\alpha}{\kappa}_{\bar{T}}}+\norm{\mbY}_{\rksnoise{\alpha}{\kappa}_{\bar{T}}}+\norm{\mbX}_{\rksnoise{\alpha}{\kappa}_{\bar{T}}}^{2}\Bigr)\\
				&\quad+\norm{\mbX-\mbY}_{\rksnoise{\alpha}{\kappa}_{\bar{T}}}\Bigl(\norm{w'_{\textsf{Y}}}_{\msL_{\eta;\bar{T}}^{\kappa}\mcC^{\beta'}}\norm{\mbX}_{\rksnoise{\alpha}{\kappa}_{\bar{T}}}+\norm{w'_{\textsf{Y}}}_{C_{\eta;\bar{T}}\mcC^{\beta'}}\norm{\mbY}_{\rksnoise{\alpha}{\kappa}_{\bar{T}}}+\norm{\widetilde{w}^{\#}_{\textsf{Y}}}_{C_{2\eta;\bar{T}}\mcC^{\beta^{\#}}}\\
				&\multiquad[8]+\norm{\widetilde{\Phi}^{\#}_{w_{\textsf{Y}}}}_{C_{2\eta;\bar{T}}\mcC^{\beta^{\#}+1}}+\norm{w'_{\textsf{Y}}}_{C_{\eta;\bar{T}}\mcC^{\beta'}}+\norm{\mbX}_{\rksnoise{\alpha}{\kappa}_{\bar{T}}}^{2}+\norm{w_{\textsf{Y}}}_{C_{\eta;\bar{T}}\mcC^{\beta}}\\
				&\multiquad[18]+\norm{w_{\textsf{X}}}_{C_{\eta;\bar{T}}\mcC^{\beta}}+\norm{\mbX}_{\rksnoise{\alpha}{\kappa}_{\bar{T}}}+\norm{\mbY}_{\rksnoise{\alpha}{\kappa}_{\bar{T}}}+1\Bigr)\\
				&\quad+(\bar{T}^{1-\frac{\beta^{\#}-\alpha-\beta-1}{2}}\vee \bar{T}^{1-\kappa})\norm{w^{\#}_{\textsf{X}}-w^{\#}_{\textsf{Y}}}_{C_{2\eta;\bar{T}}\mcC^{\beta^{\#}}}\norm{\mbX}_{\rksnoise{\alpha}{\kappa}_{\bar{T}}}.
			\end{split}
		\end{equation}
		Choosing $\bar{T}=\bar{T}(\alpha,\beta,\beta^{\#},\kappa,\eta,\norm{\mbX}_{\rksnoise{\alpha}{\kappa}_{T}})$ smaller, if necessary, we can absorb the last line of~\eqref{eq:w_sharp_lipschitz_beta_sharp_2} to obtain
		\begin{equation}\label{eq:w_sharp_lipschitz_beta_sharp_3}
			\begin{split}
				&\norm{w^{\#}_{\textsf{X}}-w^{\#}_{\textsf{X}}}_{\msL_{2\eta;\bar{T}}^{\kappa}\mcC^{\beta^{\#}}}\\
				&\lesssim\norm{\rho^{\textsf{X}}_{0}-\rho^{\textsf{Y}}_{0}}_{\mcB_{p,q}^{\beta_{0}}}\\
				&\quad+(\bar{T}^{1-\frac{\beta^{\#}-\alpha-\beta-1}{2}}\vee \bar{T}^{1-\kappa})\norm{w_{\textsf{X}}-w_{\textsf{Y}}}_{\msL_{\eta;\bar{T}}^{\kappa}\mcC^{\beta}}\Bigl(\norm{w_{\textsf{X}}}_{C_{\eta;\bar{T}}\mcC^{\beta}}+\norm{w_{\textsf{Y}}}_{C_{\eta;\bar{T}}\mcC^{\beta}}\\
				&\multiquad[22]+\norm{\mbX}_{\rksnoise{\alpha}{\kappa}_{\bar{T}}}+\norm{\mbY}_{\rksnoise{\alpha}{\kappa}_{\bar{T}}}+\norm{\mbX}_{\rksnoise{\alpha}{\kappa}_{\bar{T}}}^{2}\Bigr)\\
				&\quad+\norm{\mbX-\mbY}_{\rksnoise{\alpha}{\kappa}_{\bar{T}}}\Bigl(\norm{w'_{\textsf{Y}}}_{\msL_{\eta;\bar{T}}^{\kappa}\mcC^{\beta'}}\norm{\mbX}_{\rksnoise{\alpha}{\kappa}_{\bar{T}}}+\norm{w'_{\textsf{Y}}}_{C_{\eta;\bar{T}}\mcC^{\beta'}}\norm{\mbY}_{\rksnoise{\alpha}{\kappa}_{\bar{T}}}+\norm{\widetilde{w}^{\#}_{\textsf{Y}}}_{C_{2\eta;\bar{T}}\mcC^{\beta^{\#}}}\\
				&\multiquad[8]+\norm{\widetilde{\Phi}^{\#}_{w_{\textsf{Y}}}}_{C_{2\eta;\bar{T}}\mcC^{\beta^{\#}+1}}+\norm{w'_{\textsf{Y}}}_{C_{\eta;\bar{T}}\mcC^{\beta'}}+\norm{\mbX}_{\rksnoise{\alpha}{\kappa}_{\bar{T}}}^{2}+\norm{w_{\textsf{Y}}}_{C_{\eta;\bar{T}}\mcC^{\beta}}\\
				&\multiquad[18]+\norm{w_{\textsf{X}}}_{C_{\eta;\bar{T}}\mcC^{\beta}}+\norm{\mbX}_{\rksnoise{\alpha}{\kappa}_{\bar{T}}}+\norm{\mbY}_{\rksnoise{\alpha}{\kappa}_{\bar{T}}}+1\Bigr).
			\end{split}
		\end{equation}
		We can now plug~\eqref{eq:w_sharp_lipschitz_beta_sharp_3} into~\eqref{eq:Omega_sharp_lipschitz_3} to obtain
		\begin{equation}\label{eq:Omega_sharp_lipschitz_4}
			\begin{split}
				&\norm{\Omega^{\#}_{\textsf{X}}(\boldsymbol{w}_{\textsf{X}})-\Omega^{\#}_{\textsf{Y}}(\boldsymbol{w}_{\textsf{Y}})}_{C_{2\eta;\bar{T}}\mcC^{\alpha+\beta+2}}\\
				&\lesssim\norm{w_{\textsf{X}}-w_{\textsf{Y}}}_{\msL_{\eta;\bar{T}}^{\kappa}\mcC^{\beta}}(\norm{w_{\textsf{X}}}_{C_{\eta;\bar{T}}\mcC^{\beta}}+\norm{w_{\textsf{Y}}}_{C_{\eta;\bar{T}}\mcC^{\beta}}+\norm{\mbX}_{\rksnoise{\alpha}{\kappa}_{\bar{T}}}+\norm{\mbY}_{\rksnoise{\alpha}{\kappa}_{\bar{T}}}+\norm{\mbX}_{\rksnoise{\alpha}{\kappa}_{\bar{T}}}^{2})(1+\norm{\mbX}_{\rksnoise{\alpha}{\kappa}_{\bar{T}}})\\
				&\quad+\norm{\mbX-\mbY}_{\rksnoise{\alpha}{\kappa}_{\bar{T}}}\Bigl(\norm{w'_{\textsf{Y}}}_{\msL_{\eta;\bar{T}}^{\kappa}\mcC^{\beta'}}\norm{\mbX}_{\rksnoise{\alpha}{\kappa}_{\bar{T}}}+\norm{w'_{\textsf{Y}}}_{C_{\eta;\bar{T}}\mcC^{\beta'}}\norm{\mbY}_{\rksnoise{\alpha}{\kappa}_{\bar{T}}}+\norm{\widetilde{w}^{\#}_{\textsf{Y}}}_{C_{2\eta;\bar{T}}\mcC^{\beta^{\#}}}\\
				&\multiquad[8]+\norm{\widetilde{\Phi}^{\#}_{w_{\textsf{Y}}}}_{C_{2\eta;\bar{T}}\mcC^{\beta^{\#}+1}}+\norm{w'_{\textsf{Y}}}_{C_{\eta;\bar{T}}\mcC^{\beta'}}+\norm{\mbX}_{\rksnoise{\alpha}{\kappa}_{\bar{T}}}^{2}+\norm{w_{\textsf{Y}}}_{C_{\eta;\bar{T}}\mcC^{\beta}}\\
				&\multiquad[18]+\norm{w_{\textsf{X}}}_{C_{\eta;\bar{T}}\mcC^{\beta}}+\norm{\mbX}_{\rksnoise{\alpha}{\kappa}_{\bar{T}}}+\norm{\mbY}_{\rksnoise{\alpha}{\kappa}_{\bar{T}}}+1\Bigr)(1+\norm{\mbX}_{\rksnoise{\alpha}{\kappa}_{\bar{T}}})\\
				&\quad+\norm{\rho^{\textsf{X}}_{0}-\rho^{\textsf{Y}}_{0}}_{\mcB_{p,q}^{\beta_{0}}}\norm{\mbX}_{\rksnoise{\alpha}{\kappa}_{\bar{T}}}.
			\end{split}
		\end{equation}
		We can now plug~\eqref{eq:Omega_sharp_lipschitz_4} into~\eqref{eq:w_sharp_lipschitz_beta_1} to obtain
		\begin{equation}\label{eq:w_sharp_lipschitz_beta_2}
			\begin{split}
				&\norm{w^{\#}_{\textsf{X}}-w^{\#}_{\textsf{Y}}}_{\msL_{\eta;\bar{T}}^{\kappa}\mcC^{\beta}}\\
				&\lesssim\norm{\rho^{\textsf{X}}_{0}-\rho^{\textsf{Y}}_{0}}_{\mcB_{p,q}^{\beta_{0}}}(1+(\bar{T}^{\frac{\alpha+3}{2}-\eta}\vee \bar{T}^{1-\kappa-\eta})\norm{\mbX}_{\rksnoise{\alpha}{\kappa}_{\bar{T}}})\\
				&\quad+(\bar{T}^{\frac{\alpha+3}{2}-\eta}\vee \bar{T}^{1-\kappa-\eta})\norm{w_{\textsf{X}}-w_{\textsf{Y}}}_{\msL_{\eta;\bar{T}}^{\kappa}\mcC^{\beta}}\\
				&\quad\qquad\times(\norm{w_{\textsf{X}}}_{C_{\eta;\bar{T}}\mcC^{\beta}}+\norm{w_{\textsf{Y}}}_{C_{\eta;\bar{T}}\mcC^{\beta}}+\norm{\mbX}_{\rksnoise{\alpha}{\kappa}_{\bar{T}}}+\norm{\mbY}_{\rksnoise{\alpha}{\kappa}_{\bar{T}}}+\norm{\mbX}_{\rksnoise{\alpha}{\kappa}_{\bar{T}}}^{2})(1+\norm{\mbX}_{\rksnoise{\alpha}{\kappa}_{\bar{T}}})\\
				&\quad+(\bar{T}^{\frac{\alpha+3}{2}-\eta}\vee \bar{T}^{1-\kappa-\eta})\norm{\mbX-\mbY}_{\rksnoise{\alpha}{\kappa}_{\bar{T}}}\\
				&\quad\qquad\times\Bigl(\norm{w'_{\textsf{Y}}}_{\msL_{\eta;\bar{T}}^{\kappa}\mcC^{\beta'}}\norm{\mbX}_{\rksnoise{\alpha}{\kappa}_{\bar{T}}}+\norm{w'_{\textsf{Y}}}_{C_{\eta;\bar{T}}\mcC^{\beta'}}\norm{\mbY}_{\rksnoise{\alpha}{\kappa}_{\bar{T}}}+\norm{\widetilde{w}^{\#}_{\textsf{Y}}}_{C_{2\eta;\bar{T}}\mcC^{\beta^{\#}}}\\
				&\multiquad[8]+\norm{\widetilde{\Phi}^{\#}_{w_{\textsf{Y}}}}_{C_{2\eta;\bar{T}}\mcC^{\beta^{\#}+1}}+\norm{w'_{\textsf{Y}}}_{C_{\eta;\bar{T}}\mcC^{\beta'}}+\norm{\mbX}_{\rksnoise{\alpha}{\kappa}_{\bar{T}}}^{2}+\norm{w_{\textsf{Y}}}_{C_{\eta;\bar{T}}\mcC^{\beta}}\\
				&\multiquad[18]+\norm{w_{\textsf{X}}}_{C_{2\eta;\bar{T}}\mcC^{\beta}}+\norm{\mbX}_{\rksnoise{\alpha}{\kappa}_{\bar{T}}}+\norm{\mbY}_{\rksnoise{\alpha}{\kappa}_{\bar{T}}}+1\Bigr)(1+\norm{\mbX}_{\rksnoise{\alpha}{\kappa}_{\bar{T}}})
			\end{split}
		\end{equation}
		Choosing $\bar{T}$ smaller, if necessary, we can simplify~\eqref{eq:w_sharp_lipschitz_beta_2} to
		\begin{equation}\label{eq:w_sharp_lipschitz_beta_3}
			\begin{split}
				&\norm{w^{\#}_{\textsf{X}}-w^{\#}_{\textsf{Y}}}_{\msL_{\eta;\bar{T}}^{\kappa}\mcC^{\beta}}\\
				&\lesssim\norm{\rho^{\textsf{X}}_{0}-\rho^{\textsf{Y}}_{0}}_{\mcB_{p,q}^{\beta_{0}}}\\
				&\quad+(\bar{T}^{\frac{\alpha+3}{2}-\eta}\vee \bar{T}^{1-\kappa-\eta})\norm{w_{\textsf{X}}-w_{\textsf{Y}}}_{\msL_{\eta;\bar{T}}^{\kappa}\mcC^{\beta}}\\
				&\quad\qquad\times(\norm{w_{\textsf{X}}}_{C_{\eta;\bar{T}}\mcC^{\beta}}+\norm{w_{\textsf{Y}}}_{C_{\eta;\bar{T}}\mcC^{\beta}}+\norm{\mbX}_{\rksnoise{\alpha}{\kappa}_{\bar{T}}}+\norm{\mbY}_{\rksnoise{\alpha}{\kappa}_{\bar{T}}}+\norm{\mbX}_{\rksnoise{\alpha}{\kappa}_{\bar{T}}}^{2})(1+\norm{\mbX}_{\rksnoise{\alpha}{\kappa}_{\bar{T}}})\\
				&\quad+\norm{\mbX-\mbY}_{\rksnoise{\alpha}{\kappa}_{\bar{T}}}\Bigl(\norm{w'_{\textsf{Y}}}_{\msL_{\eta;\bar{T}}^{\kappa}\mcC^{\beta'}}\norm{\mbX}_{\rksnoise{\alpha}{\kappa}_{\bar{T}}}+\norm{w'_{\textsf{Y}}}_{C_{\eta;\bar{T}}\mcC^{\beta'}}\norm{\mbY}_{\rksnoise{\alpha}{\kappa}_{\bar{T}}}+\norm{\widetilde{w}^{\#}_{\textsf{Y}}}_{C_{2\eta;\bar{T}}\mcC^{\beta^{\#}}}\\
				&\multiquad[8]+\norm{\widetilde{\Phi}^{\#}_{w_{\textsf{Y}}}}_{C_{2\eta;\bar{T}}\mcC^{\beta^{\#}+1}}+\norm{w'_{\textsf{Y}}}_{C_{\eta;\bar{T}}\mcC^{\beta'}}+\norm{\mbX}_{\rksnoise{\alpha}{\kappa}_{\bar{T}}}^{2}+\norm{w_{\textsf{Y}}}_{C_{\eta;\bar{T}}\mcC^{\beta}}\\
				&\multiquad[18]+\norm{w_{\textsf{X}}}_{C_{2\eta;\bar{T}}\mcC^{\beta}}+\norm{\mbX}_{\rksnoise{\alpha}{\kappa}_{\bar{T}}}+\norm{\mbY}_{\rksnoise{\alpha}{\kappa}_{\bar{T}}}+1\Bigr)
			\end{split}
		\end{equation}
		We can now plug~\eqref{eq:w_sharp_lipschitz_beta_3} into~\eqref{eq:w_lipschitz_1} to obtain
		\begin{equation}\label{eq:w_lipschitz_2}
			\begin{split}
				&\norm{w_{\textsf{X}}-w_{\textsf{Y}}}_{\msL_{\eta;\bar{T}}^{\kappa}\mcC^{\beta}}\\
				&\lesssim\norm{\rho^{\textsf{X}}_{0}-\rho^{\textsf{Y}}_{0}}_{\mcB_{p,q}^{\beta_{0}}}\\
				&\quad+(\bar{T}^{\frac{\alpha+3}{2}-\eta}\vee \bar{T}^{1-\kappa-\eta})\norm{w_{\textsf{X}}-w_{\textsf{Y}}}_{\msL_{\eta;\bar{T}}^{\kappa}\mcC^{\beta}}\\
				&\quad\qquad\times(\norm{w_{\textsf{X}}}_{C_{\eta;\bar{T}}\mcC^{\beta}}+\norm{w_{\textsf{Y}}}_{C_{\eta;\bar{T}}\mcC^{\beta}}+\norm{\mbX}_{\rksnoise{\alpha}{\kappa}_{\bar{T}}}+\norm{\mbY}_{\rksnoise{\alpha}{\kappa}_{\bar{T}}}+\norm{\mbX}_{\rksnoise{\alpha}{\kappa}_{\bar{T}}}^{2})(1+\norm{\mbX}_{\rksnoise{\alpha}{\kappa}_{\bar{T}}})\\
				&\quad+\norm{\mbX-\mbY}_{\rksnoise{\alpha}{\kappa}_{\bar{T}}}\Bigl(\norm{w'_{\textsf{Y}}}_{\msL_{\eta;\bar{T}}^{\kappa}\mcC^{\beta'}}\norm{\mbX}_{\rksnoise{\alpha}{\kappa}_{\bar{T}}}+\norm{w'_{\textsf{Y}}}_{C_{\eta;\bar{T}}\mcC^{\beta'}}\norm{\mbY}_{\rksnoise{\alpha}{\kappa}_{\bar{T}}}+\norm{\widetilde{w}^{\#}_{\textsf{Y}}}_{C_{2\eta;\bar{T}}\mcC^{\beta^{\#}}}\\
				&\multiquad[8]+\norm{\widetilde{\Phi}^{\#}_{w_{\textsf{Y}}}}_{C_{2\eta;\bar{T}}\mcC^{\beta^{\#}+1}}+\norm{w'_{\textsf{Y}}}_{C_{\eta;\bar{T}}\mcC^{\beta'}}+\norm{\mbX}_{\rksnoise{\alpha}{\kappa}_{\bar{T}}}^{2}+\norm{w_{\textsf{Y}}}_{C_{\eta;\bar{T}}\mcC^{\beta}}\\
				&\multiquad[18]+\norm{w_{\textsf{X}}}_{C_{\eta;\bar{T}}\mcC^{\beta}}+\norm{\mbX}_{\rksnoise{\alpha}{\kappa}_{\bar{T}}}+\norm{\mbY}_{\rksnoise{\alpha}{\kappa}_{\bar{T}}}+1\Bigr).
			\end{split}
		\end{equation}
		Let $C$ be the implicit constant of~\eqref{eq:w_lipschitz_2}. We can now choose $\bar{T}$ smaller, if necessary, such that
		\begin{equation*}
			C(\bar{T}^{\frac{\alpha+3}{2}-\eta}\vee \bar{T}^{1-\kappa-\eta})\Bigl(\norm{w_{\textsf{X}}}_{C_{\eta;\bar{T}}\mcC^{\beta}}+\norm{w_{\textsf{Y}}}_{C_{\eta;\bar{T}}\mcC^{\beta}}+\norm{\mbX}_{\rksnoise{\alpha}{\kappa}_{\bar{T}}}+\norm{\mbY}_{\rksnoise{\alpha}{\kappa}_{\bar{T}}}+\norm{\mbX}_{\rksnoise{\alpha}{\kappa}_{\bar{T}}}^{2}\Bigr)(1+\norm{\mbX}_{\rksnoise{\alpha}{\kappa}_{\bar{T}}})<1.
		\end{equation*}
		This allows us to absorb the second term of~\eqref{eq:w_lipschitz_2} and simplify to
		\begin{equation*}
			\begin{split}
				&\norm{w_{\textsf{X}}-w_{\textsf{Y}}}_{\msL_{\eta;\bar{T}}^{\kappa}\mcC^{\beta}}\\
				&\lesssim\norm{\rho^{\textsf{X}}_{0}-\rho^{\textsf{Y}}_{0}}_{\mcB_{p,q}^{\beta_{0}}}\\
				&\quad+\norm{\mbX-\mbY}_{\rksnoise{\alpha}{\kappa}_{\bar{T}}}\Bigl(\norm{w'_{\textsf{Y}}}_{\msL_{\eta;\bar{T}}^{\kappa}\mcC^{\beta'}}\norm{\mbX}_{\rksnoise{\alpha}{\kappa}_{\bar{T}}}+\norm{w'_{\textsf{Y}}}_{C_{\eta;\bar{T}}\mcC^{\beta'}}\norm{\mbY}_{\rksnoise{\alpha}{\kappa}_{\bar{T}}}+\norm{\widetilde{w}^{\#}_{\textsf{Y}}}_{C_{2\eta;\bar{T}}\mcC^{\beta^{\#}}}\\
				&\multiquad[8]+\norm{\widetilde{\Phi}^{\#}_{w_{\textsf{Y}}}}_{C_{2\eta;\bar{T}}\mcC^{\beta^{\#}+1}}+\norm{w'_{\textsf{Y}}}_{C_{\eta;\bar{T}}\mcC^{\beta'}}+\norm{\mbX}_{\rksnoise{\alpha}{\kappa}_{\bar{T}}}^{2}+\norm{w_{\textsf{Y}}}_{C_{\eta;\bar{T}}\mcC^{\beta}}\\
				&\multiquad[18]+\norm{w_{\textsf{X}}}_{C_{\eta;\bar{T}}\mcC^{\beta}}+\norm{\mbX}_{\rksnoise{\alpha}{\kappa}_{\bar{T}}}+\norm{\mbY}_{\rksnoise{\alpha}{\kappa}_{\bar{T}}}+1\Bigr).
			\end{split}
		\end{equation*}
		Estimating $\norm{\widetilde{w}^{\#}_{\textsf{Y}}}_{C_{\eta;\bar{T}}\mcC^{\beta^{\#}}}$ and $\norm{\widetilde{\Phi}^{\#}_{w_{\textsf{Y}}}}_{C_{\eta;\bar{T}}\mcC^{\beta^{\#}}}$ as in Lemma~\ref{lem:renormalised_products_estimates}, we obtain the bound
		\begin{equation*}
			\norm{w_{\textsf{X}}-w_{\textsf{Y}}}_{\msL_{\eta;\bar{T}}^{\kappa}\mcC^{\beta}}\lesssim\norm{\rho^{\textsf{X}}_{0}-\rho^{\textsf{Y}}_{0}}_{\mcB_{p,q}^{\beta_{0}}}+\norm{\mbX-\mbY}_{\rksnoise{\alpha}{\kappa}_{\bar{T}}}(\norm{\boldsymbol{w}_{\textsf{X}}}_{\msD_{\bar{T}}}+\norm{\boldsymbol{w}_{\textsf{Y}}}_{\msD_{\bar{T}}}+\norm{\mbX}_{\rksnoise{\alpha}{\kappa}_{\bar{T}}}+\norm{\mbY}_{\rksnoise{\alpha}{\kappa}_{\bar{T}}}+1)^{2}.
		\end{equation*}
		Choosing $\bar{T}$ still smaller, if necessary, we can simplify~\eqref{eq:w_sharp_lipschitz_beta_3} to 
		\begin{equation*}
			\begin{split}
				\norm{w^{\#}_{\textsf{X}}-w^{\#}_{\textsf{Y}}}_{\msL_{\eta;\bar{T}}^{\kappa}\mcC^{\beta}}&\lesssim\norm{\rho^{\textsf{X}}_{0}-\rho^{\textsf{Y}}_{0}}_{\mcB_{p,q}^{\beta_{0}}}+\norm{w_{\textsf{X}}-w_{\textsf{Y}}}_{\msL_{\eta;\bar{T}}^{\kappa}\mcC^{\beta}}\\
				&\quad+\norm{\mbX-\mbY}_{\rksnoise{\alpha}{\kappa}_{\bar{T}}}(\norm{\boldsymbol{w}_{\textsf{X}}}_{\msD_{\bar{T}}}+\norm{\boldsymbol{w}_{\textsf{Y}}}_{\msD_{\bar{T}}}+\norm{\mbX}_{\rksnoise{\alpha}{\kappa}_{\bar{T}}}+\norm{\mbY}_{\rksnoise{\alpha}{\kappa}_{\bar{T}}}+1)^{2}.
			\end{split}
		\end{equation*}
		and we can simplify~\eqref{eq:w_sharp_lipschitz_beta_sharp_3} to 
		\begin{equation*}
			\begin{split}
				\norm{w^{\#}_{\textsf{X}}-w^{\#}_{\textsf{X}}}_{\msL_{2\eta;\bar{T}}^{\kappa}\mcC^{\beta^{\#}}}&\lesssim\norm{\rho^{\textsf{X}}_{0}-\rho^{\textsf{Y}}_{0}}_{\mcB_{p,q}^{\beta_{0}}}+\norm{w_{\textsf{X}}-w_{\textsf{Y}}}_{\msL_{\eta;\bar{T}}^{\kappa}\mcC^{\beta}}\\
				&\quad+\norm{\mbX-\mbY}_{\rksnoise{\alpha}{\kappa}_{\bar{T}}}(\norm{\boldsymbol{w}_{\textsf{X}}}_{\msD_{\bar{T}}}+\norm{\boldsymbol{w}_{\textsf{Y}}}_{\msD_{\bar{T}}}+\norm{\mbX}_{\rksnoise{\alpha}{\kappa}_{\bar{T}}}+\norm{\mbY}_{\rksnoise{\alpha}{\kappa}_{\bar{T}}}+1)^2.
			\end{split}
		\end{equation*}
		We obtain the bound
		\begin{equation*}
			\begin{split}
				&\max\{\norm{w_{\textsf{X}}-w_{\textsf{Y}}}_{\msL_{\eta;\bar{T}}^{\kappa}\mcC^{\beta}},\norm{w^{\#}_{\textsf{X}}-w^{\#}_{\textsf{Y}}}_{\msL_{\eta;\bar{T}}^{\kappa}\mcC^{\beta}},\norm{w^{\#}_{\textsf{X}}-w^{\#}_{\textsf{Y}}}_{\msL_{2\eta;\bar{T}}^{\kappa}\mcC^{\beta^{\#}}},\norm{\rho_{\textsf{X}}-\rho_{\textsf{Y}}}_{\msL_{\eta;\bar{T}}^{\kappa}\mcC^{\alpha+1}}\}\\
				&\lesssim\norm{\rho^{\textsf{X}}_{0}-\rho^{\textsf{Y}}_{0}}_{\mcB_{p,q}^{\beta_{0}}}+\norm{\mbX-\mbY}_{\rksnoise{\alpha}{\kappa}_{\bar{T}}}(\norm{\boldsymbol{w}_{\textsf{X}}}_{\msD_{\bar{T}}}+\norm{\boldsymbol{w}_{\textsf{Y}}}_{\msD_{\bar{T}}}+\norm{\mbX}_{\rksnoise{\alpha}{\kappa}_{\bar{T}}}+\norm{\mbY}_{\rksnoise{\alpha}{\kappa}_{\bar{T}}}+1)^{2}.
			\end{split}
		\end{equation*}
		This shows the local Lipschitz continuity of our paracontrolled solution up to time $\bar{T}$. 
	\end{proof}
\end{details}
\subsection{Maximal Time of Existence}\label{subsec:maximal_time_of_existence}
To establish a maximal time of existence for paracontrolled solutions to~\eqref{eq:gen_rKS_intro}, we iterate the construction of Subsection~\ref{subsec:local_well_posedness} (Theorem~\ref{thm:maximal_existence}). 

Let $(\alpha,p,q,\beta,\beta',\beta^{\#},\beta_{0},\kappa,\eta)$ satisfy~\eqref{eq:exponents}. Assume we have constructed a solution $\boldsymbol{w}\in\msD_{T_{1}}$ to~\eqref{eq:paracontrolled_equation} until time $T_{1}\in(0,T)$. Given the initial data $\rho_{T_{1}}=\ti_{T_{1}}+\ty_{T_{1}}+w_{T_{1}}\in\mcC^{\alpha+1}(\mbT^{2})$ and an enhancement $\mbX\in\rksnoise{\alpha}{\kappa}_{T}$, our paracontrolled Ansatz for the continuation is $\widetilde{\boldsymbol{w}}=(\widetilde{w},\widetilde{w}',\widetilde{w}^{\#})$, where for $t\geq0$,
\begin{equation}\label{eq:paracontrolled_equation_ctd_Ansatz}
	\widetilde{w}_{t}=\vdiv\mcI[\widetilde{w}'\pa\ti_{\shift{T_{1}}}]_{t}+\widetilde{w}^{\#}_{t}
\end{equation}
with Gubinelli derivative
\begin{equation}\label{eq:paracontrolled_equation_ctd_Gubi_der}
	\widetilde{w}'_{t}=\nabla\Phi_{\widetilde{w}_{t}}+\nabla\Phi_{\ty_{T_{1}+t}}
\end{equation}
and paracontrolled remainder
\begin{equation}\label{eq:paracontrolled_equation_ctd_remainder}
	\begin{split}
		\widetilde{w}^{\#}_{t}&=P_{t}(\rho_{T_{1}}-\ti_{T_{1}}-\ty_{T_{1}})+\vdiv\mcI[\widetilde{w}\nabla\Phi_{\widetilde{w}}]_{t}+\vdiv\mcI[\widetilde{w}\nabla\Phi_{\ty_{\shift{T_{1}}}}]_{t}+\vdiv\mcI[\ty_{\shift{T_{1}}}\nabla\Phi_{\widetilde{w}}]_{t}\\
		&\quad+\vdiv\mcI[\ty_{\shift{T_{1}}}\nabla\Phi_{\ty_{\shift{T_{1}}}}]_{t}+\vdiv\mcI[\tp_{\shift{T_{1}}}]_{t}+\vdiv\mcI[\nabla\Phi_{\ti_{\shift{T_{1}}}}\pa\ty_{\shift{T_{1}}}]_{t}\\
		&\quad+\vdiv\mcI[\ty_{\shift{T_{1}}}\pa\nabla\Phi_{\ti_{\shift{T_{1}}}}]_{t}+\vdiv\mcI[\ti_{\shift{T_{1}}}\pa\nabla\Phi_{\ty_{\shift{T_{1}}}}]_{t}+\vdiv\mcI[\widetilde{w}\pa\nabla\Phi_{\ti_{\shift{T_{1}}}}]_{t}\\
		&\quad+\vdiv\mcI[\nabla\Phi_{\ti_{\shift{T_{1}}}}\pa\widetilde{w}]_{t}+\vdiv\mcI[\ti_{\shift{T_{1}}}\pa\nabla\Phi_{\widetilde{w}}]_{t}+\vdiv\mcI[\msP_{T_{1}}(\widetilde{\boldsymbol{w}},\mbX)]_{t},
	\end{split}
\end{equation}
where the renormalised product $\msP_{T_{1}}(\widetilde{\boldsymbol{w}},\mbX)_{t}$ is given by
\begin{equation}\label{eq:paracontrolled_equation_ctd_ren_prod}
	\begin{split}
		\msP_{T_{1}}(\widetilde{\boldsymbol{w}},\mbX)_{t}&\defeq\msC(\widetilde{w}'_{t},\nabla\mcI[\ti_{\shift{T_{1}}}]_{t},\nabla\Phi_{\ti_{T_{1}+t}})+\msC(\widetilde{w}'_{t},\nabla^{2}\mcI[\Phi_{\ti_{\shift{T_{1}}}}]_{t},\ti_{T_{1}+t})\\
		&\quad+(\widetilde{w}^{\#}_{t}+\vdiv\mcI[\widetilde{w}'\pa\ti_{\shift{T_{1}}}]_{t}-\widetilde{w}'_{t}\pa\nabla\mcI[\ti_{\shift{T_{1}}}]_{t})\re\nabla\Phi_{\ti_{T_{1}+t}}\\
		&\quad+(\nabla\Phi_{\widetilde{w}_{t}^{\#}}+\nabla\vdiv\Phi_{\mcI[\widetilde{w}'\pa\ti_{\shift{T_{1}}}]_{t}}-\widetilde{w}'_{t}\pa\nabla^{2}\mcI[\Phi_{\ti_{\shift{T_{1}}}}]_{t})\re\ti_{T_{1}+t}\\
		&\quad+\widetilde{w}'_{t}\tc_{T_{1}+t}-\widetilde{w}'_{t}(\nabla P_{t}\mcI[\ti]_{T_{1}}\re\nabla\Phi_{\ti_{T_{1}+t}}+\nabla^{2}P_{t}\mcI[\Phi_{\ti}]_{T_{1}}\re\ti_{T_{1}+t}).
	\end{split}
\end{equation}
Hence, the only difference between the first iteration (Proposition~\ref{prop:w_local_existence}) and subsequent steps is the initial data $\rho_{T_{1}}-\ti_{T_{1}}-\ty_{T_{1}}$ in~\eqref{eq:paracontrolled_equation_ctd_remainder} and the additional term 
\begin{equation*}
	\widetilde{w}'_{t}(\nabla P_{t}\mcI[\ti]_{T_{1}}\re\nabla\Phi_{\ti_{T_{1}+t}}+\nabla^{2}P_{t}\mcI[\Phi_{\ti}]_{T_{1}}\re\ti_{T_{1}+t})
\end{equation*}
in the renormalised product~\eqref{eq:paracontrolled_equation_ctd_ren_prod}.
\begin{remark}\label{rem:iteration_par_remainder_discontinuous}
	The paracontrolled triple $\widetilde{\boldsymbol{w}}$ of the second iteration step is related to the paracontrolled triple $\boldsymbol{w}$ of the first iteration step by
	\begin{equation*}
		\widetilde{w}_{t}=w_{T_{1}+t},\qquad\widetilde{w}'_{t}=w'_{T_{1}+t},\qquad\widetilde{w}^{\#}_{t}=w^{\#}_{T_{1}+t}+P_{t}(w_{T_{1}}-w^{\#}_{T_{1}})
	\end{equation*}
	(for those $t$ such that both $\widetilde{\boldsymbol{w}}_{t}$ and $\boldsymbol{w}_{T_{1}+t}$ exist.) The additional terms in $\widetilde{w}^{\#}$ and $\msP_{T_{1}}(\widetilde{\boldsymbol{w}},\mbX)$ ensure that the initial data in~\eqref{eq:paracontrolled_equation_ctd_remainder} can be expressed in terms of $\rho_{T_{1}}$, $\ti_{T_{1}}$ and $\ty_{T_{1}}$. Since we assumed $\mbX\in\rksnoise{\alpha}{\kappa}_{T}$, this allows us to restart the equation as long as $\rho_{T_{1}}\in\mcC^{\alpha+1}(\mbT^{2})$. The price to pay is that the dynamics of the paracontrolled remainder changes at $T_{1}$, hence the triple will no longer be continuous on the full time interval of existence. However, this will not be an issue, as we are primarily interested in the dynamics of the solution
	\begin{equation*}
		\widetilde{\rho}_{t}\defeq\ti_{T_{1}+t}+\ty_{T_{1}+t}+\vdiv\mcI[\widetilde{w}'\pa\ti_{T_{1}+\place}]_{t}+\widetilde{w}^{\#}_{t}=\ti_{T_{1}+t}+\ty_{T_{1}+t}+\vdiv\mcI[w'\pa\ti]_{T_{1}+t}+w^{\#}_{T_{1}+t}=\rho_{T_{1}+t}
	\end{equation*}
	which will still be continuous on the full time interval.
\end{remark}
\begin{details}
	\paragraph{Proof that~\eqref{eq:paracontrolled_equation_ctd_Ansatz}--\eqref{eq:paracontrolled_equation_ctd_ren_prod} is compatible with Definition~\ref{def:paracontrolled_solution}.}
	To show that the iteration~\eqref{eq:paracontrolled_equation_ctd_Ansatz}--\eqref{eq:paracontrolled_equation_ctd_ren_prod} produces a paracontrolled solution in the sense of Definition~\ref{def:paracontrolled_solution}, we will repeatedly apply the Markov property of the heat equation
	\begin{equation}\label{eq:Markov_heat_eq}
		(\mcI[f]\circ\Theta_{T_{1}})(t)=P_{t}\mcI[f]_{T_{1}}+\mcI[f\circ\Theta_{T_{1}}]_{t}.
	\end{equation}
	Let $\rho$ be a paracontrolled solution in the sense of Definition~\ref{def:paracontrolled_solution}. Let $T_{1}>0$, $t\geq0$ and define $\widetilde{\rho}_{t}=\rho_{T_{1}+t}$, it follows that
	\begin{equation*}
		\widetilde{\rho}_{t}=\ti_{T_{1}+t}+\ty_{T_{1}+t}+\vdiv\mcI[w'\pa\ti]_{T_{1}+t}+w^{\#}_{T_{1}+t}.
	\end{equation*}
	We use the Markov property~\eqref{eq:Markov_heat_eq} to write
	\begin{equation*}
		\widetilde{\rho}_{t}=\ti_{T_{1}+t}+\ty_{T_{1}+t}+\vdiv\mcI[w'_{\shift{T_{1}}}\pa\ti_{\shift{T_{1}}}]_{t}+P_{t}\vdiv\mcI[w'\pa\ti]_{T_{1}}+w^{\#}_{T_{1}+t}.
	\end{equation*}
	Define $\widetilde{w}_{t}=w_{T_{1}+t}$, $\widetilde{w}'_{t}=w'_{T_{1}+t}$ and $\widetilde{w}^{\#}_{t}=w^{\#}_{T_{1}+t}+P_{t}\vdiv\mcI[w'\pa\ti]_{T_{1}}$. It follows that $\widetilde{w}'_{t}=\nabla\Phi_{w_{T_{1}+t}}+\nabla\Phi_{\ty_{T_{1}+t}}=\nabla\Phi_{\widetilde{w}_{t}}+\nabla\Phi_{\ty_{T_{1}+t}}$. The remainder $w^{\#}_{T_{1}+t}$ solves by Markov's property~\eqref{eq:Markov_heat_eq},
	\begin{equation*}
		\begin{split}
			w^{\#}_{T_{1}+t}&=P_{T_{1}+t}\rho_{0}+\vdiv\mcI[\Omega^{\#}(\boldsymbol{w})]_{T_{1}+t}\\
			&=P_{t}(P_{T_{1}}\rho_{0}+\vdiv\mcI[\Omega^{\#}(\boldsymbol{w})]_{T_{1}})+\vdiv\mcI[\Omega^{\#}(\boldsymbol{w})\circ\Theta_{T_{1}}]_{t}\\
			&=P_{t}(w_{T_{1}}-\vdiv\mcI[w'\pa\ti]_{T_{1}})+\vdiv\mcI[\widetilde{w}\nabla\Phi_{\widetilde{w}}]_{t}+\vdiv\mcI[\widetilde{w}\nabla\Phi_{\ty_{\shift{T_{1}}}}]_{t}+\vdiv\mcI[\ty_{\shift{T_{1}}}\nabla\Phi_{\widetilde{w}}]_{t}\\
			&\quad+\vdiv\mcI[\ty_{\shift{T_{1}}}\nabla\Phi_{\ty_{\shift{T_{1}}}}]_{t}+\vdiv\mcI[\tp_{\shift{T_{1}}}]_{t}+\vdiv\mcI[\nabla\Phi_{\ti_{\shift{T_{1}}}}\pa\ty_{\shift{T_{1}}}]_{t}+\vdiv\mcI[\ty_{\shift{T_{1}}}\pa\nabla\Phi_{\ti_{\shift{T_{1}}}}]_{t}\\
			&\quad+\vdiv\mcI[\ti_{\shift{T_{1}}}\pa\nabla\Phi_{\ty_{\shift{T_{1}}}}]_{t}+\vdiv\mcI[\widetilde{w}\pa\nabla\Phi_{\ti_{\shift{T_{1}}}}]_{t}+\vdiv\mcI[\nabla\Phi_{\ti_{\shift{T_{1}}}}\pa\widetilde{w}]_{t}+\vdiv\mcI[\ti_{\shift{T_{1}}}\pa\nabla\Phi_{\widetilde{w}}]_{t}\\
			&\quad+\vdiv\mcI[\msP(\boldsymbol{w},\mbX)\circ\Theta_{T_{1}}]_{t},
		\end{split}
	\end{equation*}
	where
	\begin{equation*}
		\begin{split}
			\msP(\boldsymbol{w},\mbX)\circ\Theta_{T_{1}}(t)&=\msP(\boldsymbol{w},\mbX)_{T_{1}+t}\\
			&=\msC(\widetilde{w}'_{t},\nabla\mcI[\ti]_{T_{1}+t},\nabla\Phi_{\ti_{T_{1}+t}})+\msC(\widetilde{w}'_{t},\nabla^{2}\mcI[\Phi_{\ti}]_{T_{1}+t},\ti_{T_{1}+t})\\
			&\quad+(w^{\#}_{T_{1}+t}+\vdiv\mcI[w'\pa\ti]_{T_{1}+t}-\widetilde{w}'_{t}\pa\nabla\mcI[\ti]_{T_{1}+t})\re\nabla\Phi_{\ti_{T_{1}+1}}\\
			&\quad+(\nabla\Phi_{w^{\#}_{T_{1}+t}}+\nabla\vdiv\Phi_{\mcI[w'\pa\ti]_{T_{1}+t}}-\widetilde{w}'_{t}\pa\nabla^{2}\mcI[\Phi_{\ti}]_{T_{1}+t})\re\ti_{T_{1}+t}+\widetilde{w}'_{t}\tc_{T_{1}+t}.
		\end{split}
	\end{equation*}
	By the Markov property~\eqref{eq:Markov_heat_eq},
	\begin{equation*}
		\begin{split}
			&\msC(\widetilde{w}'_{t},\nabla\mcI[\ti]_{T_{1}+t},\nabla\Phi_{\ti_{T_{1}+t}})+\msC(\widetilde{w}'_{t},\nabla^{2}\mcI[\Phi_{\ti}]_{T_{1}+t},\ti_{T_{1}+t})\\
			&=\msC(\widetilde{w}'_{t},\nabla\mcI[\ti_{\shift{T_{1}}}]_{t},\nabla\Phi_{\ti_{T_{1}+t}})+\msC(\widetilde{w}'_{t},\nabla^{2}\mcI[\Phi_{\ti_{\shift{T_{1}}}}]_{t},\ti_{T_{1}+t})\\
			&\quad+\msC(\widetilde{w}'_{t},P_{t}\nabla\mcI[\ti]_{T_{1}},\nabla\Phi_{\ti_{T_{1}+t}})+\msC(\widetilde{w}'_{t},P_{t}\nabla^{2}\mcI[\Phi_{\ti}]_{T_{1}},\ti_{T_{1}+t}),
		\end{split}
	\end{equation*}
	\begin{equation*}
		\begin{split}
			&\vdiv\mcI[w'\pa\ti]_{T_{1}+t}-\widetilde{w}'_{t}\pa\nabla\mcI[\ti]_{T_{1}+t}\\
			&=\vdiv\mcI[\widetilde{w}'\pa\ti_{\shift{T_{1}}}]_{t}-\widetilde{w}'_{t}\pa\nabla\mcI[\ti_{\shift{T_{1}}}]_{t}+P_{t}\vdiv\mcI[w'\pa\ti]_{T_{1}}-\widetilde{w}'_{t}\pa P_{t}\nabla\mcI[\ti]_{T_{1}},
		\end{split}
	\end{equation*}
	\begin{equation*}
		\begin{split}
			&\nabla\vdiv\Phi_{\mcI[w'\pa\ti]_{T_{1}+t}}-\widetilde{w}'_{t}\pa\nabla^{2}\mcI[\Phi_{\ti}]_{T_{1}+t}\\
			&=\nabla\vdiv\Phi_{\mcI[\widetilde{w}'\pa\ti_{\shift{T_{1}}}]_{t}}-\widetilde{w}'_{t}\pa\nabla^{2}\mcI[\Phi_{\ti_{\shift{T_{1}}}}]_{t}+P_{t}\nabla\vdiv\Phi_{\mcI[w'\pa\ti]_{T_{1}}}-\widetilde{w}'_{t}\pa P_{t}\nabla^{2}\mcI[\Phi_{\ti}]_{T_{1}}
		\end{split}
	\end{equation*}
	and, using the identity $w^{\#}_{T_{1}+t}=\widetilde{w}^{\#}_{t}- P_{t}\vdiv\mcI[w'\pa\ti]_{T_{1}}$, we obtain
	\begin{equation*}
		\begin{split}
			&\msP(\boldsymbol{w},\mbX)_{T_{1}+t}\\
			&=\msC(\widetilde{w}'_{t},\nabla\mcI[\ti_{\shift{T_{1}}}]_{t},\nabla\Phi_{\ti_{T_{1}+t}})+\msC(\widetilde{w}'_{t},\nabla^{2}\mcI[\Phi_{\ti_{\shift{T_{1}}}}]_{t},\ti_{T_{1}+t})\\
			&\quad+(\widetilde{w}^{\#}_{t}+\vdiv\mcI[\widetilde{w}'\pa\ti_{\shift{T_{1}}}]_{t}-\widetilde{w}'_{t}\pa\nabla\mcI[\ti_{\shift{T_{1}}}]_{t})\re\nabla\Phi_{\ti_{T_{1}+1}}\\
			&\quad+(\nabla\Phi_{\widetilde{w}^{\#}_{t}}+\nabla\vdiv\Phi_{\mcI[\widetilde{w}'\pa\ti_{\shift{T_{1}}}]_{t}}-\widetilde{w}'_{t}\pa\nabla^{2}\mcI[\Phi_{\ti_{\shift{T_{1}}}}]_{t})\re\ti_{T_{1}+t}+\widetilde{w}'_{t}\tc_{T_{1}+t}\\
			&\quad+\msC(\widetilde{w}'_{t},P_{t}\nabla\mcI[\ti]_{T_{1}},\nabla\Phi_{\ti_{T_{1}+t}})+\msC(\widetilde{w}'_{t},P_{t}\nabla^{2}\mcI[\Phi_{\ti}]_{T_{1}},\ti_{T_{1}+t})\\
			&\quad+(-P_{t}\vdiv\mcI[w'\pa\ti]_{T_{1}}+P_{t}\vdiv\mcI[w'\pa\ti]_{T_{1}}-\widetilde{w}'_{t}\pa P_{t}\nabla\mcI[\ti]_{T_{1}})\re\nabla\Phi_{\ti_{T_{1}+1}}\\
			&\quad+(-\nabla\Phi_{P_{t}\vdiv\mcI[w'\pa\ti]_{T_{1}}}+P_{t}\nabla\vdiv\Phi_{\mcI[w'\pa\ti]_{T_{1}}}-\widetilde{w}'_{t}\pa P_{t}\nabla^{2}\mcI[\Phi_{\ti}]_{T_{1}})\re\ti_{T_{1}+t}.
		\end{split}
	\end{equation*}
	We simplify the term
	\begin{equation*}
		\begin{split}
			&\msC(\widetilde{w}'_{t},P_{t}\nabla\mcI[\ti]_{T_{1}},\nabla\Phi_{\ti_{T_{1}+t}})+\msC(\widetilde{w}'_{t},P_{t}\nabla^{2}\mcI[\Phi_{\ti}]_{T_{1}},\ti_{T_{1}+t})\\
			&\quad+(-P_{t}\vdiv\mcI[w'\pa\ti]_{T_{1}}+P_{t}\vdiv\mcI[w'\pa\ti]_{T_{1}}-\widetilde{w}'_{t}\pa P_{t}\nabla\mcI[\ti]_{T_{1}})\re\nabla\Phi_{\ti_{T_{1}+1}}\\
			&\quad+(-\nabla\Phi_{P_{t}\vdiv\mcI[w'\pa\ti]_{T_{1}}}+P_{t}\nabla\vdiv\Phi_{\mcI[w'\pa\ti]_{T_{1}}}-\widetilde{w}'_{t}\pa P_{t}\nabla^{2}\mcI[\Phi_{\ti}]_{T_{1}})\re\ti_{T_{1}+t}\\
			&=\msC(\widetilde{w}'_{t},P_{t}\nabla\mcI[\ti]_{T_{1}},\nabla\Phi_{\ti_{T_{1}+t}})+\msC(\widetilde{w}'_{t},P_{t}\nabla^{2}\mcI[\Phi_{\ti}]_{T_{1}},\ti_{T_{1}+t})\\
			&\quad-(\widetilde{w}'_{t}\pa P_{t}\nabla\mcI[\ti]_{T_{1}})\re\nabla\Phi_{\ti_{T_{1}+1}}-(\widetilde{w}'_{t}\pa P_{t}\nabla^{2}\mcI[\Phi_{\ti}]_{T_{1}})\re\ti_{T_{1}+t}\\
			&=-\widetilde{w}'_{t}(P_{t}\nabla\mcI[\ti]_{T_{1}}\re\nabla\Phi_{\ti_{T_{1}+t}}+P_{t}\nabla^{2}\mcI[\Phi_{\ti}]_{T_{1}}\re\ti_{T_{1}+t}).
		\end{split}
	\end{equation*}
	Therefore,
	\begin{equation*}
		\begin{split}
			&\msP(\boldsymbol{w},\mbX)_{T_{1}+t}\\
			&=\msC(\widetilde{w}'_{t},\nabla\mcI[\ti_{\shift{T_{1}}}]_{t},\nabla\Phi_{\ti_{T_{1}+t}})+\msC(\widetilde{w}'_{t},\nabla^{2}\mcI[\Phi_{\ti_{\shift{T_{1}}}}]_{t},\ti_{T_{1}+t})\\
			&\quad+(\widetilde{w}^{\#}_{t}+\vdiv\mcI[\widetilde{w}'\pa\ti_{\shift{T_{1}}}]_{t}-\widetilde{w}'_{t}\pa\nabla\mcI[\ti_{\shift{T_{1}}}]_{t})\re\nabla\Phi_{\ti_{T_{1}+1}}\\
			&\quad+(\nabla\Phi_{\widetilde{w}^{\#}_{t}}+\nabla\vdiv\Phi_{\mcI[\widetilde{w}'\pa\ti_{\shift{T_{1}}}]_{t}}-\widetilde{w}'_{t}\pa\nabla^{2}\mcI[\Phi_{\ti_{\shift{T_{1}}}}]_{t})\re\ti_{T_{1}+t}\\
			&\quad+\widetilde{w}'_{t}\tc_{T_{1}+t}-\widetilde{w}'_{t}(P_{t}\nabla\mcI[\ti]_{T_{1}}\re\nabla\Phi_{\ti_{T_{1}+t}}+P_{t}\nabla^{2}\mcI[\Phi_{\ti}]_{T_{1}}\re\ti_{T_{1}+t})\\
			&=\msP_{T_{1}}(\boldsymbol{w},\mbX)_{t}.
		\end{split}
	\end{equation*}
	Hence, $\widetilde{w}^{\#}_{t}=w^{\#}_{T_{1}+t}+P_{t}\vdiv\mcI[w'\pa\ti]_{T_{1}}$ satisfies the equation
	\begin{equation*}
		\begin{split}
			&\widetilde{w}^{\#}_{t}\\
			&=P_{t}w_{T_{1}}+\vdiv\mcI[\widetilde{w}\nabla\Phi_{\widetilde{w}}]_{t}+\vdiv\mcI[\widetilde{w}\nabla\Phi_{\ty_{\shift{T_{1}}}}]_{t}+\vdiv\mcI[\ty_{\shift{T_{1}}}\nabla\Phi_{\widetilde{w}}]_{t}\\
			&\quad+\vdiv\mcI[\ty_{\shift{T_{1}}}\nabla\Phi_{\ty_{\shift{T_{1}}}}]_{t}+\vdiv\mcI[\tp_{\shift{T_{1}}}]_{t}+\vdiv\mcI[\nabla\Phi_{\ti_{\shift{T_{1}}}}\pa\ty_{\shift{T_{1}}}]_{t}+\vdiv\mcI[\ty_{\shift{T_{1}}}\pa\nabla\Phi_{\ti_{\shift{T_{1}}}}]_{t}\\
			&\quad+\vdiv\mcI[\ti_{\shift{T_{1}}}\pa\nabla\Phi_{\ty_{\shift{T_{1}}}}]_{t}+\vdiv\mcI[\widetilde{w}\pa\nabla\Phi_{\ti_{\shift{T_{1}}}}]_{t}+\vdiv\mcI[\nabla\Phi_{\ti_{\shift{T_{1}}}}\pa\widetilde{w}]_{t}+\vdiv\mcI[\ti_{\shift{T_{1}}}\pa\nabla\Phi_{\widetilde{w}}]_{t}\\
			&\quad+\vdiv\mcI[\msP_{T_{1}}(\boldsymbol{w},\mbX)]_{t},
		\end{split}
	\end{equation*}
	which is precisely~\eqref{eq:paracontrolled_equation_ctd_remainder}.
\end{details}

We can now use a technique inspired by~\cite[Derivation of~(1.27)]{mourrat_weber_17_CDI} to establish a maximal time of existence $\Tmax\leq T$ and a paracontrolled solution on $[0,\Tmax)$.
\begin{theorem}\label{thm:maximal_existence}
	Let $(\alpha,p,q,\beta,\beta',\beta^{\#},\beta_{0},\kappa,\eta)$ satisfy~\eqref{eq:exponents}, $T>0$, $\mbX\in\rksnoise{\alpha}{\kappa}_{T}$ and $\rho_{0}\in\mcB_{p,q}^{\beta_{0}}(\mbT^{2})$. Then there exists a $\Tmax\in(0,T]$ and a unique paracontrolled solution $\rho\from[0,\Tmax)\to\mcS'(\mbT^{2})$ to~\eqref{eq:gen_rKS_intro} on $[0,\Tmax)$ with enhancement $\mbX$ and initial data $\rho_{0}$, such that
	\begin{equation*}
		\Tmax=T\qquad\text{or}\qquad\lim_{t\uparrow\Tmax}\norm{\rho_{t}}_{\mcC^{\alpha+1}}=\infty.
	\end{equation*}
\end{theorem}
\begin{proof}
	Let $(\alpha,p,q,\beta,\beta',\beta^{\#},\beta_{0},\kappa,\eta)$ satisfy~\eqref{eq:exponents}. Since $\rho_{0}\in\mcB_{p,q}^{\beta_{0}}(\mbT^{2})$, we can apply Proposition~\ref{prop:w_local_existence} and run the equation for some small time $\bar{t}_{0}=\bar{t}_{0}(\norm{\mbX}_{\rksnoise{\alpha}{\kappa}_{T}},\norm{\rho_{0}}_{\mcB_{p,q}^{\beta_{0}}})\leq T$. By construction, we obtain $\rho_{t_{0}}\in\mcC^{\alpha+1}(\mbT^{2})$ for each $0<t_{0}\leq\bar{t}_{0}$.
	
	For the second iteration step, we choose a new tuple of exponents $(\alpha,\infty,\infty,\beta_{1},\beta'_{1},\beta^{\#}_{1},\alpha+1,\kappa,\eta_{1})$ satisfying the assumptions~\eqref{eq:exponents} and
	\begin{equation}\label{eq:exponents_Tmax}
		\eta_{1}\in\Bigl[\Bigl(\frac{\beta_{1}-(\alpha+1)}{2}\Bigr)\vee\Bigl(\frac{\beta^{\#}_{1}-(\alpha+1)}{4}\Bigr),\frac{\alpha+3}{2}\Bigr)\cap\Bigl(\frac{-2\alpha-4}{2},\frac{\alpha+3}{2}\Bigr).
	\end{equation}
	\begin{details}
		
		We need to show that  $(\alpha,\infty,\infty,\beta_{1},\beta'_{1},\beta^{\#}_{1},\alpha+1,\kappa,\eta_{1})$ can be chosen such that
		\begin{equation*}
			\begin{split}
				\alpha\in(-9/4,-2),&\qquad q=\infty,\\
				p=\infty,&\qquad\beta_{1}\in(-1/2,2\alpha+4),\\
				\beta'_{1}\in(-2\alpha-4,(\beta_{1}+1)\wedge(2\alpha+5)],&\qquad\beta^{\#}_{1}\in(-\alpha-2,\alpha+\beta'_{1}+2),\\
				\alpha+1\in(\beta_{1}-(\alpha+3),\beta_{1}],&\qquad\kappa\in((\beta^{\#}_{1}-\alpha-2)/2,1/2),\\
				\eta_{1}\in\Bigl[\Bigl(\frac{\beta_{1}-(\alpha+1)}{2}\Bigr)\vee\Bigl(\frac{\beta^{\#}_{1}-(\alpha+1)}{4}\Bigr),&\frac{\alpha+3}{2}\Bigr)\cap\Bigl(\frac{-2\alpha-4}{2},\frac{\alpha+3}{2}\Bigr),
			\end{split}
		\end{equation*}
		where $\alpha$ and $\kappa$ are the same as in our first application of Proposition~\ref{prop:w_local_existence} (so that we may measure the enhancement $\mbX$ in the same space $\rksnoise{\alpha}{\kappa}_{T}$ as in the first iteration step.)
		
		The interval $\beta_{1}\in(-1/2,2\alpha+4)$ is non-empty, since
		\begin{equation*}
			-1/2<2\alpha+4\quad\iff\quad-9/4<\alpha.
		\end{equation*}
		Given $\beta_{1}\in(-1/2,2\alpha+4)\subset(-1/2,\alpha+2)$, it follows by the derivation of~\eqref{eq:exponents} that the intervals $\beta'_{1}\in(-2\alpha-4,(\beta_{1}+1)\wedge(2\alpha+5)]$ and $\beta^{\#}_{1}\in(-\alpha-2,\alpha+\beta'_{1}+2)$ are non-empty. Next we need to make sure that $\alpha+1\in(\beta_{1}-(\alpha+3),\beta_{1}]$. We estimate
		\begin{equation*}
			\alpha+1<-\frac{1}{2}<\beta_{1}
		\end{equation*}
		and
		\begin{equation*}
			\beta_{1}-(\alpha+3)<\alpha+1\quad\iff\quad\beta_{1}<2\alpha+4,
		\end{equation*}
		which is already satisfied.
		
		To ensure $\kappa\in((\beta^{\#}_{1}-\alpha-2)/2,1/2)$, we choose $\beta_{1}^{\#}\leq\beta_{1}$. (We can always choose $\beta_{1}^{\#}$ closer to $-\alpha-2$ to achieve this). It then follows that 
		\begin{equation*}
			\frac{\beta^{\#}_{1}-\alpha-2}{2}\leq\frac{\beta^{\#}-\alpha-2}{2}<\kappa<\frac{1}{2}.
		\end{equation*}
		The intersection
		\begin{equation*}
			\eta_{1}\in\Bigl[\Bigl(\frac{\beta_{1}-(\alpha+1)}{2}\Bigr)\vee\Bigl(\frac{\beta^{\#}_{1}-(\alpha+1)}{4}\Bigr),\frac{\alpha+3}{2}\Bigr)\cap\Bigl(\frac{-2\alpha-4}{2},\frac{\alpha+3}{2}\Bigr)
		\end{equation*}
		is non-empty, since by the derivation of~\eqref{eq:exponents},
		\begin{equation*}
			\Bigl(\frac{\beta_{1}-(\alpha+1)}{2}\Bigr)\vee\Bigl(\frac{\beta^{\#}_{1}-(\alpha+1)}{4}\Bigr)<\frac{\alpha+3}{2};
		\end{equation*}
		and
		\begin{equation*}
			\frac{-2\alpha-4}{2}<\frac{\alpha+3}{2}\quad\iff\quad-7<3\alpha,
		\end{equation*}
		which holds true by $-7/3<-9/4<\alpha$.
		
	\end{details}
	\begin{details}
		We cannot simply set $\beta_{1}=\beta$, $\beta'_{1}=\beta'$ and $\beta^{\#}_{1}=\beta^{\#}$. For instance, $\beta\in(-1/2,\alpha+2)$, whereas by assumption $\beta_{1}\in(-1/2,2\alpha+4)$. If $\beta\in[2\alpha+4,\alpha+2)$, then $\beta=\beta_{1}$ won't be possible.
		
	\end{details}
	Using the assumption~\eqref{eq:exponents_Tmax}, we can control the additional term appearing in~\eqref{eq:paracontrolled_equation_ctd_ren_prod} via Bony's estimates (Lemma~\ref{lem:Bony}) and Schauder's estimates (Lemma~\ref{lem:Schauder}).
	
	\begin{details}
		Let $S\leq T-T_{1}$, to estimate the new term in the renormalised product~\eqref{eq:paracontrolled_equation_ctd_ren_prod}, we use Bony's estimates (Lemma~\ref{lem:Bony}) with $\beta'_{1}+2\alpha+4>0$,
		\begin{equation*}
			\begin{split}
				&\norm{\widetilde{w}'(\nabla P\mcI[\ti]_{T_{1}}\re\nabla\Phi_{\ti_{T_{1}+\place}}+\nabla^{2}P\mcI[\Phi_{\ti}]_{T_{1}}\re\ti_{T_{1}+\place})}_{C_{2\eta_{1};S}\mcC^{2\alpha+4}}\\
				&\lesssim\norm{\widetilde{w}'}_{C_{\eta_{1};S}\mcC^{\beta'_{1}}}(\norm{\nabla P\mcI[\ti]_{T_{1}}\re\nabla\Phi_{\ti_{T_{1}+\place}}}_{C_{\eta_{1};S}\mcC^{2\alpha+4}}+\norm{\nabla^{2}P\mcI[\Phi_{\ti}]_{T_{1}}\re\ti_{T_{1}+\place}}_{C_{\eta_{1};S}\mcC^{2\alpha+4}}).
			\end{split}
		\end{equation*}
		We then use Bony's estimate to bound
		\begin{equation*}
			\norm{\nabla P\mcI[\ti]_{T_{1}}\re\nabla\Phi_{\ti_{T_{1}+\place}}}_{C_{\eta_{1};S}\mcC^{2\alpha+4}}\lesssim\norm{\nabla P\mcI[\ti]_{T_{1}}}_{C_{\eta_{1};S}\mcC^{\alpha+2+2\eta_{1}}}\norm{\nabla\Phi_{\ti_{T_{1}+\place}}}_{C_{S}\mcC^{\alpha+2}}
		\end{equation*}
		and
		\begin{equation*}
			\norm{\nabla^{2}P\mcI[\Phi_{\ti}]_{T_{1}}\re\ti_{T_{1}+\place}}_{C_{\eta_{1};S}\mcC^{2\alpha+4}}\lesssim\norm{\nabla^{2}P\mcI[\Phi_{\ti}]_{T_{1}}}_{C_{\eta_{1};S}\mcC^{\alpha+3+2\eta_{1}}}\norm{\ti_{T_{1}+\place}}_{C_{S}\mcC^{\alpha+1}}.
		\end{equation*}
		We bound by Schauder's estimate (Lemma~\ref{lem:Schauder}),
		\begin{equation*}
			\begin{split}
				\norm{\nabla P\mcI[\ti]_{T_{1}}}_{C_{\eta_{1};S}\mcC^{\alpha+2+2\eta_{1}}}\lesssim\norm{P\mcI[\ti]_{T_{1}}}_{C_{\eta_{1};S}\mcC^{\alpha+3+2\eta_{1}}}&\lesssim(1\vee S^{-\eta_{1}})(1\wedge S)^{\eta_{1}}\norm{\mcI[\ti]_{T_{1}}}_{\mcC^{\alpha+3}}\\
				&\lesssim(1\vee S^{-\eta_{1}})(1\wedge S)^{\eta_{1}}\norm{\ti}_{C_{T_{1}}\mcC^{\alpha+1}},
			\end{split}
		\end{equation*}
		where $(1\vee S^{-\eta_{1}})(1\wedge S)^{\eta_{1}}=(1\wedge S^{\eta_{1}})^{-1}(1\wedge S)^{\eta_{1}}=1$.
		
	\end{details}
	\begin{details}
		To estimate the initial data in~\eqref{eq:paracontrolled_equation_ctd_remainder}, we use the triangle inequality and $\alpha+1<2\alpha+4$,
		\begin{equation*}
			\norm{\rho_{T_{1}}-\ti_{T_{1}}-\ty_{T_{1}}}_{\mcC^{\alpha+1}}\lesssim\norm{\rho_{T_{1}}}_{\mcC^{\alpha+1}}+\norm{\ti_{T_{1}}}_{\mcC^{\alpha+1}}+\norm{\ty_{T_{1}}}_{\mcC^{2\alpha+4}}.
		\end{equation*}
		
	\end{details}
	Let $R>0$ be such that $\norm{\mbX}_{\rksnoise{\alpha}{\kappa}_{T}}\leq R$ and assume that $0<t_{0}<\bar{t}_{0}$ satisfies $\norm{\rho_{t_{0}}}_{\mcC^{\alpha+1}}\leq R$. We can then use the decomposition~\eqref{eq:paracontrolled_equation_ctd_Ansatz}--\eqref{eq:paracontrolled_equation_ctd_ren_prod} and an argument analogous to Proposition~\ref{prop:w_local_existence} to re-start the equation from $\rho_{t_{0}}$ and run it until time $t_{0}+T(R)$ for some $T(R)>0$.
	
	Next, assume $\norm{\rho_{t_{0}+T(R)/2}}_{\mcC^{\alpha+1}}\leq R$, we can then re-start the equation at time $t_{0}+T(R)/2$ from $\rho_{t_{0}+T(R)/2}$ and run it until $t_{0}+T(R)/2+T(R)$ with the same exponents. We can then continue this procedure on the intervals
	\begin{equation}\label{eq:intervals_iteration}
		[0,\bar{t}_{0}],\quad[t_{0},t_{0}+T(R)],\quad\Bigl[t_{0}+\frac{T(R)}{2},t_{0}+\frac{T(R)}{2}+T(R)\Bigr],\quad\Bigl[t_{0}+T(R),t_{0}+2T(R)\Bigr],\quad\ldots
	\end{equation}
	until the first time $t\in[t_{0},T]$ such that $\norm{\rho_{t}}_{\mcC^{\alpha+1}}>R$, or, if such a $t$ does not exist, until $T$. We denote this time horizon by $\Tmax^{(R)}$. This yields a paracontrolled solution $\rho\from[0,\Tmax^{(R)}]\to\mcS'(\mbT^{2})$ to~\eqref{eq:gen_rKS_intro} with enhancement $\mbX$ and initial data $\rho_{0}\in\mcB_{p,q}^{\beta_{0}}(\mbT^{2})$, where we used that the intervals~\eqref{eq:intervals_iteration} had some overlap as in~\cite[Derivation of~(1.27)]{mourrat_weber_17_CDI} to ensure that $\rho$ is independent of the choice of iteration used in its construction. Furthermore, the paracontrolled solution is unique by an application of Lemma~\ref{lem:paracontrolled_solution_unique}.
	\begin{details}
		\paragraph{Continuity of $[0,\Tmax^{(R)}]\ni t\mapsto(1\wedge t)^{\eta}\rho_{t}\in\mcC^{\alpha+1}(\mbT^{2})$.}
		It follows by the first iteration step that $[0,\bar{t}_{0}]\ni t\mapsto(1\wedge t)^{\eta}\rho_{t}$ is continuous in $\mcC^{\alpha+1}(\mbT^{2})$. It then follows by the second iteration step with $t_{0}<\bar{t}_{0}$, that $[t_{0},t_{0}+T(R)]\ni t\mapsto(1\wedge(t-t_{0}))^{\eta_{1}}\rho_{t}$ is continuous in $\mcC^{\alpha+1}(\mbT^{2})$. Hence, by the continuity of $[\bar{t}_{0},t_{0}+T(R)]\ni t\mapsto(1\wedge t)^{\eta}(1\wedge(t-t_{0}))^{-\eta_{1}}$, it follows that
		\begin{equation*}
			[\bar{t}_{0},t_{0}+T(R)]\ni t\mapsto(1\wedge t)^{\eta}\rho_{t}=(1\wedge t)^{\eta}(1\wedge(t-t_{0}))^{-\eta_{1}}(1\wedge(t-t_{0}))^{\eta_{1}}\rho_{t}
		\end{equation*}
		is continuous in $\mcC^{\alpha+1}(\mbT^{2})$. We can then iterate the argument to obtain the continuity of $[0,\Tmax^{(R)}]\ni t\mapsto(1\wedge t)^{\eta}\rho_{t}\in\mcC^{\alpha+1}(\mbT^{2})$.
		\paragraph{The bound $\sup_{t\in[0,\Tmax^{(R)}]}(1\wedge t)^{\eta}\norm{\rho_{t}}_{\mcC^{\alpha+1}}<\infty$.}
		Let $t\in[t_{0},t_{0}+T(R)]$. By the construction of the second iteration step, we obtain
		\begin{equation*}
			\norm{\rho_{t}}_{\mcC^{\alpha+1}}\lesssim(1\wedge (t-t_{0}))^{-\eta_{1}}
		\end{equation*}
		and it is not clear a priori if we can achieve a uniform estimate for $t$ close to $t_{0}$. To solve this problem, we use that there is some overlap between $[0,\bar{t}_{0}]$ and $[t_{0},t_{0}+T(R)]$. We estimate for all $0<t\leq\bar{t}_{0}$,
		\begin{equation*}
			(1\wedge t)^{\eta}\norm{\rho_{t}}_{\mcC^{\alpha+1}}\leq\norm{\rho}_{C_{\eta;\bar{t}_{0}\mcC^{\alpha+1}}};
		\end{equation*}
		and for all $\bar{t}_{0}\leq t\leq t_{0}+T(R)$,
		\begin{equation*}
			\begin{split}
				(1\wedge t)^{\eta}\norm{\rho_{t}}_{\mcC^{\alpha+1}}&=(1\wedge t)^{\eta}(1\wedge(t-t_{0}))^{-\eta_{1}}(1\wedge(t-t_{0}))^{\eta_{1}}\norm{\rho_{t}}_{\mcC^{\alpha+1}}\\
				&\leq(1\wedge t)^{\eta}(1\wedge(t-t_{0}))^{-\eta_{1}}\norm{\rho_{\shift{t_{0}}}}_{C_{\eta_{1};T(R)}\mcC^{\alpha+1}}\\
				&\lesssim(1\wedge T)^{\eta}(1\wedge(\bar{t}_{0}-t_{0}))^{-\eta_{1}}\norm{\rho_{\shift{t_{0}}}}_{C_{\eta_{1};T(R)}\mcC^{\alpha+1}}.
			\end{split}
		\end{equation*}
		Consequently,
		\begin{equation*}
			\sup_{t\in[0,t_{0}+T(R)]}(1\wedge t)^{\eta}\norm{\rho_{t}}_{\mcC^{\alpha+1}}<\infty.
		\end{equation*}
		The subsequent iteration steps then follow by induction.
	\end{details}
	The sequence $(T^{(R)}_{\max})_{R>0}$ is increasing, hence we may define $\Tmax\defeq\sup_{R>0}\Tmax^{(R)}=\lim_{R\to\infty}\Tmax^{(R)}\in(0,T]$. Assume $\Tmax<T$, then 
	\begin{equation*}
		\lim_{t\uparrow\Tmax}\norm{\rho_{t}}_{\mcC^{\alpha+1}}=\lim_{R\uparrow\infty}\norm{\rho_{\Tmax^{(R)}}}_{\mcC^{\alpha+1}}>\lim_{R\uparrow\infty}R=\infty.
	\end{equation*}
	This yields the claim.
\end{proof}
The principal issue in comparing paracontrolled solutions is that the maximal time of existence may depend on the noise enhancement. Hence, one needs to find a metric space of functions that can blow up at separate, finite times. A suitable space was introduced in~\cite[Sec.~1.5.1]{chandra_chevyrev_hairer_shen_22}, whose construction we follow:
\begin{definition}\label{def:space_with_cemetery}
	Let $(\target,\norm{\place}_{\target})$ be a normed vector space and let $\cem$ be a cemetery state. We define $\target^{\cem}\defeq\target\cup\{\cem\}$ and equip it with the topology generated by the open balls of $\target$ and the sets of the form  $\{u:\norm{u}_{\target}>R\}\cup\{\cem\}$ with $R>0$.
\end{definition}
\begin{details}
	The topology of $\target^{\cem}$ can be metrized with the metric
	\begin{equation*}
		d^{\cem}(u,v)\defeq\norm{u-v}_{\target}\wedge(h(u)+h(v)),
	\end{equation*}
	where we set $\norm{\cem}_{\target}\defeq\infty$ and $h(u)\defeq(1+\norm{u}_{\target})^{-1}$. A proof of this statement can be found in~\cite[Sec.~1.5.1]{chandra_chevyrev_hairer_shen_22}.
	
\end{details}
Next we define the space $\target^{\sol}_{T}$ of continuous functions on $[0,T]$ with values in $\target^{\cem}$ that cannot return from the cemetery state $\cem$.
\begin{definition}\label{def:sol_space}
	We define the set
	\begin{equation*}
		\target^{\sol}_{T}\defeq\Bigl\{f\in C([0,T];\target^{\cem}):f|_{(T_{\target}^{\cem}[f],T]}\equiv\cem\Bigr\},\quad\text{where}\quad T_{\target}^{\cem}[f]\defeq T\wedge\inf\{t\in[0,T]:f(t)=\cem\}.
	\end{equation*}
\end{definition}
We equip $\target^{\sol}_{T}$ with a suitable metric.
\begin{lemma}\label{lem:Fsol_T_convergence}
	There exists a metric $D_{T}^{\target}$ on $\target^{\sol}_{T}$ such that $D_{T}^{\target}(f_{n},f)\to 0$  as $n\to\infty$ if and only if for every $L>0$,
	\begin{equation*}
		\lim_{n\to\infty}\sup_{t\in[0,T_{L}]}\norm{f_{n}(t)-f(t)}_{\target}=0,
	\end{equation*}
	where
	\begin{equation*}
		T_{L}\defeq T\wedge L\wedge\inf\{t\in[0,T]:\norm{f_{n}(t)}_{\target}>L~\text{or}~\norm{f(t)}_{\target}>L\}.
	\end{equation*}
\end{lemma}

\begin{proof}
	The existence of such a metric follows by~\cite[Lem.~1.2]{chandra_chevyrev_hairer_shen_22}.
	\begin{details}
		
		Let $f\in\target^{\sol}_{T}$. We may extend $f$ from $[0,T]$ to $[0,\infty)$ by setting $f(t)=f(T)$ for all $t>T$. It follows that $f$ is an element of $\target^{\sol}$ as defined in~\cite[Sec.~1.5.1]{chandra_chevyrev_hairer_shen_22} with identical metrics on such functions. In particular, our space $\target^{\sol}_{T}$ is well-defined by the construction in~\cite[Sec.~1.5.1]{chandra_chevyrev_hairer_shen_22}.
	\end{details}
\end{proof}
\begin{details}
	Let $\psi\from\mbR\to[0,1]$ be a smooth, non-increasing function with $\supp(\psi')\subset[1,2]$, $\psi(1)=1$ and $\psi(2)=0$. Define for $L>0$ and $f\in\target^{\sol}_{T}$, the functionals $S_{f}(t)\defeq\sup_{s\leq t}\norm{f(s)}_{\target}$, $\Theta_{L}(f)(t)\defeq(\psi(S_{f}(t)/L),f(t))$ and the joint metric 
	\begin{equation*}
		\hat{d}((a,f),(b,g))\defeq\abs{a-b}+(a\wedge b)d^{\cem}(f,g),\qquad\text{where}~a,b\in[0,1]~\text{and}~f,g\in\target^{\sol}_{T}.
	\end{equation*}
	Then,
	\begin{equation*}
		D_{T}(f,g)\defeq\sum_{L=1}^{\infty}2^{-L}\sup_{t\in[0,T\wedge L]}\hat{d}\Bigl(\Theta_{L}(f)(t),\Theta_{L}(g)(t)\Bigr).
	\end{equation*}
\end{details}
Finally, we equip $\target^{\sol}_{T}$ with a weight at $0$ to allow for irregular initial data.
\begin{definition}\label{def:sol_space_weight}
	We denote $u_{\eta\weight}(t)\defeq(1\wedge t)^{\eta}u(t)$ for each $\eta\geq0$, $u\from(0,T]\to\target^{\cem}$ and $t\in(0,T]$. We define
	\begin{equation*}
		\target^{\sol}_{\eta;T}\defeq\Bigl\{u\from(0,T]\to\target^{\cem}:u_{\eta\weight}(0)\defeq\lim_{t\to0}(1\wedge t)^{\eta}u(t)\in\target^{\cem},~u_{\eta\weight}\in\target^{\sol}_{T}\Bigr\}
	\end{equation*}
	and equip $\target^{\sol}_{\eta;T}$ with the metric $D_{\eta;T}^{\target}$ given by
	\begin{equation*}
		D_{\eta;T}^{\target}(u,v)\defeq D_{T}^{\target}(u_{\eta\weight},v_{\eta\weight}),\qquad u,v\in\target^{\sol}_{\eta;T}.
	\end{equation*}
\end{definition}
In the following, we show that the paracontrolled solutions constructed in Theorem~\ref{thm:maximal_existence} are elements of $(\mcC^{\alpha+1}(\mbT^{2}))^{\sol}_{\eta;T}$.
\begin{lemma}\label{lem:paracontrolled_sol_in_blow_up_space}
	Let $(\alpha,p,q,\beta,\beta',\beta^{\#},\beta_{0},\kappa,\eta)$ satisfy~\eqref{eq:exponents}, $T>0$, $\mbX\in\rksnoise{\alpha}{\kappa}_{T}$, $\rho_{0}\in\mcB_{p,q}^{\beta_{0}}(\mbT^{2})$ and $\rho$ be the paracontrolled solution to~\eqref{eq:gen_rKS_intro} on $[0,\Tmax)$ with enhancement $\mbX$ and initial data $\rho_{0}$, as constructed in Theorem~\ref{thm:maximal_existence}. Then $\rho\in(\mcC^{\alpha+1}(\mbT^{2}))^{\sol}_{\eta;T}$ and in particular $\Tmax=T_{\mcC^{\alpha+1}}^{\cem}[\rho_{\eta\weight}]$.
\end{lemma}
\begin{proof}
	The paracontrolled solution constructed in Theorem~\ref{thm:maximal_existence} is continuous and exists until it blows up in $\mcC^{\alpha+1}(\mbT^{2})$, from which the claim follows. 
	\begin{details}
		
		Proof that $\Tmax=T^{\cem}_{\mcC^{\alpha+1}}[\rho_{\eta\weight}]$. Let $t_{0}>0$ be as in the proof of  Theorem~\ref{thm:maximal_existence}. We claim that $\Tmax$ is the unique $S\in[t_{0},T]$ with the property
		\begin{equation}\label{eq:Tmax_fixed_point}
			S=T\wedge\sup_{R>0}\inf\Bigl\{t\in[t_{0},S):\norm{\rho_{t}}_{\mcC^{\alpha+1}}>R\Bigr\}.
		\end{equation}
		\paragraph{Proof that $\Tmax$ satisfies~\eqref{eq:Tmax_fixed_point}.}
		By definition,
		\begin{equation*}
			\Tmax=\sup_{R>0}\Tmax^{(R)}=T\wedge\sup_{R>0}\inf\Bigl\{t\in[t_{0},\Tmax):\norm{\rho_{t}}_{\mcC^{\alpha+1}}>R\Bigr\},
		\end{equation*}
		 hence it satisfies~\eqref{eq:Tmax_fixed_point}.
		\paragraph{Proof of uniqueness.}
		Let $S\neq\Tmax$ be another time that satisfies~\eqref{eq:Tmax_fixed_point}. If $S<\Tmax$, let $R>0$ be such that $S<\Tmax^{(R)}$, then
		\begin{equation*}
			S<\Tmax^{(R)}=T\wedge\inf\Bigl\{t\in[t_{0},\Tmax):\norm{\rho_{t}}_{\mcC^{\alpha+1}}>R\Bigr\}\leq T\wedge\inf\Bigl\{t\in[t_{0},S):\norm{\rho_{t}}_{\mcC^{\alpha+1}}>R\Bigr\}\leq S,
		\end{equation*}
		which yields a contradiction. If $\Tmax<S$, let $R>0$ be such that
		\begin{equation*}
			\Tmax<T\wedge\inf\Bigl\{t\in[t_{0},S):\norm{\rho_{t}}_{\mcC^{\alpha+1}}>R\Bigr\}\leq T\wedge\inf\Bigl\{t\in[t_{0},\Tmax):\norm{\rho_{t}}_{\mcC^{\alpha+1}}>R\Bigr\}\leq\Tmax,
		\end{equation*}
		which yields another contradiction. Hence $\Tmax$ is the unique $S\in[t_{0},T]$ with the property~\eqref{eq:Tmax_fixed_point}.
		
		We can use that $[t_{0},\Tmax]\ni t\mapsto(1\wedge t)^{\eta}$ is bounded above and below to deduce for each $R>0$,
		\begin{equation*}
			T\wedge\inf\Bigl\{t\in[t_{0},\Tmax):\norm{\rho_{t}}_{\mcC^{\alpha+1}}>R\Bigr\}\leq T\wedge\inf\Bigl\{t\in[t_{0},\Tmax):(1\wedge t)^{\eta}\norm{\rho_{t}}_{\mcC^{\alpha+1}}>(1\wedge T)^{\eta}R\Bigr\}
		\end{equation*}
		and
		\begin{equation*}
			T\wedge\inf\Bigl\{t\in[t_{0},\Tmax):\norm{\rho_{t}}_{\mcC^{\alpha+1}}>R\Bigr\}\geq T\wedge\inf\Bigl\{t\in[t_{0},\Tmax):(1\wedge t)^{\eta}\norm{\rho_{t}}_{\mcC^{\alpha+1}}>(1\wedge t_{0})^{\eta}R\Bigr\},
		\end{equation*}
		which implies
		\begin{equation*}
			\Tmax=T\wedge\sup_{R>0}\inf\Bigl\{t\in[t_{0},\Tmax):\norm{\rho_{t}}_{\mcC^{\alpha+1}}>R\Bigr\}=T\wedge\sup_{R>0}\inf\Bigl\{t\in[t_{0},\Tmax):(1\wedge t)^{\eta}\norm{\rho_{t}}_{\mcC^{\alpha+1}}>R\Bigr\}.
		\end{equation*}
		By definition,
		\begin{equation*}
			\begin{split}
				T^{\cem}[\rho_{\eta\weight}]=T\wedge\inf\Bigl\{t\in[0,T]:\rho_{\eta\weight}(t)=\cem\Bigr\}&=T\wedge\inf\Bigl\{t\in[0,T]:\norm{\rho_{\eta\weight}(t)}_{\mcC^{\alpha+1}}=\infty\Bigr\}\\
				&=T\wedge\sup_{R>0}\inf\Bigl\{t\in[0,T]:\norm{\rho_{\eta\weight}(t)}_{\mcC^{\alpha+1}}>R\Bigr\}.
			\end{split}
		\end{equation*}
		Using that $T^{\cem}[\rho_{\eta'\weight}]\geq t_{0}$ by construction, we obtain
		\begin{equation*}
			T^{\cem}[\rho_{\eta\weight}]=T\wedge\sup_{R>0}\inf\Bigl\{t\in[t_{0},T]:\norm{\rho_{\eta\weight}(t)}_{\mcC^{\alpha+1}}>R\Bigr\}=\Tmax.
		\end{equation*}
	\end{details}
\end{proof}
\subsection{The Renormalised Solution}\label{subsec:ren_sol}
In this subsection we combine the deterministic solution theory (Lemma~\ref{lem:paracontrolled_sol_in_blow_up_space}) with the stochastic existence of the renormalised enhancement (Theorem~\ref{thm:enhancement_existence}) to construct the renormalised solution to~\eqref{eq:gen_rKS_intro} as a random variable (Theorem~\ref{thm:convergence_renorm_sol}, Part~\ref{it:thm_existence_renormalized}). We then show that the renormalised solution is the limit in probability of solutions to~\eqref{eq:smooth_mild_sol} (Theorem~\ref{thm:convergence_renorm_sol}, Part~\ref{it:thm_convergence_to_renormalized}).
\begin{definition}\label{def:renormalised_solution}
	Let $(\alpha,p,q,\beta,\beta',\beta^{\#},\beta_{0},\kappa,\eta)$ satisfy~\eqref{eq:exponents}, $\rho_{0}\in\mcB^{\beta_0}_{p,q}(\mbT^{2})$ be some initial data, $T>0$, $\het\in C_{T}\mcH^{2}(\mbT^{2})$ and $\mbX$ be the renormalised enhancement of Theorem~\ref{thm:enhancement_existence}. We call the paracontrolled solution $\rho\in(\mcC^{\alpha+1}(\mbT^{2}))_{\eta;T}^{\sol}$ to~\eqref{eq:gen_rKS_intro} on $[0,T_{\mcC^{\alpha+1}}^{\cem}[\rho_{\eta\weight}])$ with enhancement $\mbX$ and initial data $\rho_{0}$ (in the sense of Definition~\ref{def:paracontrolled_solution}) the renormalised solution.
\end{definition}
We can now prove the main result of this section, which is similar to~\cite[Cor.~4.7~\&~Cor.~5.9]{gubinelli_15_GIP}, \cite[Thm.~3.7]{gubinelli_perkowski_17} and~\cite[Cor.~3.13]{catellier_chouk_18}.
\begin{theorem}\label{thm:convergence_renorm_sol}
	Let $(\alpha,p,q,\beta,\beta',\beta^{\#},\beta_{0},\kappa,\eta)$ satisfy~\eqref{eq:exponents}, $\rho_{0}\in\mcB^{\beta_0}_{p,q}(\mbT^{2})$ be some initial data, $T>0$, $\het\in C_{T}\mcH^{2}(\mbT^{2})$ and $\mbX$ (resp.\ $(\mbX^{\delta})_{\delta>0}$) be the renormalised (resp.\ prelimiting renormalised) enhancement of Theorem~\ref{thm:enhancement_existence}.
	\begin{enumerate}
		\item\label{it:thm_existence_renormalized}
		There exist a paracontrolled solution $\rho$ (resp.\ $(\rho^{\delta})_{\delta>0}$) in $(\mcC^{\alpha+1}(\mbT^{2}))_{\eta;T}^{\sol}$ to~\eqref{eq:gen_rKS_intro} on $[0,T_{\mcC^{\alpha+1}}^{\cem}[\rho_{\eta\weight}])$ (resp.\ on $[0,T_{\mcC^{\alpha+1}}^{\cem}[\rho^{\delta}_{\eta\weight}])$) with enhancement $\mbX$ (resp.\ $(\mbX^{\delta})_{\delta>0}$) and initial data $\rho_{0}$ in the sense of Definition~\ref{def:paracontrolled_solution} (for every $\delta>0$). The random variables $(\rho^{\delta})_{\delta>0}$ coincide with the mild solutions~\eqref{eq:smooth_mild_sol}.
		\item\label{it:thm_convergence_to_renormalized}
		It holds that $\rho^{\delta}\to\rho$ in $(\mcC^{\alpha+1}(\mbT^{2}))^{\sol}_{\eta;T}$ in probability as $\delta\to0$. In particular for each $\lambda>0$ one has that
		\begin{equation}\label{eq:convergence_frechet}
			\lim_{\delta\to0}\mbP\Bigl(\sum_{L=1}^{\infty}\frac{2^{-L}}{L}\frac{\norm{\rho-\rho^{\delta}}_{C_{\eta;T_{L}}\mcC^{\alpha+1}}}{1+\norm{\rho-\rho^{\delta}}_{C_{\eta;T_{L}}\mcC^{\alpha+1}}}>\lambda\Bigr)=0,
		\end{equation}
		where for every $L\in\mbN$ we denote 
		\begin{equation*}
			T_{L}\defeq T\wedge L\wedge\inf\Bigl\{t\in[0,T]:\norm{\rho_{\eta\weight}(t)}_{\mcC^{\alpha+1}}>L~\text{or}~\norm{\rho_{\eta\weight}^{\delta}(t)}_{\mcC^{\alpha+1}}>L\Bigr\}.
		\end{equation*}
	\end{enumerate}
\end{theorem}
\begin{proof}
	Let $(\alpha,p,q,\beta,\beta',\beta^{\#},\beta_{0},\kappa,\eta)$ satisfy~\eqref{eq:exponents}. The enhancements $\mbX$ and $(\mbX^{\delta})_{\delta>0}$ exist almost surely by an application of Theorem~\ref{thm:enhancement_existence}. It then follows by Lemma~\ref{lem:paracontrolled_sol_in_blow_up_space} that $\rho$ and $(\rho^{\delta})_{\delta>0}$ are random variables in $(\mcC^{\alpha+1}(\mbT^{2}))^{\sol}_{\eta;T}$. By Lemma~\ref{lem:Fsol_T_convergence} and the local Lipschitz continuity of the solution map in the enhancement (Lemma~\ref{lem:Local_Lipschitz}) we deduce that for every $\lambda>0$ there exists some $\nu>0$ such that $\norm{\mbX-\mbX^{\delta}}_{\rksnoise{\alpha}{\kappa}_{T}}\leq\nu$ implies $D_{\eta;T}^{\mcC^{\alpha+1}}(\rho,\rho^{\delta})\leq\lambda$.
	\begin{details}
		We want to show
		\begin{equation}\label{eq:continuity_formula}
			\forall\lambda>0~\exists\nu>0~\forall\delta>0~\norm{\mbX-\mbX^{\delta}}_{\rksnoise{\alpha}{\kappa}_{T}}\leq\nu~\implies~D_{\eta;T}^{\mcC^{\alpha+1}}(\rho,\rho^{\delta})\leq\lambda.
		\end{equation}
		Assume~\eqref{eq:continuity_formula} is not true, then
		\begin{equation}\label{eq:continuity_formula_negation}
			\exists\lambda>0~\forall\nu>0~\exists\delta>0~\norm{\mbX-\mbX^{\delta}}_{\rksnoise{\alpha}{\kappa}_{T}}\leq\nu~\wedge~D_{\eta;T}^{\mcC^{\alpha+1}}(\rho,\rho^{\delta})>\lambda.
		\end{equation}
		Let $\lambda>0$ be as in~\eqref{eq:continuity_formula_negation} and choose for every $\nu>0$, $\delta(\nu)>0$ such that $\norm{\mbX-\mbX^{\delta(\nu)}}_{\rksnoise{\alpha}{\kappa}_{T}}\leq\nu$ and $D_{\eta;T}^{\mcC^{\alpha+1}}(\rho,\rho^{\delta(\nu)})>\lambda$. We obtain $\lim_{\nu\to0}\norm{\mbX-\mbX^{\delta(\nu)}}_{\rksnoise{\alpha}{\kappa}_{T}}=0$, so by the local Lipschitz continuity and Lemma~\ref{lem:Fsol_T_convergence}, $\lim_{\nu\to0}D_{\eta;T}^{\mcC^{\alpha+1}}(\rho,\rho^{\delta(\nu)})=0$, which contradicts $D_{\eta;T}^{\mcC^{\alpha+1}}(\rho,\rho^{\delta(\nu)})>\lambda$. Therefore, \eqref{eq:continuity_formula_negation} has to be false and~\eqref{eq:continuity_formula} has to be true.
		
	\end{details}
	Hence, by an application of Theorem~\ref{thm:enhancement_existence},
	\begin{equation*}
		\mbP\Bigl(D_{\eta;T}^{\mcC^{\alpha+1}}(\rho,\rho^{\delta})>\lambda\Bigr)\leq\mbP(\norm{\mbX-\mbX^{\delta}}_{\rksnoise{\alpha}{\kappa}_{T}}>\nu)\to0\quad\text{as}~\delta\to0,
	\end{equation*}
	which yields that $\rho^{\delta}\to\rho$ in $(\mcC^{\alpha+1}(\mbT^{2}))_{\eta;T}^{\sol}$ in probability. The convergence~\eqref{eq:convergence_frechet} then follows by the explicit form of $D_{\eta;T}^{\mcC^{\alpha+1}}$, see~\cite[Sec.~1.5.1]{chandra_chevyrev_hairer_shen_22}.
	\begin{details}
		\paragraph{Proof of the claim.}
		We obtain for each $t\leq T_{L}$ that $S_{\rho_{\eta\weight}}(t)\leq L$ and $S_{\rho^{\delta}_{\eta\weight}}(t)\leq L$, hence
		\begin{equation*}
			\Theta_{L}(\rho_{\eta\weight})(t)=(1,\rho_{\eta\weight}(t)),\qquad\Theta_{L}(\rho^{\delta}_{\eta\weight})(t)=(1,\rho^{\delta}_{\eta\weight}(t))
		\end{equation*}
		and
		\begin{equation*}
			\begin{split}
				&\hat{d}\Bigl(\Theta_{L}(\rho_{\eta\weight})(t),\Theta_{L}(\rho^{\delta}_{\eta\weight})(t)\Bigr)\\
				&=d^{\cem}\Bigl(\rho_{\eta\weight}(t),\rho^{\delta}_{\eta\weight}(t)\Bigr)\\
				&=\norm{\rho_{\eta\weight}(t)-\rho^{\delta}_{\eta\weight}(t)}_{\mcC^{\alpha+1}}\wedge\Bigl(\frac{1}{1+\norm{\rho_{\eta\weight}(t)}_{\mcC^{\alpha+1}}}+\frac{1}{1+\norm{\rho^{\delta}_{\eta\weight}(t)}_{\mcC^{\alpha+1}}}\Bigr)\\
				&\geq\norm{\rho_{\eta\weight}(t)-\rho^{\delta}_{\eta\weight}(t)}_{\mcC^{\alpha+1}}\wedge\frac{2}{1+L}.
			\end{split}
		\end{equation*}
		By the definition of $D_{\eta;T}^{\mcC^{\alpha+1}}$,
		\begin{equation*}
			\begin{split}
				D_{\eta;T}^{\mcC^{\alpha+1}}(\rho,\rho^{\delta})&=\sum_{L=1}^{\infty}2^{-L}\sup_{t\in[0,T\wedge L]}\hat{d}\Bigl(\Theta_{L}(\rho_{\eta\weight})(t),\Theta_{L}(\rho^{\delta}_{\eta\weight})(t)\Bigr)\\
				&\geq\sum_{L=1}^{\infty}2^{-L}\sup_{t\leq T_{L}}\hat{d}\Bigl(\Theta_{L}(\rho_{\eta\weight})(t),\Theta_{L}(\rho^{\delta}_{\eta\weight})(t)\Bigr)\\
				&\geq\sum_{L=1}^{\infty}2^{-L}\sup_{t\leq T_{L}}\Bigl(\norm{\rho_{\eta\weight}(t)-\rho^{\delta}_{\eta\weight}(t)}_{\mcC^{\alpha+1}}\wedge\frac{2}{1+L}\Bigr)\\
				&=\sum_{L=1}^{\infty}2^{-L}\Bigl(\norm{\rho-\rho^{\delta}}_{C_{\eta;T_{L}}\mcC^{\alpha+1}}\wedge\frac{2}{1+L}\Bigr),
			\end{split}
		\end{equation*}
		where in the last equality we used that $x\mapsto(x\wedge\frac{2}{1+L})$ is continuous and non-decreasing. Let $a\geq0$ and $x\geq0$, we use the bounds
		\begin{equation}\label{eq:minimum_factorize}
			(x\wedge a)\geq(1\wedge a)(1\wedge x)
		\end{equation}
		and
		\begin{equation}\label{eq:minimum_lower_bound}
			(1\wedge x)\geq\frac{x}{1+x}
		\end{equation}
		to deduce
		\begin{equation*}
			\Bigl(\norm{\rho-\rho^{\delta}}_{C_{\eta;T_{L}}\mcC^{\alpha+1}}\wedge\frac{2}{1+L}\Bigr)\geq\Bigl(1\wedge\frac{2}{1+L}\Bigr)\frac{\norm{\rho-\rho^{\delta}}_{C_{\eta;T_{L}}\mcC^{\alpha+1}}}{1+\norm{\rho-\rho^{\delta}}_{C_{\eta;T_{L}}\mcC^{\alpha+1}}}.
		\end{equation*}
		Using that $\frac{1}{L}\leq\frac{2}{1+L}\leq1$ and for all $L\in\mbN$, we obtain
		\begin{equation*}
			\Bigl(1\wedge\frac{2}{1+L}\Bigr)=\frac{2}{1+L}\geq\frac{1}{L}.
		\end{equation*}
		Therefore,
		\begin{equation*}
			D_{\eta;T}^{\mcC^{\alpha+1}}(\rho,\rho^{\delta})\geq\sum_{L=1}^{\infty}\frac{2^{-L}}{L}\frac{\norm{\rho-\rho^{\delta}}_{C_{\eta;T_{L}}\mcC^{\alpha+1}}}{1+\norm{\rho-\rho^{\delta}}_{C_{\eta;T_{L}}\mcC^{\alpha+1}}},
		\end{equation*}
		which implies
		\begin{equation*}
			\mbP\Bigl(\sum_{L=1}^{\infty}\frac{2^{-L}}{L}\frac{\norm{\rho-\rho^{\delta}}_{C_{\eta;T_{L}}\mcC^{\alpha+1}}}{1+\norm{\rho-\rho^{\delta}}_{C_{\eta;T_{L}}\mcC^{\alpha+1}}}>\lambda\Bigr)\leq\mbP(D^{\mcC^{\alpha+1}}_{\eta;T}(\rho,\rho^{\delta})>\lambda)\to0.
		\end{equation*}
		\paragraph{Proof of~\eqref{eq:minimum_factorize}.}
		To prove~\eqref{eq:minimum_factorize}, we carry out a case distinction. Case 1, assume $a\leq1$, then
		\begin{align*}
			x\leq a&\implies(1\wedge a)(1\wedge x)=ax\leq x=(x\wedge a)\\
			a\leq x\leq1&\implies (1\wedge a)(1\wedge x)=ax\leq a=(x\wedge a)\\
			1\leq x&\implies (1\wedge a)(1\wedge x)=a=(x\wedge a).
		\end{align*}
		Case 2, assume $1<a$, then
		\begin{align*}
			x\leq1&\implies(1\wedge a)(1\wedge x)=x=(x\wedge a)\\
			1\leq x\leq a&\implies(1\wedge a)(1\wedge x)=1\leq x=(x\wedge a)\\
			a\leq x&\implies(1\wedge a)(1\wedge x)=1\leq a=(x\wedge a).
		\end{align*}
		\paragraph{Proof of~\eqref{eq:minimum_lower_bound}.}
		We establish the equivalent
		\begin{equation*}
			x\leq(1+x)(1\wedge x).
		\end{equation*}
		It suffices to note
		\begin{align*}
			x\in[0,1]&\implies x\leq x+x^{2}=(1+x)(1\wedge x)\\
			x\geq1&\implies x\leq(1+x)=(1+x)(1\wedge x).
		\end{align*}
	\end{details}
\end{proof}
\begin{example}
	The same arguments as in the proof of Theorem~\ref{thm:convergence_renorm_sol}, Part~\ref{it:thm_existence_renormalized}, also allow us to construct the mild solution to~\eqref{eq:gen_rKS_canonical} as a paracontrolled solution to~\eqref{eq:gen_rKS_intro} with enhancement $\mbX^{(\eps)}_{\delta}=(\eps^{1/2}\ti^{\delta},\eps\ty_{\can}^{\delta},\eps^{3/2}\tp_{\can}^{\delta},\eps\tc^{\delta})$ (and heterogeneity $\het=\sqrt{\rdet}$). To obtain a  limit for vanishing correlation lengths $\delta$, one then needs to assume that the noise intensity $\eps$ decays sufficiently quickly depending on the behaviour of the renormalisation. This will be the focus of an upcoming work~\cite{martini_mayorcas_22_LDP}.
\end{example}
\appendix 
\section{Besov and H\"{o}lder--Besov Spaces}\label{app:Besov}
Throughout the following section, all properties are given for mappings or distributions on $\mbT^d$ taking values in $\mbR^n$ for some $d,n\in\mbN$.
\subsection{Besov Spaces}\label{subsec:Besov_spaces}
Applying essentially the same arguments as in the proof of~\cite[Prop.~2.10]{bahouri_chemin_danchin_11} there exists a \emph{dyadic partition of unity}, i.e.\ a pair of non-negative, radially symmetric and compactly supported smooth functions $\varrho_{-1},\varrho_{0}\in C^{\infty}(\mbR^{d};[0,1])$ such that $\supp(\varrho_{-1})\subset B(0,1/2)$, $\supp(\varrho_{0})\subset \{x\in\mbR^d:9/32\leq\abs{x}\leq 1\}$ and $\sum_{k=-1}^{\infty}\varrho_{k}(x)=1$ for all $x\in\mbR^{d}$, where we denote $\varrho_k(x)\defeq\varrho_0(2^{-k}x)$ for each $k\in \mbN$.

For every $k\geq-1$ we define the Littlewood--Paley block $\Delta_{k}$ to be the Fourier multiplier $\Delta_{k}u\defeq\msF^{-1}(\varrho_{k}\msF u)$ and set
\begin{equation*}
	\Delta_{<k}u\defeq\sum_{l=-1}^{k-1}\Delta_{l}u.
\end{equation*}
As with the Fourier transform, we initially define these operators on smooth functions and then extend them by duality to $\mcS'(\mbT^d;\mbR^n)$.
\begin{definition}[Besov spaces]\label{def:Besov_space}
	Let $\alpha\in\mbR$ and $p,q\in[1,\infty]$. We define the non-homogeneous Besov 
	space $\mcB^{\alpha}_{p,q}(\mbT^{d})$ to be the completion of the smooth functions $C^{\infty}(\mbT^{d})$ under the norm
	\begin{equation*}
		\norm{u}_{\mcB^{\alpha}_{p,q}(\mbT^{d})}\defeq\norm{(2^{k\alpha}\norm{\Delta_{k}u}_{L^{p}(\mbT^{d})})_{k\in\mbN_{-1}}}_{\ell^{q}},
	\end{equation*}
	which is extended to vector resp.\ matrix-valued functions in a natural componentwise manner. Here $\ell^{q}$ denotes the usual space of $q$-summable sequences (or bounded when $q=\infty$.) When $p=q=\infty$ we recall the shorthand $\mcC^{\alpha}(\mbT^{d})\defeq\mcB^\alpha_{\infty,\infty}(\mbT^{d})$ and call $\mcC^{\alpha}(\mbT^{d})$ the H\"{o}lder--Besov space.
\end{definition}
\begin{remark}
	Note that the dyadic partition of unity obtained by~\cite[Prop.~2.10]{bahouri_chemin_danchin_11} is built from $\widetilde{\varrho}_{-1}$, $\widetilde{\varrho}_{0}$ with $\supp(\widetilde{\varrho}_{-1})\subset B(0,4/3)$ and $\supp(\widetilde{\varrho}_{0})\subset \{x\in\mbR^d:3/4\leq \abs{x}\leq 8/3\}$. However, for our purposes it is convenient to rescale these functions by a factor of $3/8$ so that the only integer in the support of $\varrho_{-1}$ is $0$. Since the Besov spaces are independent of the chosen dyadic partition of unity~\cite[Cor.~2.70]{bahouri_chemin_danchin_11} this change is harmless.
\end{remark}
Besov spaces enjoy a number of useful properties which we list below. Proofs of the following statements can be found in~\cite{bahouri_chemin_danchin_11,gubinelli_15_GIP}.
\begin{enumerate}
	\item Embeddings: There exists a constant $C>0$ such that for any $\alpha \in \mbR$, $1\leq p_1\leq p_2\leq \infty$ and $1\leq q_1\leq q_2 \leq \infty$,
	\begin{equation}\label{eq:Besov_cts_embedding}
		\norm{u}_{\mcB^{\alpha-d(1/p_1-1/p_2)}_{p_2,q_2}} \leq C \norm{u}_{\mcB^{\alpha}_{p_1,q_1}}.
	\end{equation}
	We also have the following continuous embeddings,
	\begin{align*}
		\norm{u}_{\mcB^\alpha_{p,q}}&\lesssim\norm{u}_{\mcB^{\alpha'}_{p,q}},\quad \alpha<\alpha' \in \mbR,\\
		%
		%
		\norm{u}_{\mcB^\alpha_{p,q}}&\lesssim\norm{u}_{\mcB^{\alpha'}_{p,q'}},\quad \alpha<\alpha' \in \mbR,~q< q'\in [1,\infty].
	\end{align*}
	\item Relations to $L^{p}(\mbT^{d})$-spaces: For all $p\in [1,\infty]$, one has,
	\begin{align*}
		\norm{f}_{\mcB^{0}_{p,\infty}}\lesssim\norm{f}_{L^{p}}\lesssim\norm{f}_{\mcB^{0}_{p,1}}.
	\end{align*}
\end{enumerate}
\begin{details}
	We show the separability of the Besov space $\mcB_{p,q}^{\alpha}(\mbT^{d})$ for $p,q\in[1,\infty]$ and $\alpha\in\mbR$.
	\begin{lemma}\label{lem:Besov_separable}
		Let $p,q\in[1,\infty]$ and $\alpha\in\mbR$. It holds that 
		\begin{equation}\label{eq:Besov_characterization}
			\mcB_{p,q}^{\alpha}(\mbT^{d})=\Bigl\{u\in\mcS'(\mbT^{d}):\norm{u}_{\mcB_{p,q}^{\alpha}}<\infty,\lim_{j\to\infty}2^{j\alpha}\norm{\Delta_{j}u}_{L^{p}}=0\Bigr\}
		\end{equation}
		and each $\mcB_{p,q}^{\alpha}(\mbT^{d})$ is separable. 
	\end{lemma}
	\begin{proof}
		Recall (see Definition~\ref{def:Besov_space}) that $\mcB_{p,q}^{\alpha}(\mbT^{d})$ is defined as the closure of $C^{\infty}(\mbT^{d})$ under the Besov norm $u\mapsto\norm{u}_{\mcB_{p,q}^{\alpha}}\defeq\norm{(2^{k\alpha}\norm{\Delta_{k}u}_{L^{p}})_{k\in\mbN_{-1}}}_{l^{q}}$. We first prove the characterization~\eqref{eq:Besov_characterization}.
		
		Assume $u\in\mcB_{p,q}^{\alpha}(\mbT^{d})$, we need to show that $\lim_{j\to\infty}2^{j\alpha}\norm{\Delta_{j}u}_{L^{p}}=0$. By definition, there exist a sequence $(u_{n})_{n\in\mbN}$ such that $u_{n}\in C^{\infty}(\mbT^{d})$ for every $n\in\mbN$ and $u_{n}\to u\in\mcB_{p,q}^{\alpha}(\mbT^{d})$ as $n\to\infty$. It follows that for each $j\in\mbN_{-1}$,
		\begin{equation*}
			2^{j\alpha}\norm{\Delta_j u}_{L^p}\leq2^{j\alpha}\norm{\Delta_j u-\Delta_ju_n}_{L^p}+2^{j\alpha}\norm{\Delta_j u_n}_{L^p}\leq \norm{u-u_n}_{\mcB_{p,q}^{\alpha}}+2^{j\alpha}\norm{\Delta_j u_n}_{L^p}.
		\end{equation*}
		Letting $j\to\infty$, we obtain $\lim_{j\to\infty}2^{j\alpha}\norm{\Delta_j u}_{L^p}=\norm{u-u_n}_{\mcB_{p,q}^{\alpha}}$, subsequently letting $n\to\infty$ yields $\lim_{j\to\infty}2^{j\alpha}\norm{\Delta_j u}_{L^p}=0$.
		
		Conversely assume that $u\in\mcS'(\mbT^{d})$ satisfies $\norm{u}_{\mcB_{p,q}^{\alpha}}<\infty$ and $\lim_{j\to\infty}2^{j\alpha}\norm{\Delta_{j}u}_{L^{p}}=0$. We need to show that there exists a sequence $(u_{n})_{n\in\mbN}$ such that $u_{n}\in C^{\infty}(\mbT^{d})$ for every $n\in\mbN$ and $\norm{u-u_{n}}_{\mcB_{p,q}^{\alpha}}\to0$ as $n\to\infty$. We define $u_{n}\defeq\sum_{j=-1}^{n-1}\Delta_{j}u\in C^{\infty}(\mbT^{d})$ and apply~\cite[Thm.~21.18]{vanzuijlen_22}: For $q=\infty$ we deduce $\norm{u-u_{n}}_{\mcB_{p,\infty}^{\alpha}}\lesssim\sup_{j\geq n}(2^{j\alpha}\norm{\Delta_{j}u}_{L^{p}})$, which vanishes as $n\to0$, since
		\begin{equation*}
			\lim_{n\to\infty}\sup_{j\geq n}(2^{j\alpha}\norm{\Delta_{j}u}_{L^{p}})=\limsup_{n\to\infty}2^{n\alpha}\norm{\Delta_{n}u}_{L^{p}}=\lim_{j\to\infty}2^{j\alpha}\norm{\Delta_{j}u}_{L^{p}}=0.
		\end{equation*}
		For $q<\infty$ we deduce $\norm{u-u_{n}}_{\mcB_{p,q}^{\alpha}}^{q}\lesssim\sum_{j=n}^{\infty}2^{jq\alpha}\norm{\Delta_{j}u}_{L^{p}}^{q}$, which vanishes as $n\to0$, since $\norm{u}_{\mcB_{p,q}^{\alpha}}<\infty$. This yields the characterization~\eqref{eq:Besov_characterization}.
		
		Next we show the separability of each $\mcB_{p,q}^{\alpha}(\mbT^{d})$. If $p<\infty$, let $D$ be a countable, dense subset of $L^{p}(\mbT^{d})$; and if $p=\infty$, let $D$ be a countable, dense subset of $C(\mbT^{d})$. We define the countable set 
		\begin{equation*}
			H\defeq\Bigl\{\sum_{j=-1}^{J}\Delta_jh_j:J\in\mbN_{-1},~h_j\in D,~j=1,\ldots,J\Bigr\}
		\end{equation*}
		and show that $H$ is dense in $C^{\infty}(\mbT^{d})\subset\mcB_{p,q}^{\alpha}(\mbT^{d})$. 
		
		Assume first $q=\infty$. Let $u\in C^{\infty}(\mbT^{d})$, then for each $\eps>0$ there exists a $J\in\mbN_{-1}$ such that $2^{j\alpha}\norm{\Delta_{j}u}_{L^{p}}<\eps$ for every $j>J$. For $j=-1,\ldots,J$, let $h_{j}\in D$ be such that $\norm{u-h_{j}}_{L^{p}}<\eps2^{-j\alpha}$, which combined with Poisson's summation formula and Young's convolution inequality yields
		\begin{equation*}
			\norm{\Delta_ju-\Delta_jh_{j}}_{L^p}\leq(\norm{\mathscr{F}^{-1}_{\mbR^d}\varrho_{0}}_{L^1}\vee\norm{\mathscr{F}^{-1}_{\mbR^d}\varrho_{-1}}_{L^1})\norm{u-h_{j}}_{L^p}\lesssim\eps2^{-j\alpha}.
		\end{equation*}
		It then follows that
		\begin{equation*}
			\Bigl\lVert u-\sum_{j=-1}^{J}\Delta_j h_{j}\Bigr\rVert_{\mcB_{p,\infty}^{\alpha}}\lesssim\sup_{j=-1,\ldots,J}\Bigl(2^{j\alpha}\norm{\Delta_{j}u-\Delta_{j}h_{j}}_{L^{p}}\Bigr)\vee\sup_{j=J+1,\ldots,\infty}\Bigl(2^{j\alpha}\norm{\Delta_{j}u}_{L^{p}}\Bigr)\lesssim\eps,
		\end{equation*}
		which shows that $H$ is dense in $\mcB_{p,\infty}^{\alpha}(\mbT^{d})$.
		
		Assume next $q<\infty$. Let $u\in C^{\infty}(\mbT^{d})$, then for each $\eps>0$ there exists a $J\in\mbN_{-1}$ such that $2^{j\alpha}\norm{\Delta_ju}_{L^p}<\eps2^{-j}$ for every $j>J$. For $j=-1,\ldots,J$, let $h_{j}\in D$ be such that $\norm{u-h_{j}}_{L^p}<\eps2^{-j(\alpha+1)}$, which combined with Poisson's summation formula and Young's convolution inequality yields
		\begin{equation*}
			\norm{\Delta_ju-\Delta_jh_j}_{L^p}\leq(\norm{\mathscr{F}^{-1}_{\mbR^{d}}\varrho_{0}}_{L^1}\vee\norm{\mathscr{F}^{-1}_{\mbR^{d}}\varrho_{-1}}_{L^1})\norm{u-h_j}_{L^p}\lesssim\eps2^{-j(\alpha+1)}.
		\end{equation*}
		It then follows that
		\begin{equation*}
			\Bigl\lVert u-\sum_{j=-1}^{J}\Delta_jh_j\Bigr\rVert_{\mcB_{p,q}^{\alpha}}^{q}\lesssim\sum_{j=-1}^{J}2^{jq\alpha}\norm{\Delta_{j}u-\Delta_{j}h_{j}}_{L^{p}}^{q}+\sum_{j=J+1}^{\infty}2^{jq\alpha}\norm{\Delta_{j}u}_{L^{p}}^{q}\lesssim\eps^{q}\sum_{j=-1}^{\infty}2^{-jq}\lesssim\eps^{q},
		\end{equation*}
		which shows that $H$ is dense in $\mcB_{p,q}^{\alpha}(\mbT^{d})$. This yields the claim.
	\end{proof}
	The next lemma presents a natural criterion that allows us to find elements of $\mcB_{p,q}^{\alpha}(\mbT^{d})$.
	\begin{lemma}\label{lem:Besov_criterion}
		Let $p,q\in[1,\infty]$ and $\alpha<\alpha'\in\mbR$. 
		\begin{itemize}
			\item If $q\in[1,\infty)$, then each $u\in\mcS'(\mbT^{d})$ such that $\norm{u}_{\mcB_{p,q}^{\alpha}}<\infty$ is an element of $\mcB_{p,q}^{\alpha}(\mbT^{d})$.
			\item If $q=\infty$, then each $u\in\mcS'(\mbT^{d})$ such that $\norm{u}_{\mcB_{p,\infty}^{\alpha'}}<\infty$ is an element of $\mcB_{p,\infty}^{\alpha}(\mbT^{d})$.
		\end{itemize}
	\end{lemma}
	\begin{proof}
		Let $p,q\in[1,\infty]$ and $\alpha\in\mbR$. To prove the first claim, let $q<\infty$ and $u\in\mcS'(\mbT^{d})$ be such that $\norm{u}_{\mcB_{p,q}^{\alpha}}<\infty$. It follows immediately that $\lim_{j\to\infty}2^{j\alpha}\norm{\Delta_{j}u}_{L^{p}}=0$, hence $u\in\mcB_{p,q}^{\alpha}(\mbT^{d})$ by Lemma~\ref{lem:Besov_separable}. To prove the second claim, let $q=\infty$, $\alpha<\alpha'\in\mbR$ and $u\in\mcS'(\mbT^{d})$ be such that $\norm{u}_{\mcB_{p,\infty}^{\alpha'}}<\infty$. It follows that $\norm{u}_{\mcB_{p,\infty}^{\alpha}}<\infty$ and
		\begin{equation*}
			2^{j\alpha}\norm{\Delta_{j}u}_{L^p}\leq2^{j(\alpha-\alpha')}\norm{u}_{\mcB_{p,\infty}^{\alpha'}}\to0\quad\text{as}~j\to\infty,
		\end{equation*}
		hence $u\in\mcB_{p,\infty}^{\alpha}(\mbT^{d})$ by Lemma~\ref{lem:Besov_separable}. This yields the claim.
	\end{proof}
\end{details}
We regularly work in a scale of interpolation spaces which relate temporal and spatial regularity and are suitable for solutions to parabolic PDEs.
\begin{definition}[Interpolation spaces]\label{def:interpolation_space}
	Let $T>0$, $\eta\geq0$, $\alpha\in\mbR$ and $\kappa\in(0,1]$. We define the norm
	\begin{equation*}
		\norm{u}_{\msL_{\eta;T}^{\kappa}\mcC^{\alpha}}\defeq\max\{\norm{u}_{C_{\eta;T}^{\kappa}\mcC^{\alpha-2\kappa}},\norm{u}_{C_{\eta;T}\mcC^{\alpha}}\},
	\end{equation*}
	and the spaces
	\begin{equation*}
		\msL_{\eta;T}^{\kappa}\mcC^{\alpha}(\mbT^{d})=C_{\eta;T}^{\kappa}\mcC^{\alpha-2\kappa}(\mbT^{d})\cap C_{\eta;T}\mcC^{\alpha}(\mbT^{d}).
	\end{equation*}
	We set
	\begin{equation*}
		\msL_{T}^{\kappa}\mcC^{\alpha}(\mbT^{d})\defeq\msL_{0;T}^{\kappa}\mcC^{\alpha}(\mbT^{d})\defeq C_{T}^{\kappa}\mcC^{\alpha-2\kappa}(\mbT^{d})\cap C_{T}\mcC^{\alpha}(\mbT^{d})
	\end{equation*}
	 and by an abuse of notation understand $\msL^{0}_{\eta;T}\mcC^{\alpha} (\mbT^{d})=C_{\eta;T}\mcC^{\alpha}(\mbT^{d})$.
\end{definition}
\subsection{Paraproducts}\label{subsec:paraproducts}
For $u,v\in C^{\infty}(\mbT^d;\mbR)$ we define the paraproduct $\pa$ and resonant product $\re$ by
\begin{equation*}
	u\pa v\defeq\sum_{k\geq-1}\Delta_{<k-1}u\Delta_{k}v,\qquad u\re v\defeq\sum_{\abs{k-l}\leq1}\Delta_{l}u\Delta_{k}v.
\end{equation*}
Formally one has the decomposition $uv= u\pa v+ v\pa u + u\re v$. Conditions under which this decomposition is valid for (time-dependent) distributions $u$ and $v$ are given by the following version of Bony's estimates. These operators naturally extend to vector-valued and matrix-valued objects as either inner or outer products. Where the precise meaning is not clear from context it will be specified in the text.
\begin{lemma}[Bony's estimates]\label{lem:Bony}
	Let $T>0$ and $\eta,\eta_1,\eta_2\geq0$ be such that $\eta=\eta_1+\eta_2$.
	\begin{itemize}
		\item Let $\beta \in \mbR$, $u\in C_{\eta_1;T}L^\infty(\mbT^d;\mbR)$ and $v\in C_{\eta_2;T}\mcC^{\beta}(\mbT^d;\mbR)$, then $u\pa v\in C_{\eta;T}\mcC^{\beta}(\mbT^{d};\mbR)$ and
		\begin{equation*}
			\norm{u\pa v}_{C_{\eta;T}\mcC^{\beta}}\lesssim_{\beta}\norm{u}_{C_{\eta_1;T}L^\infty}\norm{v}_{C_{\eta_2;T}\mcC^{\beta}}.
		\end{equation*}
		\item Let $\beta\in\mbR$, $\alpha<0$, $u\in C_{\eta_1;T}\mcC^{\alpha}(\mbT^d;\mbR)$ and $v\in C_{\eta_2;T}\mcC^{\beta}(\mbT^d;\mbR)$, then $u\pa v\in C_{\eta;T}\mcC^{\alpha+\beta}(\mbT^{d};\mbR)$ and
		\begin{equation*}
			\norm{u\pa v}_{C_{\eta;T}\mcC^{\alpha+\beta}}\lesssim_{\alpha,\beta}\norm{u}_{C_{\eta_1;T}\mcC^{\alpha}}\norm{v}_{C_{\eta_2;T}\mcC^{\beta}}.
		\end{equation*}
		\item Let $\alpha,\beta\in\mbR$ with $\alpha+\beta>0$, $u\in C_{\eta_1;T}\mcC^{\alpha}(\mbT^d;\mbR)$ and $v\in C_{\eta_2;T}\mcC^{\beta}(\mbT^d;\mbR)$, then $u\re v\in C_{\eta;T}\mcC^{\alpha+\beta}(\mbT^{d};\mbR)$ and
		\begin{equation*}
			\norm{u\re v}_{C_{\eta;T}\mcC^{\alpha+\beta}}\lesssim_{\alpha,\beta}\norm{u}_{C_{\eta_1;T}\mcC^{\alpha}}\norm{v}_{C_{\eta_2;T}\mcC^{\beta}}.
		\end{equation*}
	\end{itemize}
\end{lemma}
\begin{proof}
	The result is a direct consequence of~\cite[Lem.~2.1]{gubinelli_15_GIP}.
\end{proof}
\subsection{Parabolic and Elliptic Regularity Estimates}
We will make use of the following interpolation inequality. Let $x\geq0$ and $\gamma\in[0,1]$, then
\begin{equation}\label{eq:interpolation}
	0\leq1-\euler^{-x}\leq x^{\gamma}.
\end{equation}
We also apply the following rapid-decay inequality. For any $r>0$, uniformly in $x\geq0$,
\begin{equation}\label{eq:rapid_decay}
	x^{r}\euler^{-x}\lesssim 1.
\end{equation}
The operators $P$, $\mcI$ and $\Phi$ introduced in Subsection~\ref{sec:notations} and their accompanying kernels $\msH$ and $\msG$ can be generalized to $\mbT^{d}$ \emph{mutatis mutandis}. For all $n\in\mbN$, we also apply $\mcI$ to $f=(f_1,\ldots,f_n)\from[0,T]\times\mbT^{d}\to\mbR^{n}$ by setting $\mcI[f]=(\mcI[f_1],\ldots,\mcI[f_n])$.
\begin{lemma}\label{lem:heat_flow}
	Let $\alpha\leq\beta\in\mbR$ and $p,q,p',q'\in[1,\infty]$ be such that $p\geq p'$ and $q\geq q'$. Then for any $t>0$,
	\begin{equation}\label{eq:HeatFlow}
		\norm{P_tf}_{\mcB^{\beta}_{p,q}}\lesssim\bigl(1\vee t^{-\frac{\beta -\alpha}{2}}\bigr)\bigl(1\vee t^{-\frac{d}{2}(\frac{1}{p'}-\frac{1}{p})}\bigr)\norm{f}_{\mcB_{p',q'}^{\alpha}}.
	\end{equation}
	Secondly, if $\alpha\leq\beta\leq\alpha+2$ then for any $t>0$,
	\begin{equation}\label{eq:HeatFlowMinusId}
		\norm{(P_{t}-1)f}_{\mcB^{\alpha}_{p,q}} \lesssim t^{\frac{\beta-\alpha}{2}} \|f\|_{\mcB^{\beta}_{p,q}}.
	\end{equation}
\end{lemma}
\begin{proof}
	We first show~\eqref{eq:HeatFlow}, which is an easy consequence of the Besov embedding and the regularising effect of the heat flow. By~\eqref{eq:Besov_cts_embedding} for any $\beta \in \mbR$ and $p,q,p',q' \in [1,\infty]$ with $p\geq p'$ and $q\geq q'$, it holds that,
	\begin{equation*}
		\norm{P_{t}f}_{\mcB^{\beta}_{p,q}}\lesssim\norm{P_{t}f}_{\mcB^{\beta+d(1/p'-1/p)}_{p',q'}}.
	\end{equation*}
	We now apply the semigroup property and the regularizing effect of the heat flow (\cite[Lem.~2.4]{bahouri_chemin_danchin_11} \&~\eqref{eq:rapid_decay}),
	\begin{equation*}
		\norm{P_{t}f}_{\mcB^{\beta+d(1/p'-1/p)}_{p',q'}}\lesssim \bigl(1\vee t^{-\frac{d}{2}(\frac{1}{p'}-\frac{1}{p})}\bigr)\norm{P_{t/2}f}_{\mcB^{\beta}_{p',q'}}\lesssim\bigl(1\vee t^{-\frac{d}{2}(\frac{1}{p'}-\frac{1}{p})}\bigr)\bigl(1\vee t^{-\frac{\beta -\alpha}{2}}\bigr)\norm{f}_{\mcB^{\alpha}_{p',q'}},
	\end{equation*}
	which yields~\eqref{eq:HeatFlow}. 
	\begin{details}
		We need to make sure that~\cite[Lem.~2.4]{bahouri_chemin_danchin_11} also applies to $f$ on the torus. By the Poisson summation formula, 
		\begin{equation*}
			\sum_{\om\in\mbZ^{d}}\euler^{2\uppi\upi\inner{\om}{x}}\euler^{-t\abs{2\uppi\om}^{2}}=\frac{1}{(4\uppi t)^{d/2}}\sum_{\om\in\mbZ^{d}}\euler^{-\frac{\abs{\om+x}^{2}}{4t}}
		\end{equation*}
		so that for $f\in C^{\infty}(\mbT^{d})$,
		\begin{equation*}
			\sum_{\om\in\mbZ^{d}}\euler^{2\uppi\upi\inner{\om}{x}}\euler^{-t\abs{2\uppi\om}^{2}}\hat{f}(\om)=\int_{\mbT^d}\sum_{\om\in\mbZ^{d}}\euler^{2\uppi\upi\inner{\om}{x-y}}\euler^{-t\abs{2\uppi\om}^{2}}f(y)\dd y=\frac{1}{(4\uppi t)^{d/2}}\sum_{\om\in\mbZ^{d}}\int_{\mbT^{d}}\euler^{-\frac{\abs{\om+x-y}^{2}}{4t}}f(y)\dd y.
		\end{equation*}
		Using that $f$ is periodic, we obtain
		\begin{equation*}
			P_tf(x)=\sum_{\om\in\mbZ^{d}}\euler^{2\uppi\upi\inner{\om}{x}}\euler^{-t\abs{2\uppi\om}^{2}}\hat{f}(\om)=\int_{\mbR^{d}}\frac{1}{(4\uppi t)^{d/2}}\euler^{-\frac{\abs{x-y}^{2}}{4t}}f(y)\dd y,
		\end{equation*}
		which we can use to extend~\cite[Lem.~2.4]{bahouri_chemin_danchin_11} to $f\in L^{p}(\mbT^{d})$.
	\end{details}
	The bound~\eqref{eq:HeatFlowMinusId} can be found in~\cite[Prop.~A.13]{mourrat_weber_17_CDI}, who cite~\cite[Prop.~6]{mourrat_weber_15_GWP} for a proof in the full space. We provide a short argument. We consider the Littlewood--Paley blocks $\Delta_{k}(P_t-1)u=(P_t-1)\Delta_{k}u$ for $k\in\mbN_{-1}$. Since $(P_t-1)\Delta_{-1}u=0$, we may assume $k\in\mbN$. We apply~\cite[Lem.~2.4 \& Lem.~2.1]{bahouri_chemin_danchin_11} to obtain the existence of some $c>0$ such that
	\begin{equation*}
		\begin{split}
			\norm{(P_t-1)\Delta_{k}f}_{L^p}&\leq\int_{0}^{t}\norm{\partial_s P_s\Delta_{k}f}_{L^p}\dd s=\int_{0}^{t}\norm{P_s\Delta\Delta_{k}f}_{L^p}\dd s\\
			&\lesssim\norm{\Delta\Delta_k f}_{L^{p}}\int_{0}^{t}\euler^{-cs2^{2k}}\dd s\lesssim\norm{\Delta_kf}_{L^{p}}2^{2k}\int_{0}^{t}\euler^{-cs2^{2k}}\dd s\\
			&\lesssim\norm{\Delta_kf}_{L^{p}}(1-\euler^{-ct2^{2k}})\lesssim\norm{\Delta_kf}_{L^{p}}t^{\frac{\beta-\alpha}{2}}2^{k(\beta-\alpha)}.
		\end{split}
	\end{equation*}
	In the last inequality, we applied~\eqref{eq:interpolation} with $(\beta-\alpha)/2\in[0,1]$. This yields the claim using the definition of the $\mcB_{p,q}^{\alpha}(\mbT^{d})$-norm.
\end{proof}
We can now establish Schauder estimates similar to~\cite[Lem.~A.9]{gubinelli_15_GIP} and~\cite[Prop.~2.7]{catellier_chouk_18}.
\begin{lemma}\label{lem:Schauder}
	Let $T>0$, $\eta\geq0$, $\alpha\leq\beta\in\mbR$, $\kappa\in[0,1]$ and $p,q\in[1,\infty]$. Then the following hold
	\begin{enumerate}
		\item\label{it:Schauder_I}
		If $\frac{\beta-\alpha}{2}+\frac{d}{2p}\leq\eta$, then
		\begin{equation}\label{eq:Schauder_1}
			\norm{Pf}_{\msL_{\eta;T}^{\kappa}\mcC^{\beta}}\lesssim (1\vee T^{-\frac{\beta-\alpha}{2}}) (1\vee T^{-\frac{d}{2p}})(1\wedge T)^{\eta}\norm{f}_{\mcB_{p,q}^{\alpha}}
		\end{equation}
		and $P\from\mcB_{p,q}^{\alpha}(\mbT^{d})\to\msL_{\eta;T}^{\kappa}\mcC^{\beta}(\mbT^{d})$ is a continuous map.
		\item\label{it:Schauder_II}
		If $\eta'\in[0,1)$, $\eta\leq\eta'$ and $\beta<\alpha+2$ are such that $\frac{\beta-\alpha}{2} \vee \kappa \leq 1-(\eta'-\eta)$, then
		\begin{equation}\label{eq:Schauder_2}
			\norm{\mcI[f]}_{\msL_{\eta;T}^{\kappa}\mcC^{\beta}}\lesssim_{T}\norm{f}_{C_{\eta';T}\mcC^{\alpha}}
		\end{equation}
		and $\mcI\from C_{\eta';T}\mcC^{\alpha}(\mbT^{d})\to\msL_{\eta;T}^{\kappa}\mcC^{\beta}(\mbT^{d})$ is a continuous map.
		In particular if $T\leq1$, then
		\begin{equation*}
			\norm{\mcI[f]}_{\msL_{\eta;T}^{\kappa}\mcC^{\beta}}\lesssim( T^{1-\frac{\beta-\alpha}{2}-(\eta'-\eta)}\vee T^{1-\kappa-(\eta'-\eta)})\norm{f}_{C_{\eta';T}\mcC^{\alpha}}.
		\end{equation*}
		\item\label{it:Schauder_continuous}
		Furthermore, $\mcI\from C_T\mcC^{\alpha}(\mbT^{d})\to \msL_{T}^{\kappa}\mcC^{\alpha+2}(\mbT^{d})$ is a continuous map.
	\end{enumerate}
\end{lemma}
\begin{proof}
	The proofs of~\eqref{eq:Schauder_1} and~\eqref{eq:Schauder_2} are simple consequences of Lemma~\ref{lem:heat_flow}, the semigroup property and the definition of the interpolation spaces, Definition~\ref{def:interpolation_space}. The continuity of $t\mapsto (1\wedge t)^{\eta}P_{t}u_0$ and $t\mapsto(1\wedge t)^{\eta}\mcI[f]_{t}$ follows from the fact that the same statement is true for smooth $u_{0}$, $f$, and then by taking limits along a smooth approximating sequence.
	\begin{details}
		\paragraph{Proof of Claim~\ref{it:Schauder_I}.}
		We first consider $Pf$ and prove~\eqref{eq:Schauder_1}. Assume $\frac{\beta-\alpha}{2}+\frac{d}{2p}\leq\eta$. To control the $C_{\eta;T}\mcC^{\beta}(\mbT^{d})$-norm of $Pf$, we use~\eqref{eq:HeatFlow} to bound
		\begin{equation*}
			\norm{Pf}_{C_{\eta;T}\mcC^{\beta}}\lesssim\sup_{t\in[0,T]}(1\vee t^{-\frac{\beta-\alpha}{2}})(1\vee t^{-\frac{d}{2p}})(1\wedge t)^{\eta}\norm{f}_{\mcB_{p,q}^{\alpha}}.
		\end{equation*}
		Next we control the H\"{o}lder seminorm. We compute for $s\leq t$ using~\eqref{eq:HeatFlowMinusId} and~\eqref{eq:HeatFlow},
		\begin{equation*}
			\begin{split}
				\norm{P_tf-P_sf}_{\mcC^{\beta-2\kappa}}=\norm{(P_{t-s}-1)P_sf}_{\mcC^{\beta-2\kappa}}&\lesssim\abs{t-s}^{\kappa}\norm{P_sf}_{\mcC^{\beta}}\\
				&\lesssim\abs{t-s}^{\kappa}(1\vee s^{-\frac{\beta-\alpha}{2}})(1\vee s^{-\frac{d}{2p}})\norm{f}_{\mcB_{p,q}^\alpha},
			\end{split}
		\end{equation*}
		which implies
		\begin{equation*}
			\sup_{s\not=t\in[0,T]}(1\wedge s\wedge t)^{\eta}\frac{\norm{P_tf-P_sf}_{\mcC^{\beta-2\kappa}}}{\abs{t-s}^{\kappa}}\lesssim\sup_{t\in[0,T]}(1\vee t^{-\frac{\beta-\alpha}{2}})(1\vee t^{-\frac{d}{2p}})(1\wedge t)^{\eta}\norm{f}_{\mcB_{p,q}^{\alpha}}.
		\end{equation*}
		Using that $[0,T]\ni t\mapsto(1\vee t^{-\frac{\beta-\alpha}{2}})(1\vee t^{-\frac{d}{2p}})(1\wedge t)^{\eta}$ is non-decreasing and continuous, we obtain~\eqref{eq:Schauder_1}.
		
		To establish the continuity of the operator $P\from\mcB_{p,q}^{\alpha}(\mbT^{d})\to\msL_{\eta;T}^{\kappa}\mcC^{\beta}(\mbT^{d})$, it suffices to show that $Pf\in\msL_{\eta;T}^{\kappa}\mcC^{\beta}(\mbT^{d})$ for every $f\in\mcB_{p,q}^{\alpha}(\mbT^{d})$. Since $C_{\eta;T}^{\kappa}\mcC^{\beta-2\kappa}(\mbT^{d})$ is defined as those functions that are finite under the norm $\norm{\place}_{C_{\eta;T}^{\kappa}\mcC^{\beta-2\kappa}}$, it is clear by~\eqref{eq:Schauder_1} that $Pf\in C_{\eta;T}^{\kappa}\mcC^{\beta-2\kappa}(\mbT^{d})$. Hence, it remains to establish $Pf\in C_{\eta;T}\mcC^{\beta}(\mbT^{d})$ (i.e.\ that $[0,T]\ni t\mapsto(1\wedge t)^{\eta}P_{t}f\in\mcC^{\beta}(\mbT^{d})$ is continuous.) By the definition of $\mcB_{p,q}^{\alpha}(\mbT^{d})$, there exists a sequence $(f_{n})_{n\in\mbN}$ of smooth approximations such that $C^{\infty}(\mbT^{d})\ni f_{n}\to f\in\mcB_{p,q}^{\alpha}(\mbT^{d})$. Since $f_{n}\to f\in\mcB_{p,q}^{\alpha}(\mbT^{d})$, we obtain by~\eqref{eq:Schauder_1} that $\lim_{n\to\infty}\norm{Pf-Pf_{n}}_{C_{\eta;T}\mcC^{\beta}}=0$; hence, by the completeness of the continuous functions in the supremum norm, it suffices to show that $Pf_{n}\in C_{\eta;T}\mcC^{\beta}(\mbT^{d})$. 
		
		Let $\eps\in(0,1]$, using that $f_{n}\in C^{\infty}(\mbT^{d})\subset\mcC^{\beta+2\eps}(\mbT^{d})$, we apply~\eqref{eq:Schauder_1} to obtain
		\begin{equation*}
			\norm{Pf_{n}}_{\msL_{T}^{\eps}\mcC^{\beta+2\eps}}\lesssim\norm{f_{n}}_{\mcC^{\beta+2\eps}},
		\end{equation*}
		which yields $Pf_{n}\in C_{T}^{\eps}\mcC^{\beta}(\mbT^{d})\subset C_{T}\mcC^{\beta}(\mbT^{d})\subset C_{\eta;T}\mcC^{\beta}(\mbT^{d})$.
		\paragraph{Proof of Claim~\ref{it:Schauder_II}.}
		Next we consider $\mcI[f]$ and prove~\eqref{eq:Schauder_2}. Assume $\eta\leq\eta'\in[0,1)$ and $\beta<\alpha+2$. To control the $C_{\eta;T}\mcC^{\beta}(\mbT^{d})$-norm of $\mcI[f]$, we use~\eqref{eq:HeatFlow} to bound
		\begin{equation*}
			\norm{\mcI[f]_{t}}_{\mcC^{\beta}}\leq\int_{0}^{t}\norm{P_{t-r}f(r)}_{\mcC^{\beta}}\dd r\lesssim\int_{0}^{t}(1\vee \abs{t-r}^{-\frac{\beta-\alpha}{2}})(1\wedge r)^{-\eta'}\dd r\norm{f}_{C_{\eta';t}\mcC^{\alpha}}.
		\end{equation*}
		Taking the supremum, we obtain
		\begin{equation*}
			\norm{\mcI[f]}_{C_{\eta;T}\mcC^{\beta}}\lesssim\Bigl(\sup_{t\in[0,T]}(1\wedge t)^{\eta}\int_{0}^{t}(1\vee \abs{t-r}^{-\frac{\beta-\alpha}{2}})(1\wedge r)^{-\eta'}\dd r\Bigr)\norm{f}_{C_{\eta';T}\mcC^{\alpha}}.
		\end{equation*}
		To bound this expression, we distinguish the cases $t\in[0,1]$, $t\in(1,2]$ and $t\in(2,\infty)$. Assume $t\in[0,1]$, then
		\begin{equation*}
			(1\wedge t)^{\eta}\int_{0}^{t}(1\vee \abs{t-r}^{-\frac{\beta-\alpha}{2}})(1\wedge r)^{-\eta'}\dd r=t^{\eta}\int_{0}^{t}\abs{t-r}^{-\frac{\beta-\alpha}{2}}r^{-\eta'}\dd r\lesssim t^{1-\frac{\beta-\alpha}{2}-(\eta'-\eta)}.
		\end{equation*}
		Assume $t\in(1,2]$, then
		\begin{equation*}
			\begin{split}
				(1\wedge t)^{\eta}\int_{0}^{t}(1\vee \abs{t-r}^{-\frac{\beta-\alpha}{2}})(1\wedge r)^{-\eta'}\dd r&=\int_{0}^{t-1}r^{-\eta'}\dd r+\int_{t-1}^{1}\abs{t-r}^{-\frac{\beta-\alpha}{2}}r^{-\eta'}\dd r+\int_{1}^{t}\abs{t-r}^{-\frac{\beta-\alpha}{2}}\dd r\\
				&\leq\int_{0}^{t}r^{-\eta'}\dd r+\int_{0}^{t}\abs{t-r}^{-\frac{\beta-\alpha}{2}}r^{-\eta'}\dd r+\int_{0}^{t}\abs{t-r}^{-\frac{\beta-\alpha}{2}}\dd r\\
				&\lesssim t^{1-\eta'}+t^{1-\frac{\beta-\alpha}{2}-\eta'}+t^{1-\frac{\beta-\alpha}{2}}.
			\end{split}
		\end{equation*}
		Assume $t\in(2,\infty)$, then
		\begin{equation*}
			\begin{split}
				(1\wedge t)^{\eta}\int_{0}^{t}(1\vee \abs{t-r}^{-\frac{\beta-\alpha}{2}})(1\wedge r)^{-\eta'}\dd r&=\int_{0}^{1}r^{-\eta'}\dd r+\int_{1}^{t-1}1\dd r+\int_{t-1}^{t}\abs{t-r}^{-\frac{\beta-\alpha}{2}}\dd r\\
				&\leq\int_{0}^{t}r^{-\eta'}\dd r+\int_{0}^{t}1\dd r+\int_{0}^{t}\abs{t-r}^{-\frac{\beta-\alpha}{2}}\dd r\\
				&\lesssim t^{1-\eta'}+t+t^{1-\frac{\beta-\alpha}{2}}.
			\end{split}
		\end{equation*}
		We use that $\frac{\beta-\alpha}{2}\leq1-(\eta'-\eta)$ to ensure $\sup_{t\in[0,1]} t^{1-\frac{\beta-\alpha}{2}-(\eta'-\eta)}<\infty$. All in all, we obtain
		\begin{equation}\label{eq:Schauder_integral_estimate}
			\sup_{t\in[0,T]}(1\wedge t)^{\eta}\int_{0}^{t}(1\vee \abs{t-r}^{-\frac{\beta-\alpha}{2}})(1\wedge r)^{-\eta'}\dd r\lesssim_{T}1.
		\end{equation}
		In particular if $T\leq1$, then
		\begin{equation*}
			\sup_{t\in[0,T]}(1\wedge t)^{\eta}\int_{0}^{t}(1\vee \abs{t-r}^{-\frac{\beta-\alpha}{2}})(1\wedge r)^{-\eta'}\dd r\lesssim T^{1-\frac{\beta-\alpha}{2}-(\eta'-\eta)}.
		\end{equation*}
		Consequently,
		\begin{equation*}
			\norm{\mcI[f]}_{C_{\eta;T}\mcC^{\beta}}\lesssim_{T}\norm{f}_{C_{\eta';T}\mcC^{\alpha}}.
		\end{equation*}
		Next we consider the H\"{o}lder seminorm of $\mcI[f]$ and wish to establish
		\begin{equation}\label{eq:I_seminorm}
			\sup_{s\not=t\in[0,T]}(1\wedge s\wedge t)^{\eta}\frac{\norm{\mcI[f]_{t}-\mcI[f]_{s}}_{\mcC^{\beta-2\kappa}}}{\abs{t-s}^{\kappa}}\lesssim_{T}\norm{f}_{C_{\eta';T}\mcC^\alpha}.
		\end{equation}
		Assume $s\leq t$. To establish~\eqref{eq:I_seminorm}, we decompose $\mcI[f]_{t}-\mcI[f]_{s}=(P_{t-s}-1)\mcI[f]_{s}+\int_{s}^{t}P_{t-r}f(r)\dd r$. We control the first summand using~\eqref{eq:HeatFlowMinusId},
		\begin{equation*}
			\norm{(P_{t-s}-1)\mcI[f]_{s}}_{\mcC^{\beta-2\kappa}}\lesssim\abs{t-s}^{\kappa}\norm*{\int_{0}^{s}P_{s-r}f(r)\dd r}_{\mcC^{\beta}}\lesssim\abs{t-s}^{\kappa}\int_{0}^{s}(1\vee \abs{s-r}^{-\frac{\beta-\alpha}{2}})(1\wedge r)^{-\eta'}\dd r \norm{f}_{C_{\eta';s}\mcC^{\alpha}}.
		\end{equation*}
		Taking the supremum, we obtain
		\begin{equation*}
			\sup_{s\not=t\in[0,T]}(1\wedge s\wedge t)^{\eta}\frac{\norm{(P_{t-s}-1)\mcI[f]_{s}}_{\mcC^{\beta-2\kappa}}}{\abs{t-s}^{\kappa}}\lesssim\Bigl(\sup_{s\not=t\in[0,T]}(1\wedge s)^{\eta}\int_{0}^{s}(1\vee\abs{s-r}^{-\frac{\beta-\alpha}{2}})(1\wedge r)^{-\eta'}\dd r\Bigr)\norm{f}_{C_{\eta';T}\mcC^{\alpha}},
		\end{equation*}
		which we can estimate as in~\eqref{eq:Schauder_integral_estimate}. In the second summand, we distinguish cases. In the case $\alpha\leq\beta-2\kappa$, we apply~\eqref{eq:HeatFlow},
		\begin{equation*}
			\norm*{\int_{s}^{t}P_{t-r}f(r)\dd r}_{\mcC^{\beta-2\kappa}}\lesssim\int_{s}^{t}(1\vee\abs{t-r}^{-\frac{\beta-\alpha}{2}+\kappa})(1\wedge r)^{-\eta'}\dd r\norm{f}_{C_{\eta';t}\mcC^\alpha}.
		\end{equation*}
		Taking the supremum, we obtain
		\begin{equation*}
			\begin{split}
				&\sup_{s\not=t\in[0,T]}(1\wedge s\wedge t)^{\eta}\frac{\norm*{\int_{s}^{t}P_{t-r}f(r)\dd r}_{\mcC^{\beta-2\kappa}}}{\abs{t-s}^{\kappa}}\\
				&\lesssim\Bigl(\sup_{s\not=t\in[0,T]}(1\wedge s\wedge t)^{\eta}\abs{t-s}^{-\kappa}\int_{s}^{t}(1\vee\abs{t-r}^{-\frac{\beta-\alpha}{2}+\kappa})(1\wedge r)^{-\eta'}\dd r\Bigr)\norm{f}_{C_{\eta';T}\mcC^\alpha}.
			\end{split}
		\end{equation*}
		We distinguish the cases $1<t-s$ and $t-s\leq1$. Assume $1<t-s$, then
		\begin{equation*}
			\begin{split}
				&\sup_{s\not=t\in[0,T]}(1\wedge s\wedge t)^{\eta}\abs{t-s}^{-\kappa}\int_{s}^{t}(1\vee \abs{t-r}^{-\frac{\beta-\alpha}{2}+\kappa})(1\wedge r)^{-\eta'}\dd r\\
				&\leq \sup_{t\in[0,T]}(1\wedge t)^{\eta}\int_{0}^{t}(1\vee \abs{t-r}^{-\frac{\beta-\alpha}{2}+\kappa})(1\wedge r)^{-\eta'}\dd r.
			\end{split}
		\end{equation*}
		We can then estimate the remaining integral as in~\eqref{eq:Schauder_integral_estimate}, using that $\beta-\alpha-2\kappa\geq0$ and $\frac{\beta-\alpha}{2}-\kappa\leq1-(\eta'-\eta)$. Assume $t-s\leq1$ and $t\in[0,1]$. We obtain
		\begin{equation*}
			\begin{split}
				&\int_{s}^{t}(1\vee \abs{t-r}^{-\frac{\beta-\alpha}{2}+\kappa})(1\wedge r)^{-\eta'}\dd r=\int_{s}^{t}\abs{t-r}^{-\frac{\beta-\alpha}{2}+\kappa}r^{-\eta'}\dd r\\
				&\leq\abs{t-s}^{\kappa}\int_{0}^{t}\abs{t-r}^{-\frac{\beta-\alpha}{2}}r^{-\eta'}\dd r\lesssim\abs{t-s}^{\kappa}t^{1-\frac{\beta-\alpha}{2}-\eta'}.
			\end{split}
		\end{equation*}
		Assume $t-s\leq1$, $t\in(1,\infty)$ and $s\in[0,1]$. We obtain
		\begin{equation*}
			\begin{split}
				&\int_{s}^{t}(1\vee \abs{t-r}^{-\frac{\beta-\alpha}{2}+\kappa})(1\wedge r)^{-\eta'}\dd r=\int_{s}^{1}\abs{t-r}^{-\frac{\beta-\alpha}{2}+\kappa}r^{-\eta'}\dd r+\int_{1}^{t}\abs{t-r}^{-\frac{\beta-\alpha}{2}+\kappa}\dd r\\
				&\leq\abs{t-s}^{\kappa}\int_{0}^{t}\abs{t-r}^{-\frac{\beta-\alpha}{2}}r^{-\eta'}\dd r+\abs{t-s}^{\kappa}\int_{0}^{t}\abs{t-r}^{-\frac{\beta-\alpha}{2}}\dd r\\
				&\lesssim\abs{t-s}^{\kappa}t^{1-\frac{\beta-\alpha}{2}-\eta'}+\abs{t-s}^{\kappa}t^{1-\frac{\beta-\alpha}{2}}.
			\end{split}
		\end{equation*}
		Assume $t-s\leq1$, $t\in(1,\infty)$ and $s\in(1,\infty)$. We obtain
		\begin{equation*}
			\begin{split}
				\int_{s}^{t}(1\vee \abs{t-r}^{-\frac{\beta-\alpha}{2}+\kappa})(1\wedge r)^{-\eta'}\dd r&=\int_{s}^{t}\abs{t-r}^{-\frac{\beta-\alpha}{2}+\kappa}\dd r\\
				&\leq\abs{t-s}^{\kappa}\int_{0}^{t}\abs{t-r}^{-\frac{\beta-\alpha}{2}}\dd r\lesssim\abs{t-s}^{\kappa}t^{1-\frac{\beta-\alpha}{2}}.
			\end{split}
		\end{equation*}
		Consequently,
		\begin{equation*}
			\sup_{s\not=t\in[0,T]}(1\wedge s\wedge t)^{\eta}\abs{t-s}^{-\kappa}\int_{s}^{t}(1\vee \abs{t-r}^{-\frac{\beta-\alpha}{2}+\kappa})(1\wedge r)^{-\eta'}\dd r\lesssim_{T}1,
		\end{equation*}
		where we used that $\frac{\beta-\alpha}{2}<1-(\eta'-\eta)$ to ensure $\sup_{t\in[0,T]}t^{1-\frac{\beta-\alpha}{2}-(\eta'-\eta)}=T^{1-\frac{\beta-\alpha}{2}-(\eta'-\eta)}<\infty$. In particular if $T\leq 1$, then
		\begin{equation*}
			\sup_{s\not=t\in[0,T]}(1\wedge s\wedge t)^{\eta}\abs{t-s}^{-\kappa}\int_{s}^{t}(1\vee \abs{t-r}^{-\frac{\beta-\alpha}{2}+\kappa})(1\wedge r)^{-\eta'}\dd r\lesssim T^{1-\frac{\beta-\alpha}{2}-(\eta'-\eta)}.
		\end{equation*}
		On the other hand, if $\beta-2\kappa<\alpha$, we estimate instead
		\begin{equation*}
			\norm*{\int_{s}^{t}P_{t-r}f(r)\dd r}_{\mcC^{\beta-2\kappa}}\lesssim\int_{s}^{t}\norm{P_{t-r}f(r)}_{\mcC^{\alpha}}\dd r\lesssim\int_{s}^{t}(1\wedge r)^{-\eta'}\dd r\norm{f}_{C_{\eta';t}\mcC^{\alpha}}.
		\end{equation*}
		Taking the supremum, we obtain
		\begin{equation*}
			\begin{split}
				&\sup_{s\not=t\in[0,T]}(1\wedge s\wedge t)^{\eta}\frac{\norm*{\int_{s}^{t}P_{t-r}f(r)\dd r}_{\mcC^{\beta-2\kappa}}}{\abs{t-s}^{\kappa}}\lesssim\Bigl(\sup_{s\not=t\in[0,T]}(1\wedge s\wedge t)^{\eta}\abs{t-s}^{-\kappa}\int_{s}^{t}(1\wedge r)^{-\eta'}\dd r\Bigr)\norm{f}_{C_{\eta';T}\mcC^{\alpha}}.
			\end{split}
		\end{equation*}
		We estimate
		\begin{equation*}
			\int_{s}^{t}(1\wedge r)^{-\eta'}\dd r\leq (1\wedge s)^{-\eta}\int_{s}^{t}(1\wedge r)^{-(\eta'-\eta)}\dd r.
		\end{equation*}
		From here on, we use $\eta\leq\eta'$ to ensure $\eta'-\eta\geq0$. Assume $t\in[0,1]$, then
		\begin{equation*}
			\int_{s}^{t}(1\wedge r)^{-(\eta'-\eta)}\dd r=\int_{s}^{t}r^{-(\eta'-\eta)}\dd r\lesssim t^{1-(\eta'-\eta)}-s^{1-(\eta'-\eta)}.
		\end{equation*}
		Assume $t\in(1,\infty)$ and $s\in[0,1]$, then
		\begin{equation*}
			\int_{s}^{t}(1\wedge r)^{-(\eta'-\eta)}\dd r=\int_{s}^{1}r^{-(\eta'-\eta)}\dd r+\int_{1}^{t}1\dd r\leq\int_{s}^{t}r^{-(\eta'-\eta)}\dd r+\int_{s}^{t}1\dd r\lesssim t^{1-(\eta'-\eta)}-s^{1-(\eta'-\eta)}+\abs{t-s}.
		\end{equation*}
		Assume $t\in(1,\infty)$ and $s\in(1,\infty)$, then
		\begin{equation*}
			\int_{s}^{t}(1\wedge r)^{-(\eta'-\eta)}\dd r=\int_{s}^{t}1\dd r\leq\abs{t-s}.
		\end{equation*}
		Using that $\kappa\leq1-(\eta'-\eta)$, we obtain
		\begin{equation*}
			\sup_{s\not=t\in[0,T]}\frac{t^{1-(\eta'-\eta)}-s^{1-(\eta'-\eta)}}{\abs{t-s}^{\kappa}}\lesssim T^{1-(\eta'-\eta)-\kappa}.
		\end{equation*}
		Consequently,
		\begin{equation*}
			\sup_{s\not=t\in[0,T]}(1\wedge s\wedge t)^{\eta}\abs{t-s}^{-\kappa}\int_{s}^{t}(1\wedge r)^{-\eta'}\dd r\lesssim_{T}1.
		\end{equation*}
		In particular if $T\leq1$,
		\begin{equation*}
			\sup_{s\not=t\in[0,T]}(1\wedge s\wedge t)^{\eta}\abs{t-s}^{-\kappa}\int_{s}^{t}(1\wedge r)^{-\eta'}\dd r\lesssim T^{1-(\eta'-\eta)-\kappa}.
		\end{equation*}
		This establishes the time regularity
		\begin{equation*}
			\sup_{s\not=t\in[0,T]}(1\wedge s\wedge t)^{\eta}\frac{\norm{\mcI[f]_{t}-\mcI[f]_{s}}_{\mcC^{\beta-2\kappa}}}{\abs{t-s}^{\kappa}}\lesssim_{T}\norm{f}_{C_{\eta';T}\mcC^\alpha}
		\end{equation*}
		and in particular if $T\leq1$, then
		\begin{equation*}
			\sup_{s\not=t\in[0,T]}(1\wedge s\wedge t)^{\eta}\frac{\norm{\mcI[f]_{t}-\mcI[f]_{s}}_{\mcC^{\beta-2\kappa}}}{\abs{t-s}^{\kappa}}\lesssim (T^{1-\frac{\beta-\alpha}{2}-(\eta'-\eta)}\vee T^{1-(\eta'-\eta)-\kappa})\norm{f}_{C_{\eta';T}\mcC^\alpha}.
		\end{equation*}
		This proves~\eqref{eq:Schauder_2}. 
		
		To establish the continuity of $\mcI\from C_{\eta';T}\mcC^{\alpha}(\mbT^{d})\to\msL_{\eta;T}^{\kappa}\mcC^{\beta}(\mbT^{d})$, it suffices to show that $\mcI[f]\in\msL_{\eta;T}^{\kappa}\mcC^{\beta}(\mbT^{d})$ for every $f\in C_{\eta';T}\mcC^{\alpha}(\mbT^{d})$. Since $C_{\eta;T}^{\kappa}\mcC^{\beta-2\kappa}(\mbT^{d})$ is defined as those functions that are finite under the norm $\norm{\place}_{C_{\eta;T}^{\kappa}\mcC^{\beta-2\kappa}}$ it is clear by~\eqref{eq:Schauder_2} that $\mcI[f]\in C_{\eta;T}^{\kappa}\mcC^{\beta-2\kappa}(\mbT^{d})$. Hence it remains to establish $\mcI[f]\in C_{\eta;T}\mcC^{\beta}(\mbT^{d})$ (i.e.\ that $[0,T]\ni t\mapsto(1\wedge t)^{\eta}\mcI[f]_{t}\in\mcC^{\beta}(\mbT^{d})$ is continuous.) By the definition of $C_{\eta';T}\mcC^{\alpha}(\mbT^{d})$, it follows that $f_{\eta'\weight}\in C_{T}\mcC^{\alpha}(\mbT^{d})$, where $f_{\eta'\weight}(t)\defeq(1\wedge t)^{\eta'}f(t)$ for all $t\in(0,T]$ and $f_{\eta'\weight}(0)\defeq\lim_{t\to0}f_{\eta'\weight}(t)$. Using the density of $C^{\infty}(\mbT^{d})\subset\mcC^{\alpha}(\mbT^{d})$ and a partition of unity (in time), we can find a sequence $(g_{n})_{n\in\mbN}$ such that $g_{n}\in C_{T}C^{\infty}(\mbT^{d})$ and $g_{n}\to f_{\eta'\weight}\in C_{T}\mcC^{\alpha}(\mbT^{d})$. Using that the map $C_{\eta';T}\mcC^{\alpha}\to C_{T}\mcC^{\alpha}(\mbT^{d})$, $f\mapsto f_{\eta'\weight} $, is an isometric isomorphism, it follows that $f_{n}(t)\defeq(1\wedge t)^{-\eta'}g_{n}(t)$ satisfies $f_{n}\in C_{\eta';T}C^{\infty}(\mbT^{d})$ and $f_{n}\to f\in C_{\eta';T}\mcC^{\alpha}(\mbT^{d})$. Since $f_{n}\to f\in C_{\eta';T}\mcC^{\alpha}(\mbT^{d})$, we obtain by~\eqref{eq:Schauder_2} that $\lim_{n\to\infty}\norm{\mcI[f]-\mcI[f_{n}]}_{C_{\eta;T}\mcC^{\beta}}\to0$; hence, by the completeness of the continuous functions in the supremum norm, it suffices to show that $\mcI[f_{n}]\in C_{\eta;T}\mcC^{\beta}(\mbT^{d})$.
		
		Let $\eps\in(0,1-\eta']$, using that $f_{n}\in C_{\eta';T}C^{\infty}(\mbT^{d})\subset C_{\eta';T}\mcC^{\beta+2\eps}(\mbT^{d})$, we apply~\eqref{eq:Schauder_2} to obtain 
		\begin{equation*}
			\norm{\mcI[f_{n}]}_{\msL_{T}^{\eps}\mcC^{\beta+2\eps}}\lesssim_{T}\norm{f_{n}}_{C_{\eta';T}\mcC^{\beta+2\eps}},
		\end{equation*}
		which yields $\mcI[f_{n}]\in C_{T}^{\eps}\mcC^{\beta}(\mbT^{d})\subset C_{T}\mcC^{\beta}\subset C_{\eta;T}\mcC^{\beta}(\mbT^{d})$.
		\paragraph{Proof of Claim~\ref{it:Schauder_continuous}.}
		Let $f\in C_{T}\mcC^{\alpha}(\mbT^{d})$, we aim to establish $\mcI[f]\in\msL_{T}^{\kappa}\mcC^{\alpha+2}(\mbT^{d})$. To control the $C_{T}\mcC^{\alpha+2}(\mbT^{d})$-norm of $\mcI[f]$, we use~\cite[Lem.~A.9]{gubinelli_15_GIP} to bound
		\begin{equation*}
			\sup_{t\in[0,T]}\norm{\mcI[f]_{t}}_{\mcC^{\alpha+2}}\lesssim\norm{f}_{C_{T}\mcC^{\alpha}}.
		\end{equation*}
		Next we control the H\"{o}lder seminorm. We compute for $0<s\leq t\leq T$ using~\eqref{eq:HeatFlowMinusId} and~\cite[Lem.~A.9]{gubinelli_15_GIP},
		\begin{equation*}
			\norm{(P_{t-s}-1)\mcI[f]_{s}}_{\mcC^{\alpha+2-2\kappa}}\lesssim\abs{t-s}^{\kappa}\norm{\mcI[f]_{s}}_{\mcC^{\alpha+2}}\lesssim\abs{t-s}^{\kappa}\norm{f}_{C_{s}\mcC^{\alpha}}.
		\end{equation*}
		Taking the supremum, we obtain
		\begin{equation*}
			\sup_{s\neq t\in[0,T]}\frac{\norm{(P_{t-s}-1)\mcI[f]_{s}}_{\mcC^{\alpha+2-2\kappa}}}{\abs{t-s}^{\kappa}}\lesssim\norm{f}_{C_{T}\mcC^{\alpha}}.
		\end{equation*}
		Next, by~\eqref{eq:HeatFlow},
		\begin{equation*}
			\norm*{\int_{s}^{t}P_{t-r}f(r)\dd r}_{\mcC^{\alpha+2-2\kappa}}\lesssim\int_{s}^{t}(1\vee\abs{t-r}^{-1+\kappa})\dd r\norm{f}_{C_{t}\mcC^{\alpha}}.
		\end{equation*}
		Taking the supremum,
		\begin{equation*}
			\sup_{s\neq t\in[0,T]}\frac{\norm*{\int_{s}^{t}P_{t-r}f(r)\dd r}_{\mcC^{\alpha+2-2\kappa}}}{\abs{t-s}^{\kappa}}\lesssim\Bigl(\sup_{s\neq t\in[0,T]}\abs{t-s}^{-\kappa}\int_{s}^{t}(1\vee\abs{t-r}^{-1+\kappa})\dd r\Bigr)\norm{f}_{C_{T}\mcC^{\alpha}}.
		\end{equation*}
		Assume $t-s<1$, then
		\begin{equation*}
			\int_{s}^{t}(1\vee\abs{t-r}^{-1+\kappa})\dd r=\int_{s}^{t}\abs{t-r}^{-1+\kappa}\dd r\lesssim\abs{t-s}^{\kappa}.
		\end{equation*}
		Assume $1<t-s$, then $\abs{t-s}^{-\kappa}\leq1$. We estimate
		\begin{equation*}
			\int_{s}^{t}(1\vee\abs{t-r}^{-1+\kappa})\dd r\leq\int_{0}^{t}(1\vee\abs{t-r}^{-1+\kappa})\dd r.
		\end{equation*}
		Using that $1\leq1+s<t$, we obtain
		\begin{equation*}
			\int_{0}^{t}(1\vee\abs{t-r}^{-1+\kappa})\dd r=\int_{0}^{t-1}1\dd r+\int_{t-1}^{t}\abs{t-r}^{-1+\kappa}\dd r\leq\int_{0}^{t}1\dd r+\int_{0}^{t}\abs{t-r}^{-1+\kappa}\dd r\lesssim t+t^{\kappa}.
		\end{equation*}
		Consequently,
		\begin{equation*}
			\sup_{s\not=t\in[0,T]}\abs{t-s}^{-\kappa}\int_{s}^{t}(1\vee\abs{t-r}^{-1+\kappa})\dd r\lesssim1+T+T^{\kappa}.
		\end{equation*}
		In particular if $T\leq1$, then
		\begin{equation*}
			\sup_{s\not=t\in[0,T]}\abs{t-s}^{-\kappa}\int_{s}^{t}(1\vee\abs{t-r}^{-1+\kappa})\dd r\lesssim1.
		\end{equation*}
		This establishes the time regularity
		\begin{equation*}
			\sup_{s\not=t\in[0,T]}\frac{\norm{\mcI[f]_{t}-\mcI[f]_{s}}_{\mcC^{\alpha+2-2\kappa}}}{\abs{t-s}^{\kappa}}\lesssim_{T}\norm{f}_{C_{T}\mcC^\alpha}.
		\end{equation*}
		Therefore,
		\begin{equation}\label{eq:Schauder_continuity}
			\norm{\mcI[f]}_{\msL_{T}^{\kappa}\mcC^{\alpha+2}}\lesssim_{T}\norm{f}_{C_{T}\mcC^{\alpha}}.
		\end{equation}
		To establish the continuity of the operator $\mcI\from C_{T}\mcC^{\alpha}(\mbT^{d})\to\msL_{T}^{\kappa}\mcC^{\alpha+2}(\mbT^{d})$, it suffices to show that $\mcI[f]\in\msL_{T}^{\kappa}\mcC^{\alpha+2}(\mbT^{d})$ for every $f\in C_{T}\mcC^{\alpha}(\mbT^{d})$. Since $C_{T}^{\kappa}\mcC^{\alpha+2-2\kappa}(\mbT^{d})$ is defined as those functions that are finite under the norm $\norm{\place}_{C_{T}^{\kappa}\mcC^{\alpha+2-2\kappa}}$, it is clear by~\eqref{eq:Schauder_continuity} that $\mcI[f]\in C_{T}^{\kappa}\mcC^{\alpha+2-2\kappa}(\mbT^{d})$. Hence, it remains to establish $\mcI[f]\in C_{T}\mcC^{\alpha+2}(\mbT^{d})$ (i.e.\ that $[0,T]\ni t\mapsto\mcI[f]_{t}\in\mcC^{\alpha+2}(\mbT^{d})$ is continuous.) Using the density of $C^{\infty}(\mbT^{d})\subset\mcC^{\alpha}(\mbT^{d})$ and a partition of unity, we can find a sequence $(f_{n})_{n\in\mbN}$ such that $\lim_{n\to\infty}\norm{f-f_{n}}_{C_{T}\mcC^{\alpha}}=0$ and $f_{n}\in C_{T}C^{\infty}(\mbT^{d})$ for every $n\in\mbN$. Since $f_{n}\to f\in C_{T}\mcC^{\alpha}(\mbT^{d})$, we obtain by~\eqref{eq:Schauder_continuity} that $\lim_{n\to\infty}\norm{\mcI[f]-\mcI[f_{n}]}_{C_{T}\mcC^{\alpha+2}}=0$; hence, by the completeness of the continuous functions in the supremum norm, it suffices to show that $\mcI[f_{n}]\in C_{T}\mcC^{\alpha+2}(\mbT^{d})$.
		
		Let $\eps\in(0,1]$, using that $f_{n}\in C_{T}C^{\infty}(\mbT^{d})\subset C_{T}\mcC^{\alpha+2\eps}(\mbT^{d})$, we apply~\eqref{eq:Schauder_continuity} to obtain
		\begin{equation*}
			\norm{\mcI[f_{n}]}_{\msL_{T}^{\eps}\mcC^{\alpha+2+2\eps}}\lesssim_{T}\norm{f_{n}}_{C_{T}\mcC^{\alpha+2\eps}},
		\end{equation*}
		which yields $\mcI[f_{n}]\in C_{T}^{\eps}\mcC^{\alpha+2}(\mbT^{d})\subset C_{T}\mcC^{\alpha+2}(\mbT^{d})$.
	\end{details}
\end{proof}
\begin{details}
	In the next lemma we show that the heat equation acting on Besov distributions satisfies the initial condition.
	\begin{lemma}
		Let $\alpha\in\mbR$ and $p,q\in[1,\infty]$.
		\begin{enumerate}
			\item\label{it:IC_heat_sg}
			Let $u_{0}\in\mcB_{p,q}^{\alpha}(\mbT^{d})$, then $\lim_{t\to0}P_{t}u_{0}=u_{0}$ in $\mcB_{p,q}^{\alpha}(\mbT^{d})$. In particular, $\lim_{t\to0}P_{t}u_{0}=u_{0}$ in $\mcS'(\mbT^{d})$.
			\item\label{it:IC_heat_resolution}
			Let $\eta'\in[0,1)$ and $f\in C_{\eta';T}\mcC^{\alpha}(\mbT^{d})$, then $\lim_{t\to0}\mcI[f]_{t}=0$ in $\mcC^{\alpha+2(1-\eta')}(\mbT^{d})$. In particular, $\lim_{t\to0}\mcI[f]_{t}=0$ in $\mcS'(\mbT^{d})$.
		\end{enumerate}
	\end{lemma}
	\begin{proof}
		Proof of Claim~\ref{it:IC_heat_sg}. A small extension of Lemma~\ref{lem:Schauder} yields $Pu_{0}\in C_{T}\mcB_{p,q}^{\alpha}(\mbT^{d})$, hence $\lim_{t\to0}P_{t}u_{0}\in\mcB_{p,q}^{\alpha}(\mbT^{d})$, which we denote by $P_{0}u_{0}$. It suffices to show that $P_{0}u_{0}=u_{0}$ in $\mcB_{p,q}^{\alpha}(\mbT^{d})$. Let $\beta<\alpha$, we obtain by~\eqref{eq:HeatFlowMinusId},
		\begin{equation*}
			\norm{P_{t}u_{0}-u_{0}}_{\mcB_{p,q}^{\beta}}\lesssim t^{\frac{\alpha-\beta}{2}}\norm{u_{0}}_{\mcB_{p,q}^{\alpha}}\to0.
		\end{equation*}
		Therefore,
		\begin{equation*}
			\norm{P_{0}u_{0}-u_{0}}_{\mcB_{p,q}^{\beta}}\leq\norm{P_{0}u_{0}-P_{t}u_{0}}_{\mcB_{p,q}^{\beta}}+\norm{P_{t}u_{0}-u_{0}}_{\mcB_{p,q}^{\beta}}\lesssim\norm{P_{0}u_{0}-P_{t}u_{0}}_{\mcB_{p,q}^{\alpha}}+\norm{P_{t}u_{0}-u_{0}}_{\mcB_{p,q}^{\beta}}\to0,
		\end{equation*}
		which implies $P_{0}u_{0}=u_{0}$ in $\mcB_{p,q}^{\beta}(\mbT^{d})$. Consequently by~\cite[Thm.~21.18]{vanzuijlen_22}, it follows that $P_{0}u_{0}=v$ in $\mcS'(\mbT^{d})$, which yields $P_{0}u_{0}=v$ in $\mcB_{p,q}^{\alpha}(\mbT^{d})$ and hence the claim.
		
		Hence, $v=u_{0}$ in $\mcB_{p,q}^{\beta}(\mbT^{d})$ but $v\neq u_{0}$ in $\mcB_{p,q}^{\alpha}(\mbT^{d})$, which contradicts the Hausdorff property of $\mcB_{p,q}^{\beta}(\mbT^{d})$. Therefore, $P_{0}u_{0}=u_{0}$ in $\mcB_{p,q}^{\alpha}(\mbT^{d})$, which yields the claim.
		
		Proof of Claim~\ref{it:IC_heat_resolution}. We obtain by Lemma~\ref{lem:Schauder} that $\mcI[f]\in C_{T}\mcC^{\alpha+2(1-\eta')}(\mbT^{d})$, hence $\lim_{t\to0}\mcI[f]_{t}\in\mcC^{\alpha+2(1-\eta')}(\mbT^{d})$, which we denote by $\mcI[f]_{0}$. It suffices to show that $\mcI[f]_{0}=0$ in $\mcC^{\alpha+2(1-\eta')}(\mbT^{d})$. Let $\alpha\leq\beta<\alpha+2(1-\eta')$, we obtain by~\eqref{eq:HeatFlow},
		\begin{equation*}
			\norm{\mcI[f]_{t}}_{\mcC^{\beta}}\leq\int_{0}^{t}\norm{P_{t-r}f(r)}_{\mcC^{\beta}}\dd r\lesssim\int_{0}^{t}(1\vee \abs{t-r}^{-\frac{\beta-\alpha}{2}})(1\wedge r)^{-\eta'}\dd r\norm{f}_{C_{\eta';t}\mcC^{\alpha}},
		\end{equation*}
		hence for $t\in(0,1)$,
		\begin{equation*}
			\norm{\mcI[f]_{t}}_{\mcC^{\beta}}\lesssim t^{1-\frac{\beta-\alpha}{2}-\eta'}\norm{f}_{C_{\eta';T}\mcC^{\alpha}}\to0.
		\end{equation*}
		Therefore, 
		\begin{equation*}
			\norm{\mcI[f]_{0}}_{\mcC^{\beta}}\leq\norm{\mcI[f]_{0}-\mcI[f]_{t}}_{\mcC^{\beta}}+\norm{\mcI[f]_{t}}_{\mcC^{\beta}}\lesssim\norm{\mcI[f]_{0}-\mcI[f]_{t}}_{\mcC^{\alpha+2(1-\eta')}}+\norm{\mcI[f]_{t}}_{\mcC^{\beta}}\to0,
		\end{equation*}
		which implies $\mcI[f]_{0}=0$ in $\mcC^{\beta}(\mbT^{d})$. Consequently by~\cite[Thm.~21.18]{vanzuijlen_22}, it follows that $\mcI[f]_{0}=0$ in $\mcC^{\alpha+2(1-\eta')}(\mbT^{d})$ which yields the claim.
	\end{proof}
\end{details}
Assume $\theta\from\mbR^d\to\mbC$ is smooth and such that $\partial^{\nu}\theta$ is of at most polynomial growth for each multi-indix $\nu\in\mbN^{d}_{0}$.
\begin{details}
	Cf.~\cite[Def.~14.6~\&~Lem.~19.2]{vanzuijlen_22}.
\end{details}
Additionally assume that $\theta$ satisfies the \emph{reality condition}
\begin{equation}\label{eq:reality_condition}
	\overline{\theta(\om)}=\theta(-\om),\qquad\om\in\mbZ^{d}.
\end{equation}
We define the Fourier multiplier acting on $u\in\mcS'(\mbT^d;\mbR)$ by the expression
\begin{equation*}
	\theta(D)u \defeq \msF^{-1}(\theta\hat{u}).
\end{equation*}
The polynomial growth condition on all partial derivatives and~\eqref{eq:reality_condition} ensure that $\theta(D)$ maps real-valued distributions to real-valued distributions, a result which we generalize in the following lemma.
\begin{lemma}\label{lem:Fourier_multiplier}
	Let $\alpha\in\mbR$, $p,q\in[1,\infty]$, $u\in\mcB^\alpha_{p,q}(\mbT^{d})$ and $k=2\floor{1+d/2}$. Assume that $\theta\from\mbR^{d}\to\mbC$ satisfies $\theta\in C^{k}(\mbR^d\setminus\{0\};\mbC)$, $\theta(0)=0$, the reality condition~\eqref{eq:reality_condition} and that there exist some $m\in\mbR$, $C>0$, such that for any multi-index $\nu\in\mbN^d_0$ with $\abs{\nu}\leq k$,
	\begin{equation*}
		\abs{\partial^\nu\theta(x)}\leq C\abs{x}^{m-\abs{\nu}},\qquad x\in\mbR^d\setminus\{0\}.
	\end{equation*}
	Then,
	\begin{equation*}
		\norm{\theta(D)u}_{\mcB^{\alpha-m}_{p,q}}\lesssim\norm{u}_{\mcB^{\alpha}_{p,q}}.
	\end{equation*}
\end{lemma}
\begin{proof}
	Since $u$ is periodic the only frequency contained in the support of $\varrho_{-1}$ is $\om=0$, hence $\theta(D)\Delta_{-1}u=0$. The remaining Littlewood--Paley blocks can then be addressed directly with~\cite[Lem.~2.2]{bahouri_chemin_danchin_11}.
\end{proof}
Lemma~\ref{lem:Fourier_multiplier} leads directly to a control on solutions to Poisson's equation and their derivatives.
\begin{lemma}\label{lem:elliptic_regularity}
	For any $\alpha\in\mbR$ and $p,q\in [1,\infty]$, it follows that
	\begin{equation*}
		\norm{\msG\ast u}_{\mcB^{\alpha}_{p,q}}\lesssim\norm{u}_{\mcB^{\alpha-2}_{p,q}}, \qquad \norm{\nabla \msG\ast u}_{\mcB^{\alpha}_{p,q}}\lesssim \norm{u}_{\mcB^{\alpha-1}_{p,q}}.
	\end{equation*}
\end{lemma}
\begin{proof}
	Simply apply Lemma~\ref{lem:Fourier_multiplier} to the multipliers $\theta_1(\omega)=\frac{1}{\abs{2\uppi\om}^{2}}\mathds{1}_{\omega\neq 0}$ and $\theta_2(\omega)=\frac{2\uppi \upi\omega}{\abs{2\uppi\om}^{2}}\mathds{1}_{\omega\neq 0}$.
\end{proof}
\subsection{Commutator Results}
Many of the results presented below are analogues and simple extensions of similar results found in \cite{perkowski_13,catellier_chouk_18} to time-weighted spaces, $C_{\eta;T}\mcC^{\alpha}(\mbT^{d};\mbR)$.
\begin{lemma}\label{lem:commutator_Fourier_multiplier_paraproduct}
	Let $T>0$, $\eta,\eta_1,\eta_2\geq0$ be such that $\eta=\eta_1+\eta_2$, $\alpha\in(-\infty,1)$, $\beta\in\mbR$ and $k=2\floor{1+d/2}$. Assume $\theta\from\mbR^{d}\to\mbC$ satisfies $\theta\in C^{k+1}(\mbR^d\setminus\{0\};\mbC)$, $\theta(0)=0$, \eqref{eq:reality_condition} and that there exist some $m\in\mbR$, $C>0$, such that for any multi-index $\nu\in\mbN^d_0$ with $\abs{\nu}\leq k+1$,
	\begin{equation*}
		\abs{\partial^\nu\theta(x)}\leq C\abs{x}^{m-\abs{\nu}},\qquad x\in\mbR^d\setminus\{0\}.
	\end{equation*}
	Let $u\in C_{\eta_1;T}\mcC^{\alpha}(\mbT^{d};\mbR)$ and $v\in C_{\eta_2;T}\mcC^{\beta}(\mbT^{d};\mbR)$, then $\theta(D)(u\pa v)-u\pa\theta(D)v\in C_{\eta;T}\mcC^{\alpha+\beta-m}(\mbT^{d};\mbR)$ and
	\begin{equation*}
		\norm{\theta(D)(u\pa v)-u\pa\theta(D)v}_{C_{\eta;T}\mcC^{\alpha+\beta-m}}\lesssim\norm{u}_{C_{\eta_1;T}\mcC^{\alpha}}\norm{v}_{C_{\eta_2;T}\mcC^{\beta}}.
	\end{equation*}
\end{lemma}
\begin{proof}
	The result is a simple extension of~\cite[Lem.~5.3.20]{perkowski_13} and \cite[Lem.~A.1]{catellier_chouk_18} to functions with prescribed blow-up in $\mcC^{\alpha}(\mbT^{d};\mbR)$ at $t=0$.
\end{proof}
Next, we consider the commutator between the operators $\mcI$ and $\pa$, a result reminiscent of~\cite[Prop.~2.7]{catellier_chouk_18}.
\begin{lemma}\label{lem:commutator_heat_paraproduct}
	Let $T>0$, $\eta,\eta_{1},\eta_{2}\in[0,1)$ be such that $\eta=\eta_1+\eta_2$, $\kappa>0$, $\alpha\in(-\infty,(1\wedge2\kappa))$, $\beta\in\mbR$ and $m\in(0,2)$. For every $u\in\msL_{\eta_1;T}^{\kappa}\mcC^{\alpha}(\mbT^{d};\mbR)$ and $v\in C_{\eta_2;T}\mcC^{\beta}(\mbT^{d};\mbR)$, it holds that $\mcI[u\pa v]-u\pa\mcI[v]\in C_{\eta;T}\mcC^{\alpha+\beta+m}(\mbT^{d};\mbR)$ and
	\begin{equation}\label{eq:commutator_heat_paraproduct_bound}
		\norm{\mcI[u\pa v]-u\pa \mcI[v]}_{C_{\eta;T}\mcC^{\alpha+\beta+m}}\lesssim_{T}\norm{u}_{\msL_{\eta_1;T}^{\kappa}\mcC^{\alpha}}\norm{v}_{C_{\eta_2;T}\mcC^{\beta}}.
	\end{equation}
\end{lemma}
\begin{proof}
	Let $u\in\msL_{\eta_1;T}^{\kappa}\mcC^{\alpha}(\mbT^{d};\mbR)$ and $v\in C_{\eta_2;T}\mcC^{\beta}(\mbT^{d};\mbR)$, we first prove the bound~\eqref{eq:commutator_heat_paraproduct_bound} and then the regularity $\mcI[u\pa v]-u\pa\mcI[v]\in C_{\eta;T}\mcC^{\alpha+\beta+m}(\mbT^{d};\mbR)$. Let $t\in[0,T]$, by definition,
	\begin{equation*}
		\begin{split}
			\mcI[u\pa v]_t-u_{t}\pa\mcI[v]_{t}&=\int_{0}^{t}P_{t-s}(u(s)\pa v(s))-u(t)\pa P_{t-s}v(s)\dd s\\
			&=\int_{0}^{t}P_{t-s}(u(s)\pa v(s))-u(s)\pa P_{t-s}v(s)\dd s+\int_{0}^{t}(u(s)-u(t))\pa P_{t-s}v(s)\dd s.
		\end{split}
	\end{equation*}
	
	To bound the first summand we apply~\cite[Lem.~5.3.20]{perkowski_13} and use that $\alpha<1$ to estimate
	\begin{equation*}
		\norm{P_{t-s}(u(s)\pa v(s))-u(s)\pa P_{t-s}v(s)}_{\mcC^{\alpha+\beta+m}}\lesssim\abs{t-s}^{-m/2}(1\wedge s)^{-\eta}\norm{u}_{C_{\eta_1;s}\mcC^{\alpha}}\norm{v}_{C_{\eta_2;s}\mcC^{\beta}}.
	\end{equation*}
	Taking the supremum, we obtain
	\begin{equation*}
		\begin{split}
			&\sup_{t\in[0,T]}(1\wedge t)^{\eta}\norm*{\int_{0}^{t}P_{t-s}(u(s)\pa v(s))-u(s)\pa P_{t-s}v(s)\dd s}_{\mcC^{\alpha+\beta+m}}\\
			&\lesssim\Bigl(\sup_{t\in[0,T]}(1\wedge t)^{\eta}\int_{0}^{t}\abs{t-s}^{-m/2}(1\wedge s)^{-\eta}\dd s\Bigr)\norm{u}_{C_{\eta_1;T}\mcC^{\alpha}}\norm{v}_{C_{\eta_2;T}\mcC^{\beta}}.
		\end{split}
	\end{equation*}
	We can use a case distinction over $t\in[0,1]$ and $t\in(1,\infty)$; and the assumptions $\eta<1$, $m<2$, to bound
	\begin{equation*}
		\sup_{t\in[0,T]}(1\wedge t)^{\eta}\int_{0}^{t}\abs{t-s}^{-m/2}(1\wedge s)^{-\eta}\dd s\lesssim_{T}1.
	\end{equation*}
	\begin{details}
		To bound this expression, we distinguish the cases $t\in[0,1]$ and $t\in(1,\infty)$. Assume $t\in[0,1]$, then
		\begin{equation*}
			(1\wedge t)^{\eta}\int_{0}^{t}\abs{t-s}^{-m/2}(1\wedge s)^{-\eta}\dd s=t^{\eta}\int_{0}^{t}\abs{t-s}^{-m/2}s^{-\eta}\dd s\lesssim t^{1-m/2}.
		\end{equation*}
		Assume $t\in(1,\infty)$, then
		\begin{equation*}
			\begin{split}
				(1\wedge t)^{\eta}\int_{0}^{t}\abs{t-s}^{-m/2}(1\wedge s)^{-\eta}\dd s&=\int_{0}^{1}\abs{t-s}^{-m/2}s^{-\eta}\dd s+\int_{1}^{t}\abs{t-s}^{-m/2}\dd s\\
				&\leq\int_{0}^{t}\abs{t-s}^{-m/2}s^{-\eta}\dd s+\int_{0}^{t}\abs{t-s}^{-m/2}\dd s\\
				&\lesssim t^{1-m/2-\eta}+t^{1-m/2}.
			\end{split}
		\end{equation*}
	\end{details}
	
	To bound the second summand, we apply Lemma~\ref{lem:heat_flow} and use that $\alpha-2\kappa<0$ to estimate
	\begin{align*}
		\norm{(u(s)-u(t))\pa P_{t-s}v(s)}_{\mcC^{\alpha+\beta+m}}&\lesssim\norm{u(s)-u(t)}_{\mcC^{\alpha-2\kappa}}\norm{P_{t-s}v(s)}_{\mcC^{\beta+2\kappa+m}}\\
		&\lesssim\abs{t-s}^{\kappa}(1\vee\abs{t-s}^{-\kappa-m/2})(1\wedge s)^{-\eta}\norm{u}_{C_{\eta_1;t}^{\kappa}\mcC^{\alpha-2\kappa}}\norm{v}_{C_{\eta_2;s}\mcC^{\beta}}.
	\end{align*}
	Taking the supremum, we obtain
	\begin{equation*}
		\begin{split}
			&\sup_{t\in[0,T]}(1\wedge t)^{\eta}\norm*{\int_{0}^{t}(u(s)-u(t))\pa P_{t-s}v(s)\dd s}_{\mcC^{\alpha+\beta+m}}\\
			&\lesssim\Bigl(\sup_{t\in[0,T]}(1\wedge t)^{\eta}\int_{0}^{t}\abs{t-s}^{\kappa}(1\vee\abs{t-s}^{-\kappa-m/2})(1\wedge s)^{-\eta}\dd s\Bigr)\norm{u}_{C_{\eta_1;T}^{\kappa}\mcC^{\alpha-2\kappa}}\norm{v}_{C_{\eta_2;T}\mcC^{\beta}}.
		\end{split}
	\end{equation*}
	We can use a case distinction over $t\in[0,1]$, $t\in(1,2]$, and $t\in(2,\infty)$; and the assumptions $\eta<1$, $m<2$, to bound
	\begin{equation*}
		\sup_{t\in[0,T]}(1\wedge t)^{\eta}\int_{0}^{t}\abs{t-s}^{\kappa}(1\vee\abs{t-s}^{-\kappa-m/2})(1\wedge s)^{-\eta}\dd s\lesssim_{T}1.
	\end{equation*}
	\begin{details}
		We distinguish the cases $t\in[0,1]$, $t\in(1,2]$ and $t\in(2,\infty)$. Assume $t\in[0,1]$, then
		\begin{equation*}
			(1\wedge t)^{\eta}\int_{0}^{t}\abs{t-s}^{\kappa}(1\vee\abs{t-s}^{-\kappa-m/2})(1\wedge s)^{-\eta}\dd s=t^{\eta}\int_{0}^{t}\abs{t-s}^{-m/2} s^{-\eta}\dd s\lesssim t^{1-m/2}.
		\end{equation*}
		Assume $t\in(1,2]$, then
		\begin{equation*}
			\begin{split}
				&(1\wedge t)^{\eta}\int_{0}^{t}\abs{t-s}^{\kappa}(1\vee\abs{t-s}^{-\kappa-m/2})(1\wedge s)^{-\eta}\dd s\\
				&=\int_{0}^{t-1}\abs{t-s}^{\kappa}s^{-\eta}\dd s+\int_{t-1}^{1}\abs{t-s}^{-m/2} s^{-\eta}\dd s+\int_{1}^{t}\abs{t-s}^{-m/2}\dd s\\
				&\leq\int_{0}^{t}\abs{t-s}^{\kappa}s^{-\eta}\dd s+\int_{0}^{t}\abs{t-s}^{-m/2} s^{-\eta}\dd s+\int_{0}^{t}\abs{t-s}^{-m/2}\dd s\\
				&\lesssim t^{1+\kappa-\eta}+t^{1-m/2-\eta}+t^{1-m/2}.
			\end{split}
		\end{equation*}
		Assume $t\in(2,\infty]$, then
		\begin{equation*}
			\begin{split}
				&(1\wedge t)^{\eta}\int_{0}^{t}\abs{t-s}^{\kappa}(1\vee\abs{t-s}^{-\kappa-m/2})(1\wedge s)^{-\eta}\dd s\\
				&=\int_{0}^{1}\abs{t-s}^{\kappa} s^{-\eta}\dd s+\int_{1}^{t-1}\abs{t-s}^{\kappa}\dd s+\int_{t-1}^{t}\abs{t-s}^{-m/2}\dd s\\
				&\leq\int_{0}^{t}\abs{t-s}^{\kappa} s^{-\eta}\dd s+\int_{0}^{t}\abs{t-s}^{\kappa}\dd s+\int_{0}^{t}\abs{t-s}^{-m/2}\dd s\\
				&\lesssim t^{1+\kappa-\eta}+t^{1+\kappa}+t^{1-m/2}.
			\end{split}
		\end{equation*}
	\end{details}
	It follows that
	\begin{equation*}
		\norm{\mcI[u\pa v]-u\pa\mcI[v]}_{C_{\eta;T}\mcC^{\alpha+\beta+m}}\lesssim_{T}\norm{u}_{\msL_{\eta_1;T}^{\kappa}\mcC^{\alpha}}\norm{v}_{C_{\eta_2;T}\mcC^{\beta}},
	\end{equation*}
	which yields~\eqref{eq:commutator_heat_paraproduct_bound}.
	
	To establish $\mcI[u\pa v]-u\pa\mcI[v]\in C_{\eta;T}\mcC^{\alpha+\beta+m}(\mbT^{d};\mbR)$, we need to show that $(0,T]\ni t\mapsto(1\wedge t)^{\eta}(\mcI[u\pa v]_{t}-u_{t}\pa\mcI[v]_{t})$ is continuous and that $\lim_{t\to0}(1\wedge t)^{\eta}(\mcI[u\pa v]_{t}-u_{t}\pa\mcI[v]_{t})$ exists in $\mcC^{\alpha+\beta+m}(\mbT^{d};\mbR)$. The continuity in $(0,T]$ follows by the completeness of the continuous functions under the supremum norm and by taking limits along a smooth approximating sequence of $v$.
	\begin{details}
		
		Let $(v_{n})_{n\in\mbN}$ be such that $\lim_{n\to\infty}\norm{v-v_{n}}_{C_{\eta_{2};T}\mcC^{\beta}}=0$ and $v_{n}\in C_{\eta_{2};T}C^{\infty}(\mbT^{d};\mbR)$ for every $n\in\mbN$. It follows by~\eqref{eq:commutator_heat_paraproduct_bound} that 
		\begin{equation*}
			\lim_{n\to\infty}\norm{\mcI[u\pa v]-u\pa\mcI[v]-(\mcI[u\pa v_{n}]-u\pa\mcI[v_{n}])}_{C_{\eta;T}\mcC^{\alpha+\beta+m}}=0,
		\end{equation*}
		hence by the completeness of the continuous functions in the supremum norm, it suffices to show that $\mcI[u\pa v_{n}]-u\pa\mcI[v_{n}]\in C_{\eta;T}\mcC^{\alpha+\beta+m}(\mbT^{d};\mbR)$.
		
		Using Lemma~\ref{lem:Bony} and that $v_{n}\in C_{\eta_{2};T}C^{\infty}(\mbT^{d};\mbR)\subset C_{\eta_{2};T}\mcC^{\abs{\alpha}+\beta+2}(\mbT^{d};\mbR)$, we can show $u\pa v_{n}\in C_{\eta;T}\mcC^{\alpha+\beta+m}(\mbT^{d};\mbR)$. We can then apply Lemma~\ref{lem:Schauder} to control $\mcI[u\pa v_{n}]\in C_{\eta;T}\mcC^{\alpha+\beta+m}(\mbT^{d};\mbR)$. Similarly, by Lemma~\ref{lem:Schauder}, $\mcI[v_{n}]\in C_{\eta_{2};T}\mcC^{\abs{\alpha}+\beta+2}(\mbT^{d};\mbR)$ and by Lemma~\ref{lem:Bony}, $u\pa\mcI[v_{n}]\in C_{\eta;T}\mcC^{\alpha+\beta+m}(\mbT^{d};\mbR)$. Therefore, $\mcI[u\pa v_{n}]-u\pa\mcI[v_{n}]\in C_{\eta;T}\mcC^{\alpha+\beta+m}(\mbT^{d};\mbR)$, which yields the claim.
		
	\end{details}
	To show the continuity at $0$, we can use the same bounds as in the derivation of~\eqref{eq:commutator_heat_paraproduct_bound} to control for every $t\in(0,1\wedge T]$,
	\begin{equation*}
		(1\wedge t)^{\eta}\norm{\mcI[u\pa v]_{t}-u_{t}\pa\mcI[v]_{t}}_{\mcC^{\alpha+\beta+m}}\lesssim t^{1-m/2},
	\end{equation*}
	which implies $\lim_{t\to0}(1\wedge t)^{\eta}(\mcI[u\pa v]_{t}-u_{t}\pa\mcI[v]_{t})=0\in\mcC^{\alpha+\beta+m}(\mbT^{d};\mbR)$. Therefore, $\mcI[u\pa v]-u\pa\mcI[v]\in C_{\eta;T}\mcC^{\alpha+\beta+m}(\mbT^{d};\mbR)$, which yields the claim.
\end{proof}
Finally, we present a commutator result between the operators $\pa$ and $\re$.
\begin{lemma}\label{lem:commutator_paraproduct_resonant}
	Let $T>0$, $\eta,\eta_1,\eta_2,\eta_3\geq0$ be such that $\eta=\eta_1+\eta_2+\eta_3$, $\alpha\in(0,1)$ and $\beta,\gamma\in\mbR$ such that $\beta+\gamma<0$ and $\alpha+\beta+\gamma>0$. We define
	 \begin{equation*}
	 	\msC(f,g,h)=(f\pa g)\re h-f(g\re h),\qquad (f,g,h)\in C_{\eta_1;T}C^{\infty}(\mbT^{d};\mbR)\times C_{\eta_2;T}C^{\infty}(\mbT^{d};\mbR)\times C_{\eta_3;T}C^{\infty}(\mbT^{d};\mbR).
	 \end{equation*}
	Then $\msC$ extends to a bounded, trilinear operator 
	\begin{equation*}
		\msC\from C_{\eta_1;T}\mcC^{\alpha}(\mbT^{d};\mbR)\times C_{\eta_2;T}\mcC^{\beta}(\mbT^{d};\mbR)\times C_{\eta_3;T}\mcC^{\gamma}(\mbT^{d};\mbR)\to C_{\eta;T}\mcC^{\alpha+\beta+\gamma}(\mbT^{d};\mbR)
	\end{equation*}
\end{lemma}
\begin{proof}
	The result is a direct consequence of~\cite[Lem.~2.4]{gubinelli_15_GIP}.
\end{proof}
\section{Shape Coefficient Estimates}\label{sec:shape_coefficient_estimates}
In this appendix we bound our shape coefficients in terms of $\abs{t-s}^{\gamma}$ and $\abs{\om_k}^{\beta}$ for some $\gamma\in[0,1]$ and $\beta\in\mbR$. We first consider the shape coefficient for $\Ypsilon{10}$ (cf.\ Definition~\ref{def:shape_coefficient_ypsilon}).

\begin{details}
	A case distinction yields an explicit expression for the shape coefficient.
	\begin{lemma}\label{lem:shape_ypsilon}
		Let $s,t\geq0$ and $\om_1,\om_2\in2\uppi\mbZ^2\setminus\{0\}$ such that $\om_1+\om_2\not=0$. Then there exists a leading term $\mathsf{SL}_{s,t}\SYpsilon(\om_1,\om_2)$ and an error term $\mathsf{SE}_{s,t}\SYpsilon(\om_1,\om_2)$, such that
		\begin{equation*}
			\mathsf{S}_{s,t}\SYpsilon(\om_1,\om_2)=\mathsf{SL}_{s,t}\SYpsilon(\om_1,\om_2)+\mathsf{SE}_{s,t}\SYpsilon(\om_1,\om_2).
		\end{equation*}
		In the case $(\om_1\perp\om_2)$, we obtain
		\begin{equation*}
			\mathsf{SL}_{s,t}\SYpsilon(\om_1,\om_2)=\frac{1}{2^3}\abs{\om_1}^{-2}\abs{\om_2}^{-2}\abs{\om_1+\om_2}^{-4}\Bigl(\abs{t-s}\abs{\om_1+\om_2}^2\euler^{-\abs{t-s}\abs{\om_1+\om_2}^2}+\euler^{-\abs{t-s}\abs{\om_1+\om_2}^2}\Bigr)
		\end{equation*}
		and
		\begin{equation*}
			\mathsf{SE}_{s,t}\SYpsilon(\om_1,\om_2)=\frac{1}{2^3}\abs{\om_1}^{-2}\abs{\om_2}^{-2}\abs{\om_1+\om_2}^{-4}\Bigl(-(t+s)\abs{\om_1+\om_2}^2\euler^{-(t+s)\abs{\om_1+\om_2}^2}-\euler^{-(t+s)\abs{\om_1+\om_2}^2}\Bigr).
		\end{equation*}
		In the case $\lnot(\om_1\perp\om_2)$, we obtain
		\begin{equation*}
			\begin{split}
				&\mathsf{SL}_{s,t}\SYpsilon(\om_1,\om_2)\\	&=\frac{1}{2^3}\abs{\om_1}^{-2}\abs{\om_2}^{-2}\abs{\om_1+\om_2}^{-2}\Bigl(\frac{1}{\abs{\om_1}^2+\abs{\om_2}^2-\abs{\om_1+\om_2}^2}(\euler^{-\abs{t-s}\abs{\om_1+\om_2}^2}-\euler^{-\abs{t-s}(\abs{\om_1}^2+\abs{\om_2}^2)})\\	&\multiquad[13]+\frac{1}{\abs{\om_1}^2+\abs{\om_2}^2+\abs{\om_1+\om_2}^2}(\euler^{-\abs{t-s}\abs{\om_1+\om_2}^2}+\euler^{-\abs{t-s}(\abs{\om_1}^2+\abs{\om_2}^2)})\Bigr)
			\end{split}
		\end{equation*}
		and
		\begin{equation*}
			\begin{split}
				&\mathsf{SE}_{s,t}\SYpsilon(\om_1,\om_2)\\
				&=\frac{1}{2^2}\abs{\om_1}^{-2}\abs{\om_2}^{-2}\Bigl(\frac{1}{\abs{\om_1}^2+\abs{\om_2}^2-\abs{\om_1+\om_2}^2}\frac{1}{\abs{\om_1}^2+\abs{\om_2}^2+\abs{\om_1+\om_2}^2}\\
				&\multiquad[8]\times (\euler^{t(\abs{\om_1+\om_2}^2-\abs{\om_1}^2-\abs{\om_2}^2)}+\euler^{s(\abs{\om_1+\om_2}^2-\abs{\om_1}^2-\abs{\om_2}^2)})\euler^{-(t+s)\abs{\om_1+\om_2}^2}\\
				&\multiquad[12]-\frac{1}{\abs{\om_1}^2+\abs{\om_2}^2-\abs{\om_1+\om_2}^2}\frac{1}{\abs{\om_1+\om_2}^2}\euler^{-(t+s)\abs{\om_1+\om_2}^2}\Bigr).
			\end{split}
		\end{equation*}
	\end{lemma}
	\begin{proof}
		We obtain by computing the integrals over $u_1$, $u_2$,
		\begin{equation*}
			\begin{split}
				&\mathsf{S}_{s,t}\SYpsilon(\om_1,\om_2)\\
				&=\frac{1}{2^2}\abs{\om_1}^{-2}\abs{\om_2}^{-2}\euler^{-(t+s)\abs{\om_1+\om_2}^2}\int_{0}^{t}\dd u_3\int_{0}^{s}\dd u_3' \euler^{-\abs{u_3-u_3'}\abs{\om_1}^2}\euler^{-\abs{u_3-u_3'}\abs{\om_2}^2}\euler^{(u_3+u_3')\abs{\om_1+\om_2}^2}.
			\end{split}
		\end{equation*}
		Our next step is to compute
		\begin{equation*}
			\int_{0}^{t}\dd u_3\int_{0}^{s}\dd u_3' \euler^{-\abs{u_3-u_3'}\abs{\om_1}^2}\euler^{-\abs{u_3-u_3'}\abs{\om_2}^2}\euler^{(u_3+u_3')\abs{\om_1+\om_2}^2}.
		\end{equation*}
		We assume without loss of generality that $s\leq t$. Indeed, if $s>t$, we can exchange the roles of $u_3$, $u_3'$ and derive the same expression. We decompose with Fubini's theorem,
		\begin{equation*}
			\begin{split}
				&\int_{0}^{t}\dd u_3\int_{0}^{s}\dd u_3' \euler^{-\abs{u_3-u_3'}\abs{\om_1}^2}\euler^{-\abs{u_3-u_3'}\abs{\om_2}^2}\euler^{(u_3+u_3')\abs{\om_1+\om_2}^2}\\
				&=\int_{0}^{s}\dd u_3'\int_{u_3'}^{t}\dd u_3 \euler^{-(u_3-u_3')\abs{\om_1}^2}\euler^{-(u_3-u_3')\abs{\om_2}^2}\euler^{(u_3+u_3')\abs{\om_1+\om_2}^2}\\
				&\quad+\int_{0}^{s}\dd u_3\int_{u_3}^{s}\dd u_3' \euler^{(u_3-u_3')\abs{\om_1}^2}\euler^{(u_3-u_3')\abs{\om_2}^2}\euler^{(u_3+u_3')\abs{\om_1+\om_2}^2}.
			\end{split}
		\end{equation*}
		This leads to the cases $(\om_1\perp\om_2)\lor\lnot(\om_1\perp\om_2)$ and $(u_3'\leq u_3)\lor( u_3< u_3')$.
		
		Assume $(\om_1\perp\om_2)$. \emph{Case $u_3'\leq u_3$.} We compute
		\begin{equation*}
			\begin{split}
				&\int_{0}^{s}\dd u_3'\int_{u_3'}^{t}\dd u_3 \euler^{-(u_3-u_3')(\abs{\om_1}^2+\abs{\om_2}^2)}\euler^{(u_3+u_3')\abs{\om_1+\om_2}^2}\\
				&=\int_{0}^{s}\dd u_3' \euler^{u_3'2\abs{\om_1+\om_2}^2}(t-u_3')\\
				&=t\frac{1}{2\abs{\om_1+\om_2}^2}(\euler^{s2\abs{\om_1+\om_2}^2}-1)-\frac{1}{(2\abs{\om_1+\om_2}^2)^2}(\euler^{s2\abs{\om_1+\om_2}^2}(s2\abs{\om_1+\om_2}^2-1)+1)\\
				&=\abs{t-s}\frac{1}{2\abs{\om_1+\om_2}^2}\euler^{s2\abs{\om_1+\om_2}^2}+\frac{1}{(2\abs{\om_1+\om_2}^2)^2}\euler^{s2\abs{\om_1+\om_2}^2}-t\frac{1}{2\abs{\om_1+\om_2}^2}-\frac{1}{(2\abs{\om_1+\om_2}^2)^2}.
			\end{split}
		\end{equation*}
		\emph{Case $u_3<u_3'$.} We obtain
		\begin{equation*}
			\begin{split}
				&\int_{0}^{s}\dd u_3\int_{u_3}^{s}\dd u_3' \euler^{(u_3-u_3')(\abs{\om_1}^2+\abs{\om_2}^2)}\euler^{(u_3+u_3')\abs{\om_1+\om_2}^2}\\
				&=\int_{0}^{s}\dd u_3 \euler^{u_32\abs{\om_1+\om_2}^2}(s-u_3)\\
				&=s\frac{1}{2\abs{\om_1+\om_2}^2}(\euler^{s2\abs{\om_1+\om_2}^2}-1)-\frac{1}{(2\abs{\om_1+\om_2}^2)^2}(\euler^{s2\abs{\om_1+\om_2}^2}(s2\abs{\om_1+\om_2}^2-1)+1)\\
				&=\frac{1}{(2\abs{\om_1+\om_2}^2)^2}\euler^{s2\abs{\om_1+\om_2}^2}-s\frac{1}{2\abs{\om_1+\om_2}^2}-\frac{1}{(2\abs{\om_1+\om_2}^2)^2}.
			\end{split}
		\end{equation*}
		All in all, if $(\om_1\perp\om_2)$,
		\begin{equation*}
			\begin{split}
				&\int_{0}^{t}\dd u_3\int_{0}^{s}\dd u_3' \euler^{-\abs{u_3-u_3'}(\abs{\om_1}^2+\abs{\om_2}^2)}\euler^{-(t+s-(u_3+u_3'))\abs{\om_1+\om_2}^2}\\
				&=\frac{1}{2}\frac{1}{\abs{\om_1+\om_2}^2}\abs{t-s}\euler^{-\abs{t-s}\abs{\om_1+\om_2}^2}+\frac{1}{2}\frac{1}{\abs{\om_1+\om_2}^4}\euler^{-\abs{t-s}\abs{\om_1+\om_2}^2}\\
				&\quad-(t+s)\frac{1}{2}\frac{1}{\abs{\om_1+\om_2}^2}\euler^{-(t+s)\abs{\om_1+\om_2}^2}-\frac{1}{2}\frac{1}{\abs{\om_1+\om_2}^4}\euler^{-(t+s)\abs{\om_1+\om_2}^2}.
			\end{split}
		\end{equation*}
		The first claim follows. Next assume $\lnot(\om_1\perp\om_2)$. \emph{Case $u_3'\leq u_3$.} We compute
		\begin{equation*}
			\begin{split}
				&\int_{0}^{s}\dd u_3'\int_{u_3'}^{t}\dd u_3 \euler^{-(u_3-u_3')(\abs{\om_1}^2+\abs{\om_2}^2)}\euler^{(u_3+u_3')\abs{\om_1+\om_2}^2}\\
				&=\frac{1}{\abs{\om_1+\om_2}^2-\abs{\om_1}^2-\abs{\om_2}^2}\int_{0}^{s}\dd u_3'\euler^{u_3'(\abs{\om_1}^2+\abs{\om_2}^2+\abs{\om_1+\om_2}^2)}\\
				&\quad\times(\euler^{t(\abs{\om_1+\om_2}^2-\abs{\om_1}^2-\abs{\om_2}^2)}-\euler^{u_3'(\abs{\om_1+\om_2}^2-\abs{\om_1}^2-\abs{\om_2}^2)})\\
				&=-\frac{1}{\abs{\om_1}^2+\abs{\om_2}^2-\abs{\om_1+\om_2}^2}\frac{1}{\abs{\om_1}^2+\abs{\om_2}^2+\abs{\om_1+\om_2}^2}\\
				&\quad\times (\euler^{s(\abs{\om_1+\om_2}^2+\abs{\om_1}^2+\abs{\om_2}^2)}-1)\euler^{t(\abs{\om_1+\om_2}^2-\abs{\om_1}^2-\abs{\om_2}^2)}\\
				&\quad+\frac{1}{\abs{\om_1}^2+\abs{\om_2}^2-\abs{\om_1+\om_2}^2}\frac{1}{2\abs{\om_1+\om_2}^2}(\euler^{s2\abs{\om_1+\om_2}^2}-1)\\
				&=-\frac{1}{\abs{\om_1}^2+\abs{\om_2}^2-\abs{\om_1+\om_2}^2}\frac{1}{\abs{\om_1}^2+\abs{\om_2}^2+\abs{\om_1+\om_2}^2}\euler^{s(\abs{\om_1+\om_2}^2+\abs{\om_1}^2+\abs{\om_2}^2)}\euler^{t(\abs{\om_1+\om_2}^2-\abs{\om_1}^2-\abs{\om_2}^2)}\\
				&\quad+\frac{1}{\abs{\om_1}^2+\abs{\om_2}^2-\abs{\om_1+\om_2}^2}\frac{1}{2\abs{\om_1+\om_2}^2}\euler^{s2\abs{\om_1+\om_2}^2}\\
				&\quad+\frac{1}{\abs{\om_1}^2+\abs{\om_2}^2-\abs{\om_1+\om_2}^2}\frac{1}{\abs{\om_1}^2+\abs{\om_2}^2+\abs{\om_1+\om_2}^2}\euler^{t(\abs{\om_1+\om_2}^2-\abs{\om_1}^2-\abs{\om_2}^2)}\\
				&\quad-\frac{1}{\abs{\om_1}^2+\abs{\om_2}^2-\abs{\om_1+\om_2}^2}\frac{1}{2\abs{\om_1+\om_2}^2}.
			\end{split}
		\end{equation*}
		\emph{Case $u_3<u_3'$.} We compute
		\begin{equation*}
			\begin{split}
				&\int_{0}^{s}\dd u_3\int_{u_3}^{s}\dd u_3' \euler^{(u_3-u_3')(\abs{\om_1}^2+\abs{\om_2}^2)}\euler^{(u_3+u_3')\abs{\om_1+\om_2}^2}\\
				&=\frac{1}{\abs{\om_1+\om_2}^2-\abs{\om_1}^2-\abs{\om_2}^2}\int_{0}^{s}\dd u_3\euler^{u_3(\abs{\om_1+\om_2}^2+\abs{\om_1}^2+\abs{\om_2}^2)}\\
				&\quad\times(\euler^{s(\abs{\om_1+\om_2}^2-\abs{\om_1}^2-\abs{\om_2}^2)}-\euler^{u_3(\abs{\om_1+\om_2}^2-\abs{\om_1}^2-\abs{\om_2}^2)})\\
				&=-\frac{1}{\abs{\om_1}^2+\abs{\om_2}^2-\abs{\om_1+\om_2}^2}\frac{1}{\abs{\om_1}^2+\abs{\om_2}^2+\abs{\om_1+\om_2}^2}\\
				&\quad\times (\euler^{s(\abs{\om_1+\om_2}^2+\abs{\om_1}^2+\abs{\om_2}^2)}-1)\euler^{s(\abs{\om_1+\om_2}^2-\abs{\om_1}^2-\abs{\om_2}^2)}\\
				&\quad+\frac{1}{\abs{\om_1}^2+\abs{\om_2}^2-\abs{\om_1+\om_2}^2}\frac{1}{2\abs{\om_1+\om_2}^2}(\euler^{s2\abs{\om_1+\om_2}^2}-1)\\
				&=-\frac{1}{\abs{\om_1}^2+\abs{\om_2}^2-\abs{\om_1+\om_2}^2}\frac{1}{\abs{\om_1}^2+\abs{\om_2}^2+\abs{\om_1+\om_2}^2}\euler^{s2\abs{\om_1+\om_2}^2}\\
				&\quad+\frac{1}{\abs{\om_1}^2+\abs{\om_2}^2-\abs{\om_1+\om_2}^2}\frac{1}{2\abs{\om_1+\om_2}^2}\euler^{s2\abs{\om_1+\om_2}^2}\\
				&\quad+\frac{1}{\abs{\om_1}^2+\abs{\om_2}^2-\abs{\om_1+\om_2}^2}\frac{1}{\abs{\om_1}^2+\abs{\om_2}^2+\abs{\om_1+\om_2}^2}\euler^{s(\abs{\om_1+\om_2}^2-\abs{\om_1}^2-\abs{\om_2}^2)}\\
				&\quad-\frac{1}{\abs{\om_1}^2+\abs{\om_2}^2-\abs{\om_1+\om_2}^2}\frac{1}{2\abs{\om_1+\om_2}^2}.
			\end{split}
		\end{equation*}
		We define the leading term,
		\begin{equation*}
			\begin{split}
				&\mathsf{SL}_{s,t}\SYpsilon(\om_1,\om_2)\\
				&\defeq\frac{1}{2^2}\abs{\om_1}^{-2}\abs{\om_2}^{-2}\Bigl(-\frac{1}{\abs{\om_1}^2+\abs{\om_2}^2-\abs{\om_1+\om_2}^2}\frac{1}{\abs{\om_1}^2+\abs{\om_2}^2+\abs{\om_1+\om_2}^2}\\
				&\multiquad[9]\times(\euler^{-\abs{t-s}(\abs{\om_1}^2+\abs{\om_2}^2)}+\euler^{-\abs{t-s}\abs{\om_1+\om_2}^2})\\
				&\multiquad[10]+\frac{1}{\abs{\om_1}^2+\abs{\om_2}^2-\abs{\om_1+\om_2}^2}\frac{1}{\abs{\om_1+\om_2}^2}\euler^{-\abs{t-s}\abs{\om_1+\om_2}^2}\Bigr)
			\end{split}
		\end{equation*}
		and the error term,
		\begin{equation*}
			\begin{split}
				&\mathsf{SE}_{s,t}\SYpsilon(\om_1,\om_2)\\
				&\defeq\frac{1}{2^2}\abs{\om_1}^{-2}\abs{\om_2}^{-2}\Bigl(\frac{1}{\abs{\om_1}^2+\abs{\om_2}^2-\abs{\om_1+\om_2}^2}\frac{1}{\abs{\om_1}^2+\abs{\om_2}^2+\abs{\om_1+\om_2}^2}\\
				&\multiquad[8]\times \euler^{t(\abs{\om_1+\om_2}^2-\abs{\om_1}^2-\abs{\om_2}^2)}\euler^{-(t+s)\abs{\om_1+\om_2}^2}\\
				&\multiquad[8]-\frac{1}{\abs{\om_1}^2+\abs{\om_2}^2-\abs{\om_1+\om_2}^2}\frac{1}{2\abs{\om_1+\om_2}^2}\euler^{-(t+s)\abs{\om_1+\om_2}^2}\\
				&\multiquad[8]+\frac{1}{\abs{\om_1}^2+\abs{\om_2}^2-\abs{\om_1+\om_2}^2}\frac{1}{\abs{\om_1}^2+\abs{\om_2}^2+\abs{\om_1+\om_2}^2}\\
				&\multiquad[8]\times \euler^{s(\abs{\om_1+\om_2}^2-\abs{\om_1}^2-\abs{\om_2}^2)}\euler^{-(t+s)\abs{\om_1+\om_2}^2}\\
				&\multiquad[10]-\frac{1}{\abs{\om_1}^2+\abs{\om_2}^2-\abs{\om_1+\om_2}^2}\frac{1}{2\abs{\om_1+\om_2}^2}\euler^{-(t+s)\abs{\om_1+\om_2}^2}\Bigr).
			\end{split}
		\end{equation*}
		All in all, if $\lnot(\om_1\perp\om_2)$,
		\begin{equation*}
			\mathsf{SL}_{s,t}\SYpsilon(\om_1,\om_2)=\mathsf{SL}_{s,t}\SYpsilon(\om_1,\om_2)+\mathsf{SE}_{s,t}\SYpsilon(\om_1,\om_2).
		\end{equation*}
		We rewrite the leading term,
		\begin{equation*}
			\begin{split}
				&\mathsf{SL}_{s,t}\SYpsilon(\om_1,\om_2)\\
				&=\frac{1}{2^2}\abs{\om_1}^{-2}\abs{\om_2}^{-2}\Bigl(-\frac{1}{\abs{\om_1}^2+\abs{\om_2}^2-\abs{\om_1+\om_2}^2}\frac{1}{\abs{\om_1}^2+\abs{\om_2}^2+\abs{\om_1+\om_2}^2}\\
				&\multiquad[8]\times(\euler^{-\abs{t-s}(\abs{\om_1}^2+\abs{\om_2}^2)}+\euler^{-\abs{t-s}\abs{\om_1+\om_2}^2})\\
				&\multiquad[8]+\frac{1}{2}\frac{1}{\abs{\om_1}^2+\abs{\om_2}^2-\abs{\om_1+\om_2}^2}\frac{1}{\abs{\om_1+\om_2}^2}(\euler^{-\abs{t-s}\abs{\om_1+\om_2}^2}-\euler^{-\abs{t-s}(\abs{\om_1}^2+\abs{\om_2}^2)})\\
				&\multiquad[8]+\frac{1}{2}\frac{1}{\abs{\om_1}^2+\abs{\om_2}^2-\abs{\om_1+\om_2}^2}\frac{1}{\abs{\om_1+\om_2}^2}(\euler^{-\abs{t-s}\abs{\om_1+\om_2}^2}+\euler^{-\abs{t-s}(\abs{\om_1}^2+\abs{\om_2}^2)})\Bigr).
			\end{split}
		\end{equation*}
		To combine the first and third summand in $\mathsf{SL}_{s,t}\SYpsilon(\om_1,\om_2)$, we note that
		\begin{equation*}
			\begin{split}
				&\frac{1}{\abs{\om_1}^2+\abs{\om_2}^2-\abs{\om_1+\om_2}^2}\Bigl(-\frac{1}{\abs{\om_1}^2+\abs{\om_2}^2+\abs{\om_1+\om_2}^2}+\frac{1}{2}\frac{1}{\abs{\om_1+\om_2}^2}\Bigr)\\
				&=\frac{1}{2}\frac{1}{\abs{\om_1}^2+\abs{\om_2}^2+\abs{\om_1+\om_2}^2}\frac{1}{\abs{\om_1+\om_2}^2}.
			\end{split}
		\end{equation*}
		Hence if $\lnot(\om_1\perp\om_2)$,
		\begin{equation*}
			\begin{split}
				&\mathsf{SL}_{s,t}\SYpsilon(\om_1,\om_2)\\
				&=\frac{1}{2^3}\abs{\om_1}^{-2}\abs{\om_2}^{-2}\abs{\om_1+\om_2}^{-2}\Bigl(\frac{1}{\abs{\om_1}^2+\abs{\om_2}^2-\abs{\om_1+\om_2}^2}(\euler^{-\abs{t-s}\abs{\om_1+\om_2}^2}-\euler^{-\abs{t-s}(\abs{\om_1}^2+\abs{\om_2}^2)})\\
				&\multiquad[13]+\frac{1}{\abs{\om_1}^2+\abs{\om_2}^2+\abs{\om_1+\om_2}^2}(\euler^{-\abs{t-s}\abs{\om_1+\om_2}^2}+\euler^{-\abs{t-s}(\abs{\om_1}^2+\abs{\om_2}^2)})\Bigr).
			\end{split}
		\end{equation*}
		Furthermore,
		\begin{equation*}
			\begin{split}
				&\mathsf{SE}_{s,t}\SYpsilon(\om_1,\om_2)\\
				&=\frac{1}{2^2}\abs{\om_1}^{-2}\abs{\om_2}^{-2}\Bigl(\frac{1}{\abs{\om_1}^2+\abs{\om_2}^2-\abs{\om_1+\om_2}^2}\frac{1}{\abs{\om_1}^2+\abs{\om_2}^2+\abs{\om_1+\om_2}^2}\\
				&\multiquad[8]\times (\euler^{t(\abs{\om_1+\om_2}^2-\abs{\om_1}^2-\abs{\om_2}^2)}+\euler^{s(\abs{\om_1+\om_2}^2-\abs{\om_1}^2-\abs{\om_2}^2)})\euler^{-(t+s)\abs{\om_1+\om_2}^2}\\
				&\multiquad[12]-\frac{1}{\abs{\om_1}^2+\abs{\om_2}^2-\abs{\om_1+\om_2}^2}\frac{1}{\abs{\om_1+\om_2}^2}\euler^{-(t+s)\abs{\om_1+\om_2}^2}\Bigr).
			\end{split}
		\end{equation*}
		The second claim follows.
	\end{proof}
	Using Lemma~\ref{lem:shape_ypsilon}, we can derive an explicit expression for $\mathsf{D}_{s,t}\SYpsilon(\om_1,\om_2)$.
	\begin{lemma}
		Let $s,t\geq0$ and $\om_1,\om_2\in2\uppi\mbZ^2\setminus\{0\}$ such that $\om_1+\om_2\not=0$. Then there exist some $\mathsf{DL}_{s,t}\SYpsilon(\om_1,\om_2)$ and $\mathsf{DE}_{s,t}\SYpsilon(\om_1,\om_2)$, such that
		\begin{equation*}
			\mathsf{D}_{s,t}\SYpsilon(\om_1,\om_2)=\mathsf{DL}_{s,t}\SYpsilon(\om_1,\om_2)+\mathsf{DE}_{s,t}\SYpsilon(\om_1,\om_2).
		\end{equation*}
		In the case $(\om_1\perp\om_2)$, we obtain
		\begin{equation*}
			\mathsf{DL}_{s,t}\SYpsilon(\om_1,\om_2)=\frac{1}{2^2}\abs{\om_1}^{-2}\abs{\om_2}^{-2}\abs{\om_1+\om_2}^{-4}\Bigl(1-\euler^{-\abs{t-s}\abs{\om_1+\om_2}^2}-\abs{t-s}\abs{\om_1+\om_2}^2\euler^{-\abs{t-s}\abs{\om_1+\om_2}^2}\Bigr)
		\end{equation*}
		and
		\begin{equation*}
			\begin{split}
				&\mathsf{DE}_{s,t}\SYpsilon(\om_1,\om_2)\\
				&=\frac{1}{2^3}\abs{\om_1}^{-2}\abs{\om_2}^{-2}\abs{\om_1+\om_2}^{-4}\Bigl(2t\abs{\om_1+\om_2}^2(\euler^{-(t+s)\abs{\om_1+\om_2}^2}-\euler^{-2t\abs{\om_1+\om_2}^2})\\
				&\multiquad[13]+2s\abs{\om_1+\om_2}^2(\euler^{-(t+s)\abs{\om_1+\om_2}^2}-\euler^{-2s\abs{\om_1+\om_2}^2})\\
				&\multiquad[20]-(\euler^{-t\abs{\om_1+\om_2}^2}-\euler^{-s\abs{\om_1+\om_2}^2})^2\Bigr).
			\end{split}
		\end{equation*}
		In the case $\lnot(\om_1\perp\om_2)$, we obtain
		\begin{equation*}
			\begin{split}
				&\mathsf{DL}_{s,t}\SYpsilon(\om_1,\om_2)\\
				&=\frac{1}{2^2}\abs{\om_1}^{-2}\abs{\om_2}^{-2}\abs{\om_1+\om_2}^{-2}\Bigl(\frac{1}{\abs{\om_1}^2+\abs{\om_2}^2-\abs{\om_1+\om_2}^2}(\euler^{-\abs{t-s}(\abs{\om_1}^2+\abs{\om_2}^2)}-\euler^{-\abs{t-s}\abs{\om_1+\om_2}^2})\\
				&\multiquad[12]+\frac{1}{\abs{\om_1}^2+\abs{\om_2}^2+\abs{\om_1+\om_2}^2}(2-\euler^{-\abs{t-s}\abs{\om_1+\om_2}^2}-\euler^{-\abs{t-s}(\abs{\om_1}^2+\abs{\om_2}^2)})\Bigr)
			\end{split}
		\end{equation*}
		and
		\begin{equation*}
			\begin{split}
				&\mathsf{DE}_{s,t}\SYpsilon(\om_1,\om_2)\\
				&=\frac{1}{2^2}\abs{\om_1}^{-2}\abs{\om_2}^{-2}\Bigl(\frac{1}{\abs{\om_1}^2+\abs{\om_2}^2-\abs{\om_1+\om_2}^2}\frac{1}{\abs{\om_1}^2+\abs{\om_2}^2+\abs{\om_1+\om_2}^2}\\
				&\multiquad[8]\times(2(\euler^{t(\abs{\om_1+\om_2}^2-\abs{\om_1}^2-\abs{\om_2}^2)}-1)(\euler^{-2t\abs{\om_1+\om_2}^2}-\euler^{-(t+s)\abs{\om_1+\om_2}^2})\\
				&\multiquad[10]+2(\euler^{s(\abs{\om_1+\om_2}^2-\abs{\om_1}^2-\abs{\om_2}^2)}-1)(\euler^{-2s\abs{\om_1+\om_2}^2}-\euler^{-(t+s)\abs{\om_1+\om_2}^2}))\\
				&\multiquad[9]-\frac{1}{\abs{\om_1+\om_2}^2}\frac{1}{\abs{\om_1}^2+\abs{\om_2}^2+\abs{\om_1+\om_2}^2}(\euler^{-t\abs{\om_1+\om_2}^2}-\euler^{-s\abs{\om_1+\om_2}^2})^2\Bigr).
			\end{split}
		\end{equation*}
	\end{lemma}
	\begin{proof}
		We discuss the case $\mathsf{DE}_{s,t}\SYpsilon(\om_1,\om_2)$, $\lnot(\om_1\perp\om_2)$, which is the only one that requires additional manipulations. By the definition of $\mathsf{DE}_{s,t}\SYpsilon(\om_1,\om_2)$,
		\begin{equation*}
			\begin{split}
				&\mathsf{DE}_{s,t}\SYpsilon(\om_1,\om_2)\\
				&=\frac{1}{2^2}\abs{\om_1}^{-2}\abs{\om_2}^{-2}\\
				&\quad\times\Bigl(\frac{1}{\abs{\om_1}^2+\abs{\om_2}^2-\abs{\om_1+\om_2}^2}\frac{1}{\abs{\om_1}^2+\abs{\om_2}^2+\abs{\om_1+\om_2}^2}\\
				&\qquad\quad\times (2\euler^{t(-\abs{\om_1+\om_2}^2-\abs{\om_1}^2-\abs{\om_2}^2)}+2\euler^{s(-\abs{\om_1+\om_2}^2-\abs{\om_1}^2-\abs{\om_2}^2)}\\
				&\multiquad[5]-2(\euler^{t(\abs{\om_1+\om_2}^2-\abs{\om_1}^2-\abs{\om_2}^2)}+\euler^{s(\abs{\om_1+\om_2}^2-\abs{\om_1}^2-\abs{\om_2}^2)})\euler^{-(t+s)\abs{\om_1+\om_2}^2})\\
				&\qquad\quad-\frac{1}{\abs{\om_1}^2+\abs{\om_2}^2-\abs{\om_1+\om_2}^2}\frac{1}{\abs{\om_1+\om_2}^2}(\euler^{-2t\abs{\om_1+\om_2}^2}+\euler^{-2s\abs{\om_1+\om_2}^2}-2\euler^{-(t+s)\abs{\om_1+\om_2}^2})\Bigr).
			\end{split}
		\end{equation*}
		We rewrite
		\begin{equation*}
			\begin{split}
				&\mathsf{DE}_{s,t}\SYpsilon(\om_1,\om_2)\\
				&=\frac{1}{2^2}\abs{\om_1}^{-2}\abs{\om_2}^{-2}\\
				&\quad\times\Bigl(\frac{1}{\abs{\om_1}^2+\abs{\om_2}^2-\abs{\om_1+\om_2}^2}\frac{1}{\abs{\om_1}^2+\abs{\om_2}^2+\abs{\om_1+\om_2}^2}\\
				&\qquad\quad\times (2\euler^{t(-\abs{\om_1+\om_2}^2-\abs{\om_1}^2-\abs{\om_2}^2)}-2\euler^{-2t\abs{\om_1+\om_2}^2}\\
				&\multiquad[5]+2\euler^{s(-\abs{\om_1+\om_2}^2-\abs{\om_1}^2-\abs{\om_2}^2)}-2\euler^{-2s\abs{\om_1+\om_2}^2}\\
				&\multiquad[5]-2(\euler^{t(\abs{\om_1+\om_2}^2-\abs{\om_1}^2-\abs{\om_2}^2)}-1)\euler^{-(t+s)\abs{\om_1+\om_2}^2}\\
				&\multiquad[5]-2(\euler^{s(\abs{\om_1+\om_2}^2-\abs{\om_1}^2-\abs{\om_2}^2)}-1)\euler^{-(t+s)\abs{\om_1+\om_2}^2})\\
				&\qquad\quad-\frac{1}{\abs{\om_1}^2+\abs{\om_2}^2-\abs{\om_1+\om_2}^2}\Bigl(\frac{1}{\abs{\om_1+\om_2}^2}-\frac{2}{\abs{\om_1}^2+\abs{\om_2}^2+\abs{\om_1+\om_2}^2}\Bigr)\\
				&\multiquad[14]\times (\euler^{-2t\abs{\om_1+\om_2}^2}+\euler^{-2s\abs{\om_1+\om_2}^2}-2\euler^{-(t+s)\abs{\om_1+\om_2}^2})\Bigr).
			\end{split}
		\end{equation*}
		We simplify the second term,
		\begin{equation*}
			\begin{split}
				&-\frac{1}{\abs{\om_1}^2+\abs{\om_2}^2-\abs{\om_1+\om_2}^2}\Bigl(\frac{1}{\abs{\om_1+\om_2}^2}-\frac{2}{\abs{\om_1}^2+\abs{\om_2}^2+\abs{\om_1+\om_2}^2}\Bigr)\\
				&\quad\times (\euler^{-2t\abs{\om_1+\om_2}^2}+\euler^{-2s\abs{\om_1+\om_2}^2}-2\euler^{-(t+s)\abs{\om_1+\om_2}^2})\\
				&=-\frac{1}{\abs{\om_1+\om_2}^2}\frac{1}{\abs{\om_1}^2+\abs{\om_2}^2+\abs{\om_1+\om_2}^2}(\euler^{-t\abs{\om_1+\om_2}^2}-\euler^{-s\abs{\om_1+\om_2}^2})^2.
			\end{split}
		\end{equation*}
		For the first term,
		\begin{equation*}
			\begin{split}
				&2\euler^{t(-\abs{\om_1+\om_2}^2-\abs{\om_1}^2-\abs{\om_2}^2)}-2\euler^{-2t\abs{\om_1+\om_2}^2}\\
				&\quad+2\euler^{s(-\abs{\om_1+\om_2}^2-\abs{\om_1}^2-\abs{\om_2}^2)}-2\euler^{-2s\abs{\om_1+\om_2}^2}\\
				&\quad-2(\euler^{t(\abs{\om_1+\om_2}^2-\abs{\om_1}^2-\abs{\om_2}^2)}-1)\euler^{-(t+s)\abs{\om_1+\om_2}^2}\\
				&\quad-2(\euler^{s(\abs{\om_1+\om_2}^2-\abs{\om_1}^2-\abs{\om_2}^2)}-1)\euler^{-(t+s)\abs{\om_1+\om_2}^2}\\
				&=2(\euler^{t(\abs{\om_1+\om_2}^2-\abs{\om_1}^2-\abs{\om_2}^2)}-1)(\euler^{-2t\abs{\om_1+\om_2}^2}-\euler^{-(t+s)\abs{\om_1+\om_2}^2})\\
				&\quad+2(\euler^{s(\abs{\om_1+\om_2}^2-\abs{\om_1}^2-\abs{\om_2}^2)}-1)(\euler^{-2s\abs{\om_1+\om_2}^2}-\euler^{-(t+s)\abs{\om_1+\om_2}^2}).
			\end{split}
		\end{equation*}
		This yields the claim.
	\end{proof}
\end{details}
\begin{lemma}\label{lem:difference_ypsilon_bound}
	Let $s,t\geq0$, $\gamma\in[0,1]$ and $\om_1,\om_2\in2\uppi\mbZ^2\setminus\{0\}$ be such that $\om_1+\om_2\not=0$.
	\begin{enumerate}
		\item In the case $(\om_1\perp\om_2)$, we obtain
		\begin{equation*}
			\mathsf{D}_{s,t}\SYpsilon(\om_1,\om_2)\lesssim\abs{t-s}^{\gamma}\abs{\om_1}^{-2}\abs{\om_2}^{-2}\abs{\om_1+\om_2}^{-4+2\gamma}.
		\end{equation*}
		\item In the case $\lnot(\om_1\perp\om_2)$, we obtain
		\begin{equation*}
			\mathsf{D}_{s,t}\SYpsilon(\om_1,\om_2)\lesssim\abs{t-s}^{\gamma}\abs{\om_1}^{-4+2\gamma}\abs{\om_2}^{-2}\abs{\om_1+\om_2}^{-2}+\abs{t-s}^{\gamma}\abs{\om_1}^{-4}\abs{\om_2}^{-2}\abs{\om_1+\om_2}^{-2+2\gamma}.
		\end{equation*}
	\end{enumerate}
	The implied constants are uniform in $\om_1$, $\om_2$.
\end{lemma}
\begin{proof}
	To bound our shape coefficients, we first derive explicit expressions, a step that was inspired by~\cite[Proof of Thm.~2.1]{tsatsoulis_weber_18}. To do so, we evaluate the exponential integrals in~\eqref{eq:shape_coefficient_ypsilon} and compute $\mathsf{D}_{s,t}\SYpsilon(\om_1,\om_2)$ through~\eqref{eq:increment_shape_coefficient_ypsilon}. We can then further decompose 
	\begin{equation*}
		\mathsf{D}_{s,t}\SYpsilon(\om_1,\om_2)=\mathsf{DL}_{s,t}\SYpsilon(\om_1,\om_2)+\mathsf{DE}_{s,t}\SYpsilon(\om_1,\om_2),
	\end{equation*}
	into the \emph{leading term} $\mathsf{DL}_{s,t}\SYpsilon(\om_1,\om_2)$ and the \emph{error term} $\mathsf{DE}_{s,t}\SYpsilon(\om_1,\om_2)$. The error term is generated by the zero initial condition of the noise, i.e.\ the remaining restriction $u_3,u_3'\geq0$ in~\eqref{eq:shape_coefficient_ypsilon}.
	
	Assume $(\om_1\perp\om_2)$, then
	\begin{equation*}
		\mathsf{DL}_{s,t}\SYpsilon(\om_1,\om_2)=\frac{1}{2^2}\abs{\om_1}^{-2}\abs{\om_2}^{-2}\abs{\om_1+\om_2}^{-4}\Bigl(1-\euler^{-\abs{t-s}\abs{\om_1+\om_2}^2}-\abs{t-s}\abs{\om_1+\om_2}^2\euler^{-\abs{t-s}\abs{\om_1+\om_2}^2}\Bigr)
	\end{equation*}
	and
	\begin{equation*}
		\begin{split}
			&\mathsf{DE}_{s,t}\SYpsilon(\om_1,\om_2)=\frac{1}{2^3}\abs{\om_1}^{-2}\abs{\om_2}^{-2}\abs{\om_1+\om_2}^{-4}\Bigl(2t\abs{\om_1+\om_2}^2(\euler^{-(t+s)\abs{\om_1+\om_2}^2}-\euler^{-2t\abs{\om_1+\om_2}^2})\\
			&\multiquad[20]+2s\abs{\om_1+\om_2}^2(\euler^{-(t+s)\abs{\om_1+\om_2}^2}-\euler^{-2s\abs{\om_1+\om_2}^2})\\
			&\multiquad[27]-(\euler^{-t\abs{\om_1+\om_2}^2}-\euler^{-s\abs{\om_1+\om_2}^2})^2\Bigr).
		\end{split}
	\end{equation*}
	We first consider $\mathsf{DL}_{s,t}\SYpsilon(\om_1,\om_2)$ and estimate by~\eqref{eq:interpolation} for every $\gamma\in[0,1]$,
	\begin{equation*}
		\mathsf{DL}_{s,t}\SYpsilon(\om_1,\om_2)\lesssim\abs{t-s}^{\gamma}\abs{\om_1}^{-2}\abs{\om_2}^{-2}\abs{\om_1+\om_2}^{-4+2\gamma}.
	\end{equation*}
	Next we estimate the error term $\mathsf{DE}_{s,t}\SYpsilon(\om_1,\om_2)$. By symmetry it is enough to consider $s\leq t$, for which
	\begin{equation*}
		\euler^{-(t+s)\abs{\om_1+\om_2}^2}-\euler^{-2s\abs{\om_1+\om_2}^2}\leq0.
	\end{equation*}
	\begin{details}
		The remaining term can be bounded with~\eqref{eq:rapid_decay},
		\begin{equation*}
			\begin{split}
				2t\abs{\om_1+\om_2}^{2}(\euler^{-(t+s)\abs{\om_1+\om_2}^2}-\euler^{-2t\abs{\om_1+\om_2}^2})&=2t\abs{\om_1+\om_2}^2\euler^{-t\abs{\om_1+\om_2}^2}(\euler^{-s\abs{\om_1+\om_2}^2}-\euler^{-t\abs{\om_1+\om_2}^2})\\
				&\lesssim(1-\euler^{-\abs{t-s}\abs{\om_1+\om_2}^2}),
			\end{split}
		\end{equation*}
		which can be controlled by~\eqref{eq:interpolation}. Consequently,
		\begin{equation*}
			\mathsf{DE}_{s,t}\SYpsilon(\om_1,\om_2)\lesssim\abs{t-s}^{\gamma}\abs{\om_1}^{-2}\abs{\om_2}^{-2}\abs{\om_1+\om_2}^{-4+2\gamma}.
		\end{equation*}
		This yields the first claim.
	\end{details}
	We can then proceed to bound the remaining non-negative term of $\mathsf{DE}_{s,t}\SYpsilon(\om_1,\om_2)$ by~\eqref{eq:interpolation} and~\eqref{eq:rapid_decay}, giving the result in the case $(\om_1\perp\om_2)$.
	
	Next assume $\lnot(\om_1\perp\om_2)$. We obtain the expressions
	\begin{equation}\label{eq:ypsilon_DL_non_orthogonal}
		\begin{split}
			&\mathsf{DL}_{s,t}\SYpsilon(\om_1,\om_2)\\
			&=\frac{1}{2^2}\abs{\om_1}^{-2}\abs{\om_2}^{-2}\abs{\om_1+\om_2}^{-2}\Bigl(\frac{1}{\abs{\om_1}^2+\abs{\om_2}^2-\abs{\om_1+\om_2}^2}(\euler^{-\abs{t-s}(\abs{\om_1}^2+\abs{\om_2}^2)}-\euler^{-\abs{t-s}\abs{\om_1+\om_2}^2})\\
			&\multiquad[13]+\frac{1}{\abs{\om_1}^2+\abs{\om_2}^2+\abs{\om_1+\om_2}^2}(2-\euler^{-\abs{t-s}\abs{\om_1+\om_2}^2}-\euler^{-\abs{t-s}(\abs{\om_1}^2+\abs{\om_2}^2)})\Bigr)
		\end{split}
	\end{equation}
	and
	\begin{equation}\label{eq:ypsilon_DE_non_orthogonal}
		\begin{split}
			&\mathsf{DE}_{s,t}\SYpsilon(\om_1,\om_2)\\
			&=\frac{1}{2^2}\abs{\om_1}^{-2}\abs{\om_2}^{-2}\Bigl(\frac{1}{\abs{\om_1}^2+\abs{\om_2}^2-\abs{\om_1+\om_2}^2}\frac{1}{\abs{\om_1}^2+\abs{\om_2}^2+\abs{\om_1+\om_2}^2}\\
			&\multiquad[8]\times(2(\euler^{t(\abs{\om_1+\om_2}^2-\abs{\om_1}^2-\abs{\om_2}^2)}-1)(\euler^{-2t\abs{\om_1+\om_2}^2}-\euler^{-(t+s)\abs{\om_1+\om_2}^2})\\
			&\multiquad[10]+2(\euler^{s(\abs{\om_1+\om_2}^2-\abs{\om_1}^2-\abs{\om_2}^2)}-1)(\euler^{-2s\abs{\om_1+\om_2}^2}-\euler^{-(t+s)\abs{\om_1+\om_2}^2}))\\
			&\multiquad[9]-\frac{1}{\abs{\om_1+\om_2}^2}\frac{1}{\abs{\om_1}^2+\abs{\om_2}^2+\abs{\om_1+\om_2}^2}(\euler^{-t\abs{\om_1+\om_2}^2}-\euler^{-s\abs{\om_1+\om_2}^2})^2\Bigr).
		\end{split}
	\end{equation}
	The first term in~\eqref{eq:ypsilon_DL_non_orthogonal},
	\begin{equation*}
		\frac{1}{\abs{\om_1}^2+\abs{\om_2}^2-\abs{\om_1+\om_2}^2}(\euler^{-\abs{t-s}(\abs{\om_1}^2+\abs{\om_2}^2)}-\euler^{-\abs{t-s}\abs{\om_1+\om_2}^2}),
	\end{equation*}
	is non-positive so that by~\eqref{eq:interpolation},
	\begin{equation*}
		\mathsf{DL}_{s,t}\SYpsilon(\om_1,\om_2)\lesssim\abs{t-s}^{\gamma}\abs{\om_1}^{-4+2\gamma}\abs{\om_2}^{-2}\abs{\om_1+\om_2}^{-2}.
	\end{equation*}
	\begin{details}
		We arrive at
		\begin{equation*}
			\begin{split}
				&\mathsf{DL}_{s,t}\SYpsilon(\om_1,\om_2)\\
				&\lesssim\abs{\om_1}^{-2}\abs{\om_2}^{-2}\abs{\om_1+\om_2}^{-2}\frac{1}{\abs{\om_1}^2+\abs{\om_2}^2+\abs{\om_1+\om_2}^2}(2-\euler^{-\abs{t-s}\abs{\om_1+\om_2}^2}-\euler^{-\abs{t-s}(\abs{\om_1}^2+\abs{\om_2}^2)}).
			\end{split}
		\end{equation*}
		By interpolation for $\gamma\in[0,1]$,
		\begin{equation*}
			\begin{split}
				&\mathsf{DL}_{s,t}\SYpsilon(\om_1,\om_2)\\
				&\lesssim\abs{\om_1}^{-2}\abs{\om_2}^{-2}\abs{\om_1+\om_2}^{-2}\frac{1}{\abs{\om_1}^2+\abs{\om_2}^2+\abs{\om_1+\om_2}^2}\abs{t-s}^{\gamma}(\abs{\om_1+\om_2}^{2\gamma}+(\abs{\om_1}^2+\abs{\om_2}^2)^{\gamma}).
			\end{split}
		\end{equation*}
		We estimate
		\begin{equation*}
			\frac{\abs{\om_1+\om_2}^{2\gamma}+(\abs{\om_1}^2+\abs{\om_2}^2)^{\gamma}}{\abs{\om_1}^2+\abs{\om_2}^2+\abs{\om_1+\om_2}^2}\leq2\frac{(\abs{\om_1}^2+\abs{\om_2}^2+\abs{\om_1+\om_2}^2)^{\gamma}}{\abs{\om_1}^2+\abs{\om_2}^2+\abs{\om_1+\om_2}^2}\lesssim\abs{\om_1}^{-2+2\gamma}.
		\end{equation*}
		This concludes the estimate of $\mathsf{DL}_{s,t}\SYpsilon(\om_1,\om_2)$.
	\end{details}
	In~\eqref{eq:ypsilon_DE_non_orthogonal}, we bound the second, non-positive term by $0$. In the first term, we distinguish cases to fix the sign of the prefactor $(\abs{\om_1}^2+\abs{\om_2}^2-\abs{\om_1+\om_2}^2)^{-1}$. Assume first $\abs{\om_1}^2+\abs{\om_2}^2<\abs{\om_1+\om_2}^2$ and by symmetry $s\leq t$.
	\begin{details}
		We rewrite
		\begin{equation*}
			\begin{split}
				&\frac{1}{\abs{\om_1}^2+\abs{\om_2}^2-\abs{\om_1+\om_2}^2}\frac{1}{\abs{\om_1}^2+\abs{\om_2}^2+\abs{\om_1+\om_2}^2}\\
				&\quad\times\Bigl((\euler^{t(\abs{\om_1+\om_2}^2-\abs{\om_1}^2-\abs{\om_2}^2)}-1)(\euler^{-2t\abs{\om_1+\om_2}^2}-\euler^{-(t+s)\abs{\om_1+\om_2}^2})\\
				&\qquad\quad+(\euler^{s(\abs{\om_1+\om_2}^2-\abs{\om_1}^2-\abs{\om_2}^2)}-1)(\euler^{-2s\abs{\om_1+\om_2}^2}-\euler^{-(t+s)\abs{\om_1+\om_2}^2})\Bigr)\\
				&=\frac{1}{\abs{\om_1+\om_2}^2-\abs{\om_1}^2-\abs{\om_2}^2}\frac{1}{\abs{\om_1}^2+\abs{\om_2}^2+\abs{\om_1+\om_2}^2}\\
				&\quad\times\Bigl(\euler^{-t(\abs{\om_1}^2+\abs{\om_2}^2)}(1-\euler^{-t(\abs{\om_1+\om_2}^2-\abs{\om_1}^2-\abs{\om_2}^2)})\euler^{-s\abs{\om_1+\om_2}^2}(1-\euler^{-(t-s)\abs{\om_1+\om_2}^2})\\
				&\qquad\quad+\euler^{-s(\abs{\om_1}^2+\abs{\om_2}^2)}(1-\euler^{-s(\abs{\om_1+\om_2}^2-\abs{\om_1}^2-\abs{\om_2}^2)})\euler^{-t\abs{\om_1+\om_2}^2}(1-\euler^{-(s-t)\abs{\om_1+\om_2}^2})\Bigr).
			\end{split}
		\end{equation*}
		By our assumption $s\leq t$,
		\begin{equation*}
			(1-\euler^{-s(\abs{\om_1+\om_2}^2-\abs{\om_1}^2-\abs{\om_2}^2)})(1-\euler^{-(s-t)\abs{\om_1+\om_2}^2})\leq0.
		\end{equation*}
		Furthermore,
		\begin{equation*}
			\euler^{-s\abs{\om_1+\om_2}^2}\leq \euler^{-s(\abs{\om_1}^2+\abs{\om_2}^2)}.
		\end{equation*}
	\end{details}
	We obtain by~\eqref{eq:interpolation} and~\eqref{eq:rapid_decay},
	\begin{equation}\label{eq:ypsilon_DE_non_orthogonal_bound_1}
		\begin{split}
			&\frac{1}{\abs{\om_1}^2+\abs{\om_2}^2-\abs{\om_1+\om_2}^2}\frac{1}{\abs{\om_1}^2+\abs{\om_2}^2+\abs{\om_1+\om_2}^2}\\
			&\quad\times\Bigl((\euler^{t(\abs{\om_1+\om_2}^2-\abs{\om_1}^2-\abs{\om_2}^2)}-1)(\euler^{-2t\abs{\om_1+\om_2}^2}-\euler^{-(t+s)\abs{\om_1+\om_2}^2})\\
			&\qquad\quad+(\euler^{s(\abs{\om_1+\om_2}^2-\abs{\om_1}^2-\abs{\om_2}^2)}-1)(\euler^{-2s\abs{\om_1+\om_2}^2}-\euler^{-(t+s)\abs{\om_1+\om_2}^2})\Bigr)\\
			&\leq\frac{1}{\abs{\om_1+\om_2}^2-\abs{\om_1}^2-\abs{\om_2}^2}\frac{1}{\abs{\om_1}^2+\abs{\om_2}^2+\abs{\om_1+\om_2}^2}\\
			&\quad\times \euler^{-(t+s)(\abs{\om_1}^2+\abs{\om_2}^2)}(1-\euler^{-t(\abs{\om_1+\om_2}^2-\abs{\om_1}^2-\abs{\om_2}^2)})(1-\euler^{-(t-s)\abs{\om_1+\om_2}^2})\\
			&\lesssim\abs{t-s}^{\gamma}\frac{\abs{\om_1+\om_2}^{2\gamma}}{\abs{\om_1}^2+\abs{\om_2}^2+\abs{\om_1+\om_2}^2}\frac{1}{\abs{\om_1}^2+\abs{\om_2}^2}\frac{t}{t+s}\\
			&\lesssim\abs{t-s}^{\gamma}\abs{\om_1+\om_2}^{-2+2\gamma}\abs{\om_1}^{-2}.
		\end{split}
	\end{equation}
	If instead $\abs{\om_1+\om_2}^2<\abs{\om_1}^2+\abs{\om_2}^2$, $s\leq t$, then
	\begin{details}
		\begin{equation*}
			\begin{split}
				&\frac{1}{\abs{\om_1}^2+\abs{\om_2}^2-\abs{\om_1+\om_2}^2}\frac{1}{\abs{\om_1}^2+\abs{\om_2}^2+\abs{\om_1+\om_2}^2}\\
				&\quad\times\Bigl((\euler^{t(\abs{\om_1+\om_2}^2-\abs{\om_1}^2-\abs{\om_2}^2)}-1)(\euler^{-2t\abs{\om_1+\om_2}^2}-\euler^{-(t+s)\abs{\om_1+\om_2}^2})\\
				&\qquad\quad+(\euler^{s(\abs{\om_1+\om_2}^2-\abs{\om_1}^2-\abs{\om_2}^2)}-1)(\euler^{-2s\abs{\om_1+\om_2}^2}-\euler^{-(t+s)\abs{\om_1+\om_2}^2})\Bigr)\\
				&=\frac{1}{\abs{\om_1}^2+\abs{\om_2}^2-\abs{\om_1+\om_2}^2}\frac{1}{\abs{\om_1}^2+\abs{\om_2}^2+\abs{\om_1+\om_2}^2}\\
				&\quad\times\Bigl((1-\euler^{-t(\abs{\om_1}^2+\abs{\om_2}^2-\abs{\om_1+\om_2}^2)})(\euler^{-(t+s)\abs{\om_1+\om_2}^2}-\euler^{-2t\abs{\om_1+\om_2}^2})\\
				&\qquad\quad+(1-\euler^{-s(\abs{\om_1}^2+\abs{\om_s}^2-\abs{\om_1+\om_2}^2)})(\euler^{-(t+s)\abs{\om_1+\om_2}^2}-\euler^{-2s\abs{\om_1+\om_2}^2})\Bigr)
			\end{split}
		\end{equation*}
	\end{details}
	\begin{equation}\label{eq:ypsilon_DE_non_orthogonal_bound_2}
		\begin{split}
			&\frac{1}{\abs{\om_1}^2+\abs{\om_2}^2-\abs{\om_1+\om_2}^2}\frac{1}{\abs{\om_1}^2+\abs{\om_2}^2+\abs{\om_1+\om_2}^2}\\
			&\quad\times\Bigl((\euler^{t(\abs{\om_1+\om_2}^2-\abs{\om_1}^2-\abs{\om_2}^2)}-1)(\euler^{-2t\abs{\om_1+\om_2}^2}-\euler^{-(t+s)\abs{\om_1+\om_2}^2})\\
			&\qquad\quad+(\euler^{s(\abs{\om_1+\om_2}^2-\abs{\om_1}^2-\abs{\om_2}^2)}-1)(\euler^{-2s\abs{\om_1+\om_2}^2}-\euler^{-(t+s)\abs{\om_1+\om_2}^2})\Bigr)\\
			&\leq\frac{1}{\abs{\om_1}^2+\abs{\om_2}^2-\abs{\om_1+\om_2}^2}\frac{1}{\abs{\om_1}^2+\abs{\om_2}^2+\abs{\om_1+\om_2}^2}\\
			&\quad\times(1-\euler^{-t(\abs{\om_1}^2+\abs{\om_2}^2-\abs{\om_1+\om_2}^2)})\euler^{-(t+s)\abs{\om_1+\om_2}^2}(1-\euler^{-(t-s)\abs{\om_1+\om_2}^2})\\
			&\lesssim\abs{t-s}^{\gamma}\abs{\om_1+\om_2}^{-2}\frac{\abs{\om_1+\om_2}^{2\gamma}}{\abs{\om_1}^2+\abs{\om_2}^2+\abs{\om_1+\om_2}^2}\frac{t}{t+s}\\
			&\leq\abs{t-s}^{\gamma}\abs{\om_1+\om_2}^{-2}\abs{\om_1}^{-2+2\gamma}.
		\end{split}
	\end{equation}
	By combining~\eqref{eq:ypsilon_DE_non_orthogonal} with~\eqref{eq:ypsilon_DE_non_orthogonal_bound_1} and~\eqref{eq:ypsilon_DE_non_orthogonal_bound_2}, we arrive at
	\begin{equation*}
		\mathsf{DE}_{s,t}\SYpsilon(\om_1,\om_2)\lesssim\abs{t-s}^{\gamma}\abs{\om_1}^{-2}\abs{\om_2}^{-2}(\abs{\om_1+\om_2}^{-2+2\gamma}\abs{\om_1}^{-2}+\abs{\om_1+\om_2}^{-2}\abs{\om_1}^{-2+2\gamma}).
	\end{equation*}
	This yields the claim.
\end{proof}
\begin{details}
	Next we consider the shape coefficient for $\PreThree{10}$ and $\PreThree{20}$. We obtain an explicit expression for $\mathsf{D}_{s,t}\SPreThree$.
	\begin{lemma}\label{lem:difference_three}
		Let $s,t\geq0$ and $\om_1,\om_2,\om_4\in2\uppi\mbZ^2\setminus\{0\}$ such that $\om_1+\om_2\not=0$. Then there exist $\mathsf{DL}_{s,t}\SPreThree(\om_1,\om_2,\om_4)$ and $\mathsf{DE}_{s,t}\SPreThree(\om_1,\om_2,\om_4)$, such that
		\begin{equation*}
			\mathsf{D}_{s,t}\SPreThree(\om_1,\om_2,\om_4)=\mathsf{DL}_{s,t}\SPreThree(\om_1,\om_2,\om_4)+\mathsf{DE}_{s,t}\SPreThree(\om_1,\om_2,\om_4).
		\end{equation*}
		In the case $(\om_1\perp\om_2)$, we obtain
		\begin{equation*}
			\begin{split}
				\mathsf{DL}_{s,t}\SPreThree(\om_1,\om_2,\om_4)&=\frac{1}{2^3}\abs{\om_1}^{-2}\abs{\om_2}^{-2}\abs{\om_1+\om_2}^{-4}\abs{\om_4}^{-2}\\
				&\quad\times\Bigl(1-\euler^{-\abs{t-s}(\abs{\om_1+\om_2}^2+\abs{\om_4}^2)}-\abs{t-s}\abs{\om_1+\om_2}^2\euler^{-\abs{t-s}(\abs{\om_1+\om_2}^2+\abs{\om_4}^2)}\Bigr)
			\end{split}
		\end{equation*}
		and
		\begin{equation*}
			\begin{split}
				&\mathsf{DE}_{s,t}\SPreThree(\om_1,\om_2,\om_4)\\
				&=\frac{1}{2^4}\abs{\om_1}^{-2}\abs{\om_2}^{-2}\abs{\om_1+\om_2}^{-4}\abs{\om_4}^{-2}\Bigl(2t\abs{\om_1+\om_2}^2(\euler^{-(t+s)\abs{\om_1+\om_2}^2}\euler^{-\abs{t-s}\abs{\om_4}^2}-\euler^{-2t\abs{\om_1+\om_2}^2})\\
				&\multiquad[15]+2s\abs{\om_1+\om_2}^2(\euler^{-(t+s)\abs{\om_1+\om_2}^2}\euler^{-\abs{t-s}\abs{\om_4}^2}-\euler^{-2s\abs{\om_1+\om_2}^2})\\
				&\multiquad[15]-\euler^{-2t\abs{\om_1+\om_2}^2}-\euler^{-2s\abs{\om_1+\om_2}^2}+2\euler^{-(t+s)\abs{\om_1+\om_2}^2}\euler^{-\abs{t-s}\abs{\om_4}^2}\Bigr).
			\end{split}
		\end{equation*}
		In the case $\lnot(\om_1\perp\om_2)$, we obtain 
		\begin{equation*}
			\begin{split}
				&\mathsf{DL}_{s,t}\SPreThree(\om_1,\om_2,\om_4)\\
				&=\frac{1}{2^3}\abs{\om_1}^{-2}\abs{\om_2}^{-2}\abs{\om_1+\om_2}^{-2}\abs{\om_4}^{-2}\\
				&\quad\times\Bigl(\frac{1}{\abs{\om_1}^2+\abs{\om_2}^2-\abs{\om_1+\om_2}^2}(\euler^{-\abs{t-s}(\abs{\om_1}^2+\abs{\om_2}^2+\abs{\om_4}^2)}-\euler^{-\abs{t-s}(\abs{\om_1+\om_2}^2+\abs{\om_4}^2)})\\
				&\qquad\quad+\frac{1}{\abs{\om_1}^2+\abs{\om_2}^2+\abs{\om_1+\om_2}^2}(1-\euler^{-\abs{t-s}(\abs{\om_1+\om_2}^2+\abs{\om_4}^2)}+1-\euler^{-\abs{t-s}(\abs{\om_1}^2+\abs{\om_2}^2+\abs{\om_4}^2)})\Bigr)
			\end{split}
		\end{equation*}
		and
		\begin{equation*}
			\begin{split}
				&\mathsf{DE}_{s,t}\SPreThree(\om_1,\om_2,\om_4)\\
				&=\frac{1}{2^3}\abs{\om_1}^{-2}\abs{\om_2}^{-2}\abs{\om_4}^{-2}\\
				&\quad\times\Bigl(\frac{1}{\abs{\om_1}^2+\abs{\om_2}^2-\abs{\om_1+\om_2}^2}\frac{1}{\abs{\om_1}^2+\abs{\om_2}^2+\abs{\om_1+\om_2}^2}\\
				&\qquad\quad\times(2(\euler^{t(\abs{\om_1+\om_2}^2-\abs{\om_1}^2-\abs{\om_2}^2)}-1)(\euler^{-2t\abs{\om_1+\om_2}^2}-\euler^{-(t+s)\abs{\om_1+\om_2}^2}\euler^{-\abs{t-s}\abs{\om_4}^2})\\
				&\multiquad[5]+2(\euler^{s(\abs{\om_1+\om_2}^2-\abs{\om_1}^2-\abs{\om_2}^2)}-1)(\euler^{-2s\abs{\om_1+\om_2}^2}-\euler^{-(t+s)\abs{\om_1+\om_2}^2}\euler^{-\abs{t-s}\abs{\om_4}^2}))\\
				&\qquad\quad+\frac{1}{\abs{\om_1+\om_2}^2}\frac{1}{\abs{\om_1}^2+\abs{\om_2}^2+\abs{\om_1+\om_2}^2}\\
				&\multiquad[11]\times(2\euler^{-(t+s)\abs{\om_1+\om_2}^2}\euler^{-\abs{t-s}\abs{\om_4}^2}-\euler^{-2t\abs{\om_1+\om_2}^2}-\euler^{-2s\abs{\om_1+\om_2}^2})\Bigr).
			\end{split}
		\end{equation*}
	\end{lemma}
	\begin{proof}
		We only discuss the case $\mathsf{DE}_{s,t}\SPreThree(\om_1,\om_2,\om_4)$, $\lnot(\om_1\perp\om_2)$. By definition,
		\begin{equation*}
			\begin{split}
				&\mathsf{DE}_{s,t}\SPreThree(\om_1,\om_2,\om_4)\\
				&=\frac{1}{2^3}\abs{\om_1}^{-2}\abs{\om_2}^{-2}\abs{\om_4}^{-2}\\
				&\quad\times\Bigl(\frac{1}{\abs{\om_1}^2+\abs{\om_2}^2-\abs{\om_1+\om_2}^2}\frac{1}{\abs{\om_1}^2+\abs{\om_2}^2+\abs{\om_1+\om_2}^2}\\
				&\qquad\quad\times (2\euler^{t(-\abs{\om_1+\om_2}^2-\abs{\om_1}^2-\abs{\om_2}^2)}+2\euler^{s(-\abs{\om_1+\om_2}^2-\abs{\om_1}^2-\abs{\om_2}^2)}\\
				&\multiquad[5]-2(\euler^{t(\abs{\om_1+\om_2}^2-\abs{\om_1}^2-\abs{\om_2}^2)}+\euler^{s(\abs{\om_1+\om_2}^2-\abs{\om_1}^2-\abs{\om_2}^2)})\euler^{-(t+s)\abs{\om_1+\om_2}^2}\euler^{-\abs{t-s}\abs{\om_4}^2})\\
				&\qquad\quad-\frac{1}{\abs{\om_1}^2+\abs{\om_2}^2-\abs{\om_1+\om_2}^2}\frac{1}{\abs{\om_1+\om_2}^2}\\
				&\multiquad[13]\times(\euler^{-2t\abs{\om_1+\om_2}^2}+\euler^{-2s\abs{\om_1+\om_2}^2}-2\euler^{-(t+s)\abs{\om_1+\om_2}^2}\euler^{-\abs{t-s}\abs{\om_4}^2})\Bigr).
			\end{split}
		\end{equation*}
		We rewrite
		\begin{equation*}
			\begin{split}
				&\mathsf{DE}_{s,t}\SPreThree(\om_1,\om_2,\om_4)\\
				&=\frac{1}{2^3}\abs{\om_1}^{-2}\abs{\om_2}^{-2}\abs{\om_4}^{-2}\\
				&\quad\times\Bigl(\frac{1}{\abs{\om_1}^2+\abs{\om_2}^2-\abs{\om_1+\om_2}^2}\frac{1}{\abs{\om_1}^2+\abs{\om_2}^2+\abs{\om_1+\om_2}^2}\\
				&\qquad\quad\times (2\euler^{t(-\abs{\om_1+\om_2}^2-\abs{\om_1}^2-\abs{\om_2}^2)}-2\euler^{-2t\abs{\om_1+\om_2}^2}\\
				&\multiquad[5]+2\euler^{s(-\abs{\om_1+\om_2}^2-\abs{\om_1}^2-\abs{\om_2}^2)}-2\euler^{-2s\abs{\om_1+\om_2}^2}\\
				&\multiquad[5]-2(\euler^{t(\abs{\om_1+\om_2}^2-\abs{\om_1}^2-\abs{\om_2}^2)}-1)\euler^{-(t+s)\abs{\om_1+\om_2}^2}\euler^{-\abs{t-s}\abs{\om_4}^2}\\
				&\multiquad[5]-2(\euler^{s(\abs{\om_1+\om_2}^2-\abs{\om_1}^2-\abs{\om_2}^2)}-1)\euler^{-(t+s)\abs{\om_1+\om_2}^2}\euler^{-\abs{t-s}\abs{\om_4}^2})\\
				&\qquad\quad-\frac{1}{\abs{\om_1}^2+\abs{\om_2}^2-\abs{\om_1+\om_2}^2}\Bigl(\frac{1}{\abs{\om_1+\om_2}^2}-\frac{2}{\abs{\om_1}^2+\abs{\om_2}^2+\abs{\om_1+\om_2}^2}\Bigr)\\
				&\multiquad[9]\times (\euler^{-2t\abs{\om_1+\om_2}^2}+\euler^{-2s\abs{\om_1+\om_2}^2}-2\euler^{-(t+s)\abs{\om_1+\om_2}^2}\euler^{-\abs{t-s}\abs{\om_4}^2})\Bigr).
			\end{split}
		\end{equation*}
		We simplify the second term,
		\begin{equation*}
			\begin{split}
				&-\frac{1}{\abs{\om_1}^2+\abs{\om_2}^2-\abs{\om_1+\om_2}^2}\Bigl(\frac{1}{\abs{\om_1+\om_2}^2}-\frac{2}{\abs{\om_1}^2+\abs{\om_2}^2+\abs{\om_1+\om_2}^2}\Bigr)\\
				&\quad\times (\euler^{-2t\abs{\om_1+\om_2}^2}+\euler^{-2s\abs{\om_1+\om_2}^2}-2\euler^{-(t+s)\abs{\om_1+\om_2}^2}\euler^{-\abs{t-s}\abs{\om_4}^2})\\
				&=\frac{1}{\abs{\om_1+\om_2}^2}\frac{1}{\abs{\om_1}^2+\abs{\om_2}^2+\abs{\om_1+\om_2}^2}(2\euler^{-(t+s)\abs{\om_1+\om_2}^2}\euler^{-\abs{t-s}\abs{\om_4}^2}-\euler^{-2t\abs{\om_1+\om_2}^2}-\euler^{-2s\abs{\om_1+\om_2}^2}).
			\end{split}
		\end{equation*}
		We simplify the first term,
		\begin{equation*}
			\begin{split}
				&2\euler^{t(-\abs{\om_1+\om_2}^2-\abs{\om_1}^2-\abs{\om_2}^2)}-2\euler^{-2t\abs{\om_1+\om_2}^2}\\
				&\quad+2\euler^{s(-\abs{\om_1+\om_2}^2-\abs{\om_1}^2-\abs{\om_2}^2)}-2\euler^{-2s\abs{\om_1+\om_2}^2}\\
				&\quad-2(\euler^{t(\abs{\om_1+\om_2}^2-\abs{\om_1}^2-\abs{\om_2}^2)}-1)\euler^{-(t+s)\abs{\om_1+\om_2}^2}\euler^{-\abs{t-s}\abs{\om_4}^2}\\
				&\quad-2(\euler^{s(\abs{\om_1+\om_2}^2-\abs{\om_1}^2-\abs{\om_2}^2)}-1)\euler^{-(t+s)\abs{\om_1+\om_2}^2}\euler^{-\abs{t-s}\abs{\om_4}^2}\\
				&=2(\euler^{t(\abs{\om_1+\om_2}^2-\abs{\om_1}^2-\abs{\om_2}^2)}-1)(\euler^{-2t\abs{\om_1+\om_2}^2}-\euler^{-(t+s)\abs{\om_1+\om_2}^2}\euler^{-\abs{t-s}\abs{\om_4}^2})\\
				&\quad+2(\euler^{s(\abs{\om_1+\om_2}^2-\abs{\om_1}^2-\abs{\om_2}^2)}-1)(\euler^{-2s\abs{\om_1+\om_2}^2}-\euler^{-(t+s)\abs{\om_1+\om_2}^2}\euler^{-\abs{t-s}\abs{\om_4}^2}).
			\end{split}
		\end{equation*}
		This yields the claim.
	\end{proof}
	We can now compute and subsequently bound explicit expressions for $\mathsf{D}_{s,t}\SPreThree$.
	\begin{lemma}\label{lem:difference_three_bound}
		Let $s,t\geq0$, $\gamma\in[0,1)$, $\eps\in(0,1-\gamma)$, $C\geq1$ and $\om_1,\om_2,\om_4\in2\uppi\mbZ^2\setminus\{0\}$ such that $\om_1+\om_2\not=0$ and $C^{-1}\abs{\om_1+\om_2}\leq\abs{\om_4}\leq C\abs{\om_1+\om_2}$.
		\begin{enumerate}
			\item In the case $(\om_1\perp\om_2)$, we obtain
			\begin{equation*}
				\mathsf{D}_{s,t}\SPreThree(\om_1,\om_2,\om_4)\lesssim\abs{t-s}^{\gamma}\abs{\om_1}^{-2}\abs{\om_2}^{-2}\abs{\om_4}^{-6+2\gamma}.
			\end{equation*}
			\item In the case $\lnot(\om_1\perp\om_2)$, we obtain
			\begin{equation*}
				\begin{split}
					\mathsf{D}_{s,t}\SPreThree(\om_1,\om_2,\om_4)&\lesssim\abs{t-s}^{\gamma}\abs{\om_1}^{-2-2\eps}\abs{\om_2}^{-2}\abs{\om_4}^{-6+2\gamma+2\eps}+\abs{t-s}^{\gamma}\abs{\om_1}^{-4}\abs{\om_2}^{-2}\abs{\om_4}^{-4+2\gamma}\\
					&\quad+\abs{t-s}^{\gamma}\abs{\om_1}^{-4+2\gamma}\abs{\om_2}^{-2}\abs{\om_4}^{-4}.
				\end{split}
			\end{equation*}
		\end{enumerate}
		The implied constants are uniform in $\om_1,\om_2,\om_4$.
	\end{lemma}
	\begin{proof}
		We decompose as in Lemma~\ref{lem:difference_three},
		\begin{equation*}
			\mathsf{D}_{s,t}\SPreThree(\om_1,\om_2,\om_4)=\mathsf{DL}_{s,t}\SPreThree(\om_1,\om_2,\om_4)+\mathsf{DE}_{s,t}\SPreThree(\om_1,\om_2,\om_4).
		\end{equation*}
		Assume $(\om_1\perp\om_2)$. We estimate for the leading term, using~\eqref{eq:interpolation} with $\gamma\in[0,1]$ and $(\om_1+\om_2)\sim\om_4$,
		\begin{equation*}
			\mathsf{DL}_{s,t}\SPreThree(\om_1,\om_2,\om_4)\lesssim\abs{t-s}^{\gamma}\abs{\om_1}^{-2}\abs{\om_2}^{-2}\abs{\om_4}^{-6+2\gamma}.
		\end{equation*}
		Next we consider $\mathsf{DE}_{s,t}\SPreThree(\om_1,\om_2,\om_4)$. By symmetry, we may assume $s\leq t$. We estimate
		\begin{equation*}
			\euler^{-(t+s)\abs{\om_1+\om_2}^2}\euler^{-\abs{t-s}\abs{\om_4}^2}-\euler^{-2s\abs{\om_1+\om_2}^2}\leq0,
		\end{equation*}
		and
		\begin{equation*}
			-\euler^{-2t\abs{\om_1+\om_2}^2}-\euler^{-2s\abs{\om_1+\om_2}^2}+2\euler^{-(t+s)\abs{\om_1+\om_2}^2}\euler^{-\abs{t-s}\abs{\om_4}^2}\leq-(\euler^{-t\abs{\om_1+\om_2}^2}-\euler^{-s\abs{\om_1+\om_2}^2})^2\leq0.
		\end{equation*}
		We bound the remaining term,
		\begin{equation*}
			\euler^{-(t+s)\abs{\om_1+\om_2}^2}\euler^{-\abs{t-s}\abs{\om_4}^2}-\euler^{-2t\abs{\om_1+\om_2}^2}\leq \euler^{-(t+s)\abs{\om_1+\om_2}^2}(1-\euler^{-\abs{t-s}\abs{\om_1+\om_2}^2}).
		\end{equation*}
		By~\eqref{eq:interpolation}, \eqref{eq:rapid_decay} and using that $(\om_1+\om_2)\sim\om_4$,
		\begin{equation*}
			2t\abs{\om_1+\om_2}^2\euler^{-(t+s)\abs{\om_1+\om_2}^2}(1-\euler^{-\abs{t-s}\abs{\om_1+\om_2}^2})\lesssim\abs{t-s}^{\gamma}\abs{\om_4}^{2\gamma}.
		\end{equation*}
		This yields the bound on $\mathsf{DE}_{s,t}\SPreThree(\om_1,\om_2,\om_4)$.
		
		Assume $\lnot(\om_1\perp\om_2)$. We first consider $\mathsf{DL}_{s,t}\SPreThree(\om_1,\om_2,\om_4)$. We can bound the non-positive correction term,
		\begin{equation*}
			\frac{1}{\abs{\om_1}^2+\abs{\om_2}^2-\abs{\om_1+\om_2}^2}(\euler^{-\abs{t-s}(\abs{\om_1}^2+\abs{\om_2}^2+\abs{\om_4}^2)}-\euler^{-\abs{t-s}(\abs{\om_1+\om_2}^2+\abs{\om_4}^2)})\leq0.
		\end{equation*}
		We obtain by~\eqref{eq:interpolation} for $\gamma\in[0,1]$, and by using that $(\om_1+\om_2)\sim\om_4$,
		\begin{equation*}
			\mathsf{DL}_{s,t}\SPreThree(\om_1,\om_2,\om_4)\lesssim\abs{t-s}^{\gamma}\abs{\om_1}^{-2}\abs{\om_2}^{-2}\abs{\om_4}^{-4}(\abs{\om_1}^2+\abs{\om_2}^2+\abs{\om_4}^2)^{-1+\gamma}.
		\end{equation*}
		Let $\gamma\in(0,1)$ and $\eps\in(0,1-\gamma)$, we obtain
		\begin{equation*}
			(\abs{\om_1}^2+\abs{\om_2}^2+\abs{\om_4}^2)^{-1+\gamma}\leq\abs{\om_4}^{-2+2\gamma+2\eps}\abs{\om_1}^{-2\eps}.
		\end{equation*}
		Consequently,
		\begin{equation*}
			\mathsf{DL}_{s,t}\SPreThree(\om_1,\om_2,\om_4)\lesssim\abs{t-s}^{\gamma}\abs{\om_1}^{-2-2\eps}\abs{\om_2}^{-2}\abs{\om_4}^{-6+2\gamma+2\eps}.
		\end{equation*}
		Next we consider $\mathsf{DE}_{s,t}\SPreThree(\om_1,\om_2,\om_4)$. We estimate the second summand,
		\begin{equation*}
			2\euler^{-(t+s)\abs{\om_1+\om_2}^2}\euler^{-\abs{t-s}\abs{\om_4}^2}-\euler^{-2t\abs{\om_1+\om_2}^2}-\euler^{-2s\abs{\om_1+\om_2}^2}\leq-(\euler^{-t\abs{\om_1+\om_2}^2}-\euler^{-s\abs{\om_1+\om_2}^2})^2\leq0.
		\end{equation*}
		For the first summand, assume $\abs{\om_1}^2+\abs{\om_2}^2<\abs{\om_1+\om_2}^2$. We write
		\begin{equation*}
			\begin{split}
				&\frac{1}{\abs{\om_1}^2+\abs{\om_2}^2-\abs{\om_1+\om_2}^2}\frac{1}{\abs{\om_1}^2+\abs{\om_2}^2+\abs{\om_1+\om_2}^2}\\
				&\quad\times\Bigl((\euler^{t(\abs{\om_1+\om_2}^2-\abs{\om_1}^2-\abs{\om_2}^2)}-1)(\euler^{-2t\abs{\om_1+\om_2}^2}-\euler^{-(t+s)\abs{\om_1+\om_2}^2}\euler^{-\abs{t-s}\abs{\om_4}^2})\\
				&\qquad\quad+(\euler^{s(\abs{\om_1+\om_2}^2-\abs{\om_1}^2-\abs{\om_2}^2)}-1)(\euler^{-2s\abs{\om_1+\om_2}^2}-\euler^{-(t+s)\abs{\om_1+\om_2}^2}\euler^{-\abs{t-s}\abs{\om_4}^2})\Bigr)\\
				&=\frac{1}{\abs{\om_1+\om_2}^2-\abs{\om_1}^2-\abs{\om_2}^2}\frac{1}{\abs{\om_1}^2+\abs{\om_2}^2+\abs{\om_1+\om_2}^2}\\
				&\quad\times\Bigl((\euler^{-t(\abs{\om_1}^2+\abs{\om_2}^2)}-\euler^{-t\abs{\om_1+\om_2}^2})\euler^{-s\abs{\om_1+\om_2}^2}(\euler^{-\abs{t-s}\abs{\om_4}^2}-\euler^{-(t-s)\abs{\om_1+\om_2}^2})\\
				&\qquad\quad+(\euler^{-s(\abs{\om_1}^2+\abs{\om_2}^2)}-\euler^{-s\abs{\om_1+\om_2}^2})\euler^{-t\abs{\om_1+\om_2}^2}(\euler^{-\abs{t-s}\abs{\om_4}^2}-\euler^{-(s-t)\abs{\om_1+\om_2}^2})\Bigr).
			\end{split}
		\end{equation*}
		By symmetry, we may assume $s\leq t$, so that
		\begin{equation*}
			(\euler^{-s(\abs{\om_1}^2+\abs{\om_2}^2)}-\euler^{-s\abs{\om_1+\om_2}^2})(\euler^{-\abs{t-s}\abs{\om_4}^2}-\euler^{-(s-t)\abs{\om_1+\om_2}^2})\leq0.
		\end{equation*}
		Furthermore,
		\begin{equation*}
			\begin{split}
				&(\euler^{-t(\abs{\om_1}^2+\abs{\om_2}^2)}-\euler^{-t\abs{\om_1+\om_2}^2})\euler^{-s\abs{\om_1+\om_2}^2}(\euler^{-\abs{t-s}\abs{\om_4}^2}-\euler^{-(t-s)\abs{\om_1+\om_2}^2})\\
				&\leq\euler^{-t(\abs{\om_1}^2+\abs{\om_2}^2)}(1-\euler^{-t(\abs{\om_1+\om_2}^2-\abs{\om_1}^2-\abs{\om_2}^2)})\euler^{-s(\abs{\om_1}^2+\abs{\om_2}^2)}(1-\euler^{-\abs{t-s}\abs{\om_1+\om_2}^2}).
			\end{split}
		\end{equation*}
		Hence we can bound by~\eqref{eq:interpolation} and~\eqref{eq:rapid_decay},
		\begin{equation*}
			\begin{split}
				&\frac{1}{\abs{\om_1+\om_2}^2-\abs{\om_1}^2-\abs{\om_2}^2}\frac{1}{\abs{\om_1}^2+\abs{\om_2}^2+\abs{\om_1+\om_2}^2}\\
				&\quad\times\Bigl((\euler^{-t(\abs{\om_1}^2+\abs{\om_2}^2)}-\euler^{-t\abs{\om_1+\om_2}^2})\euler^{-s\abs{\om_1+\om_2}^2}(\euler^{-\abs{t-s}\abs{\om_4}^2}-\euler^{-(t-s)\abs{\om_1+\om_2}^2})\\
				&\qquad\quad+(\euler^{-s(\abs{\om_1}^2+\abs{\om_2}^2)}-\euler^{-s\abs{\om_1+\om_2}^2})\euler^{-t\abs{\om_1+\om_2}^2}(\euler^{-\abs{t-s}\abs{\om_4}^2}-\euler^{-(s-t)\abs{\om_1+\om_2}^2})\Bigr)\\
				&\lesssim\frac{\abs{\om_1+\om_2}^{2\gamma}}{\abs{\om_1}^2+\abs{\om_2}^2+\abs{\om_1+\om_2}^2}\frac{t}{t+s}\frac{1}{\abs{\om_1}^2+\abs{\om_2}^2}\abs{t-s}^{\gamma}\\
				&\leq\abs{t-s}^{\gamma}\abs{\om_1}^{-2}\abs{\om_1+\om_2}^{-2+2\gamma}.
			\end{split}
		\end{equation*}
		Now assume $\abs{\om_1+\om_2}^2<\abs{\om_1}^{2}+\abs{\om_2}^{2}$, $s\leq t$. Then
		\begin{equation*}
			\begin{split}
				&\frac{1}{\abs{\om_1}^2+\abs{\om_2}^2-\abs{\om_1+\om_2}^2}\frac{1}{\abs{\om_1}^2+\abs{\om_2}^2+\abs{\om_1+\om_2}^2}\\
				&\quad\times\Bigl((\euler^{t(\abs{\om_1+\om_2}^2-\abs{\om_1}^2-\abs{\om_2}^2)}-1)(\euler^{-2t\abs{\om_1+\om_2}^2}-\euler^{-(t+s)\abs{\om_1+\om_2}^2}\euler^{-\abs{t-s}\abs{\om_4}^2})\\
				&\qquad\quad+(\euler^{s(\abs{\om_1+\om_2}^2-\abs{\om_1}^2-\abs{\om_2}^2)}-1)(\euler^{-2s\abs{\om_1+\om_2}^2}-\euler^{-(t+s)\abs{\om_1+\om_2}^2}\euler^{-\abs{t-s}\abs{\om_4}^2})\Bigr)\\
				&=\frac{1}{\abs{\om_1}^2+\abs{\om_2}^2-\abs{\om_1+\om_2}^2}\frac{1}{\abs{\om_1}^2+\abs{\om_2}^2+\abs{\om_1+\om_2}^2}\\
				&\quad\times\Bigl((1-\euler^{-t(\abs{\om_1}^2+\abs{\om_2}^2-\abs{\om_1+\om_2}^2)})(\euler^{-(t+s)\abs{\om_1+\om_2}^2}\euler^{-\abs{t-s}\abs{\om_4}^2}-\euler^{-2t\abs{\om_1+\om_2}^2})\\
				&\qquad\quad+(1-\euler^{-s(\abs{\om_1}^2+\abs{\om_2}^2-\abs{\om_1+\om_2}^2)})(\euler^{-(t+s)\abs{\om_1+\om_2}^2}\euler^{-\abs{t-s}\abs{\om_4}^2}-\euler^{-2s\abs{\om_1+\om_2}^2})\Bigr)\\
				&\leq\frac{1}{\abs{\om_1}^2+\abs{\om_2}^2-\abs{\om_1+\om_2}^2}\frac{1}{\abs{\om_1}^2+\abs{\om_2}^2+\abs{\om_1+\om_2}^2}\\
				&\quad\times(1-\euler^{-t(\abs{\om_1}^2+\abs{\om_2}^2-\abs{\om_1+\om_2}^2)})\euler^{-(t+s)\abs{\om_1+\om_2}^2}(1-\euler^{-\abs{t-s}\abs{\om_1+\om_2}^2})\\
				&\lesssim\frac{t}{t+s}\abs{\om_1+\om_2}^{-2}\frac{\abs{\om_1+\om_2}^{2\gamma}}{\abs{\om_1}^2+\abs{\om_2}^2+\abs{\om_1+\om_2}^2}\abs{t-s}^{\gamma}\\
				&\leq\abs{t-s}^{\gamma}\abs{\om_1+\om_2}^{-2}\abs{\om_1}^{-2+2\gamma}.
			\end{split}
		\end{equation*}
		This yields the claim.
	\end{proof}
\end{details}
Next we bound the shape coefficient $\mathsf{A}\STriangle$, which appears in $\PreCocktail{50}$, $\PreCocktail{30}$, $\PreCocktail{40}$ and $\Triangle{3}$ (cf.\ Definition~\ref{def:shape_coefficient_triangle}).
\begin{lemma}\label{lem:triangle_regularity}
	Let $s,t\geq0$, $k=1,2$, $\gamma\in[0,1]$ and $C\geq1$. Then uniformly in $\om_1,\om_2,\om_3\in\mbZ^2\setminus\{0\}$ such that $C^{-1}\abs{\om_1}\leq\abs{\om_2}\leq C\abs{\om_1}$, it holds that
	\begin{equation*}
		\mathsf{A}_{s,t}^{k}\STriangle(\om_1,\om_2,\om_3)\lesssim\abs{t-s}^{\gamma}\abs{\om_2}^{2\gamma}\abs{\om_3}^{-1}.
	\end{equation*}
\end{lemma}
\begin{proof}
	The claim follows by the triangle inequality and repeated applications of~\eqref{eq:interpolation}.
\end{proof}
\begin{details}
	\textbf{Detailed Proof.}
	\begin{proof}
		We estimate by the triangle inequality,
		\begin{equation}\label{eq:triangle_regularity_terms}
			\begin{split}
				&\abs{(H^{k}_{t-u_1}(\om_1)H^{j}_{t-u_2}(\om_2)-H^{k}_{s-u_1}(\om_1)H^{j}_{s-u_2}(\om_2))H^{j}_{u_1-u_2}(\om_3)}\\
				&\leq\abs{H^{k}_{t-u_1}(\om_1)(H^{j}_{t-u_2}(\om_2)-H^{j}_{s-u_2}(\om_2))H^{j}_{u_1-u_2}(\om_3)}\\
				&\quad+\abs{(H^{k}_{t-u_1}(\om_1)-H^{k}_{s-u_1}(\om_1))H^{j}_{s-u_2}(\om_2)H^{j}_{u_1-u_2}(\om_3)}.
			\end{split}
		\end{equation}
		We first bound the integrals over $u_2$. We obtain for the first summand of~\eqref{eq:triangle_regularity_terms},
		\begin{equation*}
			\begin{split}
				&\sum_{j=1}^{2}\int_{-\infty}^{u_1}\dd u_2\abs{(H^{j}_{t-u_2}(\om_2)-H^{j}_{s-u_2}(\om_2))H^{j}_{u_1-u_2}(\om_3)}\\
				&\lesssim\abs{\om_2}\abs{\om_3}\int_{-\infty}^{u_1\wedge s}\dd u_2(\euler^{-\abs{s-u_2}\abs{2\uppi\om_2}^2}-\euler^{-\abs{t-u_2}\abs{2\uppi\om_2}^2})\euler^{-\abs{u_1-u_2}\abs{2\uppi\om_3}^2}\\
				&\quad+\abs{\om_2}\abs{\om_3}\int_{s}^{u_1}\dd u_2\euler^{-\abs{t-u_2}\abs{2\uppi\om_2}^2}\euler^{-\abs{u_1-u_2}\abs{2\uppi\om_3}^2}.
			\end{split}
		\end{equation*}
		It suffices to control each term. By~\eqref{eq:interpolation} for $\gamma\in[0,1]$,
		\begin{equation*}
			\begin{split}
				&\int_{-\infty}^{t}\dd u_1\int_{-\infty}^{u_1\wedge s}\dd u_2 \euler^{-\abs{t-u_1}\abs{2\uppi\om_1}^2}(\euler^{-\abs{s-u_2}\abs{2\uppi\om_2}^2}-\euler^{-\abs{t-u_2}\abs{2\uppi\om_2}^2})\euler^{-\abs{u_1-u_2}\abs{2\uppi\om_3}^2}\\
				&\lesssim\abs{t-s}^{\gamma}\abs{\om_2}^{2\gamma}\int_{-\infty}^{t}\dd u_1\euler^{-\abs{t-u_1}\abs{2\uppi\om_1}^2}\int_{-\infty}^{u_1\wedge s}\dd u_2 \euler^{-\abs{s-u_2}\abs{2\uppi\om_2}^2}\euler^{-\abs{u_1-u_2}\abs{2\uppi\om_3}^2}\\
				&\lesssim\abs{t-s}^{\gamma}\abs{\om_2}^{2\gamma}(\abs{\om_2}^2+\abs{\om_3}^2)^{-1}\int_{-\infty}^{t}\dd u_1\euler^{-\abs{t-u_1}\abs{2\uppi\om_1}^2}\\
				&\lesssim\abs{t-s}^{\gamma}\abs{\om_2}^{2\gamma}(\abs{\om_2}^2+\abs{\om_3}^2)^{-1}\abs{\om_1}^{-2}
			\end{split}
		\end{equation*}
		and
		\begin{equation*}
			\begin{split}
				&\int_{s}^{t}\dd u_1\int_{s}^{u_1}\dd u_2 	\euler^{-\abs{t-u_1}\abs{2\uppi\om_1}^2}\euler^{-\abs{t-u_2}\abs{2\uppi\om_2}^2}\euler^{-\abs{u_1-u_2}\abs{2\uppi\om_3}^2}\\
				&\lesssim(\abs{\om_2}^2+\abs{\om_3}^2)^{-1}\int_{s}^{t}\dd u_1\euler^{-\abs{t-u_1}\abs{2\uppi\om_1}^2}\\
				&\lesssim\abs{t-s}^{\gamma}(\abs{\om_2}^2+\abs{\om_3}^2)^{-1}\abs{\om_1}^{2\gamma-2}.
			\end{split}
		\end{equation*}
		For the second term of~\eqref{eq:triangle_regularity_terms},
		\begin{equation*}
			\sum_{j=1}^{2}\int_{-\infty}^{u_1\wedge s}\dd u_2\abs{H^{j}_{s-u_2}(\om_2)H^{j}_{u_1-u_2}(\om_3)}\lesssim\abs{\om_2}\abs{\om_3}(\abs{\om_2}^{2}+\abs{\om_3}^{2})^{-1}.
		\end{equation*}
		By~\eqref{eq:interpolation},
		\begin{equation*}
			\begin{split}
				&\sum_{j=1}^{2}\int_{-\infty}^{t}\dd u_1\int_{-\infty}^{u_1\wedge s}\dd u_2\abs{(H^k_{t-u_1}(\om_1)-H^k_{s-u_1}(\om_1))H^{j}_{s-u_2}(\om_2)H^{j}_{u_1-u_2}(\om_3)}\\
				&\lesssim\abs{\om_2}\abs{\om_3}(\abs{\om_2}^{2}+\abs{\om_3}^{2})^{-1}\int_{-\infty}^{t}\dd u_1\abs{H^k_{t-u_1}(\om_1)-H^k_{s-u_1}(\om_1)}\\
				&\lesssim\abs{t-s}^{\gamma}\abs{\om_2}\abs{\om_3}(\abs{\om_2}^{2}+\abs{\om_3}^{2})^{-1}\abs{\om_1}^{-1+2\gamma}.
			\end{split}
		\end{equation*}
		Using that $C^{-1}\abs{\om_1}\leq\abs{\om_2}\leq C\abs{\om_1}$, we can simplify our estimates to
		\begin{equation*}
			\begin{split}
				&\sum_{j=1}^{2}\int_{-\infty}^{t}\dd u_1\int_{-\infty}^{u_1}\dd u_2\abs{(H^k_{t-u_1}(\om_1)H^{j}_{t-u_2}(\om_2)-H^k_{s-u_1}(\om_1)H^{j}_{s-u_2}(\om_2))H^{j}_{u_1-u_2}(\om_3)}\\
				&\lesssim\abs{t-s}^{\gamma}\abs{\om_2}^{2\gamma}\abs{\om_3}^{-1}.
			\end{split}
		\end{equation*}
		This yields the claim.
	\end{proof}
\end{details}
The following lemma controls the shape coefficient $\mathsf{A}\SVee$, which appears in $\Checkmark{10}$, $\Checkmark{20}$ and $\PreThreeloop{10}$, $\PreThreeloop{20}$ (cf.\ Definition~\ref{def:shape_coefficient_vee}).
\begin{lemma}\label{lem:v_regularity}
	Let $s,t\geq0$, $k,k'=1,2$, $\gamma\in[0,1]$ and $C,C'\geq1$. Then uniformly in $\om_1,\om_1',\om_2\in\mbZ^2\setminus\{0\}$ such that $C^{-1}\abs{\om_1}\leq\abs{\om_2}\leq C\abs{\om_1}$ and $(C')^{-1}\abs{\om_1'}\leq\abs{\om_2}\leq C'\abs{\om_1'}$, it holds that
	\begin{equation*}
		\mathsf{A}^{k,k'}_{s,t}\SVee(\om_1,\om_1',\om_2)\lesssim\abs{t-s}^{\gamma}\abs{\om_1}^{-1+\gamma}\abs{\om_1'}^{-1+\gamma}.
	\end{equation*}
\end{lemma}
\begin{proof}
	The claim follows by the triangle inequality and repeated applications of~\eqref{eq:interpolation}.
\end{proof}
\begin{details}
	\textbf{Detailed Proof.}
	\begin{proof}
		We first apply the triangle inequality to estimate
		\begin{equation*}
			\begin{split}
				&\abs{H_{t-u_1}^{k}(\om_1)H_{t-u_2}^{j}(\om_2)-H_{s-u_1}^{k}(\om_1)H_{s-u_2}^{j}(\om_2)}\abs{H_{t-u_1'}^{k'}(\om_1')H_{t-u_2}^{j}(\om_2)-H_{s-u_1'}^{k'}(\om_1')H_{s-u_2}^{j}(\om_2)}\\
				&\leq\Bigl(\abs{H_{t-u_1}^{k}(\om_1)-H_{s-u_1}^{k}(\om_1)}\abs{H_{t-u_2}^{j}(\om_2)}+\abs{H_{s-u_1}^{k}(\om_1)}\abs{H_{t-u_2}^{j}(\om_2)-H_{s-u_2}^{j}(\om_2)}\Bigr)\\
				&\quad\times\Bigl(\abs{H_{t-u_1'}^{k'}(\om_1')-H_{s-u_1'}^{k'}(\om_1')}\abs{H_{t-u_2}^{j}(\om_2)}+\abs{H_{s-u_1'}^{k'}(\om_1')}\abs{H_{t-u_2}^{j}(\om_2)-H_{s-u_2}^{j}(\om_2)}\Bigr),
			\end{split}
		\end{equation*}
		which multiplies out into $4$ summands.
		\begin{equation}\label{eq:v_regularity_decomposition}
			\begin{split}
				&\abs{H_{t-u_1}^{k}(\om_1)H_{t-u_2}^{j}(\om_2)-H_{s-u_1}^{k}(\om_1)H_{s-u_2}^{j}(\om_2)}\abs{H_{t-u_1'}^{k'}(\om_1')H_{t-u_2}^{j}(\om_2)-H_{s-u_1'}^{k'}(\om_1')H_{s-u_2}^{j}(\om_2)}\\
				&\leq\abs{H_{t-u_1}^{k}(\om_1)-H_{s-u_1}^{k}(\om_1)}\abs{H_{t-u_2}^{j}(\om_2)}\abs{H_{t-u_1'}^{k'}(\om_1')-H_{s-u_1'}^{k'}(\om_1')}\abs{H_{t-u_2}^{j}(\om_2)}\\
				&\quad+\abs{H_{t-u_1}^{k}(\om_1)-H_{s-u_1}^{k}(\om_1)}\abs{H_{t-u_2}^{j}(\om_2)}\abs{H_{s-u_1'}^{k'}(\om_1')}\abs{H_{t-u_2}^{j}(\om_2)-H_{s-u_2}^{j}(\om_2)}\\
				&\quad+\abs{H_{s-u_1}^{k}(\om_1)}\abs{H_{t-u_2}^{j}(\om_2)-H_{s-u_2}^{j}(\om_2)}\abs{H_{t-u_1'}^{k'}(\om_1')-H_{s-u_1'}^{k'}(\om_1')}\abs{H_{t-u_2}^{j}(\om_2)}\\
				&\quad+\abs{H_{s-u_1}^{k}(\om_1)}\abs{H_{t-u_2}^{j}(\om_2)-H_{s-u_2}^{j}(\om_2)}\abs{H_{s-u_1'}^{k'}(\om_1')}\abs{H_{t-u_2}^{j}(\om_2)-H_{s-u_2}^{j}(\om_2)}.
			\end{split}
		\end{equation}
		Let $\gamma\in[0,1]$. We start with the first summand of~\eqref{eq:v_regularity_decomposition}. We apply~\eqref{eq:interpolation},
		\begin{equation*}
			\begin{split}
				&\sum_{j=1}^{2}\int_{-\infty}^{t}\dd u_2\int_{-\infty}^{t}\dd u_1\int_{-\infty}^{t}\dd u_1'\abs{H_{t-u_1}^{k}(\om_1)-H_{s-u_1}^{k}(\om_1)}\abs{H_{t-u_1'}^{k'}(\om_1')-H_{s-u_1'}^{k'}(\om_1')}\abs{H_{t-u_2}^{j}(\om_2)}^2\\
				&\lesssim\int_{-\infty}^{t}\dd u_1\int_{-\infty}^{t}\dd u_1'\abs{H_{t-u_1}^{k}(\om_1)-H_{s-u_1}^{k}(\om_1)}\abs{H_{t-u_1'}^{k'}(\om_1')-H_{s-u_1'}^{k'}(\om_1')}\\
				&\lesssim\abs{t-s}^{2\gamma}\abs{\om_1}^{-1+2\gamma}\abs{\om_1'}^{-1+2\gamma}.
			\end{split}
		\end{equation*}
		For the second summand of~\eqref{eq:v_regularity_decomposition}, we first compute the integral over $u_2$. We obtain
		\begin{equation}\label{eq:v_regularity_decomposition_second_summand_integral_u2}
			\begin{split}
				&\sum_{j=1}^{2}\int_{-\infty}^{t}\dd u_2\abs{H_{t-u_2}^{j}(\om_2)}\abs{H_{t-u_2}^{j}(\om_2)-H_{s-u_2}^{j}(\om_2)}\\
				&\lesssim\abs{\om_2}^{2}\int_{-\infty}^{s}\dd u_2\euler^{-\abs{t-u_2}\abs{2\uppi\om_2}^2}\euler^{-\abs{s-u_2}\abs{2\uppi\om_2}^2}(1-\euler^{-\abs{t-s}\abs{2\uppi\om_2}^2})+\abs{\om_2}^2\int_{s}^{t}\dd u_2\euler^{-2\abs{t-u_2}\abs{2\uppi\om_2}^2}\\
				&\lesssim\abs{t-s}^{\gamma}\abs{\om_2}^{2\gamma},
			\end{split}
		\end{equation}
		so,
		\begin{equation*}
			\begin{split}
				&\sum_{j=1}^{2}\int_{-\infty}^{t}\dd u_2\int_{-\infty}^{t}\dd u_1\int_{-\infty}^{t}\dd u_1'\\
				&\quad\abs{H_{t-u_1}^{k}(\om_1)-H_{s-u_1}^{k}(\om_1)}\abs{H_{t-u_2}^{j}(\om_2)}\abs{H_{s-u_1'}^{k'}(\om_1')}\abs{H_{t-u_2}^{j}(\om_2)-H_{s-u_2}^{j}(\om_2)}\\
				&\lesssim\abs{t-s}^{\gamma}\abs{\om_2}^{2\gamma}\int_{-\infty}^{t}\dd u_1\int_{-\infty}^{t}\dd u_1'\abs{H_{t-u_1}^{k}(\om_1)-H_{s-u_1}^{k}(\om_1)}\abs{H_{s-u_1'}^{k'}(\om_1')}\\
				&\lesssim\abs{t-s}^{2\gamma}\abs{\om_2}^{2\gamma}\abs{\om_1}^{-1+2\gamma}\abs{\om_1'}^{-1}.
			\end{split}
		\end{equation*}
		For the third summand of~\eqref{eq:v_regularity_decomposition} we apply~\eqref{eq:v_regularity_decomposition_second_summand_integral_u2} and~\eqref{eq:interpolation},
		\begin{equation*}
			\begin{split}
				&\sum_{j=1}^{2}\int_{-\infty}^{t}\dd u_2\int_{-\infty}^{t}\dd u_1\int_{-\infty}^{t}\dd u_1'\\
				&\quad\abs{H_{s-u_1}^{k}(\om_1)}\abs{H_{t-u_2}^{j}(\om_2)-H_{s-u_2}^{j}(\om_2)}\abs{H_{t-u_1'}^{k'}(\om_1')-H_{s-u_1'}^{k'}(\om_1')}\abs{H_{t-u_2}^{j}(\om_2)}\\
				&\lesssim\abs{t-s}^{\gamma}\abs{\om_2}^{2\gamma}\int_{-\infty}^{t}\dd u_1\int_{-\infty}^{t}\dd u_1'\abs{H_{s-u_1}^{k}(\om_1)}\abs{H_{t-u_1'}^{k'}(\om_1')-H_{s-u_1'}^{k'}(\om_1')}\\
				&\lesssim\abs{t-s}^{2\gamma}\abs{\om_2}^{2\gamma}\abs{\om_1}^{-1}\abs{\om_1'}^{-1+2\gamma}.
			\end{split}
		\end{equation*}
		For the fourth summand of~\eqref{eq:v_regularity_decomposition},
		\begin{equation*}
			\begin{split}
				&\sum_{j=1}^{2}\int_{-\infty}^{t}\dd u_2\abs{H_{t-u_2}^{j}(\om_2)-H_{s-u_2}^{j}(\om_2)}^2\\
				&\lesssim\abs{\om_2}^{2}\int_{-\infty}^{s}\dd u_2\euler^{-2\abs{s-u_2}\abs{2\uppi\om_2}^2}(1-\euler^{-\abs{t-s}\abs{2\uppi\om_2}^2})^2+\abs{\om_2}^{2}\int_{s}^{t}\dd u_2\euler^{-2\abs{t-u_2}\abs{2\uppi\om_2}^2}\\
				&\lesssim\abs{t-s}^{\gamma}\abs{\om_2}^{2\gamma},
			\end{split}
		\end{equation*}
		so,
		\begin{equation*}
			\begin{split}
				&\sum_{j=1}^{2}\int_{-\infty}^{t}\dd u_2\int_{-\infty}^{t}\dd u_1\int_{-\infty}^{t}\dd u_1'\\
				&\quad\abs{H_{s-u_1}^{k}(\om_1)}\abs{H_{t-u_2}^{j}(\om_2)-H_{s-u_2}^{j}(\om_2)}\abs{H_{s-u_1'}^{k'}(\om_1')}\abs{H_{t-u_2}^{j}(\om_2)-H_{s-u_2}^{j}(\om_2)}\\
				&\lesssim\abs{t-s}^{\gamma}\abs{\om_2}^{2\gamma}\int_{-\infty}^{t}\dd u_1\int_{-\infty}^{t}\dd u_1'\abs{H_{s-u_1}^{k}(\om_1)}\abs{H_{s-u_1'}^{k'}(\om_1')}\\
				&\lesssim\abs{t-s}^{\gamma}\abs{\om_2}^{2\gamma}\abs{\om_1}^{-1}\abs{\om_1'}^{-1}.
			\end{split}
		\end{equation*}
		By implementing the conditions $C^{-1}\abs{\om_1}\leq\abs{\om_2}\leq C\abs{\om_1}$ and $(C')^{-1}\abs{\om_1'}\leq\abs{\om_2}\leq C'\abs{\om_1'}$, the estimates above simplify to the claimed expression.
	\end{proof}
\end{details}
\section{Summation Estimates}\label{app:summation_estimates}
\subsection{Basic Estimates}
We prove a number of summation and discrete convolution estimates that are central to our bounds.

Recall that the Fourier multiplier of the elliptic operator $\partial_{j}\Phi$, denoted by $G^{j}(\om)=2\uppi\upi\om^{j}\abs{2\uppi\upi\om}^{-2}\mathds{1}_{\om\neq0}$ for all $\om\in\mbZ^{2}$ and $j=1,2$, frequently appears in the diagrams considered in Subsection~\ref{subsec:Feynman}. The following lemma shows that $\abs{G^{j}(\om+\om_1)-G^{j}(\om_1)}$ decays like $\abs{\om+\om_{1}}^{-2}$ for every $\om_{1}\in\mbZ^{2}\setminus\{0,-\om\}$, which is one order better than $\abs{G^{j}(\om+\om_1)}\lesssim\abs{\om+\om_{1}}^{-1}$.
\begin{lemma}\label{lem:elliptic_difference}
	Let $j=1,2$, then uniformly in $\om,\om_1\in\mbZ^2$ such that $\om_1,\om+\om_1\neq0$, it holds that
	\begin{equation*}
		\abs{G^{j}(\om+\om_1)-G^{j}(\om_1)}\lesssim\abs{\om}\abs{\om+\om_1}^{-2}(1+\abs{\om}\abs{\om_1}^{-1}).
	\end{equation*}
\end{lemma}
\begin{proof}
	We compute
	\begin{equation*}
		G^{j}(\om+\om_1)-G^{j}(\om_1)=\frac{\upi}{2\uppi}\Bigl(\frac{\om^j+\om_1^j}{\abs{\om+\om_1}^2}-\frac{\om_1^j}{\abs{\om_1}^2}\Bigr),
	\end{equation*}
	which can be bounded in absolute value by
	\begin{equation*}
		\abs*{\frac{\om^j+\om_1^j}{\abs{\om+\om_1}^2}-\frac{\om_1^j}{\abs{\om_1}^2}}=\frac{\abs{\om^j\abs{\om_1}^2-\om_1^j\abs{\om}^2-\om_1^j2\inner{\om}{\om_1}}}{\abs{\om+\om_1}^2\abs{\om_1}^2}\lesssim \abs{\om}\abs{\om+\om_1}^{-2}+\abs{\om}^2\abs{\om_1}^{-1}\abs{\om+\om_1}^{-2}.
	\end{equation*}
	This yields the claim.
\end{proof}
We apply the following summation estimates repeatedly to establish the regularities of our diagrams.
\begin{lemma}\label{lem:summation_estimates}
	It holds that uniformly in $\delta>0$,
	\begin{equation}\label{eq:summation_estimates_0}
		\sum_{\substack{k\in\mbZ^2\\\abs{k}\leq\delta^{-1}}}1\lesssim(1\vee\delta^{-2})
	\end{equation}
	and
	\begin{equation}\label{eq:summation_estimates_-2}
	 	\sum_{\substack{k\in\mbZ^2\setminus\{0\}\\\abs{k}\leq\delta^{-1}}}\abs{k}^{-2}\lesssim(1\vee\log(\delta^{-1})).
	\end{equation}
	What is more, for any $\alpha>2$,
	\begin{equation}\label{eq:summation_estimates_-alpha}
		\sum_{k\in\mbZ^2\setminus\{0\}}\abs{k}^{-\alpha}\lesssim_{\alpha}1.
	\end{equation}
\end{lemma}
\begin{proof}
	The proof follows by treating the left-hand sides in~\eqref{eq:summation_estimates_0}--\eqref{eq:summation_estimates_-alpha} as Riemann sums and by bounding the associated Riemann integrals.
\end{proof}
\begin{details}
	\begin{proof}
		Let $\delta>0$. To establish the inequality~\eqref{eq:summation_estimates_0}, we estimate
		\begin{equation*}
			\sum_{\substack{k\in\mbZ^{2}\\\abs{k}\leq\delta^{-1}}}1=\sum_{\substack{k\in\mbZ^{2}\\\abs{k}\leq\delta^{-1}}}\int_{\{x:\abs{x-k}_{l^{\infty}}\leq1/2\}}1\dd x\leq\int_{\{x:0\leq\abs{x}\leq\delta^{-1}+\sqrt{2}/2\}}1\dd x\lesssim\Bigl(\delta^{-1}+\frac{\sqrt{2}}{2}\Bigr)^{2},
		\end{equation*}
		where we used that
		\begin{equation*}
			\bigcup_{\substack{k\in\mbZ^{2}\\\abs{k}\leq\delta^{-1}}}\{x:\abs{x-k}_{\ell^{\infty}}\leq1/2\}\subset\{x:0\leq\abs{x}\leq\delta^{-1}+\sqrt{2}/2\}.
		\end{equation*}
		We claim that for every $x\geq0$,
		\begin{equation}\label{eq:wedge_estimate}
			(1+x)\leq2(1\vee x).
		\end{equation}
		To prove~\eqref{eq:wedge_estimate}, it suffices to note
		\begin{align*}
			x\in[0,1]&\implies1+x\leq2=2(1\vee x),\\
			x\in(1,\infty)&\implies1+x\leq 2x=2(1\vee x).
		\end{align*}
		Hence, we can estimate
		\begin{equation*}
			\Bigl(\delta^{-1}+\frac{\sqrt{2}}{2}\Bigr)^{2}\leq(\delta^{-1}+1)^{2}\lesssim(1\vee\delta^{-1})^{2}=(1\vee\delta^{-2}),
		\end{equation*}
		which yields~\eqref{eq:summation_estimates_0}.
		
		To establish the inequality~\eqref{eq:summation_estimates_-2}, we estimate
		\begin{equation*}
			\begin{split}
				&\sum_{\substack{k\in\mbZ^2\setminus\{0\}\\\abs{k}\leq\delta^{-1}}}\frac{1}{\abs{k}^{2}}=\sum_{\substack{k\in\mbZ^2\setminus\{0\}\\\abs{k}\leq\delta^{-1}}}\frac{(\abs{k}+\frac{\sqrt{2}}{2})^{2}}{\abs{k}^{2}}\frac{1}{(\abs{k}+\frac{\sqrt{2}}{2})^{2}}\lesssim\sum_{\substack{k\in\mbZ^2\setminus\{0\}\\\abs{k}\leq\delta^{-1}}}\int_{\{x:\abs{x-k}_{\ell^{\infty}}\leq1/2\}}\frac{1}{(\abs{k}+\frac{\sqrt{2}}{2})^{2}}\dd x\\
				&\leq\sum_{\substack{k\in\mbZ^2\setminus\{0\}\\\abs{k}\leq\delta^{-1}}}\int_{\{x:\abs{x-k}_{\ell^{\infty}}\leq1/2\}}\frac{1}{\abs{x}^{2}}\dd x\leq\int_{\{x:1-\sqrt{2}/2<\abs{x}\leq\delta^{-1}+\sqrt{2}/2\}}\frac{1}{\abs{x}^{2}}\dd x\lesssim\log\Bigl(\delta^{-1}+\frac{\sqrt{2}}{2}\Bigr)-\log\Bigl(1-\frac{\sqrt{2}}{2}\Bigr).
			\end{split}
		\end{equation*}
		We claim that there exists some $C>0$ such that for every $x>0$,
		\begin{equation}\label{eq:log_estimate}
			\log\Bigl(x+\frac{\sqrt{2}}{2}\Bigr)+1\leq C(\log(x)\vee 1).
		\end{equation}
		To prove~\eqref{eq:log_estimate} with $C\geq\log\Bigl(\euler+\frac{\sqrt{2}}{2}\Bigr)+1$, we distinguish the cases $x\in(0,\euler]$ and $x>\euler$. For $x\in(0,\euler]$,
		\begin{equation*}
			\log\Bigl(x+\frac{\sqrt{2}}{2}\Bigr)+1\leq\log\Bigl(\euler+\frac{\sqrt{2}}{2}\Bigr)+1\leq C,
		\end{equation*}
		hence~\eqref{eq:log_estimate} holds for all $x\in[0,\euler]$. For $x>\euler$, we compare slopes
		\begin{equation*}
			\frac{1}{x+\frac{\sqrt{2}}{2}}\leq\frac{1}{x}\leq C\frac{1}{x},
		\end{equation*}
		where we used that $C\geq\log\Bigl(\euler+\frac{\sqrt{2}}{2}\Bigr)+1\geq1$, hence~\eqref{eq:log_estimate} holds for all $x>\euler$.
		
		With~\eqref{eq:log_estimate} established, we can bound for every $\delta>0$,
		\begin{equation*}
			\log\Bigl(\delta^{-1}+\frac{\sqrt{2}}{2}\Bigr)-\log\Bigl(1-\frac{\sqrt{2}}{2}\Bigr)\lesssim\log\Bigl(\delta^{-1}+\frac{\sqrt{2}}{2}\Bigr)+1\lesssim(\log(\delta^{-1})\vee 1),
		\end{equation*}
		which yields~\eqref{eq:summation_estimates_-2}.
		
		To establish the inequality~\eqref{eq:summation_estimates_-alpha}, we estimate
		\begin{equation*}
			\begin{split}
				&\sum_{k\in\mbZ^2\setminus\{0\}}\frac{1}{\abs{k}^{\alpha}}=\sum_{k\in\mbZ^2\setminus\{0\}}\frac{(\abs{k}+\frac{\sqrt{2}}{2})^{\alpha}}{\abs{k}^{\alpha}}\frac{1}{(\abs{k}+\frac{\sqrt{2}}{2})^{\alpha}}\lesssim_{\alpha}\sum_{k\in\mbZ^2\setminus\{0\}}\int_{\{x:\abs{x-k}_{\ell^\infty}\leq1/2\}}\frac{1}{(\abs{k}+\frac{\sqrt{2}}{2})^{\alpha}}\dd x\\
				&\leq\sum_{k\in\mbZ^2\setminus\{0\}}\int_{\{x:\abs{x-k}_{\ell^\infty}\leq1/2\}}\frac{1}{\abs{x}^{\alpha}}\dd x\leq\int_{\{x:\abs{x}>1-\sqrt{2}/2\}}\frac{1}{\abs{x}^{\alpha}}\dd x\lesssim_{\alpha}1.
			\end{split}
		\end{equation*}
		This yields the claim.
	\end{proof}
\end{details}
We make repeated use of the following convolution estimate to construct non-linear objects.
\begin{lemma}[{\cite[Lem.~3.10]{zhu_zhu_15},~\cite[Lem.~5 \& Lem.~6]{mourrat_weber_xu_16}}]\label{lem:convolution_estimates}
	Let $\alpha,\beta\in\mbR$ be such that $\alpha+\beta>2$. We have uniformly in $\om\in\mbZ^2$,
	\begin{equation*}
		\sum_{\substack{\om_1\in\mbZ^{2}\setminus\{0,\om\}\\\om_1\sim(\om-\om_1)}}\abs{\om_1}^{-\alpha}\abs{\om-\om_1}^{-\beta}\lesssim_{\alpha+\beta}(1\vee\abs{\om})^{-\alpha-\beta+2}.
	\end{equation*}
	If in addition $\alpha,\beta<2$, then we have uniformly in $\om\in\mbZ^2$,
	\begin{equation*}
		\sum_{\om_1\in\mbZ^{2}\setminus\{0,\om\}}\abs{\om_1}^{-\alpha}\abs{\om-\om_1}^{-\beta}\lesssim_{\alpha,\beta,\alpha+\beta}(1\vee\abs{\om})^{-\alpha-\beta+2}.
	\end{equation*}
\end{lemma}
\begin{details}
	Assume $\alpha=2$ and $\beta\in(0,2)$. We may still bound the convolution by losing an $\eps\in(0,\beta)$ of regularity
	\begin{equation*}
		\sum_{\om_1\in\mbZ^{2}\setminus\{0,\om\}}\abs{\om_1}^{-2}\abs{\om-\om_1}^{-\beta}\leq\sum_{\om_1\in\mbZ^{2}\setminus\{0,\om\}}\abs{\om_1}^{-2+\eps}\abs{\om-\om_1}^{-\beta}\lesssim(1\vee\abs{\om})^{-\beta+\eps}.
	\end{equation*}
\end{details}
The next convolution result is useful for estimating correlated frequencies $\om\not=\om'$.
\begin{lemma}\label{lem:convolution_twofold}
	Let $\alpha,\beta,\gamma\in(0,2)$ be such that $\alpha+\gamma>2$ and $\beta+\gamma>2$. Then for all $\om,\om'\in\mbZ^2\setminus\{0\}$ such that $\om\neq\om'$ it uniformly holds that
	\begin{equation*}
		\sum_{\om_1\in\mbZ^2\setminus\{0,\om,\om'\}}\abs{\om-\om_1}^{-\alpha}\abs{\om'-\om_1}^{-\beta}\abs{\om_1}^{-\gamma}\lesssim \abs{\om-\om'}^{-\beta}\abs{\om}^{-\alpha-\gamma+2}+\abs{\om-\om'}^{-\alpha}\abs{\om'}^{-\beta-\gamma+2}.
	\end{equation*}
\end{lemma}
\begin{proof}
	The proof follows by two applications of Lemma~\ref{lem:convolution_estimates}, one in the case $\abs{\om-\om_1}\leq \abs{\om-\om'}/2$ and the other in the complement.
\end{proof}
\begin{details}
	\begin{proof}
		Assume $\abs{\om-\om_1}\leq \abs{\om-\om'}/2$. It follows by the triangle inequality,
		\begin{equation*}
			\abs{\om-\om'}\leq\abs{\om-\om_1}+\abs{\om'-\om_1}
		\end{equation*}
		and so
		\begin{equation*}
			\abs{\om-\om'}/2\leq\abs{\om'-\om_1}.
		\end{equation*}
		Therefore by Lemma~\ref{lem:convolution_estimates}, using that $\beta>0$, $\alpha<2$, $\gamma<2$ and $\alpha+\gamma>2$,
		\begin{equation*}
			\begin{split}
				&\sum_{\substack{\om_1\in\mbZ^{2}\setminus\{0,\om,\om'\}\\\abs{\om-\om_1}\leq \abs{\om-\om'}/2}}\abs{\om-\om_1}^{-\alpha}\abs{\om'-\om_1}^{-\beta}\abs{\om_1}^{-\gamma}\\
				&\lesssim \abs{\om-\om'}^{-\beta}	\sum_{\om_1\in\mbZ^{2}\setminus\{0,\om\}}\abs{\om-\om_1}^{-\alpha}\abs{\om_1}^{-\gamma}\lesssim_{\alpha,\gamma,\alpha+\gamma}\abs{\om-\om'}^{-\beta}\abs{\om}^{-\alpha-\gamma+2}.
			\end{split}
		\end{equation*}
		Assume $\abs{\om-\om'}/2<\abs{\om-\om_1}$. We estimate by Lemma~\ref{lem:convolution_estimates}, using $\alpha>0$, $\beta,\gamma<2$, $\beta+\gamma>2$,
		\begin{equation*}
			\begin{split}
				&\sum_{\substack{\om_1\in\mbZ^{2}\setminus\{0,\om,\om'\}\\\abs{\om-\om'}/2<\abs{\om-\om_1}}}\abs{\om-\om_1}^{-\alpha}\abs{\om'-\om_1}^{-\beta}\abs{\om_1}^{-\gamma}\\
				&\lesssim\abs{\om-\om'}^{-\alpha}\sum_{\om_1\in\mbZ^{2}\setminus\{0,\om'\}}\abs{\om'-\om_1}^{-\beta}\abs{\om_1}^{-\gamma}\lesssim_{\beta,\gamma,\beta+\gamma}\abs{\om-\om'}^{-\alpha}\abs{\om'}^{-\beta-\gamma+2}.
			\end{split}
		\end{equation*}
		This yields the claim.
	\end{proof}
\end{details}
To derive finer estimates, it is useful to introduce a discrete paraproduct analogue to extend Lemma~\ref{lem:convolution_estimates}, see~\eqref{eq:definition_precsim} for the relevant notation.
\begin{lemma}\label{lem:convolution_estimate_paraproduct}
	Let $\alpha,\beta\in\mbR$ be such that $\alpha>2$ and $\beta\geq0$. We have uniformly in $\om\in\mbZ^2\setminus\{0\}$,
	\begin{equation*}
		\sum_{\substack{\om_1\in\mbZ^{2}\setminus\{0,\om\}\\\om_1\precsim(\om-\om_1)}}\abs{\om_1}^{-\alpha}\abs{\om-\om_1}^{-\beta}\lesssim_{\alpha}\abs{\om}^{-\beta}.
	\end{equation*}
\end{lemma}
\begin{proof}
	The proof is immediate by~\eqref{eq:summation_estimates_-alpha} and the bound induced by~\eqref{eq:definition_precsim}.
\end{proof}
\subsection{Double Sum Estimates}
When we take the Fourier transform of the noise $\het\boldsymbol{\xi}$, it generates convolutions of $\hat{\het}(t,\om-m_1)$ against $\dd W^{j_1}(t,m_1)$ in $m_1\in\mbZ^{2}$. We also generate convolutions in $\om_k\in\mbZ^{2}$ by constructing non-linear objects such as $\vdiv\mcI[\ti\nabla\Phi_{\ti}]$. Those steps lead to double sums over $\om_k$ and $m_k$ that do not factorize. In this section, we estimate those sums.

We apply the following estimate in Subsection~\ref{sec:Wick_contractions} to construct $\PreCocktail{50}$.
\begin{lemma}\label{lem:double_sum_PreCocktail_renormalised}
	Let $\gamma\in[0,1/2)$ and $\eps\in(0,1-2\gamma)$. Then uniformly in $\om,\om_1\in\mbZ^{2}$,  it holds that
	\begin{equation*}
		\begin{split}
			&\sum_{\substack{\om_4\in\mbZ^{2}\setminus\{0,\om,\om-\om_1\}\\(\om-\om_4)\sim\om_4}}\abs{\om-\om_4}^{-2}(1+\abs{\om}\abs{\om_4}^{-1})\abs{\om_4}^{2\gamma}\abs{\om-\om_1-\om_4}^{-1}\\
			&\quad\times\sum_{m_2\in\mbZ^{2}}(1+\abs{\om-\om_1-\om_4-m_2}^2)^{-1}(1+\abs{\om_4+m_2}^2)^{-1}\\
			&\lesssim(1\vee\abs{\om-\om_1})^{-2+\eps}(1\vee\abs{\om})^{-1+2\gamma+\eps}.
		\end{split}
	\end{equation*}
\end{lemma}
\begin{proof}
	We decompose the sum over $m_2\in\mbZ^{2}$ into the regions $m_2=\om-\om_1-\om_4$, $m_2=-\om_4$ and $m_2\in\mbZ^{2}\setminus\{\om-\om_1-\om_4,-\om_4\}$. 
	\begin{details}
		\begin{align}
			&\sum_{\substack{\om_4\in\mbZ^{2}\setminus\{0,\om,\om-\om_1\}\\(\om-\om_4)\sim\om_4}}\abs{\om-\om_4}^{-2}(1+\abs{\om}\abs{\om_4}^{-1})\abs{\om_4}^{2\gamma}\abs{\om-\om_1-\om_4}^{-1}\notag\\
			&\quad\times\sum_{m_2\in\mbZ^{2}}(1+\abs{\om-\om_1-\om_4-m_2}^2)^{-1}(1+\abs{\om_4+m_2}^2)^{-1}\notag\\
			&=2(1+\abs{\om-\om_1}^2)^{-1}\sum_{\substack{\om_4\in\mbZ^{2}\setminus\{0,\om,\om-\om_1\}\\(\om-\om_4)\sim\om_4}}\abs{\om-\om_4}^{-2}(1+\abs{\om}\abs{\om_4}^{-1})\abs{\om_4}^{2\gamma}\abs{\om-\om_1-\om_4}^{-1}\label{eq:double_sum_PreCocktail_renormalised_1}\\
			&\quad+\sum_{\substack{\om_4\in\mbZ^{2}\setminus\{0,\om,\om-\om_1\}\\(\om-\om_4)\sim\om_4}}\abs{\om-\om_4}^{-2}(1+\abs{\om}\abs{\om_4}^{-1})\abs{\om_4}^{2\gamma}\abs{\om-\om_1-\om_4}^{-1}\notag\\
			&\qquad\times\sum_{m_2\in\mbZ^{2}\setminus\{\om-\om_1-\om_4,-\om_4\}}(1+\abs{\om-\om_1-\om_4-m_2}^{2})^{-1}(1+\abs{\om_4+m_2}^{2})^{-1}.\label{eq:double_sum_PreCocktail_renormalised_2}
		\end{align}
	\end{details}
	We only give the bound that involves the sum over $m_2\in\mbZ^{2}\setminus\{\om-\om_1-\om_4,-\om_4\}$ as the other ones are simpler. We estimate
	\begin{equation*}
		\begin{split}
			&\sum_{m_2\in\mbZ^{2}\setminus\{\om-\om_1-\om_4,-\om_4\}}(1+\abs{\om-\om_1-\om_4-m_2}^{2})^{-1}(1+\abs{\om_4+m_2}^{2})^{-1}\\
			&\leq\sum_{m_2\in\mbZ^{2}\setminus\{\om-\om_1-\om_4,-\om_4\}}\abs{\om-\om_1-\om_4-m_2}^{-2}\abs{\om_4+m_2}^{-2}.
		\end{split}
	\end{equation*}
	Let $\eps\in(0,1)$; we can then apply Lemma~\ref{lem:convolution_estimates} to bound
	\begin{equation*}
		\begin{split}
			&\sum_{\substack{\om_4\in\mbZ^{2}\setminus\{0,\om,\om-\om_1\}\\(\om-\om_4)\sim\om_4}}\abs{\om-\om_4}^{-2}(1+\abs{\om}\abs{\om_4}^{-1})\abs{\om_4}^{2\gamma}\abs{\om-\om_1-\om_4}^{-1}\\
			&\quad\times\sum_{m_2\in\mbZ^{2}\setminus\{\om-\om_1-\om_4,-\om_4\}}\abs{\om-\om_1-\om_4-m_2}^{-2}\abs{\om_4+m_2}^{-2}\\
			&\lesssim(1\vee\abs{\om-\om_1})^{-2+2\eps}\sum_{\substack{\om_4\in\mbZ^{2}\setminus\{0,\om,\om-\om_1\}\\(\om-\om_4)\sim\om_4}}\abs{\om-\om_4}^{-2}(1+\abs{\om}\abs{\om_4}^{-1})\abs{\om_4}^{2\gamma}\abs{\om-\om_1-\om_4}^{-1}.
		\end{split}
	\end{equation*}
	\begin{details}
		Hence, to bound~\eqref{eq:double_sum_PreCocktail_renormalised_1} and~\eqref{eq:double_sum_PreCocktail_renormalised_2}, it suffices to estimate
		\begin{equation*}
			\sum_{\substack{\om_4\in\mbZ^{2}\setminus\{0,\om,\om-\om_1\}\\(\om-\om_4)\sim\om_4}}\abs{\om-\om_4}^{-2}(1+\abs{\om}\abs{\om_4}^{-1})\abs{\om_4}^{2\gamma}\abs{\om-\om_1-\om_4}^{-1}.
		\end{equation*}
	\end{details}
	Let $p\in(1,\infty)$, $q=p/(p-1)$ and $\delta>0$. By H\"{o}lder's inequality,
	\begin{equation*}
		\begin{split}
			&\sum_{\substack{\om_4\in\mbZ^{2}\setminus\{0,\om,\om-\om_1\}\\(\om-\om_4)\sim\om_4}}\abs{\om-\om_4}^{-2}(1+\abs{\om}\abs{\om_4}^{-1})\abs{\om_4}^{2\gamma}\abs{\om-\om_1-\om_4}^{-1}\\
			&\leq\Bigl(\sum_{\substack{\om_4\in\mbZ^{2}\setminus\{0,\om,\om-\om_1\}\\(\om-\om_4)\sim\om_4}}\abs{\om-\om_4}^{-2p}\abs{\om_4}^{\delta p}(1+\abs{\om}\abs{\om_4}^{-1})^{p}\Bigr)^{1/p}\\
			&\quad\times\Bigl(\sum_{\om_4\in\mbZ^{2}\setminus\{0,\om,\om-\om_1\}}\abs{\om_4}^{-\delta q}\abs{\om_4}^{2\gamma q}\abs{\om-\om_1-\om_4}^{-q}\Bigr)^{1/q}.
		\end{split}
	\end{equation*}
	We assume $\gamma\in[0,1/2)$, $2<p$ and $1-2/p+2\gamma<\delta<2-2/p$. It follows that
	\begin{equation*}
		p(2-\delta)>2,\qquad q(\delta-2\gamma)<2,\qquad q<2,\qquad q(\delta-2\gamma+1)>2.
	\end{equation*}
	and consequently by Lemma~\ref{lem:convolution_estimates},
	\begin{equation*}
		\begin{split}
			&\Bigl(\sum_{\substack{\om_4\in\mbZ^{2}\setminus\{0,\om,\om-\om_1\}\\(\om-\om_4)\sim\om_4}}\abs{\om-\om_4}^{-2p}\abs{\om_4}^{\delta p}(1+\abs{\om}\abs{\om_4}^{-1})^{p}\Bigr)^{1/p}\\
			&\quad\times\Bigl(\sum_{\om_4\in\mbZ^{2}\setminus\{0,\om,\om-\om_1\}}\abs{\om_4}^{-\delta q}\abs{\om_4}^{2\gamma q}\abs{\om-\om_1-\om_4}^{-q}\Bigr)^{1/q}\\
			&\lesssim(1\vee\abs{\om})^{-2+\delta+2/p}(1\vee\abs{\om-\om_1})^{-\delta+2\gamma-1+2/q}.
		\end{split}
	\end{equation*}
	Let $\eps\in(0,1-2\gamma)$; we can now set $\delta=1-2/p+2\gamma+\eps$ to conclude
	\begin{equation*}
		\sum_{\substack{\om_4\in\mbZ^{2}\setminus\{0,\om,\om-\om_1\}\\(\om-\om_4)\sim\om_4}}\abs{\om-\om_4}^{-2}(1+\abs{\om}\abs{\om_4}^{-1})\abs{\om_4}^{2\gamma}\abs{\om-\om_1-\om_4}^{-1}\lesssim(1\vee\abs{\om})^{-1+2\gamma+\eps}(1\vee\abs{\om-\om_1})^{-\eps},
	\end{equation*}
	which yields the claim.
\end{proof}
\begin{details}
	We apply the following estimate in Subsection~\ref{sec:Wick_contractions} to construct $\PreCocktail{30}$ and $\PreCocktail{40}$.
	\begin{lemma}\label{lem:double_sum_PreCocktail}
		Let $\gamma\in[0,1/2)$ and $\eps\in(0,(1-2\gamma)/2)$. Then uniformly in $\om,\om_1\in\mbZ^{2}$, it holds that 
		\begin{equation*}
			\begin{split}
				&\sum_{\substack{\om_4\in\mbZ^{2}\setminus\{0,\om-\om_1\}\\(\om-\om_4)\sim\om_4}}\abs{\om_4}^{-1+2\gamma}\abs{\om-\om_1-\om_4}^{-2}\sum_{m_2\in\mbZ^{2}}(1+\abs{\om-\om_1-\om_4-m_2}^2)^{-1}(1+\abs{\om_4+m_2}^2)^{-1}\\
				&\lesssim(1\vee\abs{\om})^{-1+2\gamma+2\eps}(1\vee\abs{\om-\om_1})^{-2+\eps}.
			\end{split}
		\end{equation*}
	\end{lemma}
	\begin{proof}
		Let $\eps\in(0,1)$. We first bound the sum over $m_2$. We distinguish the cases $m_2=\om-\om_1-\om_4$, $m_2=-\om_4$ and $m_2\in\mbZ^{2}\setminus\{\om-\om_1-\om_4,-\om_4\}$,
		\begin{align}
			&\sum_{\substack{\om_4\in\mbZ^{2}\setminus\{0,\om-\om_1\}\\(\om-\om_4)\sim\om_4}}\abs{\om_4}^{-1+2\gamma}\abs{\om-\om_1-\om_4}^{-2}\sum_{m_2\in\mbZ^{2}}(1+\abs{\om-\om_1-\om_4-m_2}^2)^{-1}(1+\abs{\om_4+m_2}^2)^{-1}\notag\\
			&=2(1+\abs{\om-\om_1}^2)^{-1}\sum_{\substack{\om_4\in\mbZ^{2}\setminus\{0,\om-\om_1\}\\(\om-\om_4)\sim\om_4}}\abs{\om_4}^{-1+2\gamma}\abs{\om-\om_1-\om_4}^{-2}\label{eq:double_sum_PreCocktail_1}\\
			&\quad+\sum_{\substack{\om_4\in\mbZ^{2}\setminus\{0,\om-\om_1\}\\(\om-\om_4)\sim\om_4}}\abs{\om_4}^{-1+2\gamma}\abs{\om-\om_1-\om_4}^{-2}\notag\\
			&\qquad\times\sum_{m_2\in\mbZ^{2}\setminus\{\om-\om_1-\om_4,-\om_4\}}(1+\abs{\om-\om_1-\om_4-m_2}^2)^{-1}(1+\abs{\om_4+m_2}^2)^{-1}.\label{eq:double_sum_PreCocktail_2}
		\end{align}
		Let $\eps\in(0,1)$. We first estimate~\eqref{eq:double_sum_PreCocktail_2} by Lemma~\ref{lem:convolution_estimates},
		\begin{equation*}
			\begin{split}
				&\sum_{\substack{\om_4\in\mbZ^{2}\setminus\{0,\om-\om_1\}\\(\om-\om_4)\sim\om_4}}\abs{\om_4}^{-1+2\gamma}\abs{\om-\om_1-\om_4}^{-2}\\
				&\quad\times\sum_{m_2\in\mbZ^{2}\setminus\{\om-\om_1-\om_4,-\om_4\}}(1+\abs{\om-\om_1-\om_4-m_2}^2)^{-1}(1+\abs{\om_4+m_2}^2)^{-1}\\
				&\lesssim(1\vee\abs{\om-\om_1})^{-2+2\eps}\sum_{\substack{\om_4\in\mbZ^{2}\setminus\{0,\om-\om_1\}\\(\om-\om_4)\sim\om_4}}\abs{\om_4}^{-1+2\gamma}\abs{\om-\om_1-\om_4}^{-2}.
			\end{split}
		\end{equation*}
		To bound~\eqref{eq:double_sum_PreCocktail_1} and~\eqref{eq:double_sum_PreCocktail_2}, it suffices to estimate
		\begin{equation*}
			\sum_{\substack{\om_4\in\mbZ^{2}\setminus\{0,\om-\om_1\}\\(\om-\om_4)\sim\om_4}}\abs{\om_4}^{-1+2\gamma}\abs{\om-\om_1-\om_4}^{-2}.
		\end{equation*}
		Assume $\gamma\in[0,1/2)$, $\eps\in(0,1-2\gamma)$, $p\in(1,\infty)$, $q=p/(p-1)$ and $\delta>0$. We estimate by H\"{o}lder's inequality,
		\begin{equation}\label{eq:PreCocktail_colour_difference}
			\begin{split}
				&\sum_{\substack{\om_4\in\mbZ^{2}\setminus\{0,\om-\om_1\}\\(\om-\om_4)\sim\om_4}}\abs{\om_4}^{-1+2\gamma}\abs{\om-\om_1-\om_4}^{-2}\leq\sum_{\substack{\om_4\in\mbZ^{2}\setminus\{0,\om-\om_1\}\\(\om-\om_4)\sim\om_4}}\abs{\om_4}^{-1+2\gamma}\abs{\om-\om_1-\om_4}^{-2+\eps}\\
				&\leq\Bigl(\sum_{\substack{\om_4\in\mbZ^{2}\setminus\{0,\om-\om_1\}\\(\om-\om_4)\sim\om_4}}\abs{\om_4}^{p(-1+\delta)}\Bigr)^{1/p}\Bigl(\sum_{\substack{\om_4\in\mbZ^{2}\setminus\{0,\om-\om_1\}}}\abs{\om_4}^{q(2\gamma-\delta )}\abs{\om-\om_1-\om_4}^{q(-2+\eps)}\Bigr)^{1/q}.
			\end{split}
		\end{equation}
		We assume $2/\eps<p$ and $-2/p+2\gamma+\eps<\delta<1-2/p$. It follows
		\begin{equation*}
			p(1-\delta)>2,\qquad q(\delta-2\gamma)<2,\qquad q(2-\eps)<2,\qquad q(\delta-2\gamma+2-\eps)>2.
		\end{equation*}
		We apply Lemma~\ref{lem:convolution_estimates} to obtain
		\begin{equation*}
			\begin{split}
				&\Bigl(\sum_{\substack{\om_4\in\mbZ^{2}\setminus\{0,\om-\om_1\}\\(\om-\om_4)\sim\om_4}}\abs{\om_4}^{p(-1+\delta)}\Bigr)^{1/p}\Bigl(\sum_{\substack{\om_4\in\mbZ^{2}\setminus\{0,\om-\om_1\}}}\abs{\om_4}^{q(2\gamma-\delta )}\abs{\om-\om_1-\om_4}^{q(-2+\eps)}\Bigr)^{1/q}\\
				&\lesssim(1\vee\abs{\om})^{-1+\delta+2/p}(1\vee\abs{\om-\om_1})^{2\gamma-\delta-2+\eps+2/q}.
			\end{split}
		\end{equation*}
		Assume $\eps\in(0,(1-2\gamma)/2)$. Upon choosing $\delta=-2/p+2\gamma+2\eps$,
		\begin{equation*}
			\sum_{\substack{\om_4\in\mbZ^{2}\setminus\{0,\om-\om_1\}\\(\om-\om_4)\sim\om_4}}\abs{\om_4}^{-1+2\gamma}\abs{\om-\om_1-\om_4}^{-2}\lesssim(1\vee\abs{\om})^{-1+2\gamma+2\eps}(1\vee\abs{\om-\om_1})^{-\eps}.
		\end{equation*}
		This yields the claim.
	\end{proof}
\end{details}
We apply the following estimate in Subsection~\ref{sec:canonical} to construct $\tl^{\delta}$, ${\PreThreeloop{10}\!}^{\delta}$ and ${\PreThreeloop{20}\!}^{\delta}$. We use the restriction $\abs{m_1}\leq\delta^{-1}$ induced by the cut-off $\varphi(\delta m_1)$ to establish a bound in terms of $\log(\delta^{-1})$.
\begin{lemma}\label{lem:sum_m_om_canonical}
	Let $\eps\in(0,1/2)$ and $\delta>0$. Then uniformly in $\om\in\mbZ^{2}\setminus\{0\}$ it holds that 
	\begin{equation}\label{eq:sum_m_om_canonical}
		\begin{split}
			&\sum_{\substack{m_1\in\mbZ^{2}\\\abs{m_1}\leq\delta^{-1}}}\sum_{\om_1\in\mbZ^{2}\setminus\{0,\om\}}(1+\abs{\om_1-m_1}^{2})^{-1}(1+\abs{\om-\om_1+m_1}^{2})^{-1}\abs{\om_1}^{-2}(1+\abs{\om}\abs{\om-\om_1}^{-1})\\
			&\lesssim\abs{\om}^{-2+3\eps}(1\vee\log(\delta^{-1})).
		\end{split}
	\end{equation}
\end{lemma}
\begin{proof}
	To bound~\eqref{eq:sum_m_om_canonical} it suffices to estimate the two parts
	\begin{equation}\label{eq:double_sum_bound_first_order}
		\sum_{\substack{m_1\in\mbZ^{2}\\\abs{m_1}\leq\delta^{-1}}}\sum_{\om_1\in\mbZ^{2}\setminus\{0,\om\}}(1+\abs{\om_1-m_1}^{2})^{-1}(1+\abs{\om-\om_1+m_1}^{2})^{-1}\abs{\om_1}^{-2}
	\end{equation}
	and
	\begin{equation}\label{eq:double_sum_bound_higher_order}
		\sum_{\substack{m_1\in\mbZ^{2}\\\abs{m_1}\leq\delta^{-1}}}\sum_{\om_1\in\mbZ^{2}\setminus\{0,\om\}}(1+\abs{\om_1-m_1}^{2})^{-1}(1+\abs{\om-\om_1+m_1}^{2})^{-1}\abs{\om_1}^{-2}\abs{\om-\om_1}^{-1}.
	\end{equation}
	
	Let us consider~\eqref{eq:double_sum_bound_first_order}. We decompose the sum over $m_1$ into the regions $m_1=0$, $m_1=-\om$ and $m_1\in\mbZ^{2}\setminus\{0,-\om\}$.
	\begin{details}
		\begin{align}
			&\sum_{\substack{m_1\in\mbZ^{2}\\\abs{m_1}\leq\delta^{-1}}}\sum_{\om_1\in\mbZ^{2}\setminus\{0,\om\}}(1+\abs{\om_1-m_1}^{2})^{-1}(1+\abs{\om-\om_1+m_1}^{2})^{-1}\abs{\om_1}^{-2}\notag\\
			&=\sum_{\om_1\in\mbZ^{2}\setminus\{0,\om\}}(1+\abs{\om_1}^{2})^{-1}(1+\abs{\om-\om_1}^{2})^{-1}\abs{\om_1}^{-2}\label{eq:double_sum_bound_first_order_special_cases_1}\\
			&\quad+\mathds{1}_{\abs{\om}\leq\delta^{-1}}\sum_{\om_1\in\mbZ^{2}\setminus\{0,\om\}}(1+\abs{\om+\om_1}^{2})^{-1}(1+\abs{\om_1}^{2})^{-1}\abs{\om_1}^{-2}\label{eq:double_sum_bound_first_order_special_cases_2}\\
			&\quad+\sum_{\substack{m_1\in\mbZ^{2}\setminus\{0,-\om\}\\\abs{m_1}\leq\delta^{-1}}}\sum_{\om_1\in\mbZ^{2}\setminus\{0,\om\}}(1+\abs{\om_1-m_1}^{2})^{-1}(1+\abs{\om-\om_1+m_1}^{2})^{-1}\abs{\om_1}^{-2}.\notag
		\end{align}
		Let $\eps\in(0,1)$. The first two terms~\eqref{eq:double_sum_bound_first_order_special_cases_1}, \eqref{eq:double_sum_bound_first_order_special_cases_2} can be estimated by Lemma~\ref{lem:convolution_estimates}. For~\eqref{eq:double_sum_bound_first_order_special_cases_1},
		\begin{align*}
			\sum_{\om_1\in\mbZ^{2}\setminus\{0,\om\}}(1+\abs{\om_1}^{2})^{-1}(1+\abs{\om-\om_1}^{2})^{-1}\abs{\om_1}^{-2}&\leq\sum_{\om_1\in\mbZ^{2}\setminus\{0,\om\}}\abs{\om-\om_1}^{-2}\abs{\om_1}^{-4}\\
			&\leq\sum_{\om_1\in\mbZ^{2}\setminus\{0,\om\}}\abs{\om-\om_1}^{-2+\eps}\abs{\om_1}^{-2+\eps}\lesssim\abs{\om}^{-2+2\eps},
		\end{align*}
		and for~\eqref{eq:double_sum_bound_first_order_special_cases_2},
		\begin{align*}
			\sum_{\om_1\in\mbZ^{2}\setminus\{0,\om\}}(1+\abs{\om+\om_1}^{2})^{-1}(1+\abs{\om_1}^{2})^{-1}\abs{\om_1}^{-2}&\leq\abs{\om}^{-4}+\sum_{\om_1\in\mbZ^{2}\setminus\{0,\om,-\om\}}\abs{\om+\om_1}^{-2}\abs{\om_1}^{-4}\\
			&\leq\abs{\om}^{-4}+\sum_{\om_1\in\mbZ^{2}\setminus\{0,\om,-\om\}}\abs{\om+\om_1}^{-2+\eps}\abs{\om_1}^{-2+\eps}\\
			&\lesssim\abs{\om}^{-2+2\eps}.
		\end{align*}
	\end{details}
	The sum over $m_1\in\mbZ^{2}\setminus\{0,-\om\}$ is given by
	\begin{equation*}
		\sum_{\substack{m_1\in\mbZ^{2}\setminus\{0,-\om\}\\\abs{m_1}\leq\delta^{-1}}}\sum_{\om_1\in\mbZ^{2}\setminus\{0,\om\}}(1+\abs{\om_1-m_1}^{2})^{-1}(1+\abs{\om-\om_1+m_1}^{2})^{-1}\abs{\om_1}^{-2}.
	\end{equation*}
	The sum over $\om_1\in\mbZ^{2}\setminus\{0,\om\}$ can then be further decomposed into the regions $\om_1=m_1$, $\om_1=\om+m_1$ and $\om_1\in\mbZ^{2}\setminus\{0,\om,m_1,\om+m_1\}$.
	\begin{details}
		\begin{align}
			&\sum_{\substack{m_1\in\mbZ^{2}\setminus\{0,-\om\}\\\abs{m_1}\leq\delta^{-1}}}\sum_{\om_1\in\mbZ^{2}\setminus\{0,\om\}}(1+\abs{\om_1-m_1}^{2})^{-1}(1+\abs{\om-\om_1+m_1}^{2})^{-1}\abs{\om_1}^{-2}\notag\\
			&=(1+\abs{\om}^{2})^{-1}\sum_{\substack{m_1\in\mbZ^{2}\setminus\{0,-\om\}\\\abs{m_1}\leq\delta^{-1}}}(\mathds{1}_{m_1\neq\om}\abs{m_1}^{-2}+\abs{\om+m_1}^{-2})\label{eq:double_sum_bound_first_order_general_cases_1}\\
			&\quad+\sum_{\substack{m_1\in\mbZ^{2}\setminus\{0,-\om\}\\\abs{m_1}\leq\delta^{-1}}}\sum_{\om_1\in\mbZ^{2}\setminus\{0,\om,m_1,\om+m_1\}}(1+\abs{\om_1-m_1}^{2})^{-1}(1+\abs{\om-\om_1+m_1}^{2})^{-1}\abs{\om_1}^{-2}.\notag
		\end{align}
		Let $\delta>0$. The first two sums~\eqref{eq:double_sum_bound_first_order_general_cases_1} can be estimated by~\eqref{eq:summation_estimates_-2},
		\begin{equation*}
			\sum_{\substack{m_1\in\mbZ^{2}\setminus\{0,-\om\}\\\abs{m_1}\leq\delta^{-1}}}(\abs{m_1}^{-2}+\abs{\om+m_1}^{-2})\lesssim(1\vee\log(\delta^{-1}))+(1\vee\log(\abs{\om}+\delta^{-1}))\lesssim\abs{\om}^{\eps}(1\vee\log(\delta^{-1})).
		\end{equation*}
		For the last inequality, we used that $\om\in\mbZ^{2}\setminus\{0\}$ and $1+\log(1+\delta^{-1})\lesssim(1\vee\log(\delta^{-1}))$ to estimate
		\begin{equation*}
			\log(\abs{\om}+\delta^{-1})\leq\log(\abs{\om}(1+\delta^{-1}))=\log(\abs{\om})+\log(1+\delta^{-1})\lesssim\abs{\om}^{\eps}(1+\log(1+\delta^{-1}))\lesssim\abs{\om}^{\eps}(1\vee\log(\delta^{-1})),
		\end{equation*}
		which yields
		\begin{equation*}
			(1\vee\log(\delta^{-1}))+(1\vee\log(\abs{\om}+\delta^{-1}))\lesssim(1+\abs{\om}^{\eps})(1\vee\log(\delta^{-1}))\lesssim\abs{\om}^{\eps}(1\vee\log(\delta^{-1})).
		\end{equation*}
		
	\end{details}
	We only give the bound that involves the sums over $m_1\in\mbZ^{2}\setminus\{0,-\om\}$ and $\om_1\in\mbZ^{2}\setminus\{0,\om,m_1,\om+m_1\}$. Using that $\om_1\in\mbZ^{2}\setminus\{0,\om,m_1,\om+m_1\}$, we may estimate
	\begin{equation*}
		\begin{split}
			&\sum_{\substack{m_1\in\mbZ^{2}\setminus\{0,-\om\}\\\abs{m_1}\leq\delta^{-1}}}\sum_{\om_1\in\mbZ^{2}\setminus\{0,\om,m_1,\om+m_1\}}(1+\abs{\om_1-m_1}^{2})^{-1}(1+\abs{\om-\om_1+m_1}^{2})^{-1}\abs{\om_1}^{-2}\\
			&\leq\sum_{\substack{m_1\in\mbZ^{2}\setminus\{0,-\om\}\\\abs{m_1}\leq\delta^{-1}}}\sum_{\om_1\in\mbZ^{2}\setminus\{0,\om,m_1,\om+m_1\}}\abs{\om_1-m_1}^{-2}\abs{\om-\om_1+m_1}^{-2}\abs{\om_1}^{-2}.
		\end{split}
	\end{equation*}
	Introducing the dyadic partition of unity $(\varrho_{q})_{q\in\mbN_{-1}}$ (cf.~\eqref{eq:definition_sim}~\&~\eqref{eq:definition_precsim}), we decompose this sum into
	\begin{align}
		&\sum_{\substack{m_1\in\mbZ^{2}\setminus\{0,-\om\}\\\abs{m_1}\leq\delta^{-1}}}\sum_{\om_1\in\mbZ^{2}\setminus\{0,\om,m_1,\om+m_1\}}\abs{\om_1-m_1}^{-2}\abs{\om-\om_1+m_1}^{-2}\abs{\om_1}^{-2}\notag\\
		&=\sum_{\substack{m_1\in\mbZ^{2}\setminus\{0,-\om\}\\\abs{m_1}\leq\delta^{-1}}}\sum_{\substack{\om_1\in\mbZ^{2}\setminus\{0,\om,m_1,\om+m_1\}\\\om_1\precsim(\om-\om_1+m_1)}}\abs{\om_1-m_1}^{-2}\abs{\om-\om_1+m_1}^{-2}\abs{\om_1}^{-2}\label{eq:double_sum_bound_first_order_1}\\
		&\quad+\sum_{\substack{m_1\in\mbZ^{2}\setminus\{0,-\om\}\\\abs{m_1}\leq\delta^{-1}}}\sum_{\substack{\om_1\in\mbZ^{2}\setminus\{0,\om,m_1,\om+m_1\}\\\om_1\succsim(\om-\om_1+m_1)}}\abs{\om_1-m_1}^{-2}\abs{\om-\om_1+m_1}^{-2}\abs{\om_1}^{-2}\label{eq:double_sum_bound_first_order_2}\\
		&\quad+\sum_{\substack{m_1\in\mbZ^{2}\setminus\{0,-\om\}\\\abs{m_1}\leq\delta^{-1}}}\sum_{\substack{\om_1\in\mbZ^{2}\setminus\{0,\om,m_1,\om+m_1\}\\\om_1\sim(\om-\om_1+m_1)}}\abs{\om_1-m_1}^{-2}\abs{\om-\om_1+m_1}^{-2}\abs{\om_1}^{-2}\label{eq:double_sum_bound_first_order_3}.
	\end{align}
	Assume $\eps<2/3$; the terms~\eqref{eq:double_sum_bound_first_order_1} and~\eqref{eq:double_sum_bound_first_order_3} can be estimated by two applications of Lemma~\ref{lem:convolution_estimates},
	\begin{equation*}
		\begin{split}
			&\sum_{\substack{m_1\in\mbZ^{2}\setminus\{0,-\om\}\\\abs{m_1}\leq\delta^{-1}}}\sum_{\substack{\om_1\in\mbZ^{2}\setminus\{0,\om,m_1,\om+m_1\}\\\abs{\om_1}\leq64/9\abs{\om-\om_1+m_1}}}\abs{\om_1-m_1}^{-2}\abs{\om-\om_1+m_1}^{-2}\abs{\om_1}^{-2}\\
			&\lesssim\sum_{\substack{m_1\in\mbZ^{2}\setminus\{0,-\om\}\\\abs{m_1}\leq\delta^{-1}}}\abs{\om+m_1}^{-2}\sum_{\substack{\om_1\in\mbZ^{2}\setminus\{0,\om,m_1,\om+m_1\}\\\abs{\om_1}\leq64/9\abs{\om-\om_1+m_1}}}\abs{\om_1-m_1}^{-2}\abs{\om_1}^{-2}\\
			&\lesssim\sum_{\substack{m_1\in\mbZ^{2}\setminus\{0,-\om\}\\\abs{m_1}\leq\delta^{-1}}}\abs{\om+m_1}^{-2}\abs{m_1}^{-2+2\eps}\lesssim\abs{\om}^{-2+3\eps}.
		\end{split}
	\end{equation*}
	The second term~\eqref{eq:double_sum_bound_first_order_2} can be estimated by Lemma~\ref{lem:convolution_estimates} and~\eqref{eq:summation_estimates_-2},
	\begin{equation*}
		\begin{split}
			&\sum_{\substack{m_1\in\mbZ^{2}\setminus\{0,-\om\}\\\abs{m_1}\leq\delta^{-1}}}\sum_{\substack{\om_1\in\mbZ^{2}\setminus\{0,\om,m_1,\om+m_1\}\\\om_1\succsim(\om-\om_1+m_1)}}\abs{\om_1-m_1}^{-2}\abs{\om-\om_1+m_1}^{-2}\abs{\om_1}^{-2}\\
			&\lesssim\sum_{\substack{m_1\in\mbZ^{2}\setminus\{0,-\om\}\\\abs{m_1}\leq\delta^{-1}}}\abs{\om+m_1}^{-2}\sum_{\substack{\om_1\in\mbZ^{2}\setminus\{0,\om,m_1,\om+m_1\}\\\om_1\succsim(\om-\om_1+m_1)}}\abs{\om_1-m_1}^{-2}\abs{\om-\om_1+m_1}^{-2}\\
			&\lesssim\abs{\om}^{-2+2\eps}\sum_{\substack{m_1\in\mbZ^{2}\setminus\{0,-\om\}\\\abs{m_1}\leq\delta^{-1}}}\abs{\om+m_1}^{-2}\lesssim\abs{\om}^{-2+3\eps}(1\vee\log(\delta^{-1})).
		\end{split}
	\end{equation*}
	
	Decomposing~\eqref{eq:double_sum_bound_first_order} as discussed and bounding the resulting terms yields
	\begin{equation*}
		\sum_{\substack{m_1\in\mbZ^{2}\\\abs{m_1}\leq\delta^{-1}}}\sum_{\om_1\in\mbZ^{2}\setminus\{0,\om\}}(1+\abs{\om_1-m_1}^{2})^{-1}(1+\abs{\om-\om_1+m_1}^{2})^{-1}\abs{\om_1}^{-2}\lesssim\abs{\om}^{-2+3\eps}(1\vee\log(\delta^{-1})).
	\end{equation*}
	Assume $\eps\in(0,1/2)$; the bound on~\eqref{eq:double_sum_bound_higher_order} then follows in a similar manner.
	\begin{details}
		We consider~\eqref{eq:double_sum_bound_higher_order}. We decompose the sum over $m_1\in\mbZ^{2}$ into the regions $m_1=0$, $m_1=-\om$, $m_1=\om$ and $m_1\in\mbZ^{2}\setminus\{0,-\om,\om\}$,
		\begin{align}
			&\sum_{\substack{m_1\in\mbZ^{2}\\\abs{m_1}\leq\delta^{-1}}}\sum_{\om_1\in\mbZ^{2}\setminus\{0,\om\}}(1+\abs{\om_1-m_1}^{2})^{-1}(1+\abs{\om-\om_1+m_1}^{2})^{-1}\abs{\om_1}^{-2}\abs{\om-\om_1}^{-1}\notag\\
			&=\sum_{\om_1\in\mbZ^{2}\setminus\{0,\om\}}(1+\abs{\om_1}^2)^{-1}(1+\abs{\om-\om_1}^{2})^{-1}\abs{\om_1}^{-2}\abs{\om-\om_1}^{-1}\label{eq:double_sum_bound_higher_order_special_cases_1}\\
			&\quad+\mathds{1}_{\abs{\om}\leq\delta^{-1}}\sum_{\om_1\in\mbZ^{2}\setminus\{0,\om\}}(1+\abs{\om+\om_1}^{2})^{-1}(1+\abs{\om_1}^{2})^{-1}\abs{\om_1}^{-2}\abs{\om-\om_1}^{-1}\label{eq:double_sum_bound_higher_order_special_cases_2}\\
			&\quad+\mathds{1}_{\abs{\om}\leq\delta^{-1}}\sum_{\om_1\in\mbZ^{2}\setminus\{0,\om\}}(1+\abs{\om-\om_1}^{2})^{-1}(1+\abs{2\om-\om_1}^{2})^{-1}\abs{\om_1}^{-2}\abs{\om-\om_1}^{-1}\label{eq:double_sum_bound_higher_order_special_cases_3}\\
			&\quad+\sum_{\substack{m_1\in\mbZ^{2}\setminus\{0,-\om,\om\}\\\abs{m_1}\leq\delta^{-1}}}\sum_{\om_1\in\mbZ^{2}\setminus\{0,\om\}}(1+\abs{\om_1-m_1}^{2})^{-1}(1+\abs{\om-\om_1+m_1}^{2})^{-1}\abs{\om_1}^{-2}\abs{\om-\om_1}^{-1}\label{eq:double_sum_bound_higher_order_general_cases}.
		\end{align}
		We estimate the first term~\eqref{eq:double_sum_bound_higher_order_special_cases_1} by
		\begin{equation*}
			\sum_{\om_1\in\mbZ^{2}\setminus\{0,\om\}}(1+\abs{\om_1}^2)^{-1}(1+\abs{\om-\om_1}^{2})^{-1}\abs{\om_1}^{-2}\abs{\om-\om_1}^{-1}\leq\sum_{\om_1\in\mbZ^{2}\setminus\{0,\om\}}\abs{\om_1}^{-4}\abs{\om-\om_1}^{-3}.
		\end{equation*}
		To decompose this sum, we introduce the dyadic partition of unity $(\varrho_{q})_{q\in\mbN_{-1}}$,
		\begin{equation*}
			\begin{split}
				&\sum_{\om_1\in\mbZ^{2}\setminus\{0,\om\}}\abs{\om_1}^{-4}\abs{\om-\om_1}^{-3}\\
				&=\sum_{\substack{\om_1\in\mbZ^{2}\setminus\{0,\om\}\\\om_1\precsim(\om-\om_1)}}\abs{\om_1}^{-4}\abs{\om-\om_1}^{-3}+\sum_{\substack{\om_1\in\mbZ^{2}\setminus\{0,\om\}\\\om_1\succsim(\om-\om_1)}}\abs{\om_1}^{-4}\abs{\om-\om_1}^{-3}+\sum_{\substack{\om_1\in\mbZ^{2}\setminus\{0,\om\}\\\om_1\sim(\om-\om_1)}}\abs{\om_1}^{-4}\abs{\om-\om_1}^{-3}\\
				&\lesssim\abs{\om}^{-3}\sum_{\substack{\om_1\in\mbZ^{2}\setminus\{0,\om\}\\\om_1\precsim(\om-\om_1)}}\abs{\om_1}^{-4}+\abs{\om}^{-4}\sum_{\substack{\om_1\in\mbZ^{2}\setminus\{0,\om\}\\\om_1\succsim(\om-\om_1)}}\abs{\om-\om_1}^{-3}+\sum_{\substack{\om_1\in\mbZ^{2}\setminus\{0,\om\}\\\om_1\sim(\om-\om_1)}}\abs{\om_1}^{-4}\abs{\om-\om_1}^{-3}\\
				&\lesssim\abs{\om}^{-3}+\abs{\om}^{-4}+\abs{\om}^{-5},
			\end{split}
		\end{equation*}
		where in the last line we applied~\eqref{eq:summation_estimates_-alpha} and Lemma~\ref{lem:convolution_estimates}. Let $\eps\in(0,1)$. We bound the term~\eqref{eq:double_sum_bound_higher_order_special_cases_2},
		\begin{equation*}
			\begin{split}
				&\sum_{\om_1\in\mbZ^{2}\setminus\{0,\om\}}(1+\abs{\om+\om_1}^{2})^{-1}(1+\abs{\om_1}^{2})^{-1}\abs{\om_1}^{-2}\abs{\om-\om_1}^{-1}\\
				&\leq\abs{\om}^{-5}+\sum_{\om_1\in\mbZ^{2}\setminus\{0,\om,-\om\}}\abs{\om+\om_1}^{-2}\abs{\om_1}^{-4}\abs{\om-\om_1}^{-1},
			\end{split}
		\end{equation*}
		which can be controlled by
		\begin{equation*}
			\begin{split}
				&\sum_{\om_1\in\mbZ^{2}\setminus\{0,\om,-\om\}}\abs{\om+\om_1}^{-2}\abs{\om_1}^{-4}\abs{\om-\om_1}^{-1}\\
				&=\sum_{\substack{\om_1\in\mbZ^{2}\setminus\{0,\om,-\om\}\\\om_1\precsim(\om-\om_1)}}\abs{\om+\om_1}^{-2}\abs{\om_1}^{-4}\abs{\om-\om_1}^{-1}+\sum_{\substack{\om_1\in\mbZ^{2}\setminus\{0,\om,-\om\}\\\om_1\succsim(\om-\om_1)}}\abs{\om+\om_1}^{-2}\abs{\om_1}^{-4}\abs{\om-\om_1}^{-1}\\
				&\quad+\sum_{\substack{\om_1\in\mbZ^{2}\setminus\{0,\om,-\om\}\\\om_1\sim(\om-\om_1)}}\abs{\om+\om_1}^{-2}\abs{\om_1}^{-4}\abs{\om-\om_1}^{-1}\\
				&\lesssim\abs{\om}^{-1}\sum_{\substack{\om_1\in\mbZ^{2}\setminus\{0,\om,-\om\}\\\om_1\precsim(\om-\om_1)}}\abs{\om+\om_1}^{-2}\abs{\om_1}^{-4}+\abs{\om}^{-4}\sum_{\substack{\om_1\in\mbZ^{2}\setminus\{0,\om,-\om\}\\\om_1\succsim(\om-\om_1)}}\abs{\om+\om_1}^{-2}\abs{\om-\om_1}^{-1}\\
				&\quad+\sum_{\substack{\om_1\in\mbZ^{2}\setminus\{0,\om,-\om\}\\\om_1\sim(\om-\om_1)}}\abs{\om+\om_1}^{-2}\abs{\om_1}^{-5}\\
				&\lesssim\abs{\om}^{-3+2\eps}+\abs{\om}^{-5+\eps}+\abs{\om}^{-5}.
			\end{split}
		\end{equation*}
		To estimate~\eqref{eq:double_sum_bound_higher_order_special_cases_3} we apply Lemma~\ref{lem:convolution_twofold},
		\begin{equation*}
			\begin{split}
				&\sum_{\om_1\in\mbZ^{2}\setminus\{0,\om\}}(1+\abs{\om-\om_1}^{2})^{-1}(1+\abs{2\om-\om_1}^{2})^{-1}\abs{\om_1}^{-2}\abs{\om-\om_1}^{-1}\\
				&\leq\abs{\om}^{-5}+\sum_{\om_1\in\mbZ^{2}\setminus\{0,\om,2\om\}}\abs{\om-\om_1}^{-3}\abs{2\om-\om_1}^{-2}\abs{\om_1}^{-2}\\
				&\leq\abs{\om}^{-5}+\sum_{\om_1\in\mbZ^{2}\setminus\{0,\om,2\om\}}\abs{\om-\om_1}^{-2+\eps}\abs{2\om-\om_1}^{-2+\eps}\abs{\om_1}^{-2+\eps}\\
				&\lesssim\abs{\om}^{-5}+\abs{\om}^{-4+3\eps}.
			\end{split}
		\end{equation*}
		We decompose~\eqref{eq:double_sum_bound_higher_order_general_cases} further into the cases $\om_1=m_1$, $\om_1=\om+m_1$ and $\om_1\in\mbZ^{2}\setminus\{0,\om,m_1,\om+m_1\}$,
		\begin{align}
			&\sum_{\substack{m_1\in\mbZ^{2}\setminus\{0,-\om,\om\}\\\abs{m_1}\leq\delta^{-1}}}\sum_{\om_1\in\mbZ^{2}\setminus\{0,\om\}}(1+\abs{\om_1-m_1}^{2})^{-1}(1+\abs{\om-\om_1+m_1}^{2})^{-1}\abs{\om_1}^{-2}\abs{\om-\om_1}^{-1}\notag\\
			&=(1+\abs{\om}^{2})^{-1}\sum_{\substack{m_1\in\mbZ^{2}\setminus\{0,-\om,\om\}\\\abs{m_1}\leq\delta^{-1}}}(\abs{m_1}^{-2}\abs{\om-m_1}^{-1}+\abs{\om+m_1}^{-2}\abs{m_1}^{-1})\label{eq:double_sum_bound_higher_order_general_cases_1}\\
			&\quad+\sum_{\substack{m_1\in\mbZ^{2}\setminus\{0,-\om,\om\}\\\abs{m_1}\leq\delta^{-1}}}\sum_{\om_1\in\mbZ^{2}\setminus\{0,\om,m_1,\om+m_1\}}(1+\abs{\om_1-m_1}^{2})^{-1}(1+\abs{\om-\om_1+m_1}^{2})^{-1}\abs{\om_1}^{-2}\abs{\om-\om_1}^{-1}.\label{eq:double_sum_bound_higher_order_general_cases_2}
		\end{align}
		We estimate~\eqref{eq:double_sum_bound_higher_order_general_cases_1} by Lemma~\ref{lem:convolution_estimates},
		\begin{equation*}
			\sum_{\substack{m_1\in\mbZ^{2}\setminus\{0,-\om,\om\}\\\abs{m_1}\leq\delta^{-1}}}(\abs{m_1}^{-2}\abs{\om-m_1}^{-1}+\abs{\om+m_1}^{-2}\abs{m_1}^{-1})\lesssim\abs{\om}^{-1+\eps}.
		\end{equation*}
		Using that $\om_1\in\mbZ^{2}\setminus\{0,\om,m_1,\om+m_1\}$, we may bound~\eqref{eq:double_sum_bound_higher_order_general_cases_2} by
		\begin{equation*}
			\begin{split}
				&\sum_{\substack{m_1\in\mbZ^{2}\setminus\{0,-\om,\om\}\\\abs{m_1}\leq\delta^{-1}}}\sum_{\om_1\in\mbZ^{2}\setminus\{0,\om,m_1,\om+m_1\}}(1+\abs{\om_1-m_1}^{2})^{-1}(1+\abs{\om-\om_1+m_1}^{2})^{-1}\abs{\om_1}^{-2}\abs{\om-\om_1}^{-1}\\
				&\leq\sum_{\substack{m_1\in\mbZ^{2}\setminus\{0,-\om,\om\}\\\abs{m_1}\leq\delta^{-1}}}\sum_{\om_1\in\mbZ^{2}\setminus\{0,\om,m_1,\om+m_1\}}\abs{\om_1-m_1}^{-2}\abs{\om-\om_1+m_1}^{-2}\abs{\om_1}^{-2}\abs{\om-\om_1}^{-1}.
			\end{split}
		\end{equation*}
		Using the dyadic partition of unity $(\varrho_{q})_{q\in\mbN_{-1}}$, we decompose this sum into
		\begin{align}
			&\sum_{\substack{m_1\in\mbZ^{2}\setminus\{0,-\om,\om\}\\\abs{m_1}\leq\delta^{-1}}}\sum_{\om_1\in\mbZ^{2}\setminus\{0,\om,m_1,\om+m_1\}}\abs{\om_1-m_1}^{-2}\abs{\om-\om_1+m_1}^{-2}\abs{\om_1}^{-2}\abs{\om-\om_1}^{-1}\notag\\
			&=\sum_{\substack{m_1\in\mbZ^{2}\setminus\{0,-\om,\om\}\\\abs{m_1}\leq\delta^{-1}}}\sum_{\substack{\om_1\in\mbZ^{2}\setminus\{0,\om,m_1,\om+m_1\}\\\om_1\precsim(\om-\om_1+m_1)}}\abs{\om_1-m_1}^{-2}\abs{\om-\om_1+m_1}^{-2}\abs{\om_1}^{-2}\abs{\om-\om_1}^{-1}\label{eq:double_sum_bound_higher_order_1}\\
			&\quad+\sum_{\substack{m_1\in\mbZ^{2}\setminus\{0,-\om,\om\}\\\abs{m_1}\leq\delta^{-1}}}\sum_{\substack{\om_1\in\mbZ^{2}\setminus\{0,\om,m_1,\om+m_1\}\\\om_1\succsim(\om-\om_1+m_1)}}\abs{\om_1-m_1}^{-2}\abs{\om-\om_1+m_1}^{-2}\abs{\om_1}^{-2}\abs{\om-\om_1}^{-1}\label{eq:double_sum_bound_higher_order_2}\\
			&\quad+\sum_{\substack{m_1\in\mbZ^{2}\setminus\{0,-\om,\om\}\\\abs{m_1}\leq\delta^{-1}}}\sum_{\substack{\om_1\in\mbZ^{2}\setminus\{0,\om,m_1,\om+m_1\}\\\om_1\sim(\om-\om_1+m_1)}}\abs{\om_1-m_1}^{-2}\abs{\om-\om_1+m_1}^{-2}\abs{\om_1}^{-2}\abs{\om-\om_1}^{-1}\label{eq:double_sum_bound_higher_order_3}.
		\end{align}
		Let $\eps<1/2$. We estimate~\eqref{eq:double_sum_bound_higher_order_1} and~\eqref{eq:double_sum_bound_higher_order_3} by Lemma~\ref{lem:convolution_twofold},
		\begin{equation*}
			\begin{split}
				&\sum_{\substack{m_1\in\mbZ^{2}\setminus\{0,-\om,\om\}\\\abs{m_1}\leq\delta^{-1}}}\sum_{\substack{\om_1\in\mbZ^{2}\setminus\{0,\om,m_1,\om+m_1\}\\\abs{\om_1}\leq64/9\abs{\om-\om_1+m_1}}}\abs{\om_1-m_1}^{-2}\abs{\om-\om_1+m_1}^{-2}\abs{\om_1}^{-2}\abs{\om-\om_1}^{-1}\\
				&\lesssim\sum_{\substack{m_1\in\mbZ^{2}\setminus\{0,-\om,\om\}\\\abs{m_1}\leq\delta^{-1}}}\abs{\om+m_1}^{-2}\sum_{\substack{\om_1\in\mbZ^{2}\setminus\{0,\om,m_1,\om+m_1\}\\\abs{\om_1}\leq64/9\abs{\om-\om_1+m_1}}}\abs{\om_1-m_1}^{-2}\abs{\om_1}^{-2}\abs{\om-\om_1}^{-1}\\
				&\lesssim\sum_{\substack{m_1\in\mbZ^{2}\setminus\{0,-\om,\om\}\\\abs{m_1}\leq\delta^{-1}}}\abs{\om+m_1}^{-2}(\abs{\om-m_1}^{-1}\abs{m_1}^{-2+2\eps}+\abs{\om-m_1}^{-2+\eps}\abs{\om}^{-1+\eps})\lesssim\abs{\om}^{-3+3\eps}.
			\end{split}
		\end{equation*}
		We estimate~\eqref{eq:double_sum_bound_higher_order_2} by
		\begin{equation*}
			\begin{split}
				&\sum_{\substack{m_1\in\mbZ^{2}\setminus\{0,-\om,\om\}\\\abs{m_1}\leq\delta^{-1}}}\sum_{\substack{\om_1\in\mbZ^{2}\setminus\{0,\om,m_1,\om+m_1\}\\\om_1\succsim(\om-\om_1+m_1)}}\abs{\om_1-m_1}^{-2}\abs{\om-\om_1+m_1}^{-2}\abs{\om_1}^{-2}\abs{\om-\om_1}^{-1}\\
				&\lesssim\sum_{\substack{m_1\in\mbZ^{2}\setminus\{0,-\om,\om\}\\\abs{m_1}\leq\delta^{-1}}}\abs{\om+m_1}^{-2}\sum_{\substack{\om_1\in\mbZ^{2}\setminus\{0,\om,m_1,\om+m_1\}\\\om_1\succsim(\om-\om_1+m_1)}}\abs{\om_1-m_1}^{-2}\abs{\om-\om_1+m_1}^{-2}\abs{\om-\om_1}^{-1}\\
				&\leq\sum_{\substack{m_1\in\mbZ^{2}\setminus\{0,-\om,\om\}\\\abs{m_1}\leq\delta^{-1}}}\abs{\om+m_1}^{-2}\sum_{\om_1\in\mbZ^{2}\setminus\{0,\om,-m_1,\om-m_1\}}\abs{\om-\om_1-m_1}^{-2}\abs{\om_1+m_1}^{-2}\abs{\om_1}^{-1}\\
				&\lesssim\abs{\om}^{-2+\eps}\sum_{\substack{m_1\in\mbZ^{2}\setminus\{0,-\om,\om\}\\\abs{m_1}\leq\delta^{-1}}}\abs{\om+m_1}^{-2}(\abs{m_1}^{-1+\eps}+\abs{\om-m_1}^{-1+\eps})\lesssim\abs{\om}^{-3+3\eps}.
			\end{split}
		\end{equation*}
		We obtain~\eqref{eq:sum_m_om_canonical} by summing all of our estimates above.
	\end{details}
\end{proof}
\begin{details}
	We use the following estimate to establish that $\tl^{\delta}\in\msL_{T}^{1-}\mcC^{1-}(\mbT^{2})$.
	\begin{lemma}\label{lem:sum_m_om_canonical_regular}
		Let $\eps\in(0,1)$ and $\delta>0$. Then uniformly in $\om\in\mbZ^{2}\setminus\{0\}$ it holds that 
		\begin{equation}\label{eq:sum_m_om_canonical_regular}
			\begin{split}
				&\sum_{\substack{m_1\in\mbZ^{2}\\\abs{m_1}\leq\delta^{-1}}}\sum_{\om_1\in\mbZ^{2}\setminus\{0,\om\}}\frac{\abs{\om_1}^{-1}}{(1+\abs{\delta m_1}^{2})^{1/2}(1+\abs{\om_1-m_1}^{2})(1+\abs{\om-\om_1+m_1}^{2})}\\
				&\lesssim_{\eps}\abs{\om}^{-2+3\eps}(1\vee(\delta^{-1}\log(\delta^{-1}))).
			\end{split}
		\end{equation}
	\end{lemma}
	\begin{proof}
		We decompose the sum over $m_1\in\mbZ^{2}$, $\abs{m_1}\leq\delta^{-1}$, into the regions $m_1=0$, $m_1=-\om$ and $m_1\in\mbZ^{2}\setminus\{0,-\om\}$,
		\begin{align}
			&\sum_{\substack{m_1\in\mbZ^{2}\\\abs{m_1}\leq\delta^{-1}}}(1+\abs{\delta m_1}^{2})^{-1/2}\sum_{\om_1\in\mbZ^{2}\setminus\{0,\om\}}(1+\abs{\om_1-m_1}^{2})^{-1}(1+\abs{\om-\om_1+m_1}^{2})^{-1}\abs{\om_1}^{-1}\notag\\
			&=\sum_{\om_1\in\mbZ^{2}\setminus\{0,\om\}}(1+\abs{\om_1}^{2})^{-1}(1+\abs{\om-\om_1}^{2})^{-1}\abs{\om_1}^{-1}\label{eq:double_sum_bound_first_order_special_cases_regular_1}\\
			&\quad+\mathds{1}_{\abs{\om}\leq\delta^{-1}}(1+\abs{\delta\om}^{2})^{-1/2}\sum_{\om_1\in\mbZ^{2}\setminus\{0,\om\}}(1+\abs{\om+\om_1}^{2})^{-1}(1+\abs{\om_1}^{2})^{-1}\abs{\om_1}^{-1}\label{eq:double_sum_bound_first_order_special_cases_regular_2}\\
			&\quad+\sum_{\substack{m_1\in\mbZ^{2}\setminus\{0,-\om\}\\\abs{m_1}\leq\delta^{-1}}}(1+\abs{\delta m_1}^{2})^{-1/2}\sum_{\om_1\in\mbZ^{2}\setminus\{0,\om\}}(1+\abs{\om_1-m_1}^{2})^{-1}(1+\abs{\om-\om_1+m_1}^{2})^{-1}\abs{\om_1}^{-1}\label{eq:double_sum_bound_first_order_general_cases_regular}.
		\end{align}
		Let $\eps\in(0,1)$. We first estimate~\eqref{eq:double_sum_bound_first_order_special_cases_regular_1},
		\begin{equation*}
			\sum_{\om_1\in\mbZ^{2}\setminus\{0,\om\}}(1+\abs{\om_1}^{2})^{-1}(1+\abs{\om-\om_1}^{2})^{-1}\abs{\om_1}^{-1}\lesssim\sum_{\om_1\in\mbZ^{2}\setminus\{0,\om\}}\abs{\om-\om_1}^{-2+\eps}\abs{\om_1}^{-2+\eps}\lesssim\abs{\om}^{-2+2\eps}.
		\end{equation*}
		We estimate~\eqref{eq:double_sum_bound_first_order_special_cases_regular_2} by
		\begin{align*}
			\sum_{\om_1\in\mbZ^{2}\setminus\{0,\om\}}(1+\abs{\om+\om_1}^{2})^{-1}(1+\abs{\om_1}^{2})^{-1}\abs{\om_1}^{-1}&\leq\abs{\om}^{-3}+\sum_{\om_1\in\mbZ^{2}\setminus\{0,\om,-\om\}}\abs{\om+\om_1}^{-2}\abs{\om_1}^{-3}\\
			&\leq\abs{\om}^{-3}+\sum_{\om_1\in\mbZ^{2}\setminus\{0,\om,-\om\}}\abs{\om+\om_1}^{-2+\eps}\abs{\om_1}^{-2+\eps}\\
			&\lesssim\abs{\om}^{-2+2\eps}.
		\end{align*}
		The term~\eqref{eq:double_sum_bound_first_order_general_cases_regular} can be further decomposed into the regions $\om_1=m_1$, $\om_1=\om+m_1$ and $\om_1\in\mbZ^{2}\setminus\{0,\om,m_1,\om+m_1\}$,
		\begin{align}
			&\sum_{\substack{m_1\in\mbZ^{2}\setminus\{0,-\om\}\\\abs{m_1}\leq\delta^{-1}}}(1+\abs{\delta m_1}^{2})^{-1/2}\sum_{\om_1\in\mbZ^{2}\setminus\{0,\om\}}(1+\abs{\om_1-m_1}^{2})^{-1}(1+\abs{\om-\om_1+m_1}^{2})^{-1}\abs{\om_1}^{-1}\notag\\
			&=(1+\abs{\om}^{2})^{-1}\sum_{\substack{m_1\in\mbZ^{2}\setminus\{0,-\om\}\\\abs{m_1}\leq\delta^{-1}}}(1+\abs{\delta m_1}^{2})^{-1/2}(\mathds{1}_{m_1\neq\om}\abs{m_1}^{-1}+\abs{\om+m_1}^{-1})\label{eq:double_sum_bound_first_order_general_cases_regular_1}\\
			&\quad+\sum_{\substack{m_1\in\mbZ^{2}\setminus\{0,-\om\}\\\abs{m_1}\leq\delta^{-1}}}(1+\abs{\delta m_1}^{2})^{-1/2}\sum_{\om_1\in\mbZ^{2}\setminus\{0,\om,m_1,\om+m_1\}}(1+\abs{\om_1-m_1}^{2})^{-1}(1+\abs{\om-\om_1+m_1}^{2})^{-1}\abs{\om_1}^{-2}\label{eq:double_sum_bound_first_order_general_cases_regular_2}.
		\end{align}
		The first two sums in~\eqref{eq:double_sum_bound_first_order_general_cases_regular_1} can be estimated by the bound $(1+\abs{\delta m_{1}}^{2})^{-1/2}\lesssim\abs{\delta m_{1}}^{-1}$ (for $\abs{m_{1}}\leq\delta^{-1}$), Young's product inequality and~\eqref{eq:summation_estimates_-2}, 
		\begin{equation*}
			\begin{split}
				&\sum_{\substack{m_1\in\mbZ^{2}\setminus\{0,-\om\}\\\abs{m_1}\leq\delta^{-1}}}(1+\abs{\delta m_1}^{2})^{-1/2}(\abs{m_1}^{-1}+\abs{\om+m_1}^{-1})\\
				&\lesssim\delta^{-1}\sum_{\substack{m_1\in\mbZ^{2}\setminus\{0,-\om\}\\\abs{m_1}\leq\delta^{-1}}}(\abs{m_1}^{-2}+\abs{m_{1}}^{-1}\abs{\om+m_1}^{-1})\\
				&\lesssim\delta^{-1}\sum_{\substack{m_1\in\mbZ^{2}\setminus\{0,-\om\}\\\abs{m_1}\leq\delta^{-1}}}(\abs{m_1}^{-2}+\abs{\om+m_1}^{-2})\\
				&\lesssim\delta^{-1}((1\vee\log(\delta^{-1}))+(1\vee\log(\abs{\om}+\delta^{-1})))\lesssim\abs{\om}^{\eps}\delta^{-1}(1\vee\log(\delta^{-1})).
			\end{split}
		\end{equation*}
		Using that $\om_1\in\mbZ^{2}\setminus\{0,\om,m_1,\om+m_1\}$, we may bound the second sum in~\eqref{eq:double_sum_bound_first_order_general_cases_regular_2} by
		\begin{equation*}
			\begin{split}
				&\sum_{\om_1\in\mbZ^{2}\setminus\{0,\om,m_1,\om+m_1\}}(1+\abs{\om_1-m_1}^{2})^{-1}(1+\abs{\om-\om_1+m_1}^{2})^{-1}\abs{\om_1}^{-1}\\
				&\leq\sum_{\om_1\in\mbZ^{2}\setminus\{0,\om,m_1,\om+m_1\}}\abs{\om_1-m_1}^{-2}\abs{\om-\om_1+m_1}^{-2}\abs{\om_1}^{-1}.
			\end{split}
		\end{equation*}
		Introducing the dyadic partition of unity $(\varrho_{q})_{q\in\mbN_{-1}}$, we decompose this sum into
		\begin{align}
			&\sum_{\om_1\in\mbZ^{2}\setminus\{0,\om,m_1,\om+m_1\}}\abs{\om_1-m_1}^{-2}\abs{\om-\om_1+m_1}^{-2}\abs{\om_1}^{-1}\notag\\
			&=\sum_{\substack{\om_1\in\mbZ^{2}\setminus\{0,\om,m_1,\om+m_1\}\\\om_1\precsim(\om-\om_1+m_1)}}\abs{\om_1-m_1}^{-2}\abs{\om-\om_1+m_1}^{-2}\abs{\om_1}^{-1}\label{eq:double_sum_bound_first_order_regular_1}\\
			&\quad+\sum_{\substack{\om_1\in\mbZ^{2}\setminus\{0,\om,m_1,\om+m_1\}\\\om_1\succsim(\om-\om_1+m_1)}}\abs{\om_1-m_1}^{-2}\abs{\om-\om_1+m_1}^{-2}\abs{\om_1}^{-1}\label{eq:double_sum_bound_first_order_regular_2}\\
			&\quad+\sum_{\substack{\om_1\in\mbZ^{2}\setminus\{0,\om,m_1,\om+m_1\}\\\om_1\sim(\om-\om_1+m_1)}}\abs{\om_1-m_1}^{-2}\abs{\om-\om_1+m_1}^{-2}\abs{\om_1}^{-1}\label{eq:double_sum_bound_first_order_regular_3}.
		\end{align}
		The terms~\eqref{eq:double_sum_bound_first_order_regular_1} and~\eqref{eq:double_sum_bound_first_order_regular_3} can be estimated by two applications of Lemma~\ref{lem:convolution_estimates},
		\begin{equation*}
			\begin{split}
				&\sum_{\substack{m_1\in\mbZ^{2}\setminus\{0,-\om\}\\\abs{m_1}\leq\delta^{-1}}}(1+\abs{\delta m_1}^{2})^{-1/2}\sum_{\substack{\om_1\in\mbZ^{2}\setminus\{0,\om,m_1,\om+m_1\}\\\abs{\om_1}\leq64/9\abs{\om-\om_1+m_1}}}\abs{\om_1-m_1}^{-2}\abs{\om-\om_1+m_1}^{-2}\abs{\om_1}^{-1}\\
				&\lesssim\delta^{-1}\sum_{\substack{m_1\in\mbZ^{2}\setminus\{0,-\om\}\\\abs{m_1}\leq\delta^{-1}}}\abs{m_1}^{-1}\abs{\om+m_1}^{-2}\sum_{\substack{\om_1\in\mbZ^{2}\setminus\{0,\om,m_1,\om+m_1\}\\\abs{\om_1}\leq64/9\abs{\om-\om_1+m_1}}}\abs{\om_1-m_1}^{-2}\abs{\om_1}^{-1}\\
				&\lesssim\delta^{-1}\sum_{\substack{m_1\in\mbZ^{2}\setminus\{0,-\om\}\\\abs{m_1}\leq\delta^{-1}}}\abs{m_1}^{-2+\eps}\abs{\om+m_1}^{-2}\lesssim\delta^{-1}\abs{\om}^{-2+2\eps}.
			\end{split}
		\end{equation*}
		The second term~\eqref{eq:double_sum_bound_first_order_regular_2} can be estimated by Lemma~\ref{lem:convolution_estimates} and~\eqref{eq:summation_estimates_-2},
		\begin{equation*}
			\begin{split}
				&\sum_{\substack{m_1\in\mbZ^{2}\setminus\{0,-\om\}\\\abs{m_1}\leq\delta^{-1}}}(1+\abs{\delta m_1}^{2})^{-1/2}\sum_{\substack{\om_1\in\mbZ^{2}\setminus\{0,\om,m_1,\om+m_1\}\\\om_1\succsim(\om-\om_1+m_1)}}\abs{\om_1-m_1}^{-2}\abs{\om-\om_1+m_1}^{-2}\abs{\om_1}^{-1}\\
				&\lesssim\delta^{-1}\sum_{\substack{m_1\in\mbZ^{2}\setminus\{0,-\om\}\\\abs{m_1}\leq\delta^{-1}}}\abs{m_1}^{-1}\abs{\om+m_1}^{-1}\sum_{\substack{\om_1\in\mbZ^{2}\setminus\{0,\om,m_1,\om+m_1\}\\\om_1\succsim(\om-\om_1+m_1)}}\abs{\om_1-m_1}^{-2}\abs{\om-\om_1+m_1}^{-2}\\
				&\lesssim\delta^{-1}\abs{\om}^{-2+2\eps}\sum_{\substack{m_1\in\mbZ^{2}\setminus\{0,-\om\}\\\abs{m_1}\leq\delta^{-1}}}\abs{m_1}^{-1}\abs{\om+m_1}^{-1}\lesssim\abs{\om}^{-2+3\eps}\delta^{-1}(1\vee\log(\delta^{-1})).
			\end{split}
		\end{equation*}
		Decomposing~\eqref{eq:sum_m_om_canonical_regular} into~\eqref{eq:double_sum_bound_first_order_special_cases_regular_1}--\eqref{eq:double_sum_bound_first_order_regular_3} and bounding those terms as given above yields
		\begin{equation*}
			\begin{split}
				&\sum_{\substack{m_1\in\mbZ^{2}\\\abs{m_1}\leq\delta^{-1}}}(1+\abs{\delta m_1}^{2})^{-1/2}\sum_{\om_1\in\mbZ^{2}\setminus\{0,\om\}}(1+\abs{\om_1-m_1}^{2})^{-1}(1+\abs{\om-\om_1+m_1}^{2})^{-1}\abs{\om_1}^{-1}\\
				&\lesssim_{\eps}\abs{\om}^{-2+3\eps}(1+\delta^{-1}+\delta^{-1}(1\vee\log(\delta^{-1})))\\
				&\lesssim\abs{\om}^{-2+3\eps}(1\vee(\delta^{-1}\log(\delta^{-1}))).
			\end{split}
		\end{equation*}
		In the last line, we used that uniformly in $x\geq0$,
		\begin{equation}\label{eq:xlogx_estimate}
			1+x(1\vee\log(x))\lesssim1\vee(x\log(x)).
		\end{equation}
		Indeed, to show~\eqref{eq:xlogx_estimate}, let $W_{0}$ be the principal branch of the Lambert W function. We distinguish between $x\in[0,\euler^{W_{0}(1)}]$, $x\in(\euler^{W_{0}(1)},\euler]$ and $x\in(\euler,\infty)$. For $x\in[0,\euler^{W_{0}(1)}]$, we estimate
		\begin{equation*}
			1+x(1\vee\log(x))=1+x\lesssim1=1\vee(x\log(x)).
		\end{equation*}
		For $x\in(\euler^{W_{0}(1)},\euler]$, we estimate
		\begin{equation*}
			1+x(1\vee\log(x))=1+x\lesssim x\log(x)=1\vee(x\log(x)).
		\end{equation*}
		For $x\in(\euler,\infty)$, we estimate
		\begin{equation*}
			1+x(1\vee\log(x))=1+x\log(x)\lesssim x\log(x)=1\vee(x\log(x)),
		\end{equation*}
		which yields~\eqref{eq:xlogx_estimate}. This yields the claim.
	\end{proof}
\end{details}
\newpage
\section{Glossary}
%
%
In this glossary we collect our frequently-used symbols.
\begin{center}
	\renewcommand{\arraystretch}{1.2}
	\begin{tabular}{lll}
		\multicolumn{3}{c}{Table of distribution spaces}\\
		\toprule
		Space & Description & Reference\\
		\midrule
		$C^{\infty}(\mbT^{2})$ & The smooth functions on $\mbT^{2}$. & Subsec.~\ref{sec:notations}\\
		$\mcS'(\mbT^{2})$ & The distributions on $\mbT^{2}$. & Subsec.~\ref{sec:notations}\\
		$\mcB_{p,q}^{\alpha}(\mbT^{2})$ & The completion of $C^{\infty}(\mbT^{2})$ under the Besov-norm $\norm{\,\cdot\,}_{\mcB^{\alpha}_{p,q}}$. & Def.~\ref{def:Besov_space}\\
		$\mcC^{\alpha}(\mbT^{2})$ & The H\"{o}lder--Besov space $\mcC^{\alpha}(\mbT^{2})=\mcB_{\infty,\infty}^{\alpha}(\mbT^{2})$. & Def.~\ref{def:Besov_space}\\
		$\mcH^{\alpha}(\mbT^{2})$ & The Sobolev space $\mcH^{\alpha}(\mbT^{2})=\mcB_{2,2}^{\alpha}(\mbT^{2})$. & Subsec.~\ref{sec:notations}\\
		$C_{T}E$ & The continuous functions $f\from[0,T]\to E$. & Subsec.~\ref{sec:notations}\\
		$C_{T}^{\kappa}E$ & The $\kappa$-H\"{o}lder continuous functions $f\from[0,T]\to E$. & Subsec.~\ref{sec:notations}\\
		$C_{\eta;T}E$ & The continuous functions $f\from(0,T]\to E$ with blow-up of at most $t^{-\eta}$. & Subsec.~\ref{sec:notations}\\
		$C_{\eta;T}^{\kappa}E$ & The $\kappa$-H\"{o}lder functions $f\from(0,T]\to E$ with blow-up of at most $t^{-\eta}$. & Subsec.~\ref{sec:notations}\\
		$\msL_{\eta;T}^{\kappa}\mcC^{\alpha}(\mbT^{2})$ & The weighted interpolation space & \\
		&$\msL_{\eta;T}^{\kappa}\mcC^{\alpha}(\mbT^{2})=C_{\eta;T}^{\kappa}\mcC^{\alpha-2\kappa}(\mbT^{2})\cap C_{\eta;T}\mcC^{\alpha}(\mbT^{2})$. & Def.~\ref{def:interpolation_space}\\
		$\rksnoise{\alpha}{\kappa}_{T}$ & The space of enhanced rough noises. & Def.~\ref{def:noise_space}\\
		$\mathscr{D}_T$ & The space of paracontrolled distributions. & Def.~\ref{def:paracontrolled_space}\\
		$F^{\cem}$ & The normed vector space $F$ equipped with a cemetery state $\cem$. & Def.~\ref{def:space_with_cemetery}\\
		$F^{\sol}_{\eta;T}$ & The continuous $f\from(0,T]\to F^{\cem}$ that do not resurrect, with weight at $0$. & Def.~\ref{def:sol_space_weight}\\
		\bottomrule & & \\
		\multicolumn{3}{c}{Table of noise objects and Feynman diagrams}\\
		\toprule
		Diagram & Description & Reference\\
		\midrule
		$\boldsymbol{\xi}$ & The vector-valued space-time white-noise $\boldsymbol{\xi}=(\xi^1,\xi^2)$. &\eqref{eq:definition_white_noise}\\
		$\ti$ & $=\vdiv\mcI[\het\boldsymbol{\xi}]$. &\eqref{eq:lolli_example}\\
		$\tl$ & $=\mbE[\vdiv\mcI[\ti\nabla\Phi_{\ti}]]$. & Subsec.~\ref{subsec:Feynman}\\
		$\ty$ & $=\vdiv\mcI[\ti\nabla\Phi_{\ti}]-\tl=\Ypsilon{10}$. & Subsec.~\ref{subsec:Feynman}\\
		$\tp$ & $=\ty \re\nabla \Phi_{\ti}+\nabla\Phi_{\ty}\re\ti=\PreThree{10}+\PreThree{20}+\PreCocktail{30}+\PreCocktail{50}+\PreCocktail{40}$. & Subsec.~\ref{subsec:Feynman}\\
		$\PreCocktail{50}$ & $=\PreCocktail{10}+\PreCocktail{20}$. & Subsec.~\ref{subsec:Feynman}\\
		$\tc$ & $=\nabla\mcI[\ti]\re\nabla\Phi_{\ti}+\nabla^2\mcI[\Phi_{\ti}]\re \ti=\Checkmark{10}+\Checkmark{20}+\Triangle{3}$. & Subsec.~\ref{subsec:Feynman}\\
		$\Triangle{3}$ & $=\Triangle{1}+\Triangle{2}$. & Subsec.~\ref{subsec:Feynman}\\
		$\mbX$ & $=(\ti,\ty,\tp,\tc)$ the renormalised enhancement. & Thm.~\ref{thm:enhancement_existence}\\
		$\varphi$ & $\in C^{\infty}(\mbR^2)$ of compact support, $\supp(\varphi)\subset B(0,1)$, & \\
		& even and such that $\varphi(0)=1$. & \eqref{eq:def_cut_off} \\
		$\psi_{\delta}$ & $=\sum_{\om\in\mbZ^2}\euler^{2\uppi\upi\inner{\om}{x}}\varphi(\delta\om)$, where $\delta>0$.& \eqref{eq:def_mollifiers} \\
		$\ti^{\delta}$ & $=\vdiv\mcI[\het(\psi_{\delta}\ast\boldsymbol{\xi})]$. &\eqref{eq:lolli_approximation_example}\\
		$\mbX^{\delta}$ & $=(\ti^{\delta},\ty^{\delta},\tp^{\delta},\tc^{\delta})$ the prelimiting renormalised enhancement. & Thm.~\ref{thm:enhancement_existence}\\
		$\ty_{\can}^{\delta}$ & $=\vdiv\mcI[\ti^{\delta}\nabla\Phi_{\ti^{\delta}}]={\Ypsilon{10}\,}^{\delta}+\tl^{\delta}$. & Subsec.~\ref{subsec:Feynman}\\
		$\tp_{\can}^{\delta}$ & $=\ty_{\can}^{\delta}\re\nabla \Phi_{\ti^{\delta}}+\nabla\Phi_{\ty_{\can}^{\delta}}\re\ti^{\delta}=\tp^{\delta}+{\PreThreeloop{10}\!}^{\delta}+{\PreThreeloop{20}\!}^{\delta}$. & Subsec.~\ref{subsec:Feynman}\\
		$\mbX_{\can}^{\delta}$ & $=(\ti^{\delta},\ty_{\can}^{\delta},\tp_{\can}^{\delta},\tc^{\delta})$ the canonical enhancement. & Thm.~\ref{thm:enhancement_existence}\\
		%
		%
		$\Ypsilon{1}$ & \eqref{eq:isometry_bound_delta=0} applied to $\Ypsilon{10}$. & Subsec.~\ref{sec:diagrams_of_order_2_and_3}\\
		%
		%
		$\PreCocktail{5}$ & \eqref{eq:isometry_bound_delta=0} applied to $\PreCocktail{50}$. & Subsec.~\ref{sec:Wick_contractions}\\
		%
		%
		%
		\bottomrule
	\end{tabular}
\end{center} 
\newpage
\section*{Declarations}
\begin{itemize}
	\item \textbf{Data and Material:} N/A
	\item \textbf{Code:} N/A
	\item \textbf{Funding:} Please see the acknowledgements in Section~\ref{sec:introduction}. 
	\item \textbf{Competing Interests:} The authors have no relevant financial or non-financial interests to disclose.
\end{itemize}
\printbibliography[heading=bibintoc,title={References}]

\end{document}